\pdfoutput=1
\documentclass{amsbook}
\usepackage[utf8]{inputenc}

\usepackage[margin=1in]{geometry}
\usepackage{microtype}

\usepackage{amsmath, amsfonts, amssymb, amsthm, amsopn, mathtools}
\usepackage{eucal}

\usepackage{tikz-cd}
\usetikzlibrary{patterns}
\usepackage{bm}
\usepackage{rotating}
\usepackage{xparse}
\usepackage{pict2e}
\usepackage{enumitem}
\usepackage{lmodern}
\usepackage[pdfusetitle,unicode,colorlinks]{hyperref}
\newtheoremstyle{introthms}
{}{}{\itshape}{}{\bfseries }{}{ }
{\thmname{#1} \thmnumber{#2}. \thmnote{\bfseries{(#3)}}}

\theoremstyle{introthms}
\newtheorem{introthm}{Theorem}

\theoremstyle{plain}
\newtheorem{theorem}[equation]{Theorem}
\newtheorem{proposition}[equation]{Proposition}

\newtheorem{lemma}[equation]{Lemma}
\newtheorem{corollary}[equation]{Corollary}

\theoremstyle{definition}
\newtheorem{definition}[equation]{Definition}
\newtheorem{example}[equation]{Example}
\newtheorem{remark}[equation]{Remark}

\newtheorem{construction}[equation]{Construction}

\let\scr=\mathcal

\let\phi=\varphi

\let\into=\hookrightarrow
\let\onto=\twoheadrightarrow

\def\AA{\scr A}
\def\BB{\scr B}
\def\CC{\scr C}
\def\DD{\scr D}
\def\EE{\scr E}
\def\GG{\scr G}
\def\II{\scr I}
\def\JJ{\scr J}
\def\KK{\scr K}
\def\LL{\scr L}
\def\OO{\scr O}

\def\RR{\scr R}
\def\SS{\scr S}

\def\WW{\scr W}
\def\XX{\scr X}
\def\YY{\scr Y}
\def\ZZ{\scr Z}

\def\BBB{\widehat{\BB}}

\def\SSS{\widehat{\SS}}

\def\SSSS{\hathat{\SS}}
\def\BBBB{\hathat{\BB}}

\def\bU{\mathbf{U}}
\def\bV{\mathbf{V}}
\def\bW{\mathbf{W}}

\DeclareMathOperator{\Cat}{Cat}
\DeclareMathOperator{\ICat}{\mathsf{Cat}}
\newcommand{\CatS}{\Cat_{\infty}}
\newcommand{\CatSS}{\widehat{\Cat}_\infty}

\DeclareMathOperator{\IFilt}{\mathsf{Filt}}
\DeclareMathOperator{\IwFilt}{\mathsf{wFilt}}
\DeclareMathOperator{\Acc}{Acc}
\DeclareMathOperator{\Pos}{Pos}

\newcommand{\LPr}{\operatorname{Pr}^{\operatorname{L}}}
\newcommand{\RPr}{\operatorname{Pr}^{\operatorname{R}}}
\newcommand{\LPrS}{\LPr_\infty}
\newcommand{\RPrS}{\RPr_\infty}
\newcommand{\ILPr}{\mathsf{Pr}^{\operatorname{L}}}
\newcommand{\IRPr}{\mathsf{Pr}^{\operatorname{R}}}
\newcommand{\LTop}{\operatorname{Top}^{\operatorname{L}}}
\newcommand{\RTop}{\operatorname{Top}^{\operatorname{R}}}
\newcommand{\LTopEt}{\operatorname{Top}^{\operatorname{L, \acute{e}t}}}

\newcommand{\LTopS}{\LTop_\infty}
\newcommand{\RTopS}{\RTop_\infty}
\newcommand{\LTopEtS}{\LTopEt_\infty}

\newcommand{\ILTop}{\mathsf{Top}^{\operatorname{L}}}
\newcommand{\IRTop}{\mathsf{Top}^{\operatorname{R}}}
\newcommand{\ihom}{\underline{\operatorname{Hom}}}

\DeclareMathOperator{\Fun}{Fun}
\DeclareMathOperator{\PSh}{PSh}
\DeclareMathOperator{\Ind}{Ind}
\DeclareMathOperator{\Shv}{Sh}
\DeclareMathOperator{\IShv}{\mathsf{Sh}}
\DeclareMathOperator{\Loc}{Loc}
\DeclareMathOperator{\ILoc}{\mathsf{Loc}}
\DeclareMathOperator{\LConst}{LConst}
\DeclareMathOperator{\ILConst}{\mathsf{LConst}}
\DeclareMathOperator{\Ret}{Ret}
\DeclareMathOperator{\IRet}{\mathsf{Ret}}
\DeclareMathOperator{\Sub}{Sub}
\DeclareMathOperator{\IPt}{\mathsf{Pt}}
\DeclareMathOperator{\Grpd}{Grpd}
\DeclareMathOperator{\Mod}{Mod}
\DeclareMathOperator{\IMod}{\mathsf{Mod}}
\DeclareMathOperator{\Alg}{Alg}
\DeclareMathOperator{\IAlg}{\mathsf{Alg}}
\DeclareMathOperator{\ICAlg}{\mathsf{CAlg}}
\DeclareMathOperator{\CAlg}{CAlg}
\DeclareMathOperator{\CMon}{CMon}
\DeclareMathOperator{\ICMon}{\mathsf{CMon}}
\DeclareMathOperator{\Op}{Op}
\DeclareMathOperator{\IOp}{\mathsf{Op}}
\DeclareMathOperator{\POp}{POp}
\DeclareMathOperator{\IPerf}{\mathsf{Perf}}
\DeclareMathOperator{\Contr}{\mathsf{Contr}}
\DeclareMathOperator{\Flat}{Flat}
\newcommand{\iFun}[1][\BB]{\underline{\mathsf{Fun}}_{#1}}
\DeclareMathOperator{\ISub}{\mathsf{Sub}}
\DeclareMathOperator{\IPos}{\mathsf{Pos}}
\DeclareMathOperator{\Disc}{Disc}
\newcommand{\ILLoc}{\ILoc^{\operatorname{L}}}
\newcommand{\LLoc}{\operatorname{Loc}^{\operatorname{L}}}
\DeclareMathOperator{\ICov}{\mathsf{Cov}}

\newcommand{\IFun}[1][\BB]{\underline{\mathsf{Fun}}_{#1}}
\newcommand{\IPSh}[1][\BB]{\underline{\mathsf{PSh}}_{#1}}
\newcommand{\IInd}[1][\BB]{\underline{\mathsf{Ind}}_{#1}}
\newcommand{\ISml}[1][\BB]{\underline{\mathsf{Small}}_{#1}}
\newcommand{\IFlat}[1][\BB]{\underline{\mathsf{Flat}}_{#1}}

\newcommand{\Simp}[1]{#1_{\Delta}}
\newcommand{\mSimp}[1]{#1_{\Delta}^+}
\newcommand{\SGrpd}{\mathsf{Seg}^\simeq}
\newcommand{\SGrpdEff}{\SGrpd_{\operatorname{eff}}}
\newcommand{\Free}[2][\BB]{\Univ[#1]\left[#2\right]}
\newcommand{\Comma}[3]{{#1}\downarrow_{#2}{#3}}

\DeclareMathOperator{\Cart}{Cart}
\DeclareMathOperator{\ICart}{\mathsf{Cart}}

\DeclareMathOperator{\ICocart}{\mathsf{Cocart}}
\DeclareMathOperator{\St}{St}
\DeclareMathOperator{\RFib}{RFib}
\DeclareMathOperator{\LFib}{LFib}
\DeclareMathOperator{\ILFib}{\mathsf{LFib}}
\DeclareMathOperator{\IRFib}{\mathsf{RFib}}

\DeclareMathOperator{\ITw}{\mathsf{Tw}}

\DeclareMathOperator{\Set}{Set}
\DeclareMathOperator{\Sp}{Sp}
\DeclareMathOperator{\ISp}{\mathsf{Sp}}

\DeclareMathOperator{\Fin}{Fin}
\DeclareMathOperator{\IFin}{\mathsf{Fin}}
\DeclareMathOperator{\Assoc}{Assoc}

\NewDocumentCommand{\Univ}{o}{%
	\IfNoValueTF{#1}{%
		\I{\Omega}%
	}{%
		\I{\Omega}_{#1}%
	}%
}

\DeclareMathOperator{\const}{const}
\DeclareMathOperator{\diag}{diag}
\DeclareMathOperator{\pr}{pr}
\DeclareMathOperator{\id}{id}
\DeclareMathOperator{\ev}{ev}

\DeclareMathOperator{\res}{res}
\DeclareMathOperator{\lift}{lift}

\newcommand{\bil}{{\operatorname{bil}}}
\newcommand{\cart}{{\operatorname{cart}}}
\newcommand{\acc}{{\operatorname{acc}}}
\newcommand{\fin}{{\operatorname{fin}}}
\newcommand{\LAdj}{{\operatorname{LAdj}}}
\newcommand{\RAdj}{{\operatorname{RAdj}}}
\newcommand{\disc}{{\operatorname{disc}}}
\newcommand{\Zar}{{\operatorname{Zar}}}
\newcommand{\ex}{{\operatorname{ex}}}
\newcommand{\lin}{{\operatorname{lin}}}
\newcommand{\dual}{{\operatorname{dual}}}
\newcommand{\reg}{{\operatorname{reg}}}
\newcommand{\alg}{{\operatorname{alg}}}
\newcommand{\geom}{{\operatorname{geom}}}
\newcommand{\lex}{{\operatorname{lex}}}
\newcommand{\cocart}{{\operatorname{cocart}}}
\newcommand{\cc}{{\operatorname{cc}}}
\newcommand{\gen}{{\operatorname{gen}}}
\newcommand{\gp}{{\operatorname{gpd}}}
\newcommand{\op}{{\operatorname{op}}}

\newcommand{\sh}{{\operatorname{sh}}}

\newcommand{\lax}{{\operatorname{lax}}}
\newcommand{\compact}{{\operatorname{cpt}}}
\renewcommand{\c}{{\operatorname{cont}}}
\newcommand{\Lad}{{\operatorname{L}}}
\newcommand{\Rad}{{\operatorname{R}}}

\newcommand{\cocont}[1]{#1\operatorname{-cc}}
\newcommand{\cont}[1]{#1\operatorname{-cont}}
\newcommand{\mult}[1]{#1\operatorname{-mult}}
\newcommand{\core}{\simeq}
\newcommand{\cpt}[1]{#1\operatorname{-cpt}}

\let\lim=\relax
\DeclareMathOperator*{\lim}{lim}
\DeclareMathOperator*{\colim}{colim}
\DeclareMathOperator{\Map}{map}
\newcommand{\map}[1]{\Map_{#1}}
\newcommand{\Hom}{\underline{\operatorname{Hom}}}
\newcommand{\Eq}{\mathrm{eq}}

\newcommand{\ord}[1]{\langle#1\rangle}
\newcommand{\Over}[2]{#1_{\hspace{-1pt}/#2}}
\newcommand{\Under}[2]{#1_{\hspace{-1pt}#2/}}
\newcommand{\sslash}{\mathbin{/\mkern-6mu/}}
\newcommand{\I}[1]{\mathsf{#1}}

\makeatletter
\g@addto@macro\bfseries{\boldmath}
\makeatother

\makeatletter
\newcommand{\hathatInternal}[2]{%
	\begingroup%
	\let\macc@kerna\z@%
	\let\macc@kernb\z@%
	\let\macc@nucleus\@empty%
	\widehat{\raisebox{#2}{\vphantom{\ensuremath{#1}}}\smash{\widehat{#1}}}%
	\endgroup%
}
\makeatother
\newcommand{\hathat}[1]{\mathchoice
	{\hathatInternal{#1}{.2ex}}
	{\hathatInternal{#1}{.2ex}}
	{\hathatInternal{#1}{-1.5pt}}
	{\hathatInternal{#1}{1pt}}
}

\numberwithin{equation}{subsection}
\numberwithin{section}{chapter}

\title{Presentability and topoi in internal higher category theory}
\author{Louis Martini}
\address{Alfred Getz' vei 1, 7034 Trondheim, Norway}
\email{louis.o.martini@ntnu.no}
\author{Sebastian Wolf}
\address{Universitätsstraße 31, 93053 Regensburg, Germany}
\email{sebastian1.wolf@mathematik.uni-regensburg.de}
\date{\today}

\begin{document}
\frontmatter
\begin{abstract}
	The goal of this article is to develop the theory of presentable categories and topoi internal to an arbitrary $\infty$-topos $\BB$.
	Our main results are internal analogues of Lurie's and Lurie-Simpson's characterisations of presentable $\infty$-categories and $\infty$-topoi.
	In the process, we introduce a theory of internal filteredness and accessible internal categories and establish a number of structural results about presentable $\BB$-categories such as adjoint functor theorems and the existence of an internal analogue of the Lurie tensor product.
	We also compare these internal notions with external variants.
	We show that $\BB$-modules embed fully faithfully into presentable $\BB$-categories and prove that there is an equivalence between topoi internal to $\BB$ and $\infty$-topoi over $\BB$.
	We also include a number of applications of our results, such as a general version of Diaconescu's theorem for $\infty$-topoi and a characterisation of locally contractible geometric morphisms in terms of smoothness.
\end{abstract}
\maketitle

\setcounter{tocdepth}{1}
\tableofcontents


\chapter*{Introduction}
	\subsection*{Motivation and Main results}
	Internal higher category theory is an approach to study sheaves of $\infty$-categories on an $\infty$-topos $\BB$ by regarding them as internal categories in $\BB$, i.e.\ as simplicial objects in $\BB$ that satisfy the Segal conditions and univalence. 
	Such structures, which we simply refer to as \emph{$\BB$-categories}, come up in many areas of algebra and geometry in which the objects of interest carry additional data that can be captured through a parametrisation by an $\infty$-topos.
	For instance, this idea has already found application in equivariant homotopy theory.
	Indeed, in their work on \emph{parametrized homotopy theory and higher category theory}, Barwick, Dotto, Glasman, Nardin and Shah study categories that are paramerized over presheaves on the orbit category of a finite group $ G $, see for instance \cite{barwick2016parametrized}, \cite{shah2016}, \cite{zbMATH07706485}, \cite{shah2022parametrized}.
	Today, the methods introuced there are widely used and have become standard tools in equivariant homotopy theory.
	A different kind of example arises from topology.
	Recall that any continuous map of topological spaces $X\to Y$ gives rise to a geometric morphism $\Shv(X) \to \Shv(Y)$ of $\infty$-topoi.
	From this one can obtain a $\Shv(Y)$-category $\IShv(X)$ whose underlying $\infty$-category recovers $\Shv(X)$ but carries additional information and captures topological properties of the map, see e.g. \cite{MW23}.
	
	The main reason why systematically studying sheaves of $ \infty $-categories through this lens is useful is that this often allows to hide the complexity of the topos $ \BB $.
	Indeed, using this perspective, many definitions and statements that are complicated to formulate or prove for sheaves of $ \infty $-categories become familiar statements from higher categories, just interpreted internally to $ \BB $.
	We have already made use of this point of view in our previous papers on the subject ~\cite{MYoneda, MWColimits,MCocartesian} and established an abundance of basic techniques for working with $\BB$-categories, such as the theory of adjunctions, Kan extensions, and (co)limits.
	In this paper, we continue the story by developing the theory of presentable $ \infty $-categories and $ \infty $-topoi internal to an $ \infty $-topos $ \BB $.
	
	With the rise of higher categories, particularly through Luries foundational work \cite{htt}, the notion of presentable $\infty$-categories has gained a central role in the theory of higher categories.
	This is due to the many favourable properties of presentable $\infty$-categories, such as the presence of adjoint functor theorems and the existence of a well-behaved and explicit tensor product.
	Furthermore, almost all cocomplete $ \infty $-categories that arise in practice are in fact presentable, which allows for wide applications of these general results.
	Therefore, it seems desirable to develop an analogue of this notion in the setting of $\BB$-categories, which is one of the main goals of this paper.
	
	In our previous paper \cite{MWColimits}, we developed a formalism that allows one to freely add colimits indexed by a fixed class of diagram shapes to an arbitrary $\BB$-category $\I{C}$.
	We apply this machinery to develop the notion of \emph{accessible} $\BB$-categories. 
	These are precisely the $\BB$-categories that are obtained by freely adding \emph{$\I{U}$-filtered} colimits to a $\BB$-category, where $\I{U}$ is what we call a \emph{sound doctrine}, a concept that is to be thought of as an internal analogue of regular cardinals.
	Combining accessibility with the condition of being cocomplete then yields the notion of presentability.
	Our first main theorem is the following characterization of presentable $ \BB $-categories:
	
	\begin{introthm}
		\label{thm:A}
		For a large $\BB$-category $\I{D}$, the following are equivalent:
		\begin{enumerate}
			\item $\I{D}$ is presentable;
			\item there is a $\BB$-category $\I{C}$ and a sound doctrine $\I{U}$ such that $\I{D}$ is a $\I{U}$-accessible Bousfield localisation of $\IPSh(\I{C})$;
			\item $\I{D}$ is accessible and cocomplete;
			\item there is a sound doctrine $\I{U}$ such that $\I{D}$ is $\I{U}$-accessible and $\I{D}^{\cpt{\I{U}}}$ is $\op(\I{U})$-cocomplete;
			\item there is a doctrine $\I{U}$ and a small $\op(\I{U})$-cocomplete $\BB$-category $\I{C}$ such that $\I{D}\simeq\IShv_{\Univ}^{\I{U}}(\I{C})$ (see Definition~\ref{def:Usheaves}).
			\item The following two conditions are satisfied:
				\begin{enumerate}
					\item the associated sheaf $\I{D}\colon\BB^\op\to\CatSS$ takes values in the $\infty$-category $\LPrS$ of presentable $\infty$-categories and colimit-preserving functors;
					\item for every map $s\colon B\to A$ in $\BB$ the associated transition functor $s^\ast\colon \I{D}(A)\to\I{D}(B)$ admits a left adjoint $s_!$, and for every pullback square in $\BB$ the induced commutative square in $\LPrS$ is left adjointable.
				\end{enumerate}
		\end{enumerate}
	\end{introthm}
	
	Note that the equivalence of conditions (1)-(5) is parallel to the characterization of presentable $\infty$-categories due to Lurie and Simpson~\cite[Theorem~5.5.1.1]{htt}.
	Condition (6), on the other hand, gives a very explicit and simple way of checking whether a $ \BB $-category is presentable in practice.
	Having such a plethora of equivalent characterisations of the notion of presentability at hand, it is straightforward to prove the expected theorems that revolve around these objects: among other things, we prove adjoint functor theorems for presentable $\BB$-categories, and we discuss how one can construct (internal) limits and colimits of such $\BB$-categories.
	Furthermore, we construct the \emph{tensor product} of presentable $\BB$-categories and we show:
	
	\begin{introthm}
        \label{thm:B}
		There is a strong monoidal and fully faithful functor $\Mod_{\BB}(\LPrS)\into \LPr(\BB)$, where $\BB$ is viewed as a presentably symmetric monoidal $\infty$-category via its cartesian monoidal structure and where $\LPr(\BB)$ is the $\infty$-category of presentable $\BB$-categories and cocontinuous functors.
		Furthermore this functor is an equivalence whenever $ \BB $ is generated under colimits by $ (-1) $-truncated objects.
	\end{introthm}
	
	Many of the explicit examples of presentable $\BB$-categories that we have in mind arise in the following way:
	To any geometric morphism $ f_\ast \colon \mathcal{X} \to \BB$ of $\infty$-topoi with left adjoint $f^\ast$, one can associate a presentable $\BB$-category $\I{X}$ given by the sheaf of $\infty$-categories $\Over{\XX}{f^\ast(-)}$ on $\BB$ (see \cite[Example 5.3.10]{MWColimits}). In fact, $\BB$-categories that arise via the above construction have even more favourable properties than just presentability. They give examples of $\BB$-topoi, the second theme of this paper.
	
	The message that we wish to convey is that studying relative $\infty$-topoi over an arbitrary base $\BB$ in this way is useful because it is in many ways not more complicated than the study of ordinary $\infty$-topoi.
	Indeed, using the machinery that we already developed, one can interpret the notion of an $\infty$-topos internally in $\BB$, leading to the theory of \emph{$\BB$-topoi}, which can be worked with in essentially the same way as with usual $\infty$-topoi.  However, when adopting an external point of view, the theory of $\BB$-topoi turns out to be equivalent to the theory of relative $\infty$-topoi over $\BB$ (see Theorem~\ref{thm:D}).
	
	More precisely, we use the straightening equivalence for (co)cartesian fibrations between $\BB$-categories in~\cite{MCocartesian} to define the notion of \emph{descent} for $ \BB $-categories by requiring that the unstraightening of the codomain fibration straightens to a limit preserving functor.
	A $ \BB $\emph{-topos} is then simply defined to be a presentable $ \BB $-category satisfying descent.
	Our first main result on $\BB$-topoi is the following characterisation:
	\begin{introthm}
		\label{thm:C}
		For a large $\BB$-category $\I{X}$, the following are equivalent:
		\begin{enumerate}
			\item $\I{X}$ is a $\BB$-topos;
			\item $\I{X}$ satisfies the \emph{internal Giraud axioms}:
			\begin{enumerate}
				\item $\I{X}$ is presentable;
				\item $\I{X}$ has universal colimits;
				\item groupoid objects in $\I{X}$ are effective;
				\item $\Univ$-colimits in $\I{X}$ are disjoint.
			\end{enumerate}
			\item there is a $\BB$-category $\I{C}$ such that $\I{X}$ arises as a left exact and accessible Bousfield localisation of $\IPSh(\I{C})$;
			\item $\I{X}$ is $\Univ$-cocomplete and takes values in the $\infty$-category $\LTopEtS$ of $\infty$-topoi and \'etale algebraic morphisms;
			\item $\I{X}$ is an $\LTopEtS$-valued sheaf that preserves pushouts.
		\end{enumerate}
	\end{introthm}
	The first three items in Theorem~\ref{thm:C} can be understood as the $\BB$-categorical analogue of Lurie's characterisation of $\infty$-topoi in~\cite[Theorem~6.1.0.1]{htt}. By contrast, the last two items of the theorem provide an \emph{external} characterisation of $\BB$-topoi, i.e.\ one in terms of the underlying sheaves of $\infty$-categories.
	
	Theorem~\ref{thm:C} is the main ingredient that goes into proving our comparison of $\BB$-topoi with $\infty$-topoi over $\BB$, which comprises our second main result on $\BB$-topoi:
	\begin{introthm}
		\label{thm:D}
		The global sections functor induces an equivalence $\Gamma\colon \RTop(\BB)\simeq\Over{(\RTopS)}{\BB}$ between the $\infty$-category of $\BB$-topoi and the $\infty$-category of $\infty$-topoi over $\BB$.
	\end{introthm}
	The inverse to the equivalence in Theorem~\ref{thm:D} can also be described very explicitly: 
	It is precisely the construction already described above that sends a geometric morphism $f_\ast \colon \XX \to \BB$ to the $ \BB $-category given by the sheaf of $\infty$-categories $f_\ast\Univ[\XX]=\Over{\XX}{f^\ast(-)}$ on $\BB$.
	
	The principle that relative $\infty$-topoi and internal topoi are to be thought of as essentially the same has been well-known for $1$-topoi for quite some time, thanks to the work of Moens~\cite{Moens1982}. In fact, large parts of Johnstone's \emph{Sketches of an Elephant}~\cite{johnstone2002} are precisely about passing back and forth between internal and relative topos theory. In this paper, we begin an analogous, albeit much more modest journey in the $\infty$-categorical context. Our long-standing goal is to study relative $\infty$-topoi that arise in geometric situations through the lens of internal higher topos theory.
	Here, we lay the groundwork for this and give a small insight into how the internal machinery can be used to answer questions in (relative) higher topos theory.
	
	\subsection*{Applications}
	As one application of our general theory on presentability, we study compactly generated $ \BB $-categories.
	We prove that for any $ \mathbb{E}_\infty $-ring object $ R $ in $ \BB $, the $ \BB $-category of $R$-modules $\IMod_{R}^\BB $ is compactly generated as a $ \BB $-category and we identify the compact generators with the subcategory of $ \IMod_{R}^\BB $ spanned by the dualisable objects, see Theorem~\ref{thm:Dualizable=Compact=Perfect}.
	If $ R $ is just an $ \mathbb{E}_\infty $-ring, this can be used to give a quick proof that the dualisable objects in the category of sheaves of $ R $-modules on $ \BB $ are precisely locally constant sheaves with perfect values, see Corollary~\ref{cor:Dualizable=LocallyPerfect}.
	
	Furthermore, we use the theory of filtered $ \BB $-categories, that we develope for the study of accessible $ \BB $-categories, to define the notion of a \emph{flat functor} from a small cateory $ \CC $ to an $ \infty $-topos $ \BB $.
	This allows to prove a general version of Diaconescu's theorem, see Theorem~\ref{cor:Diaconescu} which shows that flat functors are exactly those such that the Yoneda-extension of $ \PSh(\CC) \to \BB $ is left exact.
	In fact, we prove a more general relative version of Diaconsecu's Theorem over any $ \infty $-topos $ \BB $, and the above result is the special case $ \BB = \mathcal{S}$, see Theorem~\ref{thm:internalDiaconescu}.
	We also show that if $ \BB $ is hypercomplete, our notion of flatness agrees with the explicit condition given in \cite{Raptis2022}.
	
	Moreover, we apply the correspondence between $\BB$-topoi and $\infty$-topoi over $\BB$ to derive a formula for the pullback of $\infty$-topoi.
	In fact, if $\XX\to \BB$ and $\YY\to\BB$ are geometric morphisms, their pullback is given by the \emph{product} of the associated $\BB$-topoi.
	But in perfect analogy to the situation for $\infty$-topoi, the product of $\BB$-topoi can be computed by using the \emph{tensor product} of the underlying presentable $\BB$-categories, Proposition~\ref{prop:coproductsBTopoi}.
	Combined with Theorem~\ref{thm:B}, this implies that whenever $ \BB $ is generated under colimits by $ (-1) $-truncated objects, the fiber product of topoi $ \XX \times_\BB \YY $ agrees with the relative tensor product $ \XX \otimes_\BB \YY $ in $ \LPr $, see Corollary~\ref{cor:O-trunc_Tensor=Product}.
	
	We also apply the framework developed in this paper to derive an internal characterisation of \emph{smooth} geometric morphisms of $\infty$-topoi.
	A map $f_\ast\colon\XX\to\BB$ is said to be smooth if for every rectangle
	\begin{equation*}
		\begin{tikzcd}
			\WW^\prime\arrow[d]\arrow[r]& \ZZ^\prime\arrow[d]\arrow[r] & \XX\arrow[d]\\
			\WW\arrow[r] & \ZZ\arrow[r] & \BB
		\end{tikzcd}
	\end{equation*}
	of $\infty$-topoi in which both squares are pullbacks, the left square is vertically left adjointable. In other words, $f_\ast$ is smooth if it stably satisfies smooth base change. 
	In Theorem~\ref{thm:locallyContractibleSmoothBaseChange} we characterise the class of smooth geometric morphisms as precisely those whose associated $\BB$-topoi are \emph{locally contractible}, i.e.\ for which the left adjoint of the (internal) global sections functor (i.e.\ the unique geometric morphism to the final $\BB$-topos) admits a further left adjoint.
	There is a dual notion to smoothness, that of a geometric morphism being \emph{proper}~\cite[\S~7.3]{htt}.
	In a companion paper~\cite{MW23}, we will characterise this class of maps as the internally \emph{compact} $\BB$-topoi.

	Finally, we introduce the notion of a localic $\BB$-topos, the main result being that if $ \BB $ is itself a localic $ \infty $-topos, then there is an equivalence between localic $ \BB $-topoi and locales over the locale of subobjects of the terminal object $ \Sub_{\BB} $.
	We also study a number of compactness conditions for $ \BB $-locales, which we will be useful to apply the aforementioned characterization of proper geometric morphisms in topology.
	
	\subsection*{Other related work}
	In parametrised higher category theory, Jay Shah developed a notion of parametrised filteredness, which he used to construct free cocompletions under filtered colimits in the parametrised setting~\cite{zbMATH07706485}. Within the same framework, Kaif Hilman introduced and studied a notion of presentability~\cite{hilman2022}. 
	In particular, Hilman discussed parametrised accessibility, the adjoint functor theorem and tensor products of presentable parametrised $\infty$-categories.
	We would like to mention the work of Ad\'amek-Borceux-Lack-Rosick\'y~\cite{Adamek2002} and Charles Rezk~\cite{rezk2021} on the concept of accessibility relative to an arbitrary class of $1$-categories and $\infty$-categories, respectively. Their ideas substantially influenced our approach to internal accessibility and presentability.
	Moen's theorem, that can be thought of as a predecessor of Theorem~\ref{thm:D}, has been proven in the context of synthetic $ (\infty,1)$-category theory by Jonathan Weinberger in \cite{Weinberger2022}.
	Locally contractible geometric morphisms of $ \infty $-topoi were also introduced and studied, by Avraham Aizenbud and Shachar Carmeli in \cite{aizenbud2021relative} and by Marco Volpe in \cite{volpe2023operations}.

 \subsection*{Acknowledgements}
We would like to thank our respective advisors Rune Haugseng and Denis-Charles Cisinksi for their advice and support while writing this article.
We are also grateful to Geoerge Raptis and Daniel Schäppi for interesting discussions about Diaconescu's theorem.
Furthermore we would like to thank Mathieu Anel for interesting discussions about the content of this paper.
Moreover, we are grateful to Bastiaan Cnossen for many helpful suggestions.
The first-named author was partially supported by the project Pure Mathematics in Norway, funded by Trond Mohn Foundation and Tromsø Research Foundation.
The second-named author gratefully acknowledges support from the SFB 1085 Higher Invariants in Regensburg, funded by the DFG.

\mainmatter
\chapter{Background on $\BB$-category theory}
\label{chapter:Background}

\section{Conventions and notation}
In this section we recall the basic framework of higher categories internal to an $\infty$-topos from~\cite{MYoneda}. We refer the reader to~\cite{MYoneda} for proofs and a more detailed discussion.

\subsection{Initial remarks}
\label{sec:conventions}
Throughout this paper we freely make use of the language of higher category theory. We will generally follow a model-independent approach to higher categories. This means that as a general rule, all statements and constructions that are considered herein will be invariant under equivalences in the ambient $\infty$-category, and we will always be working within such an ambient $\infty$-category.

We denote by $\Delta$ the simplex category, i.e.\ the category of non-empty totally ordered finite sets with order-preserving maps. Every natural number $n\in\mathbb N$ can be considered as an object in $\Delta$ by identifying $n$ with the totally ordered set $\ord{n}=\{0,\dots n\}$. For $i=0,\dots,n$ we denote by $\delta^i\colon \ord{n-1}\to \ord{n}$ the unique injective map in $\Delta$ whose image does not contain $i$. Dually, for $i=0,\dots n$ we denote by $\sigma^i\colon \ord{n+1}\to \ord{n}$ the unique surjective map in $\Delta$ such that the preimage of $i$ contains two elements. Furthermore, if $S\subset n$ is an arbitrary subset of $k$ elements, we denote by $\delta^S\colon \ord{k}\to \ord{n}$ the unique injective map in $\Delta$ whose image is precisely $S$. In the case that $S$ is an interval, we will denote by $\sigma^S\colon \ord{n}\to \ord{n-k}$ the unique surjective map that sends $S$ to a single object. If $\CC$ is an $\infty$-category, we refer to a functor $C\colon\Delta^{\op}\to\CC$ as a simplicial object in $\CC$. We write $C_n$ for the image of $n\in\Delta$ under this functor, and we write $d_i$, $s_i$, $d_S$ and $s_S$ for the image of the maps $\delta^i$, $\sigma^i$, $\delta^S$ and $\sigma^S$ under this functor. Dually, a functor $C^{\bullet}\colon \Delta\to\CC$ is referred to as a cosimplicial object in $\CC$. In this case we denote the image of $\delta^i$, $\sigma^i$, $\delta^S$ and $\sigma^S$ by $d^i$, $s^i$, $d^S$ and $\sigma^S$.

The $1$-category $\Delta$ embeds fully faithfully into the $\infty$-category of $\infty$-categories by means of identifying posets with $0$-categories and order-preserving maps between posets with functors between such $0$-categories. We denote by $\Delta^n$ the image of $\ord{n} \in\Delta$ under this embedding.

\subsection{Set-theoretical foundations}
\label{sec:setTheory}
Once and for all we will fix three Grothendieck universes $\bU\in\bV\in\bW$ that contain the first infinite ordinal $\omega$. A set is \emph{small} if it is contained in $\bU$, \emph{large} if it is contained in $\bV$ and \emph{very large} if it is contained in $\bW$. An analogous naming convention will be adopted for $\infty$-categories and $\infty$-groupoids. The large $\infty$-category of small $\infty$-groupoids is denoted by $\SS$, and the very large $\infty$-category of large $\infty$-groupoids by $\SSS$. The (even larger) $\infty$-category of very large $\infty$-groupoids will be denoted by $\SSSS$. Similarly, we denote the large $\infty$-category of small $\infty$-categories by $\CatS$ and the very large $\infty$-category of large $\infty$-categories by $\CatSS$. We shall not need the $\infty$-category of very large $\infty$-categories in this article.

\subsection{$\infty$-topoi}
\label{sec:inftyTopoi}
For $\infty$-topoi $\AA$ and $\BB$, a \emph{geometric morphism} is a functor $f_\ast\colon \BB\to \AA$ that admits a left exact left adjoint, and an \emph{algebraic morphism} is a left exact functor $f^\ast\colon \AA\to \BB$ that admits a right adjoint. The \emph{global sections} functor is the unique geometric morphism $\Gamma_{\BB}\colon \BB\to \SS$ into the $\infty$-topos of $\infty$-groupoids $\SS$. Dually, the unique algebraic morphism originating from $\SS$ is denoted by $\const_{\BB}\colon \SS\to \BB$ and referred to as the \emph{constant sheaf} functor. We will often omit the subscripts if they can be inferred from the context.
For an object $A \in \BB$, we denote the induced étale geometric morphism by $(\pi_A)_\ast \colon \BB_{/A} \rightarrow \BB$.

\subsection{Universe enlargement}
\label{sec:universeEnlargement}
If $\BB$ is an $\infty$-topos, we define its \emph{universe enlargement} $\BBB=\Shv_{\SSS}(\BB)$, where the right-hand side denotes the $\infty$-category of presheaves $\BB^{\op}\to\SSS$ which preserve small limits; this is an $\infty$-topos relative to the larger universe $\bV$~\cite[Remark~6.3.5.17]{htt}. Moreover, the Yoneda embedding gives rise to an inclusion $\BB\into\BBB$ that commutes with small limits and colimits and with the internal hom~\cite[Proposition~2.4.4]{MYoneda}. The operation of enlarging universes is transitive: when defining the $\infty$-topos $\BBBB$ relative to $\bW$ as the universe enlargement of $\BBB$ with respect to the inclusion $\bV\in\bW$, the $\infty$-category $\BBBB$ is equivalent to the universe enlargement of $\BB$ with respect to $\bU\in\bW$~\cite[Remark~2.4.1]{MYoneda}.

\subsection{Factorisation systems} 
\label{sec:factorisationSystems}
If $\CC$ is a presentable $\infty$-category and if $S$ is a small set of maps in $\CC$, there is a unique factorisation system $(\LL,\RR)$ in which a map is contained in $\RR$ if and only if it is \emph{right orthogonal} to the maps in $S$, and where $\LL$ is dually defined as the set of maps that are left orthogonal to the maps in $\RR$. We refer to $\LL$ as the \emph{saturation} of $S$; this is the smallest set of maps containing $S$ that is stable under pushouts, contains all equivalences and is stable under small colimits in $\Fun(\Delta^1,\CC)$. An object $c\in\CC$ is said to be \emph{$S$-local} if the unique morphism $c\to 1$ is contained in $\RR$. 

If $\CC$ is cartesian closed, one can analogously construct a factorisation system $(\LL^\prime,\RR^\prime)$ in which $\RR^\prime$ is the set of maps in $\BB$ that are \emph{internally} right orthogonal to the maps in $S$. Explicitly, a map is contained in $\RR^\prime$ if and only if it is right orthogonal to maps of the form $s\times \id_c$ for any $s\in S$ and any $c\in \CC$. The left complement $\LL^\prime$ is comprised of the maps in $\CC$ that are left orthogonal to the maps in $\RR^\prime$ and is referred to as the \emph{internal} saturation of $S$. Equivalently, $\LL^\prime$ is the saturation of the set of maps $s\times\id_c$ for $s\in S$ and $c\in\CC$. An object $c\in\CC$ is said to be \emph{internally $S$-local} if the unique morphism $c\to 1$ is contained in $\RR^\prime$. 

Given any factorisation system $(\LL,\RR)$ in $\CC$ in which $\LL$ is the saturation of a small set of maps in $\CC$, the inclusion $\RR\into\Fun(\Delta^1,\CC)$ admits a left adjoint that carries a map $f\in\Fun(\Delta^1,\CC)$ to the map $r\in\RR$ that arises from the unique factorisation $f\simeq rl$ into maps $l\in \LL$ and $r\in \RR$. By taking fibres over an object $c\in\CC$, one furthermore obtains a Bousfield localisation $\Over{\CC}{c}\leftrightarrows\Over{\RR}{c}$ such that if $f\colon d\to c$ is an object in $\Over{\CC}{c}$ and if $f\simeq rl$ is its unique factorisation into maps $l\in \LL$ and $r\in \RR$, the adjunction unit is given by $l$.

\section{The language of $\BB$-categories}

\subsection{Simplicial objects, $\BB$-categories and $\BB$-groupoids}
\label{sec:BCategories}
If $\BB$ is an arbitrary $\infty$-topos, we denote by $\Simp{\BB}=\Fun(\Delta^{\op},\BB)$ the $\infty$-topos of simplicial objects in $\BB$. Note that the adjunction $(\const\dashv \Gamma)\colon\SS\leftrightarrows\BB$ yields via postcomposition an induced adjunction $(\const\dashv\Gamma)\colon \Simp{\SS}\leftrightarrows \Simp\BB$ on the level of simplicial objects. We will often implicitly identify a simplicial $\infty$-groupoid $K$ with its image in $\Simp\BB$ along $\const_{\BB}$ and refer to it as the \emph{constant simplicial object} in $\BB$ associated with $K$.

For every $n\geq 1$, we denote by $I^n=\Delta^1\sqcup_{\Delta^0}\cdots\sqcup_{\Delta^0}\Delta^1\into\Delta^n$ the $n$-spine, viewed as a simplicial $\infty$-groupoid. Furthermore, we denote by $E^1=(\Delta^0\sqcup\Delta^0)\sqcup_{(\Delta^1\sqcup\Delta^1)}\Delta^3$ the walking equivalence.
\begin{definition}[{\cite[Definitions~3.1.5 and~3.2.1]{MYoneda}}]
	\label{def:BCategories}
	A \emph{$\BB$-category} is a simplicial object $\I{C}\in\Simp\BB$ that is internally local with respect to $I^2\into\Delta^2$ (Segal conditions) and $E^1\to 1$ (univalence). We denote by $\Cat(\BB)\into\Simp\BB$ the full subcategory spanned by the $\BB$-categories. A \emph{$\BB$-groupoid} is a simplicial object $G\in\Simp\BB$ which is internally local with respect to $\Delta^1\to\Delta^0$. We denote by $\Grpd(\BB)\into\Simp\BB$ the full subcategory spanned by the $\BB$-groupoids.
\end{definition}

\begin{remark}[{\cite[Proposition~3.2.7]{MYoneda}}]
	\label{rem:BCategoriesExplicitly}
	More explicitly, a simplicial object $C$ is a $\BB$-category if and only if for all $n\geq 2$ the maps $C_n\to C_1\times_{C_0} \cdots\times_{C_0} C_1$ as well as the map $C_0\to (C_0\times C_0)\times_{C_1\times C_1} C_3$ are equivalences.
\end{remark}
\begin{remark}
	\label{rem:walkingEquivalence}
	There are several non-equivalent definitions of the walking equivalence. For example, Charles Rezk~\cite[\S~6]{rezk2001} defines the walking equivalence as the simplicial set $J$ that arises as the nerve of the category with two objects and a unique isomorphism between them. Our model $E^1$ (that we adopted from~\cite[Notation~1.1.12]{lurie2009b}), on the other hand, is comprised of a map together with \emph{separate} left and right inverses. Nevertheless, either choice gives rise to the same notion of $\BB$-categories: there is a natural map $E^1\to J$ which is contained in the internal saturation of $I^2\into\Delta^2$, i.e.\ which becomes an equivalence when imposing the Segal conditions. This can be extracted from the discussion in~\cite[\S~6]{rezk2001}, see also~\cite[\S~2.4]{rasekh2018}.
\end{remark}

\begin{proposition}[{\cite[Proposition~3.2.9, Remark~3.2.10 and Proposition~3.2.11]{MYoneda}}]
	\label{prop:CatBPresentable}
	The inclusion $\Cat(\BB)\into\Simp\BB$ preserves filtered colimits and admits a left adjoint which preserves finite products. Therefore, $\Cat(\BB)$ is presentable and an exponential ideal in $\Simp\BB$, so in particular cartesian closed.
\end{proposition}
We will denote by $\IFun(-,-)$ the internal hom in $\Cat(\BB)$ and refer to it as the \emph{functor $\BB$-category} bifunctor.

\begin{proposition}[{\cite[after Corollary~3.2.12]{MYoneda}}]
	\label{prop:GroupoidificationCore}
	A simplicial object in $\BB$ is a $\BB$-groupoid if and only if it is contained in the essential image of the diagonal embedding $\iota\colon\BB\into\Simp\BB$, and every $\BB$-groupoid is automatically a $\BB$-category. Moreover, the resulting embedding $\BB\simeq \Grpd(\BB)\into\Cat(\BB)$ admits both a left adjoint $(-)^{\gp}$ (the \emph{groupoidification functor}) and a right adjoint $(-)^\core$ (the \emph{core $\BB$-groupoid functor}). Explicitly, if $\I{C}$ is a $\BB$-category, one has $\I{C}^\gp \simeq \colim_{\Delta^{\op}}\I{C}$ and $\I{C}^{\simeq}\simeq \I{C}_0$.
\end{proposition}

\begin{definition}
	\label{def:oppositeBCategory}
	If $\I{C}$ is a $\BB$-category, we denote by $\I{C}^\op$ the simplicial object that is obtained by precomposing $\I{C}\colon\Delta^\op\to\BB$ with the involution $(-)^\op\colon\Delta\simeq\Delta$ that carries $\ord{n}$ (viewed as a $0$-category) to its opposite $\ord{n}^\op$. The simplicial object $\I{C}^\op$ is again a $\BB$-category that we refer to as the \emph{opposite $\BB$-category} of $\I{C}$.
\end{definition}

\begin{remark}
	\label{rem:OppositeIdentityGroupoids}
	The equivalence $(-)^\op\colon\Cat(\BB)\simeq\Cat(\BB)$ from Definition~\ref{def:oppositeBCategory} restricts to the identity on $\Grpd(\BB)$. In fact, this follows immediately from the observation that $\BB$-groupoids are constant in the simplicial direction, see Proposition~\ref{prop:GroupoidificationCore}.
\end{remark}

\begin{remark}[{\cite[\S~3.3]{MYoneda}}]
	\label{rem:functorialityBCategories}
	If $f_\ast\colon \BB\to\AA$ is a geometric morphism and if $f^\ast$ is the associated algebraic morphism, postcomposition induces an adjunction $f^\ast\dashv f_\ast\colon \Cat(\AA)\leftrightarrows\Cat(\BB)$. In particular, one obtains an adjunction $\const_{\BB}\dashv\Gamma_{\BB}\colon \CatS\leftrightarrows\Cat(\BB)$. We will often implicitly identify an $\infty$-category $\CC$ with the associated \emph{constant $\BB$-category} $\const_{\BB}(\CC)\in\Cat(\BB)$. Furthermore, if the geometric morphism $f_\ast$ is \emph{\'etale}, the further left adjoint $f_!$ of $f^\ast$ also induces a functor $f_!\colon \Cat(\BB)\to\Cat(\AA)$ that identifies $\Cat(\BB)$ with $\Over{\Cat(\AA)}{f_! 1}$. 
\end{remark}

By making use of the adjunction $\const_{\BB}\dashv\Gamma_{\BB}\colon \CatS\leftrightarrows\Cat(\BB)$ and the internal hom $\IFun(-,-)$ as well as the product $-\times -$ in $\Cat(\BB)$, one can define bifunctors
\begin{align}
	\Fun_{\BB}(-,-)=\Gamma_{\BB}\circ\IFun(-,-)&\colon \Cat(\BB)^\op\times\Cat(\BB)\to\CatS \tag{Functor $\infty$-category} \\
	(-)^{(-)}=\IFun(\const_{\BB}(-),-)&\colon \CatS^\op\times\Cat(\BB)\to\Cat(\BB) \tag{Powering}\\
	-\otimes - = \const_{\BB}(-)\times -&\colon\CatS\times\Cat(\BB)\to\Cat(\BB) \tag{Tensoring}
\end{align}
which fit into equivalences
\begin{equation*}
	\map{\Cat(\BB)}(-\otimes -, -)\simeq \map{\CatS}(-,\Fun_{\BB}(-,-))\simeq\map{\Cat(\BB)}(-, (-)^{(-)})
\end{equation*}
(see~\cite[\S~3.4]{MYoneda}). In particular, we have $\Fun_{\BB}(-,-)^\core\simeq\map{\Cat(\BB)}(-,-)$, so that $\Fun_{\BB}(-,-)$ gives rise to a $\CatS$-enrichement of $\Cat(\BB)$ and therefore an $(\infty,2)$-categorical enhancement of $\Cat(\BB)$~\cite[Remark~3.4.3]{MYoneda}.

\begin{remark}[{\cite[Proposition~3.1.2]{MYoneda}}]
	\label{rem:identificationObjectOfNMorphisms}
	There is an equivalence of functors $\id_{\Cat(\BB)}\simeq ((-)^{\Delta^\bullet})^\simeq$. In other words, for any $\BB$-category $\I{C}$ and any integer $n\geq 0$ one may canonically identify $\I{C}_n\simeq (\I{C}^{\Delta^n})_0$.
\end{remark}

We conclude this section with a remark on \emph{large} $\BB$-categories: observe that postcomposition with the universe enlargement $\BB\into\BBB$ from \S~\ref{sec:universeEnlargement} determines an inclusion $\Cat(\BB)\into\Cat(\BBB)$ that is natural in $\BB$ both with respect to geometric and algebraic morphisms of $\infty$-topoi~\cite[\S~3.3]{MYoneda}. Furthermore, the inclusion commutes with small limits and the internal hom~\cite[Proposition~3.4.1]{MYoneda} and therefore also the tensoring, powering and functor $\infty$-category bifunctors~\cite[Corollary~3.4.2]{MYoneda}. We refer to the objects in $\Cat(\BBB)$ as \emph{large} $\BB$-categories (or as $\BBB$-categories) and to the objects in $\Cat(\BB)$ as \emph{small} $\BB$-categories. If not specified otherwise, every $\BB$-category is small. Note, however, that by replacing the universe $\bU$ with the larger universe $\bV$ (i.e.\ by working internally to $\BBB$), every statement about $\BB$-categories carries over to one about large $\BB$-categories as well. Also, we will often omit specifying the relative size of a $\BB$-category if it is evident from the context, and we will continue writing $\IFun(\I{C},\I{D})$ for the internal hom even if $\I{C}$ and $\I{D}$ are large.

\subsection{$\BB$-categories as sheaves of $\infty$-categories}
\label{sec:parametrisedCategories}
One may equivalently regard a $\BB$-category as a \emph{sheaf} of $\infty$-categories on $\BB$, by which we mean a functor $\BB^{\op}\to\CatS$ that preserves small limits:
\begin{proposition}[{\cite[Proposition~3.5.1 and Remark~3.5.6]{MYoneda}}]
	\label{prop:equivalenceBCategoriesSheaves}
	There is a natural equivalence of $\infty$-categories $\Cat(\BB)\simeq\Shv_{\CatS}(\BB)$ that sends $\I{C}\in\Cat(\BB)$ to the sheaf $\Fun_{\BB}(\iota(-),\I{C})$ (where $\iota\colon \BB\into\Cat(\BB)$ is the diagonal embedding) and that restricts along the diagonal embedding $\iota\colon \BB\into\Cat(\BB)$ to the equivalence $\BB\simeq \Shv_{\SS}(\BB)$ that is determined by the Yoneda embedding.
\end{proposition}
Hereafter, we will often implicitly identify a $\BB$-category $\I{C}$ with the associated sheaf $\Fun_{\BB}(\iota(-),\I{C})$. That is, we usually write $\I{C}(A)=\Fun_{\BB}(\iota(A),\I{C})$ for the $\infty$-category of \emph{local sections} over $A\in \BB$, and we write $s^\ast\colon \I{C}(A)\to\I{C}(B)$ for the restriction functor along a map $s\colon B\to A$ in $\BB$.

\begin{remark}[{cf.~\cite[Remark~3.1.1]{MYoneda}}]
	\label{rem:sheafFromBCategoryExplicitly}
	More explicitly, the $\infty$-category $\I{C}(A)=\Fun_{\BB}(\iota(A),\I{C})$ is given by the complete Segal space whose space of $n$-morphisms is given by the $\infty$-groupoid $\map{\BB}(A, \I{C}_n)$. In particular, the equivalence $\Cat(\BB)\simeq\Shv_{\CatS}(\BB)$ from Proposition~\ref{prop:equivalenceBCategoriesSheaves} commutes both with taking core $\BB$-groupoids and opposite $\BB$-categories, in the sense that we have equivalences of sheaves $\I{C}^\core(-)\simeq\I{C}(-)^\core$ and $\I{C}^\op(-)\simeq\I{C}(-)^\op$.
\end{remark}

\begin{remark}
	\label{rem:internalVsParametrised}
	One may interpret Proposition~\ref{prop:equivalenceBCategoriesSheaves} as a correspondence between \emph{internal} and \emph{parametrised} higher category theory. Both approaches have their specific advantages: the upshot of the internal approach is that one can often use a statement about $\infty$-categories and simply interpret it internally in $\BB$ in order to obtain the corresponding statement for $\BB$-categories. On the other hand, it is usually easier to construct a particular $\BB$-category via its associated sheaf of $\infty$-categories. In fact, most examples that are of practical interest arise in this way.
\end{remark}

\begin{remark}[{\cite[\S~3.5]{MYoneda}}]
	\label{rem:BCategoriesSheavesFunctoriality}
	The equivalence $\Cat(\BB)\simeq\Shv_{\CatS}(\BB)$ is natural in $\BB$: if $f_\ast\colon \BB\to\AA$ is a geometric morphism and $f^\ast$ denotes its left adjoint, one obtains commutative squares
	\begin{equation*}
		\begin{tikzcd}
			\Cat(\BB)\arrow[d, "f_\ast"]\arrow[r, "\simeq"] & \Shv_{\CatS}(\BB)\arrow[d, "f_\ast"] & & \Cat(\BB)\arrow[from=d, "f^\ast"]\arrow[r, "\simeq"] & \Shv_{\CatS}(\BB)\arrow[from=d, "f^\ast"]\\
			\Cat(\AA)\arrow[r, "\simeq"] & \Shv_{\CatS}(\AA) && \Cat(\AA)\arrow[r, "\simeq"] & \Shv_{\CatS}(\AA).
		\end{tikzcd}
	\end{equation*}
	Explicitly, $f_\ast\colon \Shv_{\CatS}(\BB)\to\Shv_{\CatS}(\AA)$ is given by restriction along $f^\ast\colon \AA\to\BB$. In particular, we may identify $\I{C}(1)\simeq\Gamma_{\BB}(\I{C})$ for every $\BB$-category $\I{C}$. Furthermore,  $f^\ast\colon\Shv_{\CatS}(\AA)\to\Shv_{\CatS}(\BB)$ is given by left Kan extension along $f^\ast\colon\AA\to\BB$. Thus, if the latter functor admits an additional left adjoint $f_!$, then $f^\ast\colon\Shv_{\CatS}(\AA)\to\Shv_{\CatS}(\BB)$ is simply given by precomposition with $f_!$.
\end{remark}

\begin{remark}[{\cite[Proposition~3.5.1]{MYoneda}}]
	\label{rem:BCategoriesSheavesSize}
	The equivalence between $\BB$-categories and sheaves of $\infty$-categories respects universe enlargement in the following sense: there is a commutative square
	\begin{equation*}
		\begin{tikzcd}
			\Cat(\BB)\arrow[r, "\simeq"]\arrow[d, hookrightarrow] & \Shv_{\CatS}(\BB)\arrow[d, hookrightarrow]\\
			\Cat(\BBB)\arrow[r, "\simeq"] & \Shv_{\CatSS}(\BB)
		\end{tikzcd}
	\end{equation*}
	in which the lower horizontal equivalence is obtained by sending a large $\BB$-category $\I{C}$ to $\Fun_{\BBB}(\iota(-), \I{C})$, where $\iota\colon \BB\into\BBB\into \Cat(\BBB)$ is the inclusion.
\end{remark}

We conclude this section by noting that the sheaf-theoretic perspective on $\BB$-categories also gives rise to a \emph{fibrational} point of view: on account of the inclusion $\Shv_{\CatSS}(\BB)\into\PSh_{\CatSS}(\BB)$ and by making use of the straightening/unstraightening equivalence $\PSh_{\CatSS}(\BB)\simeq\Cart(\BB)$ between $\CatSS$-valued presheaves on $\BB$ and \emph{cartesian fibrations} over $\BB$ (see~\cite[\S~3.2]{htt}), we obtain a full embedding $\Cat(\BBB)\into\Cart(\BB)$ which sends a (large) $\BB$-category $\I{C}$ to its underlying cartesian fibration $\int\I{C}\to\BB$.

\subsection{Objects and morphisms in $\BB$-categories}
\label{sec:objectsMorphisms}
Observe that by combining Proposition~\ref{prop:equivalenceBCategoriesSheaves} with the two-variable adjunctions between the bifunctors $\Fun_{\BB}(-,-)$, $-\otimes -$ and $(-)^{(-)}$, one obtains equivalences
\begin{equation*}
	\I{C}^{\Delta^n}(A)^\core\simeq\map{\Cat(\BB)}(A, \I{C}^{\Delta^n})\simeq \map{\Cat(\BB)}(\Delta^n\otimes A, \I{C})\simeq \map{\CatS}(\Delta^n, \I{C}(A))
\end{equation*}
for every $A\in\BB$, every $\I{C}\in\Cat(\BB)$ and each $n\in\mathbb N$
(where we leave the diagonal embedding $\BB\into\Cat(\BB)$ implicit). Moreover, by combining Proposition~\ref{prop:GroupoidificationCore} with Remark~\ref{rem:identificationObjectOfNMorphisms}, we may furthermore compute
\begin{equation*}
	\map{\Cat(\BB)}(A,\I{C}^{\Delta^n})\simeq\map{\BB}(A, \I{C}_n).
\end{equation*}
In other words, the datum of a map $A\to \I{C}^{\Delta^n}$ in $\Cat(\BB)$ is equivalent to that of a map $\Delta^n\otimes A\to \I{C}$ in $\Cat(\BB)$, a map $A\to\I{C}_n$ in $\BB$ as well as a functor $\Delta^n\to\I{C}(A)$ of $\infty$-categories.

\begin{definition}
	\label{def:nMorphisms}
	Let $\I{C}$ be a $\BB$-category and let $A\in\BB$ be an object. For a given integer $n\geq 0$, an \emph{$n$-morphism in $\I{C}$ in context $A$} is a map $A\to \I{C}^{\Delta^n}$ in $\Cat(\BB)$. If $n=0$, we simply speak of an \emph{object} in $\I{C}$ in context $A$, and for $n=1$ we refer to such a map as a \emph{morphism} in $\I{C}$ in context $A$. Given objects $c,d\colon A\rightrightarrows \I{C}$, one defines the \emph{mapping $\Over{\BB}{A}$-groupoid} $\map{\I{C}}(c,d)$ as the pullback
	\begin{equation*}
		\begin{tikzcd}
			\map{\I{C}}(c,d)\arrow[r]\arrow[d] & \I{C}_1\arrow[d, "{(d_1,d_0)}"]\\
			A\arrow[r, "{(c,d)}"]& \I{C}_0\times\I{C}_0.
		\end{tikzcd}
	\end{equation*}
	We denote a section $f\colon A\to \map{\I{C}}(c,d)$ by $f\colon c\to d$. 
\end{definition}

\begin{remark}[{\cite[\S~3.6]{MYoneda}}]
	Equivalently, the mapping $\Over{\BB}{A}$-groupoid $\map{\I{C}}(c,d)$ can be defined as the pullback of $(d_1,d_0)\colon\I{C}^{\Delta^1}\to\I{C}\times\I{C}$ along $(c,d)\colon A\to\I{C}\times\I{C}$.
\end{remark}

\begin{remark}
	Viewed as an $\SS$-valued sheaf on $\Over{\BB}{A}$, the object $\map{\I{C}}(c,d)$ from Definition~\ref{def:nMorphisms} is given by the assignment
	\begin{equation*}
		\Over{\BB}{A}\ni(s\colon B\to A)\mapsto \map{\I{C}(B)}(s^\ast c,s^\ast d)
	\end{equation*}
	where $s^\ast c= c s$ and likewise for $d$.
\end{remark}

More generally, if $c_0,\dots,c_n$ are objects in context $A$ in $\I{C}$, one writes $\map{\I{C}}(c_0,\dots,c_n)$ for the pullback of $(d_n,\dots,d_0)\colon \I{C}_n\to \I{C}_0^{n+1}$ along the map $(c_0,\dots,c_n)\colon A\to \I{C}_0^{n+1}$. Using the Segal conditions, one obtains an equivalence
\begin{equation*}
	\map{\I{C}}(c_0,\dots,c_n)\simeq\map{\I{C}}(c_0,c_1)\times_A\cdots\times_A\map{\I{C}}(c_{n-1},c_n).
\end{equation*}
By combining this identification with the map $\map{\I{C}}(c_0,\dots,c_n)\to\map{\I{C}}(c_0,c_n)$ that is induced by the map $d_{\{0,n\}}\colon \I{C}_n\to\I{C}_1$, one obtains a composition map
\begin{equation*}
	\map{\I{C}}(c_0,c_1)\times_A\cdots\times_A\map{\I{C}}(c_{n-1},c_n)\to\map{\I{C}}(c_0,c_n).
\end{equation*}
Given maps $f_i\colon c_{i-1}\to c_i$ in $\I{C}$ for $i=1,\dots, n$, we write $f_1\cdots f_n$ for their composition. By making use of the simplicial identities, it is straightforward to verify that composition is associative and unital, i.e.\ that the relations $f(gh)\simeq (fg)h$ and $f\id\simeq f\simeq \id f$ as well as their higher analogues hold whenever they make sense, see~\cite[Proposition~5.4]{rezk2001} for a proof.

\begin{remark}
	\label{rem:contexts}
	As a $\BB$-category $\I{C}$ is determined by the associated sheaf of $\infty$-categories on $\BB$ but not just by the underlying $\infty$-category $\Gamma_{\BB}(\I{C})$ of global sections, it is crucial that we allow objects and morphisms in $\I{C}$ to have arbitrary context $A\in\BB$. In other words, we need to allow objects and morphisms to be only \emph{locally} defined, where by the term \emph{local} we allude to the point of view that the base $\infty$-topos $\BB$ can be thought of as a spatial object. Alternatively, this phenomenon can be viewed as a shadow of the notion of contexts in type theory (hence the name), where they are needed to keep track of the types of the variables that occur in a formula. More precisely, when regarding the theory of $\BB$-categories as a model of simplicial homotopy type theory~\cite{shulman2017}, the type-theoretic notion of contexts exactly translates into our notion of contexts.
\end{remark}

\begin{remark}
	\label{rem:localityPrincipleObjectsMorphisms}
	At first, the fact that objects and morphisms of a $\BB$-category $\I{C}$ have non-global context $A$ might appear to complicate things, but in practice this is usually not the case: in fact, by making use of the adjunction $(\pi_A)_!\dashv\pi_A^\ast\colon \Over{\BB}{A}\leftrightarrows\BB$ and by the observations made in Remark~\ref{rem:functorialityBCategories}, the datum of an object $c\colon A\to \I{C}$ precisely corresponds to that of an object $\bar c\colon1_{\Over{\BB}{A}}\to \pi_A^\ast\I{C}$, where $\pi_A^\ast\I{C}\in\Cat(\Over{\BB}{A})$ is the image of $\I{C}$ along the base change functor $\pi_A^\ast\colon \Cat(\BB)\to\Cat(\Over{\BB}{A})\simeq\Over{\Cat(\BB)}{A}$. In other words, upon replacing $\BB$ with $\Over{\BB}{A}$ and $\I{C}$ with $\pi_A^\ast\I{C}$, object in context $A$ are turned into objects in global context. Very often, we will make use of this correspondence in order to be able to restrict our attention to objects and morphisms in global context.
\end{remark}

\begin{remark}
	\label{rem:tautologicalObject}
	Observe that for every $\BB$-category $\I{C}$ there is a distinguished object $\tau\colon \I{C}_0\to \I{C}$ that is determined by the counit of the adjunction $\iota\dashv (-)_0\colon \Cat(\BB)\leftrightarrows\BB$ from Proposition~\ref{prop:GroupoidificationCore}. We refer to $\tau$ as the \emph{tautological} object of $\I{C}$. By definition, \emph{every} object $c\colon A\to\I{C}$ arises as a pullback of $\tau$, in the sense that we have $c\simeq c^\ast \tau$ (where $c^\ast\colon \I{C}(\I{C}_0)\to\I{C}(A)$ is the restriction functor). In that way, many questions about an arbitrary object in a $\BB$-category can be reduced to questions about the tautological object.
\end{remark}

We conclude this section with a discussion of \emph{equivalences} in $\BB$-categories. To that end, given any object $c\colon A\to\I{C}$ in a $\BB$-category $\I{C}$, let us denote by $\id_c\colon c\to c$ the morphism that is determined by the lift $s_0 c\colon A\to \I{C}_0\to\I{C}_1$ of $(c,c)\colon A\to \I{C}_0\times\I{C}_0$.
\begin{definition}
	\label{def:equivalencesBCategories}
	A morphism $f\colon c\to d$ in $\I{C}$ is an \emph{equivalence} if there are maps $g\colon c\to d$ and $h\colon c\to d$ (all in context $A$) such that $gf\simeq \id_c$ and $fh\simeq \id_d$.
\end{definition}

As a consequence of univalence, one finds:
\begin{proposition}[{\cite[Corollary~3.6.3]{MYoneda}}]
	\label{prop:characterisationEquivalences}
	A map $f\colon A\to \I{C}^{\Delta^1}$ in a $\BB$-category $\I{C}$ is an {equivalence} precisely if it factors through $s_0\colon \I{C}\into\I{C}^{\Delta^1}$.
\end{proposition}
In other words, every equivalence $f\colon A\to\I{C}_1$ is equivalent (in the $\infty$-groupoid $\I{C}_1(A)$) to an identity.

\subsection{Fully faithful functors and full subcategories}
\label{sec:fullyFaithfulFunctors}
A functor $f\colon \I{C}\to\I{D}$ between $\BB$-categories is said to be \emph{fully faithful} if it is internally right orthogonal to the map $\Delta^0\sqcup\Delta^0\to \Delta^1$. Dually, a functor is \emph{essentially surjective} if is (internally) left orthogonal to the class of fully faithful functors. Therefore, it formally follows that fully faithful functors are stable under small limits in $\Fun(\Delta^1,\Cat(\BB))$ and are preserved by the endofunctor $\IFun(\I{C},-)$ for every $\BB$-category $\I{C}$~\cite[Proposition~3.8.4]{MYoneda}. Moreover, a functor of $\BB$-categories is an equivalence if and only if it is fully faithful and essentially surjective~\cite[Proposition~3.8.3]{MYoneda}, and every functor can be uniquely factored into an essentially surjective and a fully faithful functor. In other words, the \emph{essential image} of a functor between $\BB$-categories is well-defined.

Fully faithful and essentially surjective functors can be characterised as follows:
\begin{proposition}[{\cite[Proposition~3.8.6 and~3.8.7]{MYoneda}}]
	\label{prop:classificationFullSubcategories}
	For a functor $f\colon \I{C}\to\I{D}$ of $\BB$-categories, the following are equivalent:
	\begin{enumerate}
		\item The functor $ f $ is fully faithful;
		\item the square
		\begin{equation*}
			\begin{tikzcd}
				\I{C}_1\arrow[r, "f_1"]\arrow[d] & \I{D}_1\arrow[d]\\
				\I{C}_0\times \I{C}_0\arrow[r, "f_0\times f_0"] & \I{D}_0\times \I{D}_0
			\end{tikzcd}
		\end{equation*}
		is a pullback;
		\item for every $A\in\BB$ and any two objects $c_0,c_1\colon A\to \I{C}$ in context $A$, the morphism
		\begin{equation*}
			\map{\I{C}}(c_0,c_1)\to\map{\I{D}}(f(c_0), f(c_1))
		\end{equation*}
		that is induced by $ f $ is an equivalence in $\Over{\BB}{A}$;
		\item for every $A\in \BB$ the functor $f(A)\colon \I{C}(A)\to\I{D}(A)$ of $\infty$-categories is fully faithful.
	\end{enumerate}
\end{proposition}
\begin{proposition}[{\cite[Corollary~3.8.12]{MYoneda}}]
	\label{prop:characterisationEssentiallySurjective}
	A functor $f\colon \I{C}\to\I{D}$ is essentially surjective if and only if $f_0\colon\I{C}_0\to\I{D}_0$ is a cover (i.e.\ an effective epimorphism) in $\BB$.
\end{proposition}

Fully faithful functors are in particular $(-1)$-truncated maps in $\Cat(\BB)$, hence the full subcategory $\Sub_{\mathrm{full}}(\I{D})\into\Over{\Cat(\BB)}{\I{D}}$ that is spanned by the fully faithful functors into $\I{D}$ is a poset whose objects we call \emph{full subcategories} of $\I{D}$.
\begin{proposition}[{\cite[Proposition~3.9.3]{MYoneda}}]
	\label{prop:fullSubcategoriesParametrisation}
	Taking core $\BB$-groupoids determines an equivalence of posets $\Sub_{\mathrm{full}}(\I{D})\simeq \Sub(\I{D}_0)$ between the poset of full subcategories of $\I{D}$ and the poset of subobjects of $\I{D}_0\in\BB$. 
\end{proposition}
In particular, Proposition~\ref{prop:fullSubcategoriesParametrisation} implies that specifying a full subcategory of $\I{D}$ is equivalent to specifying a subobject of $\I{D}_0$. Therefore, if $(d_i\colon A_i\to \I{D})_{i\in I}$ is a family of objects in $\I{D}$, we may define the full subcategory of $\I{D}$ that is \emph{spanned} by these objects as the unique full subcategory of $\I{D}$ whose core $\BB$-groupoid is given by the image of the induced morphism $(d_i)\colon\bigsqcup_i A_i\to \I{D}$ in $\BB$~\cite[Definition~3.9.7]{MYoneda}. Note that this is possible even if the family is large~\cite[Remark~3.9.8]{MYoneda}.

\subsection{Monomorphisms and non-full subcategories}
A functor $f\colon \I{C}\to\I{D}$ between $\BB$-categories is a monomorphism (i.e.\ a $(-1)$-truncated map in the $\infty$-category $\Cat(\BB)$) precisely if it is internally right orthogonal to the codiagonal $\Delta^0\sqcup\Delta^0\to\Delta^0$. Dually, we say that a functor is a \emph{strong epimorphism} if it is (internally) left orthogonal to the class of monomorphisms. As in the previous section, this formally implies that monomorphisms satisfy certain stability conditions and that every functor can be uniquely factored into a strong epimorphism followed by a monomorphism. In other words, the \emph{$1$-image} of a functor is well-defined.

We have the following characterisation of monomorphisms of $\BB$-categories:
\begin{proposition}[{\cite[Proposition~B.1.2]{MWColimits}}]
	\label{prop:characterisationMonomorphismMappingGroupoids}
	Let $f\colon\I{C}\to \I{D}$ be a functor between $\BB$-categories. Then the following are equivalent:
	\begin{enumerate}
		\item $f$ is a monomorphism;
		\item $f^{\core}$ is a monomorphism in $\BB$, and for any $A\in \BB$ and any two objects $c_0,c_1\colon A\to \I{C}$ in context $A\in\BB$, the morphism
		\begin{equation*}
			\map{\I{C}}(c_0,c_1)\to\map{\I{D}}(f(c_0), f(c_1))
		\end{equation*}
		that is induced by $f$ is a monomorphism in $\Over{\BB}{A}$;
		\item for every $A\in \BB$ the functor $f(A)\colon \I{C}(A)\to\I{D}(A)$ is a monomorphism of $\infty$-categories;
	\end{enumerate}
\end{proposition}
\begin{example}[{\cite[Example~B.1.3]{MWColimits}}]
	For every $\BB$-category $\I{C}$, the map $\I{C}^\core\to\I{C}$ is a monomorphism.
\end{example}

We may now define the poset $\Sub_{\Cat(\BB)}(\I{C})$ of \emph{subcategories} of a $\BB$-category $\I{C}$ as the full subcategory of $(-1)$-truncated objects in $\Over{\Cat(\BB)}{\I{C}}$. We would like to have an analogue of Proposition~\ref{prop:fullSubcategoriesParametrisation} for non-full subcategories that gives us a way to construct a subcategory of $\I{C}$ by specifying a collection of morphisms in $\I{C}$. To that end, note that evaluation at $\ord{1}\in\Delta$ gives rise to a map $(-)_1\colon \Sub_{\Cat(\BB)}(\I{C})\to\Sub_{\BB}(\I{C}_1)$. We now obtain:
\begin{proposition}[{\cite[Proposition~B.2.7]{MWColimits}}]
	\label{prop:parametrisationSubcategories}
	For every $\BB$-category $\I{C}$, the canonical morphism of posets $ (-)_1\colon \Sub_{\Cat(\BB)}(\I{C})\to\Sub_{\BB}(\I{C}_1)$ exhitibs $\Sub_{\Cat(\BB)}(\I{C})$ as a reflective subposet of $\Sub_{\BB}(\I{C}_1)$. In other words, $(-)_1$ is fully faithful and has a left adjoint $\langle - \rangle\colon \Sub_{\BB}(\I{C}_1)\to\Sub_{\Cat(\BB)}(\I{C})$ that carries a subobject $E\into \I{C}_1$ to the subcategory $\langle E\rangle\into \I{C}$ \emph{generated} by $E$.
\end{proposition}
In the situation of Proposition~\ref{prop:parametrisationSubcategories}, note that we in general have little control over the subcategory $\langle E\rangle$: we don't know which objects of $\I{C}$ are contained in $\langle E\rangle$, and there might also be many more morphisms in $\langle E \rangle$ than are contained in the subobject $E\into \I{C}_1$. The issue is, of course, that a priori the maps in $E$ need not be closed under composition and equivalences. Luckily, this is the only obstruction:
\begin{proposition}[{\cite[Proposition~B.2.9]{MWColimits}}]
	\label{prop:classificationSubcategories}
	For any $\BB$-category $\I{C}$, a subobject $E\into\I{C}_1$ lies in the essential image of the inclusion $\Sub_{\Cat(\BB)}(\I{C})\into\Sub_{\BB}(\I{C}_1)$ if and only if
	\begin{enumerate}
		\item it is closed under equivalences, i.e.\ the map $(s_0d_1,s_0d_0)\colon E\sqcup E\to \I{C}_1$ factors through $E\into\I{C}_1$;
		\item it is closed under composition, i.e.\ the restriction of the composition map $d_1\colon \I{C}_1\times_{\I{C}_0}\I{C}_1\to \I{C}_1$ along the inclusion $E\times_{\I{C}_0}E\into\I{C}_1\times_{\I{C}_0}\I{C}_1$ factors through $E\into\I{C}_1$.
	\end{enumerate}
\end{proposition}

\subsection{The $\BB$-categories of $\BB$-groupoids and $\BB$-categories}
\label{sec:universe}
By straightening the codomain fibration $\Fun(\Delta^1,\BB)\to\BB$, one obtains a functor $\Over{\BB}{-}\colon \BB^{\op}\to\CatSS$ that preserves small limits since $\BB$ is an $\infty$-topos~\cite[Proposition~6.1.3.9]{htt}. In other words, $\Over{\BB}{-}$ is a sheaf of (large) $\infty$-categories and therefore (by Remark~\ref{rem:BCategoriesSheavesSize}) determined by a large $\BB$-category $\Univ[\BB]$ that we refer to as the \emph{universe for $\BB$-groupoids}~\cite[\S~3.7]{MYoneda}. We will often omit the subscript if it is clear from the context. By definition, we have equivalences $\Univ(A)\simeq\Over{\BB}{A}\simeq\Grpd(\Over{\BB}{A})$. In other words, the objects in $\Univ$ in context $A$ are precisely given by the $\Over{\BB}{A}$-groupoids, an observation which justifies its name. Moreover, we have:
\begin{proposition}[{\cite[Proposition~4.4.11]{MWColimits}}]
	\label{prop:mappingObjectsInternalUniverse}
	For every $A\in\BB$ and every two $\Over{\BB}{A}$-groupoids $\I{G}$, $\I{H}$, viewed as objects of $\Univ$ in context $A$, there is an equivalence
	\begin{equation*}
		\map{\Univ}(\I{G},\I{H})\simeq \Hom_{\Over{\BB}{A}}(\I{G},\I{H})
	\end{equation*}
	(where $\Hom_{\Over{\BB}{A}}(-,-)$ denotes the internal  hom in $\Over{\BB}{A}$)  that is natural in $\I{G}$ and $\I{H}$. 
\end{proposition}

We refer to a full subcategory of $\Univ$ as a \emph{subuniverse}. It follows from item~(4) of Proposition~\ref{prop:classificationFullSubcategories} and the definition of $\Univ$ that every such subuniverse corresponds precisely to \emph{local class} of morphisms in $\BB$, i.e.\ a class $S$ that satisfies the condition that a morphism $p\colon P\to A$ in $\BB$ is contained in $S$ if and only if it is \emph{locally} contained in $S$, i.e.\ if and only if there is a cover $(s_i)\colon\bigsqcup_i A_i\onto A$ in $\BB$ such that the maps $s_i^\ast(p)\colon A_i\times_A P\to A_i$ are contained in $S$ (see~\cite[\S~6.1.3 and Proposition~6.2.3.14]{htt}). In other words, we have:
\begin{proposition}[{\cite[Proposition~3.9.12]{MYoneda}}]
	\label{prop:classificationSubuniverses}
	There is an equivalence between the partially ordered set of local classes in $\BB$ and $\Sub_{\mathrm{full}}(\Univ)$.
\end{proposition}
For a given local class $S$, we denote the associated subuniverse by $\Univ[S]$.

\begin{example}[{see the discussion towards the end of~\cite[\S~4.5]{MYoneda}}]
	\label{ex:smallLargeUniverse}
	Let us say that a map $p\colon P\to A$ in $\BBB$ is \emph{small} if for every map $A^\prime\to A$ in which $A^\prime \in \BB$, the pullback $A^\prime\times_A P$ is contained in $\BB$ as well. This determines a local class of morphisms in $\BBB$ and therefore by Proposition~\ref{prop:classificationSubuniverses} a subuniverse of $\Univ[\BBB]\in\Cat(\BBBB)$ which can be identified with $\Univ[\BB]\in\Cat(\BBB)\into\Cat(\BBBB)$. This exhibits $\Univ[\BB]$ as a full subcategory of $\Univ[\BBB]$.
\end{example}

It is also possible, by a similar construction, to define the (large) $\BB$-category of all (small) $\BB$-categories. Namely, it can be deduced from Remark~\ref{rem:functorialityBCategories} that the assignment $A\mapsto \Cat(\Over{\BB}{A})$ determines a functor $\BB^\op\to\CatSS$ (that carries a map $f\colon B\to A$ in $\BB$ to the pullback functor $f^\ast\colon \Cat(\Over{\BB}{A})\to\Cat(\Over{\BB}{B})$), and one can check that this functor preserves small limits~\cite[\S~A.1]{MWColimits}. 
Therefore, one obtains a (large) $\BB$-category $\ICat_{\BB}$ that is referred to as the \emph{$\BB$-category of $\BB$-categories}. Analogous to Proposition~\ref{prop:mappingObjectsInternalUniverse}, one obtains:
\begin{proposition}[{\cite[Corollary~4.5.5]{MWColimits}}]
	\label{prop:morphismsCatB}
	For every $A\in\BB$ and every two $\Over{\BB}{A}$-categories $\I{C}$ and $\I{D}$, viewed as objects of $\ICat_{\BB}$ in context $A$, there is an equivalence
	\begin{equation*}
		\map{\ICat_{\BB}}(\I{C},\I{D})\simeq \IFun[\Over{\BB}{A}](\I{G},\I{H})^\core
	\end{equation*}
	in $\Over{\BB}{A}$ that is natural in $\I{C}$ and $\I{D}$.
\end{proposition}
One moreover has the following internal enhancements of the functors appearing in Proposition~\ref{prop:GroupoidificationCore} and Remark~\ref{rem:OppositeIdentityGroupoids}:
\begin{proposition}[{\cite[Proposition~3.2.14 and Remark~A.4]{MWColimits}}]
	\label{prop:GroupoidificationCoreInternally}
	The inclusion $\Over{\BB}{A}\into\Cat(\Over{\BB}{A})$ and both its left adjoint $(-)^\gp$ and its right adjoint $(-)^\core$ are functorial in $A\in\BB$ and therefore determine functors
	\begin{align*}
		\iota&\colon \Univ\into\ICat_{\BB}\\
		(-)^\gp&\colon\ICat_{\BB}\to\Univ\\
		(-)^\core&\colon\ICat_{\BB}\to\Univ
	\end{align*}
	of which the first is fully faithful. Similarly, the equivalence $(-)^\op\colon \Cat(\Over{\BB}{A})\simeq\Cat(\Over{\BB}{A})$ is functorial in $A\in\BB$ and thus gives rise to an equivalence $(-)^\op\colon \ICat_{\BB}\simeq\ICat_{\BB}$.
\end{proposition}
In general, we will refer to a full subcategory $\I{U}\into\ICat_{\BB}$ as an \emph{internal class} of $\BB$-categories. Thus, Proposition~\ref{prop:GroupoidificationCoreInternally} implies that $\Univ$ provides an example of such an internal class. We may also regard $\ICat_{\BB}$ itself as an internal class of \emph{large} $\BB$-categories in the following way:
\begin{example}[{\cite[Remark~A.5]{MWColimits}}]
	The inclusion $\Cat(\Over{\BB}{A})\into\Cat(\Over{\BBB}{A})$ is natural in $A\in\BB$ and therefore determines a fully faithful functor $\ICat_{\BB}\into\ICat_{\BBB}$ of very large $\BB$-categories (cf.~Remark~\ref{rem:BCategoriesSheavesSize}).
\end{example}

\subsection{Left fibrations and the Grothendieck construction}
\label{sec:leftFibrations}
We call a functor $p\colon \I{P}\to\I{C}$ between $\BB$-categories a \emph{left fibration} if it is internally right orthogonal to the map $d^1\colon \Delta^0\into\Delta^1$. A functor that is internally left orthogonal to the class of left fibrations is said to be \emph{initial}. In this way, one obtains a factorisation system between initial maps and left fibrations.
\begin{proposition}[{\cite[Proposition~4.1.3]{MYoneda}}]
	\label{prop:leftFibrationsExplicitly}
	For a map $p\colon \I{P}\to\I{C}$ of $\BB$-categories, the following are equivalent:
	\begin{enumerate}
		\item $p$ is a left fibration;
		\item for every $n\geq 1$ the commutative square
		\begin{equation*}
			\begin{tikzcd}
				\I{P}_n\arrow[r, "p_n"]\arrow[d, "d_{\{0\}}"] & \I{C}_n\arrow[d, "d_{\{0\}}"]\\
				\I{P}_0\arrow[r, "p_0"] & \I{C}_0
			\end{tikzcd}
		\end{equation*}
		is a pullback.
		\item for every $A\in\BB$ the map $p(A)\colon\I{P}(A)\to\I{C}(A)$ is a left fibration of $\infty$-categories.
	\end{enumerate}
\end{proposition}
The restriction of the codomain fibration $d_0\colon\Fun(\Delta^1,\Cat(\BB))\to\Cat(\BB)$ to the full subcategory of $\Fun(\Delta^1,\Cat(\BB))$ that is spanned by the left fibrations is a cartesian fibration (as left fibrations are stable under pullback) and therefore determines via straightening a functor $\LFib\colon \Cat(\BB)^\op\to\CatSS$. By precomposing this functor with the product bifunctor $-\times -\colon \BB\times\Cat(\BB)\to\Cat(\BB)$ (where we leave the diagonal embedding $\BB\into\Cat(\BB)$ implicit), we therefore end up with a functor
\begin{equation*}
	\LFib(-\times -)\colon \Cat(\BB)^\op\to\PSh_{\CatSS}(\BB),\quad \I{C}\mapsto \ILFib_{\I{C}}=\LFib(-\times\I{C}).
\end{equation*}

\begin{theorem}[{\cite[Theorem~4.5.1]{MYoneda}}]
	\label{thm:straightening}
	For every $\BB$-category $\I{C}$, the presheaf $\ILFib_{\I{C}}$ is a sheaf and therefore defines a large $\BB$-category. Furthermore, there is an equivalence
	\begin{equation*}
		\ILFib_{\I{C}}\simeq\IFun(\I{C},\Univ)
	\end{equation*}
	of large $\BB$-categories that is natural in $\I{C}\in\Cat(\BB)$.
\end{theorem}
\begin{remark}
	Theorem~\ref{thm:straightening} is the $\BB$-categorical analogue of straightening/unstraightening for left fibrations~\cite[Theorem~2.2.1.2]{htt}.
\end{remark}

\begin{remark}
	\label{rem:LFibExplicitly}
	By means of the projection $\pr_0\colon A\times\I{C}\to A$, every functor $p\colon\I{P}\to A\times\I{C}$ can be regarded as a map in $\Over{\Cat(\BB)}{A}\simeq\Cat(\Over{\BB}{A})$ (cf.~Remark~\ref{rem:functorialityBCategories}). Now since the forgetful functor $(\pi_A)_!\colon \Over{\BB}{A}\to\BB$ creates pullbacks, it follows (using Proposition~\ref{prop:leftFibrationsExplicitly}) that $p$ is a left fibration of $\Over{\BB}{A}$-categories if and only if it is a left fibration of $\BB$-categories. Consequently, the functor $(\pi_A)_!$ induces an equivalence
	\begin{equation*}
		\LFib_{\Over{\BB}{A}}(\pi_A^\ast\I{C})\simeq\LFib_{\BB}(A\times\I{C})
	\end{equation*}
	(where the subscript indicates internal to which $\infty$-topos we are taking left fibrations). In other words, the objects of $\ILFib_{\I{C}}$ in context $A$ are precisely given by the left fibrations of $\Over{\BB}{A}$-categories over $\pi_A^\ast\I{C}$.
\end{remark}

\begin{remark}
	\label{rem:rightFibrations}
	Dually, a functor $p\colon \I{P}\to\I{C}$ of $\BB$-categories is a \emph{right fibration} if it is internally right orthogonal to $d^0\colon\Delta^0\into\Delta^1$, and a functor that is contained in the internal saturation of the latter map is said to be \emph{final}. Equivalently, $p$ is a right fibration precisely if $p^\op$ (see Definition~\ref{def:oppositeBCategory}) is a left fibration, and a functor $j$ is final if and only if $j^\op$ is initial. Again, one obtains a factorisation system between final maps and right fibrations, and by the same construction as for left fibrations (or by simply dualising this construction in the appropriate way) one ends up with a functor
	\begin{equation*}
		\RFib(-\times -)\colon \Cat(\BB)^\op\to\PSh_{\CatSS}(\BB),\quad \I{C}\mapsto \IRFib_{\I{C}}=\RFib(\I{C}\times -).
	\end{equation*}
	For every $\BB$-category $\I{C}$, we have $\IRFib_{\I{C}}\simeq \ILFib_{\I{C}^\op}$, hence $\IRFib_{\I{C}}$ defines a large $\BB$-category as well, and one furthermore obtains a natural straightening/unstraightening equivalence
	\begin{equation*}
		\IRFib_{\I{C}}\simeq\IPSh(\I{C}),
	\end{equation*}
	where $\IPSh(\I{C})=\IFun(\I{C}^\op,\Univ)$ is the large $\BB$-category of \emph{presheaves} on $\I{C}$.
\end{remark}

\subsection{Cocartesian fibrations and straightening}
\label{sec:cocartesianFibrations}
We may generalise the discussion in \S~\ref{sec:leftFibrations} to  \emph{cocartesian fibrations}:
\begin{definition}[{\cite[Proposition~3.1.4 and Definition~5.2.3]{MCocartesian}}]
	A functor $p\colon\I{P}\to\I{C}$ is a \emph{cocartesian fibration} if
	\begin{enumerate}
		\item $p(A)$ is a cocartesian fibration of $\infty$-categories for every $A\in\BB$;
		\item for every map $s\colon B\to A$ in $\BB$ the functor $s^\ast\colon \I{P}(A)\to\I{P}(B)$ carries $p(A)$-cocartesian maps to $p(B)$-cocartesian maps.
	\end{enumerate}
	A \emph{cocartesian functor} between cocartesian fibrations $p\colon \I{P}\to\I{C}$ and $q\colon \I{Q}\to\I{C}$ is a map $f\colon \I{P}\to\I{Q}$ over $\I{C}$ such that $f(A)$ carries $p(A)$-cocartesian maps to $q(A)$-cocartesian maps for every $A\in\BB$.
	The (large) $\BB$-category $\ICocart_{\I{C}}$ of cocartesian fibrations over $\I{C}$ is defined as the sheaf on $\BB$ that carries $A\in\BB$ to the (non-full) subcategory of $\Over{\Cat(\Over{\BB}{A})}{\pi_A^\ast\I{C}}$ that is spanned by the cocartesian fibrations over $\pi_A^\ast\I{C}$ and the cocartesian functors between them.
\end{definition}

\begin{remark}
	There is also a more intrinsic definition of cocartesian fibrations in $\BB$-category theory, but we shall have no use for it.
\end{remark}
\begin{remark}
	One can define \emph{cartesian fibrations} over a $\BB$-category $\I{C}$ in the evident dual way, and one similarly obtains a $\BB$-category $\ICart_{\I{C}}$ of cartesian fibrations over $\I{C}$.
\end{remark}

(Co)cartesian fibrations of $\BB$-categories model $\ICat_{\BB}$-valued (co)presheaves on $\I{C}$, in that there are the following \emph{straightening equivalences} that generalise Lurie's straightening equivalences~\cite[Theorem~3.2.0.2]{htt} to $\BB$-categories:

\begin{proposition}[{\cite[Theorem~6.3.1 and Remark~6.3.5]{MCocartesian}}]
	For every $\BB$-category $\I{C}$, there are equivalences
\begin{equation*}
	\St_{\I{C}}\colon\ICocart_{\I{C}}\simeq \IFun(\I{C}, \ICat_{\BB})
\end{equation*}
and
\begin{equation*}
	\St_{\I{C}}\colon\ICart_{\I{C}}\simeq\IFun(\I{C}^\op, \ICat_{\BB})
\end{equation*}
which are natural in $\I{C}$.
\end{proposition}

\subsection{Slice $\BB$-categories and initial objects}
\label{sec:sliceBCategories}
We now turn to the most important example of a left fibration:
\begin{definition}
	\label{def:sliceBCategories}
	For any $\BB$-category $\I{C}$ and any object $c\colon A\to \I{C}$, one defines the \emph{slice $\BB$-category} $\Under{\I{C}}{c}$ via the pullback
	\begin{equation*}
		\begin{tikzcd}
			\Under{\I{C}}{c}\arrow[d, "(\pi_c)_!"]\arrow[r] & \IFun(\Delta^1,\I{C})\arrow[d, "{(d^1,d^0)}"]\\
			A\times\I{C}\arrow[r, "{c\times\id}"] & \I{C}\times\I{C}.
		\end{tikzcd}
	\end{equation*}
\end{definition}

\begin{remark}[{\cite[Remark~4.2.2]{MYoneda}}]
	\label{rem:baseChangeSlice}
	In the situation of Definition~\ref{def:sliceBCategories}, Remark~\ref{rem:localityPrincipleObjectsMorphisms} allows us to transpose $c\colon A\to\I{C}$ to an object $\bar{c}\colon 1_{\Over{\BB}{A}}\to \pi_A^\ast\I{C}$. Thus, we can also define the slice $\Over{\BB}{A}$-category $\Under{(\pi_A^\ast\I{C})}{\bar c}$, which also comes with a projection $(\pi_{\bar c})_!\colon \Under{(\pi_A^\ast\I{C})}{\bar c}\to\pi_A^\ast\I{C}$. This turns out to produce the same result, in the sense that when applying the forgetful functor $(\pi_A)_!\colon\Cat(\Over{\BB}{A})\to\Cat(\BB)$ to the map $(\pi_{\bar c})_!\colon \Under{(\pi_A^\ast\I{C})}{\bar c}\to\pi_A^\ast\I{C}$, we recover the map $(\pi_c)_!\colon \Under{\I{C}}{c}\to A\times\I{C}$ from Definition~\ref{def:sliceBCategories}. Thus, when regarded as  a $\Over{\BB}{A}$-category, we may identify $\Under{\I{C}}{c}$ with $\Under{(\pi_A^\ast\I{C})}{\bar c}$.
\end{remark}

\begin{remark}
	\label{rem:dualSliceBCategory}
	Dually, by performing the pullback of $(d^1,d^0)$ along $\id\times c\colon \I{C}\times A\to\I{C}\times\I{C}$, one defines the slice $\BB$-category $\Over{\I{C}}{c}$ together with its projection $(\pi_c)_!\colon \Over{\I{C}}{c}\to\I{C}\times A$. Alternatively, this $\BB$-category can be defined via the identity $\Over{\I{C}}{c}\simeq(\Under{\I{C}^\op}{c})^\op$.
\end{remark}

\begin{proposition}[{\cite[Proposition~4.2.7]{MYoneda}}]
	\label{prop:sliceProjectionLeftFibration}
	For every object $c\colon A\to\I{C}$ in a $\BB$-category $\I{C}$, the functor $(\pi_c)_!\colon\Under{\I{C}}{c}\to A\times\I{C}$ is a left fibration of $\BB$-categories.
\end{proposition}

\begin{remark}
	By Remark~\ref{rem:baseChangeSlice}, the functor $(\pi_c)_!$ in Proposition~\ref{prop:sliceProjectionLeftFibration} can be regarded as a map in $\Cat(\Over{\BB}{A})$ and is as such a left fibration as well (by either applying Proposition~\ref{prop:sliceProjectionLeftFibration} to the transposed object $\bar c\colon 1_{\Over{\BB}{A}}\to\pi_A^\ast\I{C}$ or by using Remark~\ref{rem:LFibExplicitly}).
\end{remark}

\begin{definition}
	\label{def:initialObject}
	Let $\I{C}$ be a $\BB$-category. An object $c\colon A\to \I{C}$ is said to be \emph{initial} if the transpose map $1\to \pi_A^\ast\I{C}$ defines an initial functor in $\Cat(\Over{\BB}{A})$. 
\end{definition}

\begin{remark}
	In the situation of Definition~\ref{def:initialObject}, one dually says that $c$ is \emph{final} if the transpose map $1\to \pi_A^\ast \I{C}$ defines a final functor in $\Cat(\Over{\BB}{A})$.
\end{remark}

\begin{remark}[{\cite[Remark~4.3.7]{MYoneda}}]
	\label{rem:clarificationInitialObject}
	The forgetful functor $(\pi_A)_!\colon\Cat(\Over{\BB}{A})\to\Cat(\BB)$ creates initial maps for every $A\in\BB$. Therefore, if $\I{C}$ is a $\BB$-category, an object $c\colon A\to\I{C}$ is initial if and only if the map $(c,\id)\colon A\to\I{C}\times A$ is an initial functor in $\Cat(\BB)$.
\end{remark}

Observe that if $c\colon A\to\I{C}$ is an object in a $\BB$-category $\I{C}$, the identity $\id_c\colon A\to \I{C}^{\Delta^1}$ takes values in $\Under{\I{C}}{c}$. We therefore obtain a section $\id_c\colon A\to \Under{\I{C}}{c}$ of the structure map $\Under{\I{C}}{c}\to A$ (which coincides with the image of $\id_{\bar c}\colon 1_{\Over{\BB}{A}}\to \Under{(\pi_A^\ast\I{C})}{\bar c}$ along the forgetful functor $(\pi_A)_!$, see Remark~\ref{rem:baseChangeSlice}).

\begin{proposition}[{\cite[Proposition~4.3.9 and Remark~4.3.10]{MYoneda}}]
	\label{prop:initialityCanonicalSection}
	For any $\BB$-category and any object $c\colon A\to\I{C}$, the section $\id_c\colon A\to \Under{\I{C}}{c}$ is initial as a map in $\Cat(\Over{\BB}{A})$ and therefore defines an initial object of $\Under{\I{C}}{c}$.
\end{proposition}

\begin{corollary}[{\cite[Corollary~4.3.19]{MYoneda}}]
	\label{cor:factorisationInternalObject}
	Let $\I{C}$ be a $\BB$-category and let $c\colon A\to\I{C}$ be an object in $\I{C}$. The factorisation of $c$ into an initial map and a left fibration is given by the composition $\pr_1(\pi_c)_!\id_c\colon A\to \Under{(\I{C})}{c}\to \I{C}$ where $\pr_1\colon A\times\I{C}\to\I{C}$ is the projection.
\end{corollary}

\begin{proposition}[{\cite[Proposition~4.3.20]{MYoneda}}]
	\label{prop:characterisationinitialObject}
	Let $\I{C}$ be a $\BB$-category. For any object $c\colon A\to \I{C}$, the following are equivalent:
	\begin{enumerate}
		\item  $c$ is an initial object;
		\item the projection $(\pi_c)_!\colon\Under{\I{C}}{c}\to A\times \I{C}$ is an equivalence;
		\item for any object $d\colon B\to \I{C}$ the map $\map{\I{C}}(\pr_0^\ast c,\pr_1^\ast d)\to A\times B$ is an equivalence in $\BB$.
	\end{enumerate}
\end{proposition}

\begin{corollary}[{\cite[Corollary~4.3.21]{MYoneda}}]
	\label{cor:universalPropertyInitialObject}
	Let $\I{C}$ be a $\BB$-category and let $c$ and $d$ be objects in $\I{C}$ in context $A\in\BB$ such that $c$ is initial. Then there is a unique map $c\to d$ in $\I{C}$ in context $A$ that is an equivalence if and only if $d$ is initial as well.
\end{corollary}

\subsection{Yoneda's lemma}
\label{sec:YonedaLemma}
The theory of left fibrations can be used to derive a version of Yoneda's lemma for $\BB$-categories. First, we need a functorial version of the mapping $\BB$-groupoid construction. To that end, let us denote by $-\star -\colon\Delta\times\Delta\to\Delta$ the ordinal sum bifunctor. We may now define:

\begin{definition}[{\cite[Definition~4.2.4]{MYoneda}}]
	\label{def:twistedArrow}
	Let $\epsilon\colon \Delta\to\Delta$ denote the functor $\ord{n}\mapsto \ord{n}^{\op}\star \ord{n}$. For any $\BB$-category $\I{C}$, we define the \emph{twisted arrow $\BB$-category} $\ITw(\I{C})$ to be the simplicial object given by the composition
	\begin{equation*}
		\Delta^{\op}\xrightarrow{\epsilon^{\op}} \Delta^{\op}\xrightarrow{\I{C}} \BB.
	\end{equation*}
	This defines a functor $\ITw\colon \Cat(\BB)\to\Simp\BB$.
\end{definition}

Note that the functor $\epsilon$ in Definition~\ref{def:twistedArrow} comes along with two canonical natural transformations
\begin{equation*}
	(-)^{\op}\to \epsilon \leftarrow \id_{\Delta}
\end{equation*}
which induces a map of simplicial objects
\begin{equation*}
	\ITw(\I{C})\to \I{C}^{\op}\times\I{C}
\end{equation*}
that is natural in $\I{C}$.
\begin{proposition}[{\cite[Proposition~4.2.5]{MYoneda}}]
	For every $\BB$-category $\I{C}$, the simplicial object $\ITw(\I{C})$ is a $\BB$-category, and the map $\ITw(\I{C})\to\I{C}^\op\times\I{C}$ is a left fibration.
\end{proposition}

By applying the straightening/unstraightening equivalence from Theorem~\ref{thm:straightening} to the left fibration $\ITw(\I{C})\to\I{C}^\op\times\I{C}$, one now ends up with a bifunctor
\begin{equation*}
	\map{\I{C}}\colon\I{C}^\op\times\I{C}\to\Univ
\end{equation*}
that sends a pair of objects $(c,d)\colon A\to\I{C}^{\op}\times\I{C}$ to the object $\map{\I{C}}(c,d)\in\Over{\BB}{A}$ from Definition~\ref{def:nMorphisms}.
Upon transposing this bifunctor across the adjunction $\I{C}^\op\times -\dashv \IFun(\I{C}^\op,-)$, one obtains the \emph{Yoneda embedding}
\begin{equation*}
	h_{\I{C}}\colon\I{C}\to\IPSh(\I{C}).
\end{equation*}

\begin{theorem}[{\cite[Theorem~4.7.8]{MYoneda}}]
	\label{thm:YonedaLemma}
	For any $\BB$-category $\I{C}$, there is a commutative diagram
	\begin{equation*}
		\begin{tikzcd}
			\I{C}^{\op}\times \IPSh(\I{C})\arrow[dr, "\ev"'] 
			\arrow[r, "h\times \id"] & {\IPSh(\I{C})^{\op}}\times\IPSh(\I{C})\arrow[d, "{\map{\IPSh(\I{C})}(-,-)}"]  \\
			& \Univ
		\end{tikzcd}
	\end{equation*}
	in $\Cat(\BBB)$ (where $\ev$ is the evaluation map).
\end{theorem}

\begin{corollary}[{\cite[Corollary~4.7.16]{MYoneda}}]
	\label{cor:YonedaEmbedding}
	For every $\BB$-category $\I{C}$, the Yoneda embedding $h_{\I{C}}$ is fully faithful.
\end{corollary}

\begin{remark}[{\cite[Proposition~4.7.20]{MYoneda}}]
	\label{rem:representableFunctorsFibrationalCriterion}
	Explicitly, an object $A\to \IPSh(\I{C})$ is contained in $\I{C}$ if and only if the associated right fibration $p\colon \I{P}\to \I{C}\times A$ admits a final section $A\to \I{P}$ over $A$ (i.e.\ if $\I{P}$ has a final object in global context when viewed as a $\Over{\BB}{A}$-category). If this is the case, one obtains an equivalence $\Over{\I{C}}{c}\simeq \I{P}$ over $\I{C}\times A$ where $c$ is the image of the final section $A\to \I{P}$ along the functor $\I{P}\to \I{C}$.
\end{remark}

\subsection{Adjunctions}
\label{sec:adjunctions}
Recall that the bifunctor $\Fun_{\BB}(-,-)$ exhibits $\Cat(\BB)$ as an $\infty$-category enriched in $\CatS$ and therefore in particular as an $(\infty,2)$-category. One can therefore make sense of the usual $2$-categorical definition of an adjunction in $\Cat(\BB)$. Explicitly, this can be spelled out as follows:
\begin{definition}[{\cite[Definition~3.1.1]{MWColimits}}]
	\label{def:internalAdjunction}
	Let $\I{C}$ and $ \I{D}$ be $\BB$-categories. An \emph{ adjunction} between $\I{C}$ and $\I{D}$ is a tuple $(l,r,\eta,\epsilon)$, where $l\colon \I{C}\to\I{D}$ and $r\colon\I{D}\to\I{C}$ are functors and where $\eta\colon \id_{ \I{D}}\to rl$ and $\epsilon \colon lr\to \id_{ \I{C}}$ are maps such that there are commutative triangles
	\begin{equation*}
		\begin{tikzcd}
			l\arrow[r, "l\eta"]\arrow[dr, "\id"'] & lrl \arrow[d, "\epsilon l"] & & rlr \arrow[from=r, "\eta r"']\arrow[d, "r\epsilon"'] & r \arrow[dl, "\id"]\\
			& l & & r
		\end{tikzcd}
	\end{equation*}
	in $\Fun_{\BB}(\I{C}, \I{D})$ and in $\Fun_{\BB}(\I{D}, \I{C})$, respectively. We denote such an adjunction by $l\dashv r$, and we refer to $\eta$ as the \emph{unit} and to $\epsilon$ as the \emph{counit} of the adjunction. We say that a pair $(l,r)\colon \I{C}\leftrightarrows \I{D}$ \emph{defines an adjunction} if there exist transformations $\eta$ and $\epsilon$ as above such that the tuple $(l,r,\eta,\epsilon)$ is an adjunction.
\end{definition}
Note that in order to make sense of Definition~\ref{def:internalAdjunction}, one does not actually require the full $(\infty,2)$-categorical structure on $\Cat(\BB)$. In fact, the only additional structure one needs to give a meaning to the maps appearing in the two commutative triangles is the bifunctor $\Fun_{\BB}(-,-)$.

Analogous to the situation in (higher) category theory, one can alternatively define the notion of an adjunction between $\BB$-categories via the datum of a bifunctorial equivalence on mapping $\BB$-groupoids:
\begin{proposition}[{\cite[Proposition~3.3.4 and Corollary~3.3.5]{MWColimits}}]
	\label{prop:characterisationAdjunctionMappingGroupoids}
	A pair of functors $(l, r)\colon\I{C}\leftrightarrows\I{D}$ between $\BB$-categories defines an adjunction if and only if there is an equivalence of functors
	\begin{equation*}
		\alpha\colon\map{\I{D}}(l(-),-)\simeq\map{\I{C}}(-,r(-)).
	\end{equation*}
	In particular, the functor $r\colon\I{D}\to\I{C}$ admits a left adjoint if and only if the copresheaf $\map{\I{C}}(c,r(-))$ is corepresentable for every object $c\colon A\to\I{D}$, and the dual statement for the existence of a \emph{right} adjoint holds as well.
\end{proposition}

\begin{remark}
	In view of Remark~\ref{rem:representableFunctorsFibrationalCriterion}, Proposition~\ref{prop:characterisationAdjunctionMappingGroupoids} implies that a functor $r\colon \I{D}\to\I{C}$ admits a left adjoint if and only if for every $c\colon A\to\I{C}$ the right fibration $\Under{\I{D}}{c}=\Under{\I{C}}{c}\times_{\I{C}}\I{D}\to A\times\I{D}$ admits an initial section $A\to \Under{\I{D}}{c}$ over $A$.
\end{remark}

Lastly, we also have a convenient and very explicit sheaf-theoretic criterion for the existence of a left or right adjoint:
\begin{proposition}[{\cite[Proposition~3.2.9]{MWColimits}}]
	\label{prop:existenceAdjointsBeckChevalley}
	A functor $r\colon\I{C}\to\I{D}$ in $\Cat(\BB)$ is a right adjoint if and only if the following two conditions hold:
	\begin{enumerate}
		\item For every object $A\in \BB$, the induced functor $r(A)\colon \I{C}(A)\to \I{D}(A)$ is the right adjoint in an adjunction $(l_A,r(A),\eta_A,\epsilon_A)$.
		\item For every morphism $s\colon B\to A$ in $\BB$, the mate of the commutative square
		\begin{equation*}
			\begin{tikzcd}
				\I{C}(A)\arrow[r, "r(A)"] \arrow[d, "s^\ast"'] & \I{D}(A) \arrow[d, "s^\ast"] \arrow[dl, Rightarrow, "\simeq"', shorten=2mm]\arrow[d]\\
				\I{C}(B)\arrow[r, "r(B)"]  & \I{D}(B)
			\end{tikzcd}
		\end{equation*}
		commutes.
	\end{enumerate}
	If this is the case, then the left adjoint $l$ of $r$ is given on objects $A\in \BB$ by $l_A$ and on morphisms $s\colon B\to A$ by the mate of the commutative square defined by $r(s)$. The dual statement for the existence of a \emph{right} adjoint holds as well.
\end{proposition}

\subsection{Limits and colimits}
\label{sec:limitsColimits}
Using the notion of an adjunction between $\BB$-categories, it is now straightforward to define internal limits and colimits:

\begin{definition}[{\cite[Propositions~4.1.12 and~4.2.4]{MWColimits}}]
	\label{def:limitsColimits}
	If $\I{I}$ and $\I{C}$ are $\BB$-categories, the \emph{colimit} functor $\colim_{\I{I}}$ is defined to be the left adjoint of the diagonal map $\diag\colon\I{C}\to\IFun(\I{I},\I{C})$, provided that it exists. Dually, the \emph{limit} functor $\lim_{\I{I}}$ is defined as the right adjoint of the diagonal map. A functor $f\colon\I{C}\to\I{D}$ \emph{preserves} $\I{I}$-indexed colimits if the canonical map $\colim_{\I{I}} f_\ast\to f\colim_{\I{I}}$ (which is given by the mate transformation of the equivalence $f_\ast \diag\simeq\diag f$) is an equivalence. Dually, $f$ preserves $\I{I}$-indexed limits if the map $f\lim_{\I{I}}\to\lim_{\I{I}}f_\ast$ is an equivalence.
\end{definition}
\begin{remark}
	 In the context of Definition~\ref{def:limitsColimits}, even if a colimit or limit functor does not exist, one can define the colimit $\colim d$ of a diagram $d\colon A\to\IFun(\I{I},\I{C})$ as a corepresenting object of the copresheaf $\map{\IFun(\I{I},\I{C})}(d,\diag(-))$, and the limit $\lim d$ of the diagram $d$ as a representing object of the presheaf $\map{\IFun(\I{I},\I{C})}(\diag(-), d)$. Using Remark~\ref{rem:representableFunctorsFibrationalCriterion}, this precisely means that $d$ admits a colimit if and only if the left fibration $\Under{\I{C}}{d}=\Under{\IFun(\I{I}, \I{C})}{d}\times_{\IFun(\I{I},\I{C})}\I{C}\to A\times\I{C}$ admits an initial section $A\to \Under{\I{C}}{d}$ over $A$. We refer to $\Under{\I{C}}{d}$ as the $\Over{\BB}{A}$-category of \emph{cocones} over the diagram $d$. The initial section $A\to \Under{\I{C}}{d}$ then picks out the \emph{colimit cocone} under $d$. 
\end{remark}
By making use of Proposition~\ref{prop:existenceAdjointsBeckChevalley}, one now arrives at the following two fundamental classes of examples for internal (co)limits:

\begin{example}[{\cite[Example~4.1.13]{MWColimits}}]
	\label{ex:groupoidalLimitsColimits}
	Let $\I{C}$ be a $\BB$-category and $\I{G}$ be a $\BB$-groupoid. Then the following two conditions are equivalent:
	\begin{enumerate}
		\item $\I{C}$ admits $\I{G}$-indexed colimits;
		\item for every $A\in\BB$ the functor $\pi_{\I{G}}^\ast\colon \I{C}(A)\to\I{C}(\I{G}\times A)$ admits a left adjoint $(\pi_{\I{G}})_!$ such that for every map $s\colon B\to A$ in $\BB$ the natural morphism $(\pi_{\I{G}})_!s^\ast\to s^\ast (\pi_{\I{G}})_!$ is an equivalence.
	\end{enumerate}
	Similarly, if $f\colon \I{C}\to\I{D}$ is a functor of $\BB$-categories that admit $\I{G}$-indexed colimits, then the following are equivalent:
	\begin{enumerate}
		\item $f$ preserves $\I{G}$-indexed colimits;
		\item for every $A\in \BB$ the natural map $(\pi_G)_! f(\I{G}\times A)\to f(A)(\pi_G)_!$ is an equivalence.
	\end{enumerate}
	The dual statements for $\I{G}$-indexed limits hold as well.
\end{example}

\begin{example}[{\cite[Example~4.1.14]{MWColimits}}]
	\label{ex:externalLimitsColimits}
	Let $\I{C}$ be a $\BB$-category and let $\II$ be an $\infty$-category. Then the following two conditions are equivalent:
	\begin{enumerate}
		\item $\I{C}$ admits $\II$-indexed colimits;
		\item for every $A\in\BB$ the $\infty$-category $\I{C}(A)$ admits $\II$-indexed colimits, and for every map $s\colon B\to A$ in $\BB$ the functor $s^\ast\colon \I{C}(A)\to\I{C}(B)$ preserves such colimits. 
	\end{enumerate}
	Similarly, if $f\colon \I{C}\to\I{D}$ is a functor of $\BB$-categories that admit $\II$-indexed colimits, then the following are equivalent:
	\begin{enumerate}
	\item $f$ preserves $\II$-indexed colimits;
	\item for every $A\in \BB$ the functor $f(A)\colon \I{C}(A)\to\I{D}(A)$ preserves $\II$-indexed colimits.
	\end{enumerate}
	The dual statements for $\II$-indexed limits hold as well.
\end{example}

\subsection{Cocompleteness}
\label{sec:Cocompleteness}
Using the notion of internal (co)limits, we may now define what it means for a (large) $\BB$-category to be \emph{$\I{U}$-(co)complete} and a functor of (large) $\BB$-categories to be \emph{$\I{U}$-(co)continuous} with respect to an arbitrary internal class (i.e.\ full subcategory) $\I{U}\into\ICat_{\BB}$:

\begin{definition}[{\cite[Definition~5.2.1]{MWColimits}}]
	\label{def:cocomplete}
	Let $\I{U}$ be an internal class of $\BB$-categories. A $\BB$-category $\I{C}$ is said to be \emph{$\I{U}$-cocomplete} if $\pi_A^\ast\I{C}$ admits $\I{I}$-indexed colimits for every object $\I{I}\in\I{U}(A)$ and every $A\in\BB$. Similarly, if $f\colon \I{C}\to\I{D}$ is a functor between $\BB$-categories that are both $\I{U}$-cocomplete, we say that $f$ is \emph{$\I{U}$-cocontinuous} if $\pi_A^\ast f$ preserves $\I{I}$-indexed colimits for any $A\in\BB$ and any $\I{I}\in\I{U}(A)$. We simply say that a (large) $\BB$-category $\I{C}$ is \emph{cocomplete} if it is $\ICat_{\BB}$-cocomplete (when viewing $\ICat_{\BB}$ as an internal class of large $\BB$-categories), and we call a functor between cocomplete (large) $\BB$-categories \emph{cocontinuous} if it is $\ICat_{\BB}$-cocontinuous.
	
	Dually, we say that a $\BB$-category $\I{C}$ is \emph{$\I{U}$-complete} if $\pi_A^\ast\I{C}$ admits $\I{I}$-indexed limits for every object $\I{I}\in\I{U}(A)$ and every $A\in\BB$. If $f\colon \I{C}\to\I{D}$ is a functor between $\BB$-categories that are both $\I{U}$-complete, we say that $f$ is \emph{$\I{U}$-continuous} if $\pi_A^\ast f$ preserves $\I{I}$-indexed limits for any $A\in\BB$ and any $\I{I}\in\I{U}(A)$. We simply say that a (large) $\BB$-category $\I{C}$ is \emph{complete} if it is $\ICat_{\BB}$-complete, and we call a functor between complete (large) $\BB$-categories \emph{continuous} if it is $\ICat_{\BB}$-continuous.
\end{definition}
The notions of $\I{U}$-(co)completeness and $\I{U}$-(co)continuity have the expected stability properties:
\begin{proposition}[{\cite[Proposition~5.2.5, 5.2.6 and 5.2.7]{MWColimits}}]
	\label{prop:UCocontinuityProperties}
	Let $\I{U}$ be an arbitrary internal class of $\BB$-categories. Then:
	\begin{enumerate}
		\item for every $\BB$-category $\I{C}$, the functor $\IFun(\I{C},-)$ preserves $\I{U}$-(co)completeness and $\I{U}$-(co)continuity;
		\item For every $\I{U}$-(co)complete $\BB$-category $\I{D}$, the functor $\IFun(-,\I{D})$ carries every map in $\Cat(\BB)$ to a $\I{U}$-(co)continuous functor.
		\item every right (left) adjoint functor between $\I{U}$-(co)complete $\BB$-categories is $\I{U}$-(co)continuous;
		\item every reflective subcategory of a $\I{U}$-(co)complete $\BB$-category (i.e.\ full subcategory for which the inclusion admits a left adjoint) is $\I{U}$-(co)complete as well.
	\end{enumerate}
\end{proposition}

\begin{example}[{\cite[Example~5.2.8]{MWColimits}}]
	\label{ex:UnivCocomplete}
	The universe $\Univ$ is a complete and cocomplete $\BB$-category, and therefore so is $\IPSh(\I{C})$ for every $\I{C}\in\Cat(\BB)$. Since $\ICat_{\BB}$ arises as a reflective subcategory of $\IPSh(\Delta)$, this implies that $\ICat_{\BB}$ is complete and cocomplete as well.
\end{example}

\begin{example}[{\cite[Proposition~5.4.2]{MWColimits}}]
	\label{ex:UcolimitsGroupoids}
	Let $S$ be a local class of maps in $\BB$ and let $\Univ[S]$ be the associated subuniverse, where we view $\Univ[S]$ as an internal class of $\BB$-categories. Then a $\BB$-category $\I{C}$ is $\Univ[S]$-cocomplete if and only if the following two conditions are satisfied:
	\begin{enumerate}
		\item for every map $p\colon P\to A$ in $S$, the functor $p^\ast\colon \I{C}(A)\to\I{C}(P)$ admits a left adjoint $p_!$;
		\item for every pullback square
		\begin{equation*}
			\begin{tikzcd}
				Q\arrow[r, "t"]\arrow[d, "q"] & P\arrow[d, "p"]\\
				B\arrow[r, "s"] & A
			\end{tikzcd}
		\end{equation*}
		in $\BB$ in which $p$ and $q$ are contained in $S$, the natural map $q_!t^\ast\to s^\ast p_!$ is an equivalence.
	\end{enumerate}
	Furthermore, a functor $f\colon\I{C}\to\I{D}$ between  $\Univ[S]$-cocomplete $\BB$-categories is $\Univ[S]$-cocontinuous precisely if for every map $p\colon P\to A$ in $S$ the natural map $p_! f(P)\to f(A)p_!$ is an equivalence.
\end{example}

\begin{example}[{\cite[Proposition~5.4.5]{MWColimits}}]
	\label{prop:LConstCocomplete}
	If $\KK\subset\CatS$ is a class of $\infty$-categories, we denote by $\ILConst_{\KK}\into\ICat_{\BB}$ the internal class of \emph{locally $\KK$-constant $\BB$-categories}, which is defined by the sheaf that sends $A\in\BB$ to the full subcategory of $\Cat(\Over{\BB}{A})$ that is spanned by the {locally $\KK$-constant $\Over{\BB}{A}$-categories}. Here a $\Over{\BB}{A}$-category $\I{C}$ is said to be locally $\KK$-constant if there is a cover $(s_i)\colon\bigsqcup_i A_i\onto A$ in $\BB$ such that each $s_i^\ast\I{C}$ is equivalent to $\const_{\BB_{/A_i}}(K_i)$ for some $K_i\in\KK$. We will write $\ILConst$ for the internal class of \emph{all} locally constant $\BB$-categories (i.e.\ for the case where $\KK=\CatS$). Now a $\BB$-category $\I{C}$ is $\ILConst_{\KK}$-cocomplete precisely if the sheaf $\I{C}(-)$ takes values in the $\infty$-category $\CatS^{\cocont{\KK}}$ of $\KK$-cocomplete $\infty$-categories and $\KK$-cocontinuous functors. Furthermore, a functor $f\colon\I{C}\to\I{D}$ between $\ILConst_\KK$-cocomplete $\BB$-categories is $\ILConst_\KK$-cocontinuous if and only if for all $A\in\BB$ the functor $f(A)$ is $\KK$-cocontinuous.
\end{example}

The notions of cocompleteness and cocontinuity (i.e\ the case $\I{U}=\ICat_{\BB}$) will be of particular interest to us, so it is important that we have good control over them. Luckily, we have a way to decompose arbitrary internal colimits into \emph{$\BB$-groupoidal} and \emph{$\infty$-categorical} ones:
\begin{proposition}[{\cite[Proposition~5.4.1]{MWColimits}}]
	\label{prop:CocompleteGroupoidsExternal}
	A (large) $\BB$-category $\I{C}$ is cocomplete if and only if it is both $\Univ$- and $\ILConst$-cocomplete, and a functor between cocomplete large $\BB$-categories is cocontinuous if and only if it is both $\Univ$- and $\ILConst$-cocontinuous.
\end{proposition}
As a consequence, we may deduce from the two examples above:
\begin{corollary}[{\cite[Corollary~5.4.7]{MWColimits}}]
	\label{cor:cocompletenessExplicitly}
	A (large) $\BB$-category $\I{C}$ is cocomplete if and only if
	\begin{enumerate}
		\item $\I{C}(-)$ takes values in the $\infty$-category $\CatS^\cc$ of cocomplete $\infty$-categories and cocontinuous functors;
		\item for every $s\colon B\to A$ in $\BB$ the transition functor $s^\ast\colon \I{C}(A)\to\I{C}(B)$ has a left adjoint;
		\item $\I{C}(-)$ sends every pullback square in $\BB$ to a left adjointable square of $\infty$-categories.
	\end{enumerate}
	Moreover, a functor $f\colon \I{C}\to\I{D}$ of cocomplete $\BB$-categories is cocontinuous if and only if for each $A\in\BB$ the functor $f(A)$ preserves colimits, and for every map $p\colon B\to A$ in $\BB$ the natural map $p_! f(P)\to f(A) p_!$ is an equivalence.
\end{corollary}

Lastly, we note that the collection of $\I{U}$-cocomplete $\BB$-categories and $\I{U}$-cocontinuous functors can be assembled into a $\BB$-category as well:
\begin{definition}[{\cite[Definitions~5.3.1 and~5.3.3 and Remark~5.3.2]{MWColimits}}]
	\label{def:CatU}
	For every internal class $\I{U}$ of $\BB$-categories, the (large) $\BB$-category $\ICat_{\BB}^{\cocont{\I{U}}}$ is defined as the sheaf of $\infty$-categories on $\BB$ that carries $A\in\BB$ to the (non-full) subcategory of $\Cat(\Over{\BB}{A})$ that is spanned by the $\pi_A^\ast\I{U}$-cocomplete $\Over{\BB}{A}$-categories and the $\pi_A^\ast\I{U}$-cocontinuous functors between them. If $\I{C}$ and $\I{D}$ are $\I{U}$-cocomplete $\BB$-categories, we denote by $\IFun^{\cocont{\I{U}}}(\I{C},\I{D})$ the full subcategory of $\IFun(\I{C},\I{D})$ that is spanned in context $A\in\BB$ by the $\pi_A^\ast\I{U}$-cocontinuous functors $\pi_A^\ast\I{C}\to\pi_A^\ast\I{D}$.
\end{definition}

\subsection{Kan extensions}
\label{sec:Kan extensions}
(Ordinary) Kan extensions of functors between $\BB$-categories can be defined via their usual universal property:
\begin{definition}
	\label{def:KanExtension}
	A \emph{left Kan extension} of a functor $F\colon \I{C}\to\I{E}$ along $f \colon \I{C} \to \I{D}$ is a functor $f_!F\colon \I{D}\to\I{E}$ together with an equivalence 
	\begin{equation*}
		\map{\IFun(\I{D},\I{E})}(f_! F,-)\simeq \map{\IFun(\I{C},\I{E})}(F, f^\ast(-)).
	\end{equation*}
	The \emph{functor of left Kan extension} $f_!$ is the left adjoint of $f^\ast\colon \IFun(\I{D},\I{E})\to\IFun(\I{C},\I{E})$, provided that it exists.
	
	Dually, a \emph{right Kan extension} of $F$ along $f$ is a functor $f_\ast F\colon \I{D}\to\I{E}$ together with an equivalence
	\begin{equation*}
		\map{\IFun(\I{D},\I{E})}(-,f_\ast F)\simeq \map{\IFun(\I{C},\I{E})}(f^\ast(-), F).
	\end{equation*}
	The \emph{functor of right Kan extension} $f_\ast$ is the right adjoint of $f^\ast\colon \IFun(\I{D},\I{E})\to\IFun(\I{C},\I{E})$, provided that it exists.
\end{definition}

The main existence theorem of left Kan extensions in $\BB$-category theory can be phrased as follows:
\begin{proposition}[{\cite[Theorem~6.3.5]{MWColimits}}]
	\label{prop:existenceKanExtension}
	Let $\I{U}$ be an internal class of $\BB$-categories such that for every object $d\colon A\to\I{D}$ in context $A\in\BB$ the $\Over{\BB}{A}$-category $\Over{\I{C}}{d}$ is contained in $\I{U}(A)$.
	Then, whenever $\I{E}$ is $\I{U}$-cocomplete, the functor of left Kan extension $f_!$ exists. Moreover, $f_!$ is fully faithful whenever $f$ is fully faithful.
\end{proposition}

A large $\BB$-category $\I{D}$ is said to be \emph{locally small} if the mapping $\BB$-groupoid functor $\map{\I{D}}(-,-)$ takes values in $\Univ[\BB]\into\Univ[\BBB]$ (see~\cite[\S~4.7]{MYoneda}). Proposition~\ref{prop:existenceKanExtension} now implies:
\begin{corollary}[{\cite[Corollary~6.3.7]{MWColimits}}]
	\label{cor:leftKanExtensionsSmall}
	Let $f\colon \I{C}\to\I{D}$ be a functor of $\BB$-categories such that $\I{C}$ is small and $\I{D}$ is locally small (but not necessarily small). 
	If $\I{E}$ is a cocomplete large $\BB$-category, the functor of left Kan extension $f_!$ always exists.
\end{corollary}
\begin{example}[{\cite[Proposition~3.3.1]{MWColimits}}]
	\label{ex:LKEPresheaves}
	Let $f\colon \I{C}\to\I{D}$ be a functor of $\BB$-categories. Then the functor of left Kan extension $f_!\colon\IPSh(\I{C})\to\IPSh(\I{D})$ can be explicitly described as follows: if $F\colon\I{C}^\op\to\Univ$ is a presheaf and $\Over{\I{C}}{F}\to\I{C}$ is the associated right fibration, the right fibration $\Over{\I{D}}{f_!(F)}\to\I{D}$ (i.e.\ the one classified by $f_!(F)\colon\I{D}^\op\to\Univ$) is the unique right fibration over $\I{D}$ that fits into a commutative square
	\begin{equation*}
		\begin{tikzcd}
			\Over{\I{C}}{F}\arrow[d] \arrow[r, "j"] &\Over{\I{D}}{f_!(F)}\arrow[d]\\
			\I{C}\arrow[r, "f"] & \I{D}
		\end{tikzcd}
	\end{equation*}
	in which $j$ is a final functor.
\end{example}


\subsection{Free $\I{U}$-cocompletion and the universal property of presheaves}
The Yoneda embedding $h_{\I{C}}\colon\I{C}\into\IPSh(\I{C})$ exhibits the large $\BB$-category $\IPSh(\I{C})$ as the \emph{universal} cocomplete $\BB$-category that admits a map from $\I{C}$:
\begin{proposition}[{\cite[Theorem~7.1.1]{MWColimits}}]
	\label{prop:universalPropertyPSh}
	For every $\BB$-category $\I{C}$ and every cocomplete large $\BB$-category $\I{E}$, the functor of left Kan extension $(h_{\I{C}})_!$ along the Yoneda embedding $h_{\I{C}}\colon \I{C}\into\IPSh(\I{C})$ induces an equivalence
	\begin{equation*}
		(h_{\I{C}})_!\colon\IFun(\I{C},\I{E})\simeq \IFun^\cc(\IPSh(\I{C}),\I{E}).
	\end{equation*}
\end{proposition}
One can leverage Proposition~\ref{prop:universalPropertyPSh} to construct the \emph{free $\I{U}$-cocompletion} $\IPSh^{\I{U}}(\I{C})$ of any $\BB$-category $\I{C}$:
\begin{definition}[{\cite[Definition~7.1.6]{MWColimits}}]
	\label{def:freeCocompletion}
	For every $\BB$-category $\I{C}$ and every internal class $\I{U}$ of $\BB$-categories, we define the \emph{free $\I{U}$-cocompletion} $\IPSh^{\I{U}}(\I{C})$ of $\I{C}$ as the smallest full subcategory of $\IPSh(\I{C})$ that contains $\I{C}$ and is closed under $\I{U}$-colimits in $\IPSh(\I{C})$ (i.e.\ that is $\I{U}$-cocomplete and for which the inclusion is $\I{U}$-cocontinuous). 
\end{definition}
The terminology in Definition~\ref{def:freeCocompletion} is justified by the following result:
\begin{proposition}[{\cite[Theorem~7.1.13]{MWColimits}}]
	\label{prop:freeCocompletion}
	Let $\I{C}$ be a $\BB$-category, let $\I{U}$ be an internal class of $\BB$-categories and let $\I{E}$ be a $\I{U}$-cocomplete (large) $\BB$-category. Then the functor of left Kan extension along $h_{\I{C}}\colon \I{C}\into\IPSh^{\I{U}}(\I{C})$ exists and determines an equivalence
	\begin{equation*}
		(h_{\I{C}})_!\colon \IFun(\I{C},\I{E})\simeq \IFun^{\cocont{\I{U}}}(\IPSh^{\I{U}}(\I{C}),\I{E}) .
	\end{equation*}
\end{proposition}

\subsection{Localisations}
\label{sec:localisations}
Lastly, we mention that there is also a $\BB$-categorical analogue of the theory of \emph{(Dwyer-Kan) localisations}:
\begin{definition}[{\cite[Definition~C.3]{MWColimits}}]
	\label{def:localisationAtSomeS}
	Let $\I{C}$ be a $\BB$-category and let $\I{S}\to \I{C}$ be a functor. The \emph{localisation} of $\I{C}$ at $\I{S}$ is the $\BB$-category $\I{S}^{-1}\I{C}$ defined by the pushout
	\begin{equation*}
		\begin{tikzcd}
			\I{S}\arrow[d]\arrow[r] \arrow[dr, phantom, "\lrcorner", very near end]& \I{S}^{\gp}\arrow[d]\\
			\I{C}\arrow[r, "L"] & \I{S}^{-1}\I{C}
		\end{tikzcd}
	\end{equation*}
	in $\Cat(\BB)$.
	We refer to the map $L\colon\I{C}\to \I{S}^{-1}\I{C}$ as the \emph{localisation functor} that is associated with the map $\I{S}\to\I{C}$. We say that this is a \emph{Bousfield} localisation if $L$ admits a right adjoint.
\end{definition}

\begin{remark}[{\cite[Proposition~3.4.6 and Remark~3.4.8]{MWColimits}}]
	If $L\colon \I{C}\to\I{S}^{-1}\I{C}$ is a Bousfield localisation, then the right adjoint of $L$ is automatically fully faithful and therefore exhibits $\I{S}^{-1}\I{C}$ as a reflective subcategory of $\I{C}$. Conversely, any reflective subcategory $\I{D}\into\I{C}$ (with left adjoint $L\colon\I{C}\to\I{D}$) can be identified with the localisation of $\I{C}$ at $L^{-1}(\I{D}^\core)\into\I{C}$.
\end{remark}

Given a functor $\I{S}\to\I{C}$ as well as another $\BB$-category $\I{D}$, we denote by $\IFun(\I{C},\I{D})_{\I{S}}$ the full subcategory of $\IFun(\I{C},\I{D})$ that is spanned by the objects $A\to\IFun(\I{C},\I{D})$ in arbitrary context $A\in\BB$ that encode functors $\pi_A^\ast\I{C}\to\pi_A^\ast\I{D}$ for which the composition $\pi_A^\ast\I{S}\to\pi_A^\ast\I{C}\to\pi_A^\ast\I{D}$ factors through the inclusion $\pi_A^\ast\I{D}^\core\into\pi_A^\ast\I{D}$.
\begin{proposition}[{\cite[Proposition~C.13]{MWColimits}}]
	\label{prop:universalPropertyLocalisation}
	Let $\I{C}$ be a $\BB$-category and let $\I{S}\to\I{C}$ be a functor. Then precomposition with the localisation functor $L\colon\I{C}\to \I{S}^{-1}\I{C}$ induces an equivalence
	\begin{equation*}
		L^\ast\colon\IFun(\I{S}^{-1}\I{C},\I{D})\simeq \IFun(\I{C},\I{D})_{\I{S}}
	\end{equation*}
	for any $\BB$-category $\I{D}$.
\end{proposition}

\chapter{Accessible and presentable $\BB$-categories}
\label{chap:AccessiblePresentable}

Our theory of presentable $\BB$-categories relies on the interplay between internal classes $\I{U}$ of $\BB$-categories and their associated internal classes $\IFilt_{\I{U}}$ of \emph{$\I{U}$-filtered} $\BB$-categories. We study the concept of $\I{U}$-filteredness in \S~\ref{chap:filtered}. 
In particular, we study certain conditions on the internal class $\I{U}$ which guarantee that every $\BB$-category can be decomposed into a $\I{U}$-filtered colimit of objects in $\I{U}$. 
This is a technical condition which is crucial for the development of accessibility in the world of $\BB$-categories. 
In~\S~\ref{chap:cardinals}, we discuss how one can construct an ample amount of internal classes $\I{U}$ which satisfy these conditions. 
Building upon these rather technical preparations, we then define the concept of $\I{U}$-accessibility in~\S~\ref{chap:accessible} and prove a few basic results that we will need for our discussion of presentable $\BB$-categories. For example, we give a characterisation of $\I{U}$-accessible $\BB$-categories by making use of the notion of \emph{$\I{U}$-compactness}. 
In \S~\ref{chap:presentable}, we introduce and study presentable $\BB$-categories. 
Aside from discussing multiple characterisations of these $\BB$-categories, we prove an adjoint functor theorem and discuss limits and colimits of presentable $\BB$-categories. In the final two sections of this chapter, we switch gears and study a few concepts of higher algebra in the world of $\BB$-categories. 
In \S~\ref{chap:higherAlgebra}, we set up the main framework and use it to characterise dualisable objects in the $\BB$-category of modules over an $\mathbb{E}_\infty$-ring in $\BB$. 
In \S~\ref{chap:tensorProductBCategories}, we discuss tensor products of $\BB$-categories and in particular a symmetric monoidal structure on the $\BB$-category of presentable $\BB$-categories. We use this structure to realise $\BB$-modules in the $\infty$-category of presentable $\infty$-categories as presentable $\BB$-categories.

\section{Filtered $\BB$-categories}
\label{chap:filtered}
Classically, if $\kappa$ is a (regular) cardinal, a $1$-category $\JJ$ is said to be \emph{$\kappa$-filtered} if the colimit functor $\colim_{\JJ}\colon \Fun(\JJ,\Set)\to\Set$ commutes with $\kappa$-small limits. In~\cite{htt}, Lurie generalised this concept to the notion of a \emph{$\kappa$-filtered $\infty$-category} $\JJ$, which is an $\infty$-category for which $\colim_{\JJ}\colon\Fun(\JJ,\SS)\to\SS$ preserves $\kappa$-small limits. The main goal of this section is to discuss an analogous concept for $\BB$-categories. Following ideas originally introduced in $1$-category theory by Ad\'amek-Borceux-Lack-Rosick\'y~\cite{Adamek2002} and later generalised to $\infty$-categories by Charles Rezk~\cite{rezk2021}, we will introduce the notion of a \emph{$\I{U}$-filtered $\BB$-category}, where $\I{U}$ is an arbitrary internal class, i.e.\ a full subcategory of the large $\BB$-category $\ICat_{\BB}$ of $\BB$-categories (cf.~\S~\ref{sec:universe}). The main definitions and basic properties of such $\I{U}$-filtered $\BB$-categories are discussed in \S~\ref{sec:UFilteredness}. In \S~\ref{sec:weaklyUFiltered}, we introduce a slightly weaker notion, that of a \emph{weakly} $\I{U}$-filtered $\BB$-category. Classically, a $\kappa$-filtered ($\infty$-)category can be equivalently described as an ($\infty$-)category in which every $\kappa$-small diagram has a cocone. The notion of weak $\I{U}$-filteredness is a generalisation of this condition. However, as the terminology suggests, this notion is a priori weaker than that of $\I{U}$-filteredness. Following~\cite{Adamek2002}, we will call an internal class $\I{U}$ a \emph{doctrine} if both conditions happen to be equivalent. In \S~\ref{sec:regularClasses} and \S~\ref{sec:decompositionProperty}, we will study two other important properties of internal classes: \emph{regularity} and the \emph{decomposition property}. Recall that a cardinal $\kappa$ is said to be regular if $\kappa$-small sets are closed under $\kappa$-small sums. The notion of regularity for internal classes aims at capturing this property in the world of $\BB$-categories. The decomposition property, on the other hand, is the condition that every $\BB$-category can be obtained as a $\I{U}$-filtered colimit of objects in $\I{U}$. Hence, this notion can be viewed as an analogue to the fact that every ($\infty$-)category is a $\kappa$-filtered colimit of $\kappa$-small ($\infty$-)categories. We will make use of this property when we discuss the notion of $\I{U}$-compactness in \S~\ref{sec:UCompactObjects}.

\subsection{$\I{U}$-filtered $\BB$-categories}
\label{sec:UFilteredness}
In this section, we introduce and study the notion of $\I{U}$-filteredness in the world of $\BB$-categories, where $\I{U}$ is an arbitrary internal class. We begin with the following definition, which is an evident generalisation of the classical concept of a $\kappa$-filtered ($\infty$-)category:
\begin{definition}
	\label{def:filteredCategories}
	For any internal class $\I{U}$ of $\BB$-categories, a $\BB$-category $\I{J}$ is said to be \emph{$\I{U}$-filtered} if the colimit functor $\colim\colon\IFun(\I{J},  \Univ)\to\Univ$ is $\I{U}$-continuous. We define the internal class $\IFilt_{\I{U}}$ of $\I{U}$-filtered categories as the full subcategory of $\ICat_{\BB}$ that is spanned by those $\Over{\BB}{A}$-categories $\I{J}$ that are $\pi_A^\ast\I{U}$-filtered, for every $A\in\BB$.
\end{definition}

\begin{remark}
	\label{rem:BCFilt}
	In the situation of Definition~\ref{def:filteredCategories}, the fact that $\I{U}$-continuity is a local condition~\cite[Remark~5.2.3]{MWColimits} implies that \emph{every} object $A\to\IFilt_{\I{U}}$ is $\pi_A^\ast\I{U}$-filtered (which a priori has no reason to be true). In particular, the sheaf associated with $\IFilt_{\I{U}}$ is given on local sections over $A\in\BB$ by the full subcategory of $\Cat(\Over{\BB}{A})$ that is spanned by the $\pi_A^\ast\I{U}$-filtered categories. For any $A\in\BB$, we therefore obtain a canonical equivalence $\pi_A^\ast\IFilt_{\I{U}}\simeq \IFilt_{\pi_A^\ast\I{U}}$.
\end{remark}

\begin{remark}
	\label{rem:FiltGaloisConnection}
	Clearly, if $\I{U}\into\I{V}$ is an inclusion of internal classes, every $\I{V}$-filtered $\BB$-category is in particular $\I{U}$-filtered. Therefore, one obtains an inclusion $\IFilt_{\I{V}}\into\IFilt_{\I{U}}$.
\end{remark}

\begin{remark}
	\label{rem:dualDefinitionFiltered}
	If $\I{I}$ and $\I{J}$ are $\BB$-categories, note that the horizontal mate of the commutative square
	\begin{equation*}
		\begin{tikzcd}
			\IFun(\I{I}\times\I{J},  \Univ)\arrow[from=r, "\diag_\ast"'] \arrow[d, "\colim_\ast"] & \IFun(\I{J},  \Univ)\arrow[d, "\colim"]\\
			\IFun(\I{I},  \Univ)\arrow[from=r, "\diag"] & \Univ
		\end{tikzcd}
	\end{equation*}
	(with respect to the two adjunctions $\diag\dashv \lim$ and $\diag_\ast\dashv\lim_\ast$) is equivalent to the horizontal mate of the commutative square
	\begin{equation*}
		\begin{tikzcd}
			\IFun(\I{I},  \Univ)\arrow[r, "\diag_\ast"] \arrow[d, "\lim"]& \IFun(\I{I}\times\I{J},  \Univ)\arrow[d, "\lim_\ast"]\\
			\Univ\arrow[r, "\diag"] & \IFun(\I{J},  \Univ).
		\end{tikzcd}
	\end{equation*}
	(with respect to the two adjunctions $\colim\dashv\diag$ and $\colim_\ast\dashv\diag_\ast$). As a consequence, the functor $\colim\colon \IFun(\I{J},  \Univ)\to\Univ$ commutes with $\I{I}$-indexed limits if and only if the functor $\lim\colon\IFun(\I{I},  \Univ)\to\Univ$ commutes with $\I{J}$-indexed colimits. Thus, if $\I{U}$ is an internal class of $\BB$-categories, a $\BB$-category $\I{J}$ is $\I{U}$-filtered precisely if for all $A\in\BB$ and all $\I{I}\in\I{U}(A)$ the limit functor $\lim\colon \IFun(\I{I},  \Univ[\Over{\BB}{A}])\to\Univ[\Over{\BB}{A}]$ commutes with $\pi_A^\ast\I{J}$-indexed colimits.
\end{remark}

Recall from~\cite[\S~4.2]{MYoneda} the definition of the \emph{right cone} $\I{J}^\triangleright$ of a $\BB$-category $\I{J}$. It comes with an inclusion $\iota\colon\I{J}\into\I{J}^\triangleright$ such that for every $\BB$-category $\I{C}$ that admits $\I{J}$-indexed colimits, the functor of left Kan extension $\iota_!$ exists and carries an $\I{I}$-indexed diagram in $\I{C}$ to its colimit cocone~\cite[Proposition~6.3.8]{MWColimits}. We now obtain:
\begin{proposition}
	\label{prop:characterisationFilteredCategoryCocones}
	A $\BB$-category $\I{J}$ is $\I{U}$-filtered with respect to some internal class $\I{U}$ if and only if the inclusion $\iota_!\colon \IFun(\I{J},  \Univ)\into\IFun(\I{J}^\triangleright,  \Univ)$ is $\I{U}$-continuous.
\end{proposition}
\begin{proof}
	By definition, the functor $\iota_!$ is $\I{U}$-continuous if and only if for all $A\in\BB$ the functor $\pi_A^\ast(\iota_!)\simeq (\pi_A^\ast\iota)_!$ (cf.~\cite[Lemma~4.2.3]{MYoneda} and~\cite[Corollary~3.1.9]{MWColimits}) preserves limits of $\I{I}$-indexed diagrams for all $\I{I}\in\I{U}(A)$, and $\I{J}$ is $\I{U}$-filtered if and only if the colimit functor $\colim\colon \IFun[\Over{\BB}{A}](\pi_A^\ast\I{J},\Univ[\Over{\BB}{A}])\to\Univ[\Over{\BB}{A}]$ commutes with $\I{I}$-indexed limits for all $\I{I}\in\I{U}(A)$. By replacing $\BB$ with $\Over{\BB}{A}$, it therefore suffices to show that for any $\I{I}\in\I{U}(1)$, the functor $\iota_!$ commutes with $\I{I}$-indexed limits if and only if $\colim\colon \IFun(\I{J},  \Univ)\to\Univ$ preserves $\I{I}$-limits. Note that by combining~\cite[Proposition~4.6.1]{MWColimits} with the fact that the cone point $\infty\colon 1\to \I{I}^\triangleright$ is final, one finds that the colimit functor $\colim\colon \IFun(\I{J}^\triangleright,  \Univ)\to\Univ$ is given by evaluation at $\infty$. As a consequence,~\cite[Proposition~4.3.1]{MWColimits} implies that the colimit functor $\colim\colon \IFun(\I{J}^\triangleright,  \Univ)\to\Univ$ preserves $\I{I}$-indexed limits. Owing to the commutative diagram
	\begin{equation*}
		\begin{tikzcd}
			\IFun(\I{J},  \Univ)\arrow[dr, "\colim_{\I{J}}"] \arrow[d, hookrightarrow, "\iota_!"]& \\
			\IFun(\I{J}^\triangleright,  \Univ)\arrow[r, "\colim_{\I{J}^\triangleright}"'] & \Univ,
		\end{tikzcd}
	\end{equation*}
	the functoriality of mates thus implies that $i_!$ preserving $\I{I}$-indexed limits implies that $\colim_{\I{J}}$ commutes with $\I{I}$-indexed limits as well. The converse direction, on the other hand, follows from combining the functoriality of the mate construction with the straightforward observation that $(\iota^\ast,\infty^\ast)\colon \IFun(\I{J}^\triangleright,  \Univ)\to\IFun(\I{J},  \Univ)\times\Univ$ is a conservative functor.
\end{proof}
By an analogous argument as in the proof of Proposition~\ref{prop:characterisationFilteredCategoryCocones} and by furthermore using Remark~\ref{rem:dualDefinitionFiltered}, one obtains:
\begin{proposition}
	\label{prop:characterisationFilteredCategoryCones}
	A category $\I{J}$ in $\BB$ is $\I{U}$-filtered with respect to some internal class $\I{U}$ if and only if for all $A\in\BB$ and all $\I{I}\in\I{U}(A)$ the functor of right Kan extension
	\begin{equation*}
		\iota_\ast\colon \IFun(\I{I},  \Univ[\Over{\BB}{A}])\into\IFun(\I{I}^\triangleleft,\Univ[\Over{\BB}{A}])
	\end{equation*}
	preserves $\pi_A^\ast\I{J}$-indexed colimits.\qed
\end{proposition}

\begin{remark}
	\label{rem:FiltUColimitClass}
	By Proposition~\ref{prop:characterisationFilteredCategoryCones} and~\cite[Proposition~4.6.1]{MWColimits}, if $\I{J}\to\I{K}$ is a final functor such that $\I{J}$ is $\I{U}$-filtered, then $\I{K}$ must be $\I{U}$-filtered as well. Since the final $\BB$-category $1$ is trivially $\I{U}$-filtered for every choice of internal class $\I{U}$, this means that $\IFilt_{\I{U}}$ is a \emph{colimit class} in the sense of~\cite[Definition~5.3.5]{MWColimits}.
\end{remark}

For later use, let us note the following closure property of $\I{U}$-filtered $\BB$-categories:
\begin{proposition}
	\label{prop:UFilteredUCocomplete}
	For any internal class $\I{U}$, the internal class $\IFilt_{\I{U}}$ is closed under $\IFilt_{\I{U}}$-colimits in $\ICat_{\BB}$.
\end{proposition}
\begin{proof}
	By Remark~\ref{rem:BCFilt}, it suffices to show that if $\I{J}$ is a $\I{U}$-filtered $\BB$-category and $d\colon \I{J}\to\IFilt_{\I{U}}$ is a diagram, its colimit $\I{K}$ in $\ICat_{\BB}$ is also $\I{U}$-filtered. Given any $\I{I}\in\I{U}(1)$, Proposition~\ref{prop:decompositionExistenceColimits} shows that the functor of right Kan extension $\iota_\ast\colon \IFun(\I{I},  \Univ)\into\IFun(\I{I}^\triangleleft, \Univ)$ commutes with $\I{K}$-indexed colimits. As for any $A\in\BB$ the $\Over{\BB}{A}$-category $\pi_A^\ast\I{K}$ is the colimit of $\pi_A^\ast d$, the same argument also shows that for all $\I{I}\in\I{U}(A)$ the functor $\iota_\ast\colon\IFun(\I{I}^\triangleleft,\Univ[\Over{\BB}{A}])\into\IFun(\I{I},  \Univ[\Over{\BB}{A}])$ commutes with $\pi_A^\ast\I{K}$-indexed colimits. Hence Proposition~\ref{prop:characterisationFilteredCategoryCones} implies that $\I{K}$ is $\I{U}$-filtered.
\end{proof}

\subsection{Weakly $\I{U}$-filtered $\BB$-categories}
\label{sec:weaklyUFiltered}
Recall from~\cite[Remark~5.2.2]{MWColimits} that if $\I{U}$ is an internal class, we denote by $\op(\I{U})$ the internal class that arises as the image of $\I{U}$ along the equivalence $(-)^\op\colon \ICat_{\BB}\simeq\ICat_{\BB}$.
In practice, we will often require that every $\op(\I{U})$-cocomplete $\BB$-category is $\I{U}$-filtered. However, this is not true for every internal class $\I{U}$, not even in the case $\BB=\SS$~\cite[\S 6]{rezk2021}. In this section, we will therefore study a slightly weaker notion than that of a filtered $\I{U}$-category, which will encompass the class of $\op(\I{U})$-cocomplete $\BB$-categories. We adopted the idea of weak $\I{U}$-filteredness from Charles Rezk~\cite{rezk2021}, who in turn generalised ideas from~\cite{Adamek2002} to $\infty$-categories.

\begin{definition}
	\label{def:weaklyFiltered}
	If $\I{U}$ is an internal class of $\BB$-categories, a $\BB$-category $\I{J}$ is \emph{weakly $\I{U}$-filtered} if for every $A\in\BB$ and every $\I{I}\in\I{U}(A)$ the diagonal functor $\pi_A^\ast\I{J}\to\IFun[\Over{\BB}{A}](\I{I}^{\op},\pi_A^\ast\I{J})$ is final. We define the internal class $\IwFilt_{\I{U}}$ as the full subcategory of $\ICat_{\BB}$ that is spanned by the weakly $\pi_A^\ast\I{U}$-filtered $\Over{\BB}{A}$-categories, for every $A\in\BB$.
\end{definition}

\begin{remark}
	\label{rem:BCwFilt}
	In the situation of Definition~\ref{def:weaklyFiltered}, as the condition of a functor of $\BB$-categories being final is \emph{local} in $\BB$~\cite[Remark~4.4.9]{MYoneda}, every object $A\to\IwFilt_{\I{U}}$ is weakly $\pi_A^\ast\I{U}$-filtered. In particular, there is a canonical equivalence $\pi_A^\ast\IwFilt_{\I{U}}\simeq\IwFilt_{\pi_A^\ast\I{U}}$ for all $A\in\BB$.
\end{remark}
\begin{remark}
	\label{rem:wFiltGaloisConnection}
	If $\I{U}\into\I{V}$ is an inclusion of internal classes, every weakly $\I{V}$-filtered $\BB$-category is in particular weakly $\I{U}$-filtered. One therefore obtains an inclusion $\IwFilt_{\I{V}}\into\IwFilt_{\I{U}}$.
\end{remark}

\begin{example}
	\label{ex:UCocompleteWeaklyFiltered}
	By Quillen's theorem A for $\BB$-categories~\cite[Corollary~4.4.8]{MYoneda}, every functor that admits a left adjoint is final. Consequently, every $\op(\I{U})$-cocomplete $\BB$-category $\I{I}$ is in particular weakly $\I{U}$-filtered.
\end{example}

\begin{proposition}
	\label{prop:characterisationWeaklyFiltered}
	A $\BB$-category $\I{J}$ is weakly $\I{U}$-filtered if and only if for every $A\in\BB$ and every $\I{I}\in\I{U}(A)$ the colimit functor $\colim\colon \IFun[\Over{\BB}{A}](\pi_A^\ast\I{J},\Univ[\Over{\BB}{A}])\to\Univ[\Over{\BB}{A}]$ preserves $\I{I}$-indexed limits of corepresentables.
\end{proposition}
\begin{proof}
	To begin with, note that since the colimit of a diagram $f\colon\I{J}\to\Univ$ is given by the groupoidification of the associated left fibration $\Under{\I{J}}{f}$~\cite[Proposition~4.4.1]{MWColimits}, the colimit of every corepresentable is given by the final object in $\Univ$. In other words, there is a commutative square
	\begin{equation*}
		\begin{tikzcd}
			\I{J}^{\op}\arrow[r, hookrightarrow, "h_{\I{J}^{\op}}"]\arrow[d, "\pi_{\I{J}^\op}"] & \IFun(\I{J},  \Univ)\arrow[d, "\colim"]\\
			1\arrow[r, "1_{\Univ}", hookrightarrow] & \Univ.
		\end{tikzcd}
	\end{equation*}
	As a result, for any diagram $d\colon\I{I}\to\I{J}^{\op}$, the presheaf $\colim h_{\I{J}^{\op}}d$ is equivalent to the constant functor $1_{\Univ}\pi_{\I{J}^\op}\colon \I{J}^{\op}\to\Univ$. As the inclusion $1_{\Univ}\into\Univ$ admits a left adjoint~\cite[Example~4.1.11]{MWColimits} and is therefore continuous~\cite[Proposition~5.2.5]{MWColimits}, we conclude that the limit $\lim(\colim h_{\I{J}^{\op}} d)$ is given by the final object in $\Univ$. Hence the canonical map
	\begin{equation*}
		\colim(\lim h_{\I{J}^{\op}} d)\to\lim(\colim h_{\I{J}^{\op}} d)
	\end{equation*}
	is an equivalence if and only if the domain of this map is the final object as well. On account of the chain of equivalences
	\begin{align*}
		\lim h_{\I{J}^{\op}}d&\simeq\map{\IFun(\I{J},  \Univ)}(h_{\I{J}^{\op}}(-),\lim h_{\I{J}^\op}d)\\
		&\simeq\map{\IFun(\I{I},\IFun(\I{J},  \Univ))}(\diag h_{\I{J}^{\op}}(-), h_{\I{J}^{\op}} d)\\
		&\simeq \map{\IFun(\I{I},\I{J}^{\op})}(\diag(-), d),
	\end{align*}
	the functor $\lim h_{\I{J}^{\op}}d$ classifies the left fibration $\Under{\I{J}}{d^{\op}}\to\I{J}$. Hence $\colim(\lim h_{\I{J}^{\op}}d)$ is the final object if and only if $(\Under{\I{J}}{d^{\op}})^\gp\simeq 1$. By replacing $\BB$ with $\Over{\BB}{A}$, the same argumentation goes through for any $\I{I}\in\I{U}(A)$ and any diagram $d\colon \I{I}\to\pi_A^\ast\I{J}^{\op}$. By Quillen's theorem A for $\BB$-categories~\cite[Corollary~4.4.8]{MYoneda}, the result thus follows.
\end{proof}

\begin{corollary}
	\label{cor:inclusionWFiltFilt}
	For every internal class $\I{U}$ of $\BB$-categories, any $\I{U}$-filtered $\BB$-category is weakly $\I{U}$-filtered. In other words, there is an inclusion $\IFilt_{\I{U}}\into\IwFilt_{\I{U}}$ of internal classes.\qed
\end{corollary}

Following the terminology introduced in~\cite{Adamek2002}, we may now make the following definition:
\begin{definition}
	\label{def:soundness}
	An internal class $\I{U}$ of $\BB$-categories is said to be \emph{sound} if the inclusion $\IwFilt_{\I{U}}\into\IFilt_{\I{U}}$ is an equivalence. It is called \emph{weakly sound} if for every $A\in\BB$, every $\op(\pi_A^\ast\I{U})$-cocomplete $\Over{\BB}{A}$-category is $\pi_A^\ast\I{U}$-filtered.
\end{definition}

\begin{remark}
	\label{rem:BCSoundness}
	On account of Remark~\ref{rem:BCFilt} and Remark~\ref{rem:BCwFilt}, the \'etale base change of a (weakly) sound internal class is (weakly) sound as well.
\end{remark}

We finish this section with another characterisation of weakly $\I{U}$-filtered $\BB$-categories that will be useful later. Recall from~\cite[Definition~6.2.1]{MWColimits} that if $\I{C}$ is an arbitrary $\BB$-category and $\I{V}$ is an arbitrary internal class of $\BB$-categories, we denote by $\ISml^{\I{V}}(\I{C})\into\IPSh(\I{C})$ the full subcategory that is spanned by those objects $F\colon A\to\IPSh(\I{C})$ for which the domain of the associated right fibration $\Over{\I{C}}{F}$ is contained in $\I{V}^{\colim}(A)$ (where $\I{V}^{\colim}$ is the smallest colimit class containing $\I{V}$, see~\cite[Definition~5.3.5]{MWColimits}). We now obtain:
\begin{proposition}
	\label{prop:weaklyFilteredInd}
	A $\BB$-category $\I{J}$ is weakly $\I{U}$-filtered if and only if the inclusion
	\begin{equation*}
		\I{J}\into\ISml^{\op(\I{U})}(\I{J})
	\end{equation*}
	induced by the Yoneda embedding is final.
\end{proposition}
\begin{proof}
	By Quillen's theorem A for $\BB$-categories~\cite[Corollary~4.4.8]{MYoneda}, the inclusion $\I{J}\into\ISml^{\op(\I{U})}(\I{J}$) is final if and only if for every $A\in\BB$ and every $\op(\I{U})$-small presheaf $F\colon A\to \ISml^{\op(\I{U})}(\I{J})$ the groupoidification of the $\Over{\BB}{A}$-category $\Under{\I{J}}{F}$ is the final object in $\Over{\BB}{A}$. By the same reasoning, $\I{J}$ is weakly $\I{U}$-filtered if and only if for every $\I{I}\in\I{U}(A)$ and every diagram $d\colon \I{I}^\op\to\pi_A^\ast\I{J}$ the groupoidification of the $\Over{\BB}{A}$-category $\Under{\pi_A^\ast\I{J}}{d}$ is final in $\Over{\BB}{A}$. Hence it suffices to show that for every such diagram $d$, there is an object $F\colon A\to\ISml^{\op(\I{U})}(\I{J})$ such that $\Under{\pi_A^\ast\I{J}}{d}\simeq\Under{\I{J}}{F}$, and vice versa. By replacing $\BB$ with $\Over{\BB}{A}$ and by using~\cite[Remark~6.2.2]{MWColimits}, we may assume that $A\simeq 1$. Now by~\cite[Proposition~6.2.6]{MWColimits}, the colimit of $h_{\I{J}} d\colon\I{I}\to\IPSh(\I{J})$ is contained in $\ISml^{\op(\I{U})}(\I{J})$ and therefore defines a $\I{U}$-small presheaf $F$. By construction, we have an equivalence $\Under{\ISml^{\I{U}}(\I{J})}{F}\simeq \Under{\ISml^{\I{U}}(\I{J})}{h_{\I{J}}d}$ whose pullback along the Yoneda embedding determines an equivalence $\Under{\I{J}}{d}\simeq \Under{\I{J}}{F}$. Hence, if $\Under{\I{J}}{F}^\gp$ is final, so is $\Under{\I{J}}{d}^\gp$. Conversely, if we are given an arbitrary $\I{U}$-small presheaf $F$, the fact that $\Under{\I{J}}{F}^\gp$ being final is \emph{local} in $\BB$ implies (by definition of what it means for a presheaf to be $\I{U}$-small) that we may safely assume that there is a diagram $d\colon\I{I}^\op\to\I{J}$ with $\I{I}\in\I{U}(1)$ such that $F\simeq \colim h_{\I{J}}d$. By the same argument as above, we thus conclude that if $\Under{\I{J}}{d}^\gp$ is final, so is $\Under{\I{J}}{F}$, which finishes the proof.
\end{proof}

\subsection{Regular classes}
\label{sec:regularClasses}
Recall that a cardinal $\kappa$ is said to be \emph{regular} if it is infinite and if any $\kappa$-small union of $\kappa$-small sets is still $\kappa$-small. In this section, we will study an analogue of this condition in the context of internal classes of $\BB$-categories. To that end, recall from the discussion in~\cite[\S~6.1]{MCocartesian} that the Yoneda embedding $\Delta\into\IPSh(\Delta)$ (where $\Delta$ is implicitly regarded as a constant $\BB$-category) factors through the embedding $\ICat_{\BB}\into\IPSh(\Delta)$, so that we may regard $\Delta$ as an internal class of $\BB$-categories. We may now define:
\begin{definition}
	\label{def:regularisation}
	An internal class $\I{U}$ is said to be \emph{right regular} if $\I{U}$ contains $\Delta$ and if $\I{U}$ is closed under $\I{U}$-colimits in $\ICat_{\BB}$. We define the \emph{right regularisation} $\I{U}^\reg_\rightarrow$ of $\I{U}$ to be the smallest right regular class that contains $\I{U}$.
	
	Dually, $\I{U}$ is called \emph{left regular} if it contains $\Delta$ and is closed under $\op(\I{U})$-colimits in $\ICat_{\BB}$, and we define the \emph{left regularisation} $\I{U}^\reg_\leftarrow$ as the smallest left regular class that contains $\I{U}$.
	
	Finally, we say that $\I{U}$ is \emph{regular} if it is both left and right regular, and we define the \emph{regularisation} $\I{U}^\reg$ of  $\I{U}$ as the smallest regular class that contains $\I{U}$.
\end{definition}

\begin{remark}
	An internal class $\I{U}$ of $\BB$-categories is left regular if and only if $\op(\I{U})$ is right regular, and there is an evident equivalence $\op(\I{U}^\reg_\leftarrow)\simeq \op(\I{U})^\reg_\rightarrow$ of internal classes.
	In particular, if we have an equivalence $\I{U}\simeq\op(\I{U})$ of internal classes, then the notions of left and right regularity collaps to the notion of regularity, and the left/right regularisation of $\I{U}$ is already its regularisation (cf.\ Corollary~\ref{cor:regularisationMinimalCondition} below).
\end{remark}

\begin{remark}
	\label{rem:BCRegularisation}
	By the same argument as in the proof of~\cite[Proposition~7.1.11]{MWColimits}, there is an equivalence $\pi_A^\ast(\I{U}^\reg_\rightarrow)\simeq(\pi_A^\ast\I{U})^\reg_\rightarrow$ for any internal class $\I{U}$ and any $A\in\BB$. In particular, the \'etale base change of a right regular class is still right regular. Similar observations can be made for the (left) regularisation of $\I{U}$.
\end{remark}

\begin{proposition}
	\label{prop:invarianceCocompletenessRegularisation}
	For every internal class $\I{U}$ of $\BB$-categories, a $\BB$-category is $\I{U}$-cocomplete if and only if it is $\I{U}^\reg_\rightarrow$-cocomplete, and a functor between $\BB$-categories is $\I{U}$-cocontinuous if and only if it is $\I{U}^\reg_\rightarrow$-cocontinuous.
	
	Dually, a $\BB$-category is $\I{U}$-complete if and only if it is $\I{U}^\reg_\leftarrow$-complete, and a functor between $\BB$-categories is $\I{U}$-complete if and only if it is $\I{U}^\reg_\leftarrow$-continuous.
\end{proposition}
\begin{proof}
	We only prove the first statement, the second one follows by dualisation. So let $\I{C}$ be a $\I{U}$-cocomplete $\BB$-category, and let $\I{V}$ be the largest internal class of $\BB$-categories subject to the condition that $\I{C}$ is $\I{V}$-cocomplete. Clearly $\I{V}$ contains $\Delta$ since every $\BB$-category is $\Delta$-cocomplete (cf.~\cite[Remark~5.2.4]{MWColimits}). Moreover, Proposition~\ref{prop:decompositionExistenceColimits} implies that for any $\I{I}\in\I{V}(1)$ and any diagram $d\colon\I{I}\to\I{V}$ with colimit $\I{K}$, the $\BB$-category $\I{C}$ admits $\I{K}$-indexed colimits, which implies that $\I{K}\in \I{V}(1)$ by maximality of $\I{V}$. Upon replacing $\BB$ with $\Over{\BB}{A}$ and repeating the same argument, one concludes that $\I{V}$ is closed under $\I{V}$-colimits in $\ICat_{\BB}$ and must therefore contain $\I{U}^\reg_\rightarrow$. An analogous argument also shows that every $\I{U}$-cocontinuous functor is $\I{U}^\reg_\rightarrow$-cocontinuous.
\end{proof}

\begin{corollary}
	\label{cor:regularisationMinimalCondition}
	The right (left) regularisation of an internal class $\I{U}$ is the smallest internal class that contains $\I{U}$ and $\Delta$ and that is closed under $\I{U}$-colimits ($\op(\I{U})$-colimits) in $\ICat_{\BB}$.
\end{corollary}
\begin{proof}
	This is an immediate conseqence of the observation that by Proposition~\ref{prop:invarianceCocompletenessRegularisation}, an internal class $\I{V}$ of $\BB$-categories is closed under $\I{U}$-colimits ($\op(\I{U})$-colimits) in $\ICat_{\BB}$ if and only if it is closed under $\I{U}^\reg_\rightarrow$-colimits
	($\op(\I{U}^\reg_\leftarrow)$-colimits) in $\ICat_{\BB}$.
\end{proof}

\begin{proposition}
	\label{prop:regularisationFilt}
	For every internal class $\I{U}$, the inclusion $\IFilt_{\I{U}^\reg_\leftarrow}\into\IFilt_{\I{U}}$ is an equivalence.
\end{proposition}
\begin{proof}
	In light of Remark~\ref{rem:BCRegularisation} and Remark~\ref{rem:BCFilt}, it suffices to show that every $\I{U}$-filtered $\BB$-category $\I{J}$ is already $\I{U}^\reg_\leftarrow$-filtered. This amounts to showing that the functor $\colim_{\I{J}}\colon\IFun(\I{I},  \Univ)\to\Univ$
	is $\I{U}^\reg_\leftarrow$-continuous. By Proposition~\ref{prop:invarianceCocompletenessRegularisation}, this is immediate.
\end{proof}
\begin{corollary}
	\label{cor:regularisationSoundness}
	The left regularisation of a (weakly) sound internal class is also (weakly) sound.
\end{corollary}
\begin{proof}
	Suppose that $\I{U}$ is sound, i.e.\ that $\IFilt_{\I{U}}\into\IwFilt_{\I{U}}$ is an equivalence. Since $\I{U}\into\I{U}^\reg_\leftarrow$ implies that we have an inclusion $\IwFilt_{\I{U}^\reg_\leftarrow}\into\IwFilt_{\I{U}}$, Proposition~\ref{prop:regularisationFilt} implies that the inclusion $\IFilt_{\I{U}^\reg_\leftarrow}\into\IwFilt_{\I{U}^\reg}$ is also an equivalence, hence $\I{U}^\reg_\leftarrow$ is sound. The case where $\I{U}$ is weakly sound follows from a similar argument.
\end{proof}

For the study of accessibility and presentability of $\BB$-categories, we will generally need to restrict our attention to those internal classes of $\BB$-categories that are themselves \emph{small} $\BB$-categories. It will therefore be useful to give such internal classes a dedicated name. Again following~\cite{Adamek2002}, we thus define:
\begin{definition}
	\label{def:doctrine}
	An internal class $\I{U}$ of $\BB$-categories is a \emph{doctrine} if $\I{U}$ is a small $\BB$-category.
\end{definition}
\begin{proposition}
	\label{prop:regularisationDoctrine}
	The (left/right) regularisation of a doctrine is still a doctrine.
\end{proposition}
\begin{proof}
	It suffices to show that any doctrine $\I{U}$ is contained in a regular doctrine $\I{V}$. We will explicitly construct such a doctrine in~\S~\ref{sec:kappaSmall} below, cf.~Remark~\ref{rem:regularClassesExhaustive}.
\end{proof}

\subsection{The decomposition property}
\label{sec:decompositionProperty}
It is well-known that for every regular cardinal $\kappa$, any $\infty$-category can be written as a $\kappa$-filtered colimit of $\kappa$-small $\infty$-categories. In order to obtain a well-behaved notion of accessibility internal to $\BB$, it will be crucial to have an analogue of this property for $\BB$-categories. This leads us to the following definition:

\begin{definition}
	\label{def:decompositionProperty}
	An internal class $\I{U}$ of $\BB$-categories is said to have the \emph{decomposition property} if for every $A\in\BB$ and every $\Over{\BB}{A}$-category $\I{C}$, there is a $\pi_A^\ast\I{U}$-filtered $\Over{\BB}{A}$-category $\I{J}$ and a diagram $d\colon\I{J}\to\pi_A^\ast\I{U}$ with colimit $\I{C}$.
\end{definition}

\begin{remark}
	\label{rem:decompositionPropertyVariant}
	In the situation of Definition~\ref{def:decompositionProperty}, by applying the decomposition property to $\I{D}=\I{C}^{\op}$, one deduces that $\I{C}$ can also be obtained as a $\pi_A^\ast\I{U}$-filtered colimit of a diagram in $\op(\pi_A^\ast\I{U})$.
\end{remark}

The main goal of this section is to show:
\begin{proposition}
	\label{prop:decompositionColimitsCatB}
	Every left regular and weakly sound internal class $\I{U}$ has the decomposition property.
\end{proposition}

Before we can prove Proposition~\ref{prop:decompositionColimitsCatB}, we need a few preparations.
\begin{lemma}
	\label{lem:coYonedaGeneralised}
	Let $\I{C}$ be a small $\BB$-category and let $\I{D}\into\IPSh(\I{C})$ be a full subcategory that contains $\I{C}$. Then any presheaf $F\colon\I{C}^{\op}\to\Univ$ is the colimit of the diagram $\Over{\I{D}}{F}\to\I{D}\into\IPSh(\I{C})$.
\end{lemma}
\begin{proof}
	By Proposition~\ref{prop:sliceFibrationColimits}, it suffices to show that the final object in $\Over{\IPSh(\I{C})}{F}$ is the colimit of the inclusion $\Over{\I{D}}{F}\into\Over{\IPSh(\I{C})}{F}$. In light of the inclusions $\Over{\I{C}}{F}\into\Over{\I{D}}{F}\into\Over{\IPSh(\I{C})}{F}$ and by making use of the equivalence $\Over{\IPSh(\I{C})}{F}\simeq\IPSh(\Over{\I{C}}{F})$ from~\cite[Lemma~6.1.5]{MWColimits}, we may thus assume that $F$ is the final object in $\IPSh(\I{C})$. Moreover, since the inclusion $\Univ[\BB]\into\Univ[\BBB]$ is cocontinuous~\cite[Example~5.2.8]{MWColimits}, we may enlarge our universe and thus assume without loss of generality that $\I{D}$ is \emph{small}.
	Now let $i\colon\I{C}\into\I{D}$ and $j\colon\I{D}\into\IPSh(\I{C})$ be the inclusions. 
	Since the identity on $\IPSh(\I{C})$ is the left Kan extension of the Yoneda embedding $h$ along itself~\cite[Theorem~6.1.1]{MWColimits}, we obtain equivalences $j\simeq j^\ast j_!i_!(h)\simeq i_!(h)$, where the functor $i_!$ exists since $\I{D}$ is small~\cite[Corollary~6.3.7]{MWColimits}. Therefore, the identity on $\IPSh(\I{C})$ is also the left Kan extension of $j$ along itself. The claim now follows from the explicit description of the left Kan extension in~\cite[Remark 6.3.6]{MWColimits}, which implies that we have equivalences $1_{\IPSh(\I{C})}\simeq j_!(j)(1_{\IPSh(\I{C})})\simeq \colim j$. 
\end{proof}

\begin{lemma}
	\label{lem:pullbackCocompleteRightFibration}
	Let $\I{U}$ be an internal class of $\BB$-categories and let
	\begin{equation*}
		\begin{tikzcd}
			\I{Q}\arrow[r, "g"]\arrow[d, "q"] & \I{P}\arrow[d, "p"]\\
			\I{D}\arrow[r, "f"] & \I{C}
		\end{tikzcd}
	\end{equation*}
	be a pullback square in $\Cat(\BB)$ in which $\I{D}$, $\I{C}$ and $\I{P}$ are $\I{U}$-cocomplete and both $f$ and $p$ are $\I{U}$-cocontinuous. Then $\I{Q}$ is $\I{U}$-cocomplete and both $q$ and $g$ are $\I{U}$-cocontinuous.
\end{lemma}
\begin{proof}
	By replacing $\BB$ with $\Over{\BB}{A}$ for $A\in\BB$ if necessary and using~\cite[Remarks~4.1.3 and~4.2.2]{MWColimits}, it will suffice to prove that any diagram $d\colon \I{K}\to\I{Q}$ with $\I{K}\in\I{U}(1)$ admits a colimit in $\I{Q}$ and that furthermore $q$ preserves this colimit. The pullback square in the statement of the lemma induces a commutative diagram
	\begin{equation*}
		\begin{tikzcd}[row sep={3em,between origins}, column sep={4em,between origins}]
			& 1\arrow[rr]\arrow[dd]\arrow[dl, "\overline{d}"'] && 1\arrow[dd]\arrow[dl, "\overline{gd}"']\\
			\Under{\I{Q}}{d}\arrow[dd, "q_\ast", near start]\arrow[rr, crossing over, "g_\ast", near end] && \Under{\I{P}}{gd} & \\
			& 1\arrow[rr]\arrow[dl, "\overline{qd}"'] && 1\arrow[dl, "\overline{pgd}"']\\
			\Under{\I{D}}{qd}\arrow[rr, "f_\ast", near end] && \Under{\I{C}}{pgd}\arrow[from=uu, crossing over, "p_\ast", near start]
		\end{tikzcd}
	\end{equation*}
	in which the front square is a pullback, the three cocones $\overline{qd}$, $\overline{gd}$ and $\overline{pgd}$ are colimit cocones and the cocone $\overline{d}$ is determined by the universal property of pullbacks. To finish the proof, it suffices to show that $\overline{d}$ is a colimit cocone, i.e.\ initial. Given any $\overline{d}^\prime\colon 1\to\Under{\I{Q}}{{d}}$, we obtain a pullback square
	\begin{equation*}
		\begin{tikzcd}
			\map{\Under{\I{Q}}{d}}(\overline{d},\overline{d}^\prime)\arrow[r]\arrow[d] & \map{\Under{\I{P}}{gd}}(\overline{gd}, g_\ast\overline{d}^\prime)\arrow[d]\\
			\map{\Under{\I{D}}{fd}}(\overline{qd}, q_\ast\overline{d}^\prime)\arrow[r] & \map{\Under{\I{C}}{pgd}}(\overline{pgd}, p_\ast g_\ast \overline{d}^\prime)
		\end{tikzcd}
	\end{equation*}
	in $\BB$. Since $\overline{qd}$, $\overline{gd}$ and $\overline{pgd}$ are initial, it follows that the cospan in the lower right corner is constant on the final object $1\in\BB$, hence $\map{\Under{\I{Q}}{d}}(\overline{d},\overline{d}^\prime)\simeq 1$. By replacing $\BB$ with $\Over{\BB}{A}$ and $\overline{d}$ with $\pi_A^\ast(\overline{d})$, the same is true for any object $\overline{d}^\prime\colon A\to\Under{\I{Q}}{d}$. As a consequence, $\overline{d}$ must be initial.
\end{proof}

\begin{proof}[{Proof of Proposition~\ref{prop:decompositionColimitsCatB}}]
	By Remark~\ref{rem:BCRegularisation} and Remark~\ref{rem:BCSoundness}, it suffices to show that every $\BB$-category $\I{C}$ is a $\I{U}$-filtered colimit of a diagram in $\I{U}$.
	As $\I{U}$ by regularity contains $\Delta$ and since the localisation functor $\IPSh(\Delta)\to\ICat_{\BB}$ is cocontinuous, we deduce from Lemma~\ref{lem:coYonedaGeneralised} that $\I{C}$ arises as the colimit of the diagram $\Over{\I{U}}{\I{C}}\to\I{U}\into\ICat_{\BB}$. We therefore only need to show that $\Over{\I{U}}{\I{C}}$ is $\I{U}$-filtered. Using that $\I{U}$ is weakly sound, it will suffice to show that $\Over{\I{U}}{\I{C}}$ is $\op(\I{U})$-cocomplete.
	By Corollary~\ref{cor:sliceCategoryUCocomplete}, the $\BB$-category $\Over{(\ICat_{\BB})}{\I{C}}$ is cocomplete and the projection $(\pi_{\I{C}})_!$ is cocontinuous. As the inclusion $\I{U}\into\ICat_{\BB}$ is closed under $\op(\I{U})$-colimits, the desired result follows from Lemma~\ref{lem:pullbackCocompleteRightFibration}.
\end{proof}

\begin{corollary}
	\label{cor:criterionCocompletenesDecomposition}
	Let $\I{U}$ be a weakly sound internal class of $\BB$-categories. Then a (large) $\BB$-category $\I{C}$ is cocomplete if and only if $\I{C}$ is both $\op(\I{U})$- and $\IFilt_{\I{U}}$-cocomplete. Similarly, a functor $f\colon \I{C}\to\I{D}$ between cocomplete (large) $\BB$-categories is cocontinuous if and only if it is both $\op(\I{U})$- and $\IFilt_{\I{U}}$-cocontinuous.
\end{corollary}
\begin{proof}
	We prove the first statement, the second one follows by a similar argument. Since the claim is clearly necessary, it suffices to prove the converse. So let us assume that $\I{C}$ is both $\op(\I{U})$- and $\IFilt_{\I{U}}$-cocomplete. By Proposition~\ref{prop:invarianceCocompletenessRegularisation} and Proposition~\ref{prop:regularisationFilt}, we may assume without loss of generality that $\I{U}$ is left regular. Proposition~\ref{prop:decompositionColimitsCatB} now implies that $\I{U}$ has the decomposition property. By definition and in light of Remark~\ref{rem:decompositionPropertyVariant}, this means that $(\op(\I{U})\cup\IFilt_{\I{U}})^\reg_{\rightarrow}=\ICat_{\BB}$. Appealing once more to Proposition~\ref{prop:invarianceCocompletenessRegularisation}, the claim follows.
\end{proof}

\subsection{$\I{U}$-compact objects}
\label{sec:UCompactObjects}
Recall from~\cite[\S~7.2]{MWColimits} that if $\I{V}$ is an internal class and if $\I{C}$ is a $\I{V}$-cocomplete $\BB$-category, we say that an object $c\colon A\to\I{C}$ is \emph{$\I{V}$-cocontinuous} if the functor $\map{\I{C}}(c,-)\colon \pi_A^\ast\I{C}\to\Univ[\Over{\BB}{A}]$ is $\pi_A^\ast\I{V}$-cocontinuous. In this section, we specialise this concept to the case where $\I{V}=\IFilt_{\I{U}}$ for some internal class $\I{U}$. This leads us to the notion of a \emph{$\I{U}$-compact object}, which is the internal analogue of the concept of a $\kappa$-compact object in an $\infty$-category, where $\kappa$ is a cardinal.

\begin{definition}
	\label{def:UCompact}
	Let $\I{U}$ be an internal class of $\BB$-categories, and let $\I{C}$ be a $\IFilt_{\I{U}}$-cocomplete $\BB$-category. An object $c\colon A\to \I{C}$ in context $A\in\BB$ is said to be  \emph{$\I{U}$-compact} if it is $\IFilt_{\I{U}}$-cocontinuous, i.e.\ if the functor $\map{\I{C}}(c,-)\colon \pi_A^\ast\I{C}\to\Univ[\Over{\BB}{A}]$ is $\IFilt_{\pi_A^\ast\I{U}}$-cocontinuous. We denote by $\I{C}^{\cpt{\I{U}}}$ the full subcategory of $\I{C}$ that is spanned by the $\I{U}$-compact objects.
\end{definition}

\begin{remark}
	\label{rem:BCUComp}
	In the situation of Definition~\ref{def:UCompact}, an object $c\colon A\to \I{C}$ is contained in $\I{C}^{\cpt{\I{U}}}$ if and only if it is $\I{U}$-compact. This is a direct consequence of~\cite[Remark~7.2.2]{MWColimits}. Together with Remark~\ref{rem:BCFilt}, this implies that if $A\in\BB$ is an arbitrary object in $\BB$, there is a natural equivalence $\pi_A^\ast(\I{C}^{\cpt{\I{U}}})\simeq (\pi_A^\ast\I{C})^{\cpt{\pi_A^\ast\I{U}}}$.
\end{remark}

\begin{lemma}
	\label{lem:stabilityFiltUCocontinuousFunctors}
	Let $\I{U}$ be an internal class of $\BB$-categories, and let $\I{C}$ be a $\IFilt_{\I{U}}$-cocomplete $\BB$-category. Then the full subcategory $\IFun^{\cocont{\IFilt_{\I{U}}}}(\I{C},\Univ)\into\IFun(\I{C},\Univ)$ of $\IFilt_{\I{U}}$-cocontinuous functors is closed under $\I{U}$-limits.
\end{lemma}
\begin{proof}
	Using~\cite[Remark~5.3.4]{MWColimits} and Remark~\ref{rem:BCFilt}, it will suffice to show that whenever $\I{I}$ is a $\BB$-category that is contained in $\I{U}(1)$ and $d\colon \I{I}\to \IFun(\I{C},\Univ)^{\IFilt_{\I{U}}}$ is a diagram, then the limit $\lim d$ in $\IFun(\I{C},\Univ)$ is $\IFilt_{\I{U}}$-cocontinuous.  We may compute $\lim d$ as the composition
	\begin{equation*}
		\begin{tikzcd}
			\I{C}\arrow[r, "d^\prime"] & \IFun(\I{I},  \Univ)\arrow[r, "\lim_{\I{I}}"]  & \Univ,
		\end{tikzcd}
	\end{equation*}
	where $d^\prime$ is the transpose of $d$.
	Since $\lim_{\I{I}}$ is $\IFilt_{\I{U}}$-cocontinuous by Remark~\ref{rem:dualDefinitionFiltered}, it thus suffices to show that $d^\prime$ is $\IFilt_{\I{U}}$-cocontinuous as well. This can be extracted as a special case of Lemma~\ref{lem:limitpreservationswap} below.
\end{proof}
\begin{proposition}
	\label{prop:UCompactUcocomplete}
	Let $\I{U}$ be an internal class and let $\I{C}$ be an $\op(\I{U})$- and $\IFilt_{\I{U}}$-cocomplete $\BB$-category. Then the subcategory $\I{C}^{\cpt{\I{U}}}\into\I{C}$ is closed under $\op(\I{U})$-colimits in $\I{C}$.
\end{proposition}
\begin{proof}
	By using Remark~\ref{rem:BCUComp}, it suffices to show that whenever $\I{I}$ is a $\BB$-category that is contained in $\I{U}(1)$ and $d\colon \I{I}^\op\to\I{C}^{\cpt{\I{U}}}$ is a diagram, the colimit $\colim d$ in $\I{C}$ is $\I{U}$-compact. As we have noted in~\cite[Proposition~5.2.9]{MWColimits}, the Yoneda embedding $h_{\I{C}^{\op}}\colon\I{C}^{\op}\into\IFun(\I{C},\Univ)$ is $\op(\I{U})$-continuous, so that we can identify $\map{\I{C}}(\colim d,-)$ with the limit of the diagram $h_{\I{C}^{\op}} d^{\op}$. The desired result now follows from Lemma~\ref{lem:stabilityFiltUCocontinuousFunctors}.
\end{proof}

\begin{definition}
	\label{def:retracts}
	If $\I{C}\into\I{D}$ is a fully faithful functor of $\BB$-categories, the $\BB$-category $\IRet_{\I{D}}(\I{C})$ of \emph{retracts} of $\I{C}$ in $\I{D}$ is the full subcategory of $\I{D}$ that is spanned by those objects $d\colon A\to \I{D}$ in context $A\in\BB$ for which there is an object $c\colon A\to \I{C}$ and a commutative diagram
	\begin{equation*}
		\begin{tikzcd}[column sep=small]
			& c \arrow[dr]&\\
			d\arrow[ur]\arrow[rr, "\id"] && d.
		\end{tikzcd}
	\end{equation*}
\end{definition}

\begin{remark}
	\label{rem:retracts}
	In the situation of Definition~\ref{def:retracts}, there are inclusions $\I{C}\into\IRet_{\I{D}}(\I{C})\into\I{D}$. Furthermore, an object $d\colon A\to \I{D}$ is contained in $\IRet_{\I{C}}(\I{D})$ precisely if there is a cover $(s_i)\colon \bigsqcup_i A_i\onto A$ such that $s_i^\ast(d)\colon A_i\to \I{D}$ is a retract of an object $c\colon A_i\to \I{C}$. Therefore, if $\I{C}$ is small and $\I{D}$ is locally small (in the sense of~\cite[Definition~4.7.1]{MYoneda}), then $\IRet_{\I{D}}(\I{C})$ is small as well: in fact, by~\cite[Proposition~4.7.4]{MYoneda}, this follows once we verify that $\IRet_{\I{D}}(\I{C})_0$ is small. Since the latter admits a small cover
	\begin{equation*}
		\bigsqcup_{G\in\GG} \bigsqcup_{d\in \Ret_{\I{D}(G)}(\I{C}(G))} G\onto \IRet_{\I{D}}(\I{C})_0
	\end{equation*}
	where $\GG\subset \BB$ is a small generating subcategory and where $\Ret_{\I{D}(G)}(\I{C}(G))$ denotes the full subcategory of $\I{D}(G)$ that is spanned by the retracts of $\I{C}(G)$, which is clearly a small $\infty$-category, this is immediate.
\end{remark}

\begin{lemma}
	\label{lem:retractsCocontinuousFunctors}
	Let $\I{U}$ be an internal class of $\BB$-categories and let $\I{C}$ be a $\I{U}$-cocomplete $\BB$-category. Then the full subcategory $\IFun^{\cocont{\I{U}}}(\I{C},\Univ)\into\IFun(\I{C},\Univ)$ of $\I{U}$-cocontinuous functors is closed under retracts.
\end{lemma}
\begin{proof}
	By~\cite[Remark~5.3.4]{MWColimits}, it will suffice to show that whenever a copresheaf $F\colon\I{C}\to\Univ$ is a retract of a $\I{U}$-cocontinuous functor $G\colon\I{C}\to\Univ$, then $F$ is $\I{U}$-cocontinuous as well. Let $R=\Delta^2\sqcup_{\Delta^1}\Delta^0$ be the walking retract diagram, i.e.\ the quotient of $\Delta^2$ that is obtained by collapsing $d^1\colon\Delta^1\into\Delta^2$ to a point. Then the datum of retract $F\to G\to F$ is tantamount to a map $r\colon\I{C}\to\Univ^R$. Since the retract of an equivalence is an equivalence as well, the functor $d_{\{1\}}\colon\Univ^R\to \Univ$ that is obtained by evaluation at $\{1\}\in\Delta^2\onto R$ must be conservative. By combining this observation with the fact that $d_{\{1\}}$ is cocontinuous, the equivalence $d_{\{1\}}r\simeq G$ and the functoriality of mates, we conclude that the map $\colim_{\I{I}} r_\ast\to r\colim_{\I{I}}$ is an equivalence for every $\I{I}\in\I{U}(1)$. Upon replacing $\BB$ with $\Over{\BB}{A}$ and repeating the same argument, we thus find that $r$ is $\I{U}$-cocontinuous. As we recover $F$ by postcomposing $r$ with the cocontinuous functor $d_{\{0\}}\colon \Univ^R\to \Univ$, the claim follows.
\end{proof}

\begin{proposition}
	\label{prop:RetractsUCompact}
	Let $\I{U}$ be an internal class of $\BB$-categories and let $\I{C}$ be a $\IFilt_{\I{U}}$-cocomplete $\BB$-category. Then $\I{C}^{\cpt{\I{U}}}$ is closed under retracts in $\I{C}$, in the sense that the inclusion $\I{C}^{\cpt{\I{U}}}\into\IRet_{\I{C}}(\I{C}^{\cpt{\I{U}}})$ is an equivalence.
\end{proposition}
\begin{proof}
	It suffices to show that the retract of a $\I{U}$-compact object in $\I{C}$ is $\I{U}$-compact as well, which immediately follows from Lemma~\ref{lem:retractsCocontinuousFunctors}.
\end{proof}

We conclude this section with a characterisation of $\I{U}$-compact objects in presheaf $\BB$-categories. This will require the following lemma:
\begin{lemma}
	\label{lem:compactObjectRetract}
	Let $\I{U}$ be an internal class of $\BB$-categories and let $\I{C}\into\I{D}$ be a full inclusion of $\BB$-categories such that $\I{D}$ is $\IFilt_{\I{U}}$-cocomplete. Let $\I{J}$ be a $\I{U}$-filtered $\BB$-category, let $d\colon \I{J}\to\I{C}\into\I{D}$ be a diagram and suppose that $F=\colim d$ is a $\I{U}$-compact object in $\I{D}$. Then $F$ is contained in $\IRet_{\I{D}}(\I{C})$.
\end{lemma}
\begin{proof}
	The object $F$ being $\I{U}$-compact implies that the canonical map
	\begin{equation*}
		\phi\colon \colim \map{\I{D}}(F, d(-))\to \map{\I{D}}(F,F)
	\end{equation*}
	must be an equivalence.
	Thus the identity on $F$ gives rise to a global section
	\begin{equation*}
		\id_F \colon 1 \to \colim \map{\I{D}}(F, d(-)).
	\end{equation*}
	
	Let $ p \colon \I{P} \to \I{J} $ be the left fibration that is classified by the copresheaf $ \map{\I{D}}(F, d(-)) $.
	Since the map $ \I{P} \to \I{P}^\gp \simeq \colim \map{\I{D}}(F, d(-))$ is essentially surjective (by~\cite[Lemma~3.8.8]{MYoneda}), the map $\I{P}_0\to\I{P}^\gp$ is a cover in $\BB$~\cite[Corollary~3.9.5]{MYoneda}, so that we can find a cover $ s \colon A \onto 1 $ in $\BB$ and a local section $ x \colon A \to \I{P} $ such that the composite with $\I{P} \to \I{P}^\gp  $ recovers $ \pi_{A}^*\id_F $. Let $j=p(x)$. Then $x$ defines an object $f\colon A\to \I{P}\vert_j\simeq \map{\I{D}}(\pi_A^\ast F,d(j))$ that is carried to $\pi_A^\ast\id_F$ by the canonical morphism $\map{\I{D}}(\pi_A^\ast F,d(j))\to \map{\I{D}}(\pi_A^\ast F,\pi_A^\ast D)$.  In other words, composing $f\colon \pi_A^\ast F\to d(j)$ with the map $d(j) \to \pi_A^* F $ into the colimit yields $\pi_A^\ast F$. As this precisely means that $ \pi_A^* F $ is a retract of $ d(j) $, the claim follows.
\end{proof}
\begin{proposition}
	\label{prop:characterisationCompactObjectsPSh}
	Let $\I{U}$ be an internal class of $\BB$-categories that has the decomposition property, and let $\I{C}$ be a $\BB$-category. Then there is an equivalence
	\begin{equation*}
		\IPSh(\I{C})^{\cpt{\I{U}}}\simeq\IRet_{\IPSh(\I{C})}(\ISml^{\op(\I{U})}(\I{C}))
	\end{equation*}	
	of full subcategories in $\IPSh(\I{C})$. In particular, $\IPSh(\I{C})^{\cpt{\I{U}}}$ is small.
\end{proposition}
\begin{proof}
	Yoneda's lemma implies that every representable presheaf is $\I{U}$-compact. By combining this observation with Proposition~\ref{prop:RetractsUCompact} and Proposition~\ref{prop:UCompactUcocomplete}, one thus obtains an inclusion
	\begin{equation*}
		\IRet_{\IPSh(\I{C})}(\ISml^{\op(\I{U})}(\I{C}))\into\IPSh(\I{C})^{\cpt{\I{U}}}.
	\end{equation*}
	As for the converse inclusion, suppose that $F\colon\I{C}^{\op}\to\Univ$ is a $\I{U}$-compact presheaf. By Remark~\ref{rem:decompositionPropertyVariant}, there exists a $\I{U}$-filtered $\BB$-category $\I{J}$ and a diagram $d\colon \I{J}\to\op(\I{U})$ such that $\Over{\I{C}}{F}\simeq\colim d$ in $\ICat_{\BB}$. Proposition~\ref{prop:decompositionColimitsPSh} then shows that $F$ is the colimit of a $\I{J}$-indexed diagram in $\ISml^{\op(\I{U})}(\I{C})$.
	As $F$ is $\I{U}$-compact and $\I{J}$ is $\I{U}$-filtered, Lemma~\ref{lem:compactObjectRetract} shows that $F$ is locally a retract of an object in $\ISml^{\op(\I{U})}(\I{C})$. By Remark~\ref{rem:BCUComp} and~\cite[Remark~6.2.2]{MWColimits}, if $F\colon A\to \IPSh(\I{C})^{\cpt{\I{U}}}$ is an arbitrary object, we can replace $\BB$ by $\Over{\BB}{A}$ and carry out the same argument as above, which shows that $\IPSh(\I{C})^{\cpt{\I{U}}}$ is contained in $\IRet_{\IPSh(\I{C})}(\ISml^{\op(\I{U})}(\I{C}))$.
\end{proof}

\begin{corollary}
	\label{cor:inclusionCompactFinal}
	Let $\I{U}$ be an internal class of $\BB$-categories that has the decomposition property, and let $\I{J}$ be a weakly $\I{U}$-filtered $\BB$-category. Then the inclusion $\I{J}\into\IPSh(\I{J})^{\cpt{\I{U}}}$ is final.
\end{corollary}
\begin{proof}
	Proposition~\ref{prop:characterisationCompactObjectsPSh} shows that any $\I{U}$-compact presheaf $F\colon \I{J}^{\op}\to\Univ$ arises as a retract of some object $G\colon 1\to\ISml^{\op(\I{U})}(\I{J})$ after passing to a suitable cover of $1\in\BB$.
	Thus, the right fibration $\ISml^{\op(\I{U})}(\I{J})_{/F} \to\ISml^{\op(\I{U})}(\I{J})$ is \emph{locally} a retract of a representable right fibration, so that we must have $(\ISml^{\op(\I{U})}(\I{J})_{/F})^\gp \simeq 1$ as the latter property can be checked locally in $\BB$.
	By Remark~\ref{rem:BCUComp} and~\cite[Remark~6.2.2]{MWColimits}, we may replace $\BB$ with $\Over{\BB}{A}$ to arrive at the same conclusion for any object $F\colon A\to\IPSh(\I{J})^{\cpt{\I{U}}}$. 
	Using Quillen's theorem A~\cite[Corollary~4.4.8]{MYoneda}, this shows that the inclusion $\ISml^{\op(\I{U})}(\I{J}) \into\IPSh(\I{J})^{\cpt{\I{U}}}$ is final. Hence the claim follows from Proposition~\ref{prop:weaklyFilteredInd}.
\end{proof}

\begin{corollary}
	\label{cor:regularClassSoundWeaklySound}
	A left regular class $\I{U}$ is sound if and only if it is weakly sound.
\end{corollary}
\begin{proof}
	Using Remarks~\ref{rem:BCFilt} and~\ref{rem:BCSoundness}, it suffices to show that whenever $\I{U}$ is weakly sound, every weakly $\I{U}$-filtered $\BB$-category $\I{J}$ is $\I{U}$-filtered. Since Proposition~\ref{prop:decompositionColimitsCatB} implies that $\I{U}$ has the decomposition property, Corollary~\ref{cor:inclusionCompactFinal} shows that the inclusion $\I{J}\into\IPSh(\I{J})^{\cpt{\I{U}}}$ is final.  By Proposition~\ref{prop:UCompactUcocomplete}, the $\BB$-category $\IPSh(\I{J})^{\cpt{\I{U}}}$ is $\op(\I{U})$-cocomplete and therefore $\I{U}$-filtered since $\I{U}$ is by assumption weakly sound. Now if $\I{I}\in\I{U}(1)$ is chosen arbitrarily, the fact that $\IFun(\I{I},  \Univ)$ is cocomplete allows us to extend any diagram $d\colon \I{J}\to\IFun(\I{I},  \Univ)$ to a diagram $d^\prime\colon\IPSh(\I{J})^{\cpt{\I{U}}} \to \IFun(\I{I},  \Univ)$, using the universal property of presheaf $\BB$-categories. As the inclusion $\I{J}\into\IPSh(\I{J})^\cpt{\I{U}}$ is final, the limit functor $\lim_{\I{I}}\colon\IFun(\I{I},  \Univ) \to \Univ$ preserves the colimit of $d$ if and only if it preserves the colimit of $d^\prime$ (see~\cite[Proposition~4.6.1]{MWColimits}), which is indeed the case as $\IPSh(\I{C})^{\cpt{\I{U}}}$ is $\I{U}$-filtered. By replacing $\BB$ with $\Over{\BB}{A}$ and carrying out the same argument (which is possible by Remark~\ref{rem:BCUComp}), this already implies that $\I{J}$ must be $\I{U}$-filtered, as desired.
\end{proof}

\section{Cardinality in internal higher category theory}
\label{chap:cardinals}
In our treatment of accessibility and presentability for $\BB$-categories later in this paper, we will rely on the existence of an ample amount of doctrines that satisfy the decomposition property. Therefore, it will be crucial to know that there are sufficiently many \emph{(left) regular and sound doctrines} in any $\infty$-topos $\BB$. The main objective of this section is to construct such internal classes. More precisely, our approach is to first construct what we call the \emph{canonical bifiltration} of the $\BB$-category $\ICat_{\BB}$, i.e.\ a $2$-dimensional filtration by internal classes which can be regarded as a way to order $\BB$-categories by \emph{size}. The first dimension of this bifiltration is parametrised by \emph{cardinals}, and the second one by the poset of \emph{local classes} in $\BB$. The canonical bifiltration will be \emph{exhaustive}, so that every $\BB$-category can be assigned an upper bound in size, and it will be exclusively comprised of regular doctrines. We carry out the construction of this bifiltration in \S~\ref{sec:bifiltration}. In \S~\ref{sec:kappaSmall}, we discuss how one can extract a particularly well-behaved subfiltration from the canonical bifiltration that is still exhaustive and in which each member is \emph{sound}. The latter will be parametrised by a class of cardinals that satisfy a property which depends on the $\infty$-topos $\BB$ and that we refer to as \emph{$\BB$-regularity}. Finally, we discuss a particular member of the canonical bifiltration in \S~\ref{sec:finiteBCategories}, that of \emph{finite} $\BB$-categories.

\subsection{The canonical bifiltration of the $\BB$-category of $\BB$-categories}
\label{sec:bifiltration}
Recall that when $\KK$ is an arbitrary class of $\infty$-categories (i.e.\ a full subcategory of $\CatS$), we denote by $\ILConst_{\KK}$ the essential image of the canonical functor $\KK\into \CatS\to \ICat_{\BB}$ in which the second map is obtained as the transpose of $\const_{\BB}\colon \CatS\to \Cat(\BB)=\Gamma\ICat_{\BB}$. If $S$ is a local class of morphisms in $\BB$, we denote by $\ord{\KK,S}$ the internal class of $\BB$-categories that is generated by $\ILConst_{\KK}$ and $\Univ[S]$.

\begin{definition}
	\label{def:canonicalBifiltration}
	Let $\KK\subset\CatS$ be a class of $\infty$-categories and let $S$ be a local class of morphisms in $\BB$. We define the internal class  $\ICat_{\BB}^{\ord{\KK,S}}$ of  \emph{$\ord{\KK,S}$-small $\BB$-categories} as the left regularisation of $\ord{\KK,S}$. We denote its underlying $\infty$-category of global sections by $\Cat(\BB)^{\ord{\KK,S}}$.
\end{definition}
\begin{remark}
	\label{rem:BCCanonicalBifiltration}
	In the situation of Definition~\ref{def:canonicalBifiltration}, let us denote by $\pi_A^\ast S$ the class of those maps in $\Over{\BB}{A}$ whose underlying map in $\BB$ is contained in $S$. Since $(\pi_A)_!$ preserves small colimits and covers, this is still a local class, and one has a natural equivalence $\pi_A^\ast\Univ[S]\simeq\Univ[\pi_A^\ast S]$ of subuniverses. With this understood, Remark~\ref{rem:BCRegularisation} gives rise to a canonical equivalence $\pi_A^\ast\ICat_{\BB}^{\ord{\KK,S}}\simeq\ICat_{\Over{\BB}{A}}^{\ord{\KK,\pi_A^\ast S}}$ for every $A\in\BB$.
\end{remark}

By combining Proposition~\ref{prop:invarianceCocompletenessRegularisation} with~\cite[Example~5.4.11]{MWColimits}, one finds:
\begin{proposition}
	\label{prop:KSCocompleteness}
	A (large) $\BB$-category $\I{C}$ is $\ICat_{\BB}^{\ord{\KK,S}}$-complete precisely if
	\begin{enumerate}
		\item The $\infty$-category $\I{C}(A)$ admits limits indexed by objects in $\KK$, and for every map $s\colon B\to A$ the transition functor $s^* \colon \I{C}(A) \rightarrow \I{C}(B)$ preserves these limits;
		\item For every map $p \colon P \to A $ in $ S $, the functor $s^*$ admits a right adjoint $s_* \colon \I{C}(P) \to \I{C}(A)$ and for every cartesian square
		\begin{equation*}
			\begin{tikzcd}
				Q\arrow[r, "t"]\arrow[d, "q"] & P\arrow[d, "p"]\\
				B\arrow[r, "s"] & A
			\end{tikzcd}
		\end{equation*}
		in $\BB$ in which $p$ (and therefore $q$) are contained in $S$, the natural map $s^\ast p_\ast \to q_\ast t^\ast$ is an equivalence.
	\end{enumerate}
	Moreover,  a functor $f\colon \I{C}\to\I{D}$ of $\ICat_{\BB}^{\ord{\KK,S}}$-complete $\BB$-categories is $\ICat_{\BB}^{\ord{\KK,S}}$-continuous precisely if for all $A\in\BB$ the functor $f(A)$ preserves limits indexed by objects in $\KK$, and for all maps $p\colon P\to A$ in $S$ the natural morphism $f(A)p_\ast \to p_\ast f(P)$ is an equivalence.
	
	The dual statements about cocompleteness and cocontinuity (both understood with respect to the right regular class $\op(\ICat_{\BB}^{\ord{\KK,S}})$) hold as well.\qed
\end{proposition}

In the situation of Definition~\ref{def:canonicalBifiltration}, note that whenever $\KK$ is a doctrine (i.e.\ a small $\infty$-category) and $S$ is bounded (i.e.\ the subuniverse $\Univ[S]$ that corresponds to $S$ is small), Proposition~\ref{prop:regularisationDoctrine} implies that $\ICat_{\BB}^{\ord{\KK,S}}$ is a doctrine. Therefore, assigning to a pair $(\KK,S)$ the regular class $\ICat_{\BB}^{\ord{\KK,S}}$ defines a map of posets
\begin{equation*}
	\Sub_{\mathrm{full}}^{\mathrm{small}}(\CatS)\times\Sub_{\mathrm{full}}^{\mathrm{small}}(\Univ)\to \Sub_{\mathrm{full}}^{\mathrm{small}}(\ICat_{\BB})
\end{equation*}
that we refer to as the \emph{canonical bifiltration of $\ICat_{\BB}$}. 

\begin{remark}
	The canonical bifiltration is \emph{exhaustive}. In fact, if $\I{C}$ is an arbitrary $\Over{\BB}{A}$-category, we may find a small $\infty$-category $\JJ$ and a diagram $d\colon \JJ\to \Cat(\Over{\BB}{A})$ with colimit $\I{C}$ such that for all $j\in\JJ$ one has $d(j)\simeq\Delta^n\otimes B$ for some $n\geq 0$ and some $B\in\Over{\BB}{A}$. Note that $\Delta^n\otimes B$ can be identified with the $B$-indexed colimit of the constant diagram in $\ICat_{\Over{\BB}{A}}$ with value $\Delta^n$. Therefore, by choosing $\KK$ to be the doctrine of $\infty$-categories spanned by the single object $\JJ\in\CatS$ and choosing $S$ to be the bounded local class that is generated by the maps $B\to A$, we find that $\I{C}$ is $\ord{\KK, S}$-small.
\end{remark}

\begin{example}
	\label{ex:CatSLConstRegularisation}
	Since every $\BB$-category is a small colimit of objects of the form $\Delta^n\otimes A$ with $n\geq 0$ and $A\in\BB$, we deduce that the regularisation of $\ord{\CatS, \mathrm{all}}$ is $\ICat_{\BB}$ (where the local class $\mathrm{all}$ is the class of all morphisms in $\BB$).
\end{example}

\subsection{$\kappa$-small $\BB$-categories}
\label{sec:kappaSmall}

Let $\kappa$ be a cardinal. Recall from~\cite[\S~6.1.6]{htt} that a map $p\colon P\to A$ in $\BB$ is said to be \emph{relatively $\kappa$-compact} if for every $\kappa$-compact $B\in\BB$ and every map $s\colon B\to A$, the pullback $s^\ast P$ is $\kappa$-compact as well. We denote by $\cpt{\kappa}$ the local class of morphisms in $\BB$ that is generated by the relatively $\kappa$-compact morphisms, and we let $\Univ[\BB]^\kappa$ be the associated subuniverse. Explicitly, a map $p\colon P\to A$ is contained in $\cpt{\kappa}$ precisely if there is a cover $(s_i)_i\colon \bigsqcup_i A_i\onto A$ such that $s_i^\ast p$ is relatively $\kappa$-compact for each $i$.

Let us denote by $\CatS^\kappa$ the doctrine of $\kappa$-small $\infty$-categories. We may now define:
\begin{definition}
	\label{def:kappaSmall}
	A $\BB$-category is said to be \emph{$\kappa$-small} if it is $\ord{\CatS^\kappa,\cpt{\kappa}}$-small. We will use the notation $\ICat_{\BB}^\kappa=\ICat_{\BB}^{\ord{\CatS^\kappa, \cpt{\kappa}}}$ to denote the internal class of $\kappa$-small $\BB$-categories, and we denote its underlying $\infty$-category of global sections by $\Cat(\BB)^\kappa$.
\end{definition}
\begin{remark}
	\label{rem:BCKappaSmallCategories}
	Note that for general $A\in\BB$ there is no reason to expect an equivalence $\pi_A^\ast\Univ[\BB]^\kappa\simeq\Univ[\Over{\BB}{A}]^\kappa$. Therefore, we can also not expect to have an equivalence $\pi_A^\ast\ICat_{\BB}^\kappa\simeq\ICat_{\Over{\BB}{A}}^\kappa$. The situation improves, however, when $A$ is assumed to be $\kappa$-compact. In this case, the observation that an object in $\Over{\BB}{A}$ is $\kappa$-compact if and only if its underlying object in $\BB$ is $\kappa$-compact implies that a map in $\Over{\BB}{A}$ is relatively $\kappa$-compact if and only its underlying map in $\BB$ is relatively $\kappa$-compact, so that we obtain an equivalence $\pi_A^\ast\Univ[\BB]^\kappa\simeq\Univ[\Over{\BB}{A}]^\kappa$. By using Remark~\ref{rem:BCCanonicalBifiltration}, this equivalence in turn induces an equivalence $\pi_A^\ast\ICat_{\BB}^\kappa\simeq\ICat_{\Over{\BB}{A}}^\kappa$.
\end{remark}
The internal class $\ICat_{\BB}^{\kappa}$ is not very well-behaved for arbitrary cardinals $\kappa$. Therefore, we will restrict our attention to a certain class of cardinals that are in a sense \emph{adapted} to the $\infty$-topos $\BB$.

\begin{definition}
	\label{def:BRegularCardinal}
	We say that cardinal $\kappa$ is \emph{$\BB$-regular} if
	\begin{enumerate}
		\item $\kappa$ is regular and uncountable;
		\item $\BB$ is $\kappa$-accessible;
		\item the full subcategory $\BB^{\cpt{\kappa}}\into\BB$ of $\kappa$-compact objects in $\BB$ is closed under finite limits and subobjects in $\BB$.
	\end{enumerate}
\end{definition}

\begin{remark}
	\label{rem:SRegularCardinals}
	Every uncountable regular cardinal $\kappa$ is $\SS$-regular. In fact, condition~(2) is immediate, and $1\in\SS$ is certainly $\kappa$-compact. Moreover, if $P=A\times_C B$ is a pullback of $\kappa$-compact $\infty$-groupoids, descent implies $P\simeq \colim_{a\in A} P\vert_a$. 
	Since $\kappa$-compact $\infty$-groupoids are precisely those which are $\kappa$-small~\cite[Corollary~5.4.1.5]{htt} and since $\kappa$-compact objects in $\SS$ are stable under $\kappa$-small colimits, it suffices to show that $P\vert_a$ is $\kappa$-compact. We may therefore reduce to the case where $A\simeq 1$. By the same reasoning, we can assume $B\simeq 1$ as well. But then $P$ can be identified with a mapping $\infty$-groupoid of $C$, which is $\kappa$-small by again making use of~\cite[Corollary~5.4.1.5]{htt}. Finally, the identification of $\kappa$-compact $\infty$-groupoids with $\kappa$-small $\infty$-groupoids also shows that these are stable under subobjects. Hence condition~(3) is satisfied as well.
\end{remark}

\begin{remark}
	\label{rem:enoughBRegularCardinals}
	Note that there is an ample amount of $\BB$-regular cardinals, in the sense that if $\kappa^\prime$ is an arbitrary cardinal one can always find a larger cardinal $\kappa\geq \kappa^\prime$ that is $\BB$-regular. Indeed, by enlarging $\kappa^\prime$ if necessary one can always arrange for $\BB$ to be $\kappa^\prime$-accessible. Then for any (uncountable) $\kappa\gg\kappa^\prime$ (in the sense of~\cite[Definition~A.2.6.3]{htt}) for which $\BB^{\cpt{\kappa^\prime}}$ is $\kappa$-small, an object in $\BB$ is $\kappa$-compact if and only if the underlying presheaf on $\BB^{\cpt{\kappa^\prime}}$ takes values in the full subcategory $\SS^{\cpt{\kappa}}\into\SS$ of $\kappa$-compact $\infty$-groupoids~\cite[Lemma~5.4.7.5]{htt}. In combination with Remark~\ref{rem:SRegularCardinals}, this shows that $\kappa$ is $\BB$-regular. In particular, this argument shows that we can always find a $\BB$-regular $\kappa$ such that $\kappa\gg \kappa^\prime$.
\end{remark}

\begin{remark}
	\label{rem:BCkappaSmall}
	If $\kappa$ is a $\BB$-regular cardinal, then $\kappa$ is also $\Over{\BB}{A}$-regular for every \emph{$\kappa$-compact} object $A\in\BB$. In fact, since an object in $\Over{\BB}{A}$ is $\kappa$-compact if and only if its image along $(\pi_A)_!$ is $\kappa$-compact, every object in $\Over{\BB}{A}$ is a $\kappa$-filtered colimit of $\kappa$-compact objects, which shows that~(2) is satisfied. Condition~(3) follows from $A$ being $\kappa$-compact, together with the fact that $(\pi_A)_!$ preserves pullbacks (and consequently also subobjects).
\end{remark}

\begin{remark}
	\label{rem:BPresentationCompactObjects}
	For every $\BB$-regular cardinal $\kappa$, the $\infty$-topos $\BB$ admits a presentation by the full subcategory $\BB^{\cpt{\kappa}}\subset \BB$ of $\kappa$-compact objects, in the sense that its Yoneda extension $\PSh(\BB^{\cpt{\kappa}})\to \BB$ is a left exact and accessible Bousfield localisation~\cite[Proposition~6.1.5.2]{htt}. Moreover, the inclusion $\BB\into\PSh(\BB^{\cpt{\kappa}})$ commutes with $\kappa$-filtered colimits. In particular, the global sections functor $\Gamma\colon \BB\to\SS$ commutes with $\kappa$-filtered colimits, so that $\const_{\BB}$ restricts to a functor $\SS^{\cpt{\kappa}}\to \BB^{\cpt{\kappa}}$.
\end{remark}

The first main result in this section will be the following characterisation of $\kappa$-small $\BB$-categories when $\kappa$ is $\BB$-regular:
\begin{proposition}
	\label{prop:characterisationKappaSmallCategories}
	Let $\kappa$ be a $\BB$-regular cardinal, and let $\I{C}$ be a $\BB$-category. Then the following are equivalent:
	\begin{enumerate}
		\item $\I{C}$ is $\kappa$-small;
		\item $\I{C}$ is a $\kappa$-compact object in $\Cat(\BB)$;
		\item $\I{C}$ is contained in the smallest full subcategory of $\Cat(\BB)$ that is spanned by objects of the form $\Delta^n\otimes G$ for $n\geq 0$ and $G\in\BB^{\cpt{\kappa}}$ and that is closed under $\kappa$-small colimits;
		\item $\I{C}$ is a $\kappa$-compact object in $\Simp\BB$;
		\item $\I{C}_0$ and $\I{C}_1$ are $\kappa$-compact objects in $\BB$.
	\end{enumerate}
\end{proposition}
\begin{remark}
	\label{rem:characterisationKappaSmallCategories}
	On account of Remark~\ref{rem:BCKappaSmallCategories}, Proposition~\ref{prop:characterisationKappaSmallCategories} implies that for every $\kappa$-compact object $A\in\BB$, we can identify $\ICat_{\BB}^\kappa(A)$ with the full subcategory of $\kappa$-compact objects in $\Cat(\Over{\BB}{A})$.
\end{remark}
The proof of Proposition~\ref{prop:characterisationKappaSmallCategories} requires a few preparations. We begin by establishing that the class of relatively $\kappa$-compact maps in $\BB$ is already local.
\begin{lemma}
	\label{lem:stabilityRelativelyKappaSmallSum}
	Let $\kappa$ be a $\BB$-regular cardinal and let $I$ be a small set. For every $i\in I$, let $P_i\to A_i$ be a relatively $\kappa$-compact map in $\BB$. Then $\bigsqcup_i P_i\to \bigsqcup_i A_i$ is relatively $\kappa$-compact.
\end{lemma}
\begin{proof}
	Let $G$ be $\kappa$-compact, and let $s\colon G\to \bigsqcup_i A_i$ be a map. Write $I=\colim_j I_j$ as a $\kappa$-filtered union of its $\kappa$-small subsets, so that one obtains equivalences $\bigsqcup_i A_i\simeq\colim_j\bigsqcup_{i\in I_j} A_i$ and $\bigsqcup_i P_i\simeq \colim_j \bigsqcup_{i\in I_j} P_i$. As $G$ is $\kappa$-compact, there is some $j$ such that $s$ factors through the inclusion $\bigsqcup_{i\in I_{j}}A_i\into\bigsqcup_i A_i$. By descent, we obtain a pullback diagram
	\begin{equation*}
		\begin{tikzcd}
			\bigsqcup_{i\in I_j} P_i\arrow[d]\arrow[r] & \bigsqcup_i P_i\arrow[d]\\
			\bigsqcup_{i\in I_j} A_i\arrow[r] & \bigsqcup_i A_i,
		\end{tikzcd}
	\end{equation*}
	which implies that the pullback of $\bigsqcup_i P_i\to\bigsqcup_i A_i$ to $G$ is equivalent to the pullback of $\bigsqcup_{i\in I_j}P_i\to\bigsqcup_{i\in I_j} A_i$ to $G$. By again using descent, this pullback can be identified with the coproduct $\bigsqcup_{i\in I_j} P_i\times_{A_i} G_i$, where $G_i=G\times_{\bigsqcup_i A_i} A_i$. As $G_i$ is a subobject of $G$ and therefore $\kappa$-compact, the fibre product $P_i\times_{A_i} G_i$ is $\kappa$-compact as well. Since $I_j$ is $\kappa$-small, we conclude that also $\bigsqcup_{i\in I_j} P_i\times_{A_i} G_i$ is $\kappa$-compact, as desired.
\end{proof}

\begin{proposition}
	\label{prop:relativelyKappaCompactLocal}
	Let $\kappa$ be a $\BB$-regular cardinal. Then every object in $\Univ[\BB]^\kappa$ is already relatively $\kappa$-compact. In other words, the class of relatively $\kappa$-compact maps in $\BB$ is local.
\end{proposition}
\begin{proof}
	Let $P\to A$ be an object in $\Univ[\BB]^\kappa(A)$. By definition, there is a cover $(s_i)\colon \bigsqcup_{i\in I} A_i\onto A$ such that $s_i^\ast P\to A_i$ is relatively $\kappa$-compact. By Lemma~\ref{lem:stabilityRelativelyKappaSmallSum}, the map $\bigsqcup_i P_i\to \bigsqcup_i A_i$ is relatively $\kappa$-compact. The result therefore follows once we show that relatively $\kappa$-compact maps are stable under $\Delta^{\op}$-indexed colimits in $\Fun(\Delta^1,\BB)$. By~\cite[Lemma~6.1.6.6]{htt} they are stable under pushouts, so we only need to consider the case of small coproducts, which again follows from Lemma~\ref{lem:stabilityRelativelyKappaSmallSum}.
\end{proof}
\begin{remark}
	In~\cite[Proposition~6.1.6.7]{htt}, Lurie shows that the class of relatively $\kappa$-compact maps in $\BB$ is local already when $\BB$ is $\kappa$-accessible and $\BB^{\cpt{\kappa}}$ is stable under finite limits in $\BB$. However, we failed to understand how Lurie derives this result without the additional assumption that $\BB^{\cpt{\kappa}}$ is also stable under subobjects in $\BB$. Therefore, we decided to reiterate Lurie's proof with this added assumption.
\end{remark}

Next, we need to establish that every $\BB$-regular cardinal is also $\Simp\BB$-regular. This will be a consequence of the following characterisation of $\kappa$-compact simplicial objects in $\BB$:
\begin{proposition}
	\label{prop:SimpBAccessibleCompactObjects}
	If $\kappa$ is a $\BB$-regular cardinal, then the $\infty$-topos $\Simp\BB$ is $\kappa$-accessible, and if $C$ is a simplicial object in $\BB$, the following are equivalent:
	\begin{enumerate}
		\item $C$ is $\kappa$-compact;
		\item $C_n\in \BB^{\cpt{\kappa}}$ for all $n\geq 0$;
		\item $C$ is contained in the smallest subcategory of $\Simp\BB$ that is spanned by objects of the form $\Delta^n\otimes G$ for $n\geq 0$ and $G\in\BB^{\cpt{\kappa}}$ and that is closed under $\kappa$-small colimits;
	\end{enumerate}
\end{proposition}
\begin{proof}
	Remark~\ref{rem:BPresentationCompactObjects} implies that the inclusion $\Simp\BB\into\PSh(\Delta\times\BB^{\cpt{\kappa}})$ commutes with $\kappa$-filtered colimits, which immediately implies that $\Simp\BB$ is $\kappa$-accessible. Moreover, since $\Delta$ is a $\kappa$-small $\infty$-category, every simplicial object in $\BB$ that is level-wise $\kappa$-compact is also $\kappa$-compact in $\Simp\BB$~\cite[Proposition~5.3.4.13]{htt}, hence~(2) implies~(1).  If $C$ satisfies~(3), the fact that for every $k\geq 0$ the functor $(-)_k$ commutes with small colimits implies that $C_k$ is contained in the smallest full subcategory of $\BB$ that contains all objects of the form $\Delta^n_k\times G$ for $G\in\BB^{\cpt{\kappa}}$ and $n\geq 0$ and that is closed under $\kappa$-small colimits. Since $\Delta^n_k$ is a finite set, this implies that $C_k$ is $\kappa$-compact, hence~(2) follows. Finally, suppose that $C$ is $\kappa$-compact. We may write $C$ as a small colimit of objects of the form $\Delta^n\otimes G$ for $n\geq 0$ and $G\in\BB^{\cpt{\kappa}}$ and therefore by~\cite[Corollary~4.2.3.10]{htt} as a $\kappa$-filtered colimit $C\simeq\colim_i C^i$ where each $C^i$ is a $\kappa$-small colimits of objects of the form $\Delta^n\otimes G$. As $C$ is $\kappa$-compact, there is some $i_0$ such that the identity on $C$ factors through $C^{i_0}\to\I{C}$. In other words, $C$ is a retract of $C^{i_0}$. As retracts are countable and therefore a fortiori $\kappa$-small colimits,~(3) follows.
\end{proof}

\begin{corollary}
	\label{cor:BRegularImpliesSimpBRegular}
	If $\kappa$ is a $\BB$-regular cardinal, then $\kappa$ is $\Simp\BB$-regular as well. Moreover, a map in $\Simp\BB$ is relatively $\kappa$-compact if and only if it is level-wise given by a relatively $\kappa$-compact morphism in $\BB$.\qed
\end{corollary}

\begin{lemma}
	\label{lem:stabilityKappaCompactInternalHom}
	If $\kappa$ is a $\BB$-regular cardinal and if $C$ is a $\kappa$-compact simplicial object in $\BB$, then $C^{K}$ is $\kappa$-compact for every $\omega$-compact simplicial $\infty$-groupoid $K$.
\end{lemma}
\begin{proof}
	As $\kappa$-compact objects in $\Simp\BB$ are stable under retracts and as every $\omega$-compact simplicial $\infty$-groupoid is a retract of a finite colimit of $n$-simplices, we may assume without loss of generality that $K$ is a finite colimit of $n$-simplices. Therefore $C^K$ is a finite limit of objects of the form $C^{\Delta^n}$, so that Corollary~\ref{cor:BRegularImpliesSimpBRegular} implies that we may reduce to the case $K=\Delta^n$. Now on account of the identity $(C^{\Delta^n})_k\simeq (C^{\Delta^n\times\Delta^k})_0$ and by using the fact that $\Delta^n\times\Delta^k$ is again $\omega$-compact, we can identify $(C^{\Delta^n})_k$ as a finite limit of objects of the form $(C^{\Delta^l})_0\simeq C_l$, which shows that $(C^{\Delta^n})_k$ is $\kappa$-compact. By Proposition~\ref{prop:SimpBAccessibleCompactObjects}, one concludes that $C^{\Delta^n}$ is $\kappa$-compact in $\Simp\BB$.
\end{proof}

\begin{lemma}
	\label{lem:SimplicesInternallyCompact}
	For every $\omega$-compact simplicial $\infty$-groupoid $K$, the functor $(-)^K\colon\Simp\BB\to\Simp\BB$ commutes with filtered colimits.
\end{lemma}
\begin{proof}
	As every $\omega$-compact simplicial $\infty$-groupoid $K$ is a retract of a finite colimit of $n$-simplices, we may assume without loss of generality $K=\Delta^n$. As it suffices to show that$(-)^{\Delta^n}_k$ commutes with filtered colimits for all $k\geq 0$, the same argumentation as in the proof of Lemma~\ref{lem:stabilityKappaCompactInternalHom} shows that we may reduce to showing that $(-)^{\Delta^n}_0$ commutes with filtered colimits. On account of the equivalence $(-)^{\Delta^n}_0\simeq (-)_n$, this is immediate.
\end{proof}

\begin{lemma}
	\label{lem:kappaCompactFibrantReplacement}
	Let $\kappa$ be a $\BB$-regular cardinal, let $C$ be a $\kappa$-compact simplicial object in $\Simp\BB$ and let $C\to L(C)$ be the unit of the adjunction $(L\dashv i)\colon \Cat(\BB)\leftrightarrows\Simp\BB$. Then $L(C)$ is $\kappa$-compact as well.
\end{lemma}
\begin{proof}
	We will make use of the $\infty$-categorical version of the small object argument as developed in~\cite[\S~2.3]{Anel2020a}. For the convenience of the reader, we briefly explain the setup, at least in the special case that is relevant for this proof. Suppose that $S$ is a finite set of maps in $\Simp\SS$ such that for every map $s\colon K\to L$ in $S$ the functors $(-)^K$ and $(-)^L$ commute with filtered colimits in $\Simp\BB$. Let $(\LL,\RR)$ be the factorisation system in $\Simp\BB$ that is internally generated by the set $S$. To any object $C\in \Simp\BB$, we can now assign a sequence
	\begin{equation*}
		\mathbb N\to \Simp\BB,\quad k\mapsto C(k)
	\end{equation*}
	by setting $C(0)=C$ and by recursively defining a map $C(k)\to C(k+1)$ via the pushout
	\begin{equation*}
		\begin{tikzcd}
			\bigsqcup_{s\colon K\to L} L\otimes C(k)^L\sqcup_{K\otimes C(k)^L} K\otimes C(k)^K\arrow[d]\arrow[r] & C(k)\arrow[d]\\
			\bigsqcup_{s\colon K\to L} L\otimes C(k)^K\arrow[r] & C(k+1) 
		\end{tikzcd}
	\end{equation*}
	in which the coproduct ranges over all maps $s\colon K\to L$ in $S$. Then~\cite[Theorem~2.3.4]{Anel2020a} shows that the object $\colim_k C(k)$ is internally local with respect to the maps in $S$, i.e.\ contained in $\Over{\RR}{1}$, and that furthermore the map $C\to \colim_k C(k)$ is contained in $\LL$, so that it is equivalent to the unit of the adjunction $\Over{\RR}{1}\leftrightarrows\Simp\BB$ evaluated at $C\in\Simp\BB$.
	
	Now if we let $S$ be the set $\{E^1\to 1, I^2\into \Delta^2\}$, Lemma~\ref{lem:SimplicesInternallyCompact} shows that we are in the above situation. Consequently, if $C$ is a $\kappa$-compact object in $\Simp\BB$, the $\BB$-category $L(C)$ can be computed as a countable colimit of the objects $C(k)$ as constructed above. Hence it suffices to show that each $C(k)$ is $\kappa$-compact, which easily follows from $\kappa$ being $\Simp\BB$-regular (Corollary~\ref{cor:BRegularImpliesSimpBRegular}) and Lemma~\ref{lem:stabilityKappaCompactInternalHom}.
\end{proof}

\begin{proof}[{Proof of Proposition~\ref{prop:characterisationKappaSmallCategories}}]
	We first show that~(2)--(5) are equivalent.
	By combining Proposition~\ref{prop:SimpBAccessibleCompactObjects} with the Segal conditions, one finds that~(4) and~(5) are equivalent. Moreover, since $\Cat(\BB)$ is an $\omega$-accessible localisation of $\Simp\BB$, the localisation functor preserves $\kappa$-compact objects, which shows that~(4) implies~(2). Suppose now that $\I{C}$ is a $\kappa$-compact object in $\Cat(\BB)$. As in the proof of Proposition~\ref{prop:SimpBAccessibleCompactObjects}, we can find a $\kappa$-filtered $\infty$-category $\JJ$ such that $\I{C}\simeq \colim_{j\in \JJ} \I{C}^{j}$ where each $\I{C}^j$ is a $\kappa$-small colimit of objects of the form $\Delta^n\otimes G$, where $n\geq 0$ and $G\in\BB^{\cpt{\kappa}}$. Hence $\I{C}$ is a retract of some $\I{C}^j$, so that~(3) holds. Lastly, since we can compute any small colimit in $\Cat(\BB)$ by first taking the colimit of the underlying diagram in $\Simp\BB$ and then applying the reflector $L\colon \Simp\BB\to\Cat(\BB)$, Lemma~\ref{lem:kappaCompactFibrantReplacement} implies that every $\kappa$-small colimit in $\Cat(\BB)$ of objects of the form $\Delta^n\otimes G$ with $n\geq 0$ and $G\in \BB^{\cpt{\kappa}}$ is also $\kappa$-compact in $\Simp\BB$. Thus~(3) implies~(4).
	
	Finally, since $\ICat_{\BB}^\kappa$ is closed under both $\ILConst_{\CatS^\kappa}$- and $\Univ[\BB]^\kappa$-colimits and since $\Delta^n\otimes G$ can be regarded as the colimit of the constant $G$-indexed diagram with value $\Delta^n$, it is clear that~(3) implies~(1). To show the converse, let $\I{V}$ be the internal class that is spanned by those $\Over{\BB}{A}$-categories (for $A\in\BB^{\cpt{\kappa}}$) that satisfy the $\Over{\BB}{A}$-categorical analogue of~(3). Note that if $\bigsqcup_i A_i\onto 1$ is a cover by $\kappa$-compact objects and if $\I{D}$ is a $\BB$-category such that $\pi_{A_i}^\ast\I{D}$ satisfies the $\Over{\BB}{{A_i}}$-categorical analogue of condition~(3), then $\I{D}$ satisfies~(3): in fact, since we have already established that~(3) and~(4) are equivalent, this is a consequence of Proposition~\ref{prop:relativelyKappaCompactLocal} and Corollary~\ref{cor:BRegularImpliesSimpBRegular}. As a consequence, for every $\kappa$-compact object $A\in\BB$, we can identify $\I{V}(A)$ with the class of $\Over{\BB}{A}$-categories that satisfy the $\Over{\BB}{A}$-categorical version of condition~(3). As $\I{V}$ clearly contains both $\ILConst_{\CatS^\kappa}$ and $\Univ[\BB]^\kappa$, the proof will be complete once we show that $\I{V}$ is closed under both $\ILConst_{\CatS^\kappa}$- and $\Univ[\BB]^\kappa$-colimits. By our description of $\I{V}(A)$ for every $\kappa$-compact $A\in\BB$ and the fact that colimits can be computed locally~\cite[Remark~4.1.8]{MWColimits}, this is clear for the first case. To show the second case, we need to verify that for every relatively $\kappa$-compact map $p\colon P\to A$, the functor $p_!\colon \Cat(\Over{\BB}{P})\to\Cat(\Over{\BB}{A})$ restricts to a map $\I{V}(P)\to\I{V}(A)$. Using again that the class of relatively $\kappa$-compact maps is local, it is enough to consider the case where $A$ (and therefore also $P$) is $\kappa$-compact. To show the claim, we may again use the explicit description of $\I{V}(A)$ and $\I{V}(P)$ to deduce that it suffices to verify that $p_!$ carries $\kappa$-small colimits of objects in $\Cat(\Over{\BB}{P})$ of the form $\Delta^n\otimes Q$ (with $Q\to P$ relatively $\kappa$-compact) to $\kappa$-small colimits of objects in $\Cat(\Over{\BB}{A})$ of the form $\Delta^n\otimes Q$ (with $Q\to A$ $\kappa$-compact). Since $p_!$ preserves small colimits and acts by postcomposition with $p$, this follows from the fact that relatively $\kappa$-compact maps are closed under composition.
\end{proof}

\begin{corollary}
	For every $\BB$-regular cardinal $\kappa$, the internal class $\ICat_{\BB}^\kappa$ is a doctrine.
\end{corollary}
\begin{proof}
	As $\BB$ is generated by $\BB^\kappa$, Remark~\ref{rem:BCKappaSmallCategories} implies that we only need to show that the collection of $\kappa$-small $\BB$-categories is small, which is an immediate consequence of~(2) in Proposition~\ref{prop:characterisationKappaSmallCategories}
\end{proof}

By construction, the internal class $\ICat_{\BB}^\kappa$ is regular for every cardinal $\kappa$. We conclude this section by proving that whenever $\kappa$ is $\BB$-regular, the doctrine $\ICat_{\BB}^\kappa$ is sound.
\begin{lemma}
	\label{lem:externallyFilteredImpliesInternallyFiltered}
	Let $\kappa$ be a $\BB$-regular cardinal, and let $\JJ$ be a $\kappa$-filtered $\infty$-category. Then $\JJ$ is $\ICat_{\BB}^\kappa$-filtered when viewed as a constant $\BB$-category.
\end{lemma}
\begin{proof}
	By Proposition~\ref{prop:characterisationFilteredCategoryCocones}, we need to show that the inclusion $\IFun(\JJ,\Univ)\into\IFun(\JJ^\triangleright,\Univ)$ is $\ICat_{\BB}^\kappa$-continuous. Since $\JJ$ is $\kappa$-filtered, the inclusion section-wise  preserves $\kappa$-small limits. It therefore suffices to show that it is $\Univ^\kappa$-continuous. This amounts to showing that for every $A\in\BB$ and every $G\in\Univ^\kappa(A)$ the geometric morphism $\Over{\BB}{G}\to\Over{\BB}{A}$ commutes with $\JJ$-indexed colimits. As the preservation of colimits is a local condition~\cite[Remark~4.2.1]{MWColimits} and as $\BB$ is generated by the $\kappa$-compact objects in $\BB$, we may assume that $A$ is $\kappa$-compact. In light of Remark~\ref{rem:BCkappaSmall}, we may thus replace $\BB$ with $\Over{\BB}{A}$ and can therefore reduce to the case $A\simeq 1$.  As $\kappa$ is $\BB$-regular, the collection of $\kappa$-compact objects in $\BB$ is stable under finite limits. Therefore, for every $H\in \BB^{\cpt{\kappa}}$ the functor $\map{\BB}(G\times H,-)$ preserves $\JJ$-filtered colimits. By Yoneda's lemma, this implies that the functor $\Hom_{\BB}(G,-)$ also preserves $\JJ$-filtered colimits. On account of the pullback square
	\begin{equation*}
		\begin{tikzcd}
			(\pi_G)_\ast\arrow[r]\arrow[d] & \Hom_{\BB}(G,(\pi_G)_!(-))\arrow[d]\\
			\diag(1)\arrow[r, "\id_G"] & \diag(\Hom_{\BB}(G,G))
		\end{tikzcd}
	\end{equation*}
	in $\Fun(\Over{\BB}{G},\BB)$ and the fact that the cospan in the lower right corner consists of functors which preserve $\JJ$-indexed colimits, the claim follows from the fact that $\JJ$-indexed colimits commute with finite limits.
\end{proof}

\begin{lemma}
	\label{lem:CocompleteImpliesFinallyConstant}
	Let $\kappa$ be a $\BB$-regular cardinal, and let $\I{J}$ be a $\ICat_{\BB}^\kappa$-cocomplete $\BB$-category. Then the canonical functor $\Gamma\I{J}\to\I{J}$ that is obtained from the counit of the adjunction $\const_{\BB}\dashv \Gamma$ is final.
\end{lemma}
\begin{proof}
	For every $G\in\BB^{\cpt{\kappa}}$, the functor $\I{J}(1)\to\I{J}(G)$ admits a left adjoint and is therefore in particular final. In other words, if $i\colon \BB\into\PSh(\BB^{\cpt{\kappa}})$ denotes the inclusion, then the functor $\epsilon\colon\Gamma_{\PSh(\BB^{\cpt{\kappa}})}i\I{J}\to i\I{J}$ is section-wise final. But since the local sections functor $\ev_G\colon \PSh(\BB^{\cpt{\kappa}})\to\SS$ defines an algebraic morphism of $\infty$-topoi and since every algebraic morphism preserves both final functors and right fibrations, applying $\ev_G$ to any factorisation of $\epsilon$ in $\Cat(\PSh(\BB^{\cpt{\kappa}}))$ into a final functor and a right fibration yields a factorisation of $\epsilon(G)$ into a final functor and a right fibration in $\CatS$. Consequently, the map $\epsilon$ must already be final. As we recover the map $\Gamma\I{J}\to\I{J}$ by applying the algebraic morphism $L\colon\PSh(\BB^{\cpt{\kappa}})\to\BB$ to $\epsilon$, the claim follows.
\end{proof}

\begin{proposition}
	\label{prop:kappaSmallSound}
	If $\kappa$ is a $\BB$-regular cardinal, then $\ICat_{\BB}^\kappa$ is sound.
\end{proposition}
\begin{proof}
	On account of Corollary~\ref{cor:regularClassSoundWeaklySound}, it suffices to show that $\ICat_{\BB}^\kappa$ is weakly sound. Together with the fact that $\BB$ is generated by its $\kappa$-compact objects and Remark~\ref{rem:BCKappaSmallCategories}, it is therefore enough to prove that every $\ICat_{\BB}^\kappa$-cocomplete $\BB$-category $\I{J}$ is $\ICat_{\BB}^\kappa$-filtered. By Lemma~\ref{lem:CocompleteImpliesFinallyConstant} and Remark~\ref{rem:FiltUColimitClass}, we can furthermore assume that $\I{J}$ is the constant $\BB$-category associated with an $\infty$-category that admits $\kappa$-small colimits and that is therefore $\kappa$-filtered~\cite[Proposition~5.3.3.3]{htt}. As a consequence, the result follows from Lemma~\ref{lem:externallyFilteredImpliesInternallyFiltered}.
\end{proof}

\begin{remark}
	\label{rem:regularClassesExhaustive}
	As a consequence of Proposition~\ref{prop:kappaSmallSound}, if $\I{C}$ is an arbitrary $\BB$-category, there is always a regular and sound doctrine $\I{U}$ such that $\I{C}\in\I{U}(1)$. In fact, we only need to choose a $\BB$-regular cardinal $\kappa$ such that $\I{C}$ is $\kappa$-compact (and therefore $\kappa$-small by Proposition~\ref{prop:characterisationKappaSmallCategories}) and set $\I{U}=\ICat_{\BB}^\kappa$.
	More generally, if $\I{V}$ is a doctrine, we can find a $\BB$-regular cardinal $\kappa$ such that $\I{V}_0$ is $\kappa$-compact and such that the tautological object $\tau\colon \I{V}_0\to\I{V}$ corresponds to a $\kappa$-small $\Over{\BB}{\I{V}_0}$-category. As every object of $\I{V}$ (in arbitrary context $A\in\BB$) arises as a pullback of $\tau$, this implies that $\I{V}$ is contained in $\ICat_{\BB}^\kappa$.
\end{remark}

\begin{corollary}
	\label{cor:kappaCompactObjectsUniverse}
	For every $\BB$-regular cardinal $\kappa$, there is an equivalence
	\begin{equation*}
		\Univ[\BB]^{\cpt{\ICat_{\BB}^\kappa}}\simeq\Univ[\BB]^\kappa
	\end{equation*} 
	of full subcategories in $\Univ$.
\end{corollary}
\begin{proof}
	Since $\ICat_{\BB}^\kappa$ is a sound doctrine by Proposition~\ref{prop:kappaSmallSound}, it has the decomposition property. We may therefore apply Proposition~\ref{prop:characterisationCompactObjectsPSh} to deduce an equivalence
	\begin{equation*}
		\Univ[\BB]^{\cpt{\ICat_{\BB}^\kappa}}\simeq\IRet_{\Univ}(\ISml^{\ICat_{\BB}^\kappa}(1)).
	\end{equation*} 
	Note that if $\I{G}\colon 1\to \ISml^{\ICat_{\BB}^\kappa}(1)$ is an arbitrary object, there is a cover $\bigsqcup_i A_i\onto 1$ of $\BB$ (without loss of generality by $\kappa$-compact objects) and for each $i$ a $\kappa$-small $\Over{\BB}{A_i}$-category $\I{J}$ such that $\pi_{A_i}^\ast\I{G}\simeq\I{J}^\gp$. Since $\kappa$ is by definition uncountable, we thus find that $\pi_{A_i}^\ast\I{G}$ arises as a $\kappa$-small colimit of $\kappa$-compact objects in $\Over{\BB}{A_i}$ (using the characterisation of $\kappa$-small $\Over{\BB}{A_i}$-categories in Proposition~\ref{prop:characterisationKappaSmallCategories}) and is therefore itself $\kappa$-compact. Using Proposition~\ref{prop:relativelyKappaCompactLocal}, this implies that $\I{G}$ is $\kappa$-compact itself. By~\cite[Remark~6.2.2]{MWColimits}, the same argument can be carried out for every object in $\ISml^{\ICat_{\BB}^\kappa}(1)$ in context $A\in\BB^{\cpt{\kappa}}$, and since the collection of $\kappa$-compact objects in $\BB$ generate $\BB$ under small colimits, this implies that we have an inclusion
	\begin{equation*}
		\ISml^{\ICat_{\BB}^\kappa}(1)\into\Univ[\BB]^\kappa.
	\end{equation*}
	Using again that $\kappa$ is uncountable, the collection of $\kappa$-compact objects in $\BB$ is closed under retracts, and as the same is true for the class of $\kappa$-compact objects in $\Over{\BB}{A}$ for every $A\in\BB^{\cpt{\kappa}}$, we find that $\Univ[\BB]^\kappa$ is closed under retracts in $\Univ$, so that we obtain an inclusion
	\begin{equation*}
		\Univ[\BB]^{\cpt{\ICat_{\BB}^\kappa}}\into\Univ[\BB]^\kappa.
	\end{equation*} 
	Conversely, it is clear that whenever $\I{G}$ is a $\kappa$-small $\BB$-groupoid, the associated object $\I{G}\colon 1\to \Univ$ is contained in $\ISml^{\ICat_{\BB}^\kappa}(1)$. Again, the same is true for every $\kappa$-small $\Over{\BB}{A}$-groupoid whenever $A$ is $\kappa$-compact. Hence we obtain an inclusion
	\begin{equation*}
		\Univ[\BB]^\kappa\into\Univ[\BB]^{\cpt{\ICat_{\BB}^\kappa}},
	\end{equation*}
	which finishes the proof.
\end{proof}

\subsection{Finite $\BB$-categories}
\label{sec:finiteBCategories}
In this section we will discuss another important example of a regular and sound doctrine.
Recall that a quasicategory $ \CC $ is called \emph{finite} if there is a finite simplicial set and a Joyal equivalence $ K \rightarrow \CC $. 
This is equivalent to $ \CC $ being contained in the smallest subcategory of $ \CatS $ that contains $ \varnothing$, $\Delta^0$ and $ \Delta^1 $ and is closed under pushouts (see \cite[Proposition 2.5]{volpe2022verdier}). We denote the associated doctrine of $\infty$-categories by $\Fin_{\SS}$.
Let us denote by $\Eq$ the local class of equivalences in $\BB$. We may now define:
\begin{definition}
	\label{def:finiteBCategories}
	A $\BB$-category is said to be \emph{finite} if it is $(\Fin_{\SS},\Eq)$-small, and we shall denote by  $\IFin_\BB=\ICat_{\BB}^{\ord{\Fin_{\SS},\Eq}}$ the associated regular doctrine of finite $\BB$-categories. We will denote by $\Fin(\BB)$ the underlying $\infty$-category of global sections.
	We say that a $ \BB $-category $ \I{I} $ is \emph{filtered} if it is $ \IFin_\BB$-filtered.
	We will say that a $ \BB $-category \emph{has finite (co)limits} if it is $ \IFin_\BB $-(co)complete, and a functor \emph{preserves finite (co)limits} if it is $ \IFin_\BB $-(co)continuous. Dually, we say that a $\BB$-category \emph{has filtered colimits} if it is $\IFilt$-cocomplete, and a functor \emph{preserves filtered colimits} if it is $\IFilt$-cocontinuous. If $\I{C}$ is a $\BB$-category that has filtered colimits, an object $c\colon A\to\I{C}$ is said to be \emph{compact} if it is $\IFilt_{\IFin_\BB}$-compact, and we denote the full subcategory of compact objects in $\I{C}$ by $\I{C}^\compact$.
\end{definition}

\begin{remark}
	\label{rem:BCFinB}
	By Remark~\ref{rem:BCCanonicalBifiltration} and the evident fact that $\pi_A^\ast\Eq=\Eq$ as local classes in $\Over{\BB}{A}$, there is a canonical equivalence $\pi_A^\ast\IFin_\BB\simeq\IFin_{\Over{\BB}{A}}$ for all $A\in\BB$.
\end{remark}

\begin{remark}
	\label{rem:filteredContractible}
	Every filtered $\BB$-category $\I{J}$ satisfies $\I{J}^\gp\simeq 1$. In fact, by Corollary~\ref{cor:inclusionWFiltFilt} it is weakly filtered, thus in particular the unique functor $I \rightarrow \IFun(\varnothing, \I{I}) \simeq 1$ is final.
\end{remark}

\begin{proposition}
	\label{prop:FinBDoctrine}
	There is an equivalence $\IFin_{\BB}\simeq\ILConst_{\Fin_{\SS}}$ of internal classes. In other words, a finite $\BB$-category is simply a locally constant sheaf of finite $\infty$-categories.
\end{proposition}
\begin{proof}
	Since $\Univ[\Eq]\simeq 1_{\Univ}$ as full subcategories, we can describe $\IFin_\BB$ as the regularisation of $\ILConst_{\Fin_{\SS}}$.
	But since $\CatS$ is compactly generated, we may apply Corollary~\ref{cor:LConstPO} and conclude that $\ILConst_{\Fin_{\SS}}$ is already closed under $\ILConst_{\Fin_{\SS}}$-colimits in $\ICat_{\BB}$. Hence the claim follows.
\end{proof}

By Proposition~\ref{prop:KSCocompleteness}, finite limits and preservation of finite limits can be checked section-wise:
\begin{proposition}
	\label{prop:critfinlim}
	Let $ \I{C}$ be a $ \BB $-category. Then
	\label{prop:criterionforfinitelimits}
	\begin{enumerate}
		\item $ \I{C} $ has finite limits if and only if $ \I{C}(A)$ has finite limits for every $ A \in \BB$ and for every $ s \colon B \to A$ the functor $ s^* \colon \I{C}(A) \rightarrow \I{C}(B) $ preserves finite limits.
		\item  A functor $ f \colon \I{C} \rightarrow \I{D} $ between $\BB$-categories that have finite limits preserves such limits if and only if $ f(A) \colon \I{C}(A) \rightarrow \I{D}(A)$ preserves finite limits for every $ A \in \BB $.
	\end{enumerate}
	The dual statements about finite colimits hold as well.\qed
\end{proposition}

One can construct an ample amount of filtered $\BB$-categories from \emph{presheaves} of filtered $\infty$-categories:
\begin{proposition}
	\label{prop:sectionwisefilteredgivesfiltered}
	Say that $ \BB $ is given as a left exact accessible localisation $ L \colon \PSh(\CC) \rightarrow \BB $ where $ \CC $ is a small $ \infty $-category. Let $ \I{J} $ be any $ \PSh(\CC) $-category such $ \I{J}(c) $ is filtered for every $ c \in \CC$.
	Then $ L \I{J} $ is a filtered $ \BB $-category. 
\end{proposition}
\begin{proof}
	Let $i\colon\BB\into\PSh(\CC)$ be the inclusion. Since $L$ is left exact, it induces a functor of $ \PSh(\CC) $-categories $L\colon \Univ[\PSh(\CC)] \rightarrow i\Univ[\BB] $ that for every $ A \in \PSh(\CC) $ is given by 
	\[
	\Over{L}{A} \colon \Over{\PSh(\CC)}{A} \to \Over{\BB}{LA}.
	\]
	By Proposition~\ref{prop:critfinlim}, the functor $L$ thus preserves finite limits.
	Furthermore, it readily follows from \cite[Proposition 3.2.9]{MWColimits} that $L$ admits a right adjoint $i$ that is fully faithful. 
	Therefore, we have a commutative diagram
	\[\begin{tikzcd}
		\IFun(\I{J},\Univ[\PSh(\CC)]) & \Univ[\PSh(\CC)]\\
		\IFun(\I{J},i\Univ[\BB]) & i\Univ[\BB]
		\arrow["{\colim_{\I{J}}}", from=1-1, to=1-2]
		\arrow["L", from=1-2, to=2-2]
		\arrow["{\colim_{\I{J}}}"', from=2-1, to=2-2]
		\arrow["L_\ast", to=2-1, from=1-1].
	\end{tikzcd}\]
	Since there is an equivalence $i \IFun(\I{J},i\Univ[\BB])\simeq\IFun(L\I{J},\Univ[\BB])$ that is natural in $\I{J}$, the lower colimit functor in the above diagram can be identified with the functor $ i\colim_{L\I{J}} \colon i\IFun(L\I{J},\Univ[\BB])\to i\Univ[\BB]$. Using that $i$ is fully faithful, we get that this map is equivalent to the composition $ L \colim_{\I{J}} i_*$. Therefore, it suffices to show that the upper colimit functor in the above diagram preserves finite limits. To see this, since $\Over{\PSh(\CC)}{c}\simeq\PSh(\Over{\CC}{c})$ for every $c\in\CC$ and since $\CC\into\PSh(\CC)$ generates $\PSh(\CC)$ under small colimits, it suffices to show that the functor $(-)^\gp\colon\LFib(\I{J})\to\PSh(\CC)$ commutes with finite limits, cf.~Proposition~\ref{prop:critfinlim} and~\cite[Proposition~4.4.1]{MWColimits}.
	Since for every $c\in\CC$ the evaluation functor $\ev_c\colon\PSh(\CC)\to \SS$ commutes with small colimits, the lax square
	\[
	\begin{tikzcd}
		\LFib_{\PSh(\CC)}(\I{J})\arrow[r, "(-)^\gp"]\arrow[d, "\ev_c"] & \BB\arrow[d, "\ev_c"]\\
		\LFib_{\SS}(\I{J}(c))\arrow[r, "(-)^\gp"] & \SS
	\end{tikzcd}
	\]
	is commutative. By assumption and the fact that $\ev_c$ preserves limits, the functor $(-)^\gp\circ\ev_c$ commutes with finite limits, hence so does $\ev_c\circ (-)^\gp$. The claim now follows from the fact that $(\ev_c)_{c\in\CC}\colon\PSh(\CC)\to \prod_{c\in\CC}\SS$ is a conservative functor.
\end{proof}

This leads to the main result of this section:
\begin{proposition}
	\label{prop:FinBRegular}
	The doctrine $ \IFin_\BB $ is sound.
\end{proposition}
\begin{proof}
	Since $\IFin_\BB$ is by definition regular, Corollary~\ref{cor:regularClassSoundWeaklySound} implies that suffices it to show that $\IFin_\BB$ is weakly sound.
	Using Remark~\ref{rem:BCFinB}, we only need to show that every $\BB$-category $ \I{J} $ that has finite colimits is already filtered. But since $\I{J}$ in particular admits finite \emph{constant} colimits, it is section-wise filtered, hence the result follows from Proposition~\ref{prop:sectionwisefilteredgivesfiltered}.
\end{proof}

As a result of Proposition~\ref{prop:FinBRegular}, we can now classify the compact objects of $\Univ$. To that end, Let us denote by $\ILConst_{\SS^{\compact}}$ the full subcategory of $\Univ$ that arises as the essential image of the map $\SS^\compact\to\Univ$ (which is defined as the transpose of $\const_{\BB}\colon\SS^\compact\to\BB$). We now obtain:
\begin{corollary}
	\label{cor:InternallyCompactUniverse}
	There is an equivalence
	\begin{equation*}
		\Univ[\BB]^{\compact}\simeq\ILConst_{\SS^{\compact}}
	\end{equation*} 
	of full subcategories in $\Univ$.
\end{corollary}
\begin{proof}
	Since $\IFin_{\BB}$ is a sound doctrine by Proposition~\ref{prop:FinBRegular}, it has the decomposition property. We may therefore apply Proposition~\ref{prop:characterisationCompactObjectsPSh} to deduce an equivalence
	\begin{equation*}
		\Univ[\BB]^{\compact}\simeq\IRet_{\Univ}(\ISml^{\IFin_{\BB}}(1)).
	\end{equation*} 
	Hence, if $\I{G}\colon A\to \Univ^\compact$ is an arbitrary object, there is a cover $(s_i)\colon\bigsqcup_i A_i\onto A$ in $\BB$ such that $s_i^\ast\I{G}$ is a retract of an object in $\ISml^{\IFin_{\BB}}(1)$ in context $A_i$, for every $i$. By further refining this cover, we can furthermore assume that for each $i$ there is a finite $\Over{\BB}{A_i}$-category $\I{J}_i$ such that $\pi_{A_i}^\ast\I{G}$ is a retract of $\I{J}_i^\gp$. Hence $s_i^\ast\I{G}$ is a retract of an object in $\ILConst_{\SS^{\compact}}$ in context $A_i$, so that Corollary~\ref{cor:LConstPO} implies that $s_i^\ast\I{G}$ is itself contained in $\ILConst_{\SS^{\compact}}$, which necessarily implies that $\I{G}$ is contained in $\ILConst_{\SS^{\compact}}$. 
	Conversely, if $\I{G}$ is an object of $\ILConst_{\SS^\compact}$ in context $A\in\BB$, we can find a cover $(s_i)\colon\bigsqcup_i A_i\onto A$ in $\BB$ such that $s_i^\ast\I{G}$ is a constant $\Over{\BB}{A_i}$-groupoid coming from a compact $\infty$-groupoid, which in turn implies that $s_i^\ast\I{G}$ is a retract of a constant $\Over{\BB}{A_i}$-groupoid coming from a  \emph{finite} $\infty$-groupoid. As this implies that $s_i^\ast\I{G}$ is a retract of an object in $\ISml^{\IFin_{\BB}}(1)$ in context $A_i$, we conclude that $s_i^\ast\I{G}$ must be contained in $\Univ^\compact$, so that $\I{G}$ is contained in $\Univ^\compact$ as well.
\end{proof}

The goal for the remainder of this section is to discuss a more explicit description of filtered $\BB$-categories in the case where $\BB$ is \emph{hypercomplete}. To that end, recall that the filtered $\infty$-categories can be characterised as those $\infty$-categories $ \CC $ for which every map $ \KK \rightarrow \CC $ from a finite $ \infty $-category $\KK$ can be extended to a map from the cone $ \KK^\triangleright \rightarrow \CC $.
In other words, the $\infty$-category $ \CC $ is filtered if and only if for any finite $ \infty $-category $ \KK $ the functor $j^\ast\colon \Fun( \KK^\triangleright, \CC) \to \Fun(\KK,\CC) $ induced by restricting along the inclusion $ j \colon \KK \into  \KK^\triangleright $ is essentially surjective.
This characterisation admits an immediate internal analogue:

\begin{definition}
	\label{def:preFiltered}
	A $ \BB $-category $ \I{J} $ is called \emph{quasi-filtered} if for every finite $ \infty $-category $ \KK $ the functor $j^* \colon \IFun(\KK^\triangleright,\I{J}) \to \IFun(\KK,\I{J})$ is essentially surjective.
\end{definition}
As the terminology suggests, every filtered $\BB$-category is quasi-filtered. To prove this, we require the following lemma, which gives a very explicit description of the notion of quasi-filteredness:
\begin{lemma}
	\label{lem:quasi-filtered=delignefiltered}
	Let $ \I{J} $ be a $ \BB $-category.
	Then $ \I{J} $ is quasi-filtered if and only if for any $ A \in \BB $ and any diagram $ \KK \rightarrow \I{J}(A) $ where $ \KK $ is a finite $ \infty $-category there exists a cover $ (s_i)_i \colon \bigsqcup_i A_i \onto A $ in $\BB$ such that for every $ i $ we can find a map $ \KK^\triangleright \rightarrow \I{J}(A_i) $ making the diagram
	\[\begin{tikzcd}
		\KK & {\I{J}(A)} \\
		{\KK^\triangleright} & {\I{J}(A_i)}
		\arrow[from=1-1, to=1-2]
		\arrow["{s_i^*}", from=1-2, to=2-2]
		\arrow[from=1-1, to=2-1,"j", hook]
		\arrow[from=2-1, to=2-2]
	\end{tikzcd}\]
	commute.
\end{lemma}
\begin{proof}
	Let us first assume that $ \I{J} $ is quasi-filtered.
	Choose a diagram $  \KK \rightarrow \I{J}(A) $ that corresponds to a map $ A \to \IFun(\KK,\I{J})$, and let us form the pullback square
	\[\begin{tikzcd}
		P & {\IFun(\KK^\triangleright,\I{J})^\simeq} \\
		A & {\IFun(\KK,\I{J})^\simeq}.
		\arrow["{(j^*)^\simeq}", from=1-2, to=2-2]
		\arrow[from=1-1, to=1-2]
		\arrow[from=1-1, to=2-1,"s"]
		\arrow[from=2-1, to=2-2]
	\end{tikzcd}\]
	Since $ j^* $ is essentially surjective, $ (j^*)^\simeq  $ is a cover~\cite[Corollary 3.8.12]{MYoneda}, hence so is the map $ s $. 
	Thus $ s \colon P \onto A $ gives the desired cover.
	For the converse, we may pick the diagram $ \KK \rightarrow \I{J}(\IFun(\KK,\I{J})^\core) $ that is determined by the identity $\id \colon \IFun(\KK,\I{J})^\simeq \rightarrow \IFun(\KK,\I{J})^\simeq $. 
	By assumption we may now find a cover $ (s_i)_i \colon \bigsqcup_i A_i \rightarrow  \IFun(\KK, \I{J})^\simeq $ such that the diagram
	\[\begin{tikzcd}
		{\bigsqcup_i A_i} & {\IFun(\KK^\triangleright,\I{J})^\simeq} \\
		{\IFun(\KK, \I{J})^\simeq} & {\IFun(\KK, \I{J})^\simeq}
		\arrow["\id", from=2-1, to=2-2]
		\arrow["{j^*}"', from=1-2, to=2-2]
		\arrow[from=1-1, to=1-2]
		\arrow[two heads, from=1-1, to=2-1]
	\end{tikzcd}\]
	commutes.
	Thus $ j^* $ is also a cover, as desired.
\end{proof}
\begin{proposition}
	\label{prop:FilteredIsPrefiltered}
	Every filtered $\BB$-category is quasi-filtered.
\end{proposition}
\begin{proof}
	Suppose that $\I{J}$ is a filtered $\BB$-category. In light of Lemma~\ref{lem:quasi-filtered=delignefiltered}, it suffices to show that for every finite $\infty$-category $\KK$, every diagram $d\colon \KK\to \I{J}$ locally extends to a map $\KK^\triangleright\to \I{J}$. Note that $\I{J}$ being filtered implies that $\Under{\I{J}}{d}^\gp\simeq 1$. Therefore there is a cover $A\onto 1$ in $\BB$ such that $\Under{\I{J}}{d}(A)$ is non-empty. Unwinding the definitions, this exactly provides the desired local extension of $d$. 
\end{proof}

In~\cite[\'Expose V, Definition 8.11]{SGA4} Deligne chose (a $1$-categorical analogue of) Definition~\ref{def:preFiltered} to \emph{define} filtered 1-categories internal to a $1$-topos, so one might be inclined to surmise that the notions of filteredness and quasi-filteredness coincide. In light of Proposition~\ref{prop:FilteredIsPrefiltered}, the second is always implied by the first, and the converse is in fact true in the case where $\BB\simeq\SS$ (see~\cite[Proposition~5.4.1.22]{htt}). For general $\infty$-topoi, however, this is no longer the case, the obstruction being the presence of non-trivial $\infty$-connected objects:
\begin{proposition}
	\label{prop:inftyConnectiveIsPreFiltered}
	Let $ \I{G} \in \BB $ be an $ \infty $-connective object.
	Then $ \I{G} $ is a quasi-filtered $ \BB $-category.
\end{proposition}
\begin{proof}
	It is well-known (see \cite{300682}) that $\I{G}$ is $ \infty$-connective if and only if for an arbitrary finite $\infty$-category $\KK$, the diagonal map $ \I{G} \to \IFun(\KK, \I{G}) $ is a cover. This clearly implies the claim.
\end{proof}
Since any filtered $ \BB $-groupoid is necessarily equivalent to the final object, Proposition~\ref{prop:inftyConnectiveIsPreFiltered} shows that any non-trivial $ \infty $-connective object gives rise to a $ \BB $-category that is quasi-filtered but not filtered.
In the remainder of this section we will show that is essentially the only obstruction.
More precisely we will show that if $ \BB $ is hypercomplete, then any quasi-filtered $ \BB $-category is filtered.

\begin{lemma}
	\label{lem:Pre-filtIsContrInHypercomp}
	Let $ \I{C} $ be a quasi-filtered $ \BB $-category and assume that $ \BB $ is hypercomplete.
	Then $ \I{C}^\gp \simeq 1 $. 
\end{lemma}
\begin{proof}
	Since $ \BB $ is hypercomplete, it suffices to see that the diagonal map $ \I{C}^\gp \to \map{\Univ}(K, \I{C}^\gp) $ is a cover for any finite $ \infty$-groupoid $K$, as in this case $ \I{C}^\gp $ is $\infty $-connective (see \cite{300682} again).
	So it is enough to see that for every $A\in\BB$, every map $f \colon K \to \I{C}^\gp(A) $ from a finite $\infty$-groupoid $K$ locally factors through the point.
	Replacing $ \BB $ by $ \BB_{/A} $ we may assume that $ A =1 $, so that $ f $ corresponds to a map $ g \colon K \to \I{C}^\gp $.
	Now recall that since the doctrine of finite $ \BB $-categories is sound and regular, we can find a filtered $\BB$-category $\I{J}$ and a diagram $d\colon \I{J}\to\IFin_{\BB}$ with colimit $\I{C}$.
	Since $ (-)^\gp $ is cocontinuous and $K$ is a compact object of $\Univ$ by Corollary~\ref{cor:InternallyCompactUniverse},
	 we obtain an equivalence
	\begin{equation*}
		\map{\Univ}(K,\I{C}^\gp)\simeq (\Under{\I{C}}{\map{\Univ}(K, d(-)^\gp)})^\gp.
	\end{equation*}
	Therefore, we may find a cover $ \bigsqcup_k A_k \onto 1 $ and objects $j_k\colon A_k\to \I{J}$ for each $k$ such that $ \pi_{A_k}^* g $ factors through the canonical map $d(j_k)^\gp\to \pi_{A_k}^\ast\I{C}^\gp  $.
	Since $ d(j_k) $ is a finite $ \BB_{/A} $-category we may pass to a further cover and can therefore assume that $ d(j_k)$ is the constant $\Over{\BB}{A_k}$-category associated to a finite $\infty$-category. Therefore, the assumption that $ \I{C} $ is quasi-filtered implies that locally the map $d(j_k)^\gp \to \pi_{A_k}^* \I{C}^\gp  $ factors through the final object, hence the claim follows.
\end{proof}

\begin{proposition}
	\label{prop:Prefilteredimpliesfiltered}
	Suppose that $ \BB $ is hypercomplete.
	Then any quasi-filtered $ \BB $-category is filtered.
\end{proposition}
\begin{proof}
	Let $ \I{C} $ be quasi-filtered. Since $\IFin_{\BB}$ is sound, we only have to verify that for any finite $ \BB $-category $ \I{K} $ the diagonal functor $ \I{C} \to \IFun(\I{K},\I{C}) $ is final.
	Since being final is a local property, we may assume that $ \I{K} $ is the constant $\BB$-category attached to some finite $ \infty $-category $ \KK $ (see~ Proposition~\ref{prop:FinBDoctrine}).
	Now for any diagram $d \colon \KK \to \I{C} $, we will show that the slice $\BB$-category $ \I{C}_{d/} $ is again quasi-filtered.
	To see this, let $\KK^\prime$ be a finite $\infty$-category and consider an arbitrary map $f \colon \KK^\prime \to  \I{C}_{d/}(A)$ for some $ A \in \BB $.
	Passing from $ \BB $ to $ \BB_{/A} $ and using that $ \pi_A^* (\I{C}_{d/}) \simeq (\pi_A^* \I{C}) _{\pi_A^*d /}$ we may assume that $ A\simeq1 $.
	Since $ \KK $ is constant, the global sections of $ \I{C}_{/d} $ recover the slice $\infty$-category $ (\Gamma\I{C})_{d/} $.
	Therefore $ f $ is given by a map $f' \colon \KK \diamond \KK' \to \I{C}(1) $ out of the join such that the restriction along $ \KK \into \KK \diamond \KK' $ recovers $ d $.
	But since finite $ \infty $-categories are stable under the join construction, we may find a covering $ (s_i)_i \colon \bigsqcup A_i \onto 1 $ and for every $ i $ an extension $ (\KK \diamond \KK')^\triangleright \to \I{C}(A_i) $ of $ \pi_{A_i}^* f' $.
	But these precisely correspond to maps $ (\KK^\prime)^\triangleright \to \I{C}(A_i) $ extending $ s_i^* \circ f \colon \KK^\prime \to \I{C}(A_i) $, which shows that $ \I{C}_{d/} $ is quasi-filtered and that therefore $ \I{C}_{d/}^\gp \simeq 1 $ by Lemma~\ref{lem:Pre-filtIsContrInHypercomp}.
	Repeating the above argument with $ \BB_{/A} $ instead of $ \BB $ we get that the same holds for a diagram $ d $ in any context $ A $, so the claim follows from Quillen's Theorem A \cite[Corollary 4.4.8]{MYoneda}.
\end{proof}

\section{Accessible $\BB$-categories}
\label{chap:accessible}
In the classical $1$-categorical literature, a \emph{$\kappa$-accessible} $1$-category is one that can be obtained as the free cocompletion of a small $1$-category under $\kappa$-filtered colimits~\cite{Lair1981, Makkai1989}. In~\cite[\S~5.4]{htt}, Lurie generalises this concept to $\infty$-categories. In this section we will introduce and study an analogous notion for $\BB$-categories, that of a \emph{$\I{U}$-accessible} $\BB$-category for any sound doctrine $\I{U}$. As with our discussion of $\I{U}$-filteredness, we draw much of our inspiration from ideas in~\cite{Adamek2002} and~\cite{rezk2021}. Our exposition is tailored to the study of presentable $\BB$-categories in \S~\ref{chap:presentable}, so we will not provide an exhaustive treatment of accessibility for $\BB$-categories, but rather set up only the basic machinery that we will need for our discussion of presentability later on. We begin in \S~\ref{sec:accessibility} by giving the definition of a $\I{U}$-accessible $\BB$-category and proving some basic results that will be useful later. In~\S~\ref{sec:accessibleFunctors}, we discuss accessible functors. In~\S~\ref{sec:UCompactAccessible}, we give a characterisation of $\I{U}$-accessible $\BB$-categories as those that are generated by $\I{U}$-compact objects under $\I{U}$-filtered colimits. Finally, we discuss the notion of \emph{$\I{U}$-flatness} in~\S~\ref{sec:flatness}.

\subsection{Accessibility}
\label{sec:accessibility}
If $\I{U}$ is an arbitrary internal class of $\BB$-categories and if $\I{C}$ is a $\BB$-category, we will use the notation $\IInd^{\I{U}}(\I{C})=\IPSh^{\IFilt_{\I{U}}}(\I{C})$ to denote the free $\IFilt_{\I{U}}$-cocompletion of $\I{U}$. We write $\Ind_{\BB}^{\I{U}}(\I{C})$ for the underlying $\infty$-category of global sections. If $\I{U}=\IFin_\BB$, we will simply write $\IInd(\I{C})$ for the associated free $\IFilt_{\IFin_\BB}$-cocompletion and $\Ind_{\BB}(\I{C})$ for its underlying $\infty$-category of global sections. We may  now define:
\begin{definition}
	\label{def:accessibility}
	Let $\I{U}$ be a sound doctrine. A large $\BB$-category $\I{D}$ is \emph{$\I{U}$-accessible} if there is a $\BB$-category $\I{C}$ and an equivalence $\I{D}\simeq\IInd^{\I{U}}(\I{C})$. A large $\BB$-category is called \emph{accessible} if it is $\I{U}$-accessible for some sound doctrine $\I{U}$.
\end{definition}
\begin{remark}
	\label{rem:BCInd}
	By combining Remark~\ref{rem:BCFilt} with~\cite[Proposition~7.1.11]{MWColimits}, we find that for every $A\in\BB$ there is a canonical identification $\pi_A^\ast\IInd^{\I{U}}(\I{C})\simeq\IInd[\Over{\BB}{A}]^{\pi_A^\ast\I{U}}(\pi_A^\ast\I{C})$ for every $\BB$-category $\I{C}$ and every sound doctrine $\I{U}$.
\end{remark}
\begin{remark}
	\label{rem:accessibilityRegularisation}
	In light of Proposition~\ref{prop:regularisationFilt} one has $\IInd^{\I{U}}(\I{C})\simeq\IInd^{\I{U}^\reg_\leftarrow}(\I{C})$ for every $\BB$-category $\I{C}$ and every internal class $\I{U}$. In particular, a large $\BB$-category $\I{D}$ is $\I{U}$-accessible if and only if it is $\I{U}^\reg_\leftarrow$-accessible. When arguing about accessible $\BB$-categories, we can therefore always assume that $\I{U}$ is in addition \emph{left regular} (cf.\ Corollary~\ref{cor:regularisationSoundness}).
\end{remark}

Suppose that $\I{D}$ is a $\I{U}$-accessible $\BB$-category, i.e.\ that we have $\I{D}\simeq\IInd^{\I{U}}(\I{C})$ for some $\BB$-category $\I{C}$. Recall from~\cite{MWColimits}[\S~7.1] that there is an inclusion $\ISml^{\IFilt_{\I{U}}}(\I{C})\into\IInd^{\I{U}}(\I{C})$. The following proposition shows that this inclusion is in fact an equivalence.
\begin{proposition}
	\label{prop:IndFilteredFreeCocompletion}
	For any internal class $\I{U}$ of $\BB$-categories and any $\BB$-category $\I{C}$, the fully faithful functor $\ISml^{\IFilt_{\I{U}}}(\I{C})\into \IInd^{\I{U}}(\I{C})$ is an equivalence. In other words, the Yoneda embedding $\I{C}\into\ISml^{\IFilt_{\I{U}}}(\I{C})$ exhibits $\ISml^{\IFilt_{\I{U}}}(\I{C})$ as the free $\IFilt_{\I{U}}$-cocompletion of $\I{C}$.
\end{proposition}
\begin{proof}
	It will be enough to show that $\ISml^{\IFilt_{\I{U}}}(\I{C})$ is closed under $\IFilt_{\I{U}}$-colimits in $\IPSh(\I{C})$. By combining Remark~\ref{rem:BCFilt} with~\cite[Remark~6.2.2]{MWColimits}, this follows once we prove that for any $\I{U}$-filtered $\BB$-category $\I{J}$, the colimit of any diagram $d\colon \I{J}\to \ISml^{\IFilt_{\I{U}}}(\I{C})$ in $\IPSh(\I{C})$ is contained in $\ISml^{\IFilt_{\I{U}}}(\I{C})$. Let us set $F=\colim d$ and let $p\colon\Over{\I{C}}{F}\to\I{C}$ be the associated right fibration. We need to show that $\Over{\I{C}}{F}$ is $\I{U}$-filtered. On account of the equivalence $\IPSh(\I{C})\simeq\IRFib_{\I{C}}$ and in light of Lemma~\ref{lem:RFibReflectiveSubcategory}, we may regard $d$ as a diagram $d\colon \I{J}\to\IRFib_{\I{C}}\into\Over{(\ICat_{\BB})}{\I{C}}$
	that takes values in the full subcategory $\Over{(\IFilt_{\I{U}})}{\I{C}}$ (as $\IFilt_{\I{U}}$ is a colimit class by Remark~\ref{rem:FiltUColimitClass}). Let $\I{K}\to\I{C}$ be the colimit of $d$ in $\Over{(\ICat_{\BB})}{\I{C}}$. As the right fibration $p\colon\Over{\I{C}}{F}\to\I{C}$ is the image of $\I{K}\to\I{C}$ along the localisation functor $L\colon \Over{(\ICat_{\BB})}{\I{C}}\to\IRFib_{\I{C}}$ (see~\cite[Proposition~4.2.11]{MWColimits}), there is a final map $\I{K}\to\Over{\I{C}}{F}$ over $\I{C}$. It therefore suffices to show that $\I{K}$ is $\I{U}$-filtered.
	Now Proposition~\ref{prop:sliceFibrationColimits} implies that $\I{K}$ is the colimit of the diagram $(\pi_{\I{C}})_!d\colon \I{J}\to\Over{(\ICat_{\BB})}{\I{C}}\to\ICat_{\BB}$. By construction, this diagram takes values in $\IFilt_{\I{U}}$. Therefore, the result follows from Proposition~\ref{prop:UFilteredUCocomplete}.
\end{proof}
\begin{remark}
	In light of Proposition~\ref{prop:IndFilteredFreeCocompletion}, if $\I{C}$ is a $\BB$-category and if $\I{U}$ is a sound doctrine, Remark~\ref{rem:FiltUColimitClass} implies that a presheaf $F\colon A\to\IPSh(\I{C})$ in context $A\in\BB$ is contained in $\IInd^{\I{U}}(\I{C})$ if and only if the $\Over{\BB}{A}$-category $\Over{\I{C}}{F}$ is $\pi_A^\ast\I{U}$-filtered.
\end{remark}

For later use, let us record that our notion of accessibility is stable under the formation of slice $\BB$-categories:
\begin{proposition}
	\label{prop:accessibilitySliceCategory}
	Let $\I{U}$ be a sound doctrine and let $\I{D}$ be a $\I{U}$-accessible $\BB$-category. Then $\Over{\I{D}}{d}$ is $\pi_A^\ast\I{U}$-accessible, for any choice of object $d\colon A\to\I{D}$.
\end{proposition}
\begin{proof}
	Using Remark~\ref{rem:BCInd}, we may assume that $A\simeq 1$.
	Choose a $\BB$-category $\I{C}$ such that $\I{D}\simeq\IInd^{\I{U}}(\I{C})$. Let $F\colon\I{C}^{\op}\to\Univ$ be the presheaf that corresponds to $d$ under this equivalence. We then obtain a commutative diagram
	\begin{equation*}
		\begin{tikzcd}
			\Over{\I{C}}{F}\arrow[d, "p"]\arrow[r, hookrightarrow] & \Over{\IInd(\I{C})}{F}\arrow[r, hookrightarrow]\arrow[d, "(\pi_F)_!"] &\Over{\IPSh(\I{C})}{F}\arrow[d, "(\pi_F)_!"]\\
			\I{C}\arrow[r, hookrightarrow, "h_{\I{C}}"] & \IInd(\I{C})\arrow[r, hookrightarrow] & \IPSh(\I{C}) 
		\end{tikzcd}
	\end{equation*}
	in which both squares are cartesian.
	By~\cite[Lemma~6.1.5]{MWColimits}, the vertical map on the right can be identified with $p_!\colon\IPSh(\Over{\I{C}}{F})\to\IPSh(\I{C})$ such that the upper row in the above diagram recovers the Yoneda embedding $h_{\Over{\I{C}}{F}}$. With respect to this identification, a presheaf on $\Over{\I{C}}{F}$ is contained in $\Over{\IInd(\I{C})}{F}$ precisely if the domain of the associated right fibration is $\I{U}$-filtered. We therefore obtain an equivalence $\Over{\IInd^{\I{U}}(\I{C})}{F}\simeq\IInd^{\I{U}}(\Over{\I{C}}{F})$, hence the result follows.
\end{proof}

\subsection{Accessible functors}
\label{sec:accessibleFunctors}
It will be convenient to also have a notion of accessibility for functors between accessible $\BB$-categories at our disposal:
\begin{definition}
	\label{def:AccFunctor}
	Let $ \I{U} $ be a sound doctrine.
	A functor $ f \colon \I{C} \rightarrow \I{D} $ of large $ \BB $-categories is $ \I{U} $\emph{-accessible} if $ \I{C} $ and $ \I{D} $ are $ \IFilt_{\I{U}} $-cocomplete and $ f $ is $ \IFilt_{\I{U}} $-cocontinuous.
	We will call $ f $ \emph{accessible} if it is $ \I{U} $-accessible for some sound doctrine $ \I{U} $.
	We denote by $ \IFun^{\acc}(\I{C},\I{D}) $ the full subcategory spanned by those objects $ A \rightarrow \IFun(\I{C},\I{D}) $ such that the corresponding $ \BB_{/A} $-functor $ \pi_A^* \I{C} \rightarrow \pi_A^* \I{D} $ is accessible. We will denote by $\Fun_{\BB}^\acc(\I{C},\I{D})$ the underlying $\infty$-category of global sections.
\end{definition}

\begin{remark}
	\label{rem:AccSufficesForKappa}
	Let $ f \colon \I{C} \to \I{D} $ be $ \I{U} $-accessible for some sound doctrine $ \I{U} $.
	By Remark~\ref{rem:regularClassesExhaustive} we may find a $ \BB $-regular cardinal $ \kappa $ such that $ \I{U} \subset \Cat_\BB^\kappa$.
	It follows that a functor is accessible if and only if it is $ \Cat_\BB^\kappa $-accessible for some $ \BB $-regular cardinal $ \kappa $.
\end{remark}

\begin{remark}
	\label{rem:AccOfFunctorsIsLocal}
	Let $f\colon A \rightarrow \IFun^{\acc} (\I{C},\I{D})$ be an arbitrary object.
	By definition, this means that there is a cover $(s_i) \colon  \bigsqcup_i A_i \onto A $ in $ \BB $ such that the functors $ s_i^* f\colon \pi_{A_i}^\ast\I{C}\to\pi_{A_i}^\ast\I{D} $ are accessible for all $i\in I$.
	By Remark~\ref{rem:AccSufficesForKappa}, we may find a $\Over{\BB}{A}$-regular cardinal $ \kappa $ such that all $ A_i $ are $ \kappa $-compact (in $\Over{\BB}{A}$) and $ s_i^* f $ is $ \Cat_{\BB/A_i}^\kappa $-accessible for every $ \kappa $.
	Hence Remarks~\ref{rem:BCFilt} and~\ref{rem:BCKappaSmallCategories} together with \cite[Remark~5.2.3]{MWColimits} imply that $ f $ is  $\IFilt_{\Cat_{\BB/A}^\kappa}$-cocontinuous, so in particular accessible.
	Thus, an object $f\colon A \to  \IFun(\I{C},\I{D}) $ is contained in $ \IFun^\acc(\I{C},\I{D})$ if and only if $f$ defines an accessible functor between $\Over{\BB}{A}$-categories. In particular, one obtains a canonical equivalence $\pi_A^\ast\IFun^\acc(\I{C},\I{D})\simeq\IFun[\Over{\BB}{A}]^\acc(\pi_A^\ast\I{C},\pi_A^\ast\I{D})$ for every $A\in\BB$.
\end{remark}

Somewhat surprisingly, provided that both domain and codomain have a sufficient amount of colimits, accessibility of a functor between $\BB$-categories is an entirely section-wise concept:
\begin{proposition}
	\label{prop:CocompleteKappaCocontinuousInternalExternal}
	Let $\kappa$ be a $\BB$-regular cardinal and let $f\colon\I{C}\to\I{D}$ be a functor between cocomplete $\BB$-categories that is section-wise $\kappa$-accessible. Then the functor $f$ is $\IFilt_{\ICat_{\BB}^\kappa}$-accessible.
\end{proposition}
\begin{proof}
	As $\kappa$ is $\BB$-regular, Remarks~\ref{rem:BCKappaSmallCategories} and~\ref{rem:BCFilt} imply that it suffices to show that $f$ preserves the colimit of every diagram $d\colon\I{J}\to\I{C}$ with $\I{J}$ a $\ICat_{\BB}^\kappa$-filtered $\BB$-category. As $\I{C}$ is cocomplete, there exists an extension $d^\prime\colon \IPSh(\I{J})^{\cpt{\ICat_{\BB}^\kappa}}\to\I{C}$ of $d$. By Corollary~\ref{cor:inclusionCompactFinal}, the inclusion $\I{J}\into\IPSh(\I{J})^{\cpt{\ICat_{\BB}^\kappa}}$ is final, hence we may replace $\I{J}$ by $\IPSh(\I{J})^{\cpt{\ICat_{\BB}^\kappa}}$ and $d$ by $d^\prime$ and can thus assume that $\I{J}$ is $\ICat_{\BB}^{\kappa}$-cocomplete (see Proposition~\ref{prop:UCompactUcocomplete}). Using Lemma~\ref{lem:CocompleteImpliesFinallyConstant} and Remark~\ref{rem:FiltUColimitClass}, we can further reduce to the case where $\I{J}$ is the constant $\BB$-category that is associated with an $\infty$-category with $\kappa$-small colimits. As by~\cite[Proposition~5.3.3.3]{htt} every such $\infty$-category is $\kappa$-filtered, the result follows.
\end{proof}

\begin{corollary}
	\label{cor:AccIsSectWiseAccWhenCocomp}
	Let $ f \colon \I{C} \to \I{D} $ be a functor of cocomplete large $ \BB $-categories.
	Then $ f $ is accessible if and only if $f$ is section-wise accessible.
\end{corollary}
\begin{proof}
	By Remark~\ref{rem:AccSufficesForKappa}, we can assume that $ f $ is $ \Cat_\BB^\kappa $-accessible for some $ \BB $-regular cardinal $ \kappa $.
	Then it follows from Lemma~\ref{lem:externallyFilteredImpliesInternallyFiltered} that $ f $ commutes with colimits indexed by constant $ \BB $-categories attached to $ \kappa $-filtered $ \infty $-categories.
	In other words, $ f(A) $ commutes with $ \kappa $-filtered colimits for every $ A\in\BB $ and is thus section-wise accessible.
	For the converse, we pick a small full subcategory $ \GG \into\BB$ that generates $ \BB $ under small colimits. 
	Then we may find a $ \BB $-regular cardinal $ \kappa $ such that $ f(G) $ is $ \kappa $-accessible for every $ G \in\GG$. Since the preservation of colimits is a local condition~\cite[Remark~4.2.1]{MWColimits} and since every object $A\in\BB$ admits a cover by objects in $\GG$, we conclude that $f(A)$ preserves $\kappa$-filtered colimits for all $A\in\BB$.
	Therefore $ f $ is accessible by Proposition~\ref{prop:CocompleteKappaCocontinuousInternalExternal}. 
\end{proof}

\subsection{$\I{U}$-compact objects in accessible $\BB$-categories}
\label{sec:UCompactAccessible}
In~\cite[Proposition~5.4.2.2]{htt}, Lurie characterises $\kappa$-accessible $\infty$-categories as those that are generated by a small collection of $\kappa$-compact objects under $\kappa$-filtered colimits. In this section, our goal is to obtain an analogue of this statement for accessible $\BB$-categories. We begin with the following characterisation of the $\I{U}$-compact objects in a $\I{U}$-accessible $\BB$-category:

\begin{proposition}
	\label{prop:characterisationCompactObjectAccessible}
	Let $\I{U}$ be an internal class of $\BB$-categories, let $\I{C}$ be a $\BB$-category and let $\I{D}=\IInd^{\I{U}}(\I{C})$. Then there is an equivalence $\I{D}^{\cpt{\I{U}}}\simeq\IRet_{\I{D}}(\I{C})$ of full subcategories in $\I{D}$. In particular, $\I{D}^{\cpt{\I{U}}}$ is small.
\end{proposition}
\begin{proof}
	In light of Remark~\ref{rem:retracts}, the second claim follows immediately from the first. Now by Yoneda's lemma and the fact that the inclusion $\I{D}\into\IPSh(\I{C})$ is closed under $\IFilt_{\I{U}}$-colimits, every representable presheaf on $\I{C}$ defines a $\I{U}$-compact object in $\I{D}$. In other words, one obtains an inclusion $\I{C}\into\I{D}^{\cpt{\I{U}}}$. By combining this observation with Proposition~\ref{prop:RetractsUCompact}, one obtains an inclusion $\IRet_{\I{D}}(\I{C})\into\I{D}^{\cpt{\I{U}}}$. Conversely, let $F\colon A\to\I{D}^{\cpt{\I{U}}}$ be an arbitrary object. We need to show that $F$ is contained in $\IRet_{\I{D}}(\I{C})$. Upon replacing $\BB$ with $\Over{\BB}{A}$ (which is made possible by Remark~\ref{rem:BCUComp} and~\ref{rem:BCInd}), we can assume $A\simeq 1$. The desired result thus follows from Lemma~\ref{lem:compactObjectRetract}.
\end{proof}

We can now state and prove our characterisation of $\I{U}$-accessible $\BB$-categories. To that end, if $\I{D}$ is a $\IFilt_{\I{U}}$-cocomplete $\BB$-category and $\I{C}\into\I{D}$ is a full subcategory, we shall say that $\I{D}$ is \emph{generated} under $\IFilt_{\I{U}}$-colimits by $\I{C}$ if $\I{D}$ is the smallest full subcategory of itself that is closed under $\IFilt_{\I{U}}$-colimits and contains $\I{C}$. We now obtain:
\begin{proposition}
	\label{prop:characterisationAccessibility}
	Let $\I{U}$ be a sound doctrine and let $\I{D}$ be a large $\BB$-category. Then the following are equivalent:
	\begin{enumerate}
		\item $\I{D}$ is $\I{U}$-accessible;
		\item $\I{D}$ is locally small and $\IFilt_{\I{U}}$-cocomplete, the (a priori large) $\BB$-category $\I{D}^{\cpt{\I{U}}}$ is small and generates $\I{D}$ under $\IFilt_{\I{U}}$-colimits;
		\item $\I{D}$ is $\IFilt_{\I{U}}$-cocomplete, and there is a small full subcategory $\I{C}\into\I{D}$ such that $\I{C}\into\I{D}^{\cpt{\I{U}}}$ and such that $\I{C}$ generates $\I{D}$ under $\IFilt_{\I{U}}$-colimits.
	\end{enumerate}
\end{proposition}
\begin{proof}
	If $\I{D}$ is $\I{U}$-accessible, there is a small $\BB$-category $\I{C}$ and an equivalence $\I{D}\simeq\IInd^{\I{U}}(\I{C})$. In particular, $\I{D}$ is locally small and $\IFilt_{\I{U}}$-cocomplete. Furthermore, Proposition~\ref{prop:characterisationCompactObjectAccessible} implies that $\I{D}^{\cpt{\I{U}}}$ is small. Since $\I{D}$ is generated by $\I{C}$ under $\IFilt_{\I{U}}$-colimits and therefore by $\I{D}^{\cpt{\I{U}}}$, we conclude that~(1) implies~(2). Moreover,~(2) trivially implies~(3), and the fact that~(3) implies~(1) immediately follows from~\cite[Proposition~7.2.4]{MWColimits}.
\end{proof}

Recall from~\cite[Definition~3.4.7]{MWColimits} that a localisation $L\colon\I{D}\to\I{E}$ is a \emph{Bousfield localisation} if $L$ admits a (necessarily fully faithful) right adjoint $i$. Proposition~\ref{prop:characterisationAccessibility} now implies:
\begin{corollary}
	\label{cor:AccessibleLocalisation}
	Let $\I{U}$ be a sound doctrine and let $\I{D}$ be a $\I{U}$-accessible $\BB$-category. Suppose that $\I{E}$ is a Bousfield localisation of $\I{D}$ such that the inclusion $i\colon\I{E}\into\I{D}$ is $\IFilt_{\I{U}}$-cocontinuous. Then $\I{E}$ is $\I{U}$-accessible as well.
\end{corollary}
\begin{proof}
	Let $\I{C}\into\I{E}$ be the image of $\I{D}^{\cpt{\I{U}}}$ along the localisation functor $L\colon\I{D}\to\I{E}$. As $\I{E}$ is locally small and $\I{D}^{\cpt{\I{U}}}$ is small by Proposition~\ref{prop:characterisationAccessibility}, the $\BB$-category $\I{C}$ is small as well~\cite[Lemma~4.7.5]{MYoneda}. In light of the adjunction $L\dashv i$, the assumption that $i$ is $\IFilt_{\I{U}}$-cocontinous implies that $L$ preserves $\I{U}$-compact objects. In other words, we have $\I{C}\into\I{D}^{\cpt{\I{U}}}$. By Proposition~\ref{prop:characterisationAccessibility}, the large $\BB$-category $\I{D}$ is generated by $\I{D}^{\cpt{\I{U}}}$ under $\IFilt_{\I{U}}$-colimits, i.e.\ $\I{D}$ is the smallest full subcategory of itself that contains $\I{D}^{\cpt{\I{U}}}$ and that is closed under $\IFilt_{\I{U}}$-colimits. Let $\I{E}^\prime\into\I{E}$ be the smallest full subcategory that contains $\I{C}$ and that is closed under $\IFilt_{\I{U}}$-colimits, and let us consider the commutative diagram
	\begin{equation*}
		\begin{tikzcd}
			\I{D}^{\cpt{\I{U}}}\arrow[d, twoheadrightarrow] \arrow[r, hook] & \I{D}^\prime\arrow[d]\arrow[r, hook] & \I{D}\arrow[d, twoheadrightarrow, "L"]\\
			\I{C}\arrow[r, hook] & \I{E}^\prime\arrow[r, hook] & \I{E} 
		\end{tikzcd}
	\end{equation*}
	in which the right square is a pullback. Since $L$ is cocontinuous, the inclusion $\I{D}^\prime\into\I{D}$ is closed under $\IFilt_{\I{U}}$-colimits (using Lemma~\ref{lem:pullbackCocompleteRightFibration}) and must therefore be an equivalence. As the inclusion $i$ is a section of $L$, this implies that the inclusion $\I{E}^\prime\into\I{E}$ is an equivalence as well. By using Proposition~\ref{prop:characterisationAccessibility}, we thus conclude that $\I{D}$ is $\I{U}$-accessible.
\end{proof}

\begin{definition}
	\label{def:accessibleLocalisation}
	A Bousfield localisation $L\colon\I{D}\to\I{E}$ of a $\I{U}$-accessible $\BB$-category $\I{D}$ is said to be \emph{$\I{U}$-accessible} if the inclusion $\I{E}\into\I{D}$ is $\I{U}$-accessible. More generally, a Bousfield localisation $L\colon \I{D}\to\I{E}$ is accessible if there is a sound doctrine $\I{U}$ such that $\I{D}$ is $\I{U}$-accessible and the inclusion $\I{E}\into\I{D}$ is $\IFilt_{\I{U}}$-cocontinuous.
\end{definition}

\begin{remark}
	\label{rem:AccessibilityLocal}
	Proposition~\ref{prop:characterisationAccessibility} also shows that accessibility is a \emph{local} condition: if $\bigsqcup_i A_i\onto 1$ is a cover in $\BB$ and if $\I{U}$ is a sound doctrine, then a large $\BB$-category $\I{D}$ being $\I{U}$-accessible is equivalent to each $\pi_{A_i}^\ast\I{D}$ being $\pi_{A_i}^\ast\I{U}$-accessible. In fact, Remark~\ref{rem:BCInd} shows that we have an equivalence $\pi_{A}^\ast\IInd^{\I{U}}(\I{C})\simeq\IInd[\Over{\BB}{A}]^{\pi_{A}^\ast\I{U}}(\pi_{A}^\ast\I{C})$ for every $\BB$-category $\I{C}$ and every $A\in\BB$, hence the condition is necessary. To show that it is sufficient, first recall that since $\IFilt_{\I{U}}$-cocompleteness is a local condition~\cite[Remark~5.2.3]{MWColimits}, we deduce that $\I{D}$ must be $\IFilt_{\I{U}}$-cocomplete. Moreover, if $\I{E}\into\I{D}$ is the smallest full subcategory that is closed under $\IFilt_{\I{U}}$-colimits in $\I{D}$ and that contains $\I{D}^{\cpt{\I{U}}}$, the fact that $\pi_{A_i}^\ast\I{E}$ is closed under $\IFilt_{\pi_{A_i}^\ast\I{U}}$-colimits and Remark~\ref{rem:BCUComp} imply that the inclusion $\I{E}\into\I{D}$ is locally an equivalence and therefore already an equivalence. To show that $\I{D}$ is $\I{U}$-accessible, Proposition~\ref{prop:characterisationAccessibility} thus implies that it suffices to verify that also the condition of a large $\BB$-category to be (locally) small is local in $\BB$, which is clear from the definitions.
\end{remark}

Proposition~\ref{prop:characterisationAccessibility} can furthermore be used to show that presheaf $\BB$-categories are $\I{U}$-accessible for \emph{every} choice of a sound doctrine $\I{U}$:
\begin{proposition}
	\label{prop:PShAccessible}
	For every $\BB$-category $\I{C}$ and every sound doctrine $\I{U}$, the $\BB$-category $\IPSh(\I{C})$ is $\I{U}$-accessible.
\end{proposition}
\begin{proof}
	In light of Remark~\ref{rem:accessibilityRegularisation}, we can assume that $\I{U}$ is a left regular doctrine.
	By Proposition~\ref{prop:characterisationCompactObjectsPSh}, the $\BB$-category $\IPSh(\I{C})^{\cpt{\I{U}}}$ is small. Using Proposition~\ref{prop:characterisationAccessibility}, it therefore suffices to show that every object in $\IPSh(\I{C})$ can be obtained as a $\I{U}$-filtered colimit of $\I{U}$-compact objects. If $F\colon\I{C}^{\op}\to\Univ$ is an arbitrary presheaf, Lemma~\ref{lem:coYonedaGeneralised} shows that $F$ is the colimit of the diagram $\Over{\IPSh(\I{C})^{\cpt{\I{U}}}}{F}\to\IPSh(\I{C})^{\cpt{\I{U}}}\into\IPSh(\I{C})$. By Lemma~\ref{lem:pullbackCocompleteRightFibration}, the $\BB$-category $\Over{\IPSh(\I{C})^{\cpt{\I{U}}}}{F}$ is $\op(\I{U})$-cocomplete and therefore in particular $\I{U}$-filtered (Example~\ref{ex:UCocompleteWeaklyFiltered}). Hence $F$ is contained in $\IInd^{\I{U}}(\IPSh(\I{C})^{\cpt{\I{U}}})$. Finally, upon replacing $\BB$ with $\Over{\BB}{A}$ (which is made possible by Remark~\ref{rem:BCUComp}), the same conclusion holds for objects in $\IPSh(\I{C})$ in context $A$, which finishes the proof.
\end{proof}

\subsection{Flatness}
\label{sec:flatness}
Recall from~\cite[\S~7.1]{MWColimits} that if $\I{C}$ is a $\BB$-category, the functor of left Kan extension along the Yoneda embedding $h_{\I{C}^{\op}}\colon\I{C}^{\op}\into\IFun(\I{C},\Univ)$ induces an equivalence
\begin{equation*}
	(h_{\I{C}^{\op}})_!\colon \IPSh(\I{C})\simeq\IFun^\cc(\IFun(\I{C},\Univ),\Univ)
\end{equation*}
where the right-hand side denotes the large $\BB$-category of cocontinuous functors between $\IFun(\I{C},\Univ)$ and $\Univ$.
\begin{definition}
	\label{def:UFlat}
	Let $\I{C}$ be a $\BB$-category and let $\I{U}$ be an arbitrary internal class of $\BB$-categories. A presheaf $F\colon A\to\IPSh(\I{C})$ is said to be \emph{$\I{U}$-flat} if the functor
	\begin{equation*}
		\IFun[\Over{\BB}{A}](\pi_A^\ast\I{C},\Univ[\Over{\BB}{A}])\to \Univ[\Over{\BB}{A}]
	\end{equation*}
	that is encoded by $(h_{\I{C}^{\op}})_!(F)$ is $\pi_A^\ast\I{U}$-continuous. We denote the full subcategory of $\IPSh(\I{C})$ that is spanned by the $\I{U}$-flat presheaves by $\IFlat^{\I{U}}(\I{C})$, and we denote its underlying $\infty$-category of global sections by $\Flat_{\BB}^{\I{U}}(\I{C})$.
\end{definition}

\begin{remark}
	\label{rem:UflatLocal}
	In the situation of Definition~\ref{def:UFlat}, the fact that $\I{U}$-continuity is a local condition~\cite[Remark~5.2.3]{MWColimits} together with~\cite[Remark~6.3.2]{MWColimits} a d~\cite[Lemma~4.7.14]{MYoneda} implies that the presheaf $F$ is $\I{U}$-flat if and only if for every cover $(s_i)\colon\bigsqcup A_i\onto A$ in $\BB$ the presheaf $s_i^\ast F$ is $\I{U}$-flat. In particular, \emph{every} object in $\IFlat^{\I{U}}(\I{C})$ is $\I{U}$-flat, and there is a canonical equivalence $\pi_A^\ast\IFlat^{\I{U}}(\I{C})\simeq\IFlat[\Over{\BB}{A}]^{\pi_A^\ast\I{U}}(\pi_A^\ast\I{U})$ for every $A\in\BB$.
\end{remark}

\begin{lemma}
	\label{lem:formulaYonedaExtension}
	Let $\I{C}$ be a $\BB$-category and let $F\colon\I{C}^{\op}\to\Univ$ be a presheaf. Then the Yoneda extension $(h_{\I{C}^{\op}})_!(F)$ is equivalent to the composition
	\begin{equation*}
		\begin{tikzcd}
			\IFun(\I{C},\Univ)\arrow[r, "p^\ast"] & \IFun(\Over{\I{C}}{F},\Univ)\arrow[r, "\colim"] & \Univ,
		\end{tikzcd}
	\end{equation*}
	where $p\colon\Over{\I{C}}{F}\to\I{C}$ is the right fibration that is classified by $F$.
\end{lemma}
\begin{proof}
	As both $p^\ast$ and $\colim$ are cocontinuous functors, the universal property of presheaf $\BB$-categories implies that it suffices to find an equivalence $\colim p^\ast h_{\I{C}^{\op}}\simeq F$. Let us denote by $\Over{h}{F}\colon\Over{\I{C}}{F}\into\Over{\IPSh(\I{C})}{F}$ the functor that is induced by the Yoneda embedding $h_{\I{C}}$ by taking slice $\BB$-categories.  Now there is a commutative diagram
	\begin{equation*}
		\begin{tikzcd}
			\I{C}^{\op}\arrow[r, "h_{\I{C}^{\op}}"]\arrow[d, "h_{\I{C}}^{\op}"] &\IFun(\I{C},\Univ)\arrow[r, "p^\ast"] \arrow[d, "(h_{\I{C}})_!"]&\IFun(\Over{\I{C}}{F},\Univ)\arrow[dr, "\colim", bend left=30]\arrow[d, "(\Over{h}{F})_!"'] &\\
			\IPSh(\I{C})^{\op}\arrow[r, "h_{\IPSh(\I{C})}"] &
			\IFun(\IPSh(\I{C}),\Univ)\arrow[r, "(\pi_F)_!^\ast"] &
			\IFun(\Over{\IPSh(\I{C})}{F},\Univ)\arrow[r, "\id_F^\ast"] & \Univ
		\end{tikzcd}
	\end{equation*}
	in which the commutativity of the right square follows from the straightening equivalence for right fibrations (Theorem~\ref{thm:straightening}) together with $p$ being a right fibration and therefore \emph{proper} in the sense of~\cite[\S~4.4]{MYoneda}, see~\cite[Proposition~4.4.7]{MYoneda}. In light of Yoneda's lemma, it is now immediate that the composition of the left vertical map with the lower row in the above diagram recovers $F$, as desired.
\end{proof}

Recall from~\cite[Example~4.1.11]{MWColimits} that if $\I{C}$ is a $\BB$-category that admits a final object $1_{\I{C}}\colon 1\to\I{C}$, then this object is the limit of the unique diagram $\varnothing\to\I{C}$. In other words, the map $1_{\I{C}}\colon 1\simeq\IFun(\varnothing,\I{C})\to\I{C}$ is right adjoint to the unique functor $\pi_{\I{C}}\colon\I{C}\to 1$. We will denote by $\pi\colon\id_{\I{C}}\to 1_{\I{C}}\pi_{\I{C}}$ the associated adjunction unit.
\begin{lemma}
	\label{lem:formulaMappingGroupoidsSlice}
	If $p\colon\I{P}\to\I{C}$ is a right fibration of $\BB$-categories, the commutative square
	\begin{equation*}
		\begin{tikzcd}[column sep=huge]
			\id_{\IPSh(\I{P})}\arrow[d, "\pi"]\arrow[r, "\eta"] & p^\ast p_!\arrow[d, "p^\ast p_!\pi"]\\
			1_{\IPSh(\I{P})}\pi_{\IPSh(\I{P})}\arrow[r, "\eta 1_{\IPSh(\I{P})}\pi_{\IPSh(\I{P})}"] & p^\ast p_! 1_{\IPSh(\I{P})}\pi_{\IPSh(\I{P})}
		\end{tikzcd}
	\end{equation*}
	is a pullback square in $\IFun(\IPSh(\I{P}),\IPSh(\I{P}))$.
\end{lemma}
\begin{proof}
	By~\cite[Proposition~4.3.2]{MWColimits}, it suffices to show that for every object $F\colon A\to\IPSh(\I{P})$ in context $A\in\BB$ the induced diagram
	\begin{equation*}
		\begin{tikzcd}[column sep=huge]
			F\arrow[d, "\pi"]\arrow[r, "\eta F"] & p^\ast p_!(F)\arrow[d, "p^\ast p_!\pi"]\\
			\pi_A^\ast(1_{\IPSh(\I{C})})\arrow[r, "\eta \pi_A^\ast(1_{\IPSh(\I{C})})"] & p^\ast p_!(\pi_A^\ast(1_{\IPSh(\I{C})}))
		\end{tikzcd}
	\end{equation*}
	is a pullback. Upon replacing $\BB$ with $\Over{\BB}{A}$, we can assume $A\simeq 1$. In light of the straightening equivalence for right fibrations, this diagram corresponds to the commutative square
	\begin{equation*}
		\begin{tikzcd}
			\Over{\I{P}}{F}\arrow[d] \arrow[r] & \Over{\I{P}}{F}\times_{\I{C}}\I{P}\arrow[d]\\
			\I{P}\arrow[r] & \I{P}\times_{\I{C}}\I{P}
		\end{tikzcd}
	\end{equation*}
	of right fibrations over $\I{P}$. As this square is clearly a pullback, the claim follows.
\end{proof}

\begin{lemma}
	\label{lem:descentUniverse}
	Let $\I{I}$ be a $\BB$-category and let
	\begin{equation*}
		\begin{tikzcd}
			d\arrow[d] \arrow[r, "\phi"]& h\arrow[d]\\
			\diag(\I{G})\arrow[r, "\diag(s)"] & \diag(\I{H})
		\end{tikzcd}
	\end{equation*}
	be a pullback square in $\IFun(\I{I},  \Univ)$, where $s\colon \I{G}\to\I{H}$ is an arbitrary map of $\BB$-groupoids. Then the commutative square
	\begin{equation*}
		\begin{tikzcd}
			\colim( d)\arrow[r, "\colim( \phi)"] \arrow[d]& \colim( h)\arrow[d]\\
			\I{G}\arrow[r, "s"] & \I{H}
		\end{tikzcd}
	\end{equation*}
	that is obtained by transposing the first square across the adjunction $\colim\dashv \diag$ is a pullback square as well.
\end{lemma}
\begin{proof}
	In light of the Grothendieck construction, the above pullback square corresponds to a pullback square
	\begin{equation*}
		\begin{tikzcd}
			\Over{\I{I}}{d}\arrow[r]\arrow[d] & \Over{\I{I}}{h}\arrow[d]\\
			\I{G}\arrow[r, "s"] & \I{H}
		\end{tikzcd}
	\end{equation*}
	in $\Cat(\BB)$. 	By~\cite[Proposition~4.4.1]{MWColimits}, we need to show that the groupoidification functor carries this diagram to a pullback square in $\BB$. As $s$ is a right fibration and therefore proper~\cite[Proposition~4.4.7]{MYoneda}, this is immediate.
\end{proof}

\begin{proposition}
	\label{prop:FlatIsAccessible}
	Let $\I{U}$ be a sound internal class and let $\I{C}$ be a $\BB$-category. Then there is an equivalence $\IFlat^{\I{U}}(\I{C})\simeq\IInd^{\I{U}}(\I{C})$ of full subcategories in $\IPSh(\I{C})$.
\end{proposition}
\begin{proof}
	In light of Remarks~\ref{rem:BCInd} and~\ref{rem:UflatLocal}, it will be enough to show that a presheaf $F\colon \I{C}^\op\to\Univ$ defines an object of $\IInd^{\I{U}}(\I{C})$ if and only if it is $\I{U}$-flat. So suppose first that  $F$ is contained in $\IInd^{\I{U}}(\I{C})$. Then $\Over{\I{C}}{F}$ is $\I{U}$-filtered. Let $p\colon \Over{\I{C}}{F}\to\I{C}$ be the projection. By Lemma~\ref{lem:formulaYonedaExtension}, the Yoneda extension $(h_{\I{C}^{\op}})_! F$ can be computed as the composition
	\begin{equation*}
		\begin{tikzcd}
			\IFun(\I{C},\Univ)\arrow[r, "p^\ast"] & \IFun(\Over{\I{C}}{F},\Univ)\arrow[r, "\colim"] & \Univ,
		\end{tikzcd}
	\end{equation*}
	and since both $p^\ast$ and $\colim$ are $\I{U}$-continuous, we deduce that $F$ is $\I{U}$-flat.
	
	Conversely, suppose that $F$ is $\I{U}$-flat. By Lemma~\ref{lem:formulaMappingGroupoidsSlice}, the commutative square
	\begin{equation*}
		\begin{tikzcd}[column sep={17em,between origins}]
			h_{\Over{\I{C}}{F}}\arrow[r, "\eta h_{\Over{\I{C}}{F}}"] \arrow[d, "\pi h_{\Over{\I{C}}{F}}"]& p^\ast p_! h_{\Over{\I{C}}{F}}\arrow[d, "p^\ast p_!\pi h_{\Over{\I{C}}{F}}"]\\
			1_{\IPSh(\Over{\I{C}}{F})}\pi_{\Over{\I{C}}{F}}\arrow[r, "\eta 1_{\IPSh(\Over{\I{C}}{F})}\pi_{\Over{\I{C}}{F}}"] \arrow[d, "\simeq"]& p^\ast p_!1_{\IPSh(\Over{\I{C}}{F})}\pi_{\Over{\I{C}}{F}}\arrow[d, "\simeq"]\\
			\diag_{\Over{\I{C}}{F}}(1_{\IPSh(\Over{\I{C}}{F})})\arrow[r, "\diag_{\Over{\I{C}}{F}}(\eta 1_{\IPSh(\Over{\I{C}}{F})})"] & \diag_{\Over{\I{C}}{F}}(p^\ast F)
		\end{tikzcd}
	\end{equation*}
	is a pullback in $\IFun(\Over{\I{C}}{F},\IPSh(\Over{\I{C}}{F}))$. By~\cite[Proposition~6.1.1]{MWColimits} the composition of the two vertical maps on the left is a colimit cocone, hence so is the composition of the two vertical maps on the right, for $p^\ast p_!$ preserves all colimits. Let $d\colon \I{K}\to (\Over{\I{C}}{F})^\op$ be a diagram with $\I{K}\in \I{U}(1)$. By postcomposition with $\lim_{\I{K}}d^\ast\colon \IPSh(\Over{\I{C}}{F})\to\IFun(\I{K},\Univ)\to \Univ$, the above pullback square induces a cartesian square
	\begin{equation*}
		\begin{tikzcd}
			\lim_{\I{K}}d^\ast h_{\Over{\I{C}}{F}}\arrow[d] \arrow[r] & \lim_{\I{K}}d^\ast p^\ast p_! h_{\Over{\I{C}}{F}}\arrow[d]\\
			\diag_{\Over{\I{C}}{F}}(1_{\Univ})\arrow[r]& \diag_{\Over{\I{C}}{F}}(\lim_{\I{K}}d^\ast p^\ast F).
		\end{tikzcd}
	\end{equation*}
	We claim that the right vertical map in the this last diagram is still a colimit cocone. To see this, note that the equivalence $\IFun(\Over{\I{C}}{F},\IFun(\I{K},\Univ)\simeq\IFun(\I{K},\IFun(\Over{\I{C}}{F},\Univ))$ carries the diagram $d^\ast p^\ast p_! h_{\Over{\I{C}}{F}}$ to the composition
	\begin{equation*}
		\begin{tikzcd}
			\I{K}\arrow[r, "d"] & (\Over{\I{C}}{F})^\op\arrow[r, "p^\op"] & \I{C}^{\op}\arrow[r, "h_{\I{C}^{\op}}"] & \IFun(\I{C},\Univ)\arrow[r, "p^\ast"] & \IFun(\Over{\I{C}}{F},\Univ).
		\end{tikzcd}
	\end{equation*}
	Now the functor $\lim_{\I{K}}$ preserving the colimit of $d^\ast p^\ast p_! h_{\Over{\I{C}}{F}}$ is equivalent to the colimit functor $\colim_{\Over{\I{C}}{F}}$ preserving the limit of the diagram $p^\ast h_{\I{C}^{\op}}p^{\op}d$ (cf.\ the argument in Remark~\ref{rem:dualDefinitionFiltered}). As the functor $p^\ast$ commutes with all limits, this in turn follows once $\colim_{\Over{\I{C}}{F}}p^\ast$ preserves the limit of $h_{\I{C}^{\op}}p^\op d$, which follows from the equivalence $\colim_{\Over{\I{C}}{F}}p^\ast\simeq (h_{\I{C}^{\op}})_!(F)$ from Lemma~\ref{lem:formulaYonedaExtension} and the assumption that $F$ is $\I{U}$-flat. As a consequence, we now deduce from Lemma~\ref{lem:descentUniverse} that the map $\colim_{\Over{\I{C}}{F}}\lim_{\I{K}} d^\ast h_{\Over{\I{C}}{F}}\to 1_{\Univ}$ must be an equivalence. By Proposition~\ref{prop:characterisationWeaklyFiltered}, this precisely means that $\Under{(\Over{\I{C}}{F})}{d^{\op}}^\gp\simeq 1$. As $d$ was chosen arbitrarily and as replacing $\BB$ with $\Over{\BB}{A}$ allows us to derive the same conclusion for any diagram $d\colon A\to \IFun(\I{K},(\Over{\I{C}}{F})^\op)$ in context $A\in\BB$, this shows that $\Over{\I{C}}{F}$ is weakly $\I{U}$-filtered and therefore $\I{U}$-filtered by soundness of $\I{U}$. Hence $F$ is contained in $\IInd^{\I{U}}(\I{C})$.
\end{proof}

\section{Presentable $\BB$-categories}
\label{chap:presentable}

In this section we introduce and study \emph{presentable} $\BB$-categories. Classically, a (locally) presentable $1$-category is one that is locally small and is generated by a small collection of $\kappa$-compact objects under small colimits~\cite{Gabriel1971}. In~\cite[\S~5.5]{htt}, Lurie generalised this concept to $\infty$-categories. In particular, his treatment contains a multitude of equivalent characterisations of presentability~\cite[Theorem~5.5.1.1]{htt}. One of the main goals of this section is to obtain a comparable result for $\BB$-categories. As a starting point, we chose to \emph{define} a presentable $\BB$-category as a \emph{Bousfield} localisation of a presheaf $\BB$-category at a (small) subcategory. To make sense of this, we need to study the notion of \emph{local objects} in a $\BB$-category, which we do in \S~\ref{sec:localObjects}. In \S~\ref{sec:presentability}, we formally define presentable $\BB$-categories and prove our main result about various different characterisations of this condition (Theorem~\ref{thm:characterisationPresentableCategories}), building upon our work on accessible $\BB$-categories. In \S~\ref{sec:adjointFunctorTheorem}, we discuss adjoint functor theorems for presentable $\BB$-categories, and in \S~\ref{sec:PrLB} we construct large $\BB$-categories of presentable $\BB$-categories and show that these are complete and cocomplete. Finally, we discuss the notion of \emph{$\I{U}$-sheaves} in \S~\ref{sec:sheaves}: these are $\I{U}$-continuous functors $\I{C}^{\op}\to\I{D}$, where $\I{C}$ is an $\op(\I{U})$-cocomplete $\BB$-category and $\I{D}$ is a large complete $\BB$-category. We show that if $\I{D}$ is presentable, such $\I{U}$-sheaves form a presentable $\BB$-category as well, and that this provides yet another equivalent characterisation of the notion of presentability.

\subsection{Local objects}
\label{sec:localObjects}
Recall from~\cite[\S~C.1]{MWColimits} the definition of a localisation of a $\BB$-category.
If $j\colon\I{S}\to\I{D}$ is a functor of $\BB$-categories, we obtain a localisation functor $L\colon \I{D}\to\I{S}^{-1}\I{D}=\I{D}\sqcup_{\I{S}}\I{S}^{\gp}$. If $\I{E}$ is an arbitrary $\BB$-category, $L$ satisfies the universal property that $L^\ast\colon\IFun(\I{S}^{-1}\I{D},\I{E})\to\IFun(\I{D},\I{E})$ is fully faithful and identifies the domain with the full subcategory $\IFun(\I{D},\I{E})_{\I{S}}$ that is spanned by those functors $\pi_A^\ast\I{S}\to\pi_A^\ast\I{D}$ whose restriction along $\pi_A^\ast(j)$ factors through the inclusion $\pi_A^\ast\I{D}^\simeq\into\pi_A^\ast\I{D}$. We may now define:
\begin{definition}
	\label{def:localObjects}
	If $\I{S}\to\I{D}$ is a functor, we define the associated \emph{$\BB$-category $\ILoc_{\I{S}}(\I{D})$ of $\I{S}$-local objects in $\I{D}$} as the full subcategory of $\I{D}$ that is defined via the pullback
	\begin{equation*}
		\begin{tikzcd}
			\ILoc_{\I{S}}(\I{D})\arrow[d, hookrightarrow, "i"]\arrow[r, hookrightarrow] & \IPSh(\I{S}^{-1}\I{D})\arrow[d, hookrightarrow, "L^\ast"]\\
			\I{D}\arrow[r, "h", hookrightarrow] & \IPSh(\I{D})
		\end{tikzcd}
	\end{equation*}
	in $\Cat(\BBB)$. We refer to an object $d\colon A\to\I{D}$ as being \emph{$\I{S}$-local} if it is contained in $\ILoc_{\I{S}}(\I{D})$.
\end{definition}
\begin{remark}
	\label{rem:BCLocalObjects}
	If $A\in\BB$ is an arbitrary object, we deduce from~\cite[Lemma~4.2.3 and~Lemma~4.7.14]{MYoneda} that there is a canonical equivalence $\pi_A^\ast\ILoc_{\I{S}}(\I{D})\simeq\ILoc_{\pi_A^\ast\I{S}}(\pi_A^\ast\I{D})$ of full subcategories in $\pi_A^\ast\I{D}$. In particular, this implies that an object $d\colon A\to\I{D}$ is $\I{S}$-local if and only if its transpose $\overline{d}\colon 1_{\Over{\BB}{A}}\to \pi_A^\ast\I{D}$ defines a $\pi_A^\ast\I{S}$-local object.
\end{remark}

\begin{remark}
	\label{rem:localObjectsExplicitly}
	Explicitly, an object $d\colon 1\to \I{D}$ is contained in $\ILoc_{\I{S}}$ precisely if the restriction of the presheaf $h(d)$ along $j$ factors through the inclusion $\Univ^\simeq\into \Univ$, which is the case if and only if for every map $s\colon e\to e^\prime$ in $\I{S}$ in context $A\in\BB$ the morphism  $j(s)^\ast\colon\map{\I{D}}(j(e^\prime),\pi_A^\ast d)\to\map{\I{D}}(j(e),\pi_A^\ast d)$ is an equivalence of $\Over{\BB}{A}$-groupoids (cf.~Proposition~\ref{prop:characterisationMonomorphismMappingGroupoids}). By Remark~\ref{rem:BCLocalObjects}, an analogous description holds for $\I{S}$-local objects in arbitrary context.
\end{remark}
\begin{remark}
	\label{rem:localObjectsSubcategory}
	In the situation of Definition~\ref{def:localObjects}, we deduce from~\cite[Proposition~C.8]{MWColimits} that if $\I{T}\into\I{D}$ is the \emph{$1$-image} of the map $\I{S}\to\I{D}$ (in the sense of~\cite[Definition~B.2.6]{MWColimits}), the canonical map $\I{S}^{-1}\I{D}\to\I{T}^{-1}\I{D}$ is an equivalence. Consequently, the induced map $\ILoc_{\I{T}}(\I{D})\to\ILoc_{\I{S}}(\I{D})$ must be an equivalence as well. Therefore, we may always assume that $\I{S}$ is a \emph{subcategory} of $\I{D}$.
\end{remark}
\begin{remark}
	\label{rem:LocalObjectsGenerators}
	Suppose that $(f_i\colon c_i\to d_i)_{i\in I}$ be a (small) family of maps in $\I{D}$, with $A_i\in\BB$ being the context of $f_i$. By the discussion in~\cite[Appendix~B]{MWColimits}, the subcategory $\I{S}\into\I{D}$ that is \emph{generated} by this family is given by the $1$-image of the induced map $\bigsqcup_i \Delta^1\otimes A_i\to \I{D}$. By combining Remark~\ref{rem:localObjectsExplicitly} and~\ref{rem:localObjectsSubcategory}, an object $d\colon 1\to\I{D}$ is $\I{S}$-local if and for each $i\in I$ the map
	\begin{equation*}
		f_i^\ast\colon\map{\I{D}}(d_i, \pi_{A_i}^\ast d)\to\map{\I{D}}(c_i, \pi_{A_i}^\ast d)
	\end{equation*}
	is an equivalence in $\Over{\BB}{A_i}$.
\end{remark}

The theory of local objects is intimately connected to the notion of \emph{Bousfield localisations}, i.e.\ of reflective subcategories:
\begin{proposition}
	\label{prop:BousfieldLocalisationLocalObjects}
	Let $\I{D}$ be a $\BB$-category and let $L\colon \I{D}\to\I{C}$ be a Bousfield localisation. Let $\I{S}= L^{-1}(\I{C}^\core)\into \I{D}$. Then the inclusion $\colon \I{C}\into\I{D}$ of $L$ induces an equivalence $\I{C}\simeq \ILoc_{\I{S}}(\I{D})$
	of full subcategories in $\I{D}$. Furthermore, if $\I{D}$ is $\I{U}$-accessible and $L$ is a $\I{U}$-accessible Bousfield localisation, there is a \emph{small} subcategory $\I{T}\into\I{S}$ such that $\I{C}\simeq\ILoc_{\I{T}}(\I{D})$.
\end{proposition}
\begin{proof}
	We begin with the first statement. By~\cite[Proposition~3.4.6]{MWColimits}, the functor $L\colon \I{D}\to\I{C}$ identifies $\I{C}$ with the localisation $\I{S}^{-1}\I{D}$. In light of the very definition of $\ILoc_{\I{S}}(\I{D})$, the  claim thus follows once we show that the commutative square
	\begin{equation*}
		\begin{tikzcd}
			\I{C}\arrow[d, hookrightarrow]\arrow[r, "h_{\I{C}}", hookrightarrow] & \IPSh(\I{C})\arrow[d, "L^\ast", hookrightarrow]\\
			\I{D}\arrow[r, hookrightarrow, "h_{\I{D}}"] & \IPSh(\I{D})
		\end{tikzcd}
	\end{equation*}
	is a pullback. Using~\cite[Lemma~4.7.14]{MYoneda}, it will be enough to show that if $F\colon \I{C}^\op\to\Univ$ is a presheaf such that $L^\ast(F)\colon\I{D}^\op\to\Univ$ is representable by an object $d\colon 1\to\Univ$, then $F$ is representable as well. This immediately follows from the computation
	\begin{equation*}
		F\simeq L_! L^\ast F\simeq L_! h_{\I{D}}(d)\simeq h_{\I{C}} L(d),
	\end{equation*} 
	cf.~\cite[Corollary~3.3.3]{MWColimits}. Now if $\I{D}$ is $\I{U}$-accessible and $L$ is a $\I{U}$-accessible Bousfield localisation, let us set $\I{E}=\I{D}^{\cpt{\I{U}}}$ and $\I{T}=i^{-1}(\I{S})\into\I{E}$, where $i\colon\I{E}\into\I{D}$ is the inclusion. Since $\I{E}$ is small by Proposition~\ref{prop:characterisationCompactObjectAccessible}, so is $\I{T}$, and we obtain a full inclusion $\ILoc_{\I{S}}(\I{D})\into\ILoc_{\I{T}}(\I{D})$. We need to show that this is an equivalence. By Remark~\ref{rem:BCLocalObjects}, it will be enough to show that every $\I{T}$-local object $d\colon 1\to\I{D}$ is already $\I{S}$-local. Let $\eta\colon \id_{\I{D}}\to iL$ be the adjunction unit. We then obtain a map
	\begin{equation*}
		\eta^\ast\colon\map{\I{D}}(iL(-),d)\to\map{\I{D}}(-,d),
	\end{equation*}
	and since $d$ is $\I{T}$-local the restriction of $\eta^\ast$ to $\I{E}$ is an equivalence. But as both domain and codomain of this map are $\IFilt_{\I{U}}$-cocontinuous when viewed as functors $\I{D}\to\Univ^\op$, the fact that we have $\I{D}\simeq\IInd^{\I{U}}(\I{E})$ immediately implies that $\eta^\ast$ is already an equivalence, so that $d$ is $\I{S}$-local.
\end{proof}
In the situation of Proposition~\ref{prop:BousfieldLocalisationLocalObjects}, the question naturally arises whether the converse is true: namely, whether the inclusion $i\colon\ILoc_{\I{S}}(\I{D})\into\I{D}$ always defines a Bousfield localisation (i.e.\ admits a left adjoint) for every $\BB$-category $\I{D}$ and every functor $\I{S}\to\I{D}$. In general, this is false, but there is a large class of $\BB$-categories $\I{D}$ and functors $\I{S}\to\I{D}$ for which this is nonetheless the case:
\begin{proposition}
	\label{prop:LocalObjectsBousfieldLocalisation}
	Let $\I{D}$ be an $\Univ$-cocomplete large $\BB$-category that takes values in the $\infty$-category $\LPrS$ of presentable $\infty$-categories. Let furthermore $i\colon \I{S}\to\I{D}$ be a functor where $\I{S}$ is small. Then the inclusion $i\colon\ILoc_{\I{S}}(\I{D})\into\I{D}$ admits a left adjoint and therefore exhibits $\ILoc_{\I{S}}(\I{D})$ as a Bousfield localisation of $\I{D}$. Moreover, this Bousfield localisation is accessible.
\end{proposition}
\begin{proof}
	By~\cite[Remark~C.9]{MWColimits}, we may assume without loss of generality that $\I{S}$ is a subcategory of $\I{D}$, i.e.\ that $i$ is a monomorphism.
	Let us first show that $i(A)\colon \ILoc_{\I{S}}(\I{D})(A)\into \I{D}(A)$ admits a left adjoint for every object $A\in\BB$. Choose a small subcategory of generators $\GG\into \BB$. Then an object $d\colon A\to \I{D}$ is contained in $\ILoc_{\I{S}}(\I{D})(A)$ precisely if for every $g\colon G\to A$ with $G\in\GG$ and every map $s\colon p\to q$ in $\I{S}(G)$ the induced map
	\begin{equation*}
		s^\ast\colon \map{\I{D}(G)}(q,g^\ast d)\to \map{\I{D}(G)}(p,g^\ast d)
	\end{equation*}
	is an equivalence in $\SS$ (cf.~\cite[Corollary~4.6.8]{MYoneda}). As $g^\ast$ admits a left adjoint $g_!$, the object $d$ is thus contained in $\ILoc_{\I{S}}(\I{D})(A)$ if and only if $d$ is local with respect to the set of maps
	\begin{equation*}
		T_A= \bigcup_{G\to A}\{g_!(s)~|~s\in \I{S}(G)^{\Delta^1}\}
	\end{equation*}
	in $\I{D}(A)$. By construction, $T_A$ is a small set, and since $\I{D}(A)$ is by assumption a presentable $\infty$-category, we deduce from~\cite[Proposition~5.5.4.15]{htt} that $i(A)$ admits a left adjoint $L_A$ and that $i(A)$ is accessible.
	
	Next, we show that for every map $p\colon P\to A$ in $\BB$ the natural map $L_B p^\ast\to p^\ast L_A$ is an equivalence. By~\cite[Remark~3.2.10]{MWColimits}, we only need to show that $p^\ast L_G$ sends the adjunction unit of $L_A\dashv i(A)$ to an equivalence. Recall from~\cite[Section~5.5.4]{htt} that the set of maps in $\I{D}(A)$ that is inverted by $L_A$ coincides with the~\emph{strong saturation} of $T_A$, which is the smallest set of maps in $\I{D}(A)$ containing $T_A$ that is stable under pushouts, satisfies the two out of three property and is stable under small colimits in $\I{D}(A)^{\Delta^1}$. Therefore the adjunction unit $\eta$ is contained in the strong saturation of $T_A$, and since $p^\ast$ commutes with colimits (being a morphism in $\LPrS$) we conclude that it will suffice to show that $p^\ast$ sends maps in $T_A$ to maps in the strong saturation of $T_G$. Let us therefore fix a map $g\colon G\to A$ with $G\in\GG$ as well as a map $s\in\I{S}(G)^{\Delta^1}$. Since $\I{D}$ is $\Univ$-cocomplete, we find $p^\ast g_!(s)\simeq h_! q^\ast(s)$, where $h$ and $q$ are defined via the pullback square
	\begin{equation*}
		\begin{tikzcd}
			Q\arrow[r, "h"]\arrow[d, "q"] & P\arrow[d, "b"]\\
			G\arrow[r, "g"] & A.
		\end{tikzcd}
	\end{equation*}
	Now $q^\ast(s)$ is a map in $\I{S}(P)$ and therefore inverted by $L_P$, hence $h_! q^\ast(s)$ is inverted by $L_B$ whenever $h_!$ sends maps in $T_P$ to maps in the strong saturation of $T_B$, which is immediate by definition of $T_P$.
	
	Finally, we may employ Proposition~\ref{prop:existenceAdjointsBeckChevalley} to deduce that $i$ admits a left adjoint $L$. Furthermore, as $\I{D}$ is by assumption both $\Univ$- and $\ILConst$-cocomplete and therefore cocomplete~\cite[Proposition~5.4.1]{MWColimits}, and since every reflective subcategory of a cocomplete $\BB$-category is cocomplete as well~\cite[Proposition~5.2.6]{MWColimits}, the fact that $i$ is section-wise accessible already implies that $i$ is accessible (see Corollary~\ref{cor:AccIsSectWiseAccWhenCocomp}).
\end{proof}

\begin{corollary}
	\label{cor:presentableCategoryBousfieldLocalisation}
	Let $\I{C}$ and $\I{S}$ be (small) $\BB$-categories and let $j \colon \I{S}\to\IPSh(\I{C})$ be a functor. Then there is a sound doctrine $\I{U}$ such that $\ILoc_{\I{S}}(\IPSh(\I{C}))$ is a $\I{U}$-accessible Bousfield localisation of $\IPSh(\I{C})$. Conversely, any $\I{U}$-accessible Bousfield localisation of $\IPSh(\I{C})$ can be identified with the full subcategory $\ILoc_{\I{S}}(\IPSh(\I{C}))\into\IPSh(\I{C})$ for some small $\BB$-category $\I{S}$ and some functor $\I{S}\to\IPSh(\I{C})$.
\end{corollary}
\begin{proof}
	By the straightening equivalence for right fibrations, for any $A\in\BB$ there is a natural equivalence of $\infty$-categories $\IPSh(\I{C})(A)\simeq\RFib(\I{C}\times A)$, and since the right-hand side is a localisation of the presentable $\infty$-category $\Over{\Cat(\BB)}{A\times\I{C}}$ at a small set of objects, we find that $\IPSh(\I{C})$ is section-wise given by a presentable $\infty$-category. Moreover, if $s\colon B\to A$ is a map in $\BB$, the functor $s^\ast\colon \IPSh(\I{C})(A)\to\IPSh(\I{C})(B)$ admits a right adjoint $s_\ast$ by the theory of Kan extensions~\cite[\S~6.3]{MWColimits} and therefore in particular commutes with small colimits. As $\IPSh(\I{C})$ is cocomplete, we are therefore in the situation of Proposition~\ref{prop:LocalObjectsBousfieldLocalisation}, which implies the claim.
\end{proof}

\subsection{Presentability}
\label{sec:presentability}
In this section we define the concept of a presentable $\BB$-category and discuss various characterisations of this notion.
\begin{definition}
	\label{def:presentableCategory}
	A large $\BB$-category is said to be \emph{presentable} if there exist $\BB$-categories $\I{C}$ and $\I{S}$ as well as a functor $\I{S}\to \IPSh(\I{C})$ such that $\I{D}$ is equivalent to $\ILoc_{\I{S}}(\IPSh(\I{C}))$.
\end{definition}

\begin{remark}
	\label{rem:presentableCategoryRelationsSubcategory}
	In the situation of Definition~\ref{def:presentableCategory}, the fact that $\IPSh(\I{C})$ is locally small implies that the $1$-image $\I{S}^\prime$ of the functor $\I{S}\to\IPSh(\I{C})$ (i.e.\ the subcategory of $\IPSh(\I{C})$ that is obtained by factoring the functor $\I{S}\to\IPSh(\I{C})$ into a strong epimorphism and a monomorphism) is small as well. In fact, by combining~\cite[Proposition~C.2]{MWColimits} with~\cite[Proposition~4.7.2]{MYoneda} it is clear that $\I{S}^\prime$ is locally small, hence~\cite[Proposition~4.7.4]{MYoneda} implies that $\I{S}^\prime$ is small whenever $\I{S}^\prime_0$ is contained in $\BB$, which follows in turn from the observation that $\I{S}^\prime$ is a subcategory of the essential image of $\I{S}\to\IPSh(\I{C})$, which is small by~\cite[Lemma~4.7.5]{MYoneda}. As a consequence,~\cite[Remark~C.9]{MWColimits} shows that we may always assume that $\I{S}$ is a subcategory of $\IPSh(\I{C})$.
\end{remark}

\begin{definition}
	\label{def:sectwiseAccAndPres}
	We call a large $ \BB $-category $ \I{C} $ \emph{section-wise accessible} if the associated sheaf takes values in the subcategory $ \Acc \into \CatSS $ of accessible $\infty$-categories.
	Analogously, we call $ \I{C} $ \emph{section-wise presentable} if it factors through the inclusion $ \LPrS \into \CatSS $.
\end{definition}

We now come to the main characterisation of presentable $\BB$-categories. This will require the following lemma:
\begin{lemma}
	\label{lem:endFormulaFunB}
	Let $\CC$ be an $\infty$-category such that $ \BB $ is a left exact accessible localisation of $ \PSh(\CC) $. 
	For any two $\BB$-categories $\I{C}$ and $\I{D}$, there is an equivalence
	\[
	\Fun_\BB(\I{C},\I{D}) \simeq \int_c \Fun(\I{C}(Lc),\I{D}(Lc)).
	\]
	that is natural in $ \I{C},\I{D} \in \Cat(\BB) $.
\end{lemma}
\begin{proof}
	If $i\colon\BB\into\PSh(\CC)$ is the inclusion, the observation that we have a canonical equivalence $ \Fun_{\BB}(-,-)\simeq\Fun_{\PSh(\CC)}(i(-),i(-)) $ implies that we can assume that $ \BB = \PSh(\CC) $.
	In this case, we have an equivalence $ \Cat(\BB) \simeq \PSh_{\CatS}(\CC)$. By \cite[Proposition 2.3]{glasman2016} there is therefore an equivalence
	\[
	\map{\Cat(\BB)}( \I{C},\I{D}) \simeq \int_c \map{\CatS}(\I{C}(c), \I{D}(c))
	\]
	that is natural in $ \I{C} $ and $ \I{D} $.
	Thus, for any $ \infty $-category $ \KK $ we have a chain of natural equivalences
	\begin{align*}
		\map{\CatS}(\KK, \Fun_{\BB}(\I{C},\I{D})) & \simeq \map{\Cat(\BB)}(\KK \otimes \I{C},\I{D}) \\ &\simeq \int_c \map{\CatS}(\KK \times \I{C}(c), \I{D}(c) ) \\
		&\simeq \map{\CatS}(K, \int_c \Fun(\I{C}(c),\I{D}(c)))
	\end{align*}
	which implies that we also get an equivalence
	\[
	\Fun_\BB(\I{C},\I{D} )\simeq  \int_c \Fun(\I{C}(c),\I{D}(c))
	\]
	that is natural in $\I{C}$ and $\I{D}$.
\end{proof}

\begin{theorem}
	\label{thm:characterisationPresentableCategories}
	For a large $\BB$-category $\I{D}$, the following are equivalent:
	\begin{enumerate}
		\item $\I{D}$ is presentable;
		\item there is a $\BB$-category $\I{C}$ and an accessible Bousfield localisation $L
		\colon\IPSh(\I{C})\to\I{D}$;
		\item $\I{D}$ is accessible and cocomplete;
		\item $\I{D}$ is cocomplete, and there is a $\BB$-regular cardinal $\kappa$ such that $\I{D}$ is $\ICat_{\BB}^\kappa$-accessible;
		\item $\I{D}$ is cocomplete and section-wise accessible;
		\item $\I{D}$ is $\Univ$-cocomplete and section-wise presentable.
	\end{enumerate}
\end{theorem}
\begin{proof}
	The fact that~(1) and~(2) are equivalent is an immediate consequence of Corollary~\ref{cor:presentableCategoryBousfieldLocalisation}.
	Now if we assume that~(2) is satisfied, we may find a $\BB$-regular cardinal $\kappa$ such that the inclusion $\I{D}\into\IPSh(\I{C})$ is $\IFilt_{\ICat_{\BB}^\kappa}$-cocontinuous (see Remark~\ref{rem:AccSufficesForKappa}), which by Corollary~\ref{cor:AccessibleLocalisation} implies that $\I{D}$ is $\ICat_{\BB}^\kappa$-accessible. As any reflective subcategory of a cocomplete $\BB$-category is cocomplete as well~\cite[Proposition~5.2.6]{MWColimits}, we conclude that~(4) is satisfied. Trivially,~(4) implies~(3). Lastly, if
	$\I{D}\simeq\IInd^{\I{U}}(\I{C})$ for some sound doctrine $\I{U}$ and some $\BB$-category $\I{C}$ and if $\I{D}$ is furthermore cocomplete, we deduce from~\cite[Corollary~7.1.14]{MWColimits} that the inclusion $\IInd^{\I{U}}(\I{C})\into\IPSh(\I{C})$ admits a left adjoint, hence~(3) implies~(2).
	
	To show that (2) implies (5), since~\cite[Proposition~5.2.6]{MWColimits} already shows that $\I{D}$ is cocomplete, it remains to see that $ \I{D} $ is section-wise accessible. For every $A\in\BB$, the $\infty$-category $\I{D}(A)$ is a Bousfield localisation of the presentable $\infty$-category $\IPSh(\I{C})(A)\simeq \RFib(\I{C}\times A)$. Since Corollary~\ref{cor:AccIsSectWiseAccWhenCocomp} implies that this Bousfield localisation is \emph{accessible}, one concludes that $\I{D}(A)$ is an accessible $\infty$-category.
	Furthermore, since $ \I{D} $ is also complete, the functor $ s^* \colon \I{D}(A) \rightarrow \I{D}(B)$ preserves colimits for any map $ s \colon B \to A $ in $\BB$, so it is in particular accessible.
	Thus $ \I{D} $ is section-wise accessible.
	The fact that (5) implies (6) is an immediate consequence of Proposition~\ref{prop:CocompleteGroupoidsExternal} and \cite[Proposition 5.4.5]{MWColimits}.
	
	To complete the proof, we show that (6) implies (2).
	For this, let $\CC$ be a small $\infty$-category such that there is a left exact and accessible accessible localisation $L \colon \PSh(\CC) \rightarrow \BB$.
	Let $F\colon\CC^{\op}\to\CatSS$ be the composition
	\[
	F \colon \CC^\op \into \PSh(\CC)^{\op} \xrightarrow{L} \BB^\op \xrightarrow{\I{D}} \CatSS.
	\]
	Since $ \CC $ is small we may find a regular cardinal $\kappa$ and a functor $F_0 \colon \CC^\op \rightarrow \CatS$
	such that $ F $ is given by composing $ F_0 $ with the functor $ \Ind_\kappa(-) \colon \CatS \rightarrow \CatSS $. 
	We let $\I{C}$ denote the sheafification of $ F_0$, which is a small $ \BB $-category.
	Let $\I{U}$ be the internal class that is spanned by the constant $ \BB $-categories associated to $ \kappa $-filtered $ \infty $-categories.
	We claim that $\I{D}$ is the free $ \I{U} $-completion of $\I{C}$. Note that by assumption, $\I{D}$ is both $\Univ$ and $\ILConst$-cocomplete, hence $\I{D}$ must be cocomplete (Proposition~\ref{prop:CocompleteGroupoidsExternal}) and therefore a fortiori $\I{U}$-cocomplete. As a consequence, it suffices to verify the universal property.
	Let $\I{E}$ be an arbitrary $ \I{U} $-cocomplete (large) $ \BB $-category. By~\cite[Proposition~5.4.5]{MWColimits} and~\cite[Remark~5.4.9]{MWColimits}, a functor $f \colon \I{D} \rightarrow \I{E}$ is $ \I{U} $-cocontinuous if only if for every $c \in \CC$ the functor $f(c)$
	preserves $\kappa$-filtered colimits (where we slightly abuse notation and implicitly identify $c\in\CC$ with its image along $L\colon\PSh_{\SS}(\CC)\to\BB$).
	Let us write $\Fun_\BB^{\cocont{\I{U}}}(\I{D},\I{E})$ for the global sections of $\IFun^{\cocont{\I{U}}}(\I{D},\I{E})$, i.e. the full subcategory of $ \Fun_\BB(\I{D},\I{E}) $ spanned by the $ \I{U} $-cocontinuous functors.	
	By Lemma~\ref{lem:endFormulaFunB} it follows that we have a chain of equivalences
	\begin{align*}
		\Fun_{\BB}^{\I{U}}(\I{D},\I{E}) &\simeq \int_{c \in \CC} \Fun^\kappa(\I{D}(c),\I{E}(c)) \\
		&\simeq \int_{c \in \CC} \Fun(F_0(c),\I{E}(c)) \\
		& \simeq \Fun_{\PSh(\CC)}(F_0, \I{E}) \\
		& \simeq \Fun_\BB(\I{C},\I{E})
	\end{align*}
	that is natural in $\I{E}$. Using Yoneda's lemma, this already shows that $\I{D}$ is the free $\I{U}$-cocompletion of $\I{C}$.
	In particular, it follows from the explicit description of the free $\I{U}$-cocompletion that we have a commutative triangle of fully faithful functors
	\[
	\begin{tikzcd}
		\I{D} \arrow[r, hook,"j"]                                   & \IPSh(\I{C}) \\
		\I{C} \arrow[u, "i", hook] \arrow[ru, "h"', hook] &                     
	\end{tikzcd}
	\]
	Since $\I{D}$ is cocomplete, the inclusion $j$ admits a left adjoint~\cite[Corollary~7.1.14]{MWColimits}.
	In particular, the inclusion $ j(A) \colon \I{D}(A) \rightarrow \IPSh(\I{C})(A) $ is a right adjoint functor between presentable $\infty$-categories for each object $A\in\BB$, so it follows from \cite[Proposition 5.5.1.2]{htt} that $ j(A) $ is an accessible functor. Hence Corollary~\ref{cor:AccIsSectWiseAccWhenCocomp} implies that $j$ is accessible, which completes the proof.
\end{proof}

We end this section by recording a few consequences of Theorem~\ref{thm:characterisationPresentableCategories}. We begin by noting that as Theorem~\ref{thm:characterisationPresentableCategories} implies that every presentable $\BB$-category is a reflective subcategory of $\IPSh(\I{C})$ for some  $\BB$-category $\I{C}$, we deduce from~\cite[Proposition~5.2.6]{MWColimits}:
\begin{corollary}
	\label{cor:presentableCategoryLimitsColimits}
	Every presentable $\BB$-category is complete and cocomplete.\qed
\end{corollary}

\begin{corollary}
	\label{cor:presentabilityFunctorCategories}
	Let $\I{D}$ be a presentable $\BB$-category and let $\I{K}$ be a $\BB$-category. Then $\IFun(\I{K},\I{D})$ is presentable.
\end{corollary}
\begin{proof}
	By Theorem~\ref{thm:characterisationPresentableCategories}, we may choose a $\BB$-category $\I{C}$ and a sound doctrine $\I{U}$ such that $\I{D}$ is a $\I{U}$-accessible Bousfield localisation of $\IPSh(\I{C})$. In light of~\cite[Proposition~5.2.7]{MWColimits}, this implies that the large $\BB$-category $\IFun(\I{K},\I{D})$ is a $\I{U}$-accessible Bousfield localisation of $\IFun(\I{K},\IPSh(\I{C}))\simeq\IPSh(\I{K}^{\op}\times\I{C})$, hence the result follows.
\end{proof}

\begin{corollary}
	\label{cor:presentabilitySlice}
	Let $\I{D}$ be a presentable $\BB$-category and let $d\colon A\to\I{D}$ be an arbitrary object. Then $\Over{\I{D}}{d}$ is a presentable $\Over{\BB}{A}$-category.
\end{corollary}
\begin{proof}
	We may assume that $A\simeq 1$ (cf.~Remark~\ref{rem:PresentabilityLocalCondition} below).
	Using Corollary~\ref{cor:sliceCategoryUCocomplete}, one finds that $\Over{\I{D}}{d}$ is cocomplete.
	By Theorem~\ref{thm:characterisationPresentableCategories}, it therefore suffices to show that $\Over{\I{D}}{d}$ is accessible, which is a consequence of Proposition~\ref{prop:accessibilitySliceCategory}.
\end{proof}

\begin{corollary}
	\label{cor:BousfieldLocalisationPresentable}
	Let $\I{D}$ be a presentable $\BB$-category and let $\I{S}\to\I{D}$ be a functor where $\I{S}$ is small. Then there is a sound doctrine $\I{U}$ such that $\ILoc_{\I{S}}(\I{D})$ is a $\I{U}$-accessible Bousfield localisation of $\I{D}$. In particular, $\ILoc_{\I{S}}(\I{D})$ is presentable.
\end{corollary}
\begin{proof}
	Since $\I{D}$ is cocomplete by Corollary~\ref{cor:presentableCategoryLimitsColimits} and section-wise presentable by Theorem~\ref{thm:characterisationPresentableCategories}, the claim follows from Proposition~\ref{prop:LocalObjectsBousfieldLocalisation}.
\end{proof}

\begin{remark}
	\label{rem:PresentabilityLocalCondition}
	As yet another consequence of Theorem~\ref{thm:characterisationPresentableCategories}, the condition of a large $\BB$-category to be presentable is a local condition: if $\bigsqcup_i A_i\onto 1$ is a cover in $\BB$, then a $\BB$-category $\I{D}$ is presentable if and only if each $\pi_{A_i}^\ast\I{D}$ is a presentable $\Over{\BB}{A_i}$-category. This follows from condition~(3) in Theorem~\ref{thm:characterisationPresentableCategories}, together with cocompleteness being a local condition (cf.~\cite[Remark~5.2.3]{MWColimits}) and Remark~\ref{rem:AccessibilityLocal}.
\end{remark}
\subsection{The adjoint functor theorem}
\label{sec:adjointFunctorTheorem}
Recall from~\cite[Proposition~5.2.5]{MWColimits} that any left adjoint functor between cocomplete large $\BB$-categories is cocontinuous. Therefore, if $\I{D}$ and $\I{E}$ are cocomplete large $\BB$-categories, there is a canonical inclusion
\begin{equation*}
	\IFun^{\Lad}(\I{D},\I{E})\into\IFun^\cc(\I{D},\I{E}).
\end{equation*}
If $\I{D}$ is presentable and $\I{E}$ is locally small, then this inclusion is in fact an equivalence:
\begin{proposition}[Adjoint functor theorem I]
	\label{prop:adjointFunctorTheorem}
	Let $\I{D}$ and $\I{E}$ be large $\BB$-categories such that $\I{D}$ is presentable and $\I{E}$ is cocomplete and locally small. Then every cocontinuous functor $f\colon \I{D}\to\I{E}$ admits a right adjoint. In particular, there is an equivalence
	\begin{equation*}
		\IFun^{\Lad}(\I{D},\I{E})\simeq\IFun^\cc(\I{D},\I{E})
	\end{equation*}
	of (large) $\BB$-categories.
\end{proposition}
\begin{proof}
	In light of Remark~\ref{rem:PresentabilityLocalCondition}, it is clear that the second statement immediately follows from the first. Now choose $\BB$-categories $\I{C}$ and $\I{S}$ as well as a functor $\I{S}\to\IPSh(\I{C})$ such that $\I{D}\simeq\ILoc_{\I{S}}(\IPSh(\I{C}))$. If $f\colon \I{D}\to\I{E}$ is a cocontinuous functor, then $fL\colon\IPSh(\I{C})\to \I{E}$ is cocontinuous as well and therefore a left adjoint by~\cite[Remark~7.1.4]{MWColimits}.
	To show that $f$ admits a right adjoint, we therefore only need to verify that the right adjoint $r$ of $fL$ factors through $\I{D}$. Since $\I{D}\simeq\ILoc_{\I{S}}(\IPSh(\I{C}))$ as full subcategories of $\IPSh(\I{C})$ by Theorem~\ref{thm:characterisationPresentableCategories}, this is in turn equivalent to $h_{\IPSh(\I{C})}r$ factoring through the functor
	\begin{equation*}
		L^\ast\colon \IPSh(\I{D})\into\IPSh(\IPSh(\I{C})),
	\end{equation*}
	which is clear on account of $r$ being right adjoint to $fL$.
\end{proof}
Recall from Corollary~\ref{cor:BousfieldLocalisationPresentable} that if $\I{D}$ is a presentable $\BB$-category and $\I{S}\to\I{D}$ is a functor where $\I{S}$ is small, the $\BB$-category $\ILoc_{\I{S}}(\I{D})$ is an accessible Bousfield localisation of $\I{D}$ and therefore in particular presentable. We may now use Proposition~\ref{prop:adjointFunctorTheorem} to derive a universal property of $\ILoc_{\I{S}}(\I{D})$ among presentable $\BB$-categories. To that end, recall from~\cite[Appendix~C]{MWColimits} that if $\I{E}$ is another presentable $\BB$-category, we denote by $\IFun(\I{D},\I{E})_{\I{S}}$ the full subcategory of $\IFun(\I{D},\I{E})$ that is spanned by those objects $A\to\IFun(\I{D},\I{E})$ for which the restriction of the associated functor $\pi_A^\ast\I{D}\to\pi_A^\ast\I{E}$ along $\pi_A^\ast\I{S}\to\pi_A^\ast\I{D}$ takes values in the subcategory $\pi_A^\ast\I{E}^\core$. We will denote by $\IFun^\cc(\I{D},\I{E})_{\I{S}}$ its intersection with the full subcategory $\IFun^\cc(\I{D},\I{E})$. We now obtain:
\begin{corollary}
	\label{cor:UniversalPropertyLocalObjects}
	Let $\I{S}\to\I{D}$ be a functor of $\BB$-categories where $\I{S}$ is small and $\I{D}$ is presentable, and let $\I{E}$ be another presentable $\BB$-category. Then precomposition with the left adjoint $L\colon \I{D}\to\ILoc_{\I{S}}(\I{D})$ induces an equivalence
	\begin{equation*}
		\IFun^\cc(\ILoc_{\I{S}}(\I{D}),\I{E})\simeq\IFun^\cc(\I{D},\I{E})_{\I{S}}.
	\end{equation*}
\end{corollary}
\begin{proof}
	To begin with, note that as $L$ is in particular a localisation functor (cf.~\cite[Proposition~3.4.6]{MWColimits}), the universal propery of localisations~\cite[Proposition~C.13]{MWColimits} implies that
	\begin{equation*}
		L^\ast\colon \IFun^\cc(\ILoc_{\I{S}}(\I{D}),\I{E})\to \IFun^\cc(\I{D},\I{E})
	\end{equation*}
	is fully faithful. Therefore, it suffices to identify its essential image with $\IFun^\cc(\I{D},\I{E})_{\I{S}}$. Since the restriction of $L$ along $\I{S}\to\I{D}$ takes values in $\ILoc_{\I{S}}(\I{D})^\core$, it is clear that $L^\ast$ takes values in  $\IFun^\cc(\I{D},\I{E})_{\I{S}}$, so that it suffices to show that every object $A\to \IFun^\cc(\I{D},\I{S})_{\I{S}}$ is contained in the essential image of $L^\ast$. By combining Remark~\ref{rem:BCLocalObjects} with~\cite[Remark~5.3.4 and~Remark~C.11]{MWColimits}, it will be enough to verify that any cocontinuous functor $f\colon\I{D}\to\I{E}$ whose restriction along $\I{S}\to\I{D}$ takes values in $\I{E}^\core$ factors through $L$. Note that the assumption on $f$ precisely means that $f$ factors through the localisation $l\colon \I{D}\to\I{S}^{-1}\I{D}$, so that $f^\ast\colon \IPSh(\I{E})\to\IPSh(\I{D})$ factors through $l^\ast\colon\IPSh(\I{S}^{-1}\I{D})\into\IPSh(\I{D})$. Since $f^\ast\simeq g_!$ where $g$ is the right adjoint of $f$ that is provided by Proposition~\ref{prop:adjointFunctorTheorem}, the very definition of $\ILoc_{\I{S}}(\I{D})$ implies that $g$ factors through the inclusion $i\colon\ILoc_{\I{S}}(\I{D})\into\I{D}$ via a functor $g^\prime\colon\I{E}\to\ILoc(\I{D})$. Since the composite $fi$ defines a left adjoint of $g^\prime$, the claim follows by passing to left adjoints.
\end{proof}

There is also a dual version to Proposition~\ref{prop:adjointFunctorTheorem} that classifies right adjoint functors between presentable $\BB$-categories.
\begin{proposition}[Adjoint functor theorem II]
	\label{prop:adjointfunctorII}
	Let $f \colon \I{D} \rightarrow \I{E}$ be a functor between presentable $\BB$-categories. Then the following are equivalent:
	\begin{enumerate}
		\item $f$ admits a left adjoint;
		\item $f$ is continuous and accessible;
		\item $f$ is continuous and section-wise accessible.
	\end{enumerate}
\end{proposition}
\begin{proof}
	By Corollary~\ref{cor:AccIsSectWiseAccWhenCocomp},~(2) and~(3) are equivalent. Moreover, since Theorem~\ref{thm:characterisationPresentableCategories} implies that $f$ is section-wise given by a functor between presentable $\infty$-categories, the adjoint functor theorem for presentable $\infty$-categories~\cite[Corollary 5.5.2.9]{htt} shows that~(1) implies~(3).
	For the converse, note that the same result implies that $f(A)$ admits a left adjoint $l_A$ for every $A\in\BB$.
	By~Proposition~\ref{prop:existenceAdjointsBeckChevalley}, it now suffices to see that the natural map $l_B s^\ast\to s^\ast l_A$ is an equivalence for every map $s\colon B\to A$ in $\BB$.
	This is equivalent to seeing that the transpose map $f(A)s_\ast\to s_\ast f(B)$ that is given by passing to right adjoints is an equivalence.
	But this is just another way of saying that $f$ is $\Univ$-continuous.
\end{proof}

\subsection{The large $\BB$-category of presentable $\BB$-categories}
\label{sec:PrLB}
Recall from~\cite[\S~5.3]{MWColimits} that we defined the (very large) $\BB$-category $\ICat_{\BBB}^\cc$ of cocomplete large $\BB$-categories as the subcategory of $\ICat_{\BBB}$ which is determined by the subobject of $(\ICat_{\BBB})_1$ that is spanned by the cocontinuous functors between cocomplete large $\Over{\BB}{A}$-categories for every $A\in\BB$. By~\cite[Remark~5.3.2]{MWColimits} a functor of large $\Over{\BB}{A}$-categories is contained in $\ICat_{\BBB}^\cc$ precisely if it is a cocontinuous functor between cocomplete large $\BB$-categories. We may now define:
\begin{definition}
	\label{def:PrL}
	The large $\BB$-category $\ILPr_\BB$ of presentable $\BB$-categories is defined as the full subcategory of $\ICat_{\BBB}^\cc$ that is spanned by the presentable $\Over{\BB}{A}$-categories for every $A\in\BB$. We denote by $\LPr(\BB)$ the $\infty$-category of global sections of $\ILPr_{\BB}$.
\end{definition}

\begin{remark}
	\label{rem:BCPrL}
	As presentability is a local condition (Remark~\ref{rem:PresentabilityLocalCondition}) and by~\cite[Remark~5.3.2]{MWColimits}, a large $\Over{\BB}{A}$-category defines an object in $\ILPr_{\BB}$ if and only if it is presentable, and a functor between such large $\Over{\BB}{A}$-categories is contained in $\ILPr_{\BB}$ if and only if it is cocontinuous. Consequently, the inclusion $\ILPr_{\BB}\into\ICat_{\BBB}$ identifies $\ILPr_{\BB}$ with the sheaf $\LPr(\Over{\BB}{-})$ on $\BB$. In particular, one obtains a canonical equivalence $\pi_A^\ast\ILPr_{\BB}\simeq\ILPr_{\Over{\BB}{A}}$ for every $A\in\BB$.
\end{remark}

\begin{remark}
	\label{rem:SizePrL}
	A priori, $\ILPr_{\BB}$ is a very large $\BB$-category. However, note that  the set of equivalence classes of presentable $\BB$-categories is $\bV$-small as it admits a surjection from the $\bV$-small union
	\begin{equation*}
		\bigsqcup_{\I{C}\in\Cat(\BB)} \Sub_{\text{small}}(\IPSh(\I{C}))
	\end{equation*}
	where $\Sub_{\text{small}}(\IPSh(\I{C}))$ denotes the $\bV$-small poset of \emph{small} subcategories of $\IPSh(\I{C})$. As $\ICat_{\BBB}$ is furthermore locally $\bV$-small, this shows that $\ILPr_{\BB}$ is in fact only a large $\BB$-category.
\end{remark}

Recall from~\cite[\S~6.2]{MCocartesian} that we denote by $\ICat_{\BBB}^\Lad$ the subcategory of $\ICat_{\BBB}$ that is determined by the subobject $L\into(\ICat_{\BBB})_1$ of left adjoint functors. By Proposition~\ref{prop:adjointFunctorTheorem}, the inclusion $\ILPr_{\BB}\into\ICat_{\BBB}$ factors through the inclusion $\ICat_{\BBB}^\Lad\into\ICat_{\BBB}$. Suppose now that $\I{D}$ and $\I{E}$ are presentable $\BB$-categories. By combining Proposition~\ref{prop:morphismsCatB} with the fact that $L\into (\ICat_{\BBB})_1$ is closed under equivalences and composition in the sense of Proposition~\ref{prop:classificationSubcategories} and by furthermore making use of~\cite[Remark~5.3.4]{MWColimits}, we find that the induced inclusion $\map{\ILPr_{\BB}}(\I{D},\I{E})\into\map{\ICat_{\BBB}^\Lad}(\I{D},\I{E})$ is obtained by applying the core $\BB$-groupoid functor to the equivalence
\begin{equation*}
	\IFun^\cc(\I{D},\I{E})\simeq\IFun^\Lad(\I{D},\I{E})
\end{equation*}
from Proposition~\ref{prop:adjointFunctorTheorem}. Upon replacing $\BB$ with $\Over{\BB}{A}$ and using Remark~\ref{rem:BCPrL}, the same assertion holds for objects in $\ILPr_{\BB}$ in context $A\in\BB$, so that we conclude:
\begin{proposition}
	\label{prop:PrLMappingGroupoids}
	The inclusion $\ILPr_{\BB}\into\ICat_{\BBB}^\Lad$ is fully faithful.\qed
\end{proposition}

Dually, let us denote by $\ICat_{\BBB}^\Rad$ the subcategory of $\ICat_{\BBB}$ that is determined by the subobject $R\into(\ICat_{\BBB})_1$ of right adjoint functors.
\begin{definition}
	\label{def:PrR}
	The $\BB$-category $\IRPr_{\BB}$ of presentable $\BB$-categories is defined as the full subcategory of $\ICat_{\BBB}^\Rad$ that is spanned by the presentable $\Over{\BB}{A}$-categories for every $A\in\BB$. We denote by $\RPr(\BB)$ the underlying $\infty$-category of global sections.
\end{definition}

\begin{remark}
	\label{rem:BCPrR}
	As in Remark~\ref{rem:BCPrL}, a large $\Over{\BB}{A}$-category defines an object in $\IRPr_{\BB}$ if and only if it is presentable, and a functor between such large $\Over{\BB}{A}$-categories is contained in $\IRPr_{\BB}$ if and only if it is a right adjoint. As a consequence, the large $\BB$-category $\IRPr_{\BB}$ corresponds to the sheaf $\RPr(\Over{\BB}{-})$ on $\BB$ that is spanned by the presentable $\Over{\BB}{A}$-categories and right adjoint functors. In particular, one obtains a canonical equivalence $\pi_A^\ast\IRPr_{\BB}\simeq\IRPr_{\Over{\BB}{A}}$ for every $A\in\BB$.
\end{remark}

\begin{proposition}
	\label{prop:RPrOpposite}
	There is a canonical equivalence $(\IRPr_{\BB})^\op\simeq\ILPr_{\BB}$ that carries a right adjoint functor between presentable $\BB$-categories to its left adjoint.
\end{proposition}
\begin{proof}
	By~\cite[Proposition~6.2.1]{MCocartesian}, there is such an equivalence $(\ICat_{\BBB}^\Rad)^{\op}\simeq\ICat_{\BBB}^\Lad$, and since this functor necessarily acts as the identity on the underlying core $\BB$-groupoids, it restricts to the desired equivalence by virtue of Proposition~\ref{prop:PrLMappingGroupoids}.
\end{proof}

\begin{example}
	\label{ex:tensorConstructionPresentable}
	We are now in the position to provide a large class of examples of presentable $\BB$-categories:
	recall from~\cite[Construction~A.1]{MWColimits} that there is a functor $-\otimes\Univ\colon \RPrS\to\Cat(\BBB)$ that sends a presentable $\infty$-category $\EE$ to the large $\BB$-category $\EE\otimes\Univ=\EE\otimes\Over{\BB}{-}$ (where $-\otimes -$ is Lurie's tensor product of presentable $\infty$-categories). By~\cite[Example 5.4.8]{MWColimits}, the $\BB$-category $\EE\otimes\Univ$ is cocomplete, so that Theorem~\ref{thm:characterisationPresentableCategories} implies that it is presentable as it takes values in $\LPrS$. Moreover, we deduce from ~\cite[Examples 3.2.12]{MWColimits} that whenever $g\colon \EE\to\EE^\prime$ is a map in $\RPrS$, the induced functor $g\otimes\Univ$ is a right adjoint. Consequently, we conclude that the functor $-\otimes\Univ$ takes values in $\IRPr(\BB)$. In particular, by applying this observation to $\EE=\CatS$, we find that $\ICat_{\BB}$ is presentable.
\end{example}

Our next goal is to show that $\ILPr_{\BB}$ is complete and cocomplete. For completeness, we first need a lemma. To that end, recall from~\cite[\S~5.3]{MWColimits} that we denote by $\ICat_{\BBB}^{\cocont{\Univ}}$ the subcategory of $\ICat_{\BBB}$ that is spanned by the $\Univ[\Over{\BB}{A}]$-cocontinuous functors between $\Univ[\Over{\BB}{A}]$-cocomplete (large) $\BB$-categories. We now find:
\begin{lemma}
	\label{lem:CatUnivLimits}
	The $\BB$-category $\ICat_{\BBB}^{\cocont{\Univ}}$ is $\ILConst$-complete, and the inclusion $\ICat_{\BBB}^{\cocont{\Univ}}\into\ICat_{\BBB}$ is $\ILConst$-continuous.
\end{lemma}
\begin{proof}
	By~\cite[Remark~5.3.2]{MWColimits}, it is enough to show that for any  small $\infty$-category $\KK$ and any functor $d\colon \KK \to \ICat_{\BBB}^{\cocont{\Univ}}\into\ICat_{\BBB}$, the following two conditions are satisfied:
	\begin{enumerate}
		\item $\lim d$ is $\Univ$-cocomplete;
		\item for every $\Univ$-cocomplete large $\BB$-category $\I{C}$, a functor $f\colon \I{C}\to \lim d$ is $\Univ$-cocontinuous if and only if the compositions $\I{C}\to\lim d\to d(k)$ are $\Univ$-cocontinuous for all $k\in\KK$.
	\end{enumerate}
	Recall from~\cite[Corollary~4.7.4.18]{Lurie2017} that the subcategory $\Fun^{\LAdj}(\Delta^1,\CatSS)\into\Fun(\Delta^1,\CatSS)$ that is spanned by the right adjoint functors and the left adjointable squares (i.e.\ those commutative squares of $\infty$-categories whose associated mate transformation is an equivalence) admits small limits and that the inclusion preserves small limits. Let us fix a map $p\colon P\to A$ in $\BB$, and let us denote by $\Cat(\BBB)^{\cocont{\Univ}}$ the $\infty$-category of global sections of $\ICat_{\BBB}^{\cocont{\Univ}}$. Now evaluation at $p$ defines a functor $\Cat(\BBB)\to\Fun(\Delta^1,\CatSS)$ that restricts to a map $\Cat(\BBB)^{\cocont{\Univ}}\to\Fun^{\LAdj}(\Delta^1, \CatSS)$. Since limits in $\Cat(\BBB)$ are computed section-wise, this already shows that $p^\ast\colon\lim d(A)\to\lim d(P)$ admits a left adjoint. Similarly, if $s\colon B\to A$ is a map in $\BB$ and if $q\colon Q\to B$ denotes the pullback of $p$ along $s$,  evaluating large $\BB$-categories at this pullback square yields a morphism $\Delta^1\times\Cat(\BBB)^{\cocont{\Univ}}\to\Fun^{\LAdj}(\Delta^1,\CatSS)$. Consequently, applying $\lim d$ to the very same pullback square must yield a left-adjointable square of $\infty$-categories, which implies that condition~(1) is satisfied. By the same argument, if $\I{C}$ is $\Univ$-cocomplete and if $f\colon \I{C}\to \lim d$ is a functor, evaluating $f$ at $p$ yields a commutative square of $\infty$-categories that is left-adjointable if and only if the evaluation of the composition $\I{C}\to\lim d\to d(k)$ at $p$ is left-adjointable for all $k\in \KK$. Hence~(2) follows.
\end{proof}

\begin{proposition}
	\label{prop:limitsPrL}
	The $\BB$-category $\ILPr_{\BB}$ is complete, and the inclusion $\ILPr_{\BB}\into\ICat_{\BBB}$ is continuous.
\end{proposition}
\begin{proof}
	By the dual of~Proposition~\ref{prop:CocompleteGroupoidsExternal}, it suffices to show that $\ILPr_{\BB}$ is both $\Univ$- and $\ILConst$-complete and that the inclusion $\ILPr_{\BB}$ is both $\Univ$- and $\ILConst$-continuous. Using Remark~\ref{rem:BCPrL}, this follows once we show that whenever $\I{K}$ is either given by the constant $\BB$-category $\Lambda^2_0$ or by a $\BB$-groupoid, the large $\BB$-category $\ILPr_{\BB}$ admits $\I{K}$-indexed limits and the inclusion $\ILPr_{\BB}\into\ICat_{\BBB}$ preserves $\I{K}$-indexed limits.
	
	Let us first assume that $\I{K}=\Lambda^2_0$, i.e.\ suppose that
	\begin{equation*}
		\begin{tikzcd}
			\I{Q}\arrow[d, "q"]\arrow[r, "g"] & \I{P}\arrow[d, "p"]\\
			\I{D}\arrow[r, "f"] & \I{C}
		\end{tikzcd}
	\end{equation*}
	is a pullback diagram in $\Cat(\BBB)$ in which $f$ and $p$ are cocontinuous functors between presentable $\BB$-categories.  By Theorem~\ref{thm:characterisationPresentableCategories}, the cospan determined by $f$ and $p$ takes values in $\LPrS$. Therefore,~\cite[Proposition~5.5.3.13]{htt} implies that $\I{Q}$ takes values in $\LPrS$ and that $g$ and $q$ are section-wise cocontinuous. Moreover, Lemma~\ref{lem:CatUnivLimits} shows that $\I{Q}$ is $\Univ$-cocomplete and that $g$ and $q$ are $\Univ$-cocomplete. By again making use of Theorem~\ref{thm:characterisationPresentableCategories}, we thus conclude that the $\I{Q}$ is presentable and that $g$ and $q$ are cocontinuous. Now if $\I{Z}$ is another presentable $\BB$-category, a similar argumentation shows that a functor $\I{Z}\to \I{Q}$ is cocontinuous if and only if its composition with both $g$ and $q$ are cocontinuous. In total, this shows that $\ILPr_{\BB}$ admits pullbacks and that the inclusion $\ILPr_{\BB}\into\ICat_{\BBB}$ preserves pullbacks.
	
	Let us now assume $\I{K}=\I{G}$ for some $\BB$-groupoid $\I{G}$. In order to show that $\ILPr_{\BB}$ has $\I{G}$-indexed limits and that the inclusion $\ILPr_{\BB}\into\ICat_{\BBB}$ preserves $\I{G}$-indexed limits, another application of Remark~\ref{rem:BCPrL} allows us to reduce to showing that the adjunction $(\pi_{\I{G}})_\ast\dashv\pi_{\I{G}}^\ast))\colon\Cat(\Over{\BBB}{\I{G}})\leftrightarrows\Cat(\BBB)$ restricts to an adjunction between  $\LPr(\Over{\BB}{\I{G}})$ and $\LPr(\BB)$. Recall that on the level of $\CatSS$-valued sheaves, the functor $(\pi_{\I{G}})_\ast$ is given by precomposition with $\pi_{\I{G}}^\ast$. By combining the characterisation of presentable $\BB$-categories as $\Univ$-cocomplete $\LPrS$-valued sheaves (Theorem~\ref{thm:characterisationPresentableCategories}) with the explicit description of $\Univ$-cocompleteness from Example~\ref{ex:UcolimitsGroupoids} and the section-wise characterisation of left adjoint functors (Proposition~\ref{prop:existenceAdjointsBeckChevalley}), it is therefore immediate that $(\pi_{\I{G}})_\ast$ restricts to a functor $\LPr(\Over{\BB}{\I{G}})\to\LPr(\BB)$. Moreover, the adjunction unit $\id_{\Cat(\BBB)}\to(\pi_{\I{G}})_\ast\pi_{\I{G}}^\ast$ is given by precomposition with the adjunction counit $(\pi_{\I{G}})_!\pi_{\I{G}}^\ast\to \id_{\BB}$, and the adjunction counit $\pi_{\I{G}}^\ast(\pi_{\I{G}})_\ast\to\id_{\Cat(\Over{\BBB}{A})}$ is given by precomposition with the adjunction unit $\id_{\Over{\BB}{A}}\to\pi_{\I{G}}^\ast(\pi_{\I{G}})_!$. Thus, by the section-wise characterisation of left adjoint functors and the fact that presentable $\BB$-categories are $\Univ$-cocomplete, these two maps must also restrict in the desired way, hence the result follows.
\end{proof}

\begin{proposition}
	\label{prop:limitsPrR}
	The large $\BB$-category $\IRPr_{\BB}$ is complete, and the inclusion $\IRPr_{\BB}\into\ICat_{\BBB}$ is continuous.
\end{proposition}
\begin{proof}
	As in the proof of Proposition~\ref{prop:limitsPrL}, it suffices to show that for either $\I{K}=\Lambda^2_0$ or $\I{K}=\I{G}$ for $\I{G}$ a $\BB$-groupoid, the large $\BB$-category $\IRPr_{\BB}$ admits $\I{K}$-indexed limits and the inclusion $\IRPr_{\BB}\into\ICat_{\BBB}$ preserves $\I{K}$-indexed limits. The first case follows as in the proof of Proposition~\ref{prop:limitsPrL}, by making use of the dual version of Lemma~\ref{lem:CatUnivLimits},~\cite[Theorem~5.5.3.18]{htt} and the fact that a continuous and section-wise accessible functor between presentable $\BB$-categories admits a left adjoint (Proposition~\ref{prop:adjointfunctorII}). The argument for the second case is carried out in a completely analogous way as the one in the proof of Proposition~\ref{prop:limitsPrL}, the only difference being that one must use the $\Univ$-completeness of presentable $\BB$-categories and not their $\Univ$-cocompleteness.
\end{proof}

\begin{remark}
	\label{rem:PrRGeneratedByPresheafCategories}
	As a consequence of Proposition~\ref{prop:limitsPrR}, we can furthermore deduce that $\IRPr_{\BB}$ is generated under pullbacks by presheaf $\BB$-categories. In fact, if $\I{D}$ is a presentable $\BB$-category, we may find small $\BB$-categories $\I{C}$ and $\I{S}$ and a functor $j\colon\I{S}\to\IPSh(\I{C})$ so that $\I{D}\simeq\ILoc_{\I{S}}(\IPSh(\I{C}))$. By definition of the right-hand side, we therefore obtain a pullback square
	\begin{equation*}
		\begin{tikzcd}
			\I{D}\arrow[d, hookrightarrow]\arrow[r] & \IPSh(\I{S}^\gp)\arrow[d, "\gamma^\ast"]\\
			\IPSh(\I{C})\arrow[r, "j^*h_{\IPSh(\I{C})}"] & \IPSh(\I{S})
		\end{tikzcd}
	\end{equation*}
	in $\Cat(\BBB)$, where $\gamma\colon \I{S}\to\I{S}^\gp$ is the natural map.
	By~\cite[Remark~7.1.4]{MWColimits}, the functor $j^*h_{\IPSh(\I{C})}$ is a right adjoint: its left adjoint is the left Kan extension $(h_{\I{S}})_!(j)$ of $j$ along the Yoneda embedding $h_{\I{S}}$. Since $\gamma^\ast$ is a right adjoint as well, Proposition~\ref{prop:limitsPrR} implies that this diagram is a pullback square in $\IRPr_{\BB}$.
\end{remark}

Finally, by combining Proposition~\ref{prop:limitsPrL} and Proposition~\ref{prop:limitsPrR} with Proposition~\ref{prop:RPrOpposite}, we conclude:
\begin{corollary}
	Both $\ILPr_{\BB}$ and $\IRPr_{\BB}$ are complete and cocomplete.\qed
\end{corollary}

\subsection{$\I{U}$-sheaves}
\label{sec:sheaves}
The main goal in this section is to derived yet another characterisation of presentable $\BB$-categories: that of $\BB$-categories of \emph{$\I{U}$-sheaves} on an $\op(\I{U})$-cocomplete $\BB$-category. These are defined as follows:
\begin{definition}
	\label{def:Usheaves}
	Let $\I{U}$ be an internal class and suppose that $\I{C}$ is an $\op(\I{U})$-cocomplete $\BB$-category. For any (not necessarily small) $\I{U}$-complete $\BB$-category $\I{E}$, we denote by $\IShv_{\I{E}}^{\I{U}}(\I{C})$ the full subcategory of $\IFun(\I{C}^\op, \I{E})$ that is spanned by those presheaves $F\colon A\to\IFun(\I{C}^\op, \I{E})$ (in arbitrary context $A\in\BB$) that are $\pi_A^\ast\I{U}$-continuous when viewed as functors $\pi_A^\ast\I{C}^{\op}\to\pi_A^\ast\I{E}$. We refer to such presheaves as \emph{$\I{U}$-sheaves}. For the case where $\I{U}=\ICat_{\BB}$, we will simply call them \emph{sheaves}, and we will write $\IShv_{\I{E}}(\I{C})=\IShv_{\I{E}}^{\ICat_{\BB}}(\I{C})$ for the associated $\BB$-category
\end{definition}
\begin{remark}
	\label{rem:BCSheaves}
	By~\cite[Remark~5.3.4]{MWColimits}, if $A\in\BB$ is an arbitrary object, we obtain a canonical equivalence $\pi_A^\ast\IShv_{\I{E}}^{\I{U}}(\I{C})\simeq\IShv_{\pi_A^\ast\I{E}}^{\pi_A^\ast\I{U}}(\pi_A^\ast\I{C})$.
\end{remark}

We first focus on $\Univ$-valued $\I{U}$-sheaves. 
\begin{proposition}
	\label{prop:UnivSheavesRepresentables}
	Let $\I{D}$ be a presentable $\BB$-category and let $F\colon \I{D}^{\op}\to\Univ$ be a presheaf on $\I{D}$. Then $F$ is representable if and only if $F$ is continuous. In particular, the Yoneda embedding induces an equivalence $\I{D}\simeq \IShv_{\Univ}(\I{D})$.
\end{proposition}
\begin{proof}
	By Remark~\ref{rem:BCSheaves}, the first claim implies the second, and by~\cite[Proposition~5.2.9]{MWColimits}, every representable functor is continuous, so that it suffices to prove that every continuous presheaf $F\colon\I{C}^\op\to\Univ$ is representable.
	Now $F$ being continuous is equivalent to $F^{\op}\colon \I{D}\to\Univ^{\op}$ being cocontinuous, which by Proposition~\ref{prop:adjointFunctorTheorem} is in turn equivalent to it being a left adjoint. Hence $F$ is continuous if and only if $F$ is a right adjoint. Let $l\colon \Univ\to\I{D}^{\op}$ be the left adjoint of $F$. Since the final $\BB$-groupoid $1_{\Univ}\colon 1\to \Univ$ corepresents the identity on $\Univ$, we find equivalences
	\begin{equation*}
		F\simeq \map{\Univ}(1_{\Univ}, F(-))\simeq \map{\I{D}^{\op}}(l(1_{\Univ}), -)\simeq\map{\I{D}}(-, l(1_{\Univ})),
	\end{equation*}
	hence $F$ is represented by $l(1_{\Univ})$.
\end{proof}

Next, we use Proposition~\ref{prop:UnivSheavesRepresentables} to deduce that whenever $\I{U}$ is a  \emph{doctrine}, the $\BB$-category of $\Univ$-valued $\I{U}$-sheaves on a small $\BB$-category is presentable, and that it satisfies a universal property:
\begin{proposition}
	\label{prop:USheavesPresentable}
	For any doctrine $\I{U}$, the large $\BB$-category $\IShv_{\Univ}^{\I{U}}(\I{C})$ is presentable. Moreover, for any complete large $\BB$-category $\I{E}$, restriction along the Yoneda embedding $h_{\I{C}}$ induces an equivalence
	\begin{equation*}
		h_{\I{C}}^\ast\colon\IShv_{\I{E}}(\IShv_{\Univ}^{\I{U}}(\I{C}))\simeq\IShv_{\I{E}}^{\I{U}}(\I{C})
	\end{equation*}
	of large $\BB$-categories.
\end{proposition}
\begin{proof}
	Fix a small full subcategory $\GG\into\BB$ of generators, and define the small set
	\begin{equation*}
		R=\bigsqcup_{G\in\GG}\left\{f\colon \colim h_{\I{C}}d\to h_{\I{C}}\colim d~\vert~d\colon \I{K}\to\pi_G^\ast\I{C},~\I{K}^{\op}\in\I{U}(G)\right\}
	\end{equation*}
	(where each $f$ is to be considered as a map in $\IPSh(\I{C})$ in context $G\in\GG$). We let $\I{S}_R\into\IPSh(\I{C})$ be the subcategory that is spanned by $R$. Note that since $R$ is a small set, the subcategory $\I{S}_R$ is small, so that $\I{D}=\ILoc_{\I{S}_R}(\IPSh(\I{C}))$ is a presentable $\BB$-category. Moreover, if $\I{E}$ is an arbitrary complete large $\BB$-category, the construction of $\I{S}_R$ (together with the fact that the preservation of limits can be checked locally, see~\cite[Remark~4.2.1]{MWColimits}) makes it evident that a cocontinuous functor $\IPSh(\I{C})\to\I{E}^\op$ carries the maps in $\I{S}_R$ to equivalences precisely if its restriction to $\I{C}$ is $\op(\I{U})$-cocontinuous. By replacing $\BB$ with $\Over{\BB}{A}$, the same assertion holds for any object $A\to\IFun^\cc(\IPSh(\I{C}),\I{E}^\op)$.
	As a consequence, the universal property of presheaf $\BB$-categories implies that restriction along the Yoneda embedding $h_{\I{C}}$ determines an equivalence of large $\BB$-categories $h_{\I{C}}^\ast\colon\IFun^\cc(\IPSh(\I{C}),\I{E}^\op)_{\I{S}_R}\simeq\IFun^{\cocont{\op(\I{U})}}(\I{C},\I{E}^\op)$. Upon taking opposite $\BB$-categories and using Corollary~\ref{cor:UniversalPropertyLocalObjects}, one thus obtains an equivalence $(Lh_{\I{C}})^\ast\colon\IShv_{\I{E}}(\I{D})\simeq \IShv_{\I{E}}^{\I{U}}(\I{C})$. By plugging in $\I{E}=\Univ$ into this equivalence and using proposition~\ref{prop:UnivSheavesRepresentables}, one ends up with an equivalence  $\I{D}\simeq \IShv_{\Univ}^{\I{U}}(\I{C})$ of full subcategories of $\IPSh(\I{C})$, which completes the proof.
\end{proof}

Whenever $\I{U}$ is a \emph{sound} doctrine, we can identify the $\BB$-category of $\I{U}$-sheaves on an $\op(\I{U})$-cocomplete $\BB$-category $\I{C}$ with the free $\IFilt_{\I{U}}$-cocompletion of $\I{C}$:
\begin{proposition}
	\label{prop:SheavesFlatness}
	Let $\I{U}$ be a sound internal class and let $\I{C}$ be an $\op(\I{U})$-cocomplete $\BB$-category. Then there is an equivalence $\IShv_{\Univ}^{\I{U}}(\I{C})\simeq\IInd^{\I{U}}(\I{C})$ of full subcategories of $\IPSh(\I{C})$.
\end{proposition}
\begin{proof}
	On account of Proposition~\ref{prop:FlatIsAccessible} as well as Remarks~\ref{rem:BCSheaves} and~\ref{rem:UflatLocal}, it suffices to show that a presheaf $F\colon\I{C}^{\op}\to\Univ$ is $\I{U}$-flat if and only if $F$ is $\I{U}$-continuous. As the inclusion $h_{\I{C}^{\op}}\colon \I{C}^{\op}\into\IFun(\I{C},\Univ)$ commutes with all limits that exist in $\I{C}$, the presheaf $F$ being $\I{U}$-flat immediately implies that $F$ is $\I{U}$-continuous. Conversely, suppose that $F$ is $\I{U}$-continuous. By Proposition~\ref{prop:FlatIsAccessible}, it suffices to show that $\Over{\I{C}}{F}$ is weakly $\I{U}$-filtered. By applying Lemma~\ref{lem:pullbackCocompleteRightFibration} to the pullback square
	\begin{equation*}
		\begin{tikzcd}
			\Over{\I{C}}{F}\arrow[d] \arrow[r] & \Over{\Univ^{\op}}{1_{\Univ}}\arrow[d]\\
			\I{C}\arrow[r, "F^{\op}"] & \Univ^{\op},
		\end{tikzcd}
	\end{equation*}
	(which satisfies the conditions of the lemma by Proposition~\ref{prop:sliceFibrationColimits}), we conclude that $\Over{\I{C}}{F}$ is $\op(\I{U})$-cocomplete, hence the claim follows from Example~\ref{ex:UCocompleteWeaklyFiltered}.
\end{proof}

\begin{corollary}
	\label{cor:UCocompleteIndUPresentable}
	Let $\I{U}$ be a sound doctrine and let $\I{C}$ be an $\op(\I{U})$-cocomplete $\BB$-category. Then $\IInd^{\I{U}}(\I{C})$ is presentable. Moreover, for any cocomplete large $\BB$-category $\I{E}$, restriction along the Yoneda embedding $h_{\I{C}}$ induces an equivalence
	\begin{equation*}
		h_{\I{C}}^\ast\colon\IFun^\cc(\IInd^{\I{U}}(\I{C}),\I{E})\simeq \IFun^{\cocont{\op(\I{U})}}(\I{C},\I{E})
	\end{equation*}
	of large $\BB$-categories.\qed
\end{corollary}

\begin{corollary}
	\label{cor:characterisationPresentableSheaves}
	Let $\I{D}$ be a large $\BB$-category. Then the following are equivalent:
	\begin{enumerate}
		\item $\I{D}$ is presentable;
		\item there is a sound doctrine $\I{U}$ such that $\I{D}$ is $\I{U}$-accessible and $\I{D}^{\cpt{\I{U}}}$ is $\op(\I{U})$-cocomplete;
		\item there is a doctrine $\I{U}$ and a small $\op(\I{U})$-cocomplete $\BB$-category $\I{C}$ such that one has an equivalence $\I{D}\simeq\IShv_{\Univ}^{\I{U}}(\I{C})$.
	\end{enumerate}
\end{corollary}
\begin{proof}
	By combining Theorem~\ref{thm:characterisationPresentableCategories} with Proposition~\ref{prop:UCompactUcocomplete}, it is clear that~(1) implies~(2). If~(2) is satisfied, Proposition~\ref{prop:characterisationAccessibility} implies that $\I{D}^{\cpt{\I{U}}}$ is small and that there is an equivalence $\I{D}\simeq\IInd^{\I{U}}(\I{D}^{\cpt{\I{U}}})$. In light of Proposition~\ref{prop:SheavesFlatness}, this shows that~(3) is satisfied. Finally, Proposition~\ref{prop:USheavesPresentable} shows that~(3) implies~(1).
\end{proof}
We complete this section by noting that as a consequence of the results that we have established so far, we may deduce that the $\BB$-category of sheaves between presentable $\BB$-categories is presentable as well:
\begin{corollary}
	\label{cor:SheavesPresentable}
	For every two presentable $\BB$-categories $\I{D}$ and $\I{E}$, the $\BB$-category $\IShv_{\I{E}}(\I{D})$ is presentable as well.
\end{corollary}
\begin{proof}
	By Corollary~\ref{cor:characterisationPresentableSheaves}, we may find a doctrine $\I{U}$ and a small $\op(\I{U})$-cocomplete $\BB$-category $\I{C}$ such that $\I{D}\simeq\IShv^{\I{U}}_{\Univ}(\I{C})$. Consequently, Proposition~\ref{prop:USheavesPresentable} gives rise to an equivalence $\IShv_{\I{E}}(\I{D})\simeq\IShv_{\I{E}}^{\I{U}}(\I{C})$. Therefore, it suffices to show that the right-hand side is presentable. Choose a small $\BB$-category $\I{C}^\prime$ such that $\I{E}\simeq\ILoc_{\I{S}^\prime}(\IPSh(\I{C}^\prime))$ for some $\I{S}^\prime\to\IPSh(\I{C^\prime})$ with $\I{S}^\prime$ small. We obtain a commutative square
	\begin{equation*}
		\begin{tikzcd}
			\IShv^{\I{U}}_{\I{E}}(\I{C})\arrow[d, hookrightarrow] \arrow[r, hookrightarrow] &\IFun(\I{C}^\op,\I{E})\arrow[d, hookrightarrow]\\
			\IShv^{\I{U}}_{\IPSh(\I{C}^\prime)}(\I{C})\arrow[r, hookrightarrow] & \IFun(\I{C}^\op, \IPSh(\I{C}^\prime)).
		\end{tikzcd}
	\end{equation*}
	We first claim that this square is a pullback. To see this, note that by Remarks~\ref{rem:BCSheaves} and~\ref{rem:BCLocalObjects}, it will be enough to verify that a functor $\I{C}^\op\to\I{E}$ is $\I{U}$-continuous if $\I{C}^\op\to\I{E}\into\IPSh(\I{C}^\prime)$ is $\I{U}$-continuous. This is a straightforward consequence of the fact that fully faithful functors are conservative. To proceed, note that by Corollary~\ref{cor:presentabilityFunctorCategories}, the vertical map on the right in the above diagram defines a map in $\IRPr_{\BB}$. Using Proposition~\ref{prop:limitsPrR}, the proof is thus complete once we verify that the lower horizontal map is a map in $\IRPr_{\BB}$ as well. To see this, observe that by Lemma~\ref{lem:limitpreservationswap} below, we may identify this map with the inclusion
	\begin{equation*}
		\IFun(\I{C}^\op, \IShv_{\Univ}^{\I{U}}(\I{C}))\into\IFun(\I{C}^\op,\IPSh(\I{C}^\prime))
	\end{equation*}
	that is induced by postcomposition with the inclusion $\IShv_{\Univ}^{\I{U}}(\I{C})\into\IPSh(\I{C})$. As the latter is a map in $\IRPr_{\BB}$ by Proposition~\ref{prop:USheavesPresentable}, the claim thus follows by again appealing to Corollary~\ref{cor:presentabilityFunctorCategories}.
\end{proof}

\section{Digression: Aspects of internal higher algebra}
\label{chap:higherAlgebra}
The goal of this section is to set up the basic framework of higher algebra (in the sense of~\cite{Lurie2017}) in the context of internal higher category theory. As our main goal is to use this framework to define tensor products of $\BB$-categories in~\S~\ref{chap:tensorProductBCategories}, we will restrict our attention to a few selected results instead of giving a comprehensive account of this machinery. We begin in \S~\ref{sec:BOperads} by defining symmetric monoidal $\BB$-categories (and more generally $\BB$-operads), and in~\S~\ref{sec:algebrasModules} we study algebras and modules in symmetric monoidal $\BB$-categories. As a first application of this framework, \S~\ref{sec:dualisable} contains a characterisation of dualisable objects in the $\BB$-category of modules over an $\mathbb{E}_\infty$-ring in $\BB$. The proof of this characterisation requires a few results about the notion of \emph{stability} in the world of $\BB$-categories, which we briefly discuss in \S~\ref{sec:stability}.

\subsection{$\BB$-operads, symmetric monoidal $\BB$-categories and commutative monoids}
\label{sec:BOperads}
Recall from~\cite[\S~2.3.2]{Lurie2017} the definition of the presentable $\infty$-category $\Op^\gen_\infty$ of generalised $\infty$-operads. By construction, this $\infty$-category is the (non-full) subcategory of $\Over{(\CatS)}{\Fin_*}$ (where $\Fin_*$ denotes the $1$-category of finite pointed sets) that is spanned by the generalised $\infty$-operads and the morphisms of generalised $\infty$-operads. The full subcategory of $\Op_\infty^\gen$ that is spanned by the $\infty$-operads is denoted by $\Op_\infty$. By~\cite[Corollary~2.3.2.6]{Lurie2017} the inclusion $\Op_\infty\into\Op_\infty^\gen$ admits a left adjoint, and by~\cite[\S~2.1.4]{Lurie2017} the $\infty$-category $\Op_\infty$ is presentable as well. Finally, recall from~\cite[Variant~2.1.4.13]{Lurie2017} that the (presentable) $\infty$-category $\CatS^\otimes$ is defined to be the subcategory of $\Op_\infty$ that is spanned by the symmetric monoidal $\infty$-categories and the symmetric monoidal functors. By~\cite[Proposition~2.2.4.9]{Lurie2017} the inclusion $\CatS^\otimes\into\Op_\infty$ also admits a left adjoint. In light of~\cite[Construction~A.1]{MWColimits} we may now define:
\begin{definition}
	\label{def:OperadsSMCategories}
	A \emph{generalised $\BB$-operad} is an $\Op_\infty^\gen$-valued sheaf on $\BB$, and the  $\BB$-category of generalised $\BB$-operads is defined as the large $\BB$-category $\IOp_\BB^\gen =\Op_\infty^\gen\otimes\Univ$. A generalised $\BB$-operad is said to be a \emph{$\BB$-operad} if it takes values in $\Op_\infty$, and the large $\BB$-category  $\IOp_\BB$ of $\BB$-operads is defined as the full subcategory of $\IOp_\BB^\gen$ that is spanned by the $\Over{\BB}{A}$-operads for each $A\in\BB$. Finally, a (generalised) $\BB$-operad is said to be a \emph{symmetric monoidal $\BB$-category} if it takes values in $\CatS^\otimes$, and similarly a morphism of symmetric monoidal $\BB$-categories is said to be a symmetric monoidal functor if it takes values in $\CatS^\otimes$. The large $\BB$-category $\ICat_{\BB}^\otimes$ of symmetric monoidal $\BB$-categories is defined as the subcategory of $\IOp_\BB$ that is spanned by the symmetric monoidal functors between symmetric monoidal $\BB$-categories, and the large $\BB$-category $\ICat_{\BB}^{\otimes,\lax}$ is defined as the essential image of the inclusion $\ICat_{\BB}^\otimes\into\IOp_\BB$. We refer to the maps in $\ICat_{\BB}^{\otimes,\lax}$ as \emph{lax symmetric monoidal functors}.
\end{definition}
\begin{remark}
	By construction, we have canonical equivalences $\IOp_\BB\simeq\Op_\infty\otimes\Univ$ and $\ICat_{\BB}^\otimes\simeq\CatS^\otimes\otimes\Univ$ with respect to which the inclusions $\IOp_\BB\into\IOp_\BB^\gen$ and $\ICat_{\BB}^\otimes\into\IOp_\BB$ correspond to the image of the inclusions $\Op_\infty\into\Op_\infty^\gen$ and $\CatS^\otimes\into\Op_\infty$ under the functor $-\otimes\Univ\colon\RPrS\to\RPr(\BB)$ from Example~\ref{ex:tensorConstructionPresentable}.  Consequently, the chain of inclusions
	\begin{equation*}
		\ICat_{\BB}^\otimes\into\IOp_\BB\into\IOp_\BB^\gen
	\end{equation*}
	defines morphisms in $\IRPr_{\BB}$.
\end{remark}
\begin{remark}
	\label{rem:OperadsForgetfulFunctor}
	Taking the fibre over $\ord{1}\in\Fin_*$ defines a forgetful functor $\Op_\infty\to\CatS$. By the discussion in~\cite[\S~2.1.4]{Lurie2017}, this functor defines a map in $\RPrS$. Consequently, applying $-\otimes\Univ$ to this map yields a well-defined morphism $\IOp_\BB\to\ICat_{\BB}$ in $\IRPr_{\BB}$ (cf.~\ref{ex:tensorConstructionPresentable}). Given a $\BB$-operad $\I{O}^\otimes$, we denote its image under this functor by $\I{O}$.
\end{remark}

\begin{remark}
	\label{rem:OperadsExplicitly}
	By making use of~\cite[Proposition~B.2.9]{Lurie2017}, the inclusion $\Op_\infty^\gen\into\Over{(\CatS)}{\Fin_*}$ admits a left adjoint and thus defines a morphism in $\RPrS$. Upon applying the functor $-\otimes\Univ$, we thus obtain a monomorphism $\IOp_\BB^\gen\into\Over{(\CatS)}{\Fin_*}\otimes\Univ\simeq\Over{(\ICat_{\BB})}{\Fin_*}$. Unwinding the definitions, one finds that a functor $p\colon \I{O}^\otimes\to \Fin_*$ is contained in $\IOp_\BB^\gen$ precisely if
	\begin{enumerate}
		\item for all $A\in\BB$ the pullback $(\eta^A)^\ast\I{O}^\otimes(A)\to\Fin_{\ast}$ of $p(A)$ along the adjunction unit $\eta^A\colon\Fin_*\to\Gamma_{\Over{\BB}{A}}\Fin_*$ defines a generalised $\infty$-operad, and
		\item for all maps $s\colon B\to A$ in $\BB$ the induced map $(\eta^A)^\ast \I{O}^\otimes(A)\to (\eta^B)^\ast\I{O}^\otimes(B)$ defines a morphism of generalised $\infty$-operads.
	\end{enumerate}
	Similarly, a morphism $f\colon \I{O}^\otimes\to {\I{O}^\prime}^\otimes$ in $\Over{(\ICat_{\BB})}{\Fin_*}$ between generalised $\BB$-operads is contained in $\IOp_\BB^\gen$ precisely if for all $A\in\BB$ the pullback $(\eta^A)^\ast f(A)$ defines a morphism of generalised $\infty$-operads. One can make analogous observations for the subcategories $\IOp_\BB$, $\ICat_{\BB}^{\otimes,\lax}$ and $\ICat_{\BB}^\otimes$.
\end{remark}

\begin{remark}
	\label{rem:MonoidalBCategoriesExplicitly}
	By combining Remark~\ref{rem:OperadsExplicitly} with~\cite[Proposition~3.1.7]{MCocartesian}, we may identify $\ICat_{\BB}^\otimes$ with the full subcategory of $\ICocart_{\Fin_*}$ that is spanned by those cocartesian fibrations $p\colon\I{C}^\otimes\to \Fin_*$ (in arbitrary context $A\in\BB$) for which the induced cocartesian fibration $(\eta^B)^\ast\I{C}^\otimes(B)\to\Fin_*$ defines a symmetric monoidal $\infty$-category for every $B\in\Over{\BB}{A}$.
\end{remark}

For later use, recall from~\cite[Proposition~2.3.2.9]{Lurie2017} that the functor $\CC\mapsto\Fin_{\ast}\times\CC$ determines a fully faithful map $\CatS\into\Op_\infty^\gen$ in $\RPrS$. By applying the functor $-\otimes\Univ$, we thus obtain:
\begin{proposition}
	The map $\I{C}\mapsto\Fin_{\ast}\times\I{C}$ determines a fully faithful functor $\ICat_{\BB}\into\IOp_\BB^\gen$.\qed
\end{proposition}

Recall that under the straightening equivalence, symmetric monoidal $\infty$-categories can be identified with commutative monoids in $\CatS$. Our next goal is to derive an analogous result internally. To that end, we let $\rho_i\colon \ord{n}\to \ord{1}$ be the pointed map that carries $i$ to $1$ and every other element to $0$. Given any functor $C_\bullet\colon \Fin_*\to\CC,~\ord{n}\mapsto C_n$ with values in an arbitrary $\infty$-category $\CC$, we will denote by $p_i\colon C_n\to C_1$ the image of $\rho_i$ under $C_\bullet$. We may now define:
\begin{definition}
	\label{def:CMon}
	A \emph{commutative monoid} in a $\BB$-category $\I{C}$ is a functor $M\colon\Fin_\ast\to\I{C}$ such that for all $n\geq 0$ the functors $(p_i\colon M_n\to M_1)_{1\leq i\leq n}$ exhibit $M_n$ as the product $\prod_{i=1}^n M_1$ in $\I{C}$. We define the $\BB$-category $\ICMon(\I{C})$ as the full subcategory of $\IFun(\Fin_*,\I{C})$ that is spanned by those functors $\Fin_*\to \pi_A^\ast\I{C}$ that define commutative monoids in $\pi_A^\ast\I{C}$ for all $A\in\BB$.
\end{definition}

\begin{remark}
	\label{rem:BCMonoids}
	Since the condition of a morphism in a $\BB$-category to be an equivalence is local in $\BB$, an object $A\to \IFun(\Fin_*,\I{C})$ is contained in $\ICMon(\ICat_{\BB})$ if and only if it defines a commutative monoid in $\pi_A^\ast\I{C}$. In particular, one obtains a canonical equivalence $\pi_A^\ast\ICMon(\I{C})\simeq\ICMon(\pi_A^\ast\I{C})$ for every $A\in\BB$.
\end{remark}

\begin{remark}
	Evaluation at $\ord{1}\colon 1\to\Fin_*$ defines a forgetful functor $\ICMon(\I{C})\to\I{C}$. We will usually abuse notation and identify a commutative monoid $M$ with its underlying object in $\I{C}$. By evaluating such a monoid $M$ at the unique {active} map $\ord{2}\to\ord{1}$ in $\Fin_*$, one obtains a multiplication map $\mu \colon M \times M \rightarrow M$
	that is associative and commutative up to infinite coherent homotopies. Moreover, evaluation at the unique pointed map $\ord{0}\to\ord{1}$ induces a unit $e\colon 1_{\I{C}}\to M$ such that the induced maps $\mu(e,-)$ and $\mu(-,e)$ are equivalences.
\end{remark}

By combining Remark~\ref{rem:MonoidalBCategoriesExplicitly} with the straightening equivalence for cocartesian fibrations~\cite[Theorem~6.3.1]{MCocartesian}, one now finds:
\begin{proposition}
	\label{prop:MonoidalCategoriesVsMonoids}
	The straightening equivalence restricts to an equivalence of large $\BB$-categories $\ICat_{\BB}^\otimes\simeq\ICMon(\ICat_{\BB})$.\qed
\end{proposition}

\begin{remark}
	Upon composing the equivalence $\ICat_{\BB}^\otimes\simeq\ICMon(\ICat_{\BB})$ from Proposition~\ref{prop:MonoidalCategoriesVsMonoids} with the forgetful functor $\ICMon(\ICat_{\BB})\to\ICat_{\BB}$, one recovers the functor $\ICat_{\BB}^\otimes\to\ICat_{\BB}$ from Remark~\ref{rem:OperadsForgetfulFunctor}.
	If $\I{C}^\otimes$ is a symmetric monoidal $\BB$-category, the multiplication map thus defines a bifunctor $-\otimes -\colon\I{C}\times\I{C}\to\I{C}$, and the unit is an object $1_\otimes\colon 1\to\I{C}$.
\end{remark}

A rich source for symmetric monoidal $\BB$-categories are those $\BB$-categories that admit finite products. We will denote by $\ICat_{\BB}^{\Pi}$ the subcategory of $\ICat_{\BB}$ that is spanned by the $\ILConst_{\Fin}$-continuous functors between $\ILConst_{\Fin}$-complete $\Over{\BB}{A}$-categories for all $A\in\BB$, where $\Fin$ is the $1$-category of finite sets. 
By the dual of~\cite[Corollary~5.4.6]{MWColimits} and the fact that $\CatS^\Pi$ is presentable~\cite[Lemma~4.8.4.2]{Lurie2017}, we may identify $\ICat_{\BB}^\Pi\simeq\CatS^
\Pi\otimes\Univ$.
Now recall from~\cite[Corollary~2.4.1.9]{Lurie2017} that there is a fully faithful functor $(-)^\times\colon\CatS^\Pi\into\CatS^\otimes$ that assigns to an $\infty$-category $\CC$ with finite products the associated \emph{cartesian} monoidal $\infty$-category $\CC^\times$ and that fits into a commutative diagram
\begin{equation*}
	\begin{tikzcd}[column sep={4em,between origins}]
		\CatS^\Pi\arrow[rr, hookrightarrow, "(-)^\times"]\arrow[dr, hookrightarrow] && \CatS^\otimes\arrow[dl]\\
		&\CatS &
	\end{tikzcd}
\end{equation*}
in which the right diagonal map is the forgetful functor. Note that this diagram takes values in $\RPrS$.  Upon applying the functor $-\otimes\Univ$, we thus obtain:
\begin{proposition}
	\label{prop:CartesianMonoidalBCategories}
	There is a fully faithful functor $(-)^\times\colon\ICat_{\BB}^\Pi\into\ICat_{\BB}^\otimes$ that carries a $\BB$-category $\I{C}$ to the \emph{cartesian} monoidal $\BB$-category $\I{C}^\times$ and that fits into a commutative diagram
	\begin{equation*}
		\begin{tikzcd}[column sep={4em,between origins}]
			\ICat_{\BB}^\Pi\arrow[rr, hookrightarrow, "(-)^\times"]\arrow[dr, hookrightarrow] && \ICat_{\BB}^\otimes\arrow[dl]\\
			&\ICat_{\BB}&
		\end{tikzcd}
	\end{equation*}
	of large $\BB$-categories.
\end{proposition}

\begin{remark}
	\label{rem:SymMonStronSubcats}
	Suppose that $ \I{C}^\otimes $ is a symmetric monoidal $ \BB $-category and that $ \I{D} $ is a full subcategory of $ \I{C} $ such that the tensor functor $- \otimes - \colon \I{C} \times \I{C} \to \I{C} $ restricts to a functor $ \I{D} \times \I{D}\to \I{D} $ and the unit object $1_\otimes\colon 1\to\I{C}$ is contained in $\I{D}$.
	Then we can canonically equip $ \I{D} $ with the structure of a symmetric monoidal $ \BB $-category:
	indeed, $\I{C}^\otimes $ is determined by a functor $ \I{C}^\otimes(-,-) \colon \BB^\op \times \Fin_* \to \CatS $ and we may consider the full subfunctor that consists for $ A \in \BB^\op $ and $ \left< n \right> \in \Fin_*$ of the full subcategory of $ \I{C}^\otimes(A,\left< n \right>) $ that corresponds under the equivalence
	\[
	\I{C}^\otimes(A,\left< n \right>) \simeq 	\prod_{i=1}^n \I{C}^\otimes(A,\left< 1 \right>) \simeq \prod_{i=1}^n \I{C}(A)
	\]
	to the subcategory $ \prod_{i=1}^n \I{D}(A) \subseteq  \prod_{i=1}^n \I{C}(A)$.
	By our assumption on $ - \otimes - $ this yields a well-defined functor $ \I{D}^\otimes (-,-) \colon \BB^\op \times \Fin_* \rightarrow \CatS $, and by \cite[Remark 2.2.1.12]{Lurie2017} the functor $ \I{D}^\otimes (A,-) $ defines a symmetric monoidal $\infty$-category for every $ A\in \BB $.
	So we get a functor $\I{D}^\otimes  \colon \BB^\op \rightarrow \CatS^\otimes$
	whose composition with the forgetful functor $ \CatS^\otimes \rightarrow \CatS $ recovers $ \I{D} $ and which therefore defines the desired monoidal structure on $ \I{D} $.
\end{remark}

\subsection{Algebras and modules in symmetric monoidal $\BB$-categories}
\label{sec:algebrasModules}
Recall from~\cite[\S~2.1.3]{Lurie2017} that  a \emph{commutative algebra} in a symmetric monoidal $\infty$-category $\CC^\otimes$ is a map of $\infty$-operads $\Fin_*\to \CC^\otimes$, and an \emph{associative algebra} in $\CC^\otimes$ is a map of $\infty$-operads $\Assoc\to \CC^\otimes$, where $\Assoc$ denotes the associative operad. One obtains functors $\Alg\colon \CatS^\otimes\to \CatS$ and $\CAlg\colon \CatS^\otimes\to\CatS$ that assign to a symmetric monoidal $\infty$-category $\CC^\otimes$ the $\infty$-category $\Alg(\CC)$ of associative algebras  in $\CC^\otimes$ and the $\infty$-category $\CAlg(\CC)$ of commutative algebras in $\CC^\otimes$, respectively. Note that either of these functors defines a map in $\RPrS$. In light of Example~\ref{ex:tensorConstructionPresentable}, we may thus define:
\begin{definition}
	\label{def:algebras}
	We define the maps $\IAlg\colon\ICat_{\BB}^\otimes\to\ICat_{\BB}$
	and $\ICAlg\colon\ICat_{\BB}^\otimes\to\ICat_{\BB}$ in $\RPr(\BB)$ as the maps that arise from applying $-\otimes\Univ$ to the functors $\Alg\colon\CatS^\otimes\to\CatS$ and $\CAlg\colon\CatS^\otimes\to\CatS$, respectively. For a symmetric monoidal $\BB$-category, we refer to the $\BB$-category $\IAlg(\I{C})$ ($\ICAlg(\I{C})$) as the $\BB$-category of \emph{associative (commutative) algebras} in $\I{C}$.
\end{definition}

By construction, an associative (commutative) algebra in $\I{C}^\otimes$ is by definition simply an associative (commutative) algebra in the symmetric monoidal $\infty$-category $\Gamma(\I{C}^\otimes)$.

Recall from~\cite[Proposition~2.4.2.5]{Lurie2017} that there is a commutative triangle
\begin{equation*}
	\begin{tikzcd}[column sep=small]
		\CatS^\Pi\arrow[dr, "\CMon"'] \arrow[rr, hook, "(-)^\times"] && \CatS^\otimes\arrow[dl, "\CAlg"] \\
		& \CatS &
	\end{tikzcd}
\end{equation*}
in $\RPrS$. By applying $-\otimes\Univ$ to this diagram, we therefore obtain:
\begin{proposition}
	Let $\I{C}$ be a $\BB$-category with finite products. Then there is a canonical equivalence
	\begin{equation*}
		\ICMon(\I{C})\simeq \ICAlg(\I{C}^\times)
	\end{equation*}
	of $\BB$-categories that is natural in $\I{C}$.\qed
\end{proposition}
Our next goal is to define $\BB$-categories of \emph{modules} in a symmetric monoidal $\BB$-category. 
To that end, let us briefly recall the setup from~\cite[\S~3.3.3]{Lurie2017}: we let $\KK$ be the full subcategory of $\Fun(\Delta^1,\Fin_*)$ that is spanned by the semi-inert maps, and we denote by $\KK^0\into\KK$ the full subcategory that is spanned by the null maps. We say that a morphism $f$ in $\KK$ or $\KK^0$ is inert if both $d_0(f)$ and $d_1(f)$ are inert. By making use of the projections $d_1\colon\KK\to\Fin_*$ and $d_1\colon \KK^0\to\Fin_*$, we may regard both $\KK$ and $\KK^0$ as preoperads (with the marked edges given by the inert maps). Now the two spans $\Fin_*\xleftarrow{d_1}\KK\xrightarrow{d_0}\Fin_*$ and $\Fin_*\xleftarrow{d_1}\KK^0\xrightarrow{d_0}\Fin_*$ (viewed as spans of marked simplicial sets) satisfy the conditions of~\cite[Theorem~B.4.2]{Lurie2017} and therefore determine left Quillen endofunctors $-\times_{\Fin_{\ast}}\KK$ and $-\times_{\Fin_*}\KK^0$ on the $1$-category $\POp_\infty$ of preoperads with respect to the model structure for generalised $\infty$-operads (see in particular~\cite[Proposition~3.3.3.18]{Lurie2017} for a proof of the first case, the second case follows by analogous arguments). Passing to right adjoints thus gives rise to right Quillen endofunctors $\overline{\Mod}(-)^\otimes$ and $\prescript{\mathsf{p}}{}{\CAlg(-)}$ on $\POp_\infty$, and the inclusion $\KK^0\into\KK$ determines a morphism $\overline{\Mod}(-)^\otimes\to \prescript{\mathsf{p}}{}{\CAlg(-)}$. Similarly, the span $\Fin_*\leftarrow\Fin_*\times\Fin_*\to\Fin_*$ of marked simplicial sets satisfies the conditions of~\cite[Theorem~B.4.2]{Lurie2017} as well and therefore gives rise to a left Quillen endofunctor whose right adjoint recovers the functor $\Fin_*\times\CAlg(-)\colon\POp_\infty\to\POp_\infty$. By~\cite[Remark~3.3.3.7]{Lurie2017}  the map $\KK^0\to\Fin_{\ast}\times\Fin_{\ast}$ induces a categorical equivalence $\Fin_{\ast}\times\CAlg(-)\to\prescript{\mathsf{p}}{}{\CAlg(-)}$ and therefore in particular a weak equivalence in the model structure for generalised $\infty$-operads. 

In total, these observations imply that upon passing to the underlying $\infty$-categories, one ends up with a map $\Mod(-)^\otimes\colon\Op_\infty^\gen\to\Op_\infty^\gen$ in $\RPrS$ together with a morphism $\Mod(-)^\otimes\to\Fin_*\times\CAlg(-)$, or equivalently a map $\Op_\infty^\gen\to\Fun(\Delta^1,\Op_\infty^\gen)$ in $\RPrS$.  
By making use of the evident equivalence of large $\BB$-categories $\Fun(\Delta^1,\Op_\infty^\gen)\otimes\Univ\simeq (\IOp_\BB^\gen)^{\Delta^1}$, applying $-\otimes\Univ$ yields a functor $\IMod(-)^\otimes\colon\IOp_\BB^\gen\to\IOp_\BB^\gen$ in $\IRPr_{\BB}$ together with a morphism $p\colon\IMod(-)^\otimes\to\Fin_*\times\ICAlg(-)$ of generalised $\BB$-operads. Given a symmetric monoidal $\BB$-category $\I{C}^\otimes$ and a commutative algebra $R\colon 1\to\ICAlg(\I{C})$, we will denote by $\IMod_R(\I{C})^\otimes\to\Fin_{\ast}$ the pullback of $p$ along the map $(\id,R) \colon \Fin_*\to \Fin_*\times\ICAlg(\I{C})$. By~\cite[Theorem~3.3.3.9]{Lurie2017} and Remark~\ref{rem:OperadsExplicitly}, the map $\IMod_R(\I{C})^\otimes\to\Fin_*$ defines a $\BB$-operad.
\begin{definition}
	Let $\I{C}$ be a symmetric monoidal $\BB$-category and let $R\colon 1\to\ICAlg(\I{C})$ be a commutative algebra in $\I{C}$. We define the $\BB$-category $\IMod_R(\I{C})$ of \emph{modules} over $R$ as the underlying $\BB$-category of the $\BB$-operad $\IMod_R(\I{C})^\otimes$.
\end{definition}

Our next goal is to investigate the functoriality of $\IMod_R(\I{C})^\otimes$ in $R$. To that end, note that the diagonal embedding $\Fin_{\ast}\into \KK$ induces a forgetful functor $\overline{\Mod}(-)^\otimes\to \id_{\POp_\infty}$ and therefore, by the same procedure as above, a morphism $\IMod(-)^\otimes\to\id_{\IOp_\BB^\gen}$. By combining~\cite[Corollary~3.4.3.4]{Lurie2017} with the section-wise characterisation of cartesian fibrations in $\Cat(\BB)$ (see~\cite[Proposition~3.1.7]{MCocartesian}) as well as~\cite[Remark~3.2.6]{MCocartesian}, we now find:
\begin{proposition}
	\label{prop:ModCAlgCartesianFibration}
	For any symmetric monoidal $\BB$-category $\I{C}$, the projection $\IMod(\I{C})^\otimes\to\ICAlg(\I{C})$ is a cartesian fibration, and a map in $\IMod(\I{C})^\otimes$ is cartesian if and only if its image along the forgetful functor $\IMod(\I{C})^\otimes\to\I{C}^\otimes$ is an equivalence.\qed
\end{proposition}
\begin{corollary}
	\label{cor:ModCAlgCartesianFibrationUnderlyingCategories}
	For any symmetric monoidal $\BB$-category $\I{C}$, the projection $\IMod(\I{C})\to\ICAlg(\I{C})$ is a cartesian fibration.\qed
\end{corollary}

Next we would like to establish functoriality in the opposite direction, i.e.\ to construct base change functors $\IMod_R(\I{C})^\otimes\to\IMod_S(\I{C})^\otimes$ along any algebra map $R\to S$. The existence of these functors requires the existence of geometric realisations. For the remainder of this section, we shall therefore fix an internal class $\I{U}$ of $\BB$-categories that contains $\Delta^{\op}$ and assume that $\I{C}^\otimes$ is a symmetric monoidal $\BB$-category such that $\I{C}$ is $\I{U}$-cocomplete and the tensor functor $-\otimes-\colon \I{C}\times\I{C}\to\I{C}$ is $\I{U}$-bilinear in the sense of \S~\ref{sec:bilinear} below. We now obtain:

\begin{proposition}
	\label{prop:ContravariantModuleFunctoriality}
	The projection $p\colon \IMod(\I{C})^\otimes\to\Fin_{\ast}\times\ICAlg(\I{C})$ is a cocartesian fibration.
\end{proposition}
\begin{proof}
	By using~\cite[Remark~3.2.7]{MCocartesian}, it suffices to show that for any map $\alpha\colon \ord{m}\to\ord{n}$ of finite pointed sets and any map $f\colon \Delta^1\otimes A\to \ICAlg(\I{C})$, the pair $(\pi_A^\ast(\alpha), f)$ admits a cocartesian lift. By~\cite[Theorem~4.5.3.1]{Lurie2017}, the map $\Mod(\I{C}(A))^\otimes\to \Fin_{\ast}\times\CAlg(\I{C}(A))$ is a cocartesian fibration of $\infty$-categories, so we may choose a cocartesian lift $h$ of $(\alpha, f)$ in $\Mod(\I{C}(A))$. By construction, this map defines a lift of $(\pi_A^\ast\alpha, f)$ in $\IMod(\I{C})^\otimes$ in context $A$. To show that $h$ is cocartesian in the internal sense, it now suffices to show that for any map $s\colon B\to A$ in $\BB$ the map $s^\ast(h)$ defines a cocartesian morphism with respect to $\Mod(\I{C}(B))^\otimes\to\Fin_{\ast}\times\CAlg(\I{C}(B))$. According to the proof of~\cite[Theorem~4.5.3.1]{Lurie2017}, we may find a factorisation $h\simeq h^{\prime\prime}h^\prime$, where $h^{\prime}$ defines a cocartesian morphism of the fibre $\Mod_R(\I{C}(A))\to \Fin_{\ast}$ over some $R\in\CAlg(\I{C}(A))$ and where $h^{\prime\prime}$ is contained in the fibre $\Mod(\I{C}(A))^\otimes\vert_{\ord{n}}$ and satisfies the following property:  for any $1\leq i \leq n$, there is a commutative diagram
	\begin{equation*}
		\begin{tikzcd}
			y\arrow[r, "h^{\prime\prime}"]\arrow[d] & z\arrow[d]\\
			y_i\arrow[r, "h^{\prime\prime}_i"] & z_i
		\end{tikzcd}
	\end{equation*}
	where the vertical maps are inert morphisms lying over $\rho_i\colon\ord{n}\to\ord{1}$ and $h^{\prime\prime}_i$ is a cocartesian morphism of $\Mod(\I{C}(A))\to\CAlg(\I{C}(A))$. Since every map that admits such a factorisation must be cocartesian, it is enough to show that $s^\ast(h^\prime)$ is a cocartesian morphism of $\Mod_{s^\ast R}(\I{C}(B))^\otimes\to\Fin_*$ and that $s^\ast(h_i^{\prime\prime})$ is a cocartesian morphism of $\Mod(\I{C}(B))\to\CAlg(\I{C}(B))$.
	To see the first case we equivalently have to see that the functor $ \Mod_R(\I{C}(A))^{\otimes} \to  \Mod_{s*R}(\I{C}(B))^{\otimes}$ induced by $ s^* $ is symmetric monoidal.
	This in turn follows from the description of the symmetric monoidal structure on $ \Mod_R(\I{C}(A))^{\otimes} $ via the bar-construction \cite[Theorem 4.5.2.1,Propositions 4.4.3.12 and 4.4.2.8]{Lurie2017} because $ s^* $ commutes with $ \Delta^{\op} $-indexed colimits.
	Similarly, the second case follows from the fact that the cocartesian maps in $\Mod(\I{C}(A))\to\CAlg(\I{C}(A))$ are given by the relative tensor product (see e.g. \cite[Lemma 4.5.3.5, Proposition 4.6.2.17 and Corollary 4.5.1.6]{Lurie2017}) and the latter is again described via the Bar-construction.
\end{proof}

\begin{corollary}
	For every commutative algebra $A$ in $\I{C}$, the $\BB$-operad $\IMod_A(\I{C})^\otimes$ is a symmetric monoidal $\BB$-category, and the cocartesian fibration $\IMod(\I{C})^\otimes\to\Fin_{\ast}\times\ICAlg(\I{C})$ straightens to a functor
	\begin{equation*}
		\ICAlg(\I{C})\to\ICMon(\ICat_{\BB})\simeq\ICat_{\BB}^\otimes
	\end{equation*}
	that maps $A$ to $\IMod_A(\I{C})^\otimes$.\qed
\end{corollary}

For later use we note the following observation:
\begin{proposition}
	\label{prop:ModForgetfulPreColimits}
	Suppose that $ \I{C} $ is a complete and cocomplete symmetric monoidal $ \BB $-category and that the tensor product $ - \otimes- \colon \I{C} \times \I{C} \to \I{C} $ is bilinear. 
	Let $ R \colon 1 \to \ICAlg(\I{C}) $ be a commutative algebra.
	Then 
	\begin{enumerate}
		\item The $ \BB $-category	$ \IMod_R(\I{C}) $ is complete and cocomplete.
		\item The canonical functor $ F \colon \IMod_R(\I{C}) \to \I{C} $ is continuous and cocontinuous.
		\item The tensor product  $ - \otimes_R - \colon \IMod_R(\I{C}) \times  \IMod_R(\I{C}) \to  \IMod_R(\I{C}) $ is bilinear.
	\end{enumerate}
	Furthermore if $ \I{C}$ is presentably symmetric monoidal (in the sense of Definition~\ref{def:presentablyMonoidal}) then so is $  \IMod_R(\I{C}) $.
\end{proposition}
\begin{proof}
	Since by Corollary~\ref{cor:ModCAlgCartesianFibrationUnderlyingCategories} and Proposition~\ref{prop:ContravariantModuleFunctoriality} the functor $ \IMod(\I{C}) \to \ICAlg(\I{C}) $ is both cartesian and cocartesian, the essentially unique map $ 1_{\otimes} \to R $ of algebras induces an adjunction
	\[
	(-\otimes R\dashv F)\colon \I{C}\simeq  \IMod_{1_{\otimes}}\leftrightarrow \IMod_R(\I{C}) .
	\]
	Therefore, the functor $F$ preserves all limits that exist in $\IMod_R(\I{C})$. 
	For fixed $ A \in \BB $ the functor $ F(A) $ is equivalent to the forgetful functor $ \Mod_{\pi_A^* R}(\I{C}(A)) \to \I{C}(A)$, which creates colimits by \cite[Corollary 3.4.4.6]{Lurie2017}.
	Furthemore $ F(A) $ has a left adjoint given by $ - \otimes \pi_A^* R $.
	Since the diagram
	\[\begin{tikzcd}
		{\Mod_{\pi_A^*R}(\I{C}(A))} & {\Mod_{\pi_B^*R}(\I{C}(B))} \\
		{\I{C}(A)} & {\I{C}(B)}
		\arrow["{s^*}", from=1-1, to=1-2]
		\arrow["{F(B)}", from=1-2, to=2-2]
		\arrow["{F(A)}"', from=1-1, to=2-1]
		\arrow["{s^*}"', from=2-1, to=2-2]
	\end{tikzcd}\]
	commutes for any map $ s \colon B \rightarrow A $ in $\BB$, this shows that $  \IMod_R(\I{C}) $ is $ \ILConst $-cocomplete and $ F $ is $ \ILConst $-cocontinuous.
	Next we show that $ s^* \colon \Mod_{\pi_A^*R}(\I{C}(A)) \to \Mod_{\pi_B^*R}(\I{C}(B)) $ admits a left adjoint $ s_! $ that is compatible with $ F $ in the obvious way.
	By \cite[Remark 4.7.3.15]{Lurie2017} we may find a functor $ G \colon \Mod_{\pi_B^*R}(\I{C}(B)) \to \Fun(\Delta^\op,\I{C}(B)) $ such that the composite
	\[
	\Mod_{\pi_B^*R}(\I{C}(B)) \xrightarrow{G} \Fun(\Delta^\op,\I{C}(B)) \xrightarrow{- \otimes \pi_B^*R} \Fun(\Delta^\op,\Mod_{\pi_B^*R}(\I{C}(B))) \xrightarrow{\colim} 	\Mod_{\pi_B^*R}(\I{C}(B))
	\]
	is equivalent to the identity.
	Note that since $ \I{C} $ is cocomplete, the functor $ s^* \colon \I{C}(A) \to \I{C}(B) $ admits a left adjoint $s_!$. Therefore, the composition
	\[
	\Mod_{\pi_B^*R}(\I{C}(B)) \xrightarrow{G} \Fun(\Delta^\op,\I{C}(B)) \xrightarrow{(( - \otimes \pi_A^*R)\circ s_!)_*} \Fun(\Delta^\op,\Mod_{\pi_A^*R}(\I{C}(A))) \xrightarrow{\colim } \Mod_{\pi_A^*R}(\I{C}(A))
	\]
	defines a left adjoint of the transition functor $ s^*$. Next, we claim that the canonical natural transformation $ \alpha \colon s_! F(B)  \to F(A) s_! $ is an equivalence. 
	Since all functors commute with colimits it suffices to check this for objects of the form $ M \otimes \pi_B^*R $ for some $ M \in \I{C}(B) $.
	In this case $ \alpha $ identifies with the projection formula transformation
	\[
	s_!(M \otimes \pi_B^* R) \to s_! M \otimes \pi_A^* R
	\]
	which is an equivalence since the tensor products commutes with $ \Univ $-indexed colimits in the first variable.
	It follows that for any pullback square 
	\[\begin{tikzcd}
		Q & P \\
		B & A
		\arrow["t", from=1-1, to=1-2]
		\arrow["p", from=1-2, to=2-2]
		\arrow["q"', from=1-1, to=2-1]
		\arrow["s"', from=2-1, to=2-2]
	\end{tikzcd}\]
	in $ \BB $ the canonical transformation $ q_! t^* \to s^* p_! $ is an equivalence since we may check this after composing with the forgetful functor $ F(B)$ and since $ \I{C} $ is cocomplete.
	Thus we have shown that $ \IMod_R(\I{C}) $ is cocomplete and that the forgetful functor $F$ is cocontinuous. 
	Furthermore, we note that the functor $ s^* \colon \Mod_{\pi_A^*R}(\I{C}(A)) \to \Mod_{\pi_B^*R}(\I{C}(B)) $ admits a right adjoint given by the composite 
	\[
	\Mod_{\pi_B^*R}(\I{C}(B)) \xrightarrow{s_*} \Mod_{s_* \pi_B^*R}(\I{C}(A)) \to \Mod_{\pi_A^*R}(\I{C}(A))
	\]
	where the second functor is given restriction of scalars along $ \pi_A^*R \to s_* s^* \pi_A^* R = s_* \pi_B^*R $.
	Since $ s_* $ is compatible with the forgetful functor $ F $, the same argument as above shows that $ \IMod_R(\I{C}) $ is $ \Univ $-complete and therefore complete. 
	Thus we have shown (1) and (2).
	
	We will now show (3).
	It follows from \cite[Corollary 3.4.4.6]{Lurie2017} that for some fixed $ M \in \IMod_R(\I{C}) $ in context $ A $, the functor $ - \otimes_R M $ is $ \ILConst $-cocontinuous, so it suffices to show that it is also $ \Univ $-cocontinuous.
	For this we have to see that for any $ s \colon C \to B $ in $ \BB_{/A} $ the canonical map
	\[
	s_!(N) \otimes_{\pi_B^* R}  \pi_B^* M \to  s_!(N \otimes_{\pi_C^* R} \pi_C^* M) 
	\]
	is an equivalence, by \cite[Proposition~5.4.2]{MWColimits}.
	Again using that the relative tensor product is the colimit of the bar construction, it follows that we only have to see that the canonical map
	\[
		s_!(N) \otimes (\pi_B^* R)^{\otimes n} \pi_B^* M \to  s_!(N \otimes (\pi_C^* R)^{\otimes n} \pi_C^* M) 
	\]
	is an equivalence for all $ n $.
	This follows because $ \pi_B^{*} $ is symmetric monoidal and the projection formula holds for $ \I{C} $.
	
	Finally, if $ \I{C} $ is presentably symmetric monoidal then $\IMod_R(\I{C}) $ is section-wise presentable by \cite[Theorem 3.4.4.2]{Lurie2017} and therefore presentably symmetric monoidal by what we have seen above.
\end{proof}
\subsection{Stable $\BB$-categories}
\label{sec:stability}
In this section we will define and study a few basic properties of \emph{stable} $\BB$-categories. All definitions and results are straightforward adaptions of \cite[\S 1.4.2]{Lurie2017}.

\begin{definition}
	\label{def:pointedBCat}
	A $ \BB $-category $ \I{C} $ is called \emph{pointed} if there is an object $0_\I{C}\colon 1\to\I{C}$ which is both initial and final in $\I{C}$.
\end{definition}

\begin{remark}
	\label{rem:Pointed=SheafInPointed}
	It follows from \cite[Example 4.1.14]{MWColimits} that $ \I{C} $ is pointed if and only if the associated functor $ \I{C} \colon \BB^\op \to \Cat_\infty$ factors through the subcategory $ \Cat_{\infty,*} $ of pointed $ \infty $-categories.
\end{remark}

\begin{definition}
	\label{def:StableBCat}
	A pointed $ \BB $-category $ \I{C} $ is called \emph{stable} if it is finitely complete and cocomplete and a square $ \Delta^1 \times \Delta^1 \to \pi_A^* \I{C} $ in any context $ A \in \BB$ is a pushout square if and only if it is a pullback square. A functor between stable $\BB$-categories is said to be \emph{exact} if it preserves finite colimits.
	We define the large $\BB$-category $ \ICat_\BB^{\ex} $ of stable $\BB$-categories as the subcategory of $ \ICat_\BB $ that is spanned  by the exact functors between stable $ \BB_{/A} $-categories, for every $A\in\BB$.
\end{definition}

\begin{remark}
	\label{rem:StableIsLocal}
	Since existence and preservation of (co)limits are both local properties (see~\cite[Remark~5.2.3]{MWColimits}) and since the collection of exact functors of stable $\Over{\BB}{A}$-categories is closed under composition and equivalences in the sense of~\cite[Proposition~B.2.9]{MWColimits}, it follows as in~\cite[Remark~5.3.2]{MWColimits} that a functor $ f \colon \I{C} \to \I{D} $ of $ \BB_{/A} $-categories is contained in $  \ICat_\BB^{\ex}(A) $ if and only if it is an exact functor betweeen stable $ \BB_{/A} $-categories.	 In particular, one obtains a canonical equivalence $\pi_A^\ast\ICat_{\BB}^{\ex}\simeq\ICat_{\Over{\BB}{A}}^{\ex}$ for every $A\in\BB$.
\end{remark}

\begin{remark}
	\label{rem:StableIfSheafOfStable}
	It follows from Proposition~\ref{prop:criterionforfinitelimits} and \cite[Example 4.1.14]{MWColimits} that a $ \BB $-category $ \I{C} $ is stable if and only of the associated sheaf $ \I{C} \colon \BB^\op \to  \CatS$ factors through the subcategory $ \Cat_\infty^{\ex} \hookrightarrow \CatS$ spanned by the stable $ \infty $-categories and exact functors. Consequently, we obtain a canonical identification $\ICat_{\BB}^{\ex}\simeq\CatS^{\ex}\otimes\Univ$.
\end{remark}

Following the terminology of \cite{Lurie2017}, we write $ \SS^{\fin}_\ast $ for the $\infty$-category of pointed finite $\infty$-groupoids.

\begin{definition}
	\label{def:Stabilisation}
	Let $ \I{C} $ be a $ \BB $-category with finite limits.
	We let $ \ISp(\I{C}) $ be the full subcategory of $ \IFun(\SS^{\fin}_\ast,\I{C}) $ spanned by those functors $ \SS^{\fin}_\ast \to \pi_A^* \I{C}$ in context $ A \in\BB$ that preserve the final object and send pushout squares to pullbacks.
	We denote by  $\Omega^\infty\colon \ISp(\I{C}) \to \I{C} $ the functor that is obtained by evaluation at $ S^0 \in \SS^{\fin}_\ast $.
\end{definition}

\begin{remark}
	\label{rem:ExcisivenessIsLocal}
	Let $d \colon \Delta^1 \times \Delta^1 \to \pi_A^* \SS^{\fin}_\ast $ be a square in the constant $ \BB $-category $ \SS^{\fin}_\ast $ in context $ A \in \BB $.
	Since $ \Delta^1 \times \Delta^1 $ is a finite $ \BB $-category we may find a cover $s \colon B \twoheadrightarrow A $ such that $ s^* f $ is given by a square in the $\infty$-category $\SS_*^{\fin}$, cf.~Appendix~\ref{app:locallyConstantSheaves}.
	Since being a pullback square is a local condition, this implies that a functor $ f \colon \SS^{\fin}_\ast \to \I{C} $ is contained in $ \ISp(\I{C}) $ if and only if for every $ A \in \BB $, the functor $ \SS^{\fin} \to \I{C}(A) $ that corresponds to $\pi_A^\ast f$ preserves finite limits.
\end{remark}

\begin{proposition}
	\label{prop:StablisationAdjoint}
	The inclusion functor $ \ICat_\BB^{\ex} \to \ICat_\BB^{\lex} $ from stable $ \BB $-categories to $ \BB $-categories with finite limits admits a right adjoint $ \ISp\colon \ICat_\BB^{\lex} \to \ICat_\BB^{\ex} $, that sends a $ \BB $-category $ \I{C} $ to $ \ISp(\I{C}) $.
\end{proposition}
\begin{proof}
	By \cite[Corollary 1.4.2.23]{Lurie2017} we have an adjunction
	\[
	(i\dashv\Sp)\colon \CatS^{\ex}\leftrightarrows \CatS^{\lex},
	\]
	hence $\Sp$ defines a morphism in $\RPrS$. By applying $-\otimes\Univ$, we thus obtain the desired adjunction of $\BB$-categories $\ICat_{\BB}^{\ex}\leftrightarrows\ICat_{\BB}^{\lex}$.
	Now Remark~\ref{rem:ExcisivenessIsLocal} implies that if $\I{C}$ is a stable $\BB $-category, the composition $ \BB^\op \xrightarrow{\I{C}} \Cat_\infty^{\lex} \xrightarrow{\Sp} \Cat_\infty^{\ex} $ is precisely $ \ISp(\I{C}) $.
\end{proof}

\begin{definition}
	\label{def:BSpectra}
	We call $ \ISp(\Univ) $ the $ \BB $-category of \emph{$ \BB $-spectra}.
\end{definition}

\begin{lemma}
	\label{lem:LoopsHasLeftAdjoint}
	Suppose that $ \I{C} $ is a presentable $ \BB $-category.
	Then $ \ISp(\I{C}) $ is also presentable and the functor $ \Omega^\infty \colon \ISp(\I{C}) \to \I{C} $ admits a left adjoint.
\end{lemma}
\begin{proof}
	Since $ \ISp(\I{C}) $ is given by the composite $\BB^{\op}\to \LPrS \xrightarrow{\Sp} \LPrS $, the first claim is clear as the functor $ \Sp = \operatorname{Exc}_*(\SS^{\fin}_\ast,-)$ is compatible with the mate construction.
	For the second claim we observe that for any $ A \in \BB $, the functor $ \Omega^\infty(A) $ may be identified with $ \Omega^\infty  \colon \Sp(\I{C}(A)) \to \I{C}(A)$.
	Since the compatibility with \'etale base change is clear, it follows from \cite[Corollary 5.4.7]{MWColimits} that $ \Omega^\infty $ is continuous.  
	Also $ \Omega^\infty(A) $ is accessible since it admits a left adjoint~\cite[Proposition 1.4.4.4]{Lurie2017}.
	Thus the claim follows from Proposition~\ref{prop:adjointfunctorII}.
\end{proof}

\begin{proposition}
	\label{prop:LoopsFilteredColimits}
	The functor $ \Omega^\infty \colon \Sp(\Univ) \to \Univ$ commutes with filtered colimits.
\end{proposition}
\begin{proof}
	By construction $ \Omega^\infty $ is given as the composition
	\[
	\ISp(\Univ) \hookrightarrow \IFun(\SS^{\fin}_\ast,\Univ) \xrightarrow{\ev_{S_0}} \Univ
	\]
	and therefore it suffices to see that the inclusion $ \ISp(\Univ) \hookrightarrow \IFun(\SS^{\fin}_\ast,\Univ) $ commutes with filtered colimits.
	By a similar argument as in the proof of Lemma~\ref{lem:stabilityFiltUCocontinuousFunctors}, this follows from the fact that filtered colimits in $\Univ$ commute with finite limits.
\end{proof}

\begin{remark}
	\label{rem:BSpectraMonoidal}
	By construction, there is a canonical equivalence $ \ISp(\Univ) \simeq \Sp \otimes \Univ $.
	Since the functor  $ - \otimes \Univ $ is symmetric monoidal (see the discussion before Proposition~\ref{lem:tensoringIsMonoidal}), it follows that $ \ISp(\Univ) $ canonically admits the structure $ \ISp(\Univ)^\otimes  $ of a presentably symmetric monoidal $ \BB $-category.
\end{remark}

\begin{definition}
	\label{def:BRingSpectrum}
	A \emph{commutative $ \BB $-ring spectrum} is a commutative algebra in $ \ISp(\Univ)^\otimes $.
	We denote the unit object of $ \ISp(\Univ) $ by $ S $ and call it the \emph{$ \BB $-sphere spectrum}.
	If $ R $ is a commutative $ \BB $-ring spectrum we denote the $ \BB $-category of modules over $ R $ by $ \IMod^\BB_R $.
	We will write $ \Mod^\BB_R $ for the $\infty$-category of global sections of $ \IMod^\BB_R $. 
\end{definition}

\begin{remark}
	\label{rem:SphereSpectrumCompact}
	Since $ \ISp(\Univ) $ is a presentably symmetric monoidal $ \BB $-category, it receives a unique symmetric monoidal left adjoint functor $\Univ \to \ISp(\Univ) $, see Remark~\ref{rem:UnitIsCocompletion}.
	By uniqueness, this functor automatically agrees with the functor that is given by applying $ -\otimes \Univ $ to the symmetric monoidal functor $ \Sigma^\infty_+ \colon \SS \to \Sp $.
	It follows that the right adjoint of $ \Sigma^\infty_+ \otimes \Univ $ is explicitly given by the map $		\Shv_{\Over{\BB}{-}}(\Sp) \to \Shv_{\Over{\BB}{-}}(\SS)$
	determined by precomposing with $ (\Sigma^\infty_+)^\op $.
	Since we have a commutative square
	\[\begin{tikzcd}
		\Shv_{\Over{\BB}{-}}(\Sp)& \Shv_{\Over{\BB}{-}}(\SS)\\
		{\Sp(\BB_{/-})} & {\BB_{/-}}
		\arrow[from=1-1, to=1-2]
		\arrow["\simeq", from=1-2, to=2-2]
		\arrow["\simeq"', from=1-1, to=2-1]
		\arrow["{\Omega^\infty}"', from=2-1, to=2-2]
	\end{tikzcd}\]
	it follows that $ \Sigma^\infty_+ \otimes \Univ $ is a left adjoint of $ \Omega^\infty \colon \ISp(\Univ) \to \Univ$. 
	In particular $\Omega^\infty$ is the left adjoint of the unqiue colimit preserving functor $\Univ \to \ISp(\Univ)$ and thus there is a natural equivalence $ \Omega^\infty \simeq \map{\ISp(\Univ)}(S,-) $.
\end{remark}

\begin{remark}
	In~\cite{Nardin2016}, Denis Nardin develops a notion of stability in parametrised higher category theory, and he shows that if $G$ is a finite group, the $G$-parametrised stabilisation of $G$-spaces recovers genuine $G$-spectra. However, Nardin's definition of stability is fundamentally different from ours: the key difference is that the class of parametrised $\infty$-category that he declares \emph{finite} is much larger than ours. Consequently, finite limits and colimits in parametrised higher category theory are much more complicated than finite limits and colimits in internal higher category theory, and as a result the associated notions of stability differ significantly. In particular, it is not possible to obtain genuine $G$-spectra by making use of our internal stabilisation procedure, which rather recovers \emph{naive} $G$-spectra. Since essentially the only fundamental difference is the choice of the internal class $\IFin_\BB$ of finite $\BB$-categories, one could conceive a theory of \emph{$\I{U}$-stability} that depends on a fixed internal class $\I{U}$ which replaces the role of $\IFin_{\BB}$ and which would recover Nardin's definition by choosing a particular internal class $\I{U}$. We will not discuss such questions in this paper, though.
\end{remark}

\subsection{Application: dualisable objects in module $ \BB $-categories}
\label{sec:dualisable}
If $ R $ is an $ \mathbb{E}_\infty $-ring, let us denote by $\Shv(\BB;R) = \BB \otimes \Mod_R $ the $\infty$-category of sheaves of $ R $-modules on $ \BB $.
In many geometrically interesting situations, the $\infty$-category $ \Shv(\BB;R)$ is quite far from being compactly generated. 
For example, if $ M $ is a non-compact connected manifold of dimension at least $ 1 $, the only compact object in $ \Shv(M;\mathbb{Z}) $ is the zero object by \cite[Theorem 0.1]{Neeman2001}.
In this section we will see that one may fix this issue by working internal to the $\infty$-topos $ \BB $ itself:
the $ \BB $-category $ \IMod^\BB_R $ (whose underlying $\infty$-category is $ \Shv(\BB;R)$) is always \emph{internally} compactly generated (see Theorem~\ref{thm:Dualizable=Compact=Perfect}).
Furthermore, the internally compact objects can be completely characterised in terms of the symmetric monoidal structure on $ \IMod^\BB_R $, namely as the dualisable objects.
In particular, this provides a classification result for the dualisable objects in $ \Shv(\BB;R) $ (see Corollary~\ref{cor:Dualizable=LocallyPerfect}).

\begin{definition}
	If $\I{C}$ is a $\BB$-category that admits filtered colimits (i.e. is  $\IFilt_{\IFin_\BB}$-cocomplete), let us write $\I{C}^\compact$ for the full subcategory of $\IFin_{\BB}$-compact objects. 
	We will also simply refer to these as \emph{internally compact} objects.
	We say that a cocomplete $\BB$-category $\I{C}$ is \emph{compactly generated} if the canonical map $\IInd(\I{C}^\compact) \to \I{C}$ is an equivalence.
\end{definition}

\begin{remark}
	An internally compact object $ d \colon 1 \to \I{C} $ does not necessarily yield a compact object in the $\infty$-category $\I{C}(1)$.
	For example, the final object $1_{\Univ}\colon 1 \to \Univ$ is always internally compact, but in general $1 \in \BB$ is not compact.
	Since the mapping $\infty$-groupoid functor $\map{\I{C}(1)}(d,-) \colon \I{C}(1) \to \SS$ is given by composing $\Gamma(\map{\I{C}}(d,-)) \colon \I{C}(1) \to \BB$ with the global sections functor $\Gamma \simeq \map{\BB}(1,-)$, we see that internal compactness implies compactness (for any $\BB$-category $\I{C}$) if and only if $1 \in \BB$ is compact.
\end{remark}

\begin{definition}
	\label{def:DualizableObject}
	An object $c$ in context $ A $ of a symmetric monoidal $ \BB $-category $ \I{C} $ is called \emph{dualisable} if it is dualisable in the symmmetric monoidal $ \infty $-category $ \I{C}(A) $.
	In other words, $c$ is dualisable if there is an object $ c^\vee \colon A\to\I{C}$ together with maps $ \eta\colon \pi_A^\ast(1_{\otimes}) \to c^\vee \otimes c $ and $ \varepsilon \colon c^\vee \otimes c \to \pi_A^\ast(1_{\otimes}) $ such that $ (c^\vee \otimes \epsilon) \circ  (\eta \otimes c^\vee) \simeq \id  $ and $ (c \otimes \eta ) \circ (\epsilon \otimes c) \simeq \id $.
	We denote the full subcategory of $ \I{C}$ that is spanned by the dualisable objects by $ \I{C}^{\dual}$.
\end{definition}

\begin{remark}
	\label{rem:DualiableIsLocal}
	It follows from \cite[Proposition 4.6.1.11]{Lurie2017} that an object $ c \in \I{C}(A) $ is contained in $ \I{C}^{\dual}(A) $ if and only if it is dualisable.
\end{remark}
A first easy consequence of the definitions is the following:
\begin{lemma}
	\label{lem:DualizableHasAdjoint}
	Let $\I{C}$ be a symmetric monoidal $\BB$-category and let $c\colon 1\to\I{C}$ be a dualisable object.
	Then the functor $ -\otimes c^\vee \colon\I{C}\to\I{C}$ is a right adjoint of the functor $ - \otimes c$.\qed
\end{lemma}
Suppose now that $R$ is an $\mathbb E_\infty$-ring object in $\BB$.
Since $\IMod^\BB_R $ is a presentably symmetric monoidal $ \BB $-category by Proposition~\ref{prop:ModForgetfulPreColimits}, for every object $x\colon 1\to \IMod^\BB_R$ the functor $ - \otimes x $ admits a right adjoint
\[
\Hom_R(x,-) \colon \IMod^\BB_R \to \IMod^\BB_R.
\]
Also, recall from Corollary~\ref{cor:ModCAlgCartesianFibrationUnderlyingCategories} and Proposition~\ref{prop:ContravariantModuleFunctoriality} that the forgetful functor $F \colon \IMod^\BB_R \to \IMod^\BB_S \simeq \ISp(\Univ) $ has a symmetric monoidal left adjoint $ - \otimes R \colon \ISp(\Univ) \to \IMod^\BB_R $.
We may therefore consider the composite
\[
G \colon \IMod^\BB_R \xrightarrow{\Hom_R(x,-)} \IMod^\BB_R \xrightarrow{F} \ISp(\Univ) \xrightarrow{\Omega^\infty} \Univ.
\]
We claim that $ G $ is equivalent to the functor $ \map{\IMod^\BB_R}(x,-) $. 
Indeed we have a chain of natural equivalences
\[
G \simeq \map\Univ(1_{\Univ}, G(-)) \simeq \map{\IMod^\BB_R}(R,\Hom_R(x,-)) \simeq \map{\IMod^\BB_R}(R \otimes x, -) \simeq \map{\IMod^\BB_R}(x,-).
\]

Now if $ x $ is furthermore dualisable, the functor $ \Hom_R(x,-) $ is of the form $ - \otimes x $ and therefore in particular cocontinuous.
Since both $ F $ and $ \Omega^\infty $ are $\IFilt_{\IFin_\BB}$-cocontinuous by Propositions~\ref{prop:ModForgetfulPreColimits} and \ref{prop:LoopsFilteredColimits}, we have shown:

\begin{lemma}
	\label{lem:DualizableIsCompact}
	Let $ x \colon 1\to \IMod^\BB_R $ be dualisable.
	Then $ x $ is internally compact.\qed
\end{lemma} 

\begin{definition}
	\label{def:PerfectObjects}
	We define the $ \BB $-category $ \IPerf^\BB_R$ to be the smallest full subcategory of $ \IMod^\BB_R $ that is closed under finite colimits, retracts and contains the monoidal unit $ R $.
	We call the objects in $ \IPerf^\BB_R $ \emph{perfect objects}.
\end{definition}

We now obtain:
\begin{theorem}
	\label{thm:Dualizable=Compact=Perfect}
	The $ \BB $-category $ \IMod^\BB_R $ is compactly generated, and the following full subcategories of $ \IMod^\BB_R $ are equivalent
	:
	\begin{enumerate}
		\item The full subcategory $ \IMod^{\BB,\dual}_R $ spanned by the dualisable objects.
		\item The full subcategory $ \IMod^{\BB,\compact}_R $ spanned by the internally compact objects.
		\item The full subcategory $ \IPerf^\BB_R $ spanned by the perfect objects.
	\end{enumerate}
\end{theorem}
\begin{proof}
	In light of Remarks~\ref{rem:BCUComp} and~\ref{rem:DualiableIsLocal}, Lemma~\ref{lem:DualizableIsCompact} implies that we have an inclusion $\IMod^{\BB,\dual}_R \into \IMod^{\BB,\compact}_R $.
	Furthermore, dualisable objects form a stable full subcategory that is closed under retracts, so that $ \IPerf^\BB_R \into \Mod^{\BB,\dual}_R $ is clear.
	It remains to see that every compact object is perfect.
	On account of the inclusion $\IPerf^\BB_R\into\IMod_R^{\BB,\compact}$, we deduce from \cite[Proposition 7.2.4]{MWColimits} and Corollary~\ref{cor:UCocompleteIndUPresentable} that the inclusion $ \IPerf^\BB_R \into \IMod^\BB_R$ extends to a fully faithful cocontinuous functor $i \colon \IInd(\IPerf^\BB_R) \into \IMod^\BB_R $.
	If we can show that $ i $ is an equivalence we are done by Proposition~\ref{prop:characterisationCompactObjectAccessible}.
	Since $ \IMod^\BB_R $ is cocomplete, we may now apply the adjoint functor theorem (Proposition~\ref{prop:adjointFunctorTheorem}) to deduce the existence of a right adjoint $ r $ of $ i $.
	We complete the proof by showing that the counit $ \varepsilon \colon ir \to \id $ is an equivalence.
	For this we pick an object $ X \in \IMod^\BB_R(A) $ and consider the fibre sequence
	\[
	\operatorname{fib}(\varepsilon) \to i r X \xrightarrow{\varepsilon}  X
	\]
	in $ \IMod^\BB_R(A) $. We need to show that $\operatorname{fib}(\varepsilon)\simeq 0$. For every $n$, we obtain a fibre sequence
	\[
	\map{\IMod^\BB_R}(\Sigma^n (\pi_A^\ast R), \operatorname{fib}(\varepsilon)) \to \map{\IMod^\BB_R}(\Sigma^n (\pi_A^\ast R), i r X) \xrightarrow{\varepsilon_*}	\map{\IMod^\BB_R}(\Sigma^n (\pi_A^\ast R), X)
	\]
	in $\Over{\BB}{A}$.
	As $ \Sigma^n \pi_A^\ast R $ is in the essential image of $ i $, the map $ \varepsilon_* $ is an equivalence.
	Therefore, we conclude that we have equivalences
	\[
	1 \simeq \map{\IMod^\BB_R}(\Sigma^n (\pi_A^\ast R), \operatorname{fib}(\varepsilon)) \simeq \map{\IMod^\BB_R}( \pi_A^\ast R, \Omega^n \operatorname{fib}(\varepsilon)) \simeq \Omega^\infty \Omega^n \operatorname{fib}(\varepsilon)
	\]
	for any $ n $. Thus $ \operatorname{fib}(\varepsilon) \simeq 0 $, which shows that $ \varepsilon $ is an equivalence, as desired.
\end{proof}

\begin{remark}
	\label{rem:ComparisonWithStacksProject}
	Let $ \BB_0 $ be a 1-topos and $ R $ a ring object in $ \BB_0 $.
	Let us denote the hypercomplete $ \infty $-topos associated to $ \BB_0 $ by $ \BB $.
	Then by \cite[Theorem 2.1.2.2]{lurie2018} there is an equivalence of $ \infty $-categories
	\[
	\mathbf{D}(\Shv(\BB_0;R))) \simeq \Mod^\BB_R
	\]
	where we consider $ R $ as a discrete ring object in $ \BB $.
	Furthermore one can check that the induced symmetric monoidal structure on the homotopy category of $ \mathbf{D}(\Shv(\BB_0;R)) $ is indeed the usual one.
	With these identifications in mind, the result that an object in $ \Mod^\BB_R $ is dualisable if and only if it is perfect appears as \cite[\href{https://stacks.math.columbia.edu/tag/0FPV}{Tag 0FPV}]{stacks-project}.
\end{remark}

\begin{example}
	\label{ex:perfectcomplinternallycomp}
	Let $ X $ be a (spectral/dervied) scheme and consider the Zariski-topos $ X_{\Zar} $ of $ X $. 
	The structure sheaf $ \mathcal{O}_X $ of $ X $ is a sheaf of ($ \mathbb{E}_\infty $-)rings on $ X $ and thus gives rise to an object in $ \CAlg(\Shv(X_{\Zar},\Sp)) $. 
	Thus we may consider the $ X_{\Zar} $-category $ \IMod^{X_{\Zar}}_{\mathcal{O}_X} $.
	By Theorem~\ref{thm:Dualizable=Compact=Perfect} an object $\mathcal{F} \in  \IMod^{X_{\Zar}}_{\mathcal{O}_X}(\ast) = \Mod(\mathcal{O}_X)$ is internally compact if and only if it is perfect, i.e. if and only if there is an open covering $ X = \bigcup_i U_i $ by affines such that $ \mathcal{F}_{\mid U_i} $ is contained in the smallest stable idempotent complete subcategory containing $ (\mathcal{O}_X)_{\mid U_i} = \mathcal{O}_{ U_i} $.
	Thus the full subcategory of $ \Mod_{\mathcal{O}_{X}}$ spanned by the internally compact objects is equivalent to the ususal category of perfect complexes on $ X $.
\end{example}

\begin{remark}
	\label{rem:ComparisonWithSheavesOfModules}
	Let $ R $ be an $ \mathbb{E}_\infty $-ring.
	The stable constant sheaf functor $\const\colon  \Sp \to \Shv_{\Sp}(\BB) $ is symmetric monoidal, therefore $ \underline{R} \coloneqq \const R $ defines a commutative $ \BB $-ring spectrum.
	Then there is a canonical equivalence $ \Shv(\BB;R) \simeq \Mod^\BB_{\underline{R}} $.
	Indeed, applying $ - \otimes \Shv_{\Sp}(\BB) $ to the symmetric monoidal functor $ - \otimes R \colon \Sp \to \Mod_R $ yields a symmetric monoidal functor $  \Shv_{\Sp}(\BB) \to  \Shv(\BB;R) $, whose right adjoint $ \Phi \colon \Shv(\BB;R) \to \Shv_{\Sp}(\BB) $ is lax symmetric monoidal and conservative.
	Since the forgetful functor $ \Mod_R \to \Sp $ admits a further right adjoint, so does $ \Phi $.
	By \cite[Corollary 2.5.5.3]{lurie2018} the map $ \Phi $ induces a symmetric monoidal functor
	\[
	\Shv(\BB;R) \to \Mod_{\Phi(1)}(\Shv_{\Sp}(\BB)) \simeq \Mod^\BB_{\underline{R}}
	\]
	By using~\cite[Corollary 4.7.3.16]{Lurie2017} it now easily follows that this functor is an equivalence.
\end{remark}

Combining Theorem~\ref{thm:Dualizable=Compact=Perfect} with Corollary~\ref{cor:LConstPO} we arrive at the following classification result for dualisable objects:
\begin{corollary}
	\label{cor:Dualizable=LocallyPerfect}
	Let $ R $ be an $ \mathbb{E}_\infty $-ring.
	Then an object in $ \Shv(\BB;R) $ is dualisable if and only if it is locally constant with perfect values.\qed
\end{corollary}

\begin{remark}
	\label{rem:previouscases}
	In the case of \'etale hypersheaves on a scheme $ X $ and where $ R $ is a discrete ring, the above corollary already appears in \cite[Remark 6.3.27]{cisinski2016} and for pro-\'etale sheaves on $ X $ this is shown in \cite[Corollary 3.4.3]{hemo2020}.
	The proofs in both referecnes rely on features of the specific geometric situation, more specifically on the existence of enough points and $ w $-contractible objects, respectively.
	Having a sufficient amount of machinery from internal higher category theory at our disposal, we can recover these observations with a now completely formal proof.
\end{remark}

\section{The tensor product of presentable $\BB$-categories}
\label{chap:tensorProductBCategories}
In~\cite{Lurie2017}, Lurie establishes a symmetric monoidal structure on the $\infty$-category of $\KK$-cocomplete $\infty$-categories with $\KK$-cocontinuous functors, for any class $\KK$ of $\infty$-categories. In particular, his construction gives rise to a symmetric monoidal structure on the $\infty$-category $\LPrS$ of presentable $\infty$-categories. In this section, our goal is to obtain an internal analogue of these results, i.e.\ to construct a symmetric monoidal structure on the large $\BB$-category $\ICat_{\BB}^{\cocont{\I{U}}}$ of $\I{U}$-cocomplete $\BB$-categories and $\I{U}$-cocontinuous functors, for any choice of internal class $\I{U}$. Our construction will be entirely analogous to the one in~\cite{Lurie2017}: we will define the desired symmetric monoidal $\BB$-category $\ICat_{\BB}^{\cocont{\I{U}},\otimes}\to\Fin_{\ast}$ as the subcategory of the cartesian monoidal $\BB$-category $\ICat_{\BB}^\times\to\Fin_{\ast}$ that is spanned by what we call \emph{$\I{U}$-multilinear} functors. We define and study this concept in \S~\ref{sec:bilinear}, before we go on and discuss the symmetric monoidal structure on $\ICat_{\BB}^{\cocont{\I{U}}}$ in \S~\ref{sec:tensorProductCatU}. In particular, our construction will yield a symmetric monoidal structure on the large $\BB$-category $\ILPr_{\BB}$. In \S~\ref{sec:BModules}, we make use of this structure to identify \emph{$\BB$-modules} as a full subcategory of the $\infty$-category $\LPr(\BB)$ of presentable $\BB$-categories.

\subsection{Bilinear functors}
\label{sec:bilinear}
Recall that a bilinear functor of cocomplete $\infty$-categories $\CC \times \DD \rightarrow \EE$ is a functor that preserves small colimits separately in each variable.
We will now introduce this notion in the internal setting.
It will be useful to consider functors that only preserve certain (internal) classes of colimits in each variable, so that we arrive at the following general definition:
\begin{definition}
	\label{def:bilinearfunctor}
	Let $\I{U}$ and $\I{V}$ be two internal classes of $\BB$-categories.
	Suppose that $\I{C},\I{D}$ and $\I{E}$ are $\BB$-categories such that $\I{C}$ is $\I{U}$-cocomplete, $\I{D}$ is $\I{V}$-cocomplete and $\I{E}$ is $\I{U}\cup\I{V}$-cocomplete.
	We will say that a functor $f \colon \I{C} \times \I{D} \rightarrow \I{E}$ is \emph{$(\I{U},\I{V})$-bilinear} if for any $A \in \BB$ and any two objects $c\colon A \to \I{C}$ and $d\colon A\to\I{D}$ the functor 
	\[
	f(c,-) \colon \pi_A^* \I{D} \xrightarrow{c\times\id} \pi_A^* \I{C} \times \pi_A^* \I{D} \xrightarrow{\pi_A^* f}  \pi_A^* \I{E}
	\]
	is $\pi_A^\ast\I{V}$-cocontinuous and the functor
	\[
	f(-,d) \colon \pi_A^* \I{C} \xrightarrow{\id\times d} \pi_A^* \I{C} \times \pi_A^* \I{D} \xrightarrow{\pi_A^*f} \pi_A^* \I{E}
	\]
	is $\pi_A^\ast{\I{U}}$-cocontinuous.
	We write $\IFun^{({\I{U}},\I{V})}(\I{C}\times\I{D},\I{E})$ for the full subcategory spanned by the $(\pi_A^\ast\I{U},\pi_A^\ast\I{V})$-bilinear functors for every $A\in\BB$.
	If $\I{U} = \I{V} = \ICat_\BB$ (and $\I{C}$,$\I{D}$ and $\I{E}$ are large), we will simply say that $f$ is \emph{bilinear} and write $\IFun^{\bil}(\I{C}\times\I{D},\I{E})$ for the associated $\BB$-category of bilinear functors.
\end{definition}

\begin{remark}
	\label{rem:BCBilinear}
	In the situation of Definition~\ref{def:bilinearfunctor}, the fact that $\I{U}$- and $\I{V}$-cocontinuity are local conditions~\cite[Remark~5.2.3]{MWColimits} implies that for any cover $\bigsqcup_i A_i\onto 1$ in $\BB$, a functor $f$ is $(\I{U},\I{V})$-bilinear if and only if for each $i$ the functor $\pi_{A_i}^\ast f$ is $(\pi_{A_i}^\ast\I{U},\pi_{A_i}^\ast\I{V})$-bilinear. In particular, an object $A\to \IFun(\I{C}\times\I{D},\I{E})$ in context $A\in\BB$ is contained in $\IFun(\I{C}\times\I{D},\I{E})^{(\I{U},\I{V})}$ if and only if it defines a $(\pi_A^\ast\I{U},\pi_A^\ast\I{V})$-bilinear functor, and there consequently is a canonical equivalence 
	\begin{equation*}
		\pi_A^\ast\IFun^{(\I{U},\I{V})}(\I{C}\times\I{D},\I{E})\simeq\IFun[\Over{\BB}{A}]^{(\pi_A^\ast\I{U},\pi_A^\ast\I{V})}(\pi_A^\ast\I{C}\times\pi_A^\ast\I{D},\pi_A^\ast\I{E})
	\end{equation*}
	of $\Over{\BB}{A}$-categories.
\end{remark}

\begin{lemma}
	\label{lem:limitpreservationswap}
	Let ${\I{U}}$ and ${\I{V}}$ be two internal classes and let $\I{C},\I{D}$ and $\I{E}$ be $\BB$-categories such that $\I{C}$ is ${\I{U}}$-cocomplete, $\I{D}$ is ${\I{V}}$-cocomplete and $\I{E}$ is ${\I{U}}\cup\I{V}$-cocomplete.
	Then $\IFun^{\cocont{\I{V}}}(\I{D},\I{E}) \subseteq \IFun(\I{D},\I{E})$ is closed under $\I{U}$-colimits, $\IFun^{\cocont{\I{U}}}(\I{C},\I{E})$ is closed under $\I{V}$-colimits, and there are natural equivalences
	\[
	\IFun^{(\I{U},\I{V})}(\I{C} \times \I{D},\I{E})\simeq \IFun^{\cocont{\I{U}}}(\I{C},\IFun^{\cocont{\I{V}}}(\I{D},\I{E})) \simeq \IFun^{\cocont{\I{V}}}(\I{D},\IFun^{\cocont{\I{U}}}(\I{C},\I{E})).
	\]
\end{lemma}
\begin{proof}
	By symmetry, it is enough to show that $\IFun^{\cocont{\I{V}}}(\I{D},\I{E})$ is closed under $\I{U}$-colimits and to construct the first of the two equivalences. 
	To begin with, we claim that a functor $f \colon \I{C}\times\I{D}\to\I{E}$ is $({\I{U}},{\I{V}})$-bilinear if and only if its transpose $f^\prime\colon \I{C}\to\IFun(\I{D},\I{E})$ is $\I{U}$-cocontinuous and takes values in $\IFun^{\cocont{\I{V}}}(\I{D},\I{E})$.
	To see this, note that for any $A \in \BB$ and any object $c\colon A\to\I{C}$, the functor $f^\prime(c)\colon \pi_A^\ast\I{D}\to\pi_A^\ast\I{E}$
	is by definition given by $f(c,-)$, which in turn implies that $f(c,-)$ is $\I{V}$-cocontinuous if and only if $f^\prime$ factors through the full subcategory $\IFun^{\cocont{\I{V}}}(\I{C},\I{E})$.
	Moreover, given any object $d \colon A \rightarrow \I{D}$ in context $A\in\BB$, note that the functor $f(-,d)$ is given by the composite
	\[
	\pi_A^* \I{C} \xrightarrow{\pi_A^\ast f^\prime} \pi_A^* \IFun(\I{D},\I{E}) \simeq \IFun[\Over{\BB}{A}](\pi_A^\ast \I{D}, \pi_A^* \I{E}) \xrightarrow{d^*} \pi_A^* \I{E}.
	\]
	Therefore,~\cite[Proposition~4.3.2]{MWColimits} implies that $f^\prime$ is $\I{U}$-cocontinuous if and only if $f(-,d)$ is $\I{U}$-cocontinuous for all $d\colon A\to\I{D}$ and all $A\in\BB$. Hence the claim follows. In light of Remark~\ref{rem:BCBilinear} and~\cite[Remark~5.3.2]{MWColimits}, this already implies that the equivalence $\IFun(\I{C}\times\I{D},\I{E})\simeq\IFun(\I{C},\IFun(\I{D},\I{E}))$ induces a pullback square
	\begin{equation*}
		\begin{tikzcd}
			\IFun^{(\I{U},\I{V})}(\I{C}\times\I{D},\I{E})\arrow[d, hookrightarrow]\arrow[r, hookrightarrow] & \IFun^{\cocont{\I{U}}}(\I{C},\IFun(\I{D},\I{E}))\arrow[d, hookrightarrow]\\
			\IFun(\I{C},\IFun^{\cocont{\I{V}}}(\I{D},\I{E}))\arrow[r, hookrightarrow] & \IFun(\I{C}\times\I{D},\I{E}).
		\end{tikzcd}
	\end{equation*}
	To complete the proof, it is now enough to show that $\IFun^{\cocont{\I{V}}}(\I{D},\I{E})$ is closed under $\I{U}$-colimits. In light of~\cite[Remark~5.3.2]{MWColimits}, this follows once we show that for any $\I{I}\in\I{U}(1)$ the colimit functor $\colim_{\I{I}}\colon\IFun(\I{I},\IFun(\I{D},\I{E}))\to\IFun(\I{D},\I{E})$ restricts to a functor $\IFun(\I{I},\IFun^{\cocont{\I{V}}}(\I{D},\I{E}))\to\IFun^{\cocont{\I{V}}}(\I{D},\I{E})$. As $\colim_{\I{I}}$ is cocontinuous, we get a commutative diagram
	\[
	\begin{tikzcd}
		{\IFun^{\cocont{\I{V}}}(\I{D},\IFun(\I{I},\I{E}))} \arrow[d,"(\colim_{\I{I}})_*"] \arrow[r,hook] & {\IFun(\I{D},\IFun(\I{I},\I{E}))} \arrow[d, "(\colim_{\I{I}})_*"] \\
		{\IFun^{\cocont{\I{V}}}(\I{D},\I{E})} \arrow[r,hook]               & {\IFun(\I{D},\I{E})}.                              
	\end{tikzcd}
	\]
	By what we have already shown above, the equivalence $\IFun(\I{I},\IFun(\I{D},\I{E})) \simeq \IFun(\I{D},\IFun(\I{I},\I{E}))$  restricts to an equivalence $\IFun(\I{I},\IFun^{\cocont{\I{V}}}(\I{D},\I{E})) \simeq \IFun^{\cocont{\I{V}}}(\I{D},\IFun(\I{I},\I{E}))$. Hence the previous diagram shows that the colimit functor $\colim_{\I{I}}$ restricts as desired.
\end{proof}

We will now generalise the above situation to so-called \emph{multilinear} functors. For the sake of simplicity, we will only do this in the case of one fixed internal class.
\begin{definition}
	\label{def:MultilinearFunctor}
	Let $ \I{U} $ be an internal class of $\BB$-categories and suppose that $ \I{C}_1,\dots, \I{C}_n, \I{E} $ are $ \I{U} $-cocomplete $ \BB $-categories.
	A functor $ f \colon \I{C}_1 \times .... \times \I{C}_n \rightarrow \I{E} $ is said to be $ \I{U} $\emph{-multilinear} if for every $i=1,\dots,n$ and all objects $ c_j \colon A_j\rightarrow \I{C}_j $ in context $A\in\BB$ for $ i \neq j $ the functor
	\[
	\pi_A^* \I{C}_i \xrightarrow{(c_1,\dots, \id,\dots, c_n) } \prod_{k=1}^n \pi_A^\ast\I{C}_k \xrightarrow{f} \pi_A^\ast\I{D}
	\]
	is $ \pi_A^*\I{U} $-cocontinuous.
	We will write $ \IFun^{\mult{\I{U}}}(\prod_{k=1}^n \I{C}_k, \I{E} )$ for the full subcategory spanned by the $ \pi_A^\ast\I{U} $-multilinear functors for all $A\in\BB$.
\end{definition}

\begin{remark}
	\label{rem:BCMultilinear}
	By a similar argument as in Remark~\ref{rem:BCBilinear}, the condition of a functor as in Definition~\ref{def:MultilinearFunctor} to be $\I{U}$-multilinear is local in $\BB$, which implies that there is a canonical equivalence $ \pi_A^\ast\IFun^{\mult{\I{U}}} (\prod_{k=1}^n \I{C}_k, \I{E} )\simeq \IFun[\Over{\BB}{A}]^{\mult{\pi_A^\ast\I{U}}}(\prod_{k=1}^n \pi_A^\ast\I{C}_k,\pi_A^\ast \I{E} ) $ for each $A\in\BB$.
\end{remark}

\begin{remark}
	\label{rem:UniversalMultilinFunctor}
	In the situation of Definition~\ref{def:MultilinearFunctor} we can construct the universal $ \I{U} $-multilinear functor using Proposition~\ref{prop:CocompletionRelations}.
	Namely, we may consider the collection of cocones $ \square_{k=1}^n R_k $ from Construction~\ref{constr:BoxtimesOfRi} with respect to the internal class $ \I{U} $.
	Then by construction a functor $ \prod_{k=1}^n \I{C}_k \to \I{E} $ is $ \I{U} $-multilinear if and only if it is contained in the full subcategory $ \IFun( \prod_{k=1}^n \I{C}_i, \I{E} )_{\square_{k=1}^n R_k} $.
	By Proposition~\ref{prop:CocompletionRelations}, we thus have a canonical functor $ j \colon \prod_{k=1}^n \I{C}_i \to \IPSh^{(\I{U}, \square_{k=1}^n R_k)}(\prod_{k=1}^n \I{C}_i) $ that induces an equivalence
	\[
	j^* \colon \IFun^{\cocont{\I{U}}}(\IPSh^{(\I{U}, \square_{k=1}^n R_k)}(\prod_{i=1}^n \I{C}_i), \I{E}) \xrightarrow{\simeq} \IFun^{\mult{\I{U}}}(\prod_{k=1}^n \I{C}_k, \I{E}).
	\]
\end{remark}

\subsection{The tensor product of $\I{U}$-cocomplete $\BB$-categories}
\label{sec:tensorProductCatU}
The goal of this section is to extend the results from \cite[\S~4.8.1]{Lurie2017} to the setting of $\BB$-categories.
Namely, we will construct a symmetric monoidal structure $\ICat_{\BB}^{\cocont{\I{U}},\otimes} $ on the large $\BB$-category $\ICat_{\BB}^{\cocont{\I{U}}}$ of $ \I{U} $-cocomplete $ \BB $-categories.
For this we will roughly follow the arguments in~\cite{Lurie2017}.

Recall that the (large) $ \BB $-category $ \ICat_\BB $ is complete. By Proposition~\ref{prop:CartesianMonoidalBCategories}, we therefore obtain a symmetric monoidal structure $p\colon\ICat_{\BB}^\times\to\Fin_{\ast}$ on $\ICat_{\BB}$. By construction, the pullback of $p(A)$ along the adjunction unit $\Fin_*\to \Gamma_{\Over{\BB}{A}}\Fin_*$ yields the cocartesian fibration classifying the symmetric monoidal $ \infty $-category $ \Cat(\BB_{/A})^\times $.

\begin{remark}
	Let $\alpha\colon\ord{n}\to\ord{m}$ be a map of pointed finite sets and let $f\colon x\to y$ be a map in $\ICat_{\BB}^\times$ in context $A\in\BB$ that is sent to $\pi_A^\ast \alpha$ by the cocartesian fibration $p\colon\ICat_{\BB}^\times\to \Fin_*$. Note that we may regard $f$ as a lift of $\alpha$ in the symmetric monoidal $\infty$-category $\Cat(\Over{\BB}{A})^\times$ since $\alpha$ is in the image of the adjunction unit $\Fin_*\to\Gamma_{\Over{\BB}{A}}\Fin_*$. By~\cite[Remark~3.2.6]{MCocartesian}, the map $f$ is cocartesian with respect to $p$ if and only if it is cocartesian with respect to $p(A)$ when viewed as a map in the $\infty$-category $\ICat_{\BB}^\times(A)$, which is in turn equivalent to the condition that $f$ defines a cocartesian edge in $\Cat(\Over{\BB}{A})^\times$.
\end{remark}

We define a subcategory $\ICat_{\BB}^{\cocont{\I{U}}, \otimes}$ of $\ICat_{\BB}^\times$ as follows:
Let $ f \colon x \rightarrow y  $ be a morphism in $\ICat_{\BB}^\times$ in context $A\in\BB$, and assume that $p(f)$ is contained in the image of the functor $\const_{\Over{\BB}{A}}$ and thus given by a map $\alpha\colon\ord{n}\to\ord{m}$ in the $1$-category $\Fin_*$. We now obtain equivalences $x\simeq (\I{C}_1,\dots,\I{C}_n)$ and $y\simeq(\I{D}_1,\dots,\I{D}_m)$ where the $\I{C}_i$ and $\I{D}_j$ are $\Over{\BB}{A}$-categories, and the map $f$ is determined by a collection of maps $f_j\colon\prod_{i\in\alpha^{-1}(j)}\I{C}_i\to\I{D}_j$ for $j=1,\dots,m$. We shall say that $f$ is \emph{$\I{U}$-multilinear} if the $\I{C}_i$ and $\I{D}_j$ are $\pi_A^\ast\I{U}$-cocomplete and the functors $f_j$ are $\pi_A^\ast\I{U}$-multilinear. Finally, we say that an arbitrary map $g\colon x\to y$ in $\ICat_{\BB}^\times$ in context $A\in \BB$ is \emph{locally} $\I{U}$-multilinear if there is a cover $(s_i)\colon\bigsqcup_i A_i\onto A$ such that $s_i^\ast(g)$ is $\I{U}$-multilinear for each $i$.
We let $\ICat_{\BB}^{\cocont{\I{U}}, \otimes}$ be the subcategory of $\ICat_{\BB}^\times$ that is spanned by the locally $\I{U}$-multilinear maps. 
\begin{remark}
	\label{rem:locallyMultilinearMaps}
	Since every map in the constant $\BB$-category $\Fin_*$ is \emph{locally} of the form $\alpha\colon\ord{n}\to\ord{m}$ (i.e.\ is locally contained in the image of the constant sheaf functor), for every map $g\colon x\to y$ in $\ICat_{\BB}^\times$ in context $A\in\BB$ there is a cover $(s_i)\colon\bigsqcup A_i\onto A$ in $\BB$ such that $p(s_i^\ast g)$ is given by a map $\alpha\colon\ord{n}\to\ord{m}$ of pointed finite sets. As moreover the condition that a functor is $\pi_A^\ast\I{U}$-multilinear is local in $\BB$ (see Remark~\ref{rem:BCMultilinear}), one finds that $g$ is locally $\I{U}$-multilinear if and only if $s_i^\ast(g)$ is $\I{U}$-multilinear.
\end{remark}
\begin{lemma}
	\label{lem:MultilinEdgesCompose}
	The inclusion $(\ICat_{\BB}^{\cocont{\I{U}}, \otimes})_1\into(\ICat_{\BB}^\times)_1$ identifies $(\ICat_{\BB}^{\cocont{\I{U}}, \otimes})_1$ with the subobject of $(\ICat_{\BB}^\times)_1$ that is spanned by the locally $\I{U}$-multilinear maps.
\end{lemma}
\begin{proof}
	We first show that $\I{U}$-multilinear maps are closed under composition. To that end, suppose that $f\colon x\to y$ and $f^\prime\colon y\to z$ are $\I{U}$-multilinear maps in context $A\in\BB$, and consider the commutative diagram
	\[\begin{tikzcd}
		x && y && z \\
		& {x'} && {z'} \\
		&& {y'}
		\arrow["f", from=1-1, to=1-3]
		\arrow["{f'}", from=1-3, to=1-5]
		\arrow["{h_y}"', from=1-1, to=2-2]
		\arrow["{g_y}"', from=2-2, to=1-3]
		\arrow["{h_y'}"', from=2-2, to=3-3]
		\arrow["{g_y'}"', from=3-3, to=2-4]
		\arrow["{h_z}"', from=1-3, to=2-4]
		\arrow["{g_z}"', from=2-4, to=1-5]
	\end{tikzcd}\]
	in which $h_y$, $h_y^\prime$ and $h_z$ are cocartesian and the maps $g_y$, $g_y^\prime$ and $g_z$ are sent to identity maps in $\Fin_*$. Then $f^\prime f$ being $\I{U}$-multilinear precisely means that $g_z g_y^\prime$ is determined by a tuple of $\pi_A^\ast\I{U}$-multilinear functors between $\pi_A^\ast\I{U}$-cocomplete $\Over{\BB}{A}$-categories. Unwinding the definitions, this follows immediately from the observation that $ \pi_A^*\I{U}$-multilinear functors compose in the expected way. Together with the fact that equivalences between $\pi_A^\ast\I{U}$-cocomplete $\BB$-categories are automatically $\pi_A^\ast\I{U}$-cocontinuous, this already implies that the subobject of $(\ICat_{\BB}^\times)_1$ that is spanned by the locally $\I{U}$-multilinear maps is closed under composition and equivalences in the sense of Proposition~\ref{prop:classificationSubcategories}, hence the very same proposition proves the claim.
\end{proof}
To proceed, recall that the three maps $\id_0$, $0<1$ and $\id_1$ of the poset $\Delta^1$ give rise to a decomposition $\Delta^1_1\simeq 1\sqcup 1\sqcup 1$, both when viewing $\Delta^1$ as an $\infty$-category and as a constant $\BB$-category. Therefore, if $\I{C}$ is an arbitrary $\BB$-category, we obtain an induced decomposition $(\Delta^1\otimes\I{C})_1\simeq \I{C}_1\sqcup\I{C}_1\sqcup\I{C}_1$.
By applying this observation to the case where $\I{C}=\ICat_{\BB}^\times$, we may thus define a subcategory $\I{M}^\otimes_{\I{U}}\into \Delta^1\otimes\ICat_{\BB}^\times$ via the subobject of morphisms
\begin{align*}
	(\ICat_{\BB}^\times)_1\sqcup(\ICat_{\BB}^\times)_1\vert_{(\ICat_{\BB}^{\cocont{\I{U}}, \otimes})_0}\sqcup(\ICat_{\BB}^{\cocont{\I{U}}, \otimes})_1&\into \Delta^1\otimes(\ICat_{\BB}^\times)_1
\end{align*}
where the middle summand denotes the pullback of $d_0\colon (\ICat_{\BB}^\times)_1\to(\ICat_{\BB}^\times)_0$ along the inclusion $(\ICat_{\BB}^{\cocont{\I{U}}, \otimes})_0\into (\ICat_{\BB}^\times)_0$.
Evidently, this subobject is closed under composition and equivalences in the sense of Proposition~\ref{prop:classificationSubcategories}, hence the inclusion $(\I{M}^\otimes_{\I{U}})_1\into \Delta^1\otimes(\ICat_{\BB}^\times)_1$ gives rise to an equivalence
\begin{equation*}
	(\I{M}^\otimes_{\I{U}})_1\simeq(\ICat_{\BB}^\times)_1\sqcup(\ICat_{\BB}^\times)_1\vert_{(\ICat_{\BB}^{\cocont{\I{U}}, \otimes})_0}\sqcup(\ICat_{\BB}^{\cocont{\I{U}}, \otimes})_1.
\end{equation*}
By construction, the pullback of the composition $q\colon \I{M}^\otimes_{\I{U}}\into \Delta^1\otimes\ICat_{\BB}^\times\to\Delta^1\times\Fin_*$ along the inclusion $d^1\colon\Fin_*\into\Delta^1\times\Fin_*$ recovers the cocartesian fibration $p\colon \ICat_{\BB}^\times\to\Fin_*$, and the pullback of $q$ along $d^0\colon\Fin_*\into\Delta^1\times\Fin_*$ recovers the restriction of $p$ to the subcategory $\ICat_{\BB}^{\cocont{\I{U}},\otimes}$. We now obtain:
\begin{proposition}
	\label{prop:UConstrGivesCocartFib}
	The composition $ q\colon \I{M}^\otimes_{\I{U}}\into \Delta^1\otimes\ICat_{\BB}^\times\to\Delta^1\times\Fin_*$ is a cocartesian fibration.
\end{proposition}
\begin{proof}
	Let us begin by fixing maps $\alpha\colon\ord{n}\to\ord{m}$ in the $1$-category $\Fin_*$ and $\epsilon\leq \delta$ in the poset $\Delta^1$, and let $x\colon A\to \I{M}^\otimes_{\I{U}}\vert_{(\epsilon,\ord{n})}$ be an arbitrary object in context $A\in\BB$. Let us write $ \I{V}_0 = \empty$ and $ \I{V}_1 = \I{U} $. By construction of  $\I{M}^\otimes_{\I{U}}$, the object $x$ corresponds to a tuple $(\pi_A^\ast\epsilon,\I{C}_1,\dots,\I{C}_n)$ where $\I{C}_1,\dots,\I{C}_n$ are $\pi_A^\ast\I{V}_{\epsilon}$-cocomplete $\Over{\BB}{A}$-categories. Let $f\colon x\to y$ be a cocartesian lift of $\alpha$ in $\ICat_{\BB}^\times$. For each $j=1,\dots,m$, the construction in \S~\ref{sec:cocompletionRelations} now allows us to define a map
	\[
	g_j\colon \prod_{i\in \alpha^{-1}(j)} \I{C}_i \rightarrow \I{D}_j= \IPSh^{(\I{V}_\delta,\square_{i} R_i)}(\prod_{i\in \alpha^{-1}(j)} \I{C}_i),
	\]
	and by setting $z=(\delta,\I{D}_1,\dots,\I{D}_m)$, precomposing the tuple $g=(\epsilon\leq \delta,g_1,\dots,g_m)$ with $(\id_\epsilon, f)$ defines a lift of $(\epsilon\leq \delta, \alpha)$ in $\I{M}^\otimes_{\I{U}}$.
	By Remark~\ref{rem:UniversalMultilinFunctor}, precomposition with each $g_j$ induces an equivalence
	\[
	g_j^\ast\colon \IFun[\Over{\BB}{A}]^{\mult{\pi_A^\ast\I{V}_\epsilon}} (\prod_{i} \I{C}_i, \I{E} ) \simeq\IFun[\Over{\BB}{A}]^{\cocont{\pi_A^\ast\I{V}_\delta}}(\I{D}_k,\I{E})
	\]
	for every $ \pi_A^* \I{V}_\delta$-cocomplete $\BB_{/A} $-category $ \I{E} $. By construction of $\I{M}^\otimes_{\I{U}}$ and Lemma~\ref{lem:MultilinEdgesCompose}, the underlying core $\Over{\BB}{A}$-groupoids of both domain and codomain of $g_j^\ast$ recover the mapping $\Over{\BB}{A}$-groupoids in the pullback of $q$ along the map $\id\times \ord{1}\colon\Delta^1\to \Delta^1\times\Fin_*$. As $A\in\BB$ was chosen arbitrarily and in light of Remark~\ref{rem:BCCocompletionRelations}, this shows that the functor $\I{M}^\otimes_{\I{U}}\times_{\Delta^1\times\Fin_*}\Delta^1\to\Delta^1$ that is obtained as the pullback of $q$ along $(\epsilon\leq\delta,\alpha)\colon\Delta^1\to \Delta^1\times\Fin_*$ must be a cocartesian fibration (see~\cite[Lemma~6.5.2]{MCocartesian}). As every map in the constant $\BB$-category $\Delta^1\times\Fin_*$ is \emph{locally} contained in the image of the constant sheaf functor, the pullback of $q$ along \emph{any} map $\Delta^1\to \Delta^1\times\Fin_*$ in $\Cat(\BB)$ is a cocartesian fibration after passing to a suitable cover $\bigsqcup_i A\onto 1$ in $\BB$ and must therefore be a cocartesian fibration itself, using~\cite[Remark~3.1.2]{MCocartesian}.
	In particular, $q(A)$ is a locally cocartesian fibration of $\infty$-categories for every $A\in\BB$, and since it follows from proposition~\ref{prop:cocompletionRelationsTransitivity} that the locally cocartesian maps are closed under composition, we conclude that $q(A)$ is a cocartesian fibration.
	Since cocompletions with relations are compatible with \'etale base change (Remark~\ref{rem:BCCocompletionRelations}), the transition functors $s^\ast\colon\I{M}^\otimes_{\I{U}}(A)\to\I{M}^\otimes_{\I{U}}(B)$ preserve cocartesian morphisms, which by~\cite[Proposition~3.1.7]{MCocartesian} implies that $q$ is a cocartesian fibration, as claimed.
\end{proof}

\begin{corollary}
	\label{cor:CatUMonoidal}
	The functor $\ICat_{\BB}^{\cocont{\I{U}},\otimes}\to\Fin_*$ is a cocartesian fibration that gives rise to a symmetric monoidal structure on the $\BB$-category $\ICat_{\BB}^{\cocont{\I{U}}}$.
\end{corollary}
\begin{proof}
	Since the map $\ICat_{\BB}^{\cocont{\I{U}},\otimes}\to\Fin_*$ is a pullback of the functor $q$ from Proposition~\ref{prop:UConstrGivesCocartFib}, the same proposition immediately implies the first claim. Moreover, the straightforward observation that for every $n\geq 0$ the equivalence
	\begin{equation*}
		(\ICat_{\BB}^\times)_n\simeq \prod_{i=1}^n (\ICat_{\BB}^\times)_1
	\end{equation*}
	restricts to an equivalence
	\begin{equation*}
		(\ICat_{\BB}^{\cocont{\I{U}},\otimes})_n\simeq \prod_{i=1}^n (\ICat_{\BB}^{\cocont{\I{U}},\otimes})_1
	\end{equation*}
	shows the second claim.
\end{proof}

\begin{remark}
	\label{rem:UnitIsCocompletion}
	By unstraightening the cocartesian fibration $ q $ from Proposition~\ref{prop:UConstrGivesCocartFib} we get a functor $ \Delta^1 \rightarrow \ICMon(\ICat_{\BB}) $ and therefore a morphism of symmetric monoidal $ \BB $-categories $L \colon \ICat_\BB^\times \rightarrow \ICat_\BB^{\cocont{\I{U}},\otimes} $. Note that the pullback $\Delta^1\times_{\Delta^1\times\Fin_*}\I{M}^\otimes_{\I{U}}\to \Delta^1$ of $q$ along $\id\times\ord{1}\colon\Delta^1\to\Delta^1\times\Fin_*$ is also a \emph{cartesian} fibration: in fact, by making use of~\cite[Proposition~6.5.1]{MCocartesian} this follows from the straightforward observation that the adjunction $(d^1\dashv s^0)\colon \Delta^1\otimes\ICat_{\BB}\leftrightarrows \ICat_{\BB}$ restricts to an adjunction $\Delta^1\times_{\Delta^1\times\Fin_*}\I{M}^\otimes_{\I{U}}\leftrightarrows \ICat_{\BB}$. By~\cite[Corollary~6.5.5]{MCocartesian}, this means that $L$ is the left adjoint of the inclusion $\ICat_{\BB}^{\cocont{\I{U}}}\into\ICat_{\BB}$ as provided by~\cite[Corollary~7.1.15]{MWColimits}.
	In particular we see that the $ \BB $-category underlying the tensor unit of $ \ICat_{\BB}^{\cocont{\I{U}},\otimes}$ is equivalent to the free $\I{U}$-cocompletion of the point $ \IPSh^{\I{U}}(1)$.
\end{remark}

\begin{remark}
	\label{rem:UCocompSymMonCat}
	By a similar argument as in Remark~\ref{rem:UnitIsCocompletion}, the projection $\I{M}_{\I{U}}^\otimes\to\Delta^1$ is both cartesian and cocartesian. Therefore, one also obtains an adjunction $(L\dashv i)\colon \ICat_{\BB}^{\cocont{\I{U},\otimes}}\leftrightarrows\ICat_{\BB}^\times$ in which $i$ is simply the inclusion. Since the projection $\I{M}_{\I{U}}^\otimes\to\Fin_*$ carries every map in $\I{M}_{\I{U}}^\otimes$ that is cartesian over $\Delta^1$ to an equivalence, taking global sections and pulling back along the map $\Fin_*\to\Gamma\Fin_*$ yields a relative adjunction $\Cat(\BB)^{\cocont{\I{U}},\otimes}\leftrightarrows\Cat(\BB)^\times$ over $\Fin_*$. As both maps are morphisms of $\infty$-operads, we obtain an induced adjunction
	\[
	(L\dashv i)\colon \CAlg(\Cat(\BB)^{\cocont{\I{U}}})\leftrightarrows\CAlg(\Cat(\BB))\simeq\Cat(\BB)^\otimes
	\]
	of $\infty$-categories. By unwinding the definitions, we see that a symmetric monoidal $ \BB $-category $ \I{C}^\otimes $ lies in $  \CAlg(\Cat(\BB)^{\cocont{\I{U}}}) $ if and only if its underlying $\BB$-category is $ \I{U} $-cocomplete and the functor $	-\otimes- \colon \I{C} \times \I{C} \rightarrow \I{C}$ is $ \I{U} $-bilinear.
	In particular it follows from Remark~\ref{rem:UnitIsCocompletion} that $ \IPSh^{\I{U}}(1) $ can be canonically equipped with the structure of a symmetric monoidal  $ \BB $-category $ \IPSh^{\I{U}}(1)^\otimes  $ such that $ - \otimes - \colon  \IPSh^{\I{U}}(1) \times \IPSh^{\I{U}}(1) \rightarrow \IPSh^{\I{U}}(1) $ is $ \I{U} $-bilinear and that the canonical functor $ 1 \rightarrow \IPSh^{\I{U}}(1)^\otimes $ induced by the adjunction unit is symmetric monoidal.
	So in particular we have a commuative diagram
	\[\begin{tikzcd}
		{1 \times 1} & {1} \\
		{\IPSh^{\I{U}}(1) \times \IPSh^{\I{U}}(1)} & {\IPSh^{\I{U}}(1)}.
		\arrow[from=1-1, to=1-2]
		\arrow[from=1-2, to=2-2]
		\arrow[from=1-1, to=2-1]
		\arrow[from=2-1, "-\otimes- ", to=2-2]
	\end{tikzcd}\]
	By the universal property of $ \IPSh^{\I{U}}(1) $ and Lemma~\ref{lem:limitpreservationswap} there is a unique such functor $-\otimes -$, which must therefore coincide with the product functor $ \IPSh^{\I{U}}(1) \times \IPSh^{\I{U}}(1) \to \IPSh^{\I{U}}(1)$, see \cite[Proposition 7.3.8]{MWColimits}.
\end{remark}

\begin{example}
	\label{rem:FreeCocompIsInitialPBFromalism}
	Let $ \BB = \PSh(\CC) $ for some small $ \infty $-category $ \CC $ and let $ P \subseteq \CC $ be a subcategory that is closed under pullbacks.
	Then $ P $ generates a local class $ W $ in $ \PSh(\CC) $ and therefore a full subcategory $ \Univ_W \hookrightarrow \Univ $.
	Then \cite[Propositions 5.4.2 and 5.4.5]{MWColimits} imply together with Remark~\ref{rem:UCocompSymMonCat} that the $ \infty $-category $ \CAlg(\Cat(\BB)^{\cocont{\Univ_W}}) $ is equivalent to the $ \infty $-category of \emph{pullback formalisms} in the sense of~\cite[\S 2.2]{drew2020universal}.
	By Remarks~\ref{rem:UnitIsCocompletion} and \ref{rem:UCocompSymMonCat} the initial object of $ \CAlg(\Cat(\BB)^{\cocont{{\Univ_W}}}) $ is equivalent to the free $\Univ_W $-cocompletion of the point $ \IPSh^{\Univ_W}(1) $ equipped with the cartesian monoidal structure.
	Since \cite[Example 7.3.5]{MWColimits} shows that $ \IPSh^{\Univ_W}(1) $ agrees with the geometric pullback formalism constructed in \cite[\S 4]{drew2020universal}, this gives a new proof of \cite[Theorem 3.25]{drew2020universal}.
	Furthermore our proof yields a slightly more general result because in loc. cit. the assumption that $ \CC $ is a $ 1$-category was made.
\end{example}

We will now move one universe up and consider the case where $ \I{U} = \Cat_{\BB} $ is the internal class of small $ \BB $-categories in $ \Cat_{\BBB} $.
By the above we obtain a symmetric monoidal structure $ \ICat_{\BBB}^{\cc,\otimes}$ on the very large $ \BB $-category $ \ICat_{\BBB}^{\cc} $ of cocomplete large $ \BB $-categories and cocontinuous functors.

\begin{proposition}
	\label{prop:tensorOfPresIsPres}
	The tensor product $ - \otimes - \colon \ICat_{\BB}^{\cc} \times \ICat_{\BB}^{\cc} \rightarrow \ICat_{\BB}^{\cc} $ of cocomplete $ \BB $-categories restricts to a functor $-\otimes - \colon \ILPr_\BB \times \ILPr_\BB \rightarrow \ILPr_\BB $.
	Therefore $  \ILPr_\BB $ inherits the structure of a symmetric monoidal $ \BB $-category from  $ \ICat(\BB)^{\cc,\otimes} $.
\end{proposition}
\begin{proof}
	In light of the observation that the tensor unit in $\ICat_{\BB}^\cc$ is given by the presentable $\BB$-category $\Univ$, the second claim follows from  Remark~\ref{rem:SymMonStronSubcats}, so it suffices to show the first one.
	It will be enough to see that if $ \I{D} $ and $ \I{E} $ are presentable then so is their tensor product $ \I{D} \otimes \I{E} $. By Corollary~\ref{cor:characterisationPresentableSheaves}, we may find a sound doctrine $\I{U}$ and $\I{U}$-cocomplete (small) $\BB$-categories $\I{C}$ and $\I{C}^\prime$ such that $\I{D}\simeq\IShv_{\Univ}^{\op(\I{U})}(\I{C})$ and $\I{E}\simeq\IShv_{\Univ}^{\op(\I{U})}(\I{C}^\prime)$. If $\I{X}$ is an arbitrary cocomplete large $\BB$-category, we compute
	\begin{align*}
		\IFun^\cc(\I{D}\otimes\I{E},\I{X}) &\simeq \IFun^\cc(\I{D},\IFun^\cc(\I{E},\I{X}))\\
		&\simeq \IFun^{\cocont{\I{U}}}(\I{C},\IFun^{\cocont{\I{U}}}(\I{C}^\prime,\I{X}))\\
		&\simeq \IFun^{\mult{\I{U}}}(\I{C}\times\I{C}^\prime,\I{X})\\
		&\simeq \IFun^{\cocont{\I{U}}}(\I{C}\otimes^{\I{U}}\I{C}^\prime,\I{X}),
	\end{align*}
	where the first and third equivalence are consequences of Lemma~\ref{lem:limitpreservationswap}, the second equivalence follows from Corollary~\ref{cor:UCocompleteIndUPresentable} and where $-\otimes^{\I{U}}-$ denotes the tensor product in $\ICat_{\BB}^{\cocont{\I{U}}}$. Now $\I{U}$ being a doctrine implies that the tensor product $\I{C}\otimes^{\I{U}}\I{C}^\prime$ is small (see Remark~\ref{rem:CocompletionRelationsSmall}), hence another application of Corollary~\ref{cor:UCocompleteIndUPresentable} gives rise to an equivalence $\IFun^{\cocont{\I{U}}}(\I{C}\otimes^{\I{U}}\I{C}^\prime,\I{X})\simeq \IFun^\cc(\IShv_{\Univ}^{\op(\I{U})}(\I{C}\otimes^{\I{U}}\I{C}^\prime),\I{X})$. As the same corollary shows that $\IShv_{\Univ}^{\op(\I{U})}(\I{C}\otimes^{\I{U}}\I{C}^\prime)$ is presentable and as all of the above equivalences are natural in $\I{X}$, the result follows.
\end{proof}

\begin{definition}
	\label{def:presentablyMonoidal}
	A symmetric monoidal $\BB$-category $\I{D}$ is called \emph{presentably} symmetric monoidal if $\I{D}$ is contained in the image of the inclusion $\CAlg(\LPr(\BB))\into\CAlg(\Cat(\BBB)\simeq \Cat(\BBB)^\otimes$. In other words, $\I{D}$ is presentably symmetric monoidal if $\I{D}$ is a presentable $\BB$-category and the tensor functor $-\otimes -$ is bilinear.
\end{definition}

\begin{proposition}
	\label{prop:sheaveshasunivprop}
	Let $\I{D}$ and $\I{E}$ be presentable $ \BB $-categories.
	Then there is an equivalence of $\BB$-categories
	\[
	\IFun^\Lad( \IShv_{\I{E}}(\I{D}), \I{X} ) \simeq \IFun^{\bil}(\I{D} \times \I{E},\I{X})
	\]
	that is natural in $ \I{X}\in\LPr(\BB) $ and therefore in particular an equivalence $  \IShv_{\I{E}}(\I{D}) \simeq \I{D} \otimes \I{E} $.
\end{proposition}
\begin{proof}
	Let us denote by $\IFun^\c(-,-)\into\IFun(-,-)$ the full subcategory spanned by the continuous functors. 
	We claim that we have a chain of equivalences
	\begin{align*}
		\IFun^{\bil}(\I{D} \times \I{E},\I{X}) &\simeq \IFun^\Lad(\I{D},\IFun^\Lad(\I{E},\I{X})) \\
		&\simeq \IFun^\c(\I{D}^\op, \IFun^\Lad(\I{E},\I{X})^\op)^\op \\
		&\simeq \IFun^{\c}(\I{D}^\op,\IFun^\Rad(\I{X},\I{E}))^\op \\
		& \simeq \IFun^\Rad(\I{X},\IFun^{\c}(\I{D}^\op,\I{E}))^\op \\
		& \simeq \IFun^\Lad(\IFun^{\c}(\I{D}^\op,\I{E}), \I{X}) \\
		& \simeq \IFun^\Lad(\IShv_{\I{E}}(\I{D}),\I{X})
	\end{align*}
	that are natural in $ \I{E} $.
	The first equivalence follows from Lemma~\ref{lem:limitpreservationswap}, the second and the last equivalences are obvious and the third and fifth equivalences follow from Proposition~\ref{prop:adjointfunctorII}, so it remains to argue that the fourth equivalence holds.
	
	We may choose a sound doctrine $\I{U}$ such that $\I{D} \simeq \IShv_{\Univ}^{\I{U}}(\I{C})$ for some small $\I{U}$-cocomplete $\BB$-category $\I{C}$ (cf.\ Corollary~\ref{cor:characterisationPresentableSheaves}). Using Corollary~\ref{cor:UCocompleteIndUPresentable}, we only need to see that the equivalence $ \IFun(\I{C}^\op,\IFun(\I{X},\I{E})) \simeq \IFun(\I{X},\IFun(\I{C}^\op,\I{E}))$ restricts to an equivalence
	\[
	\IFun^{\cont{\I{U}}} (\I{C}^\op,\IFun^\Rad(\I{X},\I{E}))\simeq \IFun^\Rad(\I{X},\IFun^{\cont{\I{U}}}(\I{C}^\op,\I{E}))
	\]
	(where $\IFun^{\cont{\I{U}}}(-,-)$ denotes the full subcategory of $\IFun(-,-)$ that is spanned by the $\pi_A^\ast\I{U}$-continuous functors of $\Over{\BB}{A}$-categories, for all $A\in\BB$).
	We already know from (the dual version of) Lemma~\ref{lem:limitpreservationswap} that we have an equivalence
	\[
	\IFun^{\cont{\I{U}}}(\I{C}^\op,\IFun^{\c}(\I{X},\I{E})) \simeq \IFun^{\c}( \I{X},\IFun^{\cont{\I{U}}}(\I{C}^\op , \I{E})),
	\]
	hence Proposition~\ref{prop:adjointfunctorII} together with~\cite[Remark~5.3.2]{MWColimits} and Remark~\ref{rem:AccOfFunctorsIsLocal} implies that the proof is finished once we verify that a functor $f \colon \I{X} \rightarrow \IFun( \I{C}^\op,\I{E})$ is accessible if only if its transpose $f^\prime \colon \I{C}^\op \rightarrow \IFun( \I{X}, \I{E})$ takes values in the full subcategory $\IFun^{\acc}(\I{X},\I{E})$ spanned by the accessible functors.
	If $ f $ is accessible there is some sound doctrine $ \I{U} $ such that $ f $ is $ \IFilt_{\I{U}} $-cocontinuous.
	But then it follows from Lemma~\ref{lem:limitpreservationswap} that $ f^\prime$ takes values in the $\BB$-category $ \IFun^{\cocont{\IFilt_{\I{U}}}}(\I{X}, \I{E} ) \into \IFun^{\acc}(\I{X}, \I{E} )$, as desired.
	For the converse, suppose that $ f^\prime $ takes values in $  \IFun^{\acc} (\I{X}, \I{E} )$. Let $z\colon \I{C}_0\to\I{C}$ be the tautological object. Then $f^\prime(z)\colon \pi_{\I{C}_0}^\ast\I{X}\to\pi_{\I{C}_0}^\ast\I{E}$ is $\pi_{\I{C}_0}^\ast\I{U}$-accessible for some sound doctrine $\I{U}$. Since \emph{every} object in $\I{C}$ is a pullback of $z$, this already shows that $f^\prime$ takes values in $\IFun^{\cocont{\IFilt_{\I{U}}}}(\I{X},\I{E})$, hence Lemma~\ref{lem:limitpreservationswap} shows that $f$ is accessible.
\end{proof}

For later use we also record the following explicit description of the universal bilinear functor if one of the factors is presheaf category:

\begin{lemma}
	\label{lem:explicit_descr_of_univ_bilin}
	Let $ \I{C} $ be small $ \BB $-category and $ \I{D} $ a presentable $ \BB $-category.
	Then under the identification $ \IPSh(\I{C}) \otimes \I{D} \simeq \iFun(\I{C}^{\op},\I{D}) $ the universal bilinear functor $ \tau \colon \IPSh(\I{C}) \times \I{D} \to \iFun(\I{C}^{\op},\I{D})  $ is given by the transpose of the composite
	\[
	\tau' \colon \I{C}^{\op} \times \IPSh(\I{C}) \times \I{D} \xrightarrow{\ev} \Univ[\BB] \times \I{D} \xrightarrow{- \otimes -} \I{D}.
	\]
\end{lemma}
\begin{proof}
	We at first prove the claim when $ \I{D} = \IPSh(\I{D}_0) $ as well.
	In this case it is an easy consequence of \cite[Theorem~7.1.1]{MWColimits} that the universal bilinear functor $ \tau $ is the unique bilinear functor $ \IPSh(\I{C}) \times \IPSh(\I{D}_0) \to \iFun(\I{C}^{\op},\IPSh(\I{D}_0)) \simeq \IPSh(\I{C} \times \I{D}_0) $ such that the composite
	\[
	\I{C}_0 \times \I{D}_0 \xrightarrow{h_{\I{C}} \times h_{\I{D}_0}} \IPSh(\I{C}) \times \IPSh(\I{D}_0) \to \IPSh(\I{C} \times \I{D}_0)
	\]
	is given by the Yoneda-embedding $ h_{\I{C} \times \I{D}_0} $ (see also the proof of Proposition~\ref{prop:tensorOfPresIsPres}).
	Recall that the functor $ - \otimes - $ is the unique bilinear functor that corresponds to the identity under the equivalence
	\[
	\Fun_{\BB}^{\bil}(\Univ[\BB] \times \IPSh(\I{D}_0), \IPSh(\I{D}_0)) \simeq \Fun_{\BB}^L(\IPSh(\I{D}_0), \IPSh(\I{D}_0)	).
	\]
	From this it follows that the composite
	\[
	\I{C}^{\op} \times \IPSh(\I{C}) \times \IPSh(\I{D}_0) \xrightarrow{\ev} \Univ[\BB] \times \IPSh(\I{D}_0) \xrightarrow{- \otimes -} \IPSh(\I{D}_0) 
	\]
	transposes to the functor
	\[
	\I{C}^{\op} \times\IPSh(\I{C}) \times \I{D}_0^{\op}\times\IPSh(\I{D}_0) \xrightarrow{\ev \times \ev} \Univ \times \Univ \xrightarrow{- \times - } \Univ[\BB].
	\]
	But after composing with $ \I{C}^{\op}  \times \I{C} \times \I{D_0}^{\op}  \times \I{D_0} \xrightarrow{\id \times h_{\I{C}} \times \id \times h_{\I{D}}} \I{C}^{\op} \times\IPSh(\I{C}) \times \I{D}_0^{\op}\times\IPSh(\I{D}_0) $ this functor yields $ \map{\I{C} \times \I{D}_0}(-,-) $ and thus transposes to the Yoneda-embedding $ h_{\I{C} \times \I{D}_0} $, as desired.
	Now for the case of a general presentable $ \BB $-category $ \I{D} $ we pick a Bousfield localization $ L \colon \IPSh(\I{D}_0) \to \I{D} $.
	Note that we have two commutative squares
	\[\begin{tikzcd}
		{\I{C}^{\op} \times \IPSh(\I{C}) \times \IPSh(\I{D}_0)} & {\IPSh(\I{D}_0)} \\
		{\I{C}^{\op} \times \IPSh(\I{C}) \times \I{D}} & {\I{D}}
		\arrow["{\tau'}"', shift right, from=1-1, to=1-2]
		\arrow["\tau", shift left, from=1-1, to=1-2]
		\arrow["{\id\times L}"', from=1-1, to=2-1]
		\arrow["L"', from=1-2, to=2-2]
		\arrow["{\tau'}"', shift right, from=2-1, to=2-2]
		\arrow["\tau", shift left, from=2-1, to=2-2]
	\end{tikzcd}\]
	with and without the prime.
	But by the first part of  the proof the upper two functors agree and thus so do the lower two because $ \id \times L $ has a section.
\end{proof}

\subsection{$\BB$-modules as presentable $\BB$-categories}
\label{sec:BModules}
By the discussion in the previous section, there is a symmetric monoidal functor $L\colon \ICat_{\BBB}^{\times}\to\ICat_{\BBB}^{\cc,\otimes}$ whose underlying functor of very large $\BB$-categories is the left adjoint of the inclusion $\ICat_{\BBB}^\cc\into\ICat_{\BBB}$. Upon taking global sections, we thus deduce from~\cite[Corollary~7.3.2.7]{Lurie2017} that the inclusion determines a lax symmetric monoidal functor $\Cat(\BBB)^{\cc,\otimes}\into\Cat(\BBB)^\times$ (of symmetric monoidal $\infty$-categories, i.e.\ a map of $\infty$-operads) and therefore a fortiori a lax symmetric monoidal functor $\LPr(\BB)^\otimes\into\Cat(\BBB)^\times$. Moreover, as the global sections functor $\Gamma$ preserves limits, it defines a symmetric monoidal functor $\Cat(\BBB)^\times\to\CatSS^\times$.
Since a multilinear functor in $\Cat(\BB)$ induces a multilinear functor on the underlying $\infty$-categories of global sections, it is evident that the induced map $\Cat(\BBB)^{\cc,\otimes}\to \CatSS^\times$ takes values in $\CatSS^{\cc,\otimes}\into\CatSS^\times$ and therefore defines a lax symmetric monoidal functor $\Gamma^{\cc,\otimes}\colon\Cat(\BBB)^{\cc,\otimes}\to\CatSS^{\cc,\otimes}$. Upon restricting this functor to presentable $\BB$-categories, we now end up with a lax symmetric monoidal functor $\Gamma^{\cc,\otimes}\colon \LPr(\BB)^\otimes\to (\LPrS)^\otimes$
that in turn induces a map
\begin{equation*}
	\Gamma^{\lin}\colon\LPr(\BB)\simeq\Mod_{\Univ}(\LPr(\BB))\to\Mod_{\BB}(\LPrS)
\end{equation*}
(where $\BB$ is regarded as the algebra in $\LPrS$ that is given by image of the trivial algebra $\Univ$ in $\LPr(\BB)$ along $\Gamma^{\cc,\otimes}$). Note that this is precisely the \emph{cartesian} monoidal structure on $\BB$ as the product bifunctor $\Univ\times\Univ\to\Univ$ is bilinear, cf.~\cite[Lemma~6.2.7]{MCocartesian}.

The main goal of this section is to show that $\Gamma^{\lin}$ admits a fully faithful left adjoint that embeds $\Mod_{\BB}(\LPrS)$ into $\LPr(\BB)$ and to give an explicit description of this embedding.
As a preliminary step, we need to show that the global sections functor $\Gamma\colon\LPr(\BB)\to\LPrS$ admits a left adjoint. To that end, recall from Example~\ref{ex:tensorConstructionPresentable} that there is a functor $ - \otimes \Univ \colon \RPrS \rightarrow \RPr(\BB)$ that assigns to a presentable $\infty$-category $ \DD $ the presentable $ \BB $-category that is given by the sheaf $\DD\otimes\Over{\BB}{-}$. Using Proposition~\ref{prop:RPrOpposite}, we may equivalently regard this map as a functor $\LPrS\to\LPr(\BB)$. We now obtain:
\begin{proposition}
	\label{prop:TensorIsLeftAdjointToPresGlobalSections}
	The functor $- \otimes \Univ$ is left adjoint to the global sections functor $\Gamma \colon \LPr(\BB) \rightarrow \LPrS$.
\end{proposition}
\begin{proof}
	The composite $\Gamma \circ (- \otimes \Univ)$ is by definition given by the endofunctor $- \otimes \BB \colon \LPrS \rightarrow \LPrS$
	so that the functor $\Gamma_\ast\colon \Shv_{\BB}(-)\to \id_{\RPrS}$ defines a natural transformation $\eta\colon \id_{\LPrS}\to -\otimes \BB$ upon passing to opposite $\infty$-categories.
	We need to show that the composition
	\begin{equation*}
		\tag{$\ast$}\label{eq:adjunction}
		\map{\LPr(\BB)}(\DD\otimes\Univ, \I{E}) \rightarrow \map{\LPrS}(\DD \otimes \BB, \Gamma\I{E}) \xrightarrow{\eta_{\DD}^*} \map {\LPrS}(\DD, \Gamma \I{E})
	\end{equation*}
	is an equivalence.
	Choose a regular cardinal $\kappa$ such that $\DD\simeq\Shv^{\kappa}_{\SS}(\CC)$ for some small $\infty$-category $\CC$ that admits $\kappa$-small colimits. Using Proposition~\ref{prop:USheavesPresentable}, we obtain an equivalence $\DD\otimes\Univ\simeq\IShv_{\Univ}^{\ILConst_{\kappa}}(\CC)$ with respect to which the map $\eta_{\DD}$ corresponds to the left adjoint of $\Gamma_\ast\colon\Shv_{\BB}^\kappa(\CC)\to\Shv_{\SS}^\kappa(\CC)$. Again using Proposition~\ref{prop:USheavesPresentable}, we have equivalences
	\begin{equation*}
		\map{\LPr(\BB)}(\DD\otimes\Univ,\I{E})\xrightarrow{(h_{\CC}^\BB)^\ast} \map{\Cat(\BBB)^{\cocont{\ILConst_\kappa}}}(\CC,\I{E})\simeq \map{\CatSS^{\cocont{\kappa}}}(\CC,\Gamma\I{E})\xrightarrow{(h_{\CC}^\SS)^\ast}\map{\LPrS}(\DD,\Gamma\I{E})
	\end{equation*}
	where $h_{\CC}^\BB$ is the Yoneda embedding in $\Cat(\BBB)$ and $h_{\CC}^\SS$ is the Yoneda embedding in $\CatSS$. On account of the commutative square
	\begin{equation*}
		\begin{tikzcd}
			\CC\arrow[d]\arrow[r, hookrightarrow, "h_{\CC}^\SS"] & \Shv_{\SS}^\kappa(\CC)\arrow[d, "\eta_{\DD}"]\\
			\Gamma\CC\arrow[r, hookrightarrow, "\Gamma(h_{\CC}^\BB)"] & \Shv_{\BB}^\kappa(\CC)
		\end{tikzcd}
	\end{equation*}
	in which the vertical map on the left is the unit of the adjunction $\const_{\BB}\dashv\Gamma$ (see~\cite[Lemma~6.4.5]{MCocartesian}), the composition of the above chain of equivalences recovers the map in~(\ref{eq:adjunction}), hence the claim follows.
\end{proof}

\begin{proposition}
	\label{prop:ModulesAreaReflectiveSubcat}
	The functor $\Gamma^{\lin} \colon \LPr(\BB) \to \Mod_\BB(\LPrS)$ admits a fully faithful left adjoint.
\end{proposition}
\begin{proof}
	Note that since $ \LPr(\BB) \simeq \RPr(\BB)^\op$ it follows from Proposition~\ref{prop:limitsPrR} that the global sections functor $ \Gamma \colon \LPr(\BB) \rightarrow \LPrS $ preserves colimits.
	So in light of Proposition~\ref{prop:TensorIsLeftAdjointToPresGlobalSections}  we may apply~\cite[Corollary 4.7.3.16]{Lurie2017} to the commutative triangle
	\[
	\begin{tikzcd}[column sep=small]
		\LPr(\BB) \arrow[rd, "\Gamma"'] \arrow[rr, "\Gamma^{\lin}"] &       & \Mod_{\BB}(\LPrS) \arrow[ld, "U"] \\
		& \LPrS &                                                
	\end{tikzcd}	
	\]
	(where $U$ denotes the forgetful functor), which yields the claim.
\end{proof}

We will now give a more explicit description of the left adjoint from Proposition~\ref{prop:ModulesAreaReflectiveSubcat}. To that end, observe that 
the functor $\IFun(-,\Univ)\colon \Univ^\op\to\ICat_{\BBB}$ takes values in $\IRPr_{\BB}$ and therefore determines a limit-preserving map $\BB^\op\to \RPr(\BB)\simeq(\LPr(\BB))^\op$ which by postcomposition with $\Gamma^{\lin}$ results in a limit-preserving functor $\Over{\BB}{-}\colon \BB^\op\to \Mod_{\BB}(\LPrS)^\op$.
We now get a map
\[
\Mod_{\BB}(\LPrS)^\op\times\BB^\op \xrightarrow{(-\otimes_{\BB}\Over{\BB}{-})^\op} \Mod_\BB(\LPrS)^\op \to (\LPrS)^\op\simeq \RPrS\into \CatSS
\] 
and hence by adjunction a functor $\Mod_\BB(\LPrS)^\op \rightarrow \PSh_{\CatSS}(\BB)$.

\begin{lemma}
	\label{lem:TensorConstrLandsInPrL}
	The functor $\Mod_\BB(\LPrS)^\op \rightarrow \PSh_{\CatSS}(\BB)$ factors through $\RPr(\BB)$ and thus defines a functor
	\[
	- \otimes_\BB \Univ \colon \Mod_\BB(\LPrS) \rightarrow \LPr(\BB).
	\]
\end{lemma}
\begin{proof}
	First we prove that the functor factors through $\Cat(\BBB)$.
	This amounts to showing that the functor $\DD\otimes_{\BB}\Over{\BB}{-}\colon\BB^\op\to\CatSS$
	is continuous for every $\DD \in \Mod_{\BB}(\LPrS)$. As the functor $\Over{\BB}{-}\colon \BB^\op\to \Mod_{\BB}(\LPrS)^\op$ preserves limits, this follows from the fact that $\DD\otimes_{\BB}-$, viewed as an endofunctor on $\Mod_{\BB}(\LPrS)^\op$, preserve limits as well~\cite[Corollary~4.4.2.15]{Lurie2017}.
	Next, we show that the resulting $ \BB  $-category $\DD\otimes_{\BB}\Univ$ is presentable.
	As it by construction takes values in $\RPrS$, Theorem~\ref{thm:characterisationPresentableCategories} implies that it suffices to show that $ \DD\otimes_{\BB}\Univ$ is $ \Univ $-cocomplete and that the transition functors are cocontinuous. Both statements follow from the observation that the functor $\DD\otimes_{\BB}-\colon \Mod_{\BB}(\LPrS)\to\Mod_{\BB}(\LPrS)$ can be upgraded to an $(\infty,2)$-functor (see \cite[\S 4.4]{hoyois2017higher} for details) and that for any $s \colon B \rightarrow A$ in $\BB$ the adjunction $s_!\dashv s^\ast$ is $\BB$-linear, see~\cite[Corollary~7.3.2.7]{Lurie2017}.
	To finish the proof, it remains to see that for any map of $\BB$-modules $\DD \rightarrow \EE$ the induced map $\EE \otimes_\BB \Univ \rightarrow \DD \otimes_\BB\Univ$ admits a left adjoint.
	By construction, it has one section-wise, so it suffices to check that for any map $s \colon B \rightarrow A$ in $\BB$ the induced lax square
	\[
	\begin{tikzcd}
		\DD\otimes_\BB \mathcal{\BB}_{/B} \arrow[r]           & \EE \otimes_\BB \mathcal{\BB}_{/B}           \\
		\DD \otimes_\BB \mathcal{\BB}_{/A} \arrow[r] \arrow[u] & \EE \otimes_\BB \mathcal{\BB}_{/A} \arrow[u]
	\end{tikzcd}
	\]
	commutes. Using again $(\infty,2)$-functoriality of the relative tensor product, this follows by essentially the same argument as in the proof of~\cite[Lemma~4.2.10]{MWColimits}.
\end{proof}

\begin{remark}
	It also seems natural to consider the functor
	\[
	\Mod_\BB(\LPrS) \times \BB^\op \xrightarrow{\id \times \BB_{/-}} \Mod_\BB(\LPrS) \times \Mod_{\BB}(\LPrS) \xrightarrow{- \otimes_\BB -}  \Mod_\BB(\LPrS) \to \Cat_\infty
	\]
	which by transposition also gives rise to a functor $ \Mod_\BB(\LPrS) \to \Fun(\BB^\op,\Cat_\infty) $.
	We expect that this functor takes values in $ \Cat(\BB) $ and is equivalent to $ -\otimes_\BB \Univ $.
	It is easy to see that for fixed $ \CC \in  \Mod_\BB(\LPrS) $ the two resulting presheaves of categories on $ \BB $ have the same value on objects and morphisms.
	However, a proof that they agree as functors seems two require $ (\infty,2) $-categorical techniques that are not quite available yet.
\end{remark}

\begin{lemma}
	\label{lem:TensorConstrPreColimits}
	The functor $- \otimes_\BB \Univ \colon \Mod_\BB(\LPrS) \rightarrow \LPr(\BB)$ preserves colimits.
\end{lemma}
\begin{proof}
	As limits in $\Cat(\BBB)$ are computed section-wise, it suffices to show that for every $A\in\BB$ the functor
	\[
	\Mod_\BB(\LPrS)^\op \xrightarrow{ (-\otimes_\BB \BB_{/A})^\op} \Mod_\BB(\LPrS)^\op \rightarrow (\LPrS)^\op \simeq \RPrS \rightarrow \CatSS
	\]
	preserves limits, which is obvious.
\end{proof}

\begin{proposition}
    \label{prop:otimes_Univ_is_left_adjoint}
	The functor $- \otimes_\BB \Univ$ defines a left adjoint of $\Gamma^{\lin}$.
\end{proposition}
\begin{proof}
	We show that $-\otimes_{\BB}\Univ$ is equivalent to the left adjoint $L$ of $\Gamma^{\lin}$ from Proposition~\ref{prop:ModulesAreaReflectiveSubcat}.
	Let us denote by $ - \otimes \BB \colon \LPrS \rightarrow \Mod_\BB(\LPrS) $ the left adjoint to the forgetful functor. 
	Then by the associativity of the relative tensor product (\cite[Proposition 4.4.3.14]{Lurie2017} we have equivalences
	\begin{equation}
		\label{eqn:DifferentTensorsAgree}
		(- \otimes_\BB \Univ) \circ (- \otimes \BB)   \simeq - \otimes \Univ \simeq L\circ (- \otimes \BB)\tag{$ \ast $}
	\end{equation}
	of functors from $\LPrS$ to $\LPr(\BB)$.
	By \cite[Remark 4.7.3.15]{Lurie2017} we may find a functor
	\[
	F \colon \Mod_\BB(\LPrS) \rightarrow \Fun(\Delta^\op, \LPrS)
	\]
	such that the composite 
	\[
	\Mod_\BB(\LPrS) \xrightarrow{F} \Fun(\Delta^\op,\LPrS) \xrightarrow{(- \otimes \BB)_*} \Fun(\Delta^\op,\Mod_\BB(\LPrS)) \xrightarrow{\colim_{\Delta^\op}} \Mod_\BB(\LPrS)
	\]
	is equivalent to the identity.
	From (\ref{eqn:DifferentTensorsAgree}) and Lemma~\ref{lem:TensorConstrPreColimits} it follows that the diagram
	\[
	\begin{tikzcd}
		{\Fun(\Delta^\op,\LPrS)} \arrow[r, "(-\otimes \BB)_*"] \arrow[d, "(-\otimes \BB)_*"] & {\Fun(\Delta^\op, \Mod_\BB(\LPrS))} \arrow[r, "\colim_{\Delta^\op}"] \arrow[d, "(-\otimes_\BB \Univ)_*"] & \Mod_\BB(\LPrS) \arrow[d, "- \otimes_\BB \Univ"] \\
		{\Fun(\Delta^\op, \Mod_\BB(\LPrS))} \arrow[r, "L_*"]                                      & {\Fun(\Delta^\op,\LPr(\BB))} \arrow[r, "\colim_{\Delta^\op}"]                                         & \LPr(\BB)                                             
	\end{tikzcd}
	\]
	commutes.
	Since $L$ commutes with colimits as well, we get an equivalence $L \simeq (- \otimes_\BB \Univ)$, as desired.
\end{proof}

The functor $-\otimes_{\BB}\Univ$ can be naturally extended to a strong monoidal functor. To see this, observe that since the global sections functor $\Gamma\colon\LPr(\BB)\to\LPrS$ admits an extension to a lax monoidal functor $\Gamma^{\cc,\otimes}\colon\LPr(\BB)^\otimes\to(\LPr)^\otimes$, the commutative diagram
\[
\begin{tikzcd}[column sep=small]
	\LPr(\BB) \arrow[rd, "\Gamma"'] \arrow[rr, "\Gamma^{\lin}"] &       & \Mod_{\BB}(\LPrS) \arrow[ld, "U"] \\
	& \LPrS &                                                
\end{tikzcd}	
\]
can be naturally extended to a diagram of lax monoidal functors. By passing to left adjoints, we thus obtain a commutative triangle
\[
\begin{tikzcd}[column sep=small]
	\LPr(\BB)^\otimes \arrow[from=rd, "-\otimes\Univ"] \arrow[from=rr, "-\otimes_{\BB}\Univ"'] &       & \Mod_{\BB}(\LPrS)^\otimes \arrow[from=ld, "-\otimes\BB"'] \\
	& (\LPr)^\otimes &                                                
\end{tikzcd}	
\]
of \emph{oplax} monoidal functors, see~\cite{Haugseng2020}. In order to show that the functor $-\otimes_{\BB}\Univ$ is strong monoidal, it thus suffices to show that the natural map
\begin{equation*}
	(-\otimes_{\BB} -)\otimes_{\BB}\Univ\to (-\otimes_{\BB}\Univ)\otimes (-\otimes_{\BB}\Univ)
\end{equation*}
is an equivalence. As both sides of this map preserve colimits in both variables and since every $\BB$-module can be written as a colimit of objects that are contained in the image of $-\otimes\BB$, it suffices to show that the natural map
\begin{equation*}
	(-\otimes -)\otimes\Univ\to (-\otimes\Univ)\otimes (-\otimes\Univ)
\end{equation*}
is an equivalence, i.e.\ that $-\otimes\Univ$ is strong monoidal. Recall (e.g.~from Remark~\ref{rem:PrRGeneratedByPresheafCategories}) that every presentable $\infty$-category can be obtained as a pushout (in $\LPrS$) of presheaf $\infty$-categories. The claim therefore follows from the observation that $-\otimes\Univ$ fits into a commutative square
\begin{equation*}
	\begin{tikzcd}
		\Cat(\BBB)^\times\arrow[d, "L"]\arrow[from=r, "\const_{\BB}"'] & \CatSS^\times\arrow[d, "L"]\\
		\LPr(\BB)^\otimes\arrow[from=r, "-\otimes\Univ"'] & (\LPr)^\otimes
	\end{tikzcd}
\end{equation*}
of oplax monoidal functors (which is again constructed from the associated commutative square of lax monoidal functors by passing to left adjoints) in which both vertical maps as well as $\const_{\BB}$ are strong monoidal. We conclude:
\begin{proposition}
	\label{lem:tensoringIsMonoidal}
	The functor $-\otimes_\BB\Univ\colon \Mod_{\BB}(\LPrS)\into \LPr(\BB)$ admits a natural enhancement to a strong monoidal functor.\qed
\end{proposition}

The functor $-\otimes_{\BB}\Univ$ being fully faithful raises the question what can be said about its essential image. First, we observe that there is an explicit criterion when a presentable $\BB$-category arises from a $\BB$-module:
\begin{remark}
	\label{rem:EssentialImageofModB}
	Let $ \I{C} $ be a presentable $\BB$-category. 
	Then the unit of the adjunction $ -\otimes_\BB \Univ \dashv \Gamma^{\lin}$ gives a canonical map  $ \Gamma^{\lin}(\I{C}) \otimes_\BB \Univ \to \I{C} $.
	For $A \in \BB$ the induced map $\varepsilon(A) \colon \BB_{/A} \otimes_\BB \Gamma^{\lin}(\I{C}) \to \I{C}(A)$ is the map underlying the essentially unique map of $\BB_{/A}$-modules that makes the diagram
	\[\begin{tikzcd}
		{\I{C}(1) \otimes_\BB \BB_{/A}} & {\I{C}(A))} \\
		{\I{C}(1)} & {\I{C}(1)}
		\arrow["\id", from=2-1, to=2-2]
		\arrow[from=2-2, to=1-2 ,"\pi_A^*"]
		\arrow[from=2-1, to=1-1]
		\arrow[from=1-1, to=1-2, "\varepsilon(A)"]
	\end{tikzcd}\]
	commute.
	It follows that a presentable $\BB$-category is in the essential image of $ - \otimes_\BB \Univ $ if and only if $\varepsilon(A)$ is an equivalence for all $A \in \BB$.
\end{remark}

Using the criterion from Remark~\ref{rem:EssentialImageofModB},  we are now able to write down an example of a presentable $\BB$-category that is \emph{not} in the essential image of $ - \otimes_\BB \Univ$. We learned about this example from David Gepner and Rune Haugseng.
\begin{example}
	Let us write $\Fin$ for the category of finite sets and let $\BB = \PSh(\Fin)$.
	Let $ X $ be a set with more than one element that we consider as an object in $ \BB $ via the Yoneda embedding.
	Then $ \IFun(X,\Univ_\BB) $ is a presentable $ \BB $-category that is not in the essential image of $-\otimes_{\BB}\Univ$. In fact, by Remark~\ref{rem:EssentialImageofModB} this would imply that the canonical map $ \varepsilon(X) $ being an equivalence.
	In our specific situation $ \varepsilon(X) $ is the canonical left adjoint functor $		\PSh(\Fin_{/X}) \otimes_{\PSh(\Fin)} \PSh(\Fin_{/X}) \to \PSh(\Fin_{/X \times X})$.
	Explicitly this functor is constructed by applying $\PSh(-)$ to the augmented cosimplicial diagram
	\[
	\Fin_{/X \times X} \to \Fin_{/X} \times \Fin_{/X} \rightrightarrows \Fin_{/X} \times \Fin \times \Fin_{/X} \cdots
	\]
	and then taking the induced map  $ \colim_{n \in \Delta^\op} \PSh(\Fin_{/X} \times \Fin^n \times \Fin_{/X}) \to \PSh(\Fin_{/X\times X})$ in $ \LPrS$.
	Thus, upon passing to right adjoints, we conclude that if the $\BB$-category $ \IFun(X, \Univ[\BB]) $ is contained in the essential image of $-\otimes_{\BB}\Univ$, the cosimplical diagram
	\[
	\PSh(\Fin_{/X \times X}) \to \PSh(\Fin_{/X} \times \Fin_{/X} ) \rightrightarrows \PSh(\Fin_{/X} \times \Fin \times \Fin_{/X} )\cdots 
	\]
	in $ \RPrS $ must be a limit diagram. We show that this cannot be true.
	Let us denote the map $ \Fin_{/X \times X} \to \Fin_{/X} \times \Fin^n \times \Fin_{/X} $ by $ f_n $. 
	It is given explicitly by the assignment 
	\[
	(A \to X \times X )\mapsto (A \to X \times X \xrightarrow{\pr_0} X, A, \dots, A , A \to X \times X \xrightarrow{\pr_1} X).
	\]
	Now for any $n\geq 1$ the map $ \PSh(\Fin_{/X \times X}) \to \PSh(\Fin_{/X} \times \Fin^n \times \Fin_{/X}) $ is the functor of right Kan extension $ (f_{n}^\op)_* $ along $ f_n^\op $.
	Hence, if the above cosimplicial diagram is a limit cone, the counit of the adjunctions $  (f_{n}^\op)^* \dashv (f_{n}^\op)_* $ yields an equivalence $ \colim_{n \in \Delta^{op}} (f_{n}^\op)^* (f_{n}^\op)_* F \to F $ for any $ F \in \PSh(\Fin_{/X \times X}) $. 
	For any object $ a = (A \to X, B_1,..., B_n, C \to X) \in \Fin_{/X} \times \Fin^n \times \Fin_{/X} $ we can compute $(f_{n}^\op)_* F (a) $ via the point-wise formula for right Kan extensions as a limit indexed by $ (\Fin_{/X \times X})^\op_{a/} $.
	But $ A \times B_1 \times ... \times B_n \times X \to X \times X$ defines an initial object of this category, hence we find
	\[
	F(A \to  X \times X) \simeq  \colim_{n \in \Delta^\op} F(A \times A^n \times A \to  X \times X).
	\]
	In particular, this shows that the map $ F(A \times A \to X \times X) \to F(A \to X \times X)$ induced by $ A \to A \times A $ is a cover in $\SS$.
	By taking $ F $ to be the presheaf represented by the diagonal $ X \to X \times X $, it in turn follows that the map
	\[
	\map{\Fin_{/X \times X}}(X \times X, X) \to \map{\Fin_{/X \times X}}(X , X )
	\]
	is surjective. In particular, there is a preimage of the identity $X\to X$.
	But since $ X $ has at least two elements there is no map $ \alpha $ making the diagram 
	\[\begin{tikzcd}
		{X \times X} && X \\
		& {X \times X}
		\arrow["\alpha", from=1-1, to=1-3]
		\arrow["\Delta", from=1-3, to=2-2]
		\arrow["\id"', from=1-1, to=2-2]
	\end{tikzcd}\]
	commute, which yields the desired contradiction.
\end{example}
There is, however, a class of $\infty$-topoi $\BB$ for which the functor $-\otimes_{\BB}$ turns out to be essentially surjective: those that are generated by $(-1)$-truncated objects:
\begin{proposition}
	\label{prop:Mod_B=PrL(B)_for_0-localic}
	Assume that $ \BB $ is generated under colimits by $ (-1) $-truncated objects.
	Then  $ - \otimes_\BB \Univ$ is an equivalence.
\end{proposition}
\begin{proof}
	By Propositions~\ref{prop:ModulesAreaReflectiveSubcat} and \ref{prop:otimes_Univ_is_left_adjoint} it remains to show essential surjectivity.
	Since $  - \otimes_\BB \Univ $ preserves colimits and every presentable $ \BB $-category is a pushout of presheaf $ \BB $-categories (see Remark~\ref{rem:PrRGeneratedByPresheafCategories}) it suffices to see that $ \IPSh(\I{C}) $ is in the essential image for any small $ \BB $-category $ \I{C} $.
	Furthemore, we can write $ \I{C} $ as a colimit of $ \BB $-categories of the form $ \Delta^n \otimes U $, where $ U \in \BB $ is $ (-1) $-truncated.
	Since the functor $ \IPSh(-) \colon \Cat(\BB) \to \LPr(\BB) $ that is determined by the universal property of presheaf $\BB$-categories is a (partial) left adjoint (see~\cite[Corollary~7.1.15]{MWColimits}) and therefore preserves colimits, it suffices to see that the $\IPSh(\Delta^n\otimes U)$ is in the essential image.
	Since $ \IPSh(-) $ is also symmetric monoidal by Remark~\ref{rem:UnitIsCocompletion}, we have a canonical equivalence
	\[
	\IPSh(\Delta^n \otimes U) \simeq \IPSh(\Delta^n) \otimes \IPSh(U).
	\]
	Furthermore, we have $ \IPSh(\Delta^n) \simeq \PSh(\Delta^n) \otimes \Univ \simeq (\PSh(\Delta^n) \otimes \Univ) \otimes_\BB \Univ$, and since $ -\otimes_\BB \Univ $ is symmetric monoidal by Proposition~\ref{lem:tensoringIsMonoidal}, it thus suffices to see that $ \IPSh(U) $ is in the essential image.
	By Remark~\ref{rem:EssentialImageofModB}, it follows that we need to check that for any $ A \in \BB $ the canonical map
	\[
	\BB_{/A} \otimes_{\BB} \IPSh(U)(1) \to \IPSh(U)(A)
	\]
	of $ \BB_{/A} $-modules is an equivalence.
	Since $ \BB $ is generated under colimits by $ (-1) $-truncated objects, we may assume that $ A =V $ is also $ (-1) $-truncated.
	Thus, we have to show that the canonical map
	\[
	\BB_{/V} \otimes_{\BB} \BB_{/U} \to \BB_{/U \times V}
	\]
	is an equivalence.
	For this, note that because $ U $ is $ (-1) $-truncated, we have a canonical commutative sqaure
	\[\begin{tikzcd}
		{\Delta^1} & \BB \\
		\BB & {\BB_{/U}}
		\arrow["{U \to 1}", from=1-1, to=1-2]
		\arrow["{-\times U}", from=1-2, to=2-2]
		\arrow["{\id_1}"', from=1-1, to=2-1]
		\arrow["{- \times U}"', from=2-1, to=2-2]
	\end{tikzcd}\]
	By adjunction and the universal property of presheaf $\infty$-categories, this induces a commutative square
	\[\begin{tikzcd}
		{\PSh(\Delta^1)\otimes\BB} & \BB \\
		\BB & {\BB_{/U}}
		\arrow["{(U \to 1)\otimes\BB}", from=1-1, to=1-2]
		\arrow["{-\times U}", from=1-2, to=2-2]
		\arrow["{(\id_1) \otimes\BB}"', from=1-1, to=2-1]
		\arrow["{- \times U}"', from=2-1, to=2-2]
	\end{tikzcd}\]
	in $ \Mod_{\BB}(\LPrS) $.
	We claim that this square is a pushout.
	For this it suffices to see that the underlying square in $ \LPrS $ is a pushout, i.e.\  it is a pullback after passing to right adjoints.
	The right adjoint of $ \id_1 \otimes \BB $ is simply the diagonal map $ \BB \to \BB^{\Delta^1}$, and the right adjoint of $ (U \to 1)\otimes\BB $ sends an object $ A \in \BB $ to the arrow
	\[
	A \to \Hom_{\BB}(U,A).
	\]
	Thus, we may identify the pullback, with the full subcategory of $ \BB $ spanned by those objects for which the canonical map $ 	A \to \Hom_{\BB}(U,A) $ is an equivalence.
	But because $ U $ is $ (-1) $-truncated, this subcategory is canonically equivalent to $ \BB_{/U} $, so that the above square is indeed a pushout.
	Repeating the same argument with $ \Over{\BB}{V} $ in place of $ \BB$ and $ U \times V $ in place of $ U $, we get a similiar pushout in $ \Mod_{\BB_{/U}}(\LPrS) $ with $ \BB_{/U \times V} $ in the lower right corner.
	But applying $ - \otimes_{\BB} \BB_{/V} $ to the above square, we also get a pushout
	\[\begin{tikzcd}[column sep= large]
		{\PSh(\Delta^1)\otimes{\BB_{/V}}} & {\BB_{/V}} \\
		{\BB_{/V}} & {\BB_{/U} \otimes_{\BB} \BB_{/V}}
		\arrow["{(U\times V \to V)\otimes\BB_{/V}}", from=1-1, to=1-2]
		\arrow[from=1-2, to=2-2]
		\arrow["{(\id_1) \otimes\BB_{/V}}"', from=1-1, to=2-1]
		\arrow[from=2-1, to=2-2]
	\end{tikzcd}\]
	and thus an equivalence of  $ \BB_{/V} $-modules $ \BB_{/U\times V} \simeq \BB_{/U} \otimes_\BB \BB_{/V} $.
	Furthermore this equivalence is by construction compatible with the canonical map from $ \BB $.
	Thus it is indeed the map of Remark~\ref{rem:EssentialImageofModB}, and the claim follows.
\end{proof}

Our next goal will be to show that if $ R $ is a commutative $ \BB $-ring spectrum, the presentable $ \BB $-category $ \IMod_{R}^{\BB} $ constructed in \S~\ref{sec:algebrasModules} is in the essential image of $ - \otimes_{\BB} \Univ $.
For this we need the following observation.
Let $ f \colon \CC \to \DD $ be a map in $ \CAlg(\LPr) $ and let $ A \in \CAlg(\CC) $.
The the commutative square
\[\begin{tikzcd}
	{\Mod_{A}(\CC)} & {\Mod_{f(A)}(\DD)} \\
	\CC & \DD
	\arrow[from=1-1, to=1-2]
	\arrow[from=2-1, to=1-1]
	\arrow[from=2-1, to=2-2]
	\arrow[from=2-2, to=1-2]
\end{tikzcd}\]
shows that there is a unique map of $ \DD $-modules $ \psi \colon \Mod_{A}(\CC) \otimes_{\CC} \DD \to \Mod_{f(A)}(\DD) $ making the triangle of $ \CC $-module maps
commute
\[\begin{tikzcd}
	{\Mod_{A}(\CC) \otimes_{\CC} \DD} && {\Mod_{f(A)}(\DD)} \\
	& \DD
	\arrow[from=1-1, to=1-3]
	\arrow[from=2-2, to=1-1]
	\arrow[from=2-2, to=1-3]
\end{tikzcd}\]

\begin{lemma}
	\label{lemma:tensoring_with_a_module_category}
	The map $ \psi \colon \Mod_{A}(\CC) \otimes_{\CC} \DD \to \Mod_{f(A)}(\DD) $ is an equivalence.
\end{lemma}
\begin{proof}
	This can be extracted from the discussion in \cite[\S 4.8.4 and 4.8.5]{Lurie2017}.
	More specifically, we observe that the $ \DD $-module $ \Mod_{A}(\CC) \otimes_{\CC} \DD $ together with the object $ A \otimes 1_{\DD} $ satisfies the conditions (1)-(6) of \cite[4.8.5.8]{Lurie2017}.
	Indeed, we have a $ \CC $-linear adjunction in $ \LPr $
	\[
	-\otimes A \colon \CC \leftrightarrows \Mod_{A} : G
	\]
	so that after applying $ - \otimes_{\CC} \DD $ we get a $ \DD $-linear adjunction and thus all conditions except (5) are obvious.
	To see that (5) also holds, one can argue as in the proof of \cite[Theorem 4.8.4.6]{Lurie2017}.
	By \cite[Proposition 4.8.5.8]{Lurie2017} there is thus an equivalence of $ \DD $-modules sending
	\[
	\varphi \colon \Mod_{f(A)}(\DD) \to \Mod_{A}(\CC) \otimes_{\CC} \DD
	\]
	$ f(A) $ to the unit of $ \Mod_{A}(\CC) \otimes_{\CC} \DD $.
	Furthermore the construction of the equivalence in \cite[Theorem 4.8.5.8]{Lurie2017} shows that $ \varphi $ makes the triangle
	\[\begin{tikzcd}
		{\Mod_{f(A)}(\DD)} && {\Mod_{A}(\CC) \otimes_{\CC} \DD} \\
		& \DD
		\arrow["\varphi", from=1-1, to=1-3]
		\arrow[from=2-2, to=1-1]
		\arrow[from=2-2, to=1-3]
	\end{tikzcd}\]
	of $ \DD $-modules commute and thus the inverse of $ \varphi $ has to agree with our map $ \psi $ from above.
\end{proof}

\begin{proposition}
	\label{prop:Modules_come_from_a_module}
	Let $ R $ be a commutative $ \BB $-ring spectrum.
	Then $ \IMod_{R}^{\BB} $ is in the essential image of $ - \otimes_{\BB} \Univ[\BB] $.
\end{proposition}
\begin{proof}
	Unwinding the definitions, we need to see that for any $ A \in \BB $ the canonical map of $ \BB_{/A} $-modules
	\[
	\Mod_{R}(\Sp(\BB)) \otimes_{\BB} \BB_{/A} \to \Mod_{\pi_A^* R}(\Sp(\BB_{/A}))
	\]
	is an equivalence.
	Since $ \Sp $ is an idempotent algebra in $ \LPr $ we may identify the above map with the canonical map
	\[
	\Mod_{R}(\Sp(\BB)) \otimes_{\Sp(\BB)} \Sp(\BB_{/A}) \to \Mod_{\pi_A^* R}(\Sp(\BB_{/A}))
	\]
	which is an equivalence by Lemma~\ref{lemma:tensoring_with_a_module_category}.
\end{proof}

We conclude this chapter with the following result:

\begin{proposition}
	\label{prop:GammaLinIsModuleMap}
	The functor $ \Gamma^{\lin} \colon \LPr(\BB) \to \Mod_{\BB}(\LPr) $ is $ \Mod_{\BB}(\LPr) $-linear.
	In other words for $ \I{C} \in \LPr(\BB) $ and $ \DD \in \Mod_{\BB}(\LPr) $, the canonical map
	\[
	\gamma_{\DD,\I{C}} \colon \DD \otimes_{\BB} \Gamma^{\lin}(\I{C}) \to \Gamma^{\lin}((\DD \otimes_{\BB} \Univ) \otimes^{\BB} \I{C})
	\]
	given by composing the map $ \Gamma^{\lin}(-)\otimes \Gamma^{\lin}(-) \to \Gamma^{\lin}(- \otimes -) $ with the unit $ \id \to \Gamma^{\lin}(-\otimes_{\BB} \Univ) $, is an equivalence.
\end{proposition}
\begin{proof}
	Note that by the proof of Proposition~\ref{prop:ModulesAreaReflectiveSubcat} the functor $ \Gamma^{\lin} $ commutes with all colimits.
	Since $ - \otimes_{\BB} - $ and $ - \otimes^{\BB} -$ are both bilinear and because $ \Gamma $ is conservative, we may reduce to the case where $ \DD $ is a free $ \BB $-module on some $ \EE \in \LPr $ and only have to show that the canonical map
	\[
	\EE \otimes \Gamma(\I{C}) \to \Gamma(\EE \otimes \Univ) \otimes \Gamma(\I{C}) \to \Gamma((\EE \otimes \Univ) \otimes^{\BB} \I{C})
	\]
	is an equivalence.
	Since any presentable $ \infty $-category is a pushout of presheaf categories, we may further reduce to the case where $ \EE = \PSh(\CC_0) $.
	Recall from the proof of Proposition~\ref{prop:TensorIsLeftAdjointToPresGlobalSections} that we have a canonical equivalence $ \PSh(\CC_0) \otimes \Univ \simeq \IFun(\CC_0,\Univ) $.
	Therefore it follows that we have an equivalence $ ( \PSh(\CC_0) \otimes \Univ) \otimes^{\BB} \I{C} \simeq \IFun(\CC_0^{\op},\I{C}) $ and we claim that the composite
	\[
	\PSh(\CC_0) \otimes \Gamma(\I{C}) \xrightarrow{\gamma_{\PSh(\CC_0),\I{C}}} \Gamma(( \PSh(\CC_0) \otimes \Univ) \otimes^{\BB} \I{C}) \simeq \Fun_{\BB}(\CC_0^{\op},\I{C})
	\]
	is an equivalence.
	For this observe that by construction, the canonical map
	\[
	\Gamma( \PSh(\CC_0) \otimes \Univ) \otimes \Gamma(\I{C}) \to \Gamma(( \PSh(\CC_0) \otimes \Univ) \otimes^{\BB} \I{C})
	\]
	is the unique colimit preserving functor corresponding to the bilinear functor
	\[
	\Gamma( \PSh(\CC_0) \otimes \Univ) \times \Gamma(\I{C}) \xrightarrow{\simeq} \Gamma( (\PSh(\CC_0) \otimes \Univ) \times \I{C} ) \xrightarrow{\Gamma(\tau)} \Gamma(( \PSh(\CC_0) \otimes \Univ) \otimes^{\BB} \I{C}).
	\]
	Here $ \tau $ is the universal bilinear functor of $ \BB $-categories.
	We now consider the commutative square
	\[\begin{tikzcd}[sep=large]
		{\CC_0^{\op} \times \PSh(\CC_0) \times \Gamma(\I{D})} && {\CC_0^{\op} \times \PSh(\CC_0) \times \Gamma(\I{D})} \\
		{\SS \times \Gamma(\I{D})} && {\BB \times \Gamma(\I{D})} \\
		{\Gamma(\I{D})} && {\Gamma(\I{D})}
		\arrow["{\id \times (\const_{\BB})_* \times\id}", from=1-1, to=1-3]
		\arrow["{\ev \times \id}"', from=1-1, to=2-1]
		\arrow["{\ev \times \id}", from=1-3, to=2-3]
		\arrow["{\const_{\BB} \times \id}", from=2-1, to=2-3]
		\arrow["{-\otimes-}"', from=2-1, to=3-1]
		\arrow["{- \otimes-}", from=2-3, to=3-3]
		\arrow["\id"', from=3-1, to=3-3]
	\end{tikzcd}\]
	By Lemma~\ref{lem:explicit_descr_of_univ_bilin}, the composite of the left vertical maps transposes to the universal bilinear functor $ \PSh(\CC_0) \times \Gamma(\I{C}) \to \Fun(\CC_0^{\op},\Gamma(\I{C}))$.
	But combining Lemma~\ref{lem:explicit_descr_of_univ_bilin} with the equivalence $ \Fun_\BB(\CC_0^{\op},\I{D}) \simeq \Fun(\CC_{0}^{\op},\Gamma(\I{D}))  $, the composite of the right vertical maps transposes to $ \Gamma(\tau) $, where $ \tau \colon \IPSh(\CC_0) \times \I{D} \to \IFun(\CC_0^{\op},\I{D})$ is the universal bilinear functor of $ \BB $-categories.
	Since the unit map $ \PSh(\CC_{0}) \to \Gamma(\PSh(\CC_{0})\otimes \Univ) $ can be identified with $ (\const_{\BB})_* \colon \PSh(\CC_0) \to \Fun(\CC_{0}^{\op},\BB) $ (see the proof of Proposition~\ref{prop:TensorIsLeftAdjointToPresGlobalSections}) it follows that the composite
	\[
	\PSh(\CC_0) \times \Gamma(\I{C}) \to\PSh(\CC_0) \otimes \Gamma(\I{C}) \xrightarrow{\gamma_{\PSh(\CC_0),\I{C}}} \Gamma(( \PSh(\CC_0) \otimes \Univ) \otimes^{\BB} \I{C}) \simeq \Fun_{\BB}(\CC_0^{\op},\I{C}) \simeq \Fun(\CC_0^{\op},\Gamma(\I{C}))
	\]
	can be identified with the universal bilinear functor and therefore $ \gamma_{\PSh(C_0),\I{C}} $ is an equivalence, as desired.
\end{proof}

\chapter{$\BB$-topoi}
Our definition of a $\BB$-topos will be that of a presentable $\BB$-category that satisfies the descent property.
Therefore, we will begin by setting up the theory of descent for $\BB$-categories in Section~\ref{chap:descent}.
In Section~\ref{chap:foundations}, we then proceed by developing the main concepts of $\BB$-topos theory.
Using the results from Section~\ref{chap:descent}, we characterize $ \BB $-topoi in terms of an internal version of the Giraud axioms, see Theorem~\ref{thm:characterisationBTopoi}.
Then we prove that any $ \BB $-topos is a left exact accessible localisation of a presheaf $ \BB $-category in Theorem~\ref{thm:presentationTopoi}.
As a conseqeuence, we show the equivalence of $\BB$-topoi and relative topoi over $ \BB $, see Theorem~\ref{thm:BTopoiRelativeTopoi}.
We also apply this equivalence to deduce a formula for the fiber product of $\infty$-topoi, see Proposition~\ref{prop:coproductsBTopoi}.
We conclude the section by proving Diaconescu's Theorem~\ref{thm:internalDiaconescu} for $ \BB $-topoi and using it to study \'etale geometric morphisms in section~\ref{sec:EtaleTopoi}.
Finally, Section~\ref{chap:locallyContractible} is dedicated to the discussion of locally contractible $\BB$-topoi and their relation to smooth geometric morphisms.
Finally, Section~\ref{sec:localicBTopoi} is dedicated to the discussion of localic $ \BB $-topoi.
The main result of this section is that if $ \BB $ is itself a localic $ \infty $-topos, then there is an equivalence between localic $ \BB $-topoi and locales over the locale of subobjects of the terminal object $ \Sub_{\BB} $.
We also study a number of compactness conditions for $ \BB $-locales, which we will be useful for future applications.

\section{Descent}
\label{chap:descent}
Recall that if $\CC$ is an $\infty$-category with finite limits, the codomain fibration $d_0\colon\Fun(\Delta^1,\CC)\to\CC$ is a cartesian fibration and therefore classified by a functor $\Over{\CC}{-}\colon\CC^{\op}\to\CatS$. If $\CC$ furthermore has all colimits, one says that $\CC$ satisfies \emph{descent} if $\Over{\CC}{-}$ preserves limits~\cite[\S~6.1.3]{htt}. The goal of this section is to discuss an analogus concept for $\BB$-categories. We begin in \S~\ref{sec:descent} with defining the descent property. In \S~\ref{sec:cartesianTransformations}, we bring this condition into a more explicit form using the notion of cartesian transformations. As we later want to compare descent with a $\BB$-categorical version of the Giraud axioms, we use the remainder of this section to relate the descent property with the notions of universality (\S~\ref{sec:universalColimits}) and disjointness (\S~\ref{sec:disjointColimits}) of colimits as well as effectivity of groupoid objects (\S~\ref{sec:effectiveGroupoids}). 

\subsection{The definition of descent}
\label{sec:descent}
In order to define the descent property of a $\BB$-category $\I{C}$, we first need to construct the functor $\Over{\I{C}}{-}\colon\I{C}^\op\to\ICat_{\BB}$. As we have the straightening equivalence for cartesian fibrations at our disposal~\cite[Theorem~6.3.1]{MCocartesian}, we may proceed in the same fashion as in~\cite{htt}. We begin with the following lemma:

\begin{lemma}
	\label{lem:codomainFibrationCocartesian}
	For any $\BB$-category $\I{C}$, the codomain fibration $d_0\colon\I{C}^{\Delta^1}\to\I{C}$ is a cocartesian fibration.
\end{lemma}
\begin{proof}
	In light of the identification $\Delta^1\times\Delta^1\simeq\Delta^2\sqcup_{\Delta^1}\Delta^2$,  the restriction functor $\res_{d_0}\colon\I{C}^{\Delta^1\times\Delta^1}\to\Comma{\I{C}^{\Delta^1}}{\I{C}}{\I{C}}$ can be identified with the map
	\begin{equation*}
	\I{C}^{\Delta^2\sqcup_{\Delta^1}\Delta^2}\to\I{C}^{\Delta^2} 
	\end{equation*}
	that is given by precomposition with the inclusion $\Delta^2\into\Delta^2\sqcup_{\Delta^1}\Delta^2$ of the first summand. The latter admits a retraction
	\begin{equation*}
	(\id, s_2)\colon\Delta^2\sqcup_{\Delta^1}\Delta^2\to\Delta^2
	\end{equation*}
	which is right adjoint to the inclusion. Hence precomposition with this retraction defines the desired fully faithful left adjoint $\lift_{d_1}$ of $\res_{d_1}$.
\end{proof}
By the $\BB$-categorical straightening equivalence, the cocartesian fibration $d_0\colon\I{C}^{\Delta^1}\to\I{C}$ gives rise to a functor $\Over{\I{C}}{-}\colon\I{C}\to\ICat_{\BB}$.
Note that the map $(d_1,d_0)\colon\I{C}^{\Delta^1}\to\I{C}\times\I{C}$ can be regarded as a morphism of cocartesian fibrations over $\I{C}$, where we regard the codomain as a cocartesian fibration over $\I{C}$ by virtue of the projection onto the second factor. Therefore, one obtains an induced map $\Over{\I{C}}{-}\to\diag_{\I{C}}$ in $\IFun(\I{C}, \ICat_{\BB})$, where $\diag_{\I{C}}$ is the constant functor with value $\I{C}$. Alternatively, we may regard $\Over{\I{C}}{-}$ as a functor $\I{C}\to\Over{(\ICat_{\BB})}{\I{C}}$. By construction, if $c\colon A\to\I{C}$ is an arbitrary object in context $A\in\BB$, the induced map $\Over{\I{C}}{c}\to \pi_A^\ast\I{C}$ is precisely given by the projection $(\pi_c)_!$ and therefore in particular a right fibration. Thus, the functor $\Over{\I{C}}{-}$ takes values in $\IRFib_{\I{C}}$. In particular, this implies that for any map $f\colon c\to d$ in $\I{C}$ (in arbitrary context), the induced functor $\Over{\I{C}}{c}\to\Over{\I{C}}{d}$ is a right fibration. On account of the orthogonality between right fibrations and final functors, this map is uniquely determined by the image of the final object $\id_c$. As it is moreover evident from the construction of $\Over{\I{C}}{-}$ that the image of $\id_c$ is given by $f$, we thus conclude that $\Over{\I{C}}{-}$ acts on maps by carrying $f$ to the functor $f_!\colon\Over{\I{C}}{c}\to\Over{\I{C}}{d}$ that is obtained as the image of $f$ under the Yoneda embedding $\I{C}\into\IRFib_{\I{C}}$.
\begin{proposition}
	\label{prop:codomainFibrationCartesian}
	Let $\I{C}$ be a $\BB$-category. Then the following are equivalent:
	\begin{enumerate}
	\item The codomain fibration $d_0\colon\I{C}^{\Delta^1}\to\I{C}$ is a cartesian fibration;
	\item for every map $f\colon c\to d$ in $\I{C}$, the functor $f_!\colon\Over{\I{C}}{c}\to\Over{\I{C}}{d}$ admits a right adjoint $f^\ast$;
	\item $\I{C}$ admits pullbacks.
	\end{enumerate}
\end{proposition}
\begin{proof}
	By combining Lemma~\ref{lem:codomainFibrationCocartesian} and~\cite[Lemma~7.2.2]{MCocartesian} with~\cite[Corollary~6.5.5]{MCocartesian}, the functor $d_0$ is a cartesian fibration if and only if for every map $f\colon c\to d$ in $\I{C}$ in context $A\in\BB$ the functor $f_!\colon\Over{\I{C}}{c}\to\Over{\I{C}}{d}$ admits a right adjoint $f^\ast$. Hence~(1) and~(2) are equivalent. The fact that~(1) and~(3) are equivalent follows from the section-wise characterisation of cartesian fibrations~\cite[Proposition~3.1.7]{MCocartesian}, the section-wise characterisation of the existence of pullbacks~(Example~\ref{ex:externalLimitsColimits}) and~\cite[Lemma~6.1.1.1]{htt}.
\end{proof}

\begin{remark}
	\label{rem:cartesianMorphismsCodomainFibration}
	In the situation of Proposition~\ref{prop:codomainFibrationCartesian}, note that if $d_0\colon\I{C}^{\Delta^1}\to\I{C}$ is a cartesian fibration, then a morphism in $\I{C}^{\Delta^1}$ in context $A\in\BB$ is cartesian if and only if it defines a cartesian morphism in $\I{C}(A)^{\Delta^1}$, see~\cite[Remark~3.2.6]{MCocartesian}. As this is the case precisely if the associated commutative square in $\I{C}(A)$ is a pullback square, we conclude that a map $\Delta^1\otimes A\to \I{C}^{\Delta^1}$ is a cartesian morphism if and only if it corresponds to a pullback square in $\I{C}$.
\end{remark}

If $\I{C}$ is a $\BB$-category with pullbacks, straightening the cartesian fibration $d_0\colon\I{C}^{\Delta^1}\to\I{C}$ yields a functor $\Over{\I{C}}{-}\colon\I{C}^{\op}\to\ICat_{\BB}$. By the discussion in~\cite[\S~7.2]{MCocartesian}, this functor is equivalently obtained by observing that the straightening of the \emph{cocartesian} fibration $\I{C}^{\Delta^1}\to\I{C}$ takes values in $\ICat_{\BB}^\Lad$ and by applying the equivalence $\ICat_{\BB}^\Lad\simeq(\ICat_{\BB}^\Rad)^\op$ from~\cite[Proposition~7.2.1]{MCocartesian}. We may now define:
\begin{definition}
	\label{def:descent}
	Let $\I{U}$ be an internal class of $\BB$-categories and let $\I{C}$ be a $\I{U}$-cocomplete $\BB$-category with finite limits. We say that $\I{C}$ \emph{satisfies $\I{U}$-descent} if the functor $\Over{\I{C}}{-}\colon\I{C}^{\op}\to\ICat_{\BBB}$ is $\op(\I{U})$-continuous. If $\I{X}$ is a cocomplete large $\BB$-category, we simply say that $\I{C}$ satisfies descent if $\I{C}$ satisfies $\ICat_{\BB}$-descent.
\end{definition}

\begin{remark}
	\label{rem:BCUDescent}
	The property of a $\I{U}$-cocomplete $\BB$-category $\I{C}$ with pullbacks to satisfy $\I{U}$-descent is local in $\BB$: if $\bigsqcup_i A_i\onto 1$ is a cover in $\BB$, then $\I{C}$ satisfies $\I{U}$-descent if and only if $\pi_{A_i}^\ast\I{C}$ satisfies $\pi_{A_i}^\ast\I{U}$-descent for all $i$. This follows immediately from the locality of $\I{U}$-continuity~\cite[Remark~5.2.3]{MWColimits} and from~\cite[Remark~A.3]{MWColimits}.
\end{remark}

\begin{example}
	\label{ex:LConstKDescent}
	Let $\KK$ be a class of $\infty$-categories and let $\I{C}$ be an $\ILConst_{\KK}$-cocomplete $\BB$-category with pullbacks (where $\ILConst_{\KK}$ is the essential image of the functor $\const_{\BB}\colon\KK\to \ICat_{\BB}$). Then $\I{C}$ satisfies $\ILConst_{\KK}$-descent if and only if for all $A\in\BB$ the $\infty$-category $\I{C}(A)$ satisfies $\KK$-descent. In fact, by~\cite[Corollary~6.4.10]{MCocartesian} the composition $\Gamma_{\Over{\BB}{A}}\circ\Over{\I{C}}{-}(A)$ recovers the functor $\Over{\I{C}(A)}{-}\colon \I{C}(A)^\op\to\CatSS$. Consequently, if $s\colon B\to A$ is an arbitrary map in $\BB$, postcomposing $\Over{\I{C}}{-}(A)$ with the evaluation functor $\ev_B\colon\Cat(\Over{\BB}{A})\to\CatSS$ recovers the composition $\Over{\I{C}}{-}(B)\circ s^\ast$. Using that $\I{C}$ is $\ILConst_{\KK}$-cocomplete, the functor $s^\ast\colon \I{C}(A)^\op\to\I{C}(B)^\op$ is $\op(\KK)$-continuous, and since limits in $\Cat(\Over{\BB}{A})$ are detected section-wise, the claim follows.
\end{example}

\subsection{Cartesian transformations}
\label{sec:cartesianTransformations}
The main goal of this section is to obtain a more explicit description of the descent property which will rely on the notion of \emph{cartesian} morphisms of functors:
\begin{definition}
	\label{def:cartesianNaturalTransformation}
	Let $\I{I}$ and $\I{C}$ be $\BB$-categories such that $\I{C}$ admits pullbacks. We say that a map $\phi\colon d\to d^\prime$ in $\IFun(\I{I},\I{C})$ in context $1\in\BB$ is \emph{cartesian} if for every map $i\to i^\prime$ in $\I{I}$ in context $A\in \BB$ the induced commutative square
	\begin{equation*}
		\begin{tikzcd}
			d(i)\arrow[r]\arrow[d] &d^\prime(i)\arrow[d]\\
			d(i^\prime)\arrow[r] & d^\prime(i^\prime)
		\end{tikzcd}
	\end{equation*}
	is a pullback in $\I{C}(A)$. A map $d\to d^\prime$ in context $A\in\BB$ is called cartesian if it is cartesian when viewed as a map in $\IFun[\Over{\BB}{A}](\pi_A^\ast\I{I},\pi_A^\ast\I{C})$ in context $1\in\Over{\BB}{A}$. We denote by $\Over{\IFun(\I{I},\I{C})}{d}^\cart$ the full subcategory of $\Over{\IFun(\I{I},\I{C})}{d}^\cart$ that is spanned by the cartesian maps in arbitrary context $A\in\BB$.
\end{definition} 

\begin{remark}
	\label{rem:BCCartesianMaps}
	In the situation of Definition~\ref{def:cartesianNaturalTransformation}, the property of a map $\phi d\to d^\prime$ in context $A\in\BB$ being cartesian is local in $\BB$: if $(s_i)\colon\bigsqcup_i A_i\onto A$ is a cover in $\BB$, then $\phi$ is cartesian if and only if each $s_i^\ast(\phi)$ is. In fact, by unwinding the definition, this follows from the fact that the property of a commutative square being a pullback is local in that sense. As a consequence, every object of $\Over{\IFun(\I{I},\I{C})}{d}^\cart$ in context $A$ encodes a cartesian map $d\to d^\prime$, and there is a canonical equivalence $\pi_A^\ast \Over{\IFun(\I{I},\I{C})}{d}^\cart\simeq \Over{\IFun[\Over{\BB}{A}](\pi_A^\ast\I{I},\pi_A^\ast\I{C})}{\pi_A^\ast d}^\cart$ of $\Over{\BB}{A}$-categories.
\end{remark}

\begin{lemma}
	\label{lem:cartesianIsCocartesian}
	Let $\I{I}$ and $\I{C}$ be $\BB$-categories, and suppose that $\I{C}$ admits pullbacks. Then a map $d\to d^\prime$ in $\IFun(\I{I},\I{C})$ (in arbitrary context) is cartesian if and only if the associated object in $\IFun(\I{I},\I{C}^{\Delta^1})$ is contained in $\IFun(\I{I},(\I{C}^{\Delta^1})_\sharp)$, where $(\I{C}^{\Delta^1})_\sharp\into\I{C}^{\Delta^1}$ is the subcategory that is spanned by the cartesian morphisms over $d_0\colon\I{C}^{\Delta^1}\to\I{C}$.
\end{lemma}
\begin{proof}
	This follows immediately from the description of the cartesian morphisms in $\I{C}^{\Delta^1}$ as pullback squares in $\I{C}$, see Remark~\ref{rem:cartesianMorphismsCodomainFibration}.
\end{proof}
The main goal of this section is to prove the following description of the descent property:
\begin{proposition}
	\label{prop:LimitsOfCategoriesExplicit}
	Let $\I{C}$ be a cocomplete $\BB$-category with pullbacks and let $d\colon\I{I}\to\I{C}$ be a diagram that admits a colimit in $\I{C}$. Let $\overline{d}\colon \I{I}^\triangleright\to\I{C}$ be the corresponding colimit cocone. Then the functor $\Over{\I{C}}{-}\colon\I{C}^{\op}\to\ICat_{\BB}$ carries $\overline{d}$ to a limit cone in $\ICat_{\BB}$ if and only if the restriction map $\Over{\IFun(\I{I}^\triangleright,\I{C})}{\overline{d}}\to\Over{\IFun(\I{I},\I{C})}{d}$ restricts to an equivalence
	\begin{equation*}
		\Over{\IFun(\I{I}^\triangleright,\I{C})}{\overline{d}}^{\cart}\simeq\Over{\IFun(\I{I},\I{C})}{d}^{\cart}
	\end{equation*}
	of $\BB$-categories.
\end{proposition}
The main idea for the proof of Proposition~\ref{prop:LimitsOfCategoriesExplicit} is to identify the left-hand side of the equivalence with $\Over{\I{C}}{\colim d}$ and the right-hand side with $\lim \Over{\I{C}}{d(-)}$. In order to do so, we will need the formula for limits in $\ICat_{\BB}$ that we derived in~\cite[Proposition~7.1.2]{MCocartesian}. For the convenience of the reader, we will briefly recall the main setup from~\cite{MCocartesian}.
The $\infty$-topos of \emph{marked simplicial objects} in $\BB$ is defined as $\mSimp{\BB}=\Fun(\Delta^{\op}_+,\BB)$, where $\Delta_+$ denotes the marked simplex $1$-category. Precomposition with the inclusion $\Delta\into\Delta_+$ induces a forgetful functor $(-)\vert_{\Delta}\colon\mSimp{\BB}\to\Simp\BB$ which admits a left adjoint $(-)^\flat$ and a right adjoint $(-)^\sharp$. Every cartesian fibration $p\colon\I{P}\to\I{C}$ can be equivalently encoded by a \emph{marked cartesian fibration} $p^\natural\colon\I{P}^\natural\to\I{C}^\sharp$, where $\I{P}^\natural$ is the marked simplicial object that is obtained from $\I{P}$ by marking the cartesian arrows and where a marked cartesian fibration is by definition a map that is internally right orthogonal to the collection of \emph{marked right anodyne maps} (see \cite[Definition 4.2.1]{MCocartesian}). Now if $d\colon \I{I}^\op\to\ICat_{\BB}$ is a functor and if $p\colon\I{P}\to\I{I}$ is the associated cartesian fibration, one obtains a canonical equivalence
\begin{equation*}
	\lim d\simeq( \Over{\Hom_{\mSimp\BB}(\I{I}^\sharp,\I{P}^\natural)}{\I{I}^\sharp})\vert_{\Delta},
\end{equation*}
where the right-hand side is defined via the pullback diagram
\begin{equation*}
	\begin{tikzcd}
		(	\Over{\Hom_{\mSimp\BB}(\I{I}^\sharp,\I{P}^\natural)}{\I{I}^\sharp})\vert_{\Delta}\arrow[r]\arrow[d] & \Hom_{\mSimp\BB}(\I{I}^\sharp,\I{P}^\natural)\vert_{\Delta}\arrow[d]\\
		1\arrow[r, "\id_{\I{I}^\sharp}"] & \Hom_{\mSimp\BB}(\I{I}^\sharp,\I{I}^\sharp)\vert_{\Delta}
	\end{tikzcd}
\end{equation*}
(in which $\Hom_{\mSimp\BB}(-,-)$ denotes the internal hom in $\mSimp{\BB}$).

\begin{lemma}
	\label{lem:precompositionFlatSharpFF}
	Let $K$ be a simplicial object in $\BB$ and let $p\colon\I{P}\to\I{C}$ be a cartesian fibration. Then the canonical map $K^\flat\to K^\sharp$ of marked simplicial objects in $\BB$ induces a fully faithful functor $\Hom_{\mSimp\BB}(K^\sharp,\I{P}^\natural)\vert_{\Delta}\into\Hom_{\mSimp\BB}(K^\flat,\I{P}^\natural)\vert_{\Delta}$ of $\BB$-categories.
\end{lemma}
\begin{proof}
	Let $M$ be the marked simplicial object in $\BB$ that fits into the pushout square
	\begin{equation*}
		\begin{tikzcd}
			(\Delta^0\sqcup\Delta^0)^\flat \otimes K^\flat\arrow[r]\arrow[d] & (\Delta^0\sqcup\Delta^0)^\flat \otimes K^\sharp\arrow[d]\\
			(\Delta^1)^\flat\otimes K^\flat\arrow[r] & M.
		\end{tikzcd}
	\end{equation*}
	Unwinding the definitions, we need to show that $\I{P}^\natural$ is internally local with respect to the induced map $\phi\colon M\to (\Delta^1)^\flat\otimes K^\sharp$. Since $\I{C}^\sharp$ is easily seen to be internally local with respect to $\phi$, this follows once we show that $p^\natural$ is internally right orthogonal to this map. We therefore need to verify that $\phi$ is marked right anodyne. Writing $K$ as a colimit of objects of the form $\Delta^n\otimes A$, we may assume that $K=\Delta^n\otimes A$. Moreover, since marked right anodyne morphisms are closed under products, we can assume that $A\simeq 1$. Using that the two maps $(I^n)^\flat\into (\Delta^n)^\flat$ and $(I^n)^\sharp\into (\Delta^n)^\sharp$ that are induced by the spine inclusions are marked right anodyne, we may further reduce this to $K=\Delta^1$. 
	In this case, one can apply~\cite[Lemma~4.2.3]{MCocartesian} to deduce that $\phi$ is an equivalence. Hence the claim follows.
\end{proof}

\begin{lemma}
	\label{lem:MateIsPrecomposition}
	Let $s\colon B\to A$ be a map in $\BB$, and let $P\to A$ be an arbitrary map. Let $\eta_s\colon \id_{\Over{\BB}{A}}\to s_\ast s^\ast$ be the adjunction unit. Then the value of the natural transformation $(\pi_A)_\ast\xrightarrow{(\pi_A)_\ast \eta_s}(\pi_A)_\ast s_\ast s^\ast\simeq (\pi_B)_\ast s^\ast$ at an object $p\colon P\to A$ in $\Over{\BB}{A}$ can be identified with the map
	\begin{equation*}
		\Over{\Hom_{\BB}(A,P)}{A}\to\Over{\Hom_{\BB}(B,s^\ast P)}{B}
	\end{equation*}
	that is induced by precomposition with $s$. Here $\Hom_{\BB}(-,-)$ denotes the internal hom in $\BB$, $\Over{\Hom_{\BB}(A,P)}{A}$ is the fibre of the map $\Hom_{\BB}(A,P)\to\Hom_{\BB}(A,A)$ over $\id_A$, and $\Over{\Hom_{\BB}(B,s^\ast P)}{B}$ is the fibre of the map $\Hom_{\BB}(B,s^\ast P)\to \Hom_{\BB}(B,B)$ over $\id_B$.
\end{lemma}
\begin{proof}
	Since the morphism $\id\times s\colon -\times B\to -\times A$ can be identified with the composition
	\begin{equation*}
		(\pi_B)_!\pi_B^\ast\xrightarrow{\simeq }(\pi_A)_! s_! s^\ast \pi_A^\ast \xrightarrow{(\pi_A)_!\epsilon_s\pi_A^\ast} (\pi_A)_!\pi_A^\ast
	\end{equation*}
	(in which $\epsilon_s$ is the counit of the adjunction $s_!\dashv s^\ast$), it follows by adjunction that the map $(\pi_A)_\ast\eta_s \pi_A^\ast$ can be identified with $s^\ast\colon \Hom_{\BB}(A,-)\to\Hom_{\BB}(B,-)$. Now if $p\colon P\to A$ is any map, the unique morphism $p\to \id_{\Over{\BB}{A}}$ in $\Over{\BB}{A}$ is the pullback of $\pi_A^\ast (\pi_A)_! p\to \pi_A^\ast(\pi_A)_! 1_{\Over{\BB}{A}}$.  Together with naturality of $\eta_s$, this implies that the map $(\pi_A)_\ast\eta_s(p)$ fits into a commutative diagram
	\begin{equation*}
		\begin{tikzcd}[column sep=small, row sep=small]
			& s_\ast s^\ast (\pi_A)_\ast(p) \arrow[rr]\arrow[dd] && \Hom_{\BB}(B,P)\arrow[dd, "p_\ast"]\\
			(\pi_A)_\ast(p)\arrow[rr, crossing over]\arrow[dd]\arrow[ur, "(\pi_A)_\ast \eta_s(p)"] && \Hom_{\BB}(A,P) \arrow[ur, "s^\ast"]& \\
			& 1\arrow[rr, "s"] && \Hom_{\BB}(B,A)\\
			1\arrow[rr, "\id_A"]\arrow[ur] && \Hom_{\BB}(A,A)\arrow[from=uu, crossing over, "p_\ast", near start]\arrow[ur, "s^\ast"] &
		\end{tikzcd}
	\end{equation*}
	in which the front and the back square are pullbacks. As the fibre of 
	\begin{equation*}
		p_\ast\colon \Hom_{\BB}(B,P)\to\Hom_{\BB}(B,A)
	\end{equation*} 
	over $s$ can be identified with $\Over{\Hom_{\BB}(B,s^\ast P)}{B}$, the claim follows.
\end{proof}

\begin{proof}[{Proof of Proposition~\ref{prop:LimitsOfCategoriesExplicit}}]
	 Let $\iota\colon \I{I}\into\I{I}^\triangleleft$ be the inclusion. Since $\ICat_{\BB}$ is complete, the theory of Kan extensions gives rise to an adjunction
	\begin{equation*}
		(\iota^\ast\dashv\iota_\ast)\colon\IFun(\I{I}^\triangleleft,\ICat_{\BB})\leftrightarrows\IFun(\I{I},\ICat_{\BB}),
	\end{equation*}
	see~\cite[\S~6.3]{MWColimits}.
	Given any diagram $\overline{h}\colon \I{I}^\triangleleft\to \ICat_{\BB}$, we will let $h=\iota^\ast \overline{h}$, and we denote by $\eta_{\overline{h}}\colon \overline{h}\to \iota_\ast h$ the adjunction unit.
	Now let us set $\overline h=\Over{\I{C}}{\overline{d}(-)}$, so that we get $h=\Over{\I{C}}{{d}(-)}$. Furthermore, let $p\colon \I{P}\to \I{I}^\triangleright$ be the pullback of $d_0\colon\I{C}^{\Delta^1}\to\I{C}$ along $\overline{d}$, and let $q\colon \I{Q}\to\I{I}$ be the pullback of $d_0$ along $d$. 
	According to~\cite[Proposition~7.1.2]{MCocartesian} and Lemma~\ref{lem:MateIsPrecomposition} (applied to the $\infty$-topos $\mSimp{\BB}$ and the map $\iota^\sharp$), the canonical map $\lim \eta_{\overline{h}}\colon \lim \overline{h}\to \lim\iota_\ast h$ can be identified with the functor 
	\begin{equation*}
		\label{eq:restrictionCartesianSections}
		\tag{$\ast$}
		\Over{\Hom_{\mSimp\BB}((\I{I}^\triangleright)^\sharp,\I{P}^\natural)}{(\I{I}^\triangleright)^\sharp}\vert_{\Delta}\to \Over{\Hom_{\mSimp\BB}(\I{I}^\sharp,\I{Q}^\natural)}{\I{I}^\sharp}\vert_{\Delta}
	\end{equation*}
	that is induced by precomposition with the inclusion $\iota \colon\I{I}\into\I{I}^\triangleright$. As $\Over{\I{C}}{-}$ preserving the limit of $d$ is therefore equivalent to~(\ref{eq:restrictionCartesianSections}) being an equivalence, we only need to identify this map with the functor  $		\Over{\IFun(\I{I}^\triangleright,\I{C})}{\overline{d}}^{\cart}\to\Over{\IFun(\I{I},\I{C})}{d}^{\cart}$.
	
	To see this, first note that there is an equivalence $\Hom_{\mSimp\BB}((-)^\flat,-)\vert_{\Delta}\simeq\IFun(-,(-)\vert_{\Delta})$. Therefore, we obtain a commutative diagram
	\begin{equation*}
		\begin{tikzcd}
			\Over{\Hom_{\mSimp\BB}((\I{I}^\triangleright)^\sharp,\I{P}^\natural)}{(\I{I}^\triangleright)^\sharp}\vert_{\Delta}\arrow[r]\arrow[d, hookrightarrow] & \Hom_{\mSimp\BB}((\I{I}^\triangleright)^\sharp,\I{P}^\natural)\vert_{\Delta}\arrow[r]\arrow[d, hookrightarrow] & \Hom_{\mSimp\BB}((\I{I}^\triangleright)^\sharp,(\I{C}^{\Delta^1})^\natural)\vert_{\Delta}\arrow[d, hookrightarrow]\\
			\Over{\IFun(\I{I}^\triangleright,\I{C})}{\overline{d}}\arrow[r]\arrow[d] & \IFun(\I{I}^\triangleright,\I{P})\arrow[r]\arrow[d] & \IFun(\I{I}^\triangleright,\I{C}^{\Delta^1})\arrow[d]\\
			1\arrow[r, "\id_{\I{I}^\triangleright}"] & \IFun(\I{I}^\triangleright,\I{I}^\triangleright)\arrow[r, "\overline{d}_\ast"] & \IFun(\I{I}^\triangleright,\I{C})
		\end{tikzcd}
	\end{equation*}
	in which the upper three vertical maps are induced by precomposition with the canonical map $(\I{I}^\triangleright)^\flat\to(\I{I}^\triangleright)^\sharp$. By Lemma~\ref{lem:precompositionFlatSharpFF}, they are fully faithful. Furthermore, all but the upper right square are pullbacks. But since the map $(I^\triangleright)^\flat\to (I^\triangleright)^\sharp$ is internally right orthogonal to every map that is contained in the image of $(-)^\sharp\colon \Simp\BB\into\mSimp\BB$, it must also be internally right orthogonal to $\I{P}^\natural\to (\I{C}^{\Delta^1})^\natural$ as the latter is the pullback of $\overline{d}^\sharp$. Therefore, the upper right square must also be a pullback. Note, furthermore, that since the map $(-)^\flat\to(-)^\sharp$ is an equivalence when restricted along the inclusion $\BB\into\Simp\BB$, the upper right inclusion in the above diagram identifies the domain with the essential image of the map $\Hom_{\mSimp\BB}(\I{I}^\triangleright,(\I{C}^{\Delta^1})_\sharp)\into\Hom_{\mSimp\BB}(\I{I}^\triangleright,\I{C}^{\Delta^1})$. Therefore, Lemma~\ref{lem:cartesianIsCocartesian} implies that there is an equivalence $\Over{\Hom_{\mSimp\BB}((\I{I}^\triangleright)^\sharp,\I{P}^\natural)}{(\I{I}^\triangleright)^\sharp}\vert_{\Delta}\simeq \Over{\IFun(\I{I}^\triangleright,\I{C})}{\overline{d}}^{\cart}$. By an analogous argument, one obtains an equivalence $\Over{\Hom_{\mSimp\BB}(\I{I}^\sharp,\I{Q}^\natural)}{\I{I}^\sharp}\vert_{\Delta}\simeq \Over{\IFun(\I{I},\I{C})}{{d}}^{\cart}$, hence the claim follows.
\end{proof}

\begin{remark}
	\label{rem:cartesianTransformationsCoconesExplicitly}
	In the situation of Proposition~\ref{prop:LimitsOfCategoriesExplicit}, let $\infty\colon 1\to\I{I}^\triangleleft$ be the cone point, and let
	\begin{equation*}
		(\infty^\ast \dashv \infty_\ast)\colon \IFun(\I{I}^\triangleleft,\ICat_{\BB})\leftrightarrows\ICat
	\end{equation*}
	be the induced adjunction. Let $\eta\colon \id_{\IFun(\I{I}^\triangleleft, \ICat_{\BB})}\to \infty_\ast\infty^\ast$ be the adjunction unit. By the same argument as in the proof of Proposition~\ref{prop:LimitsOfCategoriesExplicit}, evaluating $\lim_{\I{I}^\triangleleft}\eta$ at the cone $\Over{\I{C}}{\overline{d}}$ recovers the restriction map
	\begin{equation*}
		\Over{\IFun(\I{I}^\triangleright,\I{C})}{\overline{d}}^{\cart}\to \Over{\I{C}}{c}.
	\end{equation*}
	Since $\lim_{\I{I}^\triangleleft}$ can be identified with $\infty^\ast$ (owing to $\infty\colon 1\to\I{I}^\triangleleft$ being initial), the triangle identities imply that this map must be an equivalence.
	Furthermore, note that by Corollary~\ref{cor:adjunctionsSliceCategoryPullback} the restriction functor $\Over{\IFun(\I{I}^\triangleright,\I{C})}{\overline{d}}\to\Over{\I{C}}{c}$ admits a right adjoint that is is given by the composition
		\begin{equation*}
		\Over{\I{C}}{c}\xrightarrow{\Over{\diag}{c}}\Over{\IFun(\I{I}^\triangleright,\I{C})}{\diag(c)}\xrightarrow{\eta^\ast} \Over{\IFun(\I{I}^\triangleright,\I{C})}{\overline{d}}
		\end{equation*}
	(where $\eta\colon \overline{d}\to\diag(c)$ denotes the adjunction unit). Now if $c^\prime\to c$ is a map in $\I{C}$, Corollary~\ref{cor:characterisationCategoriesWithPullbacks} implies that the counit $\eta_!\eta^\ast\diag(c^\prime)\to \diag(c^\prime) $ of the adjunction $\eta_!\dashv\eta^\ast$ is given by the pullback of $\eta$ along $\diag(c^\prime)\to\diag(c)$. Since evaluation at $\infty$ preserves pullbacks and since $\Over{\diag}{c}$ is fully faithful, we conclude that evaluating the counit of the adjunction $\Over{\I{C}}{c}\leftrightarrows\Over{\IFun(\I{I}^\triangleright,\I{C})}{\overline{d}}$ at $c^\prime\to c$ must result in an equivalence. Upon replacing $\BB$ with $\Over{\BB}{A}$ and repeating the same argument, we conclude that the entire counit must be an equivalence, so that the functor $\Over{\I{C}}{c}\to\Over{\IFun(\I{I}^\triangleright,\I{C})}{\overline{d}}$ is fully faithful. Now combining the evident observation that this map takes values in $\Over{\IFun(\I{I}^\triangleright,\I{C})}{\overline{d}}^{\cart}$ with the fact that the restriction functor $\Over{\IFun(\I{I}^\triangleright,\I{C})}{\overline{d}}^{\cart}\to \Over{\I{C}}{c}$ is an equivalence, one concludes that the inclusion $\Over{\I{C}}{c}\into\Over{\IFun(\I{I}^\triangleright,\I{C})}{\overline{d}}$ identifies $\Over{\I{C}}{c}$ with $\Over{\IFun(\I{I}^\triangleright,\I{C})}{\overline{d}}^{\cart}$. In particular, the inclusion $\Over{\IFun(\I{I}^\triangleright,\I{C})}{\overline{d}}^{\cart}\into \Over{\IFun(\I{I}^\triangleright,\I{C})}{\overline{d}}$ admits a left adjoint.
\end{remark}
\begin{remark}
	\label{rem:criterionCartesianMorphismCocones}
	In the situation of Remark~\ref{rem:cartesianTransformationsCoconesExplicitly}, let $\overline{d}^\prime\to \overline{d}$ be a map in $\IFun(\I{I}^\triangleright,\I{C})$, and let us set $c^\prime=\infty^\ast(\overline{d}^\prime)$. Then the unit of the adjunction $\Over{\I{C}}{c}\leftrightarrows\Over{\IFun(\I{I}^\triangleright,\I{C})}{\overline{d}}$ evaluates at $\overline{d}^\prime$ to the natural map $\overline{d}^\prime\to\eta^\ast \diag(c^\prime)$. Therefore, the map $\overline{d}^\prime\to\overline{d}$ is cartesian precisely if the square
	\begin{equation*}
		\begin{tikzcd}
			\overline{d}^\prime\arrow[d]\arrow[r] & \diag(c^\prime)\arrow[d]\\
			\overline{d}\arrow[r] & \diag(c)
		\end{tikzcd}
	\end{equation*}
	is a pullback. As the functor $(\iota^\ast,\infty^\ast)\colon\IFun(\I{I}^\triangleright,\I{C})\to\IFun(\I{I},\I{C})\times\I{C}$ is conservative (on account of the map $(\iota,\infty)\colon \I{I}\sqcup 1\onto \I{I}^\triangleright$ being essentially surjective) and as the image of the above square along $\infty^\ast$ is always a pullback, the map $\overline{d}^\prime\to\overline{d}$ is cartesian precisely if the square
	\begin{equation*}
		\begin{tikzcd}
			{d}^\prime\arrow[d]\arrow[r] & \diag(c^\prime)\arrow[d]\\
			{d}\arrow[r] & \diag(c)
		\end{tikzcd}
	\end{equation*}
	in $\IFun(\I{I},\I{C})$ is a pullback.
\end{remark}

Suppose that $\I{U}$ is an internal class of $\BB$-categories and let $\I{C}$ be a $\I{U}$-cocomplete $\BB$-category with pullbacks. Given any $\I{I}\in\I{U}(1)$ and any diagram $d\colon \I{I}\to\I{C}$ with colimit cocone $\overline{d}\colon\I{I}^\triangleright\to \I{C}$, Proposition~\ref{prop:adjunctionsSliceCategories} implies that the functor
\begin{equation*}
	\Over{\iota^\ast}{\overline{d}}\colon\Over{\IFun(\I{I}^\triangleright,\I{C})}{\overline{d}}\to\Over{\IFun(\I{I},\I{C})}{d}
\end{equation*}
has a left adjoint that is given by $\Over{(\iota_!)}{d}$. Combining this observation with Remark~\ref{rem:cartesianTransformationsCoconesExplicitly}, we thus end up with a left adjoint $\Over{\IFun(\I{I},\I{C})}{d}^{\cart}\to\Over{\IFun(\I{I}^\triangleright,\I{C})}{\overline{d}}^{\cart}$ to the restriction functor $\Over{\IFun(\I{I}^\triangleright,\I{C})}{\overline{d}}^{\cart}\to\Over{\IFun(\I{I},\I{C})}{d}^{\cart}$ that we will refer to as the \emph{glueing functor}. In light of Proposition~\ref{prop:LimitsOfCategoriesExplicit}, the functor $\Over{\I{C}}{-}$ preserves the limit of $d$ precisely if both unit and counit of this adjunction are equivalences, i.e.\ if both the restriction functor and the glueing functor are fully faithful. We may therefore split up the notion of $\I{U}$-descent into two separate conditions:
\begin{definition}
	Let $\I{U}$ be an internal class and let $\I{C}$ be a $\I{U}$-cocomplete $\BB$-category with pullbacks. We say that $\I{C}$ has \emph{faithful $\I{U}$-descent} if for every $A\in\BB$, every $\I{I}\in \I{U}(A)$ and every diagram $d\colon\I{I}\to\pi_A^\ast\I{C}$ with colimit cocone $\overline{d}\colon\I{I}^\triangleright\to \pi_A^\ast\I{C}$, the restriction functor $\Over{\IFun[\Over{\BB}{A}](\I{I}^\triangleright,\pi_A^\ast\I{C})}{\overline{d}}^{\cart}\to\Over{\IFun[\Over{\BB}{A}](\I{I},\pi_A^\ast\I{C})}{d}^{\cart}$ is fully faithful. We say that $\I{C}$ has \emph{effective $\I{U}$-descent} if for every $A\in\BB$, every $\I{I}\in \I{U}(A)$ and every diagram $d\colon\I{I}\to\pi_A^\ast\I{C}$ with colimit cocone $\overline{d}\colon\I{I}^\triangleright\to \pi_A^\ast\I{C}$, the glueing functor $\Over{\IFun[\Over{\BB}{A}](\I{I},\pi_A^\ast\I{C})}{d}^{\cart}\to\Over{\IFun[\Over{\BB}{A}](\I{I}^\triangleright,\pi_A^\ast\I{C})}{\overline{d}}^{\cart}$ is fully faithful. If $\I{C}$ is a cocomplete large $\BB$-category, we simply say that $\I{C}$ has faithful/effective descent if it has faithful/effective $\ICat_{\BB}$-descent.
\end{definition}
\begin{remark}
	\label{rem:BCFaithfulEffectiveDescent}
	As a consequence of Remark~\ref{rem:BCCartesianMaps}, the property of $\I{C}$ having faithful/effective $\I{U}$-descent is local in $\BB$, in the sense that whenever $\bigsqcup_i A_i\onto 1$ is a cover in $\BB$, the $\BB$-category $\I{C}$ satisfies faithful/effective $\I{U}$-descent precisely if for each $i$ the $\Over{\BB}{{A_i}}$-category $\pi_{A_i}^\ast\I{C}$ has faithful/effective $\pi_{A_i}^\ast\I{U}$-descent.
\end{remark}

By unwinding how the unit and counit of the adjunction  $\Over{\IFun(\I{I}^\triangleright,\I{C})}{\overline{d}}^{\cart}\leftrightarrows\Over{\IFun(\I{I},\I{C})}{d}^{\cart}$ are computed (cf.\ Remark~\ref{rem:unitCounitSliceAdjunctions}) and by using Remark~\ref{rem:criterionCartesianMorphismCocones}, we may characterise the notion of faithful and effective $\I{U}$-descent as follows:
\begin{proposition}
	\label{prop:characterisationFaithfulDescent}
	Let $\I{U}$ be an internal class and let $\I{C}$ be a $\I{U}$-cocomplete $\BB$-category with pullbacks. Then the following are equivalent:
	\begin{enumerate}
		\item $\I{C}$ has faithful $\I{U}$-descent;
		\item for every $A\in\BB$, every $\I{I}\in \I{U}(A)$ and every cartesian map $\overline{d}^\prime\to \overline{d}$ in $\IFun[\Over{\BB}{A}](\I{I}^\triangleright,\pi_A^\ast\I{C})$ in which $\overline{d}$ is a colimit cocone, $\overline{d}^\prime$ is a colimit cocone as well;
		\item for every $A\in \BB$, every $\I{I}\in\I{U}(A)$ and every pullback diagram
		\begin{equation*}
			\begin{tikzcd}
				{d}^\prime\arrow[d]\arrow[r] & \diag(c^\prime)\arrow[d, "\diag(g)"]\\
				{d}\arrow[r, "\eta"] & \diag\colim d
			\end{tikzcd}
		\end{equation*}
		in $\IFun[\Over{\BB}{A}](\I{I},\pi_A^\ast\I{C})$ in which $\eta$ is the unit of the adjunction $\colim\dashv \diag$, the transpose map $\colim d^\prime\to c^\prime$ is an equivalence.\qed
	\end{enumerate}
\end{proposition}

\begin{proposition}
	\label{prop:characterisationEffectiveDescent}
	Let $\I{U}$ be an internal class and let $\I{C}$ be a $\I{U}$-cocomplete $\BB$-category with pullbacks. Then the following are equivalent:
	\begin{enumerate}
		\item $\I{C}$ has effective $\I{U}$-descent;
		\item for every $A\in\BB$, every $\I{I}\in \I{U}(A)$ and every cartesian map $d^\prime\to d$ in $\IFun[\Over{\BB}{A}](\I{I},\pi_A^\ast\I{C})$, the induced map between colimit cocones $\overline{d}^\prime\to\overline{d}$ is cartesian as well;
		\item for every $A\in \BB$, every $\I{I}\in\I{U}(A)$ and every cartesian map $d^\prime\to d$ in $\IFun[\Over{\BB}{A}](\I{I},\pi_A^\ast\I{C})$, the naturality square
		\begin{equation*}
			\begin{tikzcd}
				{d}^\prime\arrow[d]\arrow[r, "\eta"] & \diag(\colim d^\prime)\arrow[d]\\
				{d}\arrow[r, "\eta"] & \diag\colim d
			\end{tikzcd}
		\end{equation*}
		is a pullback.\qed
	\end{enumerate}
\end{proposition}

\begin{corollary}
	\label{cor:groupoidalDescentEtaleTransitionMaps}
	Let $S$ be a local class of maps in $\BB$ and let $\I{C}$ be an $\Univ[S]$-cocomplete $\BB$-category with pullbacks. Then the following are equivalent:
	\begin{enumerate}
		\item $\I{C}$ has $\Univ[S]$-descent;
		\item for every map $p\colon P\to A$ in $S$ the functor $p_!\colon\I{C}(P)\to\I{C}(A)$ is a right fibration;
		\item for every map $p\colon P\to A$ in $S$ the functor $\Over{(p_!)}{1_{\I{C}(P)}}\colon\I{C}(P)\to\Over{\I{C}(A)}{p_!(1_{\I{C}(P)})}$ is an equivalence.
	\end{enumerate}
\end{corollary}
\begin{proof}
	Since $\Over{(p_!)}{1_{\I{C}(P)}}$ is always final, this functor is an equivalence if and only if it is a right fibration, which is in turn equivalent to $p_!$ being a right fibration. Hence~(2) and~(3) are equivalent conditions. Now in light of the adjunction $p_!\dashv p^\ast$, a map $f\colon c^\prime\to c$ in $\I{C}(P)$ is cartesian with respect to $p_!$ precisely if the naturality square
	\begin{equation*}
		\begin{tikzcd}
			c^\prime\arrow[d, "f"]\arrow[r] & p^\ast p_! c^\prime\arrow[d, "p^\ast p_! f"]\\
			c\arrow[r] & p^\ast p_! c
		\end{tikzcd}
	\end{equation*}
	is a pullback. Therefore, Proposition~\ref{prop:characterisationEffectiveDescent} and the fact that \emph{every} map of diagrams indexed by a $\BB$-groupoid is cartesian imply that $\I{C}$ has effective $\Univ[S]$-descent if and only if for every map $p\colon P\to A$ in $S$, every morphism in $\I{C}(P)$ is cartesian with respect to $p_!$. By the same observation, Proposition~\ref{prop:characterisationFaithfulDescent} shows that $\I{C}$ has faithful $\Univ[S]$-descent if and only if for every map $p\colon P\to A$ in $S$, every object $c\in\I{C}(P)$ and every morphism $g\colon c^{\prime\prime}\to p_!(c)$ in $\I{C}(A)$, the pullback of $p^\ast(g)$ along the adjunction unit $c\to p^\ast p_! c$ defines a cartesian lift of $g$. In other words, $\I{C}$ has faithful $\Univ[S]$-descent if and only if $p_!$ is a cartesian fibration. Hence~(1) and~(2) are equivalent.
\end{proof}

For the next corollary, recall from \S~\ref{sec:bifiltration} that if $\KK$ is a class of $\infty$-categories and $S$ is a local class of morphisms in $\BB$, we denote by $\ICat_{\BB}^{\ord{\KK,S}}$ the \emph{left regularisation} of the internal class $\ILConst_{\KK}\cup \Univ[S]$, i.e.\ the smallest internal class that contains $\Delta\cup\ILConst_{\KK}\cup \Univ[S]$ and that is closed under $\ILConst_{\op(\KK)}\cup \Univ[S]$-colimits in $\ICat_{\BB}$ (where $\op(\KK)$ is the image of $\KK$ under the equivalence $(-)^\op\colon \CatS\simeq\CatS$). We now obtain:
\begin{corollary}
	\label{cor:descentForKSClasses}
	Let $\KK$ be a class of $\infty$-categories and let $S$ be a local class of maps in $\BB$. Let $\I{C}$ be a $\ICat_{\BB}^{\ord{\KK,S}}$-cocomplete $\BB$-category with pullbacks. Then $\ICat_{\BB}$ satisfies $\ICat_{\BB}^{\ord{\KK,S}}$-descent if and only if
	\begin{enumerate}
		\item for all $A\in\BB$ the $\infty$-category $\I{C}(A)$ satisfies $\KK$-descent, and
		\item for every map $p\colon P\to A$ in $S$ the functor $p_!$ is a right fibration.
	\end{enumerate}
\end{corollary}
\begin{proof}
	By Proposition~\ref{prop:invarianceCocompletenessRegularisation}, the $\BB$-category $\I{C}$ has $\ICat_{\BB}^{\ord{\KK,S}}$-descent precisely if it satisfies both $\ILConst_{\KK}$- and $\Univ[S]$-descent. By Example~\ref{ex:LConstKDescent} the first condition is equivalent to~(1), and by Corollary~\ref{cor:groupoidalDescentEtaleTransitionMaps} the second one is equivalent to~(2).
\end{proof}

\begin{example}
	\label{ex:modalityDescent}
	Let $S$ be a local class of morphisms in $\BB$ that is closed under pullbacks in $\Fun(\Delta^1,\BB)$ and that is \emph{left cancellable}, i.e.\ satisfies the condition that whenever there is a composable pair of morphisms $f$ and $g$ in $\BB$ for which $g$ is contained in $S$, then $gf$ is contained in $S$ if and only if $f$ is. Then the associated subuniverse $\Univ[S]\into\Univ$ is closed under pullbacks and under $\Univ[S]$-colimits, and for every map $s\colon B\to A$ in $S$ the functor $s_!\colon\Univ[S](B)\to\Univ[S](A)$ is a right fibration. Hence $\Univ[S]$ has $\Univ[S]$-descent. These conditions are for example satisfied if $S$ is the right complement of a factorisation system.
\end{example}

\subsection{Universality of colimits}
\label{sec:universalColimits}
The goal of this section is to establish that the notion of faithful $\I{U}$-descent is equivalent to \emph{universality} of $\I{U}$-colimits:
\begin{definition}
	\label{def:universalColimits}
	Let $\I{U}$ be an internal class of $\BB$-categories and let $\I{C}$ be a $\I{U}$-cocomplete $\BB$-category with pullbacks. We say that $\I{U}$-colimits are \emph{universal} in $\I{C}$ if for every map $f\colon c\to d$ in $\I{C}$ in context $A\in\BB$ the functor $f^\ast\colon\Over{\I{C}}{d}\to\Over{\I{C}}{c}$ is $\pi_A^\ast\I{U}$-cocontinuous. If $\I{C}$ is a cocomplete large $\BB$-category, we simply say that colimits are universal in $\I{C}$ if $\ICat_{\BB}$-colimits are universal in $\I{C}$.
\end{definition}
\begin{remark}
	\label{rem:universalColimitsWelldefined}
	In the situation of Definition~\ref{def:universalColimits}, note that by Corollary~\ref{cor:sliceCategoryUCocomplete} the $\Over{\BB}{A}$-category $\Over{\I{C}}{c}\simeq \Over{(\pi_A^\ast\I{C})}{\bar c}$ (where $\bar c\colon 1\to \pi_A^\ast\I{C}$ is the transpose of $c$) is $\pi_A^\ast\I{U}$-cocomplete for every $c\colon A\to\I{C}$. Therefore, asking for $f^\ast$ to be $\pi_A^\ast\I{U}$-cocontinuous makes sense.
\end{remark}

\begin{remark}
	\label{rem:BCUniversalColimits}
	The condition that $\I{U}$-colimits are universal in $\I{C}$ is local in $\BB$: if $\bigsqcup_i A_i\onto 1$ is a cover in $\BB$, then $\I{U}$-colimits are universal in $\I{C}$ if and only if $\pi_{A_i}^\ast\I{U}$-colimits are universal in $\pi_{A_i}^\ast\I{C}$ for each $i$. This is easily seen using the fact that $\I{U}$-cocontinuity is a local condition~\cite[Remark~5.2.3]{MWColimits}.
\end{remark}

\begin{example}
	\label{ex:universalColimitsFibrewise}
	Let $\KK$ be a class of $\infty$-categories and let $\I{C}$ be an $\ILConst_{\KK}$-cocomplete $\BB$-category with pullbacks. Then $\ILConst_{\KK}$-colimits are universal in $\I{C}$ if and only if $\KK$-colimits are universal in $\I{C}(A)$ for all $A\in\BB$. In fact, this follows immediately from the observation that for every map $f\colon c\to d$ in $\I{C}$ in context $A\in\BB$ and for every map $s\colon B\to A$ in $\BB$ the functor $f^\ast(B)$ can be identified with $(s^\ast f)^\ast\colon \Over{\I{C}(B)}{d}\to\Over{\I{C}(B)}{c}$.
\end{example}

\begin{proposition}
	\label{prop:universalColimitsFaithfulDescent}
	Let $\I{U}$ be an internal class of $\BB$-categories and let $\I{C}$ be a $\I{U}$-cocomplete $\BB$-category with pullbacks. Then $\I{U}$-colimits are universal in $\I{C}$ if and only if $\I{C}$ has faithful $\I{U}$-descent.
\end{proposition}
\begin{proof}
	Let $\I{I}$ be an object in $\I{U}(1)$, and let $f\colon c^\prime\to c$ be an arbitrary map in $\I{C}$ in context $1\in \BB$. Suppose that $d\colon \I{I}\to\Over{\I{C}}{c}$ is a diagram with colimit cocone $\overline{d}\colon d\to\diag\colim d$, which we may equivalently regard as a diagram $\overline{d}\colon\I{I}\to\Over{\I{C}}{\colim d}$. Let $g\colon \colim d\to c$ be the induced map. On account of the fact that the (vertical) mate of the commutative square
	\begin{equation*}
		\begin{tikzcd}
			\Over{\I{C}}{f^\ast(\colim d)}\arrow[r, "(f^\ast(g))_!"]\arrow[d, "(g^\ast f)_!"] & \Over{\I{C}}{c^\prime}\arrow[d, "f_!"]\\
			\Over{\I{C}}{\colim d}\arrow[r, "g_!"] & \Over{\I{C}}{c}
		\end{tikzcd}
	\end{equation*}
	commutes and since the horizontal maps in this diagram are conservative and $\I{U}$-cocontinuous, the functor $f^\ast$ preserves the colimit of $d$ if and only if the functor $(g^\ast f)^\ast$ preserves the colimit of $\overline{d}\colon\I{I}\to\Over{\I{C}}{\colim d}$. By (the proof of) Proposition~\ref{prop:sliceFibrationColimits}, the colimit of the latter is the final object in $\Over{\I{C}}{\colim d}$. Therefore, in order to show that $\I{C}$ has $\I{U}$-universal colimits, it suffices to consider those diagrams in $\Over{\I{C}}{c}$ whose colimit is the final object.
	
	Now on account of the commutative square
	\begin{equation*}
		\begin{tikzcd}
			\Over{\IFun(\I{I},\I{C})}{\diag(c^\prime)}\arrow[d, "\diag(f)_!"]\arrow[r, "\simeq"] &\IFun(\I{I},\Over{\I{C}}{c^\prime})\arrow[d, "(f_!)_\ast"]\\
			\Over{\IFun(\I{I},\I{C})}{\diag(c)}\arrow[r, "\simeq"] &\IFun(\I{I},\Over{\I{C}}{c})
		\end{tikzcd}
	\end{equation*}
	we may identify $\diag(f)^\ast$ with $(f^\ast)_\ast$. Therefore, if $d\to \diag(c)$ is a colimit cocone, the upper horizontal equivalence in the above diagram identifies its pullback $d^\prime\to\diag(c^\prime)$ along $\diag(f)$ with the composition $\I{I}\xrightarrow{\overline{d}}\Over{\I{C}}{c}\xrightarrow{f^\ast}\Over{\I{C}}{c^\prime}$. Thus, by again using Proposition~\ref{prop:sliceFibrationColimits}, the map $d^\prime\to \diag(c^\prime)$ is a colimit cocone if and only if the colimit of $f^\ast \overline{d}$ is the final object in $\Over{\I{C}}{c^\prime}$, which is in turn equivalent to $f^\ast$ preserving the colimit of $\overline{d}$. As replacing $\BB$ with $\Over{\BB}{A}$ allows us to arrive at the same conclusion for any $\I{I}\in\I{U}(A)$, Proposition~\ref{prop:characterisationFaithfulDescent} yields the claim.
\end{proof}

\begin{example}
	\label{ex:universalityGroupoidalColimitsCartesianFibration}
	Let $S$ be a local class of morphisms in $\BB$ and let $\I{C}$ be a $\Univ[S]$-cocomplete $\BB$-category with pullbacks. Then $\Univ[S]$-colimits are universal in $\I{C}$ if and only if for every map $p\colon P\to A$ in $S$ the functor $p_!\colon\I{C}({P})\to\I{C}(A)$ is a cartesian fibration. In fact, in light of Proposition~\ref{prop:universalColimitsFaithfulDescent} this follows from the argument in the proof of Corollary~\ref{cor:groupoidalDescentEtaleTransitionMaps}.
\end{example}

We end this section by relating universality of colimits with the property of being \emph{locally cartesian closed}:
\begin{proposition}
	\label{prop:universalityOfColimitsCartesianClosed}
	Let $\I{X}$ be a presentable $\BB$-category. Then $\I{X}$ has $\Univ$-universal colimits if and only if for every object $x\colon A\to \I{X}$, the $\Over{\BB}{A}$-category $\Over{(\pi_A^\ast\I{X})}{x}$ is cartesian closed, which is to say that there exists a bifunctor $\Hom_{\pi_A^\ast\I{X}}(-,-)\colon \Over{(\pi_A^\ast\I{X})}{x}^\op\times \Over{(\pi_A^\ast\I{X})}{x}\to \Over{(\pi_A^\ast\I{X})}{x}$ that fits into an equivalence
	\begin{equation*}
		\map{\Over{(\pi_A^\ast\I{X})}{x}}(-\times -, -)\simeq \map{\Over{(\pi_A^\ast\I{X})}{x}}(-, \Hom_{\pi_A^\ast\I{X}}(-,-)).
	\end{equation*}
\end{proposition}
\begin{proof}
	It will be enough to show that $\I{X}$ is cartesian closed if and only if for every $y\colon A\to \I{X}$ the functor $(\pi_y)^\ast\colon \pi_A^\ast\I{X}\to \Over{(\pi_A^\ast\I{X})}{y}$ is $\Univ[\Over{\BB}{A}]$-cocontinuous. Using Remark~\ref{rem:BCUniversalColimits}, we can assume that $A\simeq 1$. Recall that the forgetful functor $(\pi_y)_!\colon\Over{\I{X}}{y}\to\I{Y}$ is $\Univ$-cocontinuous (Corollary~\ref{cor:sliceCategoryUCocomplete}). As this functor is moreover a right fibration and therefore in particular conservative, we find that $\pi_y^\ast$ is $\Univ$-cocontinuous if and only if the composition $(\pi_y)_!\pi_y^\ast$ is. Together with Proposition~\ref{prop:characterisationCategoriesWithProducts}, this shows that $\pi_y^\ast$ being $\Univ$-cocontinuous is equivalent to $y\times -$ being $\Univ$-cocontinuous. As $\I{X}$ is presentable, this is in turn equivalent to $y\times -$ having a right adjoint $\Hom_{\I{X}}(y,-)$ (see Proposition~\ref{prop:adjointFunctorTheorem}). Clearly, this holds if $\I{X}$ is cartesian closed. Conversely, if $y\times -$ admits a right adjoint for all $y\colon A\to\I{X}$, then $\map{\I{X}}(-\times -, -)$, viewed as a functor $\I{X}^\op\times\I{X}\to\IPSh(\I{X})$, takes values in $\I{X}\into\IPSh(\I{X})$ and therefore gives rise to the desired internal hom.
\end{proof}

\subsection{Disjoint colimits}
\label{sec:disjointColimits}
If $\CC$ is an $\infty$-category with pullbacks and finite coproducts, one says that a coproduct $c_0\sqcup c_1$ in $\CC$ is \emph{disjoint} if the fibre product $c_i\times_{c_0\sqcup c_1}c_j$ is equivalent to $c_i$ if $i=j$ and the initial object otherwise. In this section, our goal is to study an internal analogue of this concept. In fact, we will define what it means for arbitrary \emph{$\BB$-groupoidal} colimits to be disjoint. To that end, recall that if $S$ is an arbitrary local class of morphisms in $\BB$, the associated subuniverse $\Univ[S]$ is contained in the free $\Univ[S]$-cocompletion $\IPSh^{\Univ[S]}(1)$, cf.~\cite[Example~7.3.5]{MWColimits}. Therefore, if $\I{C}$ is an arbitrary $\Univ[S]$-cocomplete $\BB$-category, the tensoring bifunctor
\begin{equation*}
	-\otimes -\colon\Univ[S]\times\I{C}\to\I{C}
\end{equation*}
(see~\cite[Proposition~7.3.8]{MWColimits}) is well-defined. Furthermore, note that $S$ is closed under diagonals if and only if for every $\I{G}\in\Univ[S](A)$ and every pair of objects $g,g^\prime\colon A\rightrightarrows \I{G}$ the mapping $\Over{\BB}{A}$-groupoid $\map{\I{G}}(g,g^\prime)$ is contained in $\Univ[S](A)$ as well. Therefore, we may define:
\begin{definition}
	\label{def:disjointColimits}
	Let $S$ be a local class of morphisms in $\BB$ that is closed under diagonals, and let $\I{C}$ be an $\Univ[S]$-cocomplete $\BB$-category with pullbacks. If $\I{G}\in \Univ[S](1)$ is an arbitrary object, we say that \emph{$\I{G}$-indexed colimits are disjoint in $\I{C}$} if for all diagrams $d\colon \I{G}\to\I{C}$ and for every pair of objects $g,g^\prime$ in $\I{G}$ in context $1\in\BB$ the diagram
	\begin{equation*}
		\begin{tikzcd}
			\map{\I{G}}(g,g^\prime)\otimes d(g)\arrow[d]\arrow[r] & d(g^\prime)\arrow[d]\\
			d(g)\arrow[r] & \colim d
		\end{tikzcd}
	\end{equation*}
	is a pullback. We say that \emph{$\Univ[S]$-colimits are disjoint in $\I{C}$} if for all $A\in \BB$ and all $\I{G}\in\I{U}(A)$ all $\I{G}$-indexed colimits are disjoint in $\pi_A^\ast\I{C}$.
\end{definition}

\begin{remark}
	In the situation of Definition~\ref{def:disjointColimits}, let $\overline{d}\colon \I{G}^\triangleright\to \I{C}$ be the colimit cocone associated with $d$. Then the commutative square in the definition is obtained by transposing the commutative diagram
	\begin{equation*}
		\begin{tikzcd}
			\map{\I{G}^\triangleright}(g,g^\prime)\arrow[d]\arrow[r, "\overline{d}"] & \map{\I{C}}(d(g), d(g^\prime))\arrow[d]\\
			\map{\I{G}^\triangleright}(g, \infty)\arrow[r, "\overline{d}"] & \map{\I{C}}(d(g), \colim d)
		\end{tikzcd}
	\end{equation*}
	across the equivalence $\map{\I{C}}(-\otimes -, -)\simeq\map{\Univ}(-,\map{\I{C}}(-,-))$, noting that since $\iota\colon\I{G}\into\I{G}^\triangleright$ is fully faithful the upper left corner can be identified with $\map{\I{G}}(g,g^\prime)$ and since $\infty\colon 1\to\I{G}^\triangleright$ is final the lower left corner is equivalent to the final object.
\end{remark}

\begin{example}
	\label{ex:disjointSums}
	Let us unwind Definition~\ref{def:disjointColimits} in the case where $\BB=\SS$ and where $\I{G}=\{0,1\}$. Then a diagram $d\colon \{0,1\}\to \CC$ in an $\infty$-category $\CC$ is simply given by a pair $(c_0,c_1)$ of objects in $\CC$, and its colimit is the coproduct $c_0\sqcup c_1$. Furthermore, the square in Definition~\ref{def:disjointColimits} is explicitly given by
	\begin{equation*}
		\begin{tikzcd}
			\map{\{0,1\}}(i,j)\times c_i\arrow[r]\arrow[d] & c_j\arrow[d]\\
			c_i\arrow[r] & c_0\sqcup c_1
		\end{tikzcd}
	\end{equation*}
	(for $i,j\in\{0,1\}$) and is therefore a pullback for all pairs $(i,j)$ precisely if coproducts are disjoint in $\CC$ in the usual sense.
\end{example}

\begin{remark}
	\label{rem:disjointColimitsLocalCondition}
	The property of $\Univ[S]$-colimits to be disjoint in $\I{C}$ is a local condition. More precisely, if $\I{G}\in\I{U}(1)$ is an arbitrary object and if $\bigsqcup_i A_i\onto 1$ is a cover in $\BB$, then $\I{G}$-indexed colimits are disjoint in $\I{C}$ if and only if $\pi_{A_i}^\ast\I{G}$-indexed colimits are disjoint in $\pi_{A_i}^\ast\I{C}$ for all $i$. This follows immediately from the fact that both limits and colimits are determined locally, cf.~\cite[Remark~4.1.8]{MWColimits}. As a consequence, if $\Univ[S]$ is generated by a family of objects $(\I{G}_i\colon A_i\to\Univ[S])_i$, then $\Univ[S]$-colimits are disjoint in $\I{C}$ precisely if $\I{G}_i$-indexed colimits are disjoint in $\pi_{A_i}^\ast\I{C}$ for all $i$.
\end{remark}

\begin{example}
	\label{ex:disjointColimitsFibrewise}
	Let $S$ be the local class of morphisms in $\BB$ that is generated by $\varnothing, 1$ and $2=1\sqcup 1$. Then $S$ is closed under diagonals. By using Remark~\ref{rem:disjointColimitsLocalCondition} and Example~\ref{ex:disjointSums}, one finds that $\Univ[S]$-colimits are disjoint in $\I{C}$ if and only if coproducts are disjoint in $\I{C}(A)$ for all $A\in\BB$.
\end{example}

The main goal of this section is to show:
\begin{proposition}
	\label{prop:disjointColimitsEffectiveDescent}
	Let $S$ be a local class of morphisms in $\BB$ that is closed under diagonals, and let $\I{C}$ be an $\Univ[S]$-cocomplete $\BB$-category with pullbacks in which $\Univ[S]$-colimits are universal. Then $\I{C}$ has effective $\Univ[S]$-descent if and only if $\Univ[S]$-colimits are disjoint.
\end{proposition}
In order to prove Proposition~\ref{prop:disjointColimitsEffectiveDescent}, we will need a  more explicit description of the notion of disjoint $\Univ[S]$-colimits.
The key input is the following construction:
\begin{construction}
	\label{constr:disjointColimits}
	Let $S$ be a local class of morphisms in $\BB$ that is closed under diagonals, and let $\I{C}$ be an $\Univ[S]$-cocomplete $\BB$-category with pullbacks. Suppose that $p\colon P\to A$ is a map in $S$, and let $c\colon P\to \I{C}$ be an arbitrary object. Let $\eta\colon c\to p^\ast p_! c$ be the adjunction unit, and consider the pullback square
	\begin{equation*}
		\begin{tikzcd}
			z\arrow[r]\arrow[d] & \pr_1^\ast(c)\arrow[d,"\pr_1^\ast(\eta)"]\\
			\pr_0^\ast(c)\arrow[r, "\pr_0^\ast(\eta)"] & \pr_0^\ast p^\ast p_!(c)
		\end{tikzcd}
	\end{equation*}
	in $\I{C}(P\times_A P)$ (where we implicitly identify $\pr_0^\ast p^\ast\simeq \pr_1^\ast p^\ast$). Note that if $\Delta_p\colon P\to P\times_A P$ is the diagonal map, the pullback of the above square along $\Delta_p$ yields the pullback square
	\begin{equation*}
		\begin{tikzcd}
			c\times_{p^\ast p_!(c)} c\arrow[r]\arrow[d] & c\arrow[d]\\
			c\arrow[r] & p^\ast p_!(c)
		\end{tikzcd}
	\end{equation*}
	in $\I{C}(P)$. Therefore, the diagonal map $c\to c\times_{p^\ast p_!(c)} c$ transposes to a map $\delta_p(c)\colon(\Delta_p)_!(c)\to z$.
\end{construction}

\begin{proposition}
	\label{prop:characterisationDisjointColimits}
	Let $S$ be a local class of morphisms in $\BB$ that is closed under diagonals, and let $\I{C}$ be an $\Univ[S]$-cocomplete $\BB$-category with pullbacks. Then $\Univ[S]$-colimits are disjoint in $\I{C}$ if and only if for all maps $p\colon P\to A$ and all objects $c\colon P\to\I{C}$ the map $\delta_p(c)$ from Construction~\ref{constr:disjointColimits} is an equivalence.
\end{proposition}
\begin{proof}
	By identifying $p\colon P\to A$ with a $\Over{\BB}{A}$-groupoid $\I{G}$, the object $c\colon P\to \I{C}$ corresponds to a diagram $d\colon \I{G}\to\pi_A^\ast\I{C}$. Also, the two maps $\pr_0,\pr_1\colon P\times_A P\rightrightarrows P$ correspond to objects $g$ and $g^\prime$ in $\I{G}$ in context $P\times_A P$.
	In light of these identifications, the two cospans
	\begin{equation*}
		\begin{tikzcd}
			& \pr_1^\ast(c)\arrow[d,"\pr_1^\ast(\eta_P)"] &&&& & d(g^\prime)\arrow[d]\\
			\pr_0^\ast(c)\arrow[r, "\pr_0^\ast(\eta_P)"] & \pr_0^\ast p^\ast p_!(c) &&&& d(g)\arrow[r] & \pr_0^\ast p^\ast(\colim d)
		\end{tikzcd}
	\end{equation*}
	(in context $P\times_A P$) are translated into each other. Next, we note that pulling back $g$ and $g^\prime$ along the diagonal $\Delta\colon P\to P\times_A P$ recovers the tautological object $\tau$ of $\I{G}$ (i.e.\ the one corresponding to $\id_P$). As there is a section $\id_\tau\colon P\to \map{\I{G}}(\tau,\tau)$, we thus obtain a commutative diagram
	\begin{equation*}
		\begin{tikzcd}
			d(\tau)\arrow[dr]\arrow[drr, bend left=15, "\id"]\arrow[ddr, bend right=15, "\id"'] && \\
			& \map{\I{G}}(\tau,\tau)\otimes d(\tau)\arrow[r]\arrow[d] & d(\tau)\arrow[d]\\
			& d(\tau)\arrow[r] & p^\ast(\colim d)
		\end{tikzcd}
	\end{equation*}
	in context $P$. Observe that value of the unit of the adjunction $(\Delta_!\dashv\Delta^\ast)\colon \Over{\BB}{P}\leftrightarrows\Over{\BB}{P\times_A P}$ at the final object precisely recovers the map $\id_\tau\colon P\to \map{\I{G}}(\tau,\tau)$. As the functor 
	\begin{equation*}
		-\otimes d(\tau)\colon \pi_P^\ast\Univ\to\pi_P^\ast\I{C}
	\end{equation*}
	is by construction $\pi_P^\ast\I{U}$-cocontinuous, one thus finds that the map $d(\tau)\to\map{\I{G}}(\tau,\tau)\otimes d(\tau)$ can be identified with the unit of the adjunction $(\Delta_!\dashv \Delta^\ast)\colon \I{C}(P)\leftrightarrows\I{C}(P\times_A P)$. But this precisely means that the transpose map $\Delta_! d(\tau)\to \map{\I{G}}(g,g^\prime)\otimes d(g)$.  must be an equivalence.
	As $d(\tau)$ is simply $c$, we therefore find that the two diagrams
	\begin{equation*}
		\begin{tikzcd}
			\Delta_!(c)\arrow[r]\arrow[d] & \pr_1^\ast(c)\arrow[d,"\pr_1^\ast(\eta)"] &&& \map{\I{G}}(g,g^\prime)\otimes d(g)\arrow[r]\arrow[d]& d(g^\prime)\arrow[d]\\
			\pr_0^\ast(c)\arrow[r, "\pr_0^\ast(\eta)"] & \pr_0^\ast p^\ast p_!(c) &&& d(g)\arrow[r] & \pr_0^\ast p^\ast(\colim d)
		\end{tikzcd}
	\end{equation*}
	are equivalent. To complete the proof, we still need to show that the right square being cartesian is equivalent to $\Univ[S]$-colimits being disjoint in $\I{C}$. Certainly, this is a necessary condition since this square is precisely of the form as in Definition~\ref{def:disjointColimits} (after identifying $\pr_0^\ast p^\ast(\colim d)$ with $\colim \pr_0^\ast p^\ast d$ and regarding $g$ and $g^\prime$ as objects of $\pr_0^\ast p^\ast \I{G}$ in context $1\in\Over{\BB}{P\times_A P}$). The converse follows from the observation that \emph{every} pair of objects $h,h^\prime\colon A\rightrightarrows\I{G}$ must be a pullback of $g$ and $g^\prime$, i.e.\ that the above diagram is the \emph{universal} one.
\end{proof}

\begin{proof}[{Proof of Proposition~\ref{prop:disjointColimitsEffectiveDescent}}]
	We will freely make use of the setup from Proposition~\ref{prop:characterisationDisjointColimits}. Therefore, let us fix a map $p\colon P\to A$ in $S$, and let us denote the unit and counit of the associated adjunction $p_!\dashv p^\ast$ by $\eta_p$ and $\epsilon_p$, respectively.
	We first assume that $\I{C}$ has effective $\Univ[S]$-descent. Choose an arbitrary object $c\colon P\to\I{C}$ and consider the pullback
	\begin{equation*}
		\begin{tikzcd}
			z\arrow[r, "g"]\arrow[d] & \pr_1^\ast(c)\arrow[d,"\pr_1^\ast(\eta_p)"]\\
			\pr_0^\ast(c)\arrow[r, "\pr_0^\ast(\eta_p)"] & \pr_0^\ast p^\ast p_!(c)
		\end{tikzcd}
	\end{equation*}
	in $\I{C}(P\times_A P)$. 
	By making use of the commutative diagram
	\begin{equation*}
		\begin{tikzcd}
			\pr_0^\ast(c)\arrow[r, "\pr_0^\ast\eta_p c"] \arrow[dr, "\eta_{\pr_1}\pr_0^\ast(c)"']& \pr_1^\ast p^\ast p_!(c)\arrow[d, "\pr_1^\ast(\alpha)"]\\
			&\pr_1^\ast(\pr_1)_!\pr_0^\ast(c),
		\end{tikzcd}
	\end{equation*}
	(where $\alpha$ is an equivalence owing to $\I{C}$ having $\Univ[S]$-colimits),
	we may identify the above square with the pullback square
	\begin{equation*}
		\begin{tikzcd}[column sep=large]
			z\arrow[r, "g"]\arrow[d] & \pr_1^\ast(c)\arrow[d,"\pr_1^\ast(\alpha\eta_p)"]\\
			\pr_0^\ast(c)\arrow[r, "\eta_{\pr_1}\pr_0^\ast(c)"] & \pr_1^\ast (\pr_1)_! \pr_0^\ast(c).
		\end{tikzcd}
	\end{equation*}
	Since by assumption $\Univ[S]$-colimits are universal in $\I{C}$, Proposition~\ref{prop:universalColimitsFaithfulDescent} implies that $\I{C}$ has faithful $\Univ[S]$-descent. Hence Proposition~\ref{prop:characterisationFaithfulDescent} implies that the transpose $(\pr_1)_!(z)\to c$ of $g$ must be an equivalence. Together with the commutative diagram
	\begin{equation*}
		\begin{tikzcd}[column sep=large]
			(\pr_1)_!\Delta_!(c)\arrow[d, "\simeq"]\arrow[r]\arrow[rr, bend left=25, "(\pr_1)_!(\delta_{p}(c))"] & (\pr_1)_!\Delta_!\Delta^\ast(z)\arrow[d, "\simeq"]\arrow[r, "(\pr_1)_!\epsilon_{\Delta}z"] &(\pr_1)_!(z)\arrow[d]\\
			c\arrow[r]\arrow[rr, bend right=25, "\id"] & \Delta^\ast(z)\arrow[r, "\Delta^\ast(g)"] & c,
		\end{tikzcd}
	\end{equation*}
	this observation implies that $(\pr_1)_!(\delta_{p}(c))$ is an equivalence. But since $\I{C}$ has effective and faithful $\Univ[S]$-descent, Proposition~\ref{cor:groupoidalDescentEtaleTransitionMaps} implies that $(\pr_1)_!$ is a right fibration and therefore in particular conservative. Hence $\delta_{p}(c)$ is already an equivalence.
	
	Conversely, suppose that $\Univ[S]$-colimits in $\I{C}$ are disjoint and let $f\colon c\to d$ be an arbitrary map in $\I{C}$ in context $P$.
	Consider the diagram
	\begin{equation*}
		\begin{tikzcd}
			c\arrow[dr, "\phi"]\arrow[drr, bend left, "\eta_P c"]\arrow[ddr, bend right, "f"]& &\\
			&e\arrow[d]\arrow[r] & p^\ast p_!(c)\arrow[d, "p^\ast p_!(f)"]\\
			&d\arrow[r, "\eta d"] & p^\ast p_!(d)
		\end{tikzcd}
	\end{equation*}
	in $\I{C}(P)$ in which the square is a pullback. By Proposition~\ref{prop:characterisationEffectiveDescent}, the result follows once we show that $\phi$ is an equivalence. We now obtain a pullback diagram
	\begin{equation*}
		\label{eq:stupidCube}
		\tag{$\ast$}
		\begin{tikzcd}[column sep={4em,between origins}, row sep={3em,between origins}]
			& x \arrow[rr]\arrow[dd]\arrow[dl] && \pr_1^\ast(c)\arrow[dd]\arrow[dl]\\
			\pr_0^\ast(e)\arrow[rr, crossing over] \arrow[dd]&& \pr_0^\ast p^\ast p_!(c) & \\
			& y\arrow[rr]\arrow[dl] && \pr_1^\ast(d)\arrow[dl]\\
			\pr_0^\ast(d)\arrow[rr] && \pr_0^\ast p^\ast p_!(d)\arrow[from=uu, crossing over]
		\end{tikzcd}
	\end{equation*}
	in $\Fun(\Delta^1,\I{C}(P\times_A P))$ in which the front square is obtained by applying $\pr_0^\ast$ to the pullback square in the previous diagram and the right square is given by applying $\pr_1^\ast$ to the outer square in the previous diagram. Note that by applying the functor $\Delta^\ast$ to this cube, we obtain a commutative diagram
	\begin{equation*}
		\begin{tikzcd}[column sep={3em,between origins}, row sep={1.7em,between origins}]
			& c\arrow[dddd] \arrow[dr]\arrow[drrrr, bend left=15, "\id", near end]\arrow[dddl, bend right=15, "\phi"', near start]  &&&&\\
			&& \Delta^\ast(x) \arrow[rrr]\arrow[ddll, crossing over] &&& c\arrow[dddd]\arrow[ddll]\\
			&&&&&\\
			|[alias=e]|e&&& |[alias=pc]|p^\ast p_!(c) && \\
			& d \arrow[dr]\arrow[drrrr, bend left=15, "\id", near end]\arrow[dddl, bend right=15, "\id"', near start]&&&&\\
			&& \Delta^\ast(y)\arrow[rrr]\arrow[ddll]\arrow[from=uuuu, crossing over] &&& d\arrow[ddll]\\
			&&&&&\\
			d\arrow[rrr]\arrow[from=uuuu, crossing over] &&& p^\ast p_!(d)\arrow[from=uuuu, crossing over]&&\arrow[from=e, to=pc, crossing over] 
		\end{tikzcd}
	\end{equation*}
	in which the cube defines a pullback in $\Fun(\Delta^1, \I{C}(P))$. 
	Now by disjointness of $\Univ[S]$-colimits, the map $\Delta_!(d)\to y$ must be an equivalence. Since this map fits into a commutative diagram
	\begin{equation*}
		\begin{tikzcd}
			d\arrow[r, "\eta^\prime d"] \arrow[dr] & \Delta^\ast\Delta_!(c)\arrow[d, "\simeq"]\\
			&\Delta^\ast(y)
		\end{tikzcd}
	\end{equation*}
	(in which $\eta^\prime$ denotes the unit of the adjunction $\Delta_!\dashv \Delta^\ast$) and since we have a pullback square
	\begin{equation*}
		\begin{tikzcd}
			c\arrow[d]\arrow[r] & \Delta^\ast(x)\arrow[d]\\
			d\arrow[r] & \Delta^\ast(y),
		\end{tikzcd}
	\end{equation*}
	the assumption that $\Univ[S]$-colimits are universal in $\I{C}$ and Proposition~\ref{prop:characterisationFaithfulDescent} imply that the transpose map $\Delta_!(c)\to x$ is an equivalence as well. By the argument in the beginning of the proof, applied to the top square in~(\ref{eq:stupidCube}), the commutative diagram	
	\begin{equation*}
		\begin{tikzcd}
			\pr_1^\ast(c)\arrow[r, "\pr_1^\ast\eta_p c"] \arrow[dr, "\eta_{\pr_0}\pr_1^\ast(c)"']& \pr_0^\ast p^\ast p_!(c)\arrow[d, "\simeq"]\\
			&\pr_0^\ast(\pr_0)_!\pr_1^\ast(c)
		\end{tikzcd}
	\end{equation*}
	implies that the map $(\pr_0)_!(x)\to e$ is an equivalence too. Taken together, we thus conclude that the composition
	\begin{equation*}
		c\xrightarrow{\simeq} (\pr_0)_!\Delta_!(c)\to (\pr_0)_!(x)\to e
	\end{equation*}
	is an equivalence. By its very construction, this map can be identified with $\phi$, hence the claim follows.
\end{proof}

\subsection{Effective groupoid objects}
\label{sec:effectiveGroupoids}
In this section we briefly review the notion of \emph{groupoid objects} and their relation to descent (as discussed in~\cite[\S~6.1]{htt}) in the context of $\BB$-category theory.
\begin{definition}
	\label{def:equivalenceRelation}
	Let $\I{C}$ be a $\BB$-category with pullbacks. A \emph{groupoid object} in $\I{C}$ is a functor $G_\bullet\colon\Delta^{\op}\to \I{C}$ such that for all $n\geq 0$ and every decomposition $\ord{n}\simeq \ord{k}\sqcup_{\ord{0}}\ord{l}$ the map $G_n\to G_k\times_{G_0}G_l$ is an equivalence. We denote by $\SGrpd(\I{C})$ the full subcategory of $\IFun(\Delta^{\op},\I{C})$ spanned by the groupoid objects in $\pi_A^\ast\I{C}$ for every $A\in \BB$.
\end{definition}

\begin{definition}
	\label{def:effectiveEquivalenceRelation} 
	Let $\I{C}$ be a $\BB$-category that admits $\Delta^{\op}$-indexed colimits and pullbacks. We say that a groupoid object $G_\bullet$ in $\I{C}$ is \emph{effective} if the map $G_1\to G_0\times_{\colim G_\bullet}G_0$ is an equivalence in $\I{C}$. We denote by $\SGrpdEff(\I{C})$ the full subcategory of $\SGrpd(\I{C})$ that is spanned by the effective groupoid objects in $\pi_A^\ast\I{C}$ for every $A\in\BB$. We say that \emph{groupoid objects are effective} in $\I{C}$ if the inclusion $\SGrpdEff(\I{C})\into\SGrpd(\I{C})$ is an equivalence.
\end{definition}

\begin{remark}
	\label{rem:BCGroupoidObjects}
	Since the property of a map being an equivalence is local in $\BB$, it follows immediately from the definition that an object $A\to\IFun(\Delta^\op,\I{C})$ is contained in $\SGrpd(\I{C})$ if and only if it encodes a groupoid object in $\pi_A^\ast\I{C}$, which is in turn equivalent to its transpose $\Delta^\op\to\I{C}(A)$ being a groupoid object in the conventional sense. An analogous remark can be made for effective groupoid objects. In particular, groupoid objects are effective in $\I{C}$ if and only if they are effective in $\I{C}(A)$ for each $A\in\BB$.
\end{remark}

Let $\Pos$ be the $1$-category of posets, which we always identify with $0$-categories.
Observe that the functor $(-)^\triangleright\colon\Pos\to\Pos$ that freely adjoins a final object to a partially ordered set  restricts to a functor $(-)^\triangleright\colon\Delta^{\triangleleft}\to\Delta$, and the map $\id_{\Pos}\into (-)^\triangleright$ restricts to a map $\id_{\Delta}\to (-)^\triangleright\iota$, where $\iota\colon\Delta\into\Delta^{\triangleleft}$ is the inclusion. By precomposition, we thus obtain a functor $(-)_{+1}\colon\IFun(\Delta^{\op},\I{C})\to\IFun((\Delta^{\op})^\triangleright,\I{C})$ together with a morphism $\iota^\ast(-)_{+1}\to \id_{\IFun(\Delta^{\op},\I{C})}$. Now using Remark~\ref{rem:BCGroupoidObjects}, one finds:
\begin{proposition}[{\cite[Lemma~6.1.3.7 and Remark~6.1.3.18]{htt}}]
	\label{prop:groupoidObjectCartesianTransformation}
	Let $\I{C}$ be a $\BB$-category that admits $\Delta^{\op}$-indexed colimits and pullbacks, and let $G_{\bullet}\colon\Delta^{\op}\to\I{C}$ be a simplicial object. Then $G_{\bullet +1}$ is a colimit cocone, and $G_{\bullet}$ is a groupoid object if and only if the morphism of functors $\iota^\ast G_{\bullet +1}\to G_\bullet$ is cartesian.\qed
\end{proposition}
By combining Proposition~\ref{prop:groupoidObjectCartesianTransformation} with Proposition~\ref{prop:characterisationEffectiveDescent}, we conclude:
\begin{corollary}
	\label{cor:effectiveDescentEffectiveGroupoidObjects}
	Let $\I{U}$ be the internal class that is spanned by $\Delta^{\op}\colon 1\to \ICat_{\BB}$ and let $\I{C}$ be a $\I{U}$-cocomplete $\BB$-category with pullbacks that has effective $\I{U}$-descent. Then groupoid objects are effective in $\I{C} $.\qed
\end{corollary}

\section{Foundations of $\BB$-topos theory}
\label{chap:foundations}
In this section we develop the basic theory of $\BB$-topoi. We begin in \S~\ref{sec:definitionBTopoi} by giving an axiomatic definition of this concept using the notion of descent that has been established in the previous section. By unwinding the descent condition, we furthermore establish an explicit characterisation of $\BB$-topoi in terms of the underlying $\CatSS$-valued sheaves on $\BB$. In \S~\ref{sec:freeTopoi}, we construct the \emph{free} $\BB$-topos on an arbitrary $\BB$-category, which we use in  \S~\ref{sec:presentationTopoi} to establish a characterisation of $\BB$-topoi as left exact and accessible Bousfield localisations of presheaf $\BB$-categories. In \S~\ref{sec:CatEnrichementTopoi}, we make use of this characterisation to show that the $\BB$-category of $\BB$-topoi is tensored and powered over $\ICat_{\BB}$. In \S~\ref{sec:relativeInternalTopoi}, we prove that $\BB$-topoi are entirely determined by their global sections, in the sense that the $\infty$-category of $\BB$-topoi is equivalent to that of geometric morphisms of $\infty$-topoi with codomain $\BB$. Having this simple description of $\BB$-topoi at our disposal, it is straightforward to construct limits and colimits of $\BB$-topoi, which is the topic of \S~\ref{sec:limitsColimitsTopoi}. Also, we provide an explicit formula for the coproduct of $\BB$-topoi in \S~\ref{sec:coproductTopoi}, which in particular yields a formula for the pushout in $\LTopS$. In \S~\ref{sec:Diaconescu}, we discuss a $\BB$-categorical version of Diaconescu's theorem for $\BB$-topoi, from which we deduce a universal property of \'etale $\BB$-topoi in \S~\ref{sec:EtaleTopoi}. Lastly, we discuss \emph{subterminal} $\BB$-topoi in \S~\ref{sec:sheafification}, where we derive a general formula for left exact localisations in terms of internal colimits.

\subsection{Definition and characterisation of $\BB$-topoi}
\label{sec:definitionBTopoi}
In this section we introduce the notion of a \emph{$\BB$-topos} and prove several equivalent characterisations of this concept.

Recall from Proposition~\ref{prop:critfinlim} that a $\BB$-category $\I{C}$ admits finite limits if and only if for all $A\in\BB$ the $\infty$-category $\I{C}(A)$ admits finite limits and for each map $s\colon B\to A$ in $\BB$ the functor $s^\ast\colon\I{C}(A)\to\I{C}(B)$ preserves finite limits. Similarly, a functor $f\colon \I{C}\to\I{D}$ between such $\BB$-categories preserves finite limits precisely if it does so section-wise. We may now define:
\begin{definition}
	\label{def:BTopos}
	A large $\BB$-category $\I{X}$ is a \emph{$\BB$-topos} if it is presentable and satisfies descent. A functor $f^\ast\colon \I{X}\to\I{Y}$ between $\BB$-topoi is called an \emph{algebraic morphism} if $f$ is cocontinuous and preserves finite limits. A functor $f_\ast\colon \I{Y}\to\I{X}$ between $\BB$-topoi is called a \emph{geometric morphism} if $f_\ast$ admits a left adjoint $f^\ast$ that defines an algebraic morphism. 
	
	The large $\BB$-category $\ILTop_{\BB}$ of $\BB$-topoi is defined as the subcategory of $\ICat_{\BBB}$ that is spanned by the algebraic morphisms between $\Over{\BB}{A}$-topoi, for all $A\in\BB$. Dually, the large $\BB$-category $\IRTop_\BB$ of $\BB$-topoi is defined as the subcategory of $\ICat_{\BBB}$ that is spanned by the geometric morphisms between $\Over{\BB}{A}$-topoi, for all $A\in\BB$. We denote by $\LTop(\BB)$ and $\RTop(\BB)$, respectively, the underlying $\infty$-categories of global sections.
	
	If $\I{X}$ and $\I{Y}$ are $\BB$-topoi, we will denote by $\IFun^\alg(\I{X},\I{Y})$ the full subcategory of $\IFun(\I{X},\I{Y})$ that is spanned by the algebraic morphisms $\pi_A^\ast\I{X}\to\pi_A^\ast\I{Y}$ for each $A\in\BB$. We define the $\BB$-category $\IFun^\geom(\I{Y},\I{X})$ of geometric morphisms in the evident dual way.
\end{definition}

\begin{remark}
	The fact that both $\ILTop_{\BB}$ and $\IRTop_\BB$ are large and not very large follows from Remark~\ref{rem:SizePrL}.
\end{remark}

\begin{remark}
	\label{rem:BCTopoi}
	The subobject of $(\ICat_{\BBB})_1$ that is spanned by the algebraic morphisms between $\Over{\BB}{A}$-topoi (for each $A\in\BB$) is stable under composition and equivalences in the sense of Proposition~\ref{prop:classificationSubcategories}. Since moreover cocontinuity and the property that a functor preserves finite limits are local conditions~\cite[Remark~5.2.3]{MWColimits} and on account of Remark~\ref{rem:BToposLocal} below, we conclude that a map $A\to (\ICat_{\BBB})_1$ is contained in $(\ILTop_{\BB})_1$ if and only if it defines an algebraic morphism between $\Over{\BB}{A}$-topoi. In particular, if $\I{X}$ and $\I{Y}$ are $\Over{\BB}{A}$-topoi, the image of the monomorphism
	\begin{equation*}
		\label{eq:inclusionMappingLTop}\tag{$ \ast $}
		\map{\ILTop_{\BB}}(\I{X},\I{Y})\into\map{\ICat_{\BBB}}(\I{X},\I{Y})
	\end{equation*}
	is spanned by the algebraic morphisms. Moreover, the sheaf associated to $\ILTop_{\BB}$ is given by sending $A\in\BB$ to the subcategory $\LTop(\Over{\BB}{A})\into\Cat(\Over{\BBB}{A})$, and there is consequently a canonical equivalence $\pi_A^\ast\ILTop_{\BB}\simeq\ILTop_{\Over{\BB}{A}}$. Analogous observations can be made for the $\BB$-category $\IRTop_\BB$.
	
	By the same argument, we have a canonical equivalence $\pi_A^\ast\IFun^\alg(\I{X},\I{Y})\simeq\IFun[\Over{\BB}{A}]^\alg(\pi_A^\ast\I{X},\pi_A^\ast\I{Y})$ for all $\BB$-topoi $\I{X}$ and $\I{Y}$ and all $A\in\BB$. Furthermore, by using~\cite[Corollary~4.5.5]{MWColimits}, we deduce that the inclusion in (\ref{eq:inclusionMappingLTop}) is obtained by applying the core $\BB$-groupoid functor to the inclusion of $\IFun^\alg(\I{X},\I{Y})$ into $\IFun(\I{X},\I{Y})$. Again, analogous observations can be made for geometric morphisms.
\end{remark}

By restricting the equivalence $\ICat_{\BBB}^\Rad\simeq(\ICat_{\BBB}^\Lad)^\op$ from~\cite[Proposition~7.2.1]{MCocartesian}, one finds:
\begin{proposition}
	\label{prop:DualityLTopRTop}
	There is an equivalence $(\ILTop_{\BB})^\op\simeq\IRTop_\BB$ that acts as the identity on objects and that carries an algebraic morphism to its right adjoint.\qed
\end{proposition}

Let us denote by $\LTopEtS$ the subcategory of $\LTopS$ that is spanned by the \'etale algebraic morphisms (i.e.\ those that are of the form $\pi_U^\ast\colon \XX\to\Over{\XX}{U}$ for some $\infty$-topos $\XX$ and some $U\in\XX$). By~\cite[Theorem~6.3.5.13]{htt}, this $\infty$-category admits small limits, and the inclusion $\LTopEtS\into\LTopS$ preserves small limits. The main goal of this section is to prove the following characterisation of $\BB$-topoi:
\begin{theorem}
	\label{thm:characterisationBTopoi}
	For a large $\BB$-category $\I{X}$, the following are equivalent:
	\begin{enumerate}
		\item $\I{X}$ is a $\BB$-topos;
		\item $\I{X}$ satisfies the \emph{internal Giraud axioms}:
		\begin{enumerate}
			\item $\I{X}$ is presentable;
			\item $\I{X}$ has universal colimits;
			\item groupoid objects in $\I{X}$ are effective;
			\item $\Univ$-colimits in $\I{X}$ are disjoint.
		\end{enumerate}
		\item $\I{X}$ is $\Univ$-cocomplete and takes values in $\LTopEtS$;
		\item $\I{X}$ is a $\LTopEtS$-valued sheaf that preserves pushouts.
	\end{enumerate}
\end{theorem}

\begin{remark}
	It is crucial to include the condition that all $\Univ$-groupoidal colimits are disjoint into the internal Giraud axioms, instead of just all coproducts. As a concrete example, let $\kappa$ be an uncountable regular cardinal and let $\CC\into\CatS$ be the subcategory spanned by the $\kappa$-small $\infty$-categories and \emph{cocartesian} fibrations between them. Let us set $\BB=\PSh(\CC)$ and let $\I{X}\in\Cat(\BBB)$ be the large $\BB$-category that is determined by the presheaf $\PSh(-)\colon \CC^{\op}\to\CatSS$. Since $\I{X}$ takes values in $\LTopS$ and since cocartesian fibrations are \emph{smooth}~\cite[Proposition~4.1.2.15]{htt}, we deduce from Theorem~\ref{thm:characterisationPresentableCategories}, Remark~\ref{rem:BCGroupoidObjects} and Example~\ref{ex:disjointColimitsFibrewise} that $\I{X}$ is presentable, has effective groupoid objects and that coproducts in $\I{X}$ are disjoint. Moreover, again by using that cocartesian fibrations are smooth, one easily finds that $\I{X}$ has universal colimits. Yet, the $\BB$-category $\I{X}$ cannot be a $\BB$-topos since the transition functors are in general not \'etale.
\end{remark}

Before we prove Theorem~\ref{thm:characterisationBTopoi}, let us us first record the following immediate consequence:
\begin{corollary}
	\label{cor:universeTopos}
	The universe $\Univ$ is a $\BB$-topos.\qed
\end{corollary}
\begin{remark}
	\label{rem:BToposLocal}
	As another consequence of Theorem~\ref{thm:characterisationBTopoi}, a large $\BB$-category $\I{X}$ is a $\BB$-topos if and only if there is a cover $\bigsqcup_i A_i\onto 1$ in $\BB$ such that for all $i$ the large $\Over{\BB}{i}$-category $\pi_{A_i}^\ast\I{X}$ is a $\Over{\BB}{A_i}$-topos. In fact, this most easily follows from part~(3) of the theorem, together with the fact that $\Univ$-cocompleteness can be checked locally~\cite[Remark~5.2.3]{MWColimits}.
\end{remark}

The proof of Theorem~\ref{thm:characterisationBTopoi} requires the following lemma:
\begin{lemma}
	\label{lem:criterionLTopEtalePushoutBC}
	Let
	\begin{equation*}
		\begin{tikzcd}
			\XX\arrow[r, "h^\ast"] \arrow[d, "f^\ast"] & \ZZ\arrow[d, "g^\ast"]\\
			\YY\arrow[r, "k^\ast"] & \WW
		\end{tikzcd}
	\end{equation*}
	be a commutative square in $\LTopS$, and suppose that $h^\ast$ and $k^\ast$ are \'etale. Then the square is a pushout in $\LTopS$ if and only if the mate transformation $k_! g^\ast\to f^\ast h_!$ is an equivalence.
\end{lemma}
\begin{proof}
	As $h^\ast$ is \'etale, we may replace $\ZZ$ with $\Over{\XX}{U}$ and $h^\ast$ with $\pi_{U}^\ast$, where $U=h_!(1_{\ZZ})$. By using~\cite[Remark~6.3.5.8]{htt}, the pushout of $\pi_U^\ast$ along $f^\ast$ is given by the commutative diagram
	\begin{equation*}
		\begin{tikzcd}
			\XX\arrow[r, "\pi_U^\ast"] \arrow[d, "f^\ast"] & \Over{\XX}{U}\arrow[d, "\Over{f^\ast}{U}"]\\
			\YY\arrow[r, "\pi_{f^\ast(U)}^\ast"] & \Over{\YY}{f^\ast(U)}.
		\end{tikzcd}
	\end{equation*}
	It is immediate that the mate of this square is an equivalence, so it suffices to prove the converse. Since $k^\ast$ is \'etale, we may replace $k^\ast$ with $\pi_V^\ast\colon \YY\to\Over{\YY}{V}$, where $V=k_!(1_{\WW})$. By~\cite[Remark~6.3.5.7]{htt}, the induced map $\Over{\YY}{f^\ast(U)}\to\Over{\YY}{V}$ is uniquely determined by a morphism $V\to f^\ast(U)$ in $\YY$. Unwinding the definitions, this map is precisely the value of the mate transformation $(\pi_V)_!g^\ast\to f^\ast (\pi_U)_!$ at $1_{\Over{\XX}{U}}$ and therefore an equivalence. Hence the functor $\Over{\YY}{f^\ast(U)}\to\Over{\YY}{V}$ must be an equivalence as well, which finishes the proof.
\end{proof}
\begin{proof}[{Proof of Theorem~\ref{thm:characterisationBTopoi}}]
	Let $\I{X}$ be a $\BB$-topos. By combining Propositions~\ref{prop:universalColimitsFaithfulDescent} and~\ref{prop:disjointColimitsEffectiveDescent} with Corollary~\ref{cor:effectiveDescentEffectiveGroupoidObjects}, we find that $\I{X}$ satisfies the internal Giraud axioms, so that~(1) implies~(2). If $\I{X}$ satisfies the internal Giraud axioms, then $\I{X}$ being presentable implies that it is $\Univ$-cocomplete. Moreover, Examples~\ref{ex:universalColimitsFibrewise} and~\ref{ex:disjointColimitsFibrewise} together with Remark~\ref{rem:BCGroupoidObjects} imply that $\I{X}(A)$ satisfies the $\infty$-categorical Giraud axioms in the sense of~\cite{htt} for all $A\in\BB$, so that each $\I{X}(A)$ is an $\infty$-topos. Now by Propositions~\ref{prop:disjointColimitsEffectiveDescent} and~\ref{prop:universalColimitsFaithfulDescent}, the $\BB$-category $\I{X}$ has $\Univ$-descent, hence Corollary~\ref{cor:groupoidalDescentEtaleTransitionMaps} implies that for every map $s\colon B\to A$ in $\BB$ the functor $s_!$ is a right fibration. This implies that $s^\ast$ is an \'etale geometric morphism, hence~(3) holds. The equivalence between~(3) and~(4), on the other hand, is an immediate consequence of Lemma~\ref{lem:criterionLTopEtalePushoutBC}. Finally, if $\I{X}$ satisfies condition~(3), then $\I{X}$ is presentable (see Theorem~\ref{thm:characterisationPresentableCategories}), hence the claim follows from Corollary~\ref{cor:descentForKSClasses}.
\end{proof}

\subsection{Free $\BB$-topoi}
\label{sec:freeTopoi}
The goal of this section is to construct a partial left adjoint to the forgetful functor $\ILTop_{\BB}\into\ICat_{\BBB}$ that is defined on the full subcategory $\ICat_{\BB}\into\ICat_{\BBB}$ and that carries a $\BB$-category $\I{C}$ to the associated \emph{free} $\BB$-topos $\Free{\I{C}}$. To that end, first note that if
$\ICat_{\BB}^\lex\into\ICat_{\BB}$ denotes the subcategory spanned by the left exact (i.e. $\IFin_{\Over{\BB}{A}}$-continuous) functors between $\Over{\BB}{A}$-categories with finite limits for all $A\in\BB$, then the dual of~\cite[Corollary~7.1.15]{MWColimits} implies that the inclusion admits a left adjoint $(-)^\lex\colon\ICat_{\BB}\to\ICat_{\BB}^\lex$ that carries a $\BB$-category $\I{C}$ to its free $\IFin_{\BB}$-completion $\I{C}^\lex$. Moreover, the same result implies that we have a functor $\IPSh(-)\colon \ICat_{\BB}\to\ILPr_{\BB}$ that is obtained by restricting the free cocompletion functor $\ICat_{\BBB}\to\ICat_{\BBB}^{\cc}$ in the appropriate way. By combining these two constructions, we thus end up with a well-defined functor $\Free{-}=\IPSh((-)^\lex)\colon\ICat_{\BB}\to\ILPr_{\BB}$. Our goal is to show:

\begin{proposition}
	\label{prop:UMPFreeTopoi}
	For any $\BB$-category $ \I{C} $, the large $\BB$-category $\Free{\I{C}}$ is a $\BB$-topos. Moreover, if $ \I{X} $ is another $ \BB $-topos, precomposition with the canonical map $ \I{C} \to \Free{\I{C}} $ induces an equivalence
	\[
	\IFun^\alg(\Free{\I{C}},\I{X}) \simeq \IFun(\I{C},\I{X})
	\]
	of $\BB$-categories.
\end{proposition}
The proof of Proposition~\ref{prop:UMPFreeTopoi} requires a few preparations and will be given at the end of this section. For now, let us record a few consequences of this result.
\begin{corollary}
	\label{cor:UMPFreeTopoi}
	The functor $\Free{-}$ takes values in $\ILTop_{\BB}$ and fits into an equivalence
	\begin{equation*}
		\map{\ILTop_{\BB}}(\Free{-},-)\simeq\map{\ICat_{\BBB}}(-,-)
	\end{equation*}
	of bifunctors $\ICat_{\BB}^\op\times\ILTop_{\BB}\to\Univ[\BBB]$.
\end{corollary}
\begin{proof}
	Note that if $A\in\BB$ is an arbitrary object, we deduce from~\cite[Proposition~7.1.11]{MWColimits} that the base change of the canonical map $\I{C}\to\Free{\I{C}}$ along $\pi_A^\ast$ can be identified with the canonical map $\pi_A^\ast\I{C}\to\Free[\Over{\BB}{A}]{\pi_A^\ast\I{C}}$. Thus, in light of Remark~\ref{rem:BCTopoi}, the result is an immediate consequence of Proposition~\ref{prop:UMPFreeTopoi}.
\end{proof}

Corollary~\ref{cor:UMPFreeTopoi} already implies the existence of certain colimits in $\ILTop_{\BB}$:
\begin{corollary}
	\label{cor:ColimitsFreeTopoi}
	For any diagram $d\colon \I{I}\to\ICat_{\BB}$, the induced cocone $\Free{d(-)}\to\Free{\colim d}$ is a colimit cocone in $\ILTop_{\BB}$. 
\end{corollary}
\begin{proof}
	Combine Proposition~\ref{prop:UMPFreeTopoi} with~\cite[Proposition~4.4.8]{MWColimits}.
\end{proof}
By combining Corollary~\ref{cor:ColimitsFreeTopoi} with the evident equivalence $1\simeq\varnothing^\lex$, we in particular obtain:
\begin{corollary}
	\label{cor:UniverseInitial}
	The universe $\Univ$ defines an initial object in $\ILTop_{\BB}$.\qed
\end{corollary}
In light of Corollary~\ref{cor:UniverseInitial}, we may now define:
\begin{definition}
	\label{def:globalSectionsConstantSheaf}
	Let $\I{X}$ be a $\BB$-topos. The unique algebraic morphism $\const_{\I{X}}\colon\Univ\to\I{X}$ is referred to as the \emph{constant sheaf functor}, and its right adjoint $\Gamma_{\I{X}}\colon\I{X}\to\Univ$ is called the \emph{global sections functor}.
\end{definition}
\begin{remark}
	\label{rem:GlobalSectionsExplicitly}
	If $\I{X}$ is a $\BB$-topos, then the global sections functor $\Gamma_{\I{X}}$ is equivalent to $\map{\I{X}}(1_{\I{X}},-)$, where $1_{\I{X}}$ denotes the final object in $\I{X}$. In fact, since the unique algebraic morphism $\const_{\I{X}}\colon\Univ[\BB]\to\I{X}$ is left exact and since $\map{\Univ[\BB]}(1_{\Univ},-)\simeq \id_{\Univ}$ by~\cite[Proposition~4.6.3]{MYoneda}, this follows immediately from the adjunction $\const_{\BB}\dashv \Gamma_{\BB}$.
\end{remark}

We now turn to the proof of Proposition~\ref{prop:UMPFreeTopoi}. As a first step, we need to establish that presheaf $\BB$-categories are $\BB$-topoi:
\begin{proposition}
	\label{prop:PresheafCategoriesTopoi}
	For every $\BB$-category $\I{C}$, the large $\BB$-category $\IPSh(\I{C})$ is a $\BB$-topos.
\end{proposition}
\begin{proof}
	Since $\IPSh(\I{C})$ is presentable, we only need to show that it satisfies descent. Let us first show that $\IPSh(\I{C})$ has universal colimits. Let therefore $f\colon F\to G$ be an arbitrary map of presheaves on $\I{C}$ in context $A\in\BB$. By~\cite[Lemma~4.7.14]{MYoneda}, we may replace $\BB$ with $\Over{\BB}{A}$, so that we can assume that $A\simeq 1$. By~\cite[Lemma~6.1.5]{MWColimits} there are equivalences $\Over{\IPSh(\I{C})}{F}\simeq\IPSh(\Over{\I{C}}{F})$ and $\Over{\IPSh(\I{C})}{G}\simeq\IPSh(\Over{\I{C}}{G})$ with respect to which the functor $\IPSh(\Over{\I{C}}{F})\to\IPSh(\Over{\I{C}}{G})$ that corresponds to $f_!$ carries the final presheaf on $\Over{\I{C}}{F}$ to the presheaf that classifies the right fibration $f_!\colon\Over{\I{C}}{F}\to\Over{\I{C}}{G}$. As the functor $\IPSh(\Over{\I{C}}{F})\to\IPSh(\Over{\I{C}}{G})$ is a morphism of right fibrations over $\IPSh(\I{C})$, this map is uniquely specified by the image of the final object. We thus conclude that this functor must be equivalent to the functor of left Kan extension along $f_!\colon\Over{\I{C}}{F}\to\Over{\I{C}}{G}$. Its right adjoint is simply given by precomposition with $f_!$, which defines a cocontinuous functor. Hence $f^\ast\colon \Over{\IPSh(\I{C})}{G}\to\Over{\IPSh(\I{C})}{F}$ must be cocontinuous as well, and we conclude that $\IPSh(\I{C})$ has universal colimits.
	
	To conclude the proof, we need to show that $\IPSh(\I{C})$ has effective descent. By Proposition~\ref{prop:characterisationEffectiveDescent}, this is equivalent to the condition that for every $A\in \BB$, every small $\Over{\BB}{A}$-category $\I{I}$ and every cartesian map $d^\prime\to d$ in $\IFun[\Over{\BB}{A}](\I{I},\pi_A^\ast\IPSh(\I{C}))$, the naturality square
	\begin{equation*}
		\begin{tikzcd}
			{d}^\prime\arrow[d]\arrow[r, "\eta"] & \diag(\colim d^\prime)\arrow[d]\\
			{d}\arrow[r, "\eta"] & \diag(\colim d)
		\end{tikzcd}
	\end{equation*}
	is a pullback. Upon replacing $\BB$ by $\Over{\BB}{A}$, we may assume without loss of generality $A\simeq 1$. Moreover, since limits and colimits in functor $\BB$-categories are detected object-wise~\cite[Proposition~4.3.2]{MWColimits}, we can reduce to $\I{C}\simeq 1$. In this case, the result follows from Corollary~\ref{cor:universeTopos}.
\end{proof}
Next, we need to establish an internal analogue of the well-known statement that left exact functors with values in an $\infty$-topos are equivalently \emph{flat} functors~\cite[Proposition~6.1.5.2]{htt}. The key ingredient to this result is the following lemma:
\begin{lemma}
	\label{lem:YonedaExtensionPullbackPreservation}
	Let $\I{C}$ be a $\BB$-category, let $\I{X}$ be a $\BB$-topos and let $f\colon \I{C}\to\I{X}$ be a functor. Suppose that the Yoneda extension $h_!(f)\colon \IPSh(\I{C})\to\I{X}$ of this functor preserves the limit of every cospan in $\IPSh(\I{C})$ (in arbitrary context $A\in\BB$) that is contained in the essential image of the Yoneda embedding $h\colon\I{C}\into\IPSh(\I{C})$. Then $h_!(f)$ preserves pullbacks.
\end{lemma}
\begin{proof}
	Suppose that
	\begin{equation*}
		\begin{tikzcd}
			Q\arrow[r]\arrow[d] & P\arrow[d]\\
			G\arrow[r] & F  
		\end{tikzcd}
	\end{equation*}
	is a pullback square in $\IPSh(\I{C})$. We need to show that the image of this square along $(h_{\I{C}})_!(f)$ is a pullback in $\I{X}$. 	By combining~\cite[Lemma~4.7.14]{MYoneda} with~\cite[Remark~6.3.2]{MWColimits}, we may assume without loss of generality that the above square is in context $1\in\BB$.
	
	Let us first show that the claim is true whenever $F$ is representable by an object $c\colon 1\to\I{C}$. In this case,~\cite[Lemma~6.1.5]{MWColimits} implies that there is an equivalence $\Over{\IPSh(\I{C})}{h(c)}\simeq\IPSh(\Over{\I{C}}{c})$ with respect to which the composition $(h_! f)(\pi_{h(c)})_!$ can be identified with the left Kan extension of $f(\pi_c)_!$ along the Yoneda embeddding $\Over{\I{C}}{c}\into\IPSh(\Over{\I{C}}{c})$. Therefore, by replacing $\I{C}$ with $\Over{\I{C}}{c}$, one can assume that $F\simeq 1_{\IPSh(\I{C})}$. Now the product functor $G\times -\colon \IPSh(\I{C})\to\IPSh(\I{C})$ being cocontinuous (by Proposition~\ref{prop:PresheafCategoriesTopoi})  implies that the canonical map $h_! f(G\times -)\to h_! f (G)\times h_!f(-)$ is a morphism between cocontinuous functors. On account of the universal property of presheaf $\BB$-categories, this means that we may further reduce to the case where $H$ is representable. By the same argument, the presheaf $G$ can also be assumed to be representable. In this case, the claim follows from the assumption on $h_!(f)$.
	
	We now turn to the general case. By~\cite[Proposition~6.1.1]{MWColimits}, there is a diagram $d\colon \I{I}\to \IPSh(\I{C})$ such that $F\simeq \colim d$ and such that $d$ takes values in $\I{C}\into\IPSh(\I{C})$. Let us write $\overline{d}$ for the associated colimit cocone. In light of the equivalence $\Over{\IPSh(\I{C})}{F}\simeq \Over{\IFun(\I{I},\IPSh(\I{C}))}{\overline{d}}^{\cart}$ from Remark~\ref{rem:cartesianTransformationsCoconesExplicitly} and by identifying the above pullback square with a diagram in $\Over{\IPSh(\I{C})}{F}$, we obtain a pullback diagram
	\begin{equation*}
		\begin{tikzcd}
			\overline{q}\arrow[r]\arrow[d] & \overline{p}\arrow[d]\\
			\overline{g}\arrow[r] & \overline{d}
		\end{tikzcd}
	\end{equation*}
	in $\Over{\IFun(\I{I}^\triangleright,\IPSh(\I{C}))}{\overline{d}}^{\cart}$. By the above and the fact that limits in functor $\BB$-categories can be computed object-wise~\cite[Proposition~4.3.2]{MWColimits}, the composition
	\begin{equation*}
		\Over{\IFun(\I{I}^\triangleright,\IPSh(\I{C}))}{\overline{d}}^{\cart}\to\Over{\IFun(\I{I},\IPSh(\I{C}))}{{d}}^{\cart}\xrightarrow{(h_!f)_\ast}\Over{\IFun(\I{I},\I{X})}{(h_!f)_\ast d}
	\end{equation*}
	carries the above pullback diagram of cocones to a pullback and therefore in particular to a diagram in $\Over{\IFun(\I{I},\I{X})}{(h_!f)_\ast d}^{\cart}$. By using descent in $\I{X}$ and in $\IPSh(\I{C})$ (cf.~Proposition~\ref{prop:PresheafCategoriesTopoi}) together with the fact that $h_!(f)$ is cocontinuous, this implies that the functor $(h_!f)_\ast\colon \Over{\IFun(\I{I}^\triangleright,\IPSh(\I{C}))}{\overline{d}}^{\cart}\to \Over{\IFun(\I{I}^\triangleright,\I{X})}{\overline{d}}$ preserves the above pullback. Upon evaluating the latter at the cone point $\infty\colon 1\to\I{I}^\triangleright$, we recover the image of the original pullback square along $h_!(f)$, hence the claim follows.
\end{proof}

\begin{proposition}
	\label{prop:generalFlatness}
	Let $\I{C}$ be a $\BB$-category with finite limits, and let $\I{X}$ be a $\BB$-topos. Then a functor $f\colon \I{C}\to\I{X}$ preserves finite limits if and only if its left Kan extension $h_!(f)\colon\IPSh(\I{C})\to\I{X}$ preserves finite limits.
\end{proposition}
\begin{proof}
	Since the Yoneda embedding $h\colon\I{C}\into\IPSh(\I{C})$ preserves finite limits, it is clear that the condition is sufficient. Conversely, suppose that $f$ preserves finite limits. Since the final object in $\IPSh(\I{C})$ is contained in $\I{C}$, it is clear that $h_!(f)$ preserves final objects. We therefore only need to show that this functor also preserves pullbacks, which is an immediate consequence of Lemma~\ref{lem:YonedaExtensionPullbackPreservation}.
\end{proof}
By combining Proposition~\ref{prop:generalFlatness} with the universal property of presheaf $\BB$-categories, Remark~\ref{rem:BCTopoi} and~\cite[Remark~5.3.4]{MWColimits}, we now conclude:
\begin{corollary}
	\label{cor:internalDiaconescu}
	For any $\BB$-category $\I{C}$ with finite limits and any $\BB$-topos $\I{X}$, the functor of left Kan extension along the Yoneda embedding $\I{C}\into\IPSh(\I{C})$ gives rise to an equivalence
	\begin{equation*}
		\IFun^\lex(\I{C},\I{X})\simeq\IFun^\alg(\IPSh(\I{C}),\I{X}),
	\end{equation*}
	where $\IFun^\lex(\I{C},\I{X})$ is the full subcategory of $\IFun(\I{C},\I{X})$ that is spanned by the left exact functors in arbitrary context.\qed
\end{corollary}

\begin{proof}[{Proof of Proposition~\ref{prop:UMPFreeTopoi}}]
	Combine Proposition~\ref{prop:PresheafCategoriesTopoi} with Corollary~\ref{cor:internalDiaconescu} and the universal property of free $\IFin_{\BB}$-completion, cf.~\cite[Theorem~7.1.13]{MWColimits}.
\end{proof}

\subsection{Presentations of $\BB$-topoi}
\label{sec:presentationTopoi}
Recall from Definition~\ref{def:accessibleLocalisation} that if $\I{C}$ is a $\BB$-category, a Bousfield localisation $L\colon\IPSh(\I{C})\to\I{D}$ is said to be \emph{accessible} if the inclusion $\I{D}\into\IPSh(\I{C})$ is $\IFilt_{\I{U}}$-cocontinuous for some choice of \emph{sound doctrine} $\I{U}$ (see Definitions~\ref{def:doctrine} and~\ref{def:soundness}). We will say that the localisation is \emph{left exact} if $L$ preserves finite limits. The main goal of this section is to prove the following characterisation of $\BB$-topoi:
\begin{theorem}
	\label{thm:presentationTopoi}
	A large $\BB$-category $\I{X}$ is a $\BB$-topos if and only if there is a $\BB$-category $\I{C}$ such that $\I{X}$ arises as a left exact and accessible localisation of $\IPSh(\I{C})$.
\end{theorem}

The proof of Theorem~\ref{thm:presentationTopoi} relies on the following two lemmas:
\begin{lemma}
	\label{lem:lexLocalisationPushoutLTop}
	Suppose that
	\begin{equation*}
		\begin{tikzcd}
			\XX\arrow[r, "\pi_U^\ast"] \arrow[d, "L"] & \Over{\XX}{U}\arrow[d, "L^\prime"]\\
			\YY\arrow[r, "h^\ast"] & \ZZ
		\end{tikzcd}
	\end{equation*}
	is a commutative square in $\LTopS$ in which $L$ and $L^\prime$ are Bousfield localisations. Suppose furthermore that $h^\ast$ admits a left adjoint $h_!$ and that the mate transformation $\phi\colon h_! L^\prime\to L(\pi_U)_!$ is an equivalence. Then $h^\ast$ is \'etale.
\end{lemma}
\begin{proof}
	We would like to apply~\cite[Proposition~6.3.5.11]{htt}, which says that the functor $h^\ast$ is \'etale precisely if $h_!$ is conservative and if for every map $f\colon W\to V$ in $\YY$ and every object $P\in\Over{\ZZ}{h^\ast(V)}$, the canonical map $\alpha\colon h_!(h^\ast(W)\times_{h^\ast (V)} P)\to W\times_V h_!(P)$ is an equivalence.
	
	Let us begin by showing that $h_!$ is conservative. To that end, note that if $f\colon V\to W$ is a map in $\Over{\XX}{U}$ such that $L(\pi_U)_!(f)$ is an equivalence, then $L^\prime(f)$ is an equivalence. In fact, since the adjunction unit of $(\pi_U)_!\dashv\pi_U^\ast$ exhibits $f$ as a pullback of $\pi_U^\ast(\pi_U)_!(f)$, the localisation functor $L^\prime$ being left exact implies that $L^\prime(f)$ is a pullback of $L^\prime\pi_U^\ast(\pi_U)_!(f)\simeq h^\ast L(\pi_U)_!(f)$. Since the latter is an equivalence, the claim follows. Applying this observation to a map $f$ that is contained in $\ZZ$ and using the assumption that the mate transformation $\phi\colon h_! L^\prime\to L(\pi_U)_!$ is an equivalence, we deduce that $h_!$ is indeed conservative.

	To conclude the proof, we show that the map $\alpha$ is an equivalence. From the map $f\colon W\to V$ in $\YY$ we obtain a commutative diagram
	\begin{equation*}
		\begin{tikzcd}[column sep={between origins,6em}]
			& \Over{\XX}{V}\arrow[rr, "\Over{(\pi_U^\ast)}{V}"]\arrow[dd, "f^\ast", near end]\arrow[dl, "\Over{L}{V}"] && \Over{\XX}{U\times V}\arrow[dd, "(\pi_U^\ast f)^\ast", near end]\arrow[dl, "\Over{L^\prime}{\pi_U^\ast(V)}"]\\
			\Over{\YY}{V}\arrow[rr, "\Over{h^\ast}{V}", crossing over, near end] \arrow[dd, "f^\ast", near end]&& \Over{\ZZ}{h^\ast(V)} & \\
			& \Over{\XX}{W}\arrow[rr, "\Over{(\pi_U^\ast)}{W}", near start]\arrow[dl, "\Over{L}{W}"] && \Over{\XX}{U\times W}\arrow[dl, "\Over{L^\prime}{\pi_U^\ast(W)}"]\\
			\Over{\YY}{W}\arrow[rr, "\Over{h^\ast}{W}"] && \Over{\ZZ}{h^\ast(W)}\arrow[from=uu, "(h^\ast f)^\ast", crossing over, near end] &
		\end{tikzcd}
	\end{equation*}
	in $\LTopS$ in which all of the four maps pointing to the right admit a left adjoint. Note that $\alpha$ being an equivalence for all $P\in\Over{\ZZ}{h^\ast(V)}$ precisely means that the front square is left adjointable (i.e.\ has an invertible mate transformation). Now since the mate $\phi\colon h_! i^\prime\to i(\pi_U)_!$ is by assumption an equivalence, it follows that both the top and the bottom square in the above diagram are left adjointable. Since $\pi_U^\ast$ is an \'etale algebraic morphism, the back square is left adjointable as well. Therefore, by combining the functoriality of the mate construction with the fact that the four maps in the above diagram pointing to the front are localisation functors and thus in particular essentially surjective, we conclude that the front square must be left adjointable as well, as desired.
\end{proof}

\begin{lemma}
	\label{lem:presentableCategoryUCompactLex}
	Let $\I{D}$ be a presentable $\BB$-category. Then there exists a sound doctrine $\I{U}$ such that $\I{D}$ is $\I{U}$-accessible and $\I{D}^{\cpt{\I{U}}}$ is closed under finite limits in $\I{D}$.
\end{lemma}
\begin{proof}
	Since $\I{D}$ is presentable, there exists a $\BB$-category $\I{C}$ and a sound doctrine $\I{U}$ such that $\I{D}$ arises as a $\I{U}$-accessible Bousfield localisation of $\IPSh(\I{C})$ (cf. Theorem~\ref{thm:characterisationPresentableCategories}). In particular, for every sound doctrine $\I{V}$ that contains $\I{U}$, the $\BB$-category $\I{D}$ is $\I{V}$-accessible (Corollary~\ref{cor:AccessibleLocalisation}). Therefore, for any cardinal $\kappa$ we can always find a $\BB$-regular cardinal $\tau\geq \kappa$ such that $\I{D}$ is $\ICat_{\BB}^\tau$-accessible. By Remark~\ref{rem:enoughBRegularCardinals}, we can always choose $\tau$ such that $\tau\gg\kappa$. Now $\I{D}$ being presentable implies that $\I{D}$ is section-wise accessible (Theorem~\ref{thm:characterisationPresentableCategories}). Therefore, if $\GG\into\BB$ is a small generating subcategory, we may find a regular cardinal $\kappa$ such that $\I{D}(G)$ is $\kappa$-accessible for all $G\in\GG$. Let us choose a $\BB$-regular cardinal $\tau\gg \kappa$ such that 
	\begin{enumerate}
		\item $\GG$ is contained in $\BB^{\cpt{\tau}}$;
		\item $\I{D}$ is $\ICat_{\BB}^\tau$-accessible;
		\item $\I{D}(G)^{\cpt{\kappa}}$ is $\tau$-small for all $G\in\GG$.
	\end{enumerate}
	Then~\cite[Proposition~5.4.7.4]{htt} implies that the inclusion $\I{D}(G)^{\cpt{\tau}}\into\I{D}(G)$ is closed under finite limits for all $G\in\GG$. Recall from~\cite[Corollary~4.6.8]{MYoneda} that for every object $d\colon A\to\I{D}$ the mapping functor $\map{\I{D}(A)}(d,-)$ can be identified with the composition
	\begin{equation*}
		\I{D}(A)\xrightarrow{\map{\I{D}}(d,-)(A)}\Over{\BB}{A}\xrightarrow{\Gamma_{\Over{\BB}{A}}}\SS.
	\end{equation*}
	By combining this observation with Proposition~\ref{prop:CocompleteKappaCocontinuousInternalExternal} and the fact that $\BB$ is generated by $\GG$, we find that for any $G\in\GG$ an object $d\colon G\to \I{D}$ is contained in $\I{D}^{\cpt{\ICat_{\BB}^\tau}}$ if and only if for every $H\in\GG$ and every map $s\colon H\to G$ the object $s^\ast(d)\in\I{D}(H)$ is contained in $\I{D}(H)^{\cpt{\tau}}$. Since $s^\ast$ commutes with limits, this implies that the inclusion $\I{D}^{\cpt{\ICat_{\BB}^\tau}}\into\I{D}$ is closed under finite limits.
\end{proof}

\begin{proof}[{Proof of Theorem~\ref{thm:presentationTopoi}}]
	Suppose that $\I{X}$ is a left exact and $\I{U}$-accessible localisation of $\IPSh(\I{C})$, and let us show that $\I{X}$ is a $\BB$-topos. We would like to apply Theorem~\ref{thm:characterisationBTopoi}. First, note that by choosing a $\BB$-regular cardinal $\kappa$ such that $\I{U}\into\ICat_{\BB}^\kappa$, we may assume that $\I{X}$ is a $\ICat_{\BB}^\kappa$-accessible Bousfield localisation of $\IPSh(\I{C})$. Therefore, for every $A\in\BB$ the $\infty$-category $\I{X}(A)$ is a $\kappa$-accessible and left exact Bousfield localisation of $\IPSh(\I{C})(A)$, and since the latter is an $\infty$-topos by Proposition~\ref{prop:PresheafCategoriesTopoi}, it follows that $\I{X}(A)$ is an $\infty$-topos as well. Moreover, if $s\colon B\to A$ is a map in $\BB$, the fact that $\I{X}$ is a presentable $\BB$-category (see Theorem~\ref{thm:characterisationPresentableCategories}) implies that $s^\ast\colon \I{X}(A)\to\I{X}(B)$ is continuous and cocontinuous and therefore in particular an algebraic morphism that admits a left adjoint $s_!\colon \I{X}(B)\to\I{X}(A)$. We are therefore in the situation of Lemma~\ref{lem:lexLocalisationPushoutLTop} and may thus conclude that $s^\ast$ is an \'etale algebraic morphism. Theorem~\ref{thm:characterisationBTopoi} thus implies that $\I{X}$ is a $\BB$-topos.
	
	Conversely, suppose that $\I{X}$ is a $\BB$-topos. Then $\I{X}$ is presentable, hence Lemma~\ref{lem:presentableCategoryUCompactLex} implies that there exists a sound doctrine $\I{U}$ such that $\I{X}$ is $\I{U}$-accessible and $\I{X}^{\cpt{\I{U}}}$ is closed under finite limits in $\I{X}$. Then we may identify $\I{X}\simeq\IInd^{\I{U}}(\I{X}^{\cpt{\I{U}}})$, and since $\I{X}$ is cocomplete the induced inclusion $\I{X}\into\IPSh(\I{X}^{\cpt{\I{U}}})$ admits a left adjoint $L\colon\IPSh(\I{X}^{\cpt{\I{U}}}) \to \I{X}$ which is obtained as the left Kan extension of the inclusion $\I{X}^{\cpt{\I{U}}}\into\I{X}$ (see~\cite[Corollary~7.1.14]{MWColimits}). By Proposition~\ref{prop:generalFlatness}, the functor $L$ is left exact, hence the claim follows.
\end{proof}

\begin{corollary}
	\label{cor:tensoringTopoi}
	For any $\BB$-topos $\I{X}$ and any $\BB$-category $\I{D}$, the functor $\BB$-category $\IFun(\I{D},\I{X})$ is again a $\BB$-topos.
\end{corollary}
\begin{proof}
	Choose a left exact and accessible Bousfield localisation $L\colon \IPSh(\I{C})\to\I{X}$. Then the postcomposition functor $L_\ast\colon \IPSh(\I{C}\times\I{D}^{\op})\to\IFun(\I{D},\I{X})$ is again an accessible and left exact Bousfield localisation, hence the claim follows.
\end{proof}

\begin{corollary}
	\label{cor:freeGenerationTopoi}
	A large $\BB$-category $\I{X}$ is a $\BB$-topos if and only if there is a $\BB$-category $\I{C}$ and a left exact and accessible Bousfield localisation $L\colon\Free{\I{C}}\to\I{X}$.
\end{corollary}
\begin{proof}
	By Theorem~\ref{thm:presentationTopoi}, it suffices to show that every presheaf $\BB$-topos arises as a left exact and accessible Bousfield localisation of a free $\BB$-topos. But if $\I{C}$ is a $\BB$-category, the fact that $i\colon \I{C}\into\I{C}^\lex$ is fully faithful implies that $i_\ast\colon \IPSh(\I{C})\into\Free{\I{C}}$ is fully faithful as well (see the dual of~\cite[Theorem~6.3.5]{MWColimits}), hence $i^\ast$ is a left exact and accessible Bousfield localisation.
\end{proof}

\begin{corollary}
	\label{cor:freePresentationTopoi}
	Any $ \BB $-topos $\I{X}$ is a pushout of free $\BB$-topoi.
\end{corollary}
\begin{proof}
	By Corollary~\ref{cor:freeGenerationTopoi}, we may choose a small $ \BB $-category $ \I{C} $ and a left exact and accessible Bousfield localisation $L \colon \Free{\I{C}} \to \I{X} $.
	We can therefore find a small subcategory $ \I{W} \into\Free{\I{C}} $ such that $ L $ induces an equivalence $\ILoc_{\I{W}}(\Free{\I{C}})\simeq \I{X}$ (see Theorem~\ref{thm:characterisationPresentableCategories}).
	Note that a functor $\I{X}\to\I{Y}$ between $\BB$-topoi is an algebraic morphism if and only if its precomposition with $L$ is one (this is easily deduced from Remark~\ref{rem:BCTopoi} and the explicit computation of colimits in a Bousfield localisation, cf.~\cite[Proposition~4.2.12]{MWColimits}). Therefore, by combining Corollary~\ref{cor:UniversalPropertyLocalObjects} with Proposition~\ref{prop:UMPFreeTopoi}, we deduce that the induced square
	\[\begin{tikzcd}
		\Free{\I{W}} & {\Free{\I{C}}} \\
		{\Free{\I{W}^\gp}} & {\I{X}}
		\arrow[from=1-1, to=2-1]
		\arrow[from=1-1, to=1-2]
		\arrow[from=2-1, to=2-2]
		\arrow[from=1-2, to=2-2]
	\end{tikzcd}\]
	is a pushout in $ \LTop(\BB) $.
\end{proof}

\subsection{The $\ICat_{\BBB}$-enrichment of $\ILTop_{\BB}$}
\label{sec:CatEnrichementTopoi}
Recall from~\cite[Proposition~4.5.3]{MWColimits} that $\ICat_{\BBB}$ is \emph{cartesian closed}, i.e.\ that forming functor $\BB$-categories defines a bifunctor $\IFun(-,-)\colon \ICat_{\BBB}^{\op}\times\ICat_{\BBB}\to\ICat_{\BBB}$ and therefore in particular a bifunctor $(\ILTop_{\BB})^\op\times\ILTop_{\BB}\to\ICat_{\BBB}$. Let $p\colon \I{P}\to(\ILTop_{\BB})^\op\times\ILTop_{\BB}$ be the unstraightening of the latter (in the sense of~\cite{MCocartesian}). Explicitly, an object $A\to \I{P}$ is given by a functor $\pi_A^\ast\I{X}\to\pi_A^\ast\I{Y}$ between $\Over{\BB}{A}$-topoi. 
 Let $\I{Q}\into\I{P}$ be the full subcategory that is spanned by those objects that correspond to algebraic morphisms. By Lemma~\ref{lem:subCocartesianFibration} below, the induced functor $q\colon \I{Q}\to(\ILTop_{\BB})^\op\times\ILTop_{\BB}$ is a cocartesian fibration as well and therefore classified by a bifunctor $\IFun^\alg(-,-)\colon (\ILTop_{\BB})^\op\times\ILTop_{\BB}\to\ICat_{\BBB}$.
 
 \begin{lemma}
 	\label{lem:subCocartesianFibration}
 	Let $p\colon\I{P}\to\I{C}$ be a cocartesian fibration of $\BB$-categories. Let $\I{Q}\into\I{P}$ be a full subcategory such that for each map $f\colon c\to d$ in $\I{C}$ in context $A\in\BB$ the induced functor $f_!\colon\I{P}\vert_c\to\I{P}\vert_d$ restricts to a functor $\I{Q}\vert_c\to\I{Q}\vert_d$. Then the induced functor $q\colon \I{Q}\to\I{C}$ is a cocartesian fibration as well, and the inclusion $\I{Q}\into\I{P}$ is a cocartesian functor.
 \end{lemma}
 \begin{proof}
 	Using~\cite[Proposition~3.2.5]{MCocartesian}, it will be enough to show that for any cocartesian lift $\phi\colon x\to y$ of $f$ in $\I{P}$ in which $x$ is contained in $\I{Q}\vert_c$, the object $y$ is contained in $\I{Q}\vert_d$. But this immediately follows from the assumptions, using~\cite[Remark~6.5.4]{MCocartesian}.
 \end{proof}
 
 \begin{definition}
 	We define the \emph{functor of $\BB$-points}  as the functor $\IPt_{\BB}=\IFun^\alg(-,\Univ)\colon \IRTop_{\BB}\to\ICat_{\BBB}$.
 \end{definition}

Recall that by Corollary~\ref{cor:tensoringTopoi}, if $\I{C}$ is a $\BB$-category and $\I{X}$ is a $\BB$-topos, then $\I{X}^{\I{C}}=\IFun(\I{C},\I{X})$ is a $\BB$-topos as well. Moreover, as precomposition and postcomposition preserves all limits and colimits, the bifunctor $\IFun(-,-)$ restricts to a bifunctor $(-)^{(-)}\colon \ICat_{\BB}^\op\times\ILTop_{\BB}\to\ILTop_{\BB}$ which we refer to as the \emph{powering} of $\ILTop_{\BB}$ over $\ICat_{\BB}$. This terminology is justified by the following proposition:
\begin{proposition}
	\label{prop:cotensoringLTop}
	The powering bifunctor $(-)^{(-)}$ fits into an equivalence
	\begin{equation*}
		\map{\ILTop_{\BB}}(-, (-)^{(-)})\simeq\map{\ICat_{\BBB}}^\alg(-, \IFun(-,-)).
	\end{equation*}
\end{proposition}
\begin{proof}
	If $\I{C}$ is a $\BB$-category and $\I{X}$ and $\I{Y}$ are $\BB$-topoi, then Lemma~\ref{lem:limitpreservationswap} and its dual imply that a functor $\I{X}\to\I{Y}^{\I{C}}$ defines an algebraic morphism if and only if the transpose functor $\I{C}\to\IFun(\I{X},\I{Y})$ takes values in $\IFun^\alg(\I{X},\I{Y})$.  By replacing $\BB$ by $\Over{\BB}{A}$ (which is made possible by Remark~\ref{rem:BCTopoi} and~\cite[Lemma~4.2.3]{MYoneda}) , one obtains that the same is true for any object $A\to \ICat_{\BB}^{\op}\times(\ILTop_{\BB})^\op\times\ILTop_{\BB}$. Hence, the equivalence
	\begin{equation*}
		\map{\ICat_{\BBB}}(i(-),\IFun(-,-))\simeq \map{\ICat_{\BBB}}(-,\IFun(i(-),-))
	\end{equation*}
	of functors $\ICat_{\BB}^{\op}\times\ICat_{\BBB}^{\op}\times\ICat_{\BBB}\to\Univ[\BBB]$ (where $i\colon \ICat_{\BB}\into\ICat_{\BBB}$ is the inclusion) restricts in the desired way.
\end{proof}

\begin{corollary}
	\label{cor:UMPPoints}
	The functor of $\BB$-points $\IPt$ is a partial right adjoint of the functor 
	\begin{equation*}
		\Univ^{(-)}\colon \ICat_{\BB}\to \IRTop_{\BB},
	\end{equation*}
	in the sense that there is an equivalence
	\begin{equation*}
		\map{\IRTop_{\BB}}(\Univ^{(-)}, -)\simeq\map{\ICat_{\BBB}}(-, \IPt(-))
	\end{equation*}
	of functors $\ICat_{\BB}^{\op}\times \IRTop_{\BB}\to\Univ[\BBB]$.\qed
\end{corollary}

In light of Corollary~\ref{cor:UMPPoints}, it is reasonable to define:
\begin{definition}
	\label{def:discreteBTopoi}
	If $\I{C}$ is a $\BB$-category, we refer to the $\BB$-topos $\I{C}^{\disc}=\Univ^{\I{C}}$ as the \emph{discrete} $\BB$-topos associated with $\I{C}$.
\end{definition}

Lastly, we note that the large $\BB$-category $\ILTop_{\BB}$ is also \emph{tensored} over $\ICat_{\BB}$:

\begin{proposition}
	\label{prop:tensoringLTop}
	There is a bifunctor $-\otimes -\colon\ICat_{\BB}\times\ILTop_{\BB}\to\ILTop_{\BB}$ that fits into an equivalence
	\begin{equation*}
		\map{\ILTop_{\BB}}(-\otimes -, -)\simeq\map{\ILTop_{\BB}}(-, (-)^{(-)})
	\end{equation*}
	of functors $\ICat_{\BB}^{\op}\times(\ILTop_{\BB})^\op\times\ILTop_{\BB}\to\Univ[\BBB]$.
\end{proposition}
\begin{proof}
	As an immediate consequence of the constructions, if $\I{C}$ and $\I{D}$ are $\BB$-categories, we obtain a chain of equivalences
	\begin{equation*}
	\map{\ILTop_{\BB}}(\Free{\I{C}}, 	(-)^{\I{D}})\simeq\map{\ICat_{\BBB}}(\I{C}, \IFun(\I{D},-))\simeq\map{\ICat_{\BBB}}(\I{C}\times\I{D}, -)\simeq \map{\ILTop_{\BB}}(\Free{\I{C}\times\I{D}},-),
	\end{equation*}
	which implies that the functor $\map{\ILTop_{\BB}}(\I{X}, (-)^{\I{D}})$ is representable whenever $\I{X}$ is in the image of $\Free{-}$. But since every $\BB$-topos is a pushout of such $\BB$-topoi (see Corollary~\ref{cor:freePresentationTopoi}), this functor must be representable for \emph{any} $\BB$-topos $\I{X}$. As by Remark~\ref{rem:BCTopoi} the same argument shows that this is the case for every object $\I{X}\colon A\to \ILTop_{\BB}$ and every $\Over{\BB}{A}$-category $\I{D}$, the result follows.
\end{proof}

\subsection{Relative $\infty$-topoi as $\BB$-topoi}
\label{sec:relativeInternalTopoi}
By Theorem~\ref{thm:characterisationBTopoi} and the evident fact that an algebraic morphism between $\BB$-topoi induces an algebraic morphism of $\infty$-topoi upon taking global sections, we obtain a functor $\Gamma\colon\LTop(\BB)\to\LTopS$. By making use of the fact that the universe $\Univ$ is an initial object in $\LTop(\BB)$ (Corollary~\ref{cor:UniverseInitial}), this functor factors through the projection $\Under{(\LTopS)}{\BB}\to\LTop$, so that we end up with a functor
\begin{equation*}
	\Gamma\colon \LTop(\BB)\to\Under{(\LTopS)}{\BB}.
\end{equation*}
The main goal in this section is to prove:
\begin{theorem}
	\label{thm:BTopoiRelativeTopoi}
	The global sections functor $\Gamma\colon \LTop(\BB)\to\Under{(\LTopS)}{\BB}$ is an equivalence of $\infty$-categories.
\end{theorem}
\begin{remark}
	Theorem~\ref{thm:BTopoiRelativeTopoi} implies that the datum of a $\BB$-topos $\I{X}$ is equivalent to that of a geometric morphism $f_\ast\colon \XX\to\BB$. We will refer to the latter as the geometric morphism that is \emph{associated} with $\I{X}$.
\end{remark}

\begin{remark}
	\label{rem:BToposFromRelativeToposExplicitly}
	One can describe the inverse to the equivalence $\Gamma\colon \LTop(\BB)\simeq\Under{(\LTopS)}{\BB}$ from Theorem~\ref{thm:BTopoiRelativeTopoi} explicitly as follows: Given any algebraic morphism $f^\ast\colon \BB\to \XX$, recall that we get an induced functor $f_* \colon \Cat(\widehat{\XX}) \to \Cat(\BBB)$. 
    Then Theorem~\ref{thm:characterisationBTopoi} easily implies that the large $\BB$-category $\I{X}=f_\ast\Univ[\XX]$ is a $\BB$-topos (since the associated sheaf on $\BB$ is simply given by $\Over{\XX}{f^\ast(-)}$). Moreover, the functor $f^\ast$ induces a map $\Over{\BB}{-}\to\Over{\XX}{f^\ast(-)}$ of sheaves on $\BB$ that recovers the unique algebraic morphism $\const_{\I{X}}\colon \Univ[\BB]\to\I{X}$. This implies that $\I{X}$ is the image of $f^\ast\colon \BB\to \XX$ under the equivalence from Theorem~\ref{thm:BTopoiRelativeTopoi}.
\end{remark}

The proof of Theorem~\ref{thm:BTopoiRelativeTopoi} requires a few preparations. We begin with the following lemma:
\begin{lemma}
	\label{lem:globalSectionsCocartesianFibration}
	Let $\CC$ be an $\infty$-category with an initial object $\varnothing_{\CC}$, and let $\DD$ be an $\infty$-category that admits pushouts. Then the evaluation functor $\ev_{\varnothing_{\CC}}\colon\Fun(\CC,\DD)\to\DD$ is a cocartesian fibration. Moreover, a morphism $\phi\colon F\to G$ in $\Fun(\CC,\DD)$ is cocartesian if and only if for every map $f\colon c\to c^\prime$ in $\CC$ the induced commutative square
	\begin{equation*}
		\begin{tikzcd}
			F(c)\arrow[r, "\alpha(c)"]\arrow[d, "F(f)"] & G(c)\arrow[d, "G(f)"]\\
			F(c^\prime)\arrow[r, "\alpha(c^\prime)"] & G(c^\prime)
		\end{tikzcd}
	\end{equation*}
	is a pushout in $\DD$.
\end{lemma}
\begin{proof}
	Note that the diagonal functor $\diag\colon \DD\to\Fun(\CC,\DD)$ defines a left adjoint to $\ev_{\varnothing_{\CC}}$. Therefore, we deduce from~\cite[Proposition~4.51]{Haugseng2017} that a map $\alpha\colon F\to G$ in $\Fun(\CC,\DD)$ is cocartesian if and only if for every $c\in\CC$ the square
	\begin{equation*}
		\begin{tikzcd}
			F(\varnothing_{\CC})\arrow[r, "\alpha(\varnothing_{\CC})"] \arrow[d] & G(\varnothing_{\CC})\arrow[d]\\
			F(c)\arrow[r, "\alpha(c)"] & G(c)
		\end{tikzcd}
	\end{equation*}
	is a pushout in $\DD$. The assumption that $\DD$ admits pushouts guarantees that there are enough such cocartesian maps, see~\cite[Corollary~4.52]{Haugseng2017}.
\end{proof}
By Lemma~\ref{lem:globalSectionsCocartesianFibration}, the global sections functor $\Gamma\colon\PSh_{\LTopS}(\BB)\to\LTopS$ is a cocartesian fibration and therefore induces a left fibration $\Gamma\colon \PSh_{\LTopS}(\BB)^{\cocart}\to\LTopS$, where $\PSh_{\LTopS}(\BB)^{\cocart}\into \PSh_{\LTopS}(\BB)$ is the subcategory that is spanned by the cocartesian morphisms. Moreover, observe that by Theorem~\ref{thm:characterisationBTopoi} we may regard the $\infty$-category $\LTop(\BB)$ as a (non-full) subcategory of $\PSh_{\LTopS}(\BB)$. Now the key step towards the proof of Theorem~\ref{thm:BTopoiRelativeTopoi} consists of the following proposition:
\begin{proposition}
	\label{prop:TopBFullSubcategoryCocartesianArrows}
	The (non-full) inclusion $\LTop(\BB)\into\PSh_{\LTopS}(\BB)$ fits into a commutative diagram
	\begin{equation*}
		\begin{tikzcd}
			\LTop(\BB)\arrow[dr, hookrightarrow] \arrow[r, hookrightarrow] & \PSh_{\LTopS}(\BB)^{\cocart}\arrow[d, hookrightarrow]\\
			& \PSh_{\LTopS}(\BB) 
		\end{tikzcd}
	\end{equation*}
	in which the horizontal map is fully faithful. Moreover, if $\I{X}$ is a $\BB$-topos and if $\I{X}\to F$ is a map in $\PSh_{\LTopS}(\BB)^{\cocart}$, then $F$ is contained in $\LTop(\BB)$.
\end{proposition}
\begin{proof}
	If $f^\ast\colon\I{X}\to\I{Y}$ is an algebraic morphism between $\BB$-topoi, Lemma~\ref{lem:criterionLTopEtalePushoutBC} and the fact that $f^\ast$ is cocontinuous imply that for every map $s\colon B\to A$ in $\BB$ the induced commutative square
	\begin{equation*}
		\begin{tikzcd}
			\I{X}(A)\arrow[r, "f^\ast(A)"]\arrow[d, "s^\ast"] & \I{Y}(A)\arrow[d, "s^\ast"]\\
			\I{X}(B)\arrow[r, "f^\ast(B)"] & \I{Y}(B)
		\end{tikzcd}
	\end{equation*}
	is a pushout in $\LTopS$. By Lemma~\ref{lem:globalSectionsCocartesianFibration}, this means that the underlying map of $\LTopS$-valued presheaves on $\BB$ defines a cocartesian morphism over $\LTopS$. Hence the inclusion $\LTop(\BB)\into\PSh_{\LTopS}(\BB)$ factors through the inclusion $\PSh_{\LTopS}(\BB)^{\cocart}\into \PSh_{\LTopS}(\BB)$. To finish the proof, it now suffices to show that for any cocartesian morphism $f\colon \I{X}\to F$ of $\LTopS$-valued presheaves on $\BB$, the presheaf $F$ is contained in $\LTop(\BB)$ and the map $f$ defines an algebraic morphism of $\BB$-topoi. 
	Since $f$ is a cocartesian morphism and since \'etale algebraic morphisms are closed under pushouts in $\LTopS$, we find that for every $s\colon B\to A$ in $\BB$ the induced functor $s^\ast\colon F(A)\to F(B)$ is an \'etale algebraic morphism of $\infty$-topoi.  Moreover, the pasting lemma for pushouts (and the fact that $\I{X}$ is a $\BB$-topos) imply that $F\colon \BB^{\op}\to\LTopEtS$ preserves pushouts. Hence Theorem~\ref{thm:characterisationBTopoi} implies that $F$ must be contained in $\LTop(\BB)$ whenever $F$ is a sheaf. But if $d\colon I\to\BB$ is an arbitrary diagram, then we deduce from~\cite[Corollary~4.7.4.18]{Lurie2017} that the commutative square
	\begin{equation*}
		\begin{tikzcd}
			\I{X}(\colim d)\arrow[r]\arrow[d] & F(\colim d)\arrow[d]\\
			\lim \I{X}\circ d\arrow[r] & \lim F\circ d
		\end{tikzcd}
	\end{equation*}
	is left adjointable and therefore a pushout in $\LTopS$, using Lemma~\ref{lem:criterionLTopEtalePushoutBC}. Hence, since the left vertical map is an equivalence, so is the right one, which means that $F$ is a sheaf. Finally, since $f$ is already section-wise given by an algebraic morphism of $\infty$-topoi, the map defines an algebraic morphism of $\BB$-topoi precisely if it is $\Univ$-cocontinuous, which again follows from Lemma~\ref{lem:criterionLTopEtalePushoutBC}.
\end{proof}
\begin{corollary}
	\label{cor:globalSectionsTopoiLeftFibration}
	The global sections functor $\Gamma\colon\LTop(\BB)\to\LTopS$ is a left fibration.
\end{corollary}
\begin{proof}
	By the first part of Proposition~\ref{prop:TopBFullSubcategoryCocartesianArrows}, every map in $\LTop(\BB)$ is cocartesian. By its second part, if $\I{X}$ is a $\BB$-topos and $f^\ast\colon\Gamma(\I{X})\to\ZZ$ is an arbitrary algebraic morphism, the codomain of the cocartesian lift $\I{X}\to F$ of $f^\ast$ in $\PSh_{\LTopS}(\BB)$ is again a $\BB$-topos. Hence the claim follows.
\end{proof}

\begin{proof}[{Proof of Theorem~\ref{thm:BTopoiRelativeTopoi}}]
	By Corollary~\ref{cor:globalSectionsTopoiLeftFibration}, the global sections functor $\Gamma\colon\LTop(\BB)\to\LTopS$ is a left fibration, hence so is the functor $\Gamma\colon\LTop(\BB)\to\Under{(\LTopS)}{\BB}$. Since this functor carries the initial object $\Univ$ to the initial object $\id_{\BB}$, it must be an initial functor as well. Hence $\Gamma$ is an equivalence.
\end{proof}

\subsection{Limits and colimits of $\BB$-topoi}
\label{sec:limitsColimitsTopoi}
In this section, we discuss how one can construct limits and colimits in the $\BB$-category $\ILTop_{\BB}$ of $\BB$-topoi.
The construction of \emph{limits} in $\ILTop_{\BB}$ is rather easy: they are simply computed in $\ICat_{\BBB}$. This is analogous to how limits are computed in the $\BB$-category $\ILPr_{\BB}$ of presentable $\BB$-categories, cf.~Proposition~\ref{prop:limitsPrL}. The proof of this statement follows along similar lines as well.

\begin{proposition}
	\label{prop:LTopLimits}
	The large $\BB$-category $\ILTop_{\BB}$ is complete, and the inclusion $\ILTop_{\BB}\into\ICat_{\BBB}$ is continuous.
\end{proposition}
\begin{proof}
	As in the proof of~Proposition~\ref{prop:limitsPrL}, it will be enough to show that whenever $\I{K}$ is either given by the constant $\BB$-category $\Lambda^2_0$ or by a $\BB$-groupoid, the large $\BB$-category $\ILTop_{\BB}$ admits $\I{K}$-indexed limits and the inclusion $\ILTop_{\BB}\into\ICat_{\BBB}$ preserves $\I{K}$-indexed limits.
	
	We begin with the case were $\I{K}$ is a $\BB$-groupoid. Let us set $A=\I{K}_0$. Since $(\pi_A)_\ast\colon\Cat(\Over{\BBB}{A})\to\Cat(\BBB)$ is given by precomposition with $\pi_A^\ast$, Theorem~\ref{thm:characterisationBTopoi} implies that $ (\pi_A)_\ast $ takes objects in $ \LTop(\BB_{/A}) $ to objects in $ \LTop(\BB) $. 
	Furthermore it easily follows from \cite[Proposition 5.4.2 and Proposition 5.4.5]{MWColimits} that $(\pi_A)_\ast $ therefore defines a functor $\LTop(\BB_{/A}) \to \LTop(\BB) $. Moreover, since the adjunction unit $\id_{\Cat(\BBB)}\to(\pi_A)_\ast\pi_A^\ast$ is given by precomposition with the adjunction counit $(\pi_A)_!\pi_A^\ast\to\id_{\BB}$ and vice versa for the adjunction counit, the same argument shows together with the fact that $\BB$-topoi are $\Univ$-cocomplete that these two maps must also restrict in the desired way. Hence $(\pi_A)_\ast\colon \LTop(\Over{\BB}{A})\to\LTop(\BB)$ defines a right adjoint of $\pi_A^\ast$.
	
	Now let us assume that $\I{K}=\Lambda^2_0$, i.e.\ let
	\begin{equation*}
		\begin{tikzcd}
			\I{X}\times_{\I{Z}}\I{Y}\arrow[r, "\pr_1"]\arrow[d, "\pr_0"] & \I{Y}\arrow[d, "g"]\\
			\I{X}\arrow[r, "f"] & \I{Z}
		\end{tikzcd}
	\end{equation*}
	be a pullback square in $\Cat(\BBB)$ in which the cospan in the lower right corner is contained in $\LTop(\BB)$. By~Proposition~\ref{prop:limitsPrL} this square defines a pullback in $\LPr(\BB)$, and~\cite[Proposition 6.3.2.3]{htt} implies that both $\pr_0$ and $\pr_1$ preserve finite limits. Hence the above pullback square is contained in $\LTop(\BB)$ whenever $\I{X}\times_{\I{Z}}\I{Y}$ satisfies descent. But the codomain fibration $(\I{X}\times_\I{Z}\I{Y})^{\Delta^1}\to\I{X}\times_\I{Z}\I{Y}$ can be identified with the pullback of the cospan $\pr_0^\ast(\I{X}^{\Delta^1)}\to \pr_0^\ast f^\ast(\I{Z}^{\Delta^1})\leftarrow\pr_1^\ast(\I{Y}^{\Delta^1})$ of cartesian fibrations over $\I{X}\times_{\I{Z}}\I{Y}$, which implies that we may identify $\Over{(\I{X}\times_{\I{Z}}\I{y})}{-}$ with the pullback $\Over{\I{X}}{{\pr_0^\ast(-)}}\times_{\Over{\I{Z}}{{\pr_0^\ast f^\ast(-)}}}\Over{\I{Y}}{{\pr_1^\ast(-)}}$ in $\IFun((\I{X}\times_{\I{Z}}\I{Y})^\op,\ICat_{\BBB})$. Since all four functors in the initial pullback square are continuous, we conclude that $\I{X}\times_{\I{Z}}\I{Y}$ satisfies descent provided that continuous functors are closed under pullbacks in $\IFun((\I{X}\times_{\I{Z}}\I{Y})^\op,\ICat_{\BBB})$, which follows immediately from the fact that limit functors are themselves continuous (see the proof of~Lemma~\ref{lem:stabilityFiltUCocontinuousFunctors} for more details). 
	We complete the proof by showing that if we are given another $ \BB $-topos $ \I{E} $ and algebraic morphisms $ h \colon \I{E} \to \I{X}$ and $ k \colon \I{E} \to \I{Z} $ together with an equivalence $ f \circ h \simeq g \circ k $, the induced map $ \I{E} \to \I{X} \times_{\I{Y}} \I{Z} $ is an algebraic morphism as well.
	That this map is cocontinuous follows from~Proposition~\ref{prop:limitsPrL}, and that it preserves finite limits is a consequence of the fact that this property can be checked section-wise.
\end{proof}

As a consequence of Proposition~\ref{prop:LTopLimits}, we can now upgrade the equivalence from Theorem~\ref{thm:BTopoiRelativeTopoi} to a \emph{functorial} one:
\begin{corollary}
	\label{cor:BTopRelTopFunctorial}
	Let $\Under{(\LTopS)}{(\Over{\BB}{-})}$ be the $\CatSS$-valued presheaf on $\BB$ whose associated cocartesian fibration on $\BB^{\op}$ is given by the pullback of $d_1\colon\Fun(\Delta^1,\LTopS)\to\LTopS$ along $\Over{\BB}{-}\colon\BB^{\op}\to\LTopS$. Then this presheaf is a sheaf whose associated large $\BB$-category is equivalent to $\ILTop_{\BB}$.
\end{corollary}
\begin{proof}
	To begin with, note that the functor $\Over{\Univ}{-}\colon\Univ^{\op}\to\ICat_{\BBB}$ takes values in $\ILTop_{\BB}$ (see the discussion before Definition~\ref{def:etaleBTopos} below). Thus, by combining descent with Proposition~\ref{prop:LTopLimits}, we obtain an $\Univ$-continuous functor $\Univ^{\op}\to\ILTop_{\BB}$. Hence, the underlying map of $\CatSS$-valued presheaves on $\BB$ can be regarded as a morphism in $\Fun^{\LAdj}(\BB^{\op},\CatSS)$ (in the sense of~\cite[\S~4.7.4]{Lurie2017}).
 On account of the equivalence $\Fun^{\LAdj}(\BB^{\op},\CatSS)\simeq\Fun^{\RAdj}(\BB,\CatSS)$ from~\cite[Corollary~4.7.4.18]{Lurie2017} that is furnished by passing to right adjoints, we thus obtain a morphism of functors $\varphi\colon(\Over{\BB}{-})^\op\to\LTop(\Over{\BB}{-})$ in which the functoriality on both sides is given by the right adjoints of the transition functors. 
	Let $\eta\colon \phi\to \diag_{\BB}(\phi(1))$ be the commutative square in $\Fun(\BB,\CatSS)$ that is obtained from the unit of the adjunction $\ev_{1}\dashv\diag_{\BB}\colon \Fun(\BB,\CatSS)\leftrightarrows\CatSS$. We may regard $\eta$ as a morphism in $\Fun(\BB, \CatSS^{\Delta^1})$.
	Note that for every map $s\colon B\to A$ in $\BB$ one has a commutative triangle
	\begin{equation*}
		\begin{tikzcd}[column sep=small]
			\LTop(\Over{\BB}{B})\arrow[rr, "s_\ast"]\arrow[dr, "\Gamma_{\Over{\BB}{B}}"'] && \LTop(\Over{\BB}{A})\arrow[dl, "\Gamma_{\Over{\BB}{A}}"]\\
			& \LTopS, &
		\end{tikzcd} 
	\end{equation*}
	hence Corollary~\ref{cor:globalSectionsTopoiLeftFibration} implies that $s_\ast$ is a left fibration. As the functor $s_!^\op\colon \Over{\BB}{B}^\op\to\Over{\BB}{A}^\op$ is a left fibration too, the map $\eta$ thus defines a morphism in $\Fun(\BB, \LFib)$ (where $\LFib$ is the full subcategory of $\Fun(\Delta^1,\CatSS)$ that is spanned by the left fibrations). Explicitly, this morphism carries $A\in\BB$ to the commutative square $\eta(A)\colon \phi(A)\to \phi(1)$. Now observe that the domain of $\eta$ is contained in the fibre $\LFib(\BB^\op)\into\LFib$ and the codomain is contained in the fibre $\LFib(\LTop(\BB))\into\LFib$. Moreover, for each $A\in\BB$ the functor $\Phi(A)\colon(\Over{\BB}{A})^\op\to\LTop(\Over{\BB}{A})$ carries the final object in $\Over{\BB}{A}$ to the initial object $\Over{\BB}{A}\in\LTop(\Over{\BB}{A})$ (see Corollary~\ref{cor:UniverseInitial}), hence $\Phi$ is section-wise initial. Altogether, these observations imply that 
	\begin{equation*}
		\LTop(\Over{\BB}{-})\colon \BB\to \LFib(\LTop(\BB))
	\end{equation*}
	is equivalent to the composition of $(\Over{\BB}{-})^\op\colon \BB\to \LFib(\BB^\op)$ (which is just the Yoneda embedding) with the functor of left Kan extension $\Phi(1)_!\colon\LFib(\BB^\op)\to\LFib(\LTop(\BB))$ along  $\Phi(1)\colon\BB^\op\to \LTop(\BB)$. But the latter composition is equivalent to the composition of $\Phi(1)^\op\colon\BB\to\LTop(\BB)^\op$ with the Yoneda embedding $\LTop(\BB)^\op\into\LFib(\LTop(\BB))$. By making use of the commutative diagram
	\begin{equation*}
		\begin{tikzcd}
			\BB\arrow[r, "\Phi(1)"] \arrow[rr, bend left, "\Over{\BB}{-}"]& \LTop(\BB)^\op\arrow[r, "\Gamma"]\arrow[d, hookrightarrow] & (\LTopS)^\op\arrow[d, hookrightarrow]\\
			& \LFib(\LTop(\BB))\arrow[r, "\Gamma_!"] & \LFib(\LTopS),
		\end{tikzcd}
	\end{equation*}
	the claim now follows.
\end{proof}

\begin{remark}
	\label{rem:CompareWithRelTopFunctorial}
	Corollary~\ref{cor:BTopRelTopFunctorial} in particular implies that for any map $s\colon B\to A$ in $\BB$ the transition functor $s^\ast\colon \LTop(\Over{\BB}{A})\to\LTop(\Over{\BB}{B})$ can be identified with the pushout functor
	\begin{equation*}
	  - \sqcup_{\Over{\BB}{A}} \BB_{/B}\colon \Under{(\LTopS)}{\Over{\BB}{A}}\to\Under{(\LTopS)}{\Over{\BB}{B}}.
	\end{equation*}
\end{remark}

As opposed to limits in $\ILTop_{\BB}$, general colimits of $\BB$-topoi can \emph{not} be computed on the underlying $\BB$-categories, not even after passing to the opposite $\BB$-category $\IRTop_{\BB}$. 
The existence of constant colimits follows easily from Theorem~\ref{thm:BTopoiRelativeTopoi}:
\begin{lemma}
	\label{lem:LTopConstColim}
	The large $\BB$-category $\ILTop_{\BB}$ is $ \ILConst $-cocomplete.
\end{lemma}
\begin{proof}
	In light of Remark~\ref{rem:CompareWithRelTopFunctorial}, this follows from the fact that for any map $s\colon B\to A$ in $\BB$ the $\infty$-categories $ (\LTopS)_{\BB_{/A}/} $ and $ (\LTopS)_{\BB_{/B}/} $ have colimits by \cite[Proposition~6.3.4.6]{htt} and $ -\sqcup_{\BB_{/A}} \BB_{/B} $ preserves all colimits.
\end{proof}

\begin{lemma}
	\label{lem:BTopoiGroupoidalColimits}
	The $ \BB $-category $ \ILTop_{\BB} $ is $\Univ$-cocomplete.
\end{lemma}
\begin{proof}
	By Remark~\ref{rem:BCTopoi}, it suffices to show that whenever $\I{G}$ is a $\BB$-groupoid and $d\colon \I{G}\to\ILTop_{\BB}$ is a diagram, the functor $ \map{\IFun(\I{G},\ILTop_{\BB})}(d, \diag(-)) $ is corepresentable. Note that we have an equivalence $\IFun(\I{G},-)\simeq(\pi_{\I{G}})_\ast\pi_{\I{G}}^\ast$. Therefore, Corollary~\ref{cor:freePresentationTopoi} implies that we can assume that $d$ is in the image of $\Free{-}_\ast\colon\IFun(\I{G},\ICat_{\BBB})\to\IFun(\I{G},\ILTop_{\BB})$. In this case, the claim follows from Corollary~\ref{cor:ColimitsFreeTopoi}.
\end{proof}

\begin{proposition}
	\label{prop:LTopCocomplete}
	The $ \BB $-category $ \ILTop_\BB $ is cocomplete.
\end{proposition}
\begin{proof}
	By Proposition~\ref{prop:CocompleteGroupoidsExternal}, this follows from, Lemmas~\ref{lem:LTopConstColim} and~\ref{lem:BTopoiGroupoidalColimits}.
\end{proof}

\subsection{A formula for the coproduct of $\BB$-topoi}
\label{sec:coproductTopoi}
The goal of this section is to give an explicit description of the coproduct in $\ILTop_{\BB}$. To that end, recall that by the discussion in \S~\ref{sec:tensorProductCatU} the large $\BB$-category $\ILPr_{\BB}$ of presentable $\BB$-categories is symmetric monoidal. Explicitly, if $\I{D}$ and $\I{E}$ are presentable $\BB$-categories, their tensor product $\I{D}\otimes\I{E}$ is equivalent to the $\BB$-category $\IShv_{\I{E}}(\I{D})$ of $\I{E}$-valued sheaves on $\I{D}$ (i.e.\ the full subcategory of $\IFun(\I{D}^\op,\I{E})$ spanned by the continuous functors $\pi_A^\ast\I{D}^{\op}\to\pi_A^\ast\I{E}$ for each $A\in\BB$). In light of this identification, the proof of~Proposition~\ref{prop:sheaveshasunivprop} shows that if $f^\ast\colon \I{D}\to\I{D}^\prime$ and $g^\ast\colon \I{E}\to\I{E}^\prime$ are maps in $\ILPr_{\BB}$ with right adjoints $f_\ast$ and $g_\ast$, then the functor $\id\otimes f^\ast\colon\I{D}\otimes\I{E}\to \I{D}\otimes\I{E}^\prime$ can be identified with the left adjoint of $(f_\ast)_\ast\colon \IShv_{\I{E}^\prime}(\I{D})\to\IShv_{\I{E}}(\I{D})$, and the functor $g^\ast\otimes\id\colon\I{D}\otimes\I{E}\to\I{D}^\prime\otimes\I{E}$ can be identified with the left adjoint of $(g^\ast)^\ast\colon\IShv_{\I{E}}(\I{D}^\prime)\to \IShv_{\I{E}}(\I{D})$.
\begin{proposition}
	\label{prop:coproductsBTopoi}
	If $\I{X}$ and $\I{Y}$ are $\BB$-topoi, then their tensor product $\I{X}\otimes \I{Y}$ is a $\BB$-topos as well, and the functors $\id\otimes\const_{\I{Y}}\colon\I{X}\simeq \I{X}\otimes\Univ\to\I{X}\otimes \I{Y}$ and $\const_{\I{X}}\otimes \id\colon\I{Y}\simeq\Univ\otimes\I{Y}\to\I{X}\otimes\I{Y}$ exhibit $\I{X}\otimes\I{Y}$ as the coproduct of $\I{X}$ and $\I{Y}$ in $\ILTop_{\BB}$.
\end{proposition}

Combining the above result with Proposition~\ref{prop:Mod_B=PrL(B)_for_0-localic}, we obtain the following generalisation of \cite[Corollary 1.10]{aoki2023sheavesspectrum}:

\begin{corollary}
    \label{cor:O-trunc_Tensor=Product}
	Assume that $ \BB $ is generated under colimits by $ (-1) $-truncated objects.
	Then for $ \XX,\YY \in \LTop_{\BB/} $ the canonical map
	\[
	\XX \otimes_{\BB} \YY \to \XX \sqcup_{\BB} \YY
	\]
	is an equivalence.
\end{corollary}

The proof of Proposition~\ref{prop:coproductsBTopoi} requires a few preparations and will be given at the end of this section. First, let us observe that this result provides an explicit formula for the pushout of $\infty$-topoi:

\begin{corollary}
	\label{cor:formulaPushoutTopoi}
	Given a cospan $ \XX \xleftarrow{f^*} \ZZ \xrightarrow{g^*} \YY $ in $ \LTopS $, there is a canonical equivalence
	\[
	\XX \sqcup_{\ZZ} \YY \simeq \Fun_\ZZ^{\c}(f_\ast(\Univ[\XX])^\op,g_\ast\Univ[\YY])
	\]
	in which the right-hand side denotes the full subcategory of $\Fun_\ZZ(f_\ast(\Univ[\XX])^\op,g_\ast\Univ[\YY])$ that is spanned by the continuous functors.\qed
\end{corollary}

\begin{remark}
	\label{rem:TensorProdExplicitly}
	In light of Corollary~\ref{cor:formulaPushoutTopoi}, the $\infty$-topos $ \XX \sqcup_\ZZ \YY $ admits the following explicit description:
	It is the full subcategory of the $\infty$-category of natural transformations between the two $\CatSS$-valued sheaves $\Over{\XX}{f^\ast(-)}$ and $\Over{\YY}{g^\ast(-)}$ on $\ZZ$ that is spanned spanned by those maps $ \phi \colon (\XX_{/f^*(-)})^{\op} \to \YY_{/g^*(-)} $ which satisfy that
	\begin{enumerate}
		\item the functor $ \phi(A)$ preserves limits for all $ A \in \ZZ $, and
		\item for any map $ s \colon B \to A $ in $\ZZ$ the canonical lax square
		\[\begin{tikzcd}
			{(\XX_{/f^*(B)})^\op} & {\YY_{/g^*(B)}} \\
			{(\XX_{/f^*(A)})^\op} & {\YY_{/g^*(A)}}
			\arrow["{\phi(B)}", from=1-1, to=1-2]
			\arrow["{g^*(s)_*}", from=1-2, to=2-2]
			\arrow["{f^*(s)_!}"', from=1-1, to=2-1]
			\arrow["{\phi(A)}"', from=2-1, to=2-2]
			\arrow[Rightarrow, from=2-1, to=1-2, shorten=4mm]
		\end{tikzcd}\]
		commutes.\qed
	\end{enumerate}
\end{remark}

Admittedly, the description of the pushout of $\infty$-topoi in Remark~\ref{rem:TensorProdExplicitly} is rather unwieldy in general. However, we can paint a more concrete picture in the following case:
\begin{example}
	\label{ex:PushoutPresheafCategory}
	Let $\I{X}$ be a $\BB$-topos and let $\I{C}$ be an arbitrary $\BB$-category. Then Proposition~\ref{prop:coproductsBTopoi} implies that the commutative square
	\begin{equation*}
		\begin{tikzcd}
			\Univ\arrow[d, "\const_{\I{X}}"]\arrow[r, "\diag"] & \IPSh(\I{C})\arrow[d, "(\const_{\I{X}})_\ast"]\\
			\I{X}\arrow[r, "\diag"] & \IFun(\I{C}^\op, \I{X})
		\end{tikzcd}
	\end{equation*}
	is a pushout in $\ILTop_{\BB}$. Furthermore, if $f\colon \XX\to\BB$ is the geometric morphism associated to $\I{X}$, then the lower horizontal map can be identified with the image of $\diag\colon\Univ[\XX]\to\IPSh[\XX](f^\ast\I{C})$ along $f_\ast$.
\end{example}

We now turn to the proof of Proposition~\ref{prop:coproductsBTopoi}. 
It is a straightforward adaption of the proof presented in \cite[\S 2.3]{Anel2018a} to the setting of $\BB$-categories. 
We begin with the following lemma:
\begin{lemma}
	\label{lem:ProductOfLexFunctors}
	Let $ \I{C} $ and $ \I{D} $ be $ \BB $-categories with finite limits. Then precomposition with the canonical maps $ \id_{\I{C}}\times 1_{\I{D}} \colon \I{C} \to\I{C} \times \I{D} $ and $1_{\I{C}}\times\id_{\I{D}} \colon \I{D} \to \I{C} \times \I{D}  $ induces an equivalence
	\[
	\IFun^{\lex}(\I{C}\times \I{D},\I{E}) \simeq \IFun^{\lex}(\I{C},\I{E}) \times \IFun^{\lex}(\I{D},\I{E})
	\]
	for any $\BB$-category $\I{E}$ with finite limits. In other words, these two maps exhibit $\I{C}\times\I{D}$ as the coproduct of $\I{C}$ and $\I{D}$ in $\ICat_{\BB}^\lex$.
\end{lemma}
\begin{proof}
	The composition
	\begin{align*}
		\IFun^{\lex}(\I{C},\I{E}) \times \IFun^{\lex}(\I{D},\I{E}) &\xrightarrow{\pr_0^* \times \pr_1^*} \IFun^{\lex}(\I{C}\times \I{D},\I{E}) \times \IFun^{\lex}(\I{C}\times \I{D},\I{E})\\
		&\xrightarrow{\simeq}  \IFun^{\lex}(\I{C}\times \I{D}, \I{E} \times \I{E}) \\
		&\xrightarrow{(-\times-)_*} \IFun^{\lex}(\I{C}\times \I{D},\I{E})
	\end{align*}
	defines an inverse.
\end{proof}

The rough strategy of the proof of Proposition~\ref{prop:coproductsBTopoi} is to first prove the claim for free $\BB$-topoi, which will follow from Lemma~\ref{lem:ProductOfLexFunctors}.
In order to reduce the general case to this setting we need to understand the compatibility of tensor products with localisations:

\begin{lemma}
	\label{lem:LocalisationOfTensorProds}
	Suppose that $ \I{C} $ and $ \I{D} $ are presentable $ \BB $-categories  and that $ \I{W} \into \I{C} $ and $ \I{S} \into \I{D}  $ are small subcategories.
	Let $\I{C}^\prime\into\I{C}$ be a small full subcategory that exhibits $\I{C}$ as the free $\IFilt_{\I{U}}$-cocompletion of $\I{C}^\prime$ for some sound doctrine $\I{U}$. Let $\I{D}^\prime\into\I{D}$ be chosen similarly.
	We write $ \tau \colon \I{C} \times \I{D} \to \I{C} \otimes \I{D} $ for the universal bilinear functor.
	Let us set $\I{W} \boxtimes \I{S}=(\I{W} \times (\I{D}^\prime)^\simeq) \sqcup ((\I{C}^\prime)^\simeq \times \I{S})$.
	Then the canonical map $ \I{C} \otimes \I{D} \to \ILoc_{\I{W}}(\I{C}) \otimes \ILoc_{\I{S}}(\I{D})$ induces an equivalence
	\[
	\ILoc_{\I{W} \boxtimes \I{S}}(\I{C} \otimes \I{D})  \xrightarrow{\simeq} \ILoc_{\I{W}}(\I{C}) \otimes \ILoc_{\I{S}}(\I{D}),
	\]
	where the left-hand side is the $\BB$-category of local objects with respect to $(\tau,\tau)\colon \I{W}\boxtimes\I{S}\to\I{C}\otimes\I{D}$.
\end{lemma}
\begin{proof}
	Let $ \I{E} $ be any other presentable $ \BB $-category and let us denote by $\IFun^{\bil}(\I{C}\times\I{D},\I{E})_{\I{W}\boxtimes\I{S}}$ the full subcategory of $\IFun^{\bil}(\I{C}\times\I{D},\I{E})$ that is spanned by those bilinear functors $\pi_A^\ast\I{C}\times\pi_A^\ast\I{D}\to\pi_A^\ast\I{E}$ (in arbitrary context $A\in\BB$) whose precomposition with $\pi_A^\ast(\I{W}\boxtimes\I{S})\to \pi_A^\ast\I{C}\times\pi_A^\ast\I{D}$ factors through $\pi_A^\ast\I{E}^\core$. By combining the universal property of the tensor product with Corollary~\ref{cor:UniversalPropertyLocalObjects}, we now obtain a chain of equivalences
	\begin{equation*}
		\IFun^\cc(\ILoc_{\I{W}\boxtimes\I{S}}(\I{C}\otimes\I{D}),\I{E})\simeq\IFun^\cc(\I{C}\otimes\I{D},\I{E})_{\I{W}\boxtimes\I{S}}\simeq\IFun^{\bil}(\I{C}\times\I{D},\I{E})_{\I{W}\boxtimes\I{S}}.
	\end{equation*}
	Note that a bilinear functor $f\colon \I{C}\times\I{D}\to\I{E}$ is contained in $\IFun^{\bil}(\I{C}\times\I{D},\I{E})_{\I{W}\boxtimes\I{D}}$ if and only if 
	\begin{enumerate}
		\item for any $ c \colon A \to \I{C}^{\prime}$ in context $ A \in \BB $ the functor  $\pi_A^* \I{S} \into \pi_A^* \I{D} \xrightarrow{f(c,-)} \pi_A^* \I{E}$
		factors through $ \pi_A^* \I{E}^\simeq$, and
		\item for any $ d \colon A \to \I{D}^\prime$ in context $ A \in \BB $ the functor $\pi_A^* \I{W} \into \pi_A^* \I{C} \xrightarrow{f(-,d)} \pi_A^* \I{E} $
		factors through $ \pi_A^* \I{E}^\simeq$.
	\end{enumerate}
	Let $f^\prime\colon \I{C}\to \IFun^\cc(\I{D},\I{E})$ be the image of $f$ under the equivalence $\IFun^{\bil}(\I{C}\times\I{D},\I{E})\simeq\IFun^\cc(\I{C},\IFun^\cc(\I{D},\I{E}))$ from Lemma~\ref{lem:limitpreservationswap}. Now the first condition is equivalent to the composition $ \I{C}' \into \I{C} \xrightarrow{f^\prime} \IFun^\cc(\I{D},\I{E}) $ taking values in $ \IFun^\cc(\ILoc_{\I{S}}(\I{D}), \I{E}) $.
	Note that the inclusion $ \IFun^\cc(\ILoc_{\I{S}}(\I{D}), \I{E}) \into \IFun^\cc(\I{D},\I{E}) $ is given by precomposition with $ \I{D} \to \ILoc_{\I{S}}(\I{D}) $ and is therefore cocontinuous.
	Since $ \I{C}' $ generates $ \I{C} $ under $\IFilt_{\I{U}}$-colimits, it thus follows that (1) is equivalent to $ f^\prime$ being contained in $\IFun^\cc(\I{C},\IFun^\cc(\ILoc_{\I{S}}(\I{D}),\I{E}))$. Similarly, if $ f^{\prime\prime} \colon \I{D} \to \IFun^\cc(\I{C},\I{E})$ is the other transpose of $ f $, condition~(2) is equivalent to $f^{\prime\prime}$ taking values in $\IFun^\cc(\ILoc_{\I{W}}(\I{C}),\I{E})$.
	Thus the naturality of the equivalence in Lemma~\ref{lem:limitpreservationswap} implies that $f $ satisfies (1) and (2) if and only if $f$ is contained in $\IFun^{\bil}(\ILoc_{\I{W}}(\I{C})\times\ILoc_{\I{S}}(\I{D}),\I{E})$. As the same argument can be carried out for bilinear functors in arbitrary context, this shows that the equivalence $\IFun^\cc(\I{C}\otimes\I{D},\I{E})\simeq\IFun^{\bil}(\I{C}\times\I{D},\I{E})$ restricts to an equivalence $	\IFun^\cc(\ILoc_{\I{W}\boxtimes\I{S}}(\I{C}\otimes\I{D}),\I{E})\simeq\IFun^{\bil}(\ILoc_{\I{W}}(\I{C})\times\ILoc_{\I{S}}(\I{D}),\I{E})$, which proves the claim.
\end{proof}

A similar argument as above shows the following:
\begin{lemma}
	\label{lem:LocalisationsGivePushout}
	Let $ \I{C} $ and $ \I{D} $ be presentable $ \BB $-categories and let $ \I{W} \into \I{C} $ and $ \I{S} \into \I{D}  $ be small subcategories.
	Then the commutative square
	\[\begin{tikzcd}
		{\I{C} \otimes \I{D}} & { \I{C} \otimes \ILoc_{\I{S}}(\I{D}}) \\
		{\ILoc_{\I{W}}(\I{C}) \otimes  \I{D}} & {\ILoc_{\I{W}}(\I{C}) \otimes \ILoc_{\I{S}}(\I{D})}
		\arrow[from=1-1, to=1-2]
		\arrow[from=1-2, to=2-2]
		\arrow[from=1-1, to=2-1]
		\arrow[from=2-1, to=2-2]
	\end{tikzcd}\]
	is a pushout in $\ILPr_{\BB}$.\qed
\end{lemma}

\begin{proof}[Proof of Proposition~\ref{prop:coproductsBTopoi}]
    To simplify notation, we shall write $ i_0 = \id\otimes\const_{\I{Y}} $ as well as $ i_1 = \const_{\I{X}}\otimes \id $.
	First, let us show the claim in the special case where $ \I{X} = \Free{\I{C}}$ and $ \I{Y} = \Free{\I{D}}$.
	In this situation, we have an equivalence $ \I{X} \otimes \I{Y} \simeq \IPSh(\I{C}^\lex\times \I{D}^\lex) $ with respect to which the functors $ i_0 $ and $ i_1 $ are given by left Kan extension along $ \id\times1_{\I{D}^\lex}\colon\I{C}^\lex \to\I{C}^\lex \times \I{D}^\lex  $ and $ 1_{\I{C}^\lex}\times\id\colon\I{D}^\lex \to \I{C}^\lex \times \I{D}^\lex $, respectively. By Lemma~\ref{lem:ProductOfLexFunctors}, the latter two functors exhibit $\I{C}^\lex\times\I{D}^\lex$ as the coproduct $\I{C}^\lex\sqcup\I{D}^\lex$ in $\ICat_{\BB}^\lex$. As the functor $(-)^\lex$ is a left adjoint and thus preserves coproducts, we end up with an equivalence $\I{X}\otimes\I{Y}\simeq \Free{\I{C}\sqcup\I{D}}$ with respect to which $i_0$ and $i_1$ correspond to the image of the inclusions $\I{C}\into\I{C}\sqcup\I{D}$ and $\I{D}\into\I{C}\sqcup\I{D}$ along the functor $\Free{-}$. The claim thus follows from Corollary~\ref{cor:ColimitsFreeTopoi}. 
	
	In the general case, we may choose left exact and accessible Bousfield localisations $L\colon  \Free{\I{C}}\to\I{X} $  and $L^\prime\colon \Free{\I{D}}\to\I{Y}$, cf.~Corollary~\ref{cor:freeGenerationTopoi}.
	By Lemma~\ref{lem:LocalisationsGivePushout} we have a pushout square
	\[\begin{tikzcd}
		{\Free{\I{C}} \otimes \Free{\I{D}}} & {\I{X} \otimes \Free{\I{D}}} \\
		{\Free{\I{C}} \otimes \I{Y}} & {\I{X} \otimes \I{Y}}
		\arrow[from=1-1, to=1-2]
		\arrow[from=1-2, to=2-2]
		\arrow[from=1-1, to=2-1]
		\arrow[from=2-1, to=2-2]
	\end{tikzcd}\]
	in $ \LPr(\BB) $.
	The upper horizontal functor is equivalent to the functor 
	\begin{equation*}
		L_* \colon \IFun((\I{D}^\lex)^\op, \Free{\I{C}})\to\IFun((\I{D}^\lex)^\op, \I{X})
	\end{equation*}
	Thus $ \I{X} \otimes \I{Y} $ is equivalent to the intersection of two accessible and left exact Bousfield localisations of $ \Free{\I{C}}\otimes\Free{\I{D}}$ and therefore by~\cite[Lemma 6.3.3.4]{htt} in particular a $ \BB $-topos.
	and therefore a left exact and accessible Bousfield localisation. By symmetry, the same holds for the left vertical functor.
	Since the square
	\[\begin{tikzcd}
		{\Free{\I{C}}} & {\I{X}} \\
		{\Free{\I{C}} \otimes \Free{\I{D}}} & {\I{X} \otimes \Free{\I{D}}}
		\arrow["i_0", from=1-1, to=2-1,swap]
		\arrow[from=1-1, to=1-2]
		\arrow["i_0"', from=1-2, to=2-2]
		\arrow[from=2-1, to=2-2]
	\end{tikzcd}\]
	commutes, it follows that $ i_0 \colon \I{X} \to \I{X} \otimes \Free{\I{D}}$ is left exact.
	Since $ i_0 \colon \I{X} \to \I{X} \otimes \I{Y} $ factors as the composite $ \I{X} \xrightarrow{i_0} \I{X} \otimes \Free{\I{D}} \to \I{X} \otimes \I{Y} $ it is therefore also left exact.
	The same argument shows that $ i_1 \colon \I{Y} \to \I{X} \otimes \I{Y} $ is left exact.
	Finally, note that $ L $ and $ L' $ induce a commutative square
	\[\begin{tikzcd}
		{\IFun^\alg(\I{X} \otimes \I{Y},\I{Z})} & {\IFun^\alg(\I{X},\I{Z}) \times \IFun^\alg(\I{Y},\I{Z})} \\
		{\IFun^\alg(\Free{\I{C}} \otimes \Free{\I{D}},\I{Z})} & {\IFun^\alg(\Free{\I{C}},\I{Z}) \times \IFun^\alg(\Free{\I{D}},\I{Z})}
		\arrow["{(i_0^*,i_1^*)}", from=1-1, to=1-2]
		\arrow[from=2-1, to=2-2, "\simeq"]
		\arrow[hook', from=1-2, to=2-2]
		\arrow[hook', from=1-1, to=2-1]
	\end{tikzcd}\]
	for any $ \BB $-topos $ \I{Z} $.
	As the lower horizontal map being an equivalence implies that $(i_0^\ast, i_1^\ast)$ is fully faithful, it thus suffices to see that this functor is also essentially surjective. Using Remark~\ref{rem:BCTopoi}, it will be enough to show that for any two algebraic morphisms $ f \colon\I{X} \to \I{Z} $ and $ g \colon  \I{Y} \to  \I{Z}  $ the induced map $ \Free{\I{C}}\otimes\Free{\I{D}} \to  \I{Z} $ factors through $ \Free{\I{C}} \otimes \Free{\I{D}}\to \I{X} \otimes \I{Y} $.
	This is a direct consequence of Lemma~\ref{lem:LocalisationOfTensorProds}.
\end{proof}

\subsection{Diaconescu's theorem}
\label{sec:Diaconescu}
In classical category theory, Diaconescu's theorem states that for any $1$-category $\CC$ and any $1$-topos $\XX$, a functor $f\colon \CC\to \XX$ is \emph{internally flat} if and only if its left Kan extension $h_! f\colon\PSh_{\Set}(\CC)\to\XX$ preserves finite limits, see for example~\cite[Theorem~B.3.2.7]{johnstone2002}. Here $f$ being internally flat precisely means that its internal unstraightening results in a \emph{filtered} internal category in $\XX$. For $\infty$-categories, a comparable result has been proved by Lurie~\cite[Proposition~6.1.5.2]{htt} in the special case where the $\infty$-category $\CC$ already admits finite limits. In the general case, Raptis and Schäppi proved Diaconescu's theorem under the assumption that the codomain $\XX$ is a hypercomplete $\infty$-topos~\cite{Raptis2022}.

The main goal of this section is to establish a general version of Diaconescu's theorem for $\BB$-topoi and therefore also a general version of Diaconescu's theorem for $\infty$-topoi, without any hypercompleteness assumptions.
To that end, let us say that a presheaf $F\colon\I{C}^{\op}\to\Univ$ on an arbitrary $\BB$-category $\I{C}$ is \emph{flat} if it is $\IFin_\BB$-flat in the sense of Definition~\ref{def:UFlat}. 
We will denote by $\IFlat(\I{C})$ the associated $\BB$-category of flat functors.
Recall from Proposition~\ref{prop:FinBRegular} that the doctrine $\IFin_\BB$ is sound. Therefore, Proposition~\ref{prop:FlatIsAccessible} implies:
\begin{proposition}
	\label{prop:Flatifiltered}
	For any $\BB$-category $\I{C}$, a functor $F\colon \I{C}^{\op}\to \Univ$ is flat if and only if the $\BB$-category $\Over{\I{C}}{F}$ is filtered.\qed
\end{proposition}
Diaconescu's theorem for $\BB$-topoi can now be stated as follows:
\begin{theorem}
	\label{thm:internalDiaconescu}
	Let $\I{X}$ be a $\BB$-topos with associated geometric morphism $f_\ast\colon \XX\to\BB$ and let $\I{C}$ be an arbitrary $\BB$-category. The precomposition with the Yoneda embedding induces an equivalence
	\begin{equation*}
		\IFun^\alg(\IPSh(\I{C}),\I{X})\simeq f_\ast\IFlat[\XX](f^\ast\I{C}^\op).
	\end{equation*}
\end{theorem}
Specialising to the case where $\BB\simeq \SS$, Theorem~\ref{thm:internalDiaconescu} implies:
\begin{corollary}
	\label{cor:Diaconescu}
	For any small $\infty$-category $\CC$, a functor $f\colon \CC\to \BB$ is flat if and only if its Yoneda extension $h_!f\colon\PSh_{\SS}(\CC)\to\BB$ preserves finite limits. In particular, the functor of left Kan extension along $h_{\CC}$ induces an equivalence
	\begin{equation*}
		h_!\colon\Flat_{\BB}(\CC^\op)\simeq \Fun^\alg(\PSh_{\SS}(\CC),\BB)
	\end{equation*}
	of $\infty$-categories.\qed
\end{corollary}
\begin{remark}
	Corollary~\ref{cor:Diaconescu} can be used to define morphisms of general $\infty$-sites: if $\CC$ and $\DD$ are $\infty$-sites, a functor $f\colon \CC\to \DD$ is a morphism of $\infty$-sites if the associated functor $f^\prime\colon\CC\to \Shv(\DD)$ (which is obtained by composing $f$ with the sheafified Yoneda embedding $Lh\colon \DD\to\Shv(\DD)$) is flat and if for every covering $(c_i\to c)_{i\in I}$ in $\CC$ the induced functor $\bigsqcup_i f^\prime(c_i)\to f^\prime(c)$ is a cover in $\Shv(\DD)$. Using this definition, Corollary~\ref{cor:Diaconescu} and~\cite[Lemma~6.2.3.19]{htt} imply that every morphism of $\infty$-sites $f\colon \CC\to\DD$ induces an algebraic morphism $F\colon\Shv(\CC)\to\Shv(\DD)$.
\end{remark}

The proof of Theorem~\ref{thm:internalDiaconescu} relies on the following two elementary lemmas:
\begin{lemma}
	\label{lem:GrothendieckConstructionBaseFunctoriality}
	Let $\I{X}$ be a $\BB$-topos and let $f\colon\XX\to\BB$ be the corresponding geometric morphism. Suppose that $p\colon\I{P}\to\I{C}$ is a left fibration of $\XX$-categories that is classified by a functor $g\colon \I{C}\to\Univ[\XX]$. Then the left fibration $f_\ast(p)$ of $\BB$-categories is classified by the composition $f_\ast\I{C}\xrightarrow{f_\ast(g)}\I{X}\xrightarrow{\Gamma_{\I{X}}}\Univ[\BB]$.
\end{lemma}
\begin{proof}
	Since the functor $f_\ast$ commutes with pullbacks and with powering by $\infty$-categories, the image of the universal left fibration $\Under{(\Univ[\XX])}{1}\to\Univ[\XX]$ along $f_\ast$ can be identified with $(\pi_{1_{\I{X}}})_!\colon \Under{\I{X}}{1_{\I{X}}}\to \I{X}$ and is therefore classified by $\map{\I{X}}(1_{\I{X}},-)\simeq\Gamma_{\I{X}}$. Hence the claim follows.
\end{proof}

\begin{lemma}
	\label{lem:BaseFunctorialityYonedaLemma}
	Let $\I{X}$ be a $\BB$-topos and let $f\colon \XX\to\BB$ be the corresponding geometric morphism. Then for any $\BB$-category $\I{C}$, there is a commutative square
	\begin{equation*}
		\begin{tikzcd}
			\I{C}\arrow[r, "h_{\I{C}}"] \arrow[d, "\eta"] & \IPSh(\I{C})\arrow[d, "(\const_{\I{X}})_\ast"]\\
			f_\ast f^\ast\I{C}\arrow[r, "f_\ast (h_{f^\ast\I{C}})"] & \IFun(\I{C}^\op, \I{X}).
		\end{tikzcd}
	\end{equation*}
\end{lemma}
\begin{proof}
	Transposing the Yoneda embedding $h_{f^\ast\I{C}}\colon f^\ast\I{C}\into\IPSh[\XX](f^\ast\I{C})$ across the adjunction $f^\ast\dashv f_\ast$ yields the composition
	\begin{equation*}
		\I{C}\xrightarrow{\eta} f^\ast f_\ast\I{C}\xhookrightarrow{f_\ast(h_{f^\ast\I{C}})} \IFun(\I{C}^\op, \I{X})
	\end{equation*}
	in which $\eta$ is the adjunction unit. By transposing the above map across the adjunction $\I{C}^{\op}\times -\dashv \IFun(\I{C}^\op,-)$, one ends up with the functor
	\begin{equation*}
		\I{C}^{\op}\times\I{C}\xrightarrow{\eta} f_\ast f^\ast(\I{C}^\op\times\I{C})\xrightarrow{f_\ast(\map{f^\ast\I{C}})}\I{X}.
	\end{equation*}
	On the other hand, the transpose of the composition $\I{C}\into\IPSh(\I{C})\to\IFun(\I{C}^\op, \I{X})$ yields
	\begin{equation*}
		\I{C}^{\op}\times\I{C}\xrightarrow{\map{\I{C}}} \Univ[\BB]\xrightarrow{\const_{\I{X}}}\I{X},
	\end{equation*}
	so it suffices to show that these two functors are equivalent. By Lemma~\ref{lem:GrothendieckConstructionBaseFunctoriality} the functor $\map{f_\ast f^\ast\I{C}}$ is equivalent to the composition $\Gamma_{\I{X}}\circ f_\ast(\map{f^\ast\I{C}})\colon f_\ast\I{C}^{\op}\times f_\ast\I{C}\to \I{X}\to \Univ[\BB]$. As a consequence, the morphism of functors $\map{\I{C}}\to \map{f_\ast f^\ast\I{C}}\circ\eta$ that is induced by the action of $\eta$ on mapping $\BB$-groupoids determines a morphism
	$\map{\I{C}}\to \Gamma_{\I{X}}\circ f_\ast(\map{f^\ast\I{C}})\circ\eta$
	which in turn transposes to a map
	\begin{equation*}
		\const_{\I{X}}\circ\map{\I{C}}\to f_\ast(\map{f^\ast\I{C}})\circ\eta.
	\end{equation*}
	To show that this is an equivalence, it will be enough to show that it induces an equivalence when evaluated at $(\tau,\tau)$, where $\tau$ is the tautological object in $\I{C}$, i.e.\ the one given by the identity of $\I{C}_0$. But by construction, the resulting map is simply the transpose of $\eta\colon \I{C}_1\to f_\ast f^\ast\I{C}_1$ across the adjunction $f^\ast\dashv f_\ast$ and therefore an equivalence, as desired.
\end{proof}

\begin{proof}[{Proof of Theorem~\ref{thm:internalDiaconescu}}]
	To begin with, we note that the universal property of presheaf $\BB$-categories together with Remarks~\ref{rem:BCTopoi} and~\ref{rem:UflatLocal} implies that it suffices to show that a functor $g\colon \I{C}\to\I{X}$ transposes to a flat functor $g^\prime\colon f^\ast\I{C}\to \Univ[\XX]$ if and only if its Yoneda extension $(h_{\I{C}})_!(g)\colon \IPSh(\I{C})\to\I{X}$ preserves finite limits. Note that by Lemma~\ref{lem:BaseFunctorialityYonedaLemma} (and the fact that base change along $f_\ast$ preserves cocontinuity), we have a commutative diagram
	\begin{equation*}
		\begin{tikzcd}[column sep={10em,between origins}]
			\IPSh(\I{C})\arrow[d, "(\const_{\I{X}})_\ast"]\arrow[dr,bend left, "(h_{\I{C}})_!(g)"] & \\
			\IFun(\I{C}^\op, \I{X})\arrow[r, "f_\ast(h_{f^\ast\I{C}})_!(g^\prime)"] & \I{X}.
		\end{tikzcd}
	\end{equation*}
	Therefore, $g^\prime$ being flat immediately implies that $(h_{\I{C}})_!(g)$ is an algebraic morphism, so it suffices to consider the converse implication. Suppose therefore that the left Kan extension $(h_{\I{C}})_!(g)$ preserves finite limits. We wish to show that the functor $(h_{f^\ast\I{C}})_!(g^\prime)$ preserves finite limits as well. In light of the previous commutative diagram and the fact that $(\const_{\I{X}})_\ast$ preserves finite limits, it is clear that it preserves the final object, so we only need to consider the case of pullbacks. By Lemma~\ref{lem:YonedaExtensionPullbackPreservation}, we may reduce to pullbacks of cospans in $\IPSh[\XX](f^\ast\I{C})$ (in arbitrary context $U\in \XX$) which are contained in the essential image of the Yoneda embedding $h_{f^\ast\I{C}}$. Since any such cospan is determined by a map $U\to (f^\ast\I{C})^{\Lambda^2_0}$, it factors through the core inclusion $\tau_{f^\ast\I{C}}\colon f^\ast(\I{C}_1\times_{\I{C}_0}\I{C}_1)\to (f^\ast\I{C})^{\Lambda^2_0}$, which we may regard as the  \emph{tautological} cospan. Therefore, it is enough to show that the pullback of $\tau_{f^\ast\I{C}}$ is preserved by $(h_{f^\ast\I{C}})_!(g^\prime)$. As this diagram is in context $f^\ast(\I{C}_1\times_{\I{C}_0}\I{C}_1)$, we may make use of the adjunction $f^\ast\dashv f_\ast$ to regard $\tau_{f^\ast\I{C}}$ as a cospan in $f_\ast f^\ast\I{C}$ in context $\I{C}_1\times_{\I{C}_0}\I{C}_1$. As such, it is precisely the cospan that arises as the image of the tautological cospan $\tau_{\I{C}}$ in $\I{C}$ (i.e.\ the one given by the core inclusion $\tau_{\I{C}}\colon\I{C}_1\times_{\I{C}_0}\I{C}_1\into \I{C}^{\Lambda^2_0}$) along $\eta\colon \I{C}\to f_\ast f^\ast\I{C}$.
	By again making use of Lemma~\ref{lem:BaseFunctorialityYonedaLemma}, we thus conclude that the image of $\tau_{f^\ast\I{C}}$ along $f_\ast (h_{f^\ast\I{C}})\colon f_\ast f^\ast \I{C}\into\IFun(\I{C}^\op, \I{X})$ can be identified with the image of $\tau_{\I{C}}$ along the composition $(\const_{\I{X}})_\ast\circ h_{\I{C}}$. In particular, the cospan $f_\ast (h_{f^\ast\I{C}})(\tau_{f^\ast\I{C}})$ is contained in the image of $(\const_{\I{X}})_\ast$, hence the above commutative diagram yields the claim.
\end{proof}

In the remainder of this section we will explain how our version of Diaconescu's theorem for $\infty$-topoi (Corollary~\ref{cor:Diaconescu}) relates to that of Raptis and Sch\"appi~\cite{Raptis2022} when $\BB$ is \emph{hypercomplete}. More precisely, in~\cite[Theorem 1.1 (3)]{Raptis2022} Raptis and Sch\"appi give an explicit characterisation of flat functors $\CC\to\XX$ valued in a hypercomplete $\infty$-topos $\XX$, and a prioi it is not clear how to relate this description to our substantially less explicit characterisation of flat functors in terms of internal filteredness (Proposition~\ref{prop:Flatifiltered}). Therefore, our goal is to recover the description in~\cite[Theorem 1.1 (3)]{Raptis2022} from Proposition~\ref{prop:Flatifiltered}.
To that end, suppose that there is a left exact and accessible localisation $ L \colon \PSh(\DD) \to \BB $ for some small $ \infty $-category $ \DD $, and let $i\colon\BB\into\PSh(\DD)$ be its right adjoint.
We denote by $\Under{\CC}{f}\to\CC$ the left fibration (in $\Cat(\BB)$) that is classified by $f\colon \CC\to\Univ$. By definition, it sits inside a pullback square
\[\begin{tikzcd}
	{\Under{\CC}{f}} & {\Under{\Univ}{1}} \\
	\CC & \Univ
	\arrow[from=2-1, "f" to=2-2]
	\arrow[from=1-2, to=2-2]
	\arrow[from=1-1, to=2-1]
	\arrow[from=1-1, to=1-2]
\end{tikzcd}\]
in $ \Cat(\BB) $. If $\BB$ is hypercomplete, we deduce from Propositions~\ref{prop:Flatifiltered} and~\ref{prop:Prefilteredimpliesfiltered} that $f$ being flat is equivalent to $(\Under{\CC}{f})^\op$ being quasi-filtered.
In order to obtain a more explicit understanding of the latter condition, let us first consider the constant presheaf $ \underline{\CC}\colon \DD^\op\to\CatS $ with value $\CC$ and compute the pullback
\[\begin{tikzcd}
	{\underline{\CC}_{f'/}} & {i(\Under{\Univ}{1})} \\
	{\underline{\CC}} & i(\Univ)
	\arrow[from=2-1, "f^\prime" to=2-2]
	\arrow[from=1-2, to=2-2]
	\arrow[from=1-1, to=2-1]
	\arrow[from=1-1, to=1-2]
\end{tikzcd}\]
in $ \Cat({\PSh(\DD)}) \simeq \Fun(\DD^\op,\Cat_\infty) $.
Here $ f^{\prime}$ is the transpose of $f\colon \CC\to\Univ$ across the adjunction $L\dashv i$. Note that $L(\Under{\underline{\CC}}{f^\prime})\simeq \Under{\CC}{f}$ since $L$ is left exact.
Upon evaluating the previous pullback square at any $ d \in \DD$, we obtain a commutative rectangle
\[\begin{tikzcd}
	{\underline{\CC}_{f^\prime/}}(d) & {\BB_{Ld/}} & \BB_{L(d)\sslash L(d)} \\
	\CC & \BB & {\BB_{/Ld}}
	\arrow["f", from=2-1, to=2-2]
	\arrow["\pi_{L(d)}^\ast", from=2-2, to=2-3]
	\arrow[from=1-3, to=2-3]
	\arrow[from=1-2, to=1-3]
	\arrow[from=1-1, to=2-1]
	\arrow[from=1-1, to=1-2]
	\arrow[from=1-2, to=2-2]
\end{tikzcd}\]
where the lower composite is is equivalent to $f^\prime(d) $ and all squares are pullback squares. Here ${\BB}_{L(d)\sslash L(d)}$ denotes the $\infty$-category of pointed objects in $\Over{\BB}{L(d)}$.
It follows that we can explicitly describe $ \underline{\CC}_{f^\prime/}(d) $ as the pullback in the left square such that for any map $s\colon d \to e $ in $ \DD $ the functor $s^\ast\colon\underline{\CC}_{f^\prime/}(e) \to \underline{\CC}_{f^\prime/}(d) $ is induced by pulling back the canonical functor $ s^\ast\colon\BB_{Le/} \to \BB_{Ld/} $ along $f$. To proceed, we now need the following lemma that characterises those $\CatS$-valued presheaves on $\DD$ which yield quasi-filtered $\BB$-categories upon sheafification:
\begin{lemma}
	\label{lem:PreFiltisCheckedinPsh}
	Let $ \I{C} \in \Cat(\PSh(\DD)) $. Then $ L \I{C} $ is a quasi-filtered $ \BB $-category if and only if for any  finite $ \infty $-category $ \KK $, any $ d \in \DD $ and any map $ \beta \colon \KK \to \I{C}(d)$ there exist morphisms $(s_i\colon d_i\to d)_i$ such that $ (L s_i) \colon \bigsqcup_i L(d_i)\onto L(d)$ is a cover in $ \BB $, and there are maps $ \alpha _i \colon \KK^\triangleright \to \I{C}(d_i)  $ for every $ i $ that fit into commutative diagrams
	\[\begin{tikzcd}
		{\KK^\triangleright} & {\I{C}(d_i)} \\
		\KK & {\I{C}(d)}
		\arrow["\beta", from=2-1, to=2-2]
		\arrow["{s_i^*}", from=2-2, to=1-2]
		\arrow[from=2-1, to=1-1]
		\arrow["{\alpha_i}"', from=1-1, to=1-2,swap]
	\end{tikzcd}\]
	of $\infty$-categories.
\end{lemma}
\begin{proof}
	The if part of the statement is a direct consequence of Proposition~\ref{prop:Localsectionsofsheafification}.
	For the converse we note that for any finite $ \infty $-category $ \KK $ the canonical map
	$ L\IFun[\PSh(\DD)](\KK, \I{C})^\simeq \to \IFun(\KK, L\I{C})^\simeq $ is an equivalence. 
	Now for some $ \beta \colon \KK \to \I{C} (d)$ corresponding via Yoneda's lemma to a morphism $ d \to \IFun(\KK, \I{C})^\simeq $ this shows that the projection map $ \pr_1 \colon d \times_{\IFun[\PSh(\DD)](\KK, \I{C})^\simeq} \IFun[\PSh(\DD)](\KK^\triangleright, \I{C})^\simeq \to d $ becomes a cover after applying $ L $. 
	We now pick a cover $ (t_i) \colon \bigsqcup_i d_i \to d \times_{\IFun[\PSh(\DD)](\KK, \I{C})^\simeq} \IFun[\PSh(\DD)](\KK^\triangleright, \I{C})^\simeq  $ in $ \PSh(\DD) $ by representables.
	Then the $ s_i = \pr_1 \circ t_i $ yield a cover after applying $ L $, and by Yoneda's lemma every $ s_i $ gives a commutative square as in the claim.
\end{proof}
By combining Lemma~\ref{lem:PreFiltisCheckedinPsh} with the discussion preceding it, we recover the following characterisation of flat functors in the hypercomplete case:

\begin{proposition}[{\cite[Definition 3.1 and Theorem 3.5]{Raptis2022}}]
	\label{prop:flatnessHypercompleteExplicitly}
	Suppose that $\BB$ is hypercomplete, and let $f\colon \CC\to\BB$ be a functor. Then $f$ is flat if and only if for every $d\in \DD$, every functor $\alpha\colon \KK\to\CC$ (where $\KK$ is a finite $\infty$-category) and every map $\overline{\beta}\colon\KK^\triangleleft\to \BB$ with cone point $L(d)$ such that $f\alpha\simeq\overline{\beta}\vert_{\KK}$, there are maps $(s_i\colon d_i\to d)_i$ in $\DD$ such that
	\begin{enumerate}
		\item $L(s_i)\colon\bigsqcup_i L(d_i)\onto d$ is a cover in $\BB$;
		\item for each $i$ there is a cocone $\overline{\alpha}_i\colon \KK^\triangleleft\to \CC$ extending $\alpha$, together with a morphism of cones $h\colon \Delta^1\diamond \KK\to \BB$  from the cocone $ \overline{\beta} \circ s_i$ (which is given by composing the cone point of $\beta$ with $s_i$) to $f \circ \overline{\alpha}_i$.
		\qed
	\end{enumerate}
\end{proposition}

\subsection{\'Etale $\BB$-topoi}
\label{sec:EtaleTopoi}
To prepare our discussion, note that Corollary~\ref{cor:UMPPoints} implies that the functor $(-)^{\disc}=\Univ^{(-)}\colon \Univ\into\ICat_{\BB}\to \IRTop_{\BB}$ from Definition~\ref{def:discreteBTopoi} is cocontinuous. Moreover, as this functor carries the final object $1_{\Univ}$ to $\Univ$ itself, the universal property of $\Univ$ implies that we have a functorial equivalence $(-)^{\disc}\simeq \Over{\Univ}{-}$. In particular, the functor $\Over{\Univ}{-}$ takes values in $\IRTop_{\BB}$ too. We may therefore define:
By Theorem~\ref{thm:BTopoiRelativeTopoi}, geometric morphisms $f_\ast\colon \XX\to\BB$ are in correspondence with $\BB$-topoi $f_\ast(\Univ[\XX])$. In this section, we study those $\BB$-topoi that correspond to \emph{\'etale} geometric morphisms.
\begin{definition}
	\label{def:etaleBTopos}
	A $\BB$-topos $\I{X}$ is \emph{\'etale} if there is an equivalence $\I{X}\simeq\Over{\Univ}{\I{G}}$ for some $\BB$-groupoid $\I{G}$.
\end{definition}

In~\cite[Proposition~6.3.5.5]{htt}, Lurie proved a universal property for \'etale geometric morphisms of $\infty$-topoi. In light of Theorem~\ref{thm:BTopoiRelativeTopoi}, such \'etale geometric morphisms precisely correspond to \'etale $\BB$-topoi. The main goal of ths section is to discuss how Lurie's result can also be deduced from Diaconescu's theorem. To that end,
note that if $\I{X}$ is a $\BB$-topos with associated geometric morphism $f_\ast\colon \XX\to\BB$ and if $\I{G}$ is a $\BB$-groupoid, the fact that we may identify $\I{X}\simeq f_\ast\Univ[\XX]$ implies that precomposition with the Yoneda embedding $h_{\I{G}}\colon \I{G}\into \IFun(\I{G}, \Univ)\simeq\Over{\Univ}{\I{G}}$ induces a map $\IFun(\Over{\Univ}{\I{G}},\I{X})\to f_\ast\IFun[\XX](f^\ast\I{G},\Univ[\XX])\simeq \Over{\I{X}}{\const_{\I{X}}\I{G}}$.
The universal property of \'etale $\BB$-topoi can now be formulated as follows:
\begin{proposition}
	\label{prop:UMPEtale}
	Let $\I{G}$ be a $\BB$-groupoid and let $\I{X}$ be a $\BB$-topos. Precomposition with the Yoneda embedding $h_{\I{G}}$ induces a fully faithful functor $h_{\I{G}}^\ast\colon\IFun^\alg(\Over{\Univ}{\I{G}},\I{X})\into\Over{\I{X}}{\const_{\I{X}}\I{G}}$ that fits into a pullback square
	\begin{equation*}
		\begin{tikzcd}
			\IFun^\alg(\Over{\Univ}{\I{G}},\I{X})\arrow[r, hookrightarrow, "h_{\I{G}}^\ast"] \arrow[d] & \Over{\I{X}}{\const_{\I{X}}\I{G}}\arrow[d, "(\pi_{\const_{\I{X}}\I{G}})_!"]\\
			1\arrow[r, hookrightarrow, "1_{\I{X}}"] & \I{X}.
		\end{tikzcd}
	\end{equation*}
	In particular, there is a canonical equivalence $\IFun^\alg(\Over{\Univ}{\I{G}},\I{X})\simeq\map{\I{X}}(1_{\I{X}}, \const_{\I{X}}\I{G})$.
\end{proposition}
The proof of Proposition~\ref{prop:UMPEtale} requires the following lemma:
\begin{lemma}
	\label{lem:YonedaEmbeddingGroupoid}
	For any $\BB$-groupoid $\I{G}$, the full embedding $\I{G}\into\Over{\Univ}{\I{G}}$ that is obtained by combining the Yoneda embedding $h_{\I{G}}$ the equivalence $\IFun(\I{G},\Univ)\simeq\Over{\Univ}{G}$ fits into a pullback square
	\begin{equation*}
		\begin{tikzcd}
			\I{G}\arrow[r, hookrightarrow] \arrow[d] & \Over{\Univ}{\I{G}}\arrow[d, "(\pi_{\I{G}})_!"]\\
			1\arrow[r, hookrightarrow, "1_{\Univ}"] & \Univ.
		\end{tikzcd}
	\end{equation*}
\end{lemma}
\begin{proof}
	Since we have a commutative diagram
	\begin{equation*}
		\begin{tikzcd}
			\Over{\Univ}{\I{G}}\arrow[r, "\simeq"]\arrow[dr, "(\pi_{\I{G}})_!"'] & \IFun(\I{G}, \Univ)\arrow[d, "\colim_{\I{G}}"]\\
			&\Univ,
		\end{tikzcd}
	\end{equation*}
	the claim follows once we show that a copresheaf $F\colon \I{G}\to\Univ$ is representable if and only if $\colim_{\I{G}} F\simeq 1_{\Univ}$. But $F$ is representable if and only if $\Over{\I{G}}{F}$ admits an initial object, and since the latter is a $\BB$-groupoid, this is in turn equivalent to $\Over{\I{G}}{F}\simeq 1$.  Since by~\cite[Proposition~4.4.1]{MWColimits} we have $\Over{\I{G}}{F}\simeq\colim_{\I{G}} F$, the claim follows.
\end{proof}

\begin{proof}[{Proof of Proposition~\ref{prop:UMPEtale}}]
	Let $f_\ast\colon\XX\to\BB$ be the geometric morphism that corresponds to the $\BB$-topos $\XX$.
	Since for every $U\in \XX$ an $\Over{\XX}{U}$-groupoid is filtered if and only if it is final (see Remark~\ref{rem:filteredContractible}), the Yoneda embedding $h_{f^\ast\I{G}}\colon f^\ast\I{G}\into\IFun[\XX](f^\ast\I{G},\Univ[\XX])$ induces an equivalence $f^\ast\I{G}\simeq\IFlat[\XX](f^\ast\I{G})$.
	By combining this observation with Theorem~\ref{thm:internalDiaconescu}, we thus find that precomposition with the Yoneda embedding $\I{G}\into\IFun(\I{G}, \Univ)$ yields an equivalence
	\begin{equation*}
		\IFun^\alg(\Over{\Univ}{\I{G}},\I{X})\simeq f_\ast(f^\ast\I{G}).
	\end{equation*}
	Hence the claim follows from Lemma~\ref{lem:YonedaEmbeddingGroupoid}.
\end{proof}

\begin{corollary}
	\label{cor:embeddingUniverseBTopoi}
	The functor $\Over{\Univ}{-}\colon \Univ\to \RPr_{\BB}$ factors through a cocontinuous and fully faithful embedding $\Over{\Univ}{-}\colon \Univ\into \IRTop_\BB$ whose essential image is spanned by the \'etale $\BB$-topoi.
\end{corollary}
\begin{proof}
	It is clear that this functor takes values in $\IRTop_{\BB}$, and by combining the descent property of $\Univ$ with Proposition~\ref{prop:LTopLimits}, this functor must be cocontinuous.
    It therefore suffices to show that it is fully faithful.
    As we have seen above, we may identify $\Over{\Univ}{-}$ with the restriction of the functor $(-)^{\disc}\colon \ICat_{\BB}\to \IRTop_{\BB}$ from \S~\ref{sec:CatEnrichementTopoi} along the inclusion $\Univ\into\ICat_{\BB}$. Using Corollary~\ref{cor:UMPPoints}, the claim thus follows once we show that for every $\BB$-groupoid $\I{G}$ the (partial) adjunction unit $\I{G}\to\IPt_{\BB}(\I{G}^{\disc})$ is an equivalence. By construction, this map is obtained by the transpose of the evaluation map $\ev\colon \I{G}\times\IFun(\I{G}, \Univ)\to\Univ$, which by Yoneda's lemma is precisely the inverse of the equivalence from Proposition~\ref{prop:UMPEtale}. This finishes the proof.
\end{proof}

\begin{remark}
	\label{rem:selfIndexingLex}
	The functor $\Over{\Univ}{-}\colon\Univ\into \IRTop_{\BB}$ also preserves finite limits. In fact, this is clear for the final object, and the case of binary products is an immediate consequence of the formula from Example~\ref{ex:PushoutPresheafCategory} (together with the fact that the \'etale base change of this functor along $\pi_A^\ast$ recovers the functor $\Over{(\Univ[\Over{\BB}{A}])}{-}$). This is already enough to deduce that $\Over{\Univ}{-}$ preserves pullbacks: in fact, since Corollary~\ref{cor:BTopRelTopFunctorial} and Corollary~\ref{cor:groupoidalDescentEtaleTransitionMaps} imply that $\IRTop_{\BB}$ has $\Univ$-descent, this follows from the argument in the second part of the proof of Lemma~\ref{lem:YonedaExtensionPullbackPreservation}.
\end{remark}

\subsection{Subterminal \texorpdfstring{$\BB$}{B}-topoi}
\label{sec:sheafification}
The goal of this section is to study \emph{subterminal} $\BB$-topoi. To begin with, observe that if $f_\ast\colon \XX\to\BB$ and $g_\ast\colon \YY\to\BB$ are geometric morphism where $f_\ast$ is fully faithful, then the formula that we derived in \S~\ref{sec:coproductTopoi} immediately implies that the geometric morphism $\XX\times_{\BB}\YY\to \YY$ (whose domain is the pullback in $\RTopS$) is fully faithful as well. Thus, we may define:
\begin{definition}
	\label{def:subterminalBTopos}
	A $\BB$-topos $\I{X}$ is said to be \emph{subterminal} if the global sections functor $\Gamma_{\I{X}}$ is fully faithful, or equivalently if the associated geometric morphism $f_\ast\colon \XX\to\BB$ is fully faithful.
\end{definition}
By Theorem~\ref{thm:BTopoiRelativeTopoi}, any subterminal $\BB$-topos $\I{X}$ determines and is determined by a left exact and accessible Bousfield localisation of $\BB$ and therefore in particular by a class of maps $S$ in $\BB$ for which $\Gamma(\I{X})\simeq\Loc_S(\BB)$. The main goal of this section is to characterise those collections of maps $S$ that arise from and give rise to a subterminal $\BB$-topos $\I{X}$ in this way, and to describe the associated endofunctor
\begin{equation*}
	\Gamma_{\I{X}}\const_{\I{X}}\colon\Univ\to\Univ
\end{equation*} 
by an explicit colimit formula in terms of $S$, akin to Lurie's sheafification formula from~\cite[\S~6.2.2]{htt}. 

We begin with the following definition:
\begin{definition}
	\label{def:plusConstruction}
	Let $d\colon \I{I}\to\Univ$ be a functor of $\BB$-categories, where $\I{I}$ is small.
	We define the \emph{$+$-construction} $(-)^{+}_{d}\colon \Univ\to\Univ$ relative to $d$ as the composition
	\begin{equation*}
		\Univ\xhookrightarrow{h_{\Univ}}\IPSh(\Univ)\xrightarrow{d^\ast} \IPSh(\I{C})\xrightarrow{\colim_{\I{C}^\op}}\Univ,
	\end{equation*}
	i.e.\ by the formula $(-)^+_{d} =\colim_{\I{I}^\op}\map{\Univ}(d(-),-)$.
\end{definition}

\begin{remark}
	\label{rem:PlusConstructionLex}
	If $\I{I}$ is \emph{cofiltered}, i.e.\ if $\I{I}^\op$ is filtered, then the  $+$-construction $(-)^+_d$ is left exact.
\end{remark}

\begin{remark}
	\label{rem:PlusConstructionUnit}
	If $\I{I}$ is cofiltered, then the diagonal functor $\diag_{\I{I}^\op}\colon\Univ\to\IPSh(\I{I})$ is fully faithful (which follows from $\I{I}$ being weakly contractible, see Remark~\ref{rem:filteredContractible}, and from the explicit formula of the colimit in $\Univ$ from~\cite[Proposition~4.4.1]{MWColimits}). Therefore, by applying the limit functor $\lim_{\I{I}^\op}\colon\IPSh(\I{I})\to\Univ$ to the adjunction unit $\id\to \diag_{\I{I}^\op}\colim_{\I{I}^\op}$, we end up with a natural map $\lim_{\I{I}^\op}\to\colim_{\I{I}^\op}$.
	Now suppose furthermore that the colimit of $d\colon \I{I}\to\Univ$ is the final object $1_{\Univ}\colon 1_{\BB}\to\Univ$. Then the composition
	\begin{equation*}
		\Univ\xhookrightarrow{h_{\Univ}}\IPSh(\Univ)\xrightarrow{d^\ast}\IPSh(\I{I})\xrightarrow{\lim_{\I{I}^\op}} \Univ
	\end{equation*}
	is equivalent to the identity: in fact, this follows from the observation that its left adjoint is given by the composition of $\diag_{\I{I}^\op}$ with the Yoneda extension of $d$ (see~\cite[Remark~7.1.4]{MWColimits}) and therefore preserves final objects. Thus, we obtain a natural map $\phi\colon\id_{\Univ}\to (-)^+_d$.
\end{remark}

To proceed, let us fix a (small) cofiltered $\BB$-category $\I{I}$ and a functor $d\colon \I{I}\to\Univ$ whose colimit is the final object. Since $\I{I}$ is small, there is a $\BB$-regular cardinal $\kappa$ such that the essential image of $d$ is contained in the full subcategory $\Univ[\BB]^\kappa\into\Univ$ determined by the local class of relatively $\kappa$-compact objects in $\BB$ (cf.~Proposition~\ref{prop:relativelyKappaCompactLocal}). We will call such a $\BB$-regular cardinal $\kappa$ \emph{adapted to $d$}. We will identify $\kappa$ with the linearly ordered set of ordinals $< \kappa$. Using transfinite induction, we may now construct a diagram $T_\bullet^d\colon \kappa\to\IFun(\Univ,\Univ)$ by setting $T_0^d= \id$, by defining the map $T_\tau\to T_{\tau + 1}$ to be the morphism $\phi\colon T_\tau^d\to (T_\tau^d)^+_d$ from Remark~\ref{rem:PlusConstructionUnit} and finally by setting $T_{\tau}^d= \colim_{\tau^\prime < \tau} T_{\tau^\prime}^d$ whenever $\tau$ is a limit ordinal.
\begin{definition}
	\label{def:sheafification}
	Let $d\colon \I{I}\to\Univ$ be a functor whose colimit is the final object and whose domain is a cofiltered small $\BB$-category.
	We define the \emph{sheafification functor} $(-)^{\sh}_d$ relative to the functor $d\colon \I{I}\to\Univ$ as the colimit $(-)^{\sh}_d = \colim_{\tau < \kappa} T_{\tau}^d$ in $\IFun(\Univ,\Univ)$, where $\kappa$ is an arbitrary $\BB$-regular cardinal that is adapted to $d$.
\end{definition}
\begin{remark}
	A priori, the sheafification functor $(-)^\sh_d$ depends on the choice of $\BB$-regular cardinal $\kappa$. However, since $d$ takes values in $\Univ^\kappa$ and therefore in $\kappa$-compact objects in $\Univ$ (see Corollary~\ref{cor:kappaCompactObjectsUniverse}), and since $\kappa$ (when viewed as a linearly ordered set) is $\kappa$-filtered, one can show that whenever $\tau \geq \kappa$ is another $\BB$-regular cardinal, the sheafification functor that is constructed with respect to $\tau$ is equivalent to the one constructed with respect to $\kappa$.
\end{remark}
\begin{remark}
	\label{rem:sheafificationLeftExact}
	In the situation of Definition~\ref{def:sheafification}, the sheafification functor $(-)^{\sh}_d$ is left exact since by Remark~\ref{rem:PlusConstructionUnit} it is a filtered colimit of left exact functors (see the argument in the proof of Lemma~\ref{lem:stabilityFiltUCocontinuousFunctors}).
\end{remark}

\begin{example}
	\label{ex:sheafificationLocalClass}
	Let $S$ be a bounded local class of morphisms in $\BB$ which is closed under finite limits in $\Fun(\Delta^1,\BB)$, and let $\iota\colon\Univ[S]\into\Univ$ be the associated inclusion. Then $\Univ[S]$ is small and closed under finite limits in $\Univ$. In particular, $\Univ[S]$ is cofiltered by Proposition~\ref{prop:FinBRegular} and contains the final object of $\Univ$, so that the sheafification functor $(-)^{\sh}_{\iota}$ is well-defined.
\end{example}

\begin{example}
	\label{ex:sheafificationInverseImageLocalClass}
	Let $f_\ast\colon \XX\to\BB$ be a geometric morphism, and let $S$ and $\iota$ be as in Example~\ref{ex:sheafificationLocalClass}. Then the functor $\const_{ f_\ast(\Univ[\XX])}\iota\colon \Univ[S]\to f_\ast(\Univ[\XX])$ transposes to a map $\iota^\prime\colon f^\ast(\Univ[S])\to\Univ[\XX]$ of $\XX$-categories. As $\const_{ f_\ast(\Univ[\XX])}\iota$ preserves the final object, its colimit is $1_{f_\ast(\Univ[\XX])}$, hence the colimit of $\iota^\prime$ is the final object as well. Moreover, the fact that $\Univ[S]$ is a cofiltered $\BB$-category implies that $f^\ast\Univ[\XX]$ is a cofiltered $\XX$-category: in fact, the colimit functor $\colim_{f^\ast(\Univ[S])^\op}\colon \IPSh[\XX](f^\ast(\Univ[S]))\to\Univ[\XX]$ preserves finite limits if and only if the underlying functor of $\infty$-categories  $\PSh_{\XX}(f^\ast(\Univ[S]))\to\XX$ preserves finite limits, and as the latter can be identified with the global sections of
	\begin{equation*}
		{\colim}_{\Univ[S]^\op}\colon\IFun(\Univ[S]^\op, f_\ast(\Univ[\XX]))\to f_\ast(\Univ[\XX]),
	\end{equation*}
	the claim follows from the fact that filtered colimits commute with finite limits in every $\BB$-topos (which one can see by reducing to the case of a presheaf $\BB$-topos where it readily follows from the definitions). Thus, we are in the situation of Definition~\ref{def:sheafification}, so that $(-)^{\sh}_{\iota^\prime}$ is well-defined.
\end{example}

\begin{construction}
	\label{constr:subcatFromLocalClass}
	Suppose that $S$ is a bounded local class of maps in $\BB$. Since $S$ is bounded, there is a $\BB$-regular cardinal $\kappa$ that is adapted to $\iota\colon \Univ[S]\into\Univ$. Let $S^\kappa\subset S$ be the class of maps in $S$ between $\kappa$-compact objects. We let $E\into\Univ_1$ be the subobject that is spanned by the maps $f\colon P\to Q$ in $\Over{\BB}{A}$ (for arbitrary $A\in\BB$) for which $(\pi_A)_!(f)$ is contained in $S$ and the two maps $P\to A$ and $Q\to A$ are relatively $\kappa$-compact. Since both $S$ and the class of relatively $\kappa$-compact maps are local, a map $f$ in $\Univ$ in context $A$ is contained in $E$ if and only if $(\pi_A)_!(f)\in S$ and both $P\to A$ and $Q\to A$ are relatively $\kappa$-compact. In particular, $E$ is small. We define $\I{W}\into\Univ$ as the subcategory that is generated by $E$ in the sense of~\cite[Appendix~B]{MWColimits}. Note that as $E$ is small, the subcategory $\I{W}$ is small as well.
\end{construction}

We can now state the first main result of this section:
\begin{proposition}
	\label{prop:localClassSubtoposSheafification}
	Let $S$ be a bounded local class of morphisms in $\BB$ such that $S$ is closed under finite limits in $\Fun(\Delta^1,\BB)$. Let $\I{W}\into\Univ$ be as in Construction~\ref{constr:subcatFromLocalClass}. Then $\I{X}=\ILoc_{\I{W}}(\Univ)$ is a subterminal $\BB$-topos with the property that $\Gamma(\I{X})\simeq\Loc_S(\BB)$. Moreover, the adjunction unit $\eta\colon \id \to \Gamma_{\I{X}}\const_{\I{X}}$ can be identified with the map $\id\to (-)^{\sh}_\iota$, where $\iota\colon\Univ[S]\into\Univ$ is the inclusion. 
\end{proposition}

\begin{lemma}
	\label{lem:subcatFromLocalClassLocalisation}
	Let $S$ be a bounded local class of morphisms in $\BB$. Let $\I{W}\into\Univ$ be as in Construction~\ref{constr:subcatFromLocalClass}. Then there is an equivalence
	\begin{equation*}
		\Gamma(\ILoc_{\I{W}}(\Univ))\simeq\Loc_{S}(\BB)
	\end{equation*}
	of full subcategories in $\BB$.
\end{lemma}
\begin{proof}
	By Corollary~\ref{cor:presentableCategoryBousfieldLocalisation}, the inclusion $i\colon \ILoc_{\I{W}}(\Univ)\into\Univ$ admits a left adjoint $L$ that exhibits $\ILoc_{\I{W}}(\Univ)$ as an accessible Bousfield localisation of $\Univ$.
	Note that every map in $S$ can be written as a colimit of maps in $S^\kappa$, which implies that every map in $S$ is inverted by $L$. Consequently, we have an inclusion $\Gamma(\ILoc_{\I{W}}(\Univ))\into\Loc_S(\BB)$, so that the claim follows once we verify that every $S$-local object $G\in\BB$ is $\I{W}$-local. This amounts to showing that for every $A\in\BB$ and
	and every map $s\colon P\to Q$ in $\Over{\BB}{A}$ for which $(\pi_A)_!(f)\in S$ and both $P\to A$ and $Q\to A$ are relatively $\kappa$-compact, the map
	\begin{equation*}
		s^\ast\colon \map{\Univ}(Q, \pi_A^\ast G)\to \map{\Univ}(P, \pi_A^\ast G)
	\end{equation*}
	is an equivalence (cf.~Remark~\ref{rem:LocalObjectsGenerators}). By~\cite[Proposition~4.4.11]{MWColimits}, we may identify this morphism with the map
	\begin{equation*}
		\Hom_{\Over{\BB}{A}}(Q, \pi_A^\ast G)\to \Hom_{\Over{\BB}{A}}(P, \pi_A^\ast G).
	\end{equation*}
	By evaluating the latter at any object $B\to A$ in $\Over{\BB}{A}$, we recover the morphism
	\begin{equation*}
		\map{\BB}(Q\times_{A}B, G)\to\map{\BB}(P\times_A B, G),
	\end{equation*}
	which is indeed an equivalence as the maps in $S$ are closed under base change. Hence $G$ is $\I{W}$-local, as claimed.
\end{proof}
Before we can prove Proposition~\ref{prop:localClassSubtoposSheafification}, we first need to make a few remarks on the \emph{internal hom} of a $\BB$-topos $\I{X}$. Recall from Proposition~\ref{prop:universalityOfColimitsCartesianClosed} that colimits being universal in $\I{X}$ precisely means that $\I{X}$ is \emph{cartesian closed}. We denote by
\begin{equation*}
	\Hom_{\I{X}}(-,-)\colon\I{X}^\op\times\I{X}\to\I{X}
\end{equation*}
the internal hom of $\I{X}$ that results from this observation. Note that if $f_\ast\colon\XX\to\BB$ is the geometric morphism associated with $\I{X}$, we deduce from combining~\cite[Proposition~4.4.10]{MWColimits} with~\cite[Corollary~3.1.9]{MWColimits} and~\cite[Lemma~5.2.1]{MCocartesian} that $\Hom_{\I{X}}(-,-)$ is explicitly given by the image of the bifunctor of $\XX$-categories
\begin{equation*}
	\map{\Univ[\XX]}(-,-)\colon\Univ[\XX]^\op\times\Univ[\XX]\to\Univ[\XX]
\end{equation*} 
along $f_\ast$. 
\begin{remark}
	\label{rem:internalHomGlobalSections}
	If $\I{X}$ is a $\BB$-topos, then the composition $\Gamma_{\I{X}}\circ\Hom_{\I{X}}(-,-)$ recovers the mapping bifunctor $\map{\I{X}}(-,-)$. In fact, as Remark~\ref{rem:GlobalSectionsExplicitly} implies that $\Gamma_{\I{X}}$ is corepresented by $1_{\I{X}}$, we deduce that there is a pullback square
	\begin{equation*}
		\begin{tikzcd}
			\Under{\I{X}}{1_{\I{X}}}\arrow[r]\arrow[d] & \Under{(\Univ[\BB])}{1_{\Univ[\BB]}}\arrow[d]\\
			\I{X}\arrow[r, "\Gamma_{\I{X}}"] & \Univ[\BB].
		\end{tikzcd}
	\end{equation*}
	On the other hand,~\cite[Lemma~5.2.1]{MCocartesian} implies that there also is a pullback square
	\begin{equation*}
		\begin{tikzcd}[column sep=large]
			\ITw(\I{X})\arrow[r]\arrow[d, "p_{\I{X}}"] & \Under{\I{X}}{1_{\I{X}}}\arrow[d]\\
			\I{X}^\op\times\I{X}\arrow[r, "{\Hom_{\I{X}}^{\BB}(-,-)}"] & \I{X}.
		\end{tikzcd}
	\end{equation*}
	By pasting these two pullback squares together, the claim follows.
\end{remark}

\begin{proof}[{Proof of Proposition~\ref{prop:localClassSubtoposSheafification}}]
	Let $\kappa$ and $\I{W}\into\Univ$ be as in Construction~\ref{constr:subcatFromLocalClass}, and let us denote by
	\begin{equation*}
		(L\dashv i)\colon \Univ\leftrightarrows\ILoc_{\I{W}}(\Univ)
	\end{equation*}
	the associated Bousfield localisation provided by Corollary~\ref{cor:presentableCategoryBousfieldLocalisation}. In light of Lemma~\ref{lem:subcatFromLocalClassLocalisation}, we only need to show that $\ILoc_{\I{W}}(\Univ)$ is a subterminal $\BB$-topos and to identify the adjunction unit $\id\to iL$ with the sheafification map $\id\to (-)^{\sh}_\iota$. In light of Theorem~\ref{thm:presentationTopoi} and Example~\ref{ex:sheafificationLocalClass}, the former claim is implied by the latter one, so that the proof is finished once we identify $\id\to iL$ with $\id\to (-)^{\sh}_\iota$.
	
	We only need to show that for every object $G\colon A\to \Univ$ in arbitrary context $A\in\BB$, the object $G^{\sh}_\iota$ is $\I{W}$-local and the map $G\to G^{\sh}_\iota$ is inverted by the localisation functor $L$. Note that as $\BB$ is generated by its $\kappa$-compact objects (as $\kappa$ is assumed to be $\BB$-regular), we may assume that $A$ is $\kappa$-compact. In this case, note that $\kappa$ is also $\Over{\BB}{A}$-regular and adapted to $\pi_A^\ast(\iota)$ (by Remark~\ref{rem:BCkappaSmall}) and that we may identify the base change of $(-)^+_\iota$ along $\pi_A^\ast$ with $(-)^+_{\pi_A^\ast(\iota)}\colon \Univ[\Over{\BB}{A}]\to\Univ[\Over{\BB}{A}]$. Therefore, we may also identify the base change of $(-)^{\sh}_\iota$ along $\pi_A^\ast$ with $(-)^{\sh}_{\pi_A^\ast(\iota)}$. Together with Remark~\ref{rem:BCLocalObjects}, this implies that we may replace $\BB$ with $\Over{\BB}{A}$ and $G$ with its transpose $\hat G\colon 1_{\Over{\BB}{A}}\to \Univ[\Over{\BB}{A}]$, so that we may assume without loss of generality that $A\simeq 1_{\BB}$.
	
	We first show that the map $G\to G^{\sh}_{\iota}$ is inverted by $L$, for which it will be enough to show that the map $\phi\colon G\to G^{+}_{\iota}$ is inverted by $L$, or equivalently that the map
	\begin{equation*}
		\phi^\ast\colon\map{\Univ}(G^{+}_{\iota}, i(-))\to \map{\Univ}(G, i(-))
	\end{equation*}
	is an equivalence. Note that by the triangle identities, the map $\lim_{\Univ[S]^\op}\to \colim_{(\Univ[S])^\op}$ can be identified with the composition 
	\begin{equation*}
		\lim_{\Univ[S]^\op}\simeq \colim_{\Univ[S]^\op}\diag_{\Univ[S]^\op}\lim_{\Univ[S]^\op}\to\colim_{\Univ[S]^\op}
	\end{equation*}
	where the first map is induced by the inverse of counit of $\colim_{(\Univ[S])^\op}\dashv \diag_{\Univ[S]^\op}$ and the second map is induced by the counit of $\diag_{\Univ[S]^\op}\dashv\lim_{\Univ[S]^\op}$. Moreover, observe that as $\lim_{\Univ[S]^\op}$ is given by evaluation at the final object $1_{\Univ}\colon 1\to\Univ[S]$, this functor is cocontinuous and therefore given by the left Kan extension of its restriction along the Yoneda embedding $h_{\Univ[S]}\colon\Univ[S]\into\IPSh(\Univ[S])$. As the restriction of $\lim_{\Univ[S]^\op}$ along $h_{\Univ[S]}$ can be identified with $\map{\Univ[S]}(1_{\Univ},-)$ and the latter is equivalent to the inclusion $\iota$, it follows that the left adjoint of $\lim_{\Univ[S]^\op}$ is given by $\iota^\ast h_{\Univ}$ (see~\cite[Remark~7.1.4]{MWColimits}). Altogether, these observations imply that we may decompose $\phi^\ast$ into the chain of morphisms
	\begin{align*}
		\map{\Univ}(G^+_{\iota}, i(-)) &\simeq \map{\Univ}(\colim_{\Univ[S]^\op}\iota^\ast h_{\Univ}(G), i(-))\\
		&\simeq \map{\IPSh(\Univ[S])}(\iota^\ast h_{\Univ}(G), \diag_{\Univ[S]^\op} i(-))\\
		&\to \map{\IPSh(\Univ[S])}(\iota^\ast h_{\Univ}(G), \iota^\ast h_{\Univ} \lim_{\Univ[S]^\op}\diag_{\Univ[S]^\op} i(-))\\
		&\simeq \map{\Univ}(G, i(-))
	\end{align*}
	in which the penultimate map is induced by the adjunction unit $\id \to \iota^\ast h_{\Univ}\lim_{\Univ[S]^\op}$ and where the last equivalence follows from both $\diag_{\Univ[S]^\op}$ and $\iota^\ast h_{\Univ}$ being fully faithful functors.
	Hence, it suffices to show that the map
	\begin{equation*}
		\diag_{\Univ[S]^\op}i\to \iota^\ast h_{\Univ} \lim_{\Univ[S]^\op}\diag_{\Univ[S]^\op} i
	\end{equation*} 
	is an equivalence, i.e.\ that $\diag_{\Univ[S]^\op} i$ takes value in the essential image of $\iota^\ast h_{\Univ}$. 
	To see this, note that since the restriction of the localisation functor $L\colon\Univ\to\ILoc_{\I{W}}(\Univ)$ to $\Univ[S]$ factors through the inclusion $1_{\Univ}\colon 1\into \ILoc_{\I{W}}(\Univ)$, it follows that
	we may identify $\iota^\ast h_{\Univ}i \simeq\map{\Univ}(\iota(-), i(-))$
	with the transpose of the composition
	\begin{align*}
		\Univ[S]^\op\times\ILoc_{\I{W}}(\Univ)\xrightarrow{\pr_1} \ILoc_{\I{W}}(\Univ)\xrightarrow{\map{\ILoc_{\I{W}}}(1_{\Univ},-)} \Univ,
	\end{align*}
	which is precisely $\diag_{\Univ[S]^\op}i$, as desired.
	
	We finish the proof by showing that $G^{\sh}_{\iota}$ is $\I{W}$-local. By the same reduction steps as above, it is enough to show that for every map $s\colon P\to Q$ in $S$ between $\kappa$-compact objects,  the map $s^\ast\colon \map{\Univ}(Q,G^{\sh}_{\iota})\to \map{\Univ}(P, G^{\sh}_{\iota})$ is an equivalence. Note that in light of the adjunction $(\pi_Q)_!\dashv \pi_Q^\ast\colon \Over{\Univ}{Q}\leftrightarrows \Univ$, the map $s^\ast$ can be interpreted as the morphism $s^\ast \colon \map{\Over{\Univ}{Q}}(Q, \pi_Q^\ast G^{\sh}_{\iota})\to \map{\Over{\Univ}{Q}}(P, \pi_Q^\ast G^{\sh}_{\iota})$. By Remark~\ref{rem:internalHomGlobalSections}, we can identify $\map{\Over{\Univ}{Q}}(-,-)$ with the global sections of the internal hom $\Hom_{\Over{\Univ}{Q}}(-,-)$. Hence we may as well show that the map 
	\begin{equation*}
		s^\ast\colon \Hom_{\Over{\Univ}{Q}}(Q, \pi_Q^\ast G^{\sh}_{\iota})\to \Hom_{\Over{\Univ}{Q}}(P, \pi_Q^\ast G^{\sh}_{\iota})
	\end{equation*}
	is an equivalence. In other words, by replacing $\BB$ with $\Over{\BB}{Q}$, we can reduce to the case where $Q\simeq 1$.
	Thus, what is left to show is that $t^\ast\colon G^{\sh}_{\iota}\to \map{\Univ}(P, G^{\sh}_{\iota})$ is an equivalence for every $\kappa$-compact $P\colon 1\to\Univ[S]$, where $t\colon G\to 1_{\Univ}$ is the terminal map. Note that by Corollary~\ref{cor:kappaCompactObjectsUniverse}, the fact that $P$ is $\kappa$-compact even implies that $P$ is  $\ICat_{\BB}^\kappa$-compact when viewed as an object of $\Univ$. Hence, the map $t^\ast$ can be identified with the colimit
	\begin{equation*}
		\colim_{\tau < \kappa}t^\ast_\tau\colon \colim_{\tau <\kappa}T_{\tau}^{\iota}G\to 
		\colim_{\tau < \kappa}\map{\Univ}(P, T_\tau^{\iota} G).
	\end{equation*}
	To show that this map is an equivalence, it will be sufficient to prove that for every ordinal $\tau < \kappa$ the map $\colim_{n\in\mathbb N} t^\ast_{\tau + n}$ is one. To see the latter claim, observe that for every $H\in\BB$ we have a commutative diagram
	\begin{equation*}
		\begin{tikzcd}[column sep=tiny]
			H\arrow[r, "t^\ast"]\arrow[d] &\map{\Univ}(P, H) \arrow[d, "\alpha"]\\ 
			\colim_{\Under{(\Univ[S]^\op)}{P}} \map{\Univ}(\iota(\pi_P)_!(-), H)\arrow[r, "\simeq"]  \arrow[d, "\simeq"] &\colim_{(\Univ[S])^\op} \map{\Univ}(\iota(-), H) \arrow[d, "\beta"] \\
			\colim_{(\Univ[S])^\op} \map{\Univ}(\iota(-), H) \arrow[r, "t^\ast"] &\map{\Univ}(P, \colim_{(\Univ[S])^\op}\map{\Univ}(\iota(-), H)).
		\end{tikzcd}
	\end{equation*}
	Here the composition of the two vertical maps on the left and right can be identified by $\phi$ and $\phi_\ast$, respectively. Moreover, the equivalences in this diagram are induced by the initial map $(\pi_P)_!\colon \Over{(\Univ[S])}{P}\to\Univ[S]$, the map $\alpha$ is induced by $P\colon 1\to \Univ[S]$ and $\beta$ is given by the composition
	\begin{align*}
		\colim_{(\Univ[S])^\op} \map{\Univ}(\iota(-), H)&\to \colim_{(\Univ[S])^\op} \map{\Univ}(P, \map{\Univ}(\iota(-), H)) \\
		&\to \map{\Univ}(P, \colim_{(\Univ[S])^\op}\map{\Univ}(\iota(-), H)).
	\end{align*}
	By substituting $H= T_{\tau + n}^{\iota}G$ for any $n\in\mathbb N$, we deduce that the map $t^\ast_{\tau + n}\to t^\ast_{\tau + n + 1}$ is a retract of an equivalence and must therefore be an equivalence itself.
\end{proof}
We finish this section with the following converse of Proposition~\ref{prop:localClassSubtoposSheafification}:
\begin{proposition}
	\label{prop:SubtoposLocalClass}
	Let $\I{X}$ be a subterminal $\BB$-topos. Then there is a bounded local class $S$ that is closed under finite limits in $\Fun(\Delta^1,\BB)$ such that $\Gamma(\I{X})\simeq\Loc_S(\BB)$. Moreover, for \emph{any} such local class $S$, we obtain an equivalence $\I{X}\simeq\ILoc_{\I{W}}(\Univ)$, where $\I{W}$ is as in Construction~\ref{constr:subcatFromLocalClass}, and the adjunction unit $\eta\colon \id \to \Gamma_{\I{X}}\const_{\I{X}}$ can be identified with the map $\id\to (-)^{\sh}_\iota$, where $\iota\colon\Univ[S]\into\Univ$ is the inclusion. 
\end{proposition}
\begin{proof}
	Let us denote by $j_\ast\colon\XX\into\BB$ the geometric morphism associated with $\I{X}$.
	We begin by proving the second statement, i.e.\ suppose that $S$ is a bounded local class that is closed under finite limits in $\Fun(\Delta^1,\BB)$ such that we have $\XX\simeq\Loc_S(\BB)$. Let $\I{W}\into\Univ$ be as in Construction~\ref{constr:subcatFromLocalClass}. By Lemma~\ref{lem:subcatFromLocalClassLocalisation}, we may identify $\Gamma(\ILoc_{\I{W}}(\Univ))$ with $\XX$. As Proposition~\ref{prop:localClassSubtoposSheafification} moreover implies that $\ILoc_{\I{W}}(\Univ)$ is a subterminal $\BB$-topos, Theorem~\ref{thm:BTopoiRelativeTopoi} implies that we must necessarily have $\ILoc_{\I{W}}(\Univ)\simeq\I{X}$. Hence the same proposition gives rise to the desired identification of the adjunction unit $\eta\colon \id\to\Gamma_{\I{X}}\const_{\I{X}}$.
	
	To complete the proof, it is therefore enough to show that there always exists a bounded local class $S$ that is closed under finite limits in $\Fun(\Delta^1,\BB)$ such that $\XX\simeq\Loc_S(\BB)$. To that end, choose a $\BB$-regular cardinal $\kappa$ for which $\Gamma_{\I{X}}$ is $\IFilt_{\ICat_{\BB}^\kappa}$-cocontinuous. We let $S$ be the class of relatively $\kappa$-compact maps in $\BB$ that are inverted by $j^\ast$. Since by Proposition~\ref{prop:relativelyKappaCompactLocal} the class of relatively $\kappa$-compact maps in $\BB$ is local, we find that $S$ is local as well. Moreover, $S$ is closed under finite limits in $\Fun(\Delta^1,\BB)$ as $j^\ast$ is left exact and as $\kappa$-compact objects in $\BB$ are closed under finite limits (by choice of $\kappa$). Since $S$ is inverted by $j$, we already have an inclusion $\XX\into\Loc_S(\BB)$, so that it suffices to prove that every $S$-local object in $\BB$ is contained in $\BB$. Since $\XX$ is a $\kappa$-accessible localisation of $\BB$ (using Proposition~\ref{prop:CocompleteKappaCocontinuousInternalExternal}) and $\BB$ itself is $\kappa$-accessible, we deduce from the proof of~\cite[Proposition~5.5.4.2]{htt} (or alternatively the proof of Proposition~\ref{prop:BousfieldLocalisationLocalObjects} applied in the case $\BB=\SS$) that $\XX$ is the Bousfield localisation at the class $S^\prime$ of those maps in $\BB$ between $\kappa$-compact objects which are inverted by $j$. Since every such map must be relatively $\kappa$-compact (using that $\kappa$-compact objects are closed under finite limits in $\BB$), every such map is contained in $S$. Hence the claim follows.
\end{proof}

\begin{remark}
	\label{rem:pullbackSubtopos}
	We can use our understanding of subterminal $\BB$-topoi to obtain a quite explicit understanding of pushouts in $\LTopS$ in which one of the two maps is a Bousfield localisation. In fact, suppose that $f_\ast\colon\XX\to\BB$ and $i_\ast\colon\ZZ\into\BB$ be geometric morphisms, where $i_\ast$ is fully faithful. 
	Then $i_\ast(\Univ[\ZZ])$ is a subterminal $\BB$-topos, so that Proposition~\ref{prop:SubtoposLocalClass} implies that we can find a bounded local class $S$ that is closed under compositions and finite limits in $\Fun(\Delta^1,\BB)$ such that $\ZZ=\Loc_S(\BB)$. Let $ \overline{f^*S} $ denote the smallest local class of maps in $\XX$ that contains $f^\ast(S)$. Then we claim that the functor $j_\ast\colon \ZZ\times_{\BB}\XX\into\XX$ exhibits $\ZZ\times_{\BB}\XX$ as the Bousfield localisation of $\XX$ at $\overline{f^\ast S}$. To see this, note that Proposition~\ref{prop:coproductsBTopoi} and Lemma~\ref{lem:LocalisationOfTensorProds} imply that the morphism in $\IRTop_{\BB}$ corresponding to $j_\ast$ is given by $\ILoc_{ \I{W} \boxtimes f_\ast(\Univ[\XX])^\simeq }(\Univ \otimes f_\ast(\Univ[\XX]))\into f_\ast(\Univ[\XX]) $, where $\I{W}\into\Univ$ is the subcategory from Construction~\ref{constr:subcatFromLocalClass}. Since $ \Univ \otimes f_\ast(\Univ[\XX]) \simeq f_\ast(\Univ[\XX]) $, the left-hand side can be identified with the full subcategory of local objects with respect to
	\begin{equation*}
		\I{W}\times f_\ast(\Univ[\XX])^\core \xrightarrow{\const_{i_\ast\Univ[\XX]}(-)\times -} f_\ast(\Univ[\XX]).
	\end{equation*} 
	By the same argument as in the proof of Lemma~\ref{lem:subcatFromLocalClassLocalisation}, this means that an object $U\in \XX$ is contained in $\ZZ\times_{\BB}\XX$ if and only if it is local with respect to every map in $\XX$ of the form
	\begin{equation*}
		f^\ast(s)\times_{f^\ast A} X\colon f^\ast(P)\times_{f^\ast A} X\to f^\ast(Q)\times_{f^\ast A}X
	\end{equation*}
	where $s\colon P\to Q$ is a map in $\I{W}(A)$ and $X$ is an arbitrary object in $\Over{\XX}{f^\ast A}$. By construction of $\I{W}$, the map $(\pi_A)_!(f)$ is contained in $S$, which in turn implies that $f^\ast(s)\times_{f^\ast A} X$ is in $\overline{f^\ast S}$. Hence the claim follows.
\end{remark}

\section{Locally contractible \texorpdfstring{$\BB$}{B}-topoi}
\label{chap:locallyContractible}
An $\infty$-topos $\XX$ is said to be \emph{locally contractible} if the constant sheaf functor $\const_{\XX}\colon\SS\to\XX$ admits a left adjoint $\pi_{\XX}\colon\XX\to\SS$ which is to be thought of as the functor that carries an object $U\in\XX$ to its \emph{homotopy type} (or \emph{shape}) $\pi_{\XX}(U)$. In $1$-topos theory, the corresponding notion is that of a \emph{locally connected} $1$-topos $\EE$, in which the additinal left adjoint carries an object $U\in\EE$ to its set of connected components $\pi_{0}(E)$. In this section, we study the analogous concept for $\BB$-topoi.

We begin in \S~\ref{sec:localContractibility} by recalling the notion of a locally contractible $\BB$-topos and providing a few characterisations of this concept. In \S~\ref{sec:contractibleObjects}, we show that every locally contractible $\BB$-topos is generated by its \emph{contractible objects} in a quite strong sense. Finally, in \S~\ref{sec:smoothBC} we provide a characterisation of locally contractible $\BB$-topoi in terms of \emph{smoothness} of the associated geometric morphisms.

\subsection{Local contractibility}
\label{sec:localContractibility}
The goal of this section is to define the condition of a $\BB$-topos to be locally contractible and to derive a few explicit characterisations of this concept. 
Locally contractible geometric morphisms have been introduced and studied, mostly in the context of sheaves of topological sapces, in \cite[\S 3.2]{aizenbud2021relative} and \cite[\S 3.2]{volpe2023operations}.
We give the following equivalent definition, which is a straightforward generalisation of the notion of a locally contractible $\infty$-topos to the world of $\BB$-topoi:
\begin{definition}
	\label{def:localContractibility}
	A $\BB$-topos $\I{X}$ is \emph{locally contractible} if the unique algebraic morphism $\const_{\I{X}}\colon \Univ\to \I{X}$ admits a left adjoint $\pi_{\I{X}}\colon \I{X}\to\Univ$.
	We call a geometric morphism $f_* \colon \XX \to \BB$ \emph{locally contractible} if $f_*(\Univ[\XX])$ is a locally contractible $\BB$-topos, in which case we denote by $f_!$ the additional left adjoint of $f^\ast$ (i.e.\ the functor $\Gamma(\pi_{f_\ast(\Univ[\XX])})$).
\end{definition}
\begin{remark}
	\label{rem:BCLContr}
	As the property of a functor being a right adjoint is local in $\BB$ (see~\cite[Remark~3.3.6]{MWColimits}) and by making use of Remark~\ref{rem:BCTopoi}, we find that for any cover $\bigsqcup_i A_i\to 1$ in $\BB$, the $\BB$-topos $\I{X}$ is locally contractible if and only if for every $i$ the $\Over{\BB}{A_i}$-topos $\pi_{A_i}^\ast\I{X}$ is locally contractible.
\end{remark}

\begin{remark}
	\label{rem:locallyContractibleExplicit}
	Explicitly, a geometric morphism $f_\ast\colon\XX\to\BB$ is locally contractible precisely if $f^\ast\colon\BB \to \XX$ admits a left adjoint $f_!\colon\XX\to\BB$ such that for every map $s\colon B\to A$ in $\BB$ the induced map
	\begin{equation*}
		f_!(f^\ast B\times_{f^\ast A} -)\to B\times_A f_!(-) 
	\end{equation*}
	is an equivalence. In fact, this follows from the section-wise characterisation of internal adjunctions (Proposition~\ref{prop:existenceAdjointsBeckChevalley}) together with the observation that if $f^\ast$ admits a left adjoint $f_!$, one obtains an induced left adjoint $(f_!)_A$ of $\Over{f^\ast}{A}\colon \Over{\BB}{A}\to\Over{\XX}{f^\ast A}$ for every $A\in\BB$ which is simply given by the composition
	\begin{equation*}
		\Over{\XX}{f^\ast A}\xrightarrow{\Over{(f_!)}{A}}\Over{\BB}{f_! f^\ast A}\xrightarrow{\epsilon_!}\Over{\BB}{A}
	\end{equation*}
	(in which $\epsilon\colon f_! f^\ast\to \id_{\BB}$ is the adjunction counit).
\end{remark}

\begin{example}
	\label{etaleIsLocallyContractible}
	Every \'etale $\BB$-topos is locally contractible. More precisely, one can characterise the class of \'etale $\BB$-topoi as those locally contractible $\BB$-topoi $\I{X}$ for which the additinal left adjoint $\pi_{\I{X}}$ is a conservative functor. This is an immediate consequence of~\cite[Proposition~6.3.5.11]{lurie2009b}.
\end{example}
Recall from Theorem~\ref{thm:BTopoiRelativeTopoi} and Remark~\ref{rem:BToposFromRelativeToposExplicitly} that every $\BB$-topos $\I{X}$ corresponds uniquely to a geometric morphism $f_\ast\colon \XX\to\BB$ such that $\I{X}$ can be recovered by $f_\ast\Univ[\XX]$. The goal of this section is to characterise the property that $\I{X}$ is locally contractible in terms of the geometric morphism $f_\ast$. To that end, recall that a product-preserving functor $g\colon\CC\to\DD$ between cartesian closed $\infty$-categories is said to be cartesian closed if the natural map $g(\Hom(-,-))\to\Hom(g(-),g(-))$ (in which $\Hom(-,-)$ denotes the internal hom in $\CC$ and $\DD$, respectively) is an equivalence. If $\CC$ and $\DD$ are even \emph{locally} cartesian closed and $g$ preserves finite limits, one says that $g$ is locally cartesian closed if the induced functor $\Over{g}{c}\colon\Over{\CC}{c}\to\Over{\DD}{g(d)}$ is cartesian closed for every $c\in\CC$. We now obtain:
\begin{proposition}
	\label{prop:characterisationLocallyContractible}
	Let $\I{X}$ be a $\BB$-topos and let $f_\ast\colon \XX\to\BB$ be the associated $\infty$-topos over $\BB$. Then the following are equivalent:
	\begin{enumerate}
		\item $\I{X}$ is locally contractible;
		\item the unique algebraic morphism $\const_{\I{X}}\colon\Univ[\BB]\to\I{X}$ is $\Univ$-continuous;
		\item the functor $\const_{\I{X}}\colon f^\ast(\Univ[\BB])\to \Univ[\XX]$ (which is obtained by transposing the algebraic morphism $\const_{\I{X}}\colon\Univ[\BB]\to\I{X}$ across the adjunction $f^\ast\dashv f_\ast$) is fully faithful.
		\item the functor $f^\ast\colon \BB\to\XX$ is locally cartesian closed.
	\end{enumerate}
\end{proposition}

Before we can prove Proposition~\ref{prop:characterisationLocallyContractible}, we need the following lemma:
\begin{lemma}
	\label{lem:LCCFullyFaithful}
	Let $\I{X}$ be a $\BB$-topos and let $f_\ast\colon \XX\to \BB$ be the associated $\infty$-topos over $\BB$. Let $A\in \BB$ be an arbitrary object and let $P\to A$ and $Q\to A$ be two $\Over{\BB}{A}$-groupoids. Then the morphism of $\Over{\XX}{f^\ast A}$-groupoids $\map{f^\ast(\Univ[\BB])}(P,Q)\to \map{\Univ[\XX]}(\const_{\I{X}}(P), \const_{\I{X}}(Q))$ that is induced by the action of $\const_{\I{X}}$ recovers the map $f^\ast(\Hom_{\Over{\BB}{A}}(P,Q))\to \Hom_{\Over{\XX}{f^\ast A}}(f^\ast P,f^\ast Q)$.
\end{lemma}
\begin{proof}
	Using Remark~\ref{rem:BCTopoi}, we may assume without loss of generality that $A\simeq 1$. Furthermore, by transposing across the adjunction $f^\ast\dashv f_\ast$, it suffices to show that the map $\Hom_{\BB}(P,Q)\to f_\ast\Hom_{\XX}(f^\ast P, f^\ast Q)$ can be identified with the morphism of $\BB$-groupoids $\map{\Univ[\BB]}(\I{G},\I{H})\to \map{\I{X}}(\const_{\I{X}}(\I{G}),\const_{\I{X}}(\I{H}))$.  Now the latter can be identified with $\eta_\ast\colon \map{\Univ[\BB]}(\I{G},\I{H})\to\map{\Univ[\BB]}(\I{G},\Gamma_{\I{X}}\const_{\I{X}}\I{H})$ (where $\Gamma_{\I{X}}\colon \I{X}\to\Univ$ denotes the unique geometric morphism and where $\eta$ is the adjunction unit). In light of the equivalence $f_\ast \Hom(f^\ast P, f^\ast Q)\simeq \Hom(P, f_\ast f^\ast P)$, we can also identify the former map with $\eta_\ast\colon \Hom_{\BB}(P,Q)\to \Hom_{\XX}(P,f_\ast f^\ast Q)$. Therefore, the claim follows from Proposition~\ref{prop:mappingObjectsInternalUniverse}.
\end{proof}

\begin{proof}[{Proof of Proposition~\ref{prop:characterisationLocallyContractible}}]
	Since $\const_{\I{X}}$ is cocontinuous and preserves finite limits, one deduces from Proposition~\ref{prop:decompositionExistenceColimits} and the adjoint functor theorem (Proposition~\ref{prop:adjointFunctorTheorem}) that~(1) and~(2) are equivalent. In light of Lemma~\ref{lem:LCCFullyFaithful}, it is moreover clear that~(3) and~(4) are equivalent. To complete the proof, we will show that~(2) and~(4) are equivalent. To that end, given any map $p\colon P\to A$ in $\BB$, consider the commutative diagram
	\begin{equation*}
		\begin{tikzcd}
			\Over{\XX}{f^\ast A}\arrow[r, "p^\ast"]\arrow[from=d, "\Over{f^\ast}{A}"] & \Over{\XX}{f^\ast P}\arrow[r]\arrow[from=d, "\Over{f^\ast}{P}"]\arrow[r, "p_!"] & \Over{\XX}{f^\ast A}\arrow[from=d, "\Over{f^\ast}{A}"]\\
			\Over{\BB}{A}\arrow[r, "p^\ast"] & \Over{\BB}{P}\arrow[r, "p_!"] & \Over{\BB}{A}.
		\end{tikzcd}
	\end{equation*}
	Given $q\colon Q\to A$, the natural map $f^\ast\Hom_{\Over{\BB}{A}}(P, Q)\to \Hom_{\Over{\XX}{f^\ast A}}(f^\ast P, f^\ast Q)$ is precisely obtained by evaluating the (horizontal) mate of the composite square at $q$. Since the horizontal mate of the left square being an equivalence (for every such map $p$) precisely means that $\const_{\I{X}}$ is $\Univ$-continuous, we only need to show that the mate of the left square is an equivalence if and only if the mate of the composite square is one. One direction is trivial. As for the other direction, if we know that the map $\phi\colon\Over{f^\ast}{A}p_\ast\to p_\ast\Over{f^\ast}{P}$ is an equivalence for every object in the image of $p^\ast\colon \Over{\BB}{A}\to\Over{\BB}{P}$, then the entire map has to be an equivalence since every object in $\Over{\BB}{P}$ can be written as a pullback of such objects and since both domain and codomain of $\phi$ preserves finite limits.
\end{proof}

\subsection{Contractible objects}
\label{sec:contractibleObjects}
A topological space $X$ is by definition locally contractible if it admits a basis of contractible open subsets. A priori, the definition of a locally contractible $\BB$-topos does not appear to be related to this condition at all. In this section, we reconcile the two notions by showing that local contractibility of a $\BB$-topos can be characterised by the property of it being generated under colimits by its \emph{contractible objects}. We begin with the following definition:
\begin{definition}
	\label{def:contractibleObject}
	Let $\I{X}$ be a $\BB$-topos. An object $U\colon A\to \I{X}$ is said to be \emph{contractible} if
	the functor $\map{\I{X}}(U, \const_{\I{X}}(-))\colon \Univ[\Over{\BB}{A}]\to\Univ[\Over{\BB}{A}]$ is an equivalence.
	We define the full subcategory $\Contr(\I{X})\into\I{X}$ as the fibre of the functor
	\begin{equation*}
		\const_{\I{X}}^\ast h_{\I{X}^\op}^\op\colon\I{X}\into\IFun(\I{X},\Univ)^\op\to\IFun(\Univ,\Univ)^\op
	\end{equation*}
	over the identity $\id\colon\Univ\to\Univ$.
\end{definition}
\begin{remark}
	Note that as $\Univ$ is the initial $\BB$-topos, we find that the inclusion of the identity $\id_{\Univ}\colon 1\to \IFun(\Univ,\Univ)$ determines a fully faithful functor that identifies the domain with the full subcategory $\IFun^\alg(\Univ,\Univ)$. Therefore, the functor $\Contr(\I{X})\into\I{X}$ is indeed fully faithful. Moreover, as the universal property of $\Univ[\Over{\BB}{A}]$ implies that a functor $\Univ[\Over{\BB}{A}]\to\Univ[\Over{\BB}{A}]$ is an equivalence if and only if it is equivalent to the identity, we find that an object $U\colon A\to\I{X}$ is contained in $\Contr(\I{X})$ if and only if it is contractible.
\end{remark}

\begin{remark}
	\label{rem:BCContr}
	If $A\in\BB$ is an arbitrary object, we may combine Remark~\ref{rem:BCTopoi} with~\cite[Lemma~4.7.14]{MYoneda} and~\cite[Lemma~4.2.3]{MYoneda} to deduce that there is a canonical equivalence $\pi_A^\ast\Contr(\I{X})\simeq\Contr(\pi_A^\ast\I{X})$ of full subcategories in $\pi_A^\ast\I{X}$.
\end{remark}

\begin{remark}
	\label{rem:LContrContractibleObjects}
	In the situation of Definition~\ref{def:contractibleObject}, suppose that $\I{X}$ is locally contractible. Then we obtain an equivalence $\const_{\I{X}}^\ast h_{\I{X}^\op}^\op\simeq h_{\Univ^\op}^\op\pi_{\I{X}}$. Since $h_{\Univ^\op}^\op$ is fully faithful and since the identity on $\Univ$ is corepresented by $1_{\Univ}$ (see~\cite[Proposition~4.6.3]{MYoneda}), we thus find that $\Contr(\I{X})$ arises as the fibre of $\pi_{\I{X}}\colon\I{X}\to\Univ$ over $1_{\Univ}\colon 1\into\Univ$. In particular, this means that an object $U\colon A\to\I{X}$ is contractible precisely if $\pi_{\I{X}}(U)\colon A\to\Univ$ transposes to the final object in $\Univ[\Over{\BB}{A}]$.
\end{remark}

For the remainder of this section, let us fix a $\BB$-topos $\I{X}$. 
Recall from Lemma~\ref{lem:presentableCategoryUCompactLex} that we may always find a sound doctrine  $\I{U}$ such that $\I{X}$ is $\I{U}$-accessible and $\I{X}^{\cpt{\I{U}}}$ is closed under finite limits in $\I{X}$. We will denote by $\Contr^{\cpt{\I{U}}}(\I{X})\into\I{X}$ the intersection of $\Contr(\I{X})$ with $\I{X}^{\cpt{\I{U}}}$. The main goal of this section is to prove the following proposition:

\begin{proposition}
	\label{prop:contractibleObjectsFlat}
	Let $\I{X}$ be a $\BB$-topos and let $\I{U}$ be a sound doctrine such that $\I{X}$ is $\I{U}$-accessible and $\I{X}^{\cpt{\I{U}}}$ is closed under finite limits in $\I{X}$. Then the following are equivalent:
	\begin{enumerate}
		\item $\I{X}$ is locally contractible;
		\item the left Kan extension $h_!(j)\colon \IPSh(\Contr^{\cpt{\I{U}}}(\I{X}))\to\I{X}$ of $j\colon\Contr^{\cpt{\I{U}}}(\I{X})\into\I{X}$
		along the Yoneda embedding defines a left exact and accessible Bousfield localisation;
		\item $\I{X}$ is generated by $\Contr(\I{X})$ under colimits.
	\end{enumerate}
\end{proposition}

The proof of Proposition~\ref{prop:contractibleObjectsFlat} is based on the following two lemmas:
\begin{lemma}
	\label{lem:criterionFlatness}
	Let $j\colon\I{C}\into\I{X}$ be a (small) full subcategory such that the identity on $\I{X}$ is the left Kan extension of $j$ along itself. Then $j$ is flat.
\end{lemma}
\begin{proof}
	We need to show that $h_!(j)\colon\IPSh(\I{C})\to\I{X}$ preserves finite limits. By virtue of~\cite[Proposition~6.1.1]{MWColimits}, the final object $1_{\IPSh(\I{C})}\colon1\to\IPSh(\I{C})$ is given by the colimit of the Yoneda embedding $h\colon \I{C}\into\IPSh(\I{C})$, hence $h_!(j)(1_{\IPSh(\I{C})})$ is the colimit of $j$. But as the left Kan extension of $j$ along itself is by assumption the identity on $\I{X}$, the formula from~\cite[Remark~6.3.6]{MWColimits} implies that every object $U\colon 1\to\I{X}$ is the colimit of the composition $\Over{\I{C}}{U}\into\Over{\I{X}}{U}\to \I{X}$. In particular, the final object in $\I{X}$ must be the colimit of $j$ itself. Hence $h_!(j)$ preserves the final object. To complete the proof, it therefore suffices to show that $h_!(j)$ also preserves pullbacks. By Lemma~\ref{lem:YonedaExtensionPullbackPreservation} and in light of~\cite[Remark~6.3.2]{MWColimits} and~\cite[Lemma~4.7.14]{MYoneda}, it will be enough to show that if $\sigma$ is an arbitrary cospan in $\I{C}$ in context $1\in\BB$, the functor $h_!(j)$ preserves the pullback $P$ of $h(\sigma)$. In other words, we need to prove that the induced functor $h_!(j)_\ast\colon \Over{\IPSh(\I{C})}{h(\sigma)}\to\Over{\I{X}}{j(\sigma)}$ preserves the final object. Let $Q\colon 1\to\I{X}$ be the pullback of $ j(\sigma)$. We then have a commutative diagram
	\begin{equation*}
		\begin{tikzcd}
			\Over{\I{C}}{P}\arrow[d, "\simeq"]\arrow[r, hookrightarrow] & \Over{\IPSh(\I{C})}{P}\arrow[d, "\simeq"]\arrow[r] & \Over{\I{X}}{Q}\arrow[d, "\simeq"]\\
			\Over{\I{C}}{\sigma}\arrow[r, hookrightarrow] & \Over{\IPSh(\I{C})}{h(\sigma)}\arrow[r, "h_!(j)_\ast"] & \Over{\I{X}}{j(\sigma)}
		\end{tikzcd}
	\end{equation*}
	in which the upper right horizontal functor can be identified with the composition of $\Over{h_!(j)}{P}\colon\Over{\IPSh(\I{C})}{P}\to\Over{\I{X}}{h_!(j)(P)}$ with the forgetful functor $\Over{\I{X}}{h_!(j)(P)}\to\Over{\I{X}}{Q}$ along the natural map $h_!(j)(P)\to Q$. As both of these functors are cocontinuous (see Corollary~\ref{cor:adjunctionsSliceCategoryPullback} and Proposition~\ref{prop:characterisationCategoriesWithProducts}), the upper right horizontal functor must be cocontinuous as well. By combining this observation with the identification $\Over{\IPSh(\I{C})}{P}\simeq\IPSh(\Over{\I{C}}{P})$ from~\cite[Lemma~6.1.5]{MWColimits} and the universal property of presheaf $\BB$-categories, we conclude that this functor is equivalent to the Yoneda extension of the inclusion $\Over{\I{C}}{P}\into\Over{\I{X}}{Q}$. Therefore, by using the first part of the proof, it will be enough to show that the identity on $\Over{\I{X}}{j(\sigma)}$ is the left Kan extension of the inclusion $\Over{\I{C}}{\sigma}\into\Over{\I{X}}{j(\sigma)}$ along itself. By using the criterion from~\cite[Remark~6.3.6]{MWColimits} together with the fact that any slice $\BB$-category over $\Over{\I{X}}{j(\sigma)}$ can be identified with a slice $\BB$-category over $\I{X}$, this in turn follows from the assumption that the identity on $\I{X}$ is the left Kan extension of $j$ along itself.
\end{proof}

\begin{lemma}
	\label{lem:contractibleObjectsDense}
	If $\I{X}$ is locally contractible, the identity on $\I{X}$ is the left Kan extension of the inclusion $\Contr^{\cpt{\I{U}}}(\I{X})\into \I{X}$ along itself.
\end{lemma}
\begin{proof}
	Note that we have inclusions $\Contr^{\cpt{\I{U}}}(\I{X})\into\I{X}^{\cpt{\I{U}}}\into\I{X}$ in which the left Kan extension of the second inclusion along itself is the identity on $\I{X}$. Consequently, it will be enough to show that the left Kan extension of the inclusion $\Contr^{\cpt{\I{U}}}(\I{X})\into \I{X}$ along the inclusion $\Contr^{\cpt{\I{U}}}(\I{X})\into \I{X}^{\cpt{\I{U}}}$ recovers the inclusion $\I{X}^{\cpt{\I{U}}}\into\I{X}$.
	By using~\cite[Remark~6.3.6]{MWColimits} together with Remark~\ref{rem:BCContr}, this follows once we verify that for any $\I{U}$-compact object $U\colon 1\to\I{X}$, the colimit of the induced inclusion $\Over{\Contr^{\cpt{\I{U}}}(\I{X})}{U}\into\Over{\I{X}}{U}$ is the final object. Now observe that the functor $\Over{(\pi_{\I{X}})}{U}\colon \Over{\I{X}}{U}\to\Over{\Univ}{\pi_{\I{X}}(U)}$ restricts to a functor $\Over{(\pi_{\I{X}})}{U}\colon \Over{\Contr^{\cpt{\I{U}}}(\I{X})}{U}\to \pi_{\I{X}}(U)$. We claim that the right adjoint $ (\const_{\I{X}})_{\pi_{\I{X}}(U)}\colon\Over{\Univ}{\pi_{\I{X}}(U)}\to\Over{\I{X}}{U}$ (which is constructed by composing the functor $\Over{(\const_{\I{X}})}{\pi_{\I{X}}(U)}\colon\Over{\Univ}{\pi_{\I{X}}(U)}\to\Over{\I{X}}{\const_{\I{X}}\pi_{\I{X}}(U)}$ with the pullback morphism $\eta^\ast\colon\Over{\I{X}}{\const_{\I{X}}\pi_{\I{X}}(U)}\to\Over{\I{X}}{U}$ along the adjunction unit $\eta\colon U\to\const_{\I{X}}\pi_{\I{X}}(U)$, see Appendix~\ref{sec:sliceAdjunctions}) restricts to a map $\pi_{\I{X}}(U)\to \Over{\Contr^{\cpt{\I{U}}}(\I{X})}{U}$. In fact, by making use of Remark~\ref{rem:BCContr}, it will be enough to verify that if $x\colon 1\to\pi_{\I{X}}(U)$ is an arbitrary object in context $1\in\BB$, its image along $(\const_{\I{X}})_{\pi_{\I{X}}(U)}$ is $\I{U}$-compact and contractible. By construction, this object fits into a pullback square
	\begin{equation*}
		\begin{tikzcd}
			(\const_{\I{X}})_{\pi_{\I{X}}(U)}(x)\arrow[r]\arrow[d] & 1_{\I{X}}\arrow[d, "\const_{\I{X}}(x)"]\\
			U\arrow[r, "\eta"] & \const_{\I{X}}\pi_{\I{X}}(U).
		\end{tikzcd}
	\end{equation*}
	Note that both $\pi_{\I{X}}$ and $\const_{\I{X}}$ are left adjoint to $\IFilt_{\I{U}}$-cocontinuous functors and therefore preserve $\I{U}$-compact objects. In combination with our assumption that the full subcategory of $\I{U}$-compact objects in $\I{X}$ is closed under finite limits, we thus find that $(\const_{\I{X}})_{\pi_{\I{X}}(U)}(x)$ is $\I{U}$-compact too. Furthermore, note that we may regard $\eta$ as an object in $\Over{\XX}{f^\ast (\pi_{\I{X}}(U))}=\I{X}(\pi_{\I{X}}(U))$, i.e.\ as an \emph{object} in $\I{X}$ in context $\pi_{\I{X}}(U)$. As such, $\eta$ is contractible: in fact, by Remark~\ref{rem:locallyContractibleExplicit} the object $\pi_{\I{X}}(\eta)\in \Univ[\BB](\pi_{\I{X}}(U))$ is explicitly computed as the composition
	\begin{equation*}
		\pi_{\I{X}}(U)\xrightarrow{\pi_{\I{X}}(\eta)} \pi_{\I{X}}\const_{\I{X}}\pi_{\I{X}}(U)\xrightarrow{\epsilon}\pi_{\I{X}}(U)
	\end{equation*} 
	(where in the first map $\eta$ is regarded as a \emph{morphism} in $\I{X}$ in context $1\in\BB$ and where $\epsilon$ is the counit of the adjunction $\pi_{\I{X}}\dashv \const_{\I{X}}$), hence the claim follows from the triangle identities. Now viewing $\eta$ as an object in $\I{X}$ in context $\pi_{\I{X}}(U)$, the above pullback square exhibits the global object $(\const_{\I{X}})_{\pi_{\I{X}}(U)}(x)\in \I{X}(1)=\XX$ as the image of $\eta\in\I{X}(\pi_{\I{X}}(U))$ along the transition map $x^\ast\colon \I{X}(\pi_{\I{X}}(U))\to\I{X}(1)$. Therefore, $\eta$ being a contractible object implies that $(\const_{\I{X}})_{\pi_{\I{X}}(U)}(x)$ must be contractible as well. Thus we conclude that $(\const_{\I{X}})_{\pi_{\I{X}}(U)}(x)$ is contained in $\Over{\Contr^{\cpt{\I{U}}}(\I{X})}{U}$, as claimed. 
	
	So far, our arguments have shown that we have a commutative square
	\begin{equation*}
		\begin{tikzcd}
			\Over{\Contr^{\cpt{\I{U}}}(\I{X})}{U}\arrow[r, hookrightarrow] & \Over{\I{X}}{U}\\
			\pi_{\I{X}}(U)\arrow[u]\arrow[r, hookrightarrow] & \Over{\Univ}{\pi_{\I{X}}(U)}.\arrow[u, "(\const_{\I{X}})_{\pi_{\I{X}}(U)}"]
		\end{tikzcd}
	\end{equation*}
	Since the vertical maps in this diagram are right adjoints, they are in particular final. Since furthermore $(\const_{\I{X}})_{\pi_{\I{X}}(U)}$ is cocontinuous, the colimit of the upper horizontal map is the image of the colimit of the lower horizontal map along $(\const_{\I{X}})_{\pi_{\I{X}}(U)}$. To complete the proof, it is therefore enough to prove that the colimit of the lower horizontal map is the final object in $\Over{\Univ}{\pi_{\I{X}}(U)}$. But this is simply the statement that $\pi_{\I{X}}(U)$ is the colimit of the constant diagram $\pi_{\I{X}}(U)\to 1\into\Univ$ with value $1_{\Univ}$, which is clear.
\end{proof}

\begin{proof}[{Proof of Proposition~\ref{prop:contractibleObjectsFlat}}]
	Suppose first that $\I{X}$ is locally contractible. By combining Lemmas~\ref{lem:contractibleObjectsDense} and~\ref{lem:criterionFlatness}, the map $h_!(j)\colon\IPSh(\Contr^{\cpt{\I{U}}}(\I{X}))\to\I{X}$ is left exact, so it suffices to show that this functor is a Bousfield localisation. Since it is cocontinuous, it has a right adjoint $r$ (which is automatically accessible). The counit of this adjunction carries an object $U\colon A\to\I{X}$ to the canonical map from the colimit of $\Over{\pi_A^\ast\Contr^{\cpt{\I{U}}}(\I{X})}{U}\to \pi_A^\ast\I{X}$ to $U$. By again using Lemma~\ref{lem:contractibleObjectsDense}, this map is an equivalence, hence the claim follows. If~(2) is satisfied, the map $h_!(j)\colon\IPSh(\Contr^{\cpt{\I{U}}}(\I{X}))\to\I{X}$ being a (left exact and accessible) Bousfield localisation in particular implies that $\I{X}$ is generated under colimits by $\Contr^{\cpt(\I{U})}(\I{X})$, in the sense that the smallest full subcategory of $\I{X}$ that contains $\Contr^{\cpt(\I{U})}(\I{X})$ and that is closed under $\ICat_{\BB}$-colimits in $\I{X}$ must already be $\I{X}$ itself. In particular,~(3) follows. Lastly, if~(3) is satisfied, consider the commutative diagram
	\begin{equation*}
		\begin{tikzcd}
			\Contr(\I{X})\arrow[r, hookrightarrow]\arrow[d] & \I{P}\arrow[r, hookrightarrow] \arrow[d]& \I{X}\arrow[d, "\const_{\I{X}}^\ast h_{\I{X}^\op}^\op"]\\
			1\arrow[r, hookrightarrow, "1_{\Univ}"] & \Univ\arrow[r, hookrightarrow, "h_{\Univ^\op}^\op"] & \IFun(\Univ,\Univ)^\op
		\end{tikzcd}
	\end{equation*}
	in which both squares are pullbacks. Since $h_{\Univ^\op}^\op$ and $h_{\I{X}^\op}^\op$ are cocontinuous functors~\cite[Proposition~5.2.9]{MWColimits}, the inclusion $\I{P}\into\I{X}$ is closed under $\Univ$-colimits (see Lemma~\ref{lem:pullbackCocompleteRightFibration}) and must therefore be an equivalence. Hence $\const_{\I{X}}^\ast h_{\I{X}^\op}^\op$ takes values in $\Univ\into\IFun(\Univ,\Univ)^\op$, which precisely means that $\const_{\I{X}}$ has a left adjoint $\pi_{\I{X}}$. Hence $\I{X}$ is locally contractible.
\end{proof}

\subsection{Classification of smooth geometric morphisms}
\label{sec:smoothBC}
In $1$-topos theory, it is a well-known result (see~\cite[Corollary~C.3.3.16]{johnstone2002}) that locally connected geometric morphisms are precisely the \emph{smooth} maps, i.e.\ those that satisfy \emph{smooth base change}. Our main goal in this section is to prove the $\infty$-toposic analogue of this result. We begin with the following definition:
\begin{definition}
	\label{def:smoothness}
	Let $f_\ast\colon \XX\to \BB$ be a geometric morphism of $\infty$-topoi. We say that $f_\ast$ is \emph{smooth} if for every diagram
	\begin{equation*}
		\begin{tikzcd}
			\YY^\prime\arrow[r, "k^\prime_\ast"] \arrow[d, "g_\ast^\prime"] & \YY\arrow[d, "g_\ast"]\arrow[r, "k_\ast"] & \XX\arrow[d, "f_\ast"]\\
			\AA^\prime\arrow[r, "h^\prime_\ast"] & \AA\arrow[r, "h_\ast"] & \BB
		\end{tikzcd}
	\end{equation*}
	in $\RTopS$ in which both squares are pullbacks, the mate transformation $g^\ast h^\prime_\ast\to k^\prime_\ast(g^\prime)^\ast$ is an equivalence.
\end{definition}

We may now formulate the main result of this section as follows:
\begin{theorem}
	\label{thm:locallyContractibleSmoothBaseChange}
	Let $\I{X}$ be a $\BB$-topos and let $f_\ast\colon \XX\to\BB$ be the associated geometric morphism. Then $\I{X}$ is locally contractible if and only if $f_\ast$ is smooth.
\end{theorem}

The proof of Theorem~\ref{thm:locallyContractibleSmoothBaseChange} relies on a few reduction steps. Our first goal is to establish that the property of a geometric morphism to be locally contractible is stable under taking \emph{powers} by $\BB$-categories: Recall from Proposition~\ref{prop:cotensoringLTop} that the large $\BB$-category $\ILTop_{\BB}$ admits a powering bifunctor $(-)^{(-)}\colon \ICat_{\BB}^\op\times\ILTop_{\BB}\to\ILTop_{\BB}$. We now find:
\begin{lemma}
	\label{lem:TensoringPreservesLocallyContractible}
	Let $\I{C}$ be a $\BB$-category and let $\I{X}$ be a locally contractible $\BB$-topos. Then the geometric morphism $(\Gamma_{\I{X}})_\ast\colon\I{X}^{\I{C}}\to\Univ^{\I{C}}$ exhibits $\Fun_{\BB}(\I{C},\I{X})$ as a locally contractible $\Fun_{\BB}(\I{C},\Univ)$-topos.
\end{lemma}
\begin{proof}
	Since the algebraic morphism associated with $(\Gamma_{\I{X}})_\ast$ is given by $(\const_{\I{X}})_\ast$, the functor $(\pi_{\I{X}})_\ast$ defines a further left adjoint of $(\const_{\I{X}})_\ast$. Therefore, Remark~\ref{rem:locallyContractibleExplicit} implies that we only need to show that for every map $F\to G$ in $\Fun_{\BB}(\I{C},\Univ)$ and every map $H\to (\const_{\I{X}})_\ast(G)$ in $\Fun_{\BB}(\I{C},\I{X})$, the canonical morphism
	\begin{equation*}
		(\pi_{\I{X}})_\ast((\const_{\I{X}})_\ast F\times_{(\const_{\I{X}})_\ast G}H)\to F\times_{H}(\pi_{\I{X}})_\ast H
	\end{equation*}
	is an equivalence. It will be enough to show that this map becomes an equivalence after being evaluated at an arbitrary object $c\colon A\to \I{C}$ in context $A\in\BB$. In light of Remark~\ref{rem:BCLContr} and~\cite[Lemma~4.2.3]{MYoneda}, we can replace $\BB$ with $\Over{\BB}{A}$, so that we can reduce to the case $A\simeq 1$. But as pullbacks in functor $\BB$-categories are computed object-wise~\cite[Proposition~4.3.2]{MWColimits} and as evaluating the unit and counit of the adjunction $(\pi_{\I{X}})_\ast\dashv (\const_{\I{X}})_\ast$ at $c$ recovers the unit and counit of the adjunction $\pi_{\I{X}}\dashv\const_{\I{X}}$, the claim follows from the assumption that $\I{X}$ is locally contractible and Remark~\ref{rem:locallyContractibleExplicit}.
\end{proof}

Before we can prove Theorem~\ref{thm:locallyContractibleSmoothBaseChange}, we also need the following result:
\begin{lemma}
	\label{lem:constantPresheafIsSheafLocallyContractible}
	Let $\I{X}$ be a locally contractible $\BB$-topos and let $\I{U}$ be a sound doctrine such that $\I{X}$ is $\I{U}$-accessible and $\I{X}^{\cpt{\I{U}}}$ is closed under finite limits in $\I{X}$. Then the diagonal map $\diag\colon\Univ\to\IPSh(\Contr^{\cpt{\I{U}}}(\I{X}))$ takes values in $\I{X}\into\IPSh(\Contr^{\cpt{\I{U}}}(\I{X}))$.
\end{lemma}
\begin{proof}
	We need to show that if $\eta\colon \id\to iL$ is the unit of the adjunction $L\dashv i\colon \I{X}\leftrightarrows \IPSh(\Contr^{\cpt{\I{U}}}(\I{X}))$, then the induced morphism $\eta \diag\colon \diag\to iL\diag$ is an equivalence. As we may check this object-wise and by using Remarks~\ref{rem:BCLContr},~\ref{rem:BCUComp},~\ref{rem:BCInd} and~\ref{rem:BCContr} together with~\cite[Lemma~4.2.3]{MYoneda}, we only need to show that for any object $U\colon 1\to\Contr^{\cpt{\I{U}}}(\I{X})$ the map $U^\ast\eta\diag\colon U^\ast\diag\to U^\ast iL\diag$
	is an equivalence in $\Univ$. Note that we have a chain of equivalences
	\begin{equation*}
		U^\ast iL\diag\simeq \map{\I{X}}(U,\const_{\I{X}})\simeq\map{\Univ}(\pi_{\I{X}}(U), -).
	\end{equation*}
	As $\pi_{\I{X}}(U)\simeq 1_{\Univ}$, we thus find that $U^\ast iL\diag\simeq \id$. Since also $U^\ast\diag$ is equivalent to the identity and since the universal property of $\Univ$ implies that $\map{\IFun(\Univ,\Univ)}(\id,\id)\simeq 1_{\Univ}$, the claim follows.
\end{proof}

\begin{proof}[{Proof of Theorem~\ref{thm:locallyContractibleSmoothBaseChange}}]
	Suppose first that $f_\ast$ is smooth. Then $f_\ast$ in particular satisfies condition~(2) of Proposition~\ref{prop:characterisationLocallyContractible} and is therefore locally contractible. To prove the converse direction, suppose that we have two pullback squares
	\begin{equation*}
		\begin{tikzcd}
			\YY^\prime\arrow[r, "k^\prime_\ast"]\arrow[d, "g^\prime_\ast"]&\YY\arrow[r, "k_\ast"]\arrow[d, "g_\ast"] & \XX\arrow[d, "f_\ast"]\\
			\AA^\prime\arrow[r, "h^\prime_\ast"] &\AA\arrow[r, "h_\ast"] & \BB
		\end{tikzcd}
	\end{equation*}
	of $\infty$-topoi in which $f_\ast$ is locally contractible.
	By viewing $\AA$ as a $\BB$-topos and using Theorem~\ref{thm:presentationTopoi}, we may factor $h_\ast$ into a composition $\AA\into \Fun_{\BB}(\I{C}^\op,\Univ[\BB])\to \BB$. Since the commutative square
	\begin{equation*}
		\begin{tikzcd}
			\IFun(\I{C}^\op, \I{X})\arrow[r, "\lim"]\arrow[d, "(\Gamma_{\I{X}})_\ast"] & \I{X}\arrow[d, "\Gamma_{\I{X}}"]\\
			\IPSh(\I{C})\arrow[r, "\lim"] & \Univ
		\end{tikzcd}
	\end{equation*}
	is a pullback in $\IRTop_{\BB}$ (see Example~\ref{ex:PushoutPresheafCategory}) and on account of Lemma~\ref{lem:TensoringPreservesLocallyContractible}, this allows us to reduce to the case where $h_\ast$ is already an embedding. But then $k_\ast$ must be an embedding as well, so that the mate of the left square is an equivalence if and only if the mates of the right one and the composite one are equivalences. Hence, to complete the proof, it will be enough to show that if we are given any pullback square
	\begin{equation*}
		\begin{tikzcd}
			\YY\arrow[r, "k_\ast"]\arrow[d, "g_\ast"] & \XX\arrow[d, "f_\ast"]\\
			\AA\arrow[r, "h_\ast"] & \BB
		\end{tikzcd}
	\end{equation*}
	in which $f_\ast$ is locally contractible, the mate transformation $f^\ast h_\ast\to k_\ast g^\ast$ is an equivalence. By the same argument as above (and the fact that $\const_{\I{X}}$ is continuous), we can moreover still assume that $h_\ast$ and $k_\ast$ are embeddings. To proceed, we make use of Proposition~\ref{prop:contractibleObjectsFlat} to obtain a commutative diagram
	\begin{equation*}
		\begin{tikzcd}
			\YY\arrow[r, hookrightarrow, "k_\ast"]\arrow[d, hookrightarrow] & \XX\arrow[d, hookrightarrow]\\
			\Fun_{\BB}(\Contr^{\cpt{\I{U}}}(\I{X})^\op, h_\ast(\Univ[\AA]))\arrow[r, hookrightarrow] \arrow[d, "\lim"]& \Fun_{\BB}(\Contr^{\cpt{\I{U}}}(\I{X})^\op,\Univ[\BB])\arrow[d, "\lim"]\\
			\AA\arrow[r, "h_\ast", hookrightarrow] & \BB
		\end{tikzcd}
	\end{equation*}
	in which both squares are pullbacks. Since the mate of the lower square is evidently an equivalence and since Lemma~\ref{lem:constantPresheafIsSheafLocallyContractible} implies that the diagonal map $\diag\colon\BB\to \Fun_{\BB}(\Contr^{\cpt{\I{U}}}(\I{X})^\op,\Univ[\BB])$ factors through $\XX$, it will be enough to show that the diagonal map $\diag\colon\AA\to \Fun_{\BB}(\Contr^{\cpt{\I{U}}}(\I{X})^\op, h_\ast\Univ[\AA])$ factors through $\YY$ as well. Let us therefore pick an arbitrary object $A\in\AA$. By~\cite[Lemmas~6.3.3.4]{htt}, the upper square in the above diagram is a pullback square of $\infty$-categories, hence it suffices to show that the image of $\diag(A)$ in $\Fun_{\BB}(\Contr^{\cpt{\I{U}}}(\I{X})^\op,\Univ[\BB])$ is contained in $\XX$. But as the mate of the lower square is an equivalence, this latter object is equivalent to $\diag h_\ast(A)$, hence another application of Lemma~\ref{lem:constantPresheafIsSheafLocallyContractible} yields the claim.
\end{proof}

\section{Localic $\BB$-topoi}
\label{sec:localicBTopoi}
In higher topos theory, the $1$-category of \emph{locales} (with left exact left adjoints as maps) arises as a coreflective subcategory of the $\infty$-category $\LTopS$ of $\infty$-topoi. The inclusion is given by sending a locale $L$ to the $\infty$-topos $\Shv(L)$ of sheaves on $L$, and the coreflection carries an $\infty$-topos $\XX$ to the locale $\Sub(\XX)$ of subterminal (i.e.\  $(-1)$-truncated) objects in $\XX$. An $\infty$-topos $\XX$ is said to be localic if it is equivalent to $\Shv(\Sub(\XX))$.

In this section, we give a brief exposition of the analogous story in the world of $\BB$-topoi. We do not aim to provide a comprehensive study of localic $\BB$-topoi, but rather restrict our attention to those aspects of the theory that allow us to \emph{define} the notion of a localic $\BB$-topos and to provide an external characterisation of this concept in the case where $\BB$ is itself localic.

We begin in \S~\ref{sec:BPosets} and~\ref{sec:PresentableBPosets} by providing the necessary background material on $\BB$-posets. In \S~\ref{sec:BLocales} we define and characterise $\BB$-locales, and in \S~\ref{sec:SheavesOnBLocales} we construct the $\BB$-topos of sheaves on a $\BB$-locale, which we use in~\S~\ref{sec:LocalicReflection} to show that $\BB$-locales are a coreflective localisation of $\BB$-topoi. In \S~\ref{sec:BLocalesVsRelativeLocales} we discuss how localic $\BB$-topoi correspond to localic $\infty$-topoi over $\BB$ in the case where $\BB$ is itself localic. Lastly, \S~\ref{sec:compactnessLocalicTopoi} discusses some compactness conditions of $\BB$-locales and how they are inherited by the associated $\BB$-topoi of sheaves.
Finally in \S~\ref{sec:locallyProper}, we study $ \BB $-locales arising from maps of topological spaces.

\subsection{$\BB$-posets}
\label{sec:BPosets}
Recall that an $\infty$-category $\CC$ is (equivalent to) a poset precisely if if for all objects $c$ and $d$ in $\CC$ the mapping $\infty$-groupoid $\map{\CC}(c,d)$ is $(-1)$-truncated. In this section we discuss a generalisation of this concept to $\BB$-categories.

Recall that the class of $(-1)$-truncated maps in $\BB$ is precisely the collection of morphisms that are internally right orthogonal to the map $1\sqcup 1\to 1$ (in the sense of~\cite[\S~2.5]{MYoneda}). In particular, this class is local, so that we may define:
\begin{definition}
	\label{def:subterminalBGrpd}
	The subuniverse $\ISub_{\BB}\into\Univ[\BB]$ of \emph{subterminal $\BB$-groupoids} is the full subcategory of $\Univ[\BB]$ that is determined by the local class of $(-1)$-truncated morphisms in $\BB$. A $\Over{\BB}{A}$-groupoid $\I{G}$ is said to be a subterminal $\Over{\BB}{A}$-groupoid if it defines an object of $\ISub_{\BB}$ (in context $A$).
\end{definition}

\begin{remark}
	As the functor $(\pi_A)_!\colon \Over{\BB}{A}\to\BB$ creates pullbacks, a map $P\to B$ in $\Over{\BB}{A}$ defines an object $B\to \ISub_{\Over{\BB}{A}}$ if and only if the underlying map in $\BB$ defines an object $(\pi_A)_!(B)\to \ISub_{\BB}$. Thus, the equivalence $\pi_A^\ast\Univ\simeq\Univ[\Over{\BB}{A}]$ restricts to an equivalence $\pi_A^\ast\ISub_{\BB}\simeq\ISub_{\Over{\BB}{A}}$ for every $A\in\BB$.
\end{remark}

By~\cite[Example~3.4.2]{anel2020} the internal saturation of the map $1\sqcup 1\to 1$ (i.e. the class of covers) in $\BB$ is closed under base change, so that the factorisation system between these and $(-1)$-truncated maps even forms a \emph{modality}. As a consequence, one finds:
\begin{proposition}
	\label{prop:subterminalBGrpdsPresentablyEmbedded}
	The $\BB$-category $\ISub_{\BB}$ is an accessible Bousfield localisation of $\Univ[\BB]$ (in the sense of Definition~\ref{def:accessibleLocalisation}) and therefore in particular presentable.
\end{proposition}
\begin{proof}
	By Example~\ref{ex:modalityDescent} the subcategory $\ISub_{\BB}\into\Univ[\BB]$ is reflective. Since the inclusion $\ISub_{\BB}\into\Univ[\BB]$ is section-wise accessible, Corollary~\ref{cor:AccIsSectWiseAccWhenCocomp} implies that the localisation must be accessible. 
\end{proof}

\begin{definition}
	\label{def:BPoset}
	A $\BB$-category $\I{C}$ is said to be a \emph{$\BB$-poset} if the mapping bifunctor $\map{\I{C}}$ takes values in $\ISub_{\BB}$. The full subcategory of $\ICat_{\BB}$ that is spanned by the $\Over{\BB}{A}$-posets for each $A\in\BB$ is denoted by $\IPos_{\BB}$ and its underlying $\infty$-category of global sections by $\Pos(\BB)$.
\end{definition}

\begin{remark}
	\label{rem:PosetsUniversalMappingGroupoid}
	A $\BB$-category $\I{C}$ is a $\BB$-poset precisely if the map $\I{C}_1\to\I{C}_0\times\I{C}_0$ is $(-1)$-truncated. In fact, since this map exhibits $\I{C}_1$ as the mapping $\Over{\BB}{\I{C}_0\times\I{C}_0}$-groupoid between the two objects $\pr_0,\pr_1\colon\I{C}_0\times\I{C}_0\rightrightarrows\I{C}$, this is clearly a necessary condition. The fact that it is also sufficient follows from the observation that every mapping $\Over{\BB}{A}$-groupoid of $\I{C}$ is a pullback of this map.
\end{remark}
\begin{remark}
	\label{rem:localityBPosets}
	Since the class of $(-1)$-truncated maps in $\BB$ is local, we deduce from Remark~\ref{rem:PosetsUniversalMappingGroupoid} that if $(s_i)\colon\bigsqcup_i A_i\onto A$ is a cover in $\BB$, a $\Over{\BB}{A}$-category $\I{C}\colon A\to\ICat_{\BB}$ defines a $\Over{\BB}{A}$-poset if and only if the $\Over{\BB}{A_i}$-category $s_i^\ast\I{C}$ defines a $\Over{\BB}{A_i}$-poset for every $i$. In particular, \emph{every} object of $\IPos_{\BB}$ in context $A\in\BB$ encodes a $\Over{\BB}{A}$-poset, and one has a canonical equivalence $\pi_A^\ast\IPos_{\BB}\simeq\IPos_{\Over{\BB}{A}}$.
\end{remark}

\begin{remark}
	\label{rem:BPosetsSectionwise}
	Remark~\ref{rem:PosetsUniversalMappingGroupoid} and the fact that a map is $(-1)$-truncated precisely if it is so section-wise imply that a  $\BB$-category $\I{C}$ is a $\BB$-poset if and only if $\I{C}(A)$ is an poset for every $A\in\BB$. Together with Remark~\ref{rem:localityBPosets}, this implies that we obtain an equivalence $\IPos_{\BB}\simeq \Pos\otimes \Univ$ (where $\Pos$ is the $1$-category of posets).
\end{remark}

\begin{remark}
	\label{rem:BPosetsClassically}
	Recall that if $\I{C}$ is a $\BB$-category, then the map $s_0\colon \I{C}_0\to\I{C}_1$ is a monomorphism in $\BB$. Together with Remark~\ref{rem:PosetsUniversalMappingGroupoid}, this implies that if $\I{P}$ is a $\BB$-poset, then each $\I{P}_0$ is contained in the underlying $1$-topos $\Disc(\BB)\into\BB$ of $0$-truncated objects. Consequently, we may identify $\Pos(\BB)$ with the full subcategory of $\Simp{\Disc(\BB)}$ that is spanned by the complete Segal objects $\I{P}$ for which the map $\I{P}_1\to\I{P}_0\times\I{P}_0$ is a monomorphism. Hence $\Pos(\BB)$ is equivalent to the $1$-category of internal posets in the $1$-topos $\Disc(\BB)$ in the sense of~\cite[\S~B2.3]{johnstone2002}.
\end{remark}

\subsection{Presentable $\BB$-posets}
\label{sec:PresentableBPosets}
In this section we study \emph{presentable} $\BB$-posets. We begin with the following definition:

\begin{definition}
	\label{def:subterminalObjects}
	If $\I{C}$ is a $\BB$-category, we define the full subcategory $\ISub(\I{C})\into\I{C}$ of \emph{subterminal objects} as the pullback $\I{C}\times_{\IPSh(\I{C})}\iFun(\I{C}^{\op},\ISub_{\BB})$.
\end{definition}

\begin{remark}
	\label{rem:BCSubterminal}
	If $\I{C}$ is a $\BB$-category and $A\in\BB$ is an arbitrary object, then there is a canonical equivalence $\pi_A^\ast\ISub(\I{C})\simeq \ISub(\pi_A^\ast\I{C})$ of full subcategories in $\pi_A^\ast\I{C}$.
\end{remark}

\begin{proposition}
	\label{prop:subterminalObjectsPresentable}
	For every presentable $\BB$-category $\I{D}$ there is a canonical equivalence $\ISub(\I{D})\simeq\I{D}\otimes\ISub_{\BB}$ of full subcategories in $\I{D}$. In particular, $\ISub(\I{D})$ is an accessible Bousfield localisation of $\I{D}$ and therefore presentable as well.
\end{proposition}
\begin{proof}
	Recall from Proposition~\ref{prop:UnivSheavesRepresentables} that the Yoneda embedding $\I{D}\into\IPSh(\I{D})$ identifies $\I{D}$ with the full subcategory $\IShv_{\BB}(\I{D})\into\IPSh(\I{D})$ that is spanned by the continuous functors (in arbitrary context). Therefore, it will be enough to show that the square
	\begin{equation*}
		\begin{tikzcd}
			\I{D}\otimes\ISub_{\BB}\arrow[r, hookrightarrow]\arrow[d, hookrightarrow] & \iFun(\I{D}^{\op},\ISub_{\BB})\arrow[d, hookrightarrow]\\
			\IShv_{\BB}(\I{D})\arrow[r, hookrightarrow] & \IPSh(\I{D})
		\end{tikzcd}
	\end{equation*}
	is a pullback. Together with the usual base change arguments, this means that we only need to check that a functor $F\colon \I{D}^\op\to\ISub_{\BB}$ is continuous if and only if its postcomposition with the inclusion $\ISub_{\BB}\into\Univ$ is. This follows immediately from the fact that the inclusion $\ISub_{\BB}\into\Univ$ has a left adjoint and is therefore continuous and conservative.
\end{proof}
For every presentable $\BB$-category $\I{D}$, we will denote the left adjoint of the inclusion $\ISub(\I{D})\into\I{D}$ by $(-)^{\Sub}$ and refer to it as the \emph{subterminal truncation functor}.
\begin{example}
	\label{ex:subterminalPresheaves}
	If $\I{C}$ is an arbitrary $\BB$-category, then Proposition~\ref{prop:subterminalObjectsPresentable} provides us with an equivalence of $\BB$-categories
	$\ISub(\IPSh(\I{C}))\simeq\iFun(\I{C}^{\op},\ISub_{\BB})$. Moreover, in light of the equivalence $\IPSh(\I{C})\simeq\IRFib_{\I{C}}$, we may identify $\iFun(\I{C}^{\op},\ISub_{\BB})$ with the full subcategory of $\IRFib_{\I{C}}$ that is spanned by the right fibrations $p\colon\I{P}\to\pi_A^\ast\I{C}$ (in arbitrary context $A\in\BB$) which are fully faithful, i.e.\ which are \emph{sieves} in the $\Over{\BB}{A}$-poset $\pi_A^\ast\I{C}$. To see the latter claim, first note that by Remark~\ref{rem:BCSubterminal}, we may replace $\BB$ with $\Over{\BB}{A}$ so that we can assume that $A\simeq 1$. In this case, since $p_0\colon \I{P}_0\to\I{C}_0$ can be identified with the image of the tautological object $\I{C}_0\to \I{C}$ along the functor $F\colon\I{C}^\op \to \Univ$ that classifies $p$ and since every object in $\I{C}$ (in arbitrary context) arises as a pullback of the tautological object, $F$ takes values in $\ISub_{\BB}$ if and only if $p_0$ is a monomorphism. Therefore it suffices to show that $p_0$ is monic if and only if $p$ is fully faithful. This follows from considering the commutative diagram
	\begin{equation*}
		\begin{tikzcd}
			\I{P}_1\arrow[d, "{(d_1,d_0)}"]\arrow[r, "\id"] & \I{P}_1\arrow[d]\arrow[r] & \I{C}_1\arrow[d, "{(d_1,d_0)}"]\\
			\I{P}_0\times\I{P}_0\arrow[r, "p_0\times\id"] & \I{C}_0\times\I{P}_0\arrow[r, "\id\times p_0"] & \I{C}_0\times\I{C}_0
		\end{tikzcd}
	\end{equation*}
	in which the right square is a pullback as $p$ is a right fibration and where the left square is a pullback since $p_0$ is a monomorphism.
	
	Furthermore, the above observation implies that we may compute the subterminal truncation functor $(-)^{\Sub}\colon\IPSh(\I{C})\to\ISub(\IPSh(\I{C}))$ on the level of right fibrations by taking essential images: If $F\colon \I{C}^\op\to\Univ$ is a presheaf, then $F^{\Sub}$ classifies the essential image of the right fibration $\Over{\I{C}}{F}\to\I{C}$. In fact, this is a consequence of the straightforward observation that the essential image is still a right fibration.
\end{example}

By definition, if $\I{C}$ is a $\BB$-category, then $\ISub(\I{C})$ is a $\BB$-poset. Our next goal is to show that if $\I{C}$ is {presentable}, then $\ISub(\I{C})$ can be characterised as the \emph{largest} accessible Bousfield localisation of $\I{C}$ with that property.

\begin{lemma}
	\label{lem:SubFunctorial}
	Let $(l\dashv r)\colon\I{D}\leftrightarrows\I{C}$ be an adjunction of $\BB$-categories. Then there is a commutative square
	\begin{equation*}
		\begin{tikzcd}
			\ISub(\I{D})\arrow[r, "r"]\arrow[d, hookrightarrow]& \ISub(\I{C})\arrow[d, hookrightarrow]\\
			\I{D}\arrow[r, "r"] & \I{C}
		\end{tikzcd}
	\end{equation*}
	which is a pullback when $r$ is fully faithful.
\end{lemma}
\begin{proof}
	Unwinding the definitions, it will be enough to show that we have a commutative square
	\begin{equation*}
		\begin{tikzcd}
			\iFun(\I{D}^{\op},\ISub_{\BB})\arrow[d, hookrightarrow]\arrow[r, "r_!", hookrightarrow] & \iFun(\I{C}^{\op},\ISub_{\BB})\arrow[d, hookrightarrow]\\
			\IPSh(\I{D})\arrow[r, "r_!", hookrightarrow] & \IPSh(\I{C})
		\end{tikzcd}
	\end{equation*}
	that is a pullback when $r$ is fully faithful. The existence of this square follows from observing that $r_!$ can be identified with $l^\ast$.
	If $r$ is moreover fully faithful, then $l$ is a localisation and therefore in particular essentially surjective. Thus, the second claim follows.
\end{proof}

\begin{proposition}
	\label{prop:UMPSubterminal}
	Let $\I{D}$ be a presentable $\BB$-category and $\I{P}$ be a presentable $\BB$-poset. Then composition with the inclusion $\ISub_{\BB}(\I{D})\into\I{D}$ induces an equivalence
	\begin{equation*}
		\iFun^R(\I{P}, \ISub(\I{D}))\into\iFun^R(\I{P},\I{D}).
	\end{equation*}
\end{proposition}
\begin{proof}
	By Remark~\ref{rem:BCSubterminal}, it will be enough to show that every right adjoint functor $\I{P}\to \I{D}$ factors through $\ISub(\I{D})$. This follows immediately from Lemma~\ref{lem:SubFunctorial}.
\end{proof}

In light of Remark~\ref{rem:BCSubterminal}, Proposition~\ref{prop:UMPSubterminal} implies:
\begin{corollary}
	\label{cor:presentablePosetsCoreflective}
	The full subcategory of $\ILPr_{\BB}$ that is spanned by the presentable $\Over{\BB}{A}$-posets for all $A\in\BB$ is reflective, with the left adjoint given by sending a presentable $\BB$-category $\I{D}$ to $\ISub(\I{D})$.\qed
\end{corollary}

\begin{remark}
	\label{rem:subterminalObjectsDiagonal}
	Lemma~\ref{lem:SubFunctorial} furthermore implies that if $\I{D}$ is a presentable $\BB$-category and $d\colon 1\to\I{D}$ is an arbitrary object, then $d$ is subterminal if and only if the diagonal $d\to d\times d$ is an equivalence. In fact, by choosing a presentation of $\I{D}$ as an accessible Bousfield localisation of a presheaf $\BB$-category and making use of Lemma~\ref{lem:SubFunctorial}, we may assume that $\I{D}\simeq\IPSh(\I{C})$ and hence that $d$ can be identified with a right fibration over $\I{C}$. Then the claim follows immediately from Example~\ref{ex:subterminalPresheaves}.
	In particular, this observation implies that $\ISub(\I{D})$ can be identified with the sheaf $\Sub(\I{D}(-))$ on $\BB$.
\end{remark}

Finally, we arrive at the following characterisation of presentable $\BB$-posets:
\begin{proposition}
	\label{prop:classificationPresentablePosets}
	For an (a priori large) $\BB$-category $\I{D}$, the following are equivalent:
	\begin{enumerate}
		\item $\I{D}$ is a presentable $\BB$-poset;
		\item $\I{D}\simeq \ISub(\I{E})$ for some presentable $\BB$-category $\I{E}$;
		\item $\I{D}$ is small and cocomplete;
		\item $\I{D}$ is small, and the Yoneda embedding $h\colon \I{D}\into\IPSh(\I{D})$ admits a left adjoint;
		\item $\I{D}$ is a small $\BB$-poset, and the Yoneda embedding $h\colon \I{D}\into\iFun(\I{D}^{\op},\ISub_{\BB})$ admits a left adjoint.
	\end{enumerate}
\end{proposition}
\begin{proof}
	(1) and~(2) are equivalent by Proposition~\ref{prop:subterminalObjectsPresentable}. If $\I{D}$ is a presentable $\BB$-poset, then Lemma~\ref{lem:SubFunctorial} combined with Example~\ref{ex:subterminalPresheaves} shows that there is a small $\BB$-category $\I{C}$ such that $\I{D}$ arises as a Bousfield localisation of $\iFun(\I{C}^{\op},\ISub_{\BB})$ for some small $\BB$-category $\I{C}$. Hence, as the latter is small, so is $\I{D}$.
	Therefore~(1) implies~(3). Conversely, every small and cocomplete $\BB$-category is presentable (this follows from our characterisation of presentable $\BB$-categories as the accessible and cocomplete ones, see Corollary~\ref{cor:characterisationPresentableSheaves}), and since every small and cocomplete $\infty$-category is a poset, we conclude by employing Remark~\ref{rem:BPosetsSectionwise} that~(3) implies~(1).
	Furthermore,~(3) and~(4) are equivalent by the universal property of presheaf $\BB$-categories (see \cite[Theorem 7.1.1]{MWColimits}). Finally,~(4) implies~(5) by Lemma~\ref{lem:SubFunctorial} (and since we already know that~(3) forces $\I{D}$ to be a $\BB$-poset), and~(5) implies~(3) since $\ISub_{\BB}$ is cocomplete.
\end{proof}

We end this section with the observation that all colimits in a presentable $\BB$-poset are $\BB$-groupoidal and can be computed by an explicit formula:
\begin{proposition}
	\label{prop:colimitsBLocales}
	Let $\I{D}$ be a presentable $\BB$-poset and let $d\colon\I{I}\to\I{D}$ be a diagram. Then the inclusion $\I{I}^{\core}\to\I{I}$ induces an equivalence $\colim d\vert_{\I{I}^{\core}}\simeq \colim d$. Moreover, this colimit can be explicitly computed as
	\begin{equation*}
		\colim d\simeq \bigvee_{\substack{i\colon G\to \I{I}\\ G\in\GG}} (\pi_G)_!(d(i)),
	\end{equation*}
	where $\GG\into\BB$ is a small dense full subcategory.
\end{proposition}
\begin{proof}
	Consider the full subcategory $ \EE\subset\Over{\Cat(\BB) }{\I{D}}$ spanned by the diagrams $d\colon\I{I}\to\I{D}$ for which the inclusion $ \I{I}^{\core}\to \I{I} $ induces an equivalence $\colim d\vert_{\I{I}^\core}\simeq \colim d$. To prove the first statement, we need to show that $\EE=\Over{\Cat(\BB)}{\I{D}}$.
	
	To that end, first observe that if $h_{\I{D}}\colon\I{D}\into\IPSh(\I{D})$ denotes the Yoneda embedding, then $\colim h_{\I{D}} d$ classifies the right fibration $p\colon \I{P}\to \I{D}$ that arises from factoring $d$ into a final functor and a right fibration.
	In other words, $p$ is the image of $d$ under the localisation functor $L\colon \Over{\Cat(\BB)}{\I{D}}\to\RFib(\I{D})$, and we may compute the colimit of $ d $ by applying the left adjoint $ l \colon \IRFib(\I{D}) \to \I{D} $ (which exists by Proposition~\ref{prop:classificationPresentablePosets}) to $ p $.
	Since both $ l $ and $ L $ are cocontinuous, it follows that for every $\infty$-category $\KK$ and every diagram $\phi \colon \KK \to \Over{ \Cat(\BB) }{\I{D}}$ with colimit $d\colon\I{I}\to\I{D}$, we have a canonical equivalence $	\colim lL\phi \simeq \colim d$. 
	
	Now let $\phi^\core\colon \KK\to \Over{\Cat(\BB)}{\I{D}}$ be the composition of $\phi$ with the core $\BB$-groupoid functor $(-)^{\core}$. We then obtain a natural comparison map $\colim \phi^\core\to \I{I}^\core$ which has the property that the composition of this map with the inclusion $\I{I}^\core\to\I{I}$ can be identified with the colimit of the canonical morphism $\phi^\core\to\phi$. As a consequence, we obtain maps
	\[
	\colim lL \phi^\core \to \colim  d\vert_{\I{I}^{\core}} \to \colim d
	\]
	in which the composition can be identified with $\colim lL \phi^{\core}\to\colim lL\phi$. Hence, if $\phi$ takes values in $\EE$, the latter map is an equivalence, which implies that the map $\colim d\vert_{\I{I}^{\core}}\to\colim d$ is one as well since $ \I{D} $ is a poset. Thus, we conclude that $\EE$ is closed under colimits in $\Over{\Cat(\BB)}{\I{D}}$. 
	
	Consequently, since every $ \BB $-category can be written as a colimit of $\BB$-categories of the form $ G \otimes \Delta^n $  for $G\in \BB$ and $n\in\mathbb N$ (cf.~\cite[Lemma~4.5.2]{MYoneda}), it suffices to see that every diagram of the form $d\colon G\otimes\Delta^n\to \I{D}$ is in $\EE$.
	To that end, note that the colimit of a diagram $d \colon G \otimes \Delta^n \to \I{D}$ is given by applying $(\pi_G)_!\colon \I{D}(G)\to\I{D}(1)$ to the colimit of the transposed map $d^\prime\colon \Delta^n\to \I{D}(G)$, which is simply $d^\prime(n)$.
	Likewise, the colimit of the induced diagram $ \bigsqcup_n G =  (G \otimes \Delta^n)^\core\to \I{D}$ is given by applying $(\pi_{G})_!$ to the supremum of the objects $d'(i)$ for $i \in \Delta^n$.
	Since $ d'(i) \leq d'(n) $ for all $ i \in \Delta^n $, we deduce that $d\in\EE$, as desired.
	
	As for the second statement of the proposition, note that we have an equivalence
	\[
	\I{I}^\core \simeq \colim_{\substack{i\colon G\to \I{I}\\ G\in\GG}} G
	\]
	since $ \GG$ is dense in $ \BB $. 
	Thus the description of $\colim d\vert_{\I{I}^\core}$ follows from the observation at the beginning of the proof.
\end{proof}

\subsection{$\BB$-locales}
\label{sec:BLocales}
In this section we define what it means for a $\BB$-poset to be  a \emph{$\BB$-locale} and provide a first characterisation of this notion.

\begin{definition}
	\label{def:BLocale}
	A $\BB$-category $\I{L}$ is said to be a \emph{$\BB$-locale} if
	\begin{enumerate}
		\item $\I{L}$ is a $\BB$-poset,
		\item $\I{L}$ is presentable, and
		\item colimits are universal in $\I{L}$ (in the sense of \S~\ref{sec:universalColimits}).
	\end{enumerate}
	A functor $f\colon K\to L$ between $\BB$-locales is called an algebraic morphism of $\BB$-locales if it is cocontinuous and preserves finite limits. We denote by $\ILLoc_{\BB}\into \IPos_{\BB}$ the subcategory that is spanned by the algebraic morphisms of $\Over{\BB}{A}$-locales for all $A\in\BB$.
\end{definition}

\begin{remark}
	In the situation of Definition~\ref{def:BLocale}, note that colimits are universal in $\I{L}$ if and only if for every $A\in\BB$ and every $U\colon A\to\I{L}$ the functor $U\times -\colon \pi_A^\ast\I{L}\to\pi_A^\ast\I{L}$ is cocontinuous. In fact, since for every $V\colon A\to \I{L}$ with a map $U\to V$ the diagram
	\begin{equation*}
		\begin{tikzcd}[column sep=large]
			\Over{(\pi_A^\ast\I{L})}{V}\arrow[r, "U\times_V -"] \arrow[d, "(\pi_V)_!"] & \Over{(\pi_A^\ast\I{L})}{V}\arrow[d, "(\pi_V)_!"]\\
			\pi_A^\ast\I{L}\arrow[r, "U\times -"] & \pi_A^\ast\I{L}
		\end{tikzcd}
	\end{equation*}
	commutes, the composition $(\pi_V)_!(U\times_V -)$ is cocontinuous. As $(\pi_V)_!$ is conservative, this implies that $U\times_V -$ is already cocontinuous. 
\end{remark}

\begin{remark}
	\label{rem:BLocaleLocal}
	Since the property of a $\BB$-category $\I{L}$ to be a presentable $\BB$-poset is local in $\BB$ (combine Remark~\ref{rem:PresentabilityLocalCondition} with Remark~\ref{rem:localityBPosets}) and since the property of a functor to be cocontinuous is local in $\BB$ as well \cite[Remark~5.2.3]{MWColimits}, one concludes that for every cover $\bigsqcup A_i\onto 1$ in $\BB$, the $\BB$-category $\I{L}$ is a $\BB$-locale if and only if  the $\Over{\BB}{A_i}$-category $\pi_{A_i}^\ast \I{L}$ is a $\Over{\BB}{A_i}$-locale.
\end{remark}

\begin{remark}
	\label{rem:BCLocales}
	The subobject of $(\ICat_{\BBB})_1$ that is spanned by the algebraic morphisms between $\Over{\BB}{A}$-locales (for each $A\in\BB$) is stable under composition and equivalences in the sense of Proposition~\ref{prop:classificationSubcategories}.
	Since moreover cocontinuity and the property that a functor preserves finite limits are local conditions and on account of Remark~\ref{rem:BLocaleLocal}, we conclude that a map $A\to (\ICat_{\BBB})_1$ is contained in $(\ILLoc_{\BB})_1$ if and only if it defines an algebraic morphism between $\Over{\BB}{A}$-locales. In particular, if $\I{L}$ and $\I{M}$ are $\Over{\BB}{A}$-locales, the image of the monomorphism
	\begin{equation*}
		\map{\ILLoc_{\BB}}(\I{L},\I{M})\into\map{\ICat_{\BBB}}(\I{L},\I{M})
	\end{equation*}
	is spanned by the algebraic morphisms, and there is a canonical equivalence $\pi_A^\ast\ILLoc_{\BB}\simeq\ILLoc_{\Over{\BB}{A}}$.
\end{remark}

\begin{remark}
	\label{rem:BLocalesClassically}
	In light of Remark~\ref{rem:BPosetsClassically}, it is easy to see that $\LLoc(\BB)\into\Pos(\BB)$ can be identified with the category of internal locales in $\Disc(\BB)$ in the sense of~\cite[\S~C1.6]{johnstone2002}. In other words, our notion of an internal locale coincides with the classical one.
\end{remark}

\begin{lemma}
	\label{lem:BousfieldLocalisationUniversalityColimits}
	Let $\I{D}$ be a presentable $\BB$-category with universal colimits,  and let $l\colon\I{D}\to\I{L}$ be a Bousfield localisation that preserves binary products. Suppose furthermore that $\I{L}$ is a $\BB$-poset. $\I{L}$ is a $\BB$-locale.
\end{lemma}
\begin{proof}
	We need to show that colimits are universal in $\I{L}$, i.e.\ that for every $A\in\BB$ and every object $U\colon A\to \I{L}$ the functor $U\times - \colon \pi_A^\ast\I{L}\to\pi_A^\ast\I{L}$ is cocontinuous, or equivalently has a right adjoint. Now $l$ preserving binary products implies that $U\times -$ carries every map in $\I{D}$ (in arbitrary context) that is inverted by $l$ to one that is inverted by $l$ as well. Hence the functor $\ihom_{\I{D}}(i(U), i(-))$ (where $\ihom_{\I{D}}(-,-)$ is the internal hom in $\I{D}$) takes values in $\I{L}$, which yields the claim.
\end{proof}

\begin{proposition}
	\label{prop:CharacterisationsLocales}
	For a $\BB$-category $\I{L}$, the following are equivalent:
	\begin{enumerate}
		\item $\I{L}$ is a $\BB$-locale.
		\item \begin{enumerate}
			\item $\I{L}$ takes values in the $1$-category $\LLoc$ of locales;
			\item $\I{L}$ is $\Univ[\BB]$-cocomplete;
			\item for every map $s\colon B\to A$ in $\BB$, the functor $s_!\colon \I{L}(B)\to \I{L}(A)$ is a cartesian fibration.
		\end{enumerate}
		\item $\I{L}$ is small, and the Yoneda embedding $\I{L}\into\IPSh(\I{L})$ admits a left adjoint which preserves finite products;
		\item $\I{L}$ is a small $\BB$-poset, and the Yoneda embedding $\I{L}\into\iFun(\I{L}^{\op},\ISub_{\BB})$ admits a left exact left adjoint. 
	\end{enumerate}
\end{proposition}
\begin{proof}
	First, we show that~(1) and~(2) are equivalent. To that end, if $\I{L}$ is a $\BB$-locale, then for each $A\in\BB$ the $\infty$-category $\I{L}(A)$ is a presentable poset in which colimits are universal. Therefore, $\I{L}(A)$ is a locale. Moreover, $\I{L}$ being cocomplete implies that for every map $s\colon B\to A$ in $\BB$ the transition functor $s^\ast\colon \I{L}(A)\to \I{L}(B)$ is cocontinuous. Likewise, $\I{L}$ having finite limits implies that $s^\ast$ preserves finite limits. Therefore~(2a) follows. Moreover, condition~(2b) is part of the definition of a $\BB$-locale, and condition~(2c) is a reformulation of the condition that $\Univ$-colimits are universal in $\I{L}$ (see Example~\ref{ex:universalityGroupoidalColimitsCartesianFibration}).
	Conversely, if the three conditions in~(2) are satisfied, then $\I{L}$ is both $\Univ$- and $\ILConst$-cocomplete and section-wise presentable. Hence Theorem~\ref{thm:characterisationPresentableCategories} implies that $\I{L}$ is presentable. By Remark~\ref{rem:BPosetsSectionwise}, the assumption that $\I{L}$ is section-wise given by a poset implies that $\I{L}$ is a $\BB$-poset. Finally, the fact that $\I{L}$ takes values in $\Loc$ implies that $\ILConst$-colimits are universal in $\I{L}$, so that it suffices to verify that $\Univ$-colimits are universal in $\I{L}$ as well. Again, this is a consequence of Example~\ref{ex:universalityGroupoidalColimitsCartesianFibration}.
	
	Next, if $\I{L}$ is a $\BB$-locale, then Proposition~\ref{prop:classificationPresentablePosets} implies that the Yoneda embedding $\I{L}\into\IPSh(\I{L})$ has a left adjoint $l$. Moreover, as colimits are universal in $\I{L}$ and as $\IPSh(\I{L})$ is generated by $\I{L}$ under colimits,  the comparison map $l(-\times -)\to l(-)\times l(-)$ is an equivalence already when its restriction to $\I{L}$ is one, which is trivially true. Hence~(1) implies~(3). If we assume~(3), then Proposition~\ref{prop:classificationPresentablePosets} implies that $\I{L}$ is a small $\BB$-poset and that the Yoneda embedding $\I{L}\into\iFun(\I{L}^{\op},\ISub_{\BB})$ has a left adjoint. Explicitly, this left adjoint arises as the restriction of the left adjoint $\IPSh(\I{L})\to\I{L}$ and therefore preserves finite products. But since pullbacks in $\BB$-posets coincide with binary products, this is already enough to conclude that this functor is left exact. Hence~(4) follows. Finally, if~(4) holds, then $\I{L}$ is presentable by Proposition~\ref{prop:classificationPresentablePosets}. Moreover, using Lemma~\ref{lem:BousfieldLocalisationUniversalityColimits} it will be enough to show that the subterminal truncation functor $(-)^{\Sub}\colon\IPSh(\I{L})\to\iFun(\I{L}^{\op},\ISub_{\BB})$ preserves binary products, which is an immediate consequence of Example~\ref{ex:subterminalPresheaves}.
\end{proof}

Using Proposition~\ref{prop:CharacterisationsLocales}, it is now easy to show that the $\BB$-poset of subterminal objects in a $\BB$-topos is a $\BB$-locale. More precisely, one has:

\begin{proposition}
	\label{prop:coreflectionTopoiLocales}
	The functor $\ISub\colon \ILPr_{\BB}\to \ILPr_{\BB}$ from Corollary~\ref{cor:presentablePosetsCoreflective} restricts to a functor $\ISub\colon \ILTop_{\BB}\to\ILLoc_{\BB}$.
\end{proposition}
\begin{proof}
	By combining Remark~\ref{rem:BCSubterminal} and Remark~\ref{rem:BCLocales}, it is enough to show that for every algebraic morphism $f^\ast\colon \I{X}\to\I{Y}$ of $\BB$-topoi the induced map $\ISub(f^\ast)\colon \ISub(\I{X})\to\ISub(\I{Y})$ is an algebraic morphism of $\BB$-locales. First, let us show that $\ISub(\I{X})$ (and therefore by symmetry also $\ISub(\I{Y})$) is a $\BB$-locale. To that end, choose a presentation of $\I{X}$ as a left exact and accessible Bousfield localisation $L\colon \IPSh(\I{C})\to\I{X}$. Then Remark~\ref{rem:subterminalObjectsDiagonal} implies that $L$ restricts to a left exact and accessible Bousfield localisation $\iFun(\I{C}^{\op},\ISub_{\BB})\to\ISub(\I{X})$, hence the claim follows from Proposition~\ref{prop:CharacterisationsLocales}. Second, since we already know that $\ISub(f^\ast)$ is cocontinuous, it is enough to show that it is left exact as well. But on account of Remark~\ref{rem:subterminalObjectsDiagonal}, this functor arises as the restriction of $f^\ast$ to subterminal objects, which immediately yields the claim. 
\end{proof}

\subsection{Sheaves on a $\BB$-locale}
\label{sec:SheavesOnBLocales}
In this section we introduce and study the $\BB$-category of \emph{sheaves} on a $\BB$-locale.
\begin{definition}
	\label{def:covering}
	Let $\I{L}$ be a $\BB$-locale and let $U\colon A\to \I{L}$ be an object. A \emph{covering} of $U$ is a diagram $d\colon \I{G}\to\pi_A^\ast\I{L}$ with colimit $U$, where $\I{G}$ is a $\Over{\BB}{A}$-groupoid.
\end{definition}
\begin{remark}
	\label{rem:coveringsExplicitly}
	Explicitly, a covering of $U$ is given by a map	$s\colon B\to A$ in $\BB$ together with an object $V\colon B\to \I{L}$ such that $s_!(V)\simeq U$.
\end{remark}

\begin{example}
	\label{ex:usualCoveringsAreRecovered}
	Let $\I{L}$ be a $\BB$-locale and  $U\colon A\to\I{L}$ be an object. Then every covering $(j_i\colon U_i\to U)_{i\in I}$ in $\I{L}(A)$ (in the conventional sense) can be regarded as a covering in the sense of Definition~\ref{def:covering} by setting $\I{G}=I$ and $d=(j_i)_{i\in I}$.
\end{example}

Recall from Proposition~\ref{prop:CharacterisationsLocales} that if $\I{L}$ is a $\BB$-locale, the Yoneda embedding $h_{\I{L}}\colon\I{L}\into\iFun(\I{L}^{\op},\ISub_{\BB})$ admits a (left exact) left adjoint $l$. We denote by $\eta\colon \id_{\iFun(\I{L}^{\op},\ISub_{\BB})}\to h_{\I{L}}l$ the adjunction unit.
\begin{definition}
	\label{def:coveringSieve}
	Let $\I{L}$ be a $\BB$-locale and let $d\colon \I{G}\to\pi_A^\ast\I{L}$ be a covering of an object $U\colon A\to\I{L}$. Then the induced map $\eta \colim h_{\I{L}}d\colon S_{d}=\colim h_{\I{L}}d\into h_{\I{L}}(U)$ in $\iFun(\I{L}^{\op},\ISub_{\BB})$ is referred to as the \emph{covering sieve} associated with $d$.
\end{definition}

\begin{remark}
	\label{rem:coveringSieveLocal}
	Let $\I{L}$ be a $\BB$-locale and $d\colon\I{G}\to\pi_A^\ast\I{L}$ be a covering of an object $U\colon A\to\I{L}$. Then, for every map $s\colon B\to A$ in $\BB$, we obtain an equivalence $s^\ast S_d\simeq S_{s^\ast d}$ that commutes with the canonical equivalence $s^\ast h_{\I{L}}(U)\simeq h_{\I{L}}(s^\ast U)$.
\end{remark}

\begin{definition}
	\label{def:coveringSubcategory}
	If $\I{L}$ is a $\BB$-locale, we denote by $\ICov\into\iFun(\I{L}^{\op},\ISub_{\BB})\into\IPSh(\I{L})$ the subcategory that is spanned by the covering sieves in arbitrary context.
\end{definition}

\begin{remark}
	\label{rem:localityCov}
	For every $A\in\BB$, one may identify $\pi_A^\ast\ICov\into\iFun(\pi_A^\ast\I{L}^\op,\ISub_{\Over{\BB}{A}})$ with the subcategory of covering sieves in $\pi_A^\ast\I{L}$.
\end{remark}

\begin{remark}
	\label{rem:coveringSubcategorySmall}
	The $\BB$-category $\ICov$ is small. In fact, first note that by Remark~\ref{rem:coveringSieveLocal}, the subcategory $\ICov\into\IPSh(\I{L})$ is already spanned by all covering sieves of objects in context $G\in \GG$, where $\GG$ is a small full subcategory of $\BB$ that generates $\BB$ under colimits. Furthermore, since $\I{L}$ is small, the collection of all coverings of objects in fixed context $G$ is parametrised by a small set. Hence, the full subcategory of $\IPSh(L)^{\Delta^1}$ that is spanned by the covering sieves must be small.
	In light of the very construction of a subcategory from a collection of morphisms (see \cite[\S~B.2]{MWColimits}), the claim thus follows from the fact that the $1$-image of a small $\BB$-category in a locally small $\BB$-category must also be small (see for example~\cite[Lemma~4.7.5]{MYoneda}).
\end{remark}

\begin{definition}
	\label{def:Sheaves}
	Let $\I{L}$ be a $\BB$-locale. We define the $\BB$-category $\IShv_{\BB}(\I{L})$ of \emph{sheaves} on $\I{L}$ to be the Bousfield localisation $\ILoc_{\ICov}(\IPSh(\I{C}))$. We will furthermore denote the underlying $\infty$-category of global sections of $\IShv_{\BB}(\I{L})$ by $\Shv_{\BB}(\I{L})$.
\end{definition}

\begin{remark}
	\label{rem:localityPrincipleSheaves}
	By Remark~\ref{rem:localityCov}, for every $A\in\BB$ there is a canonical equivalence $\pi_A^\ast\IShv_{\BB}(\I{L})\simeq\IShv_{\Over{\BB}{A}}(\pi_A^\ast\I{L})$ of full subcategories of $\IPSh[\Over{\BB}{A}](\pi_A^\ast\I{L})$.
\end{remark}

\begin{remark}
	\label{rem:representablesAreSheaves}
	If $\I{L}$ be a $\BB$-locale, then Proposition~\ref{prop:CharacterisationsLocales} implies that the Yoneda embedding $\I{L}\into\IPSh(\I{L})$ admits a left adjoint $l\colon \IPSh(\I{L})\to\I{L}$. By construction, this functor carries $\ICov$ into $\I{L}^\core$. In other words, $l$ factors through the sheafification functor $\IPSh(\I{L})\to \IShv_{\BB}(\I{L})$. By passing to right adjoints, this implies that the Yoneda embedding factors through the inclusion $\IShv_{\BB}(\I{L})\into\IPSh(\I{L})$, which means that every representable presheaf on $\I{L}$ is already a sheaf.
\end{remark}

The main goal of this section is to prove that $\IShv_{\BB}(\I{L})$ is a $\BB$-topos. More precisely, we will show:
\begin{proposition}
	\label{prop:SheavesFormTopos}
	For any $\BB$-locale $\I{L}$, the localisation functor $\IPSh(\I{L})\to\IShv_{\BB}(\I{L})$ preserves finite limits. In particular, $\IShv_{\BB}(\I{L})$ is a $\BB$-topos.
\end{proposition}
The proof of Proposition~\ref{prop:SheavesFormTopos} is based on the following three lemmas:
\begin{lemma}
	\label{lem:CoveringSievesExternally}
	For every $\BB$-locale $\I{L}$, the $\infty$-category $\Shv_{\BB}(\I{L})$ is the Bousfield localisation of $\PSh_{\BB}(\I{L})$ at the set
	\begin{equation*}
		W= \{(\pi_A)_!(S)\into(\pi_A)_!h(U)~\vert~ A\in\BB,~U\colon A\to\I{L},~S\into h(U)~\textrm{covering sieve}\}
	\end{equation*}
	of morphisms in $\PSh_{\BB}(\I{L})$.
\end{lemma}
\begin{proof}
	A presheaf $F\colon\I{L}^\op\to\Univ$ is a sheaf if and only if for every $A\in\BB$ and every covering sieve $S\into h(U)$ in context $A$ the morphism $\phi\colon\map{\IPSh(\I{L})}(h(U), \pi_A^\ast F)\to\map{\IPSh(\I{L})}(S, \pi_A^\ast F)$ is an equivalence in $\Over{\BB}{A}$. Recall from~\cite[Corollary~4.6.8]{MYoneda} that if $s\colon B\to A$ is a map in $\BB$, then on local sections over $A$ the map $\phi$ recovers the morphism
	\begin{equation*}
		\map{\IPSh(\I{L})(B)}(s^\ast h(U), \pi_B^\ast F)\to \map{\IPSh(\I{L})(B)}(s^\ast S, \pi_B^\ast F)
	\end{equation*}
	of mapping $\infty$-groupoids, which by adjunction can in turn be identified with the map
	\begin{equation*}
		\map{\PSh_{\BB}(\I{L})}((\pi_B)_! s^\ast h(U), F)\to \map{\PSh_{\BB}(\I{L})}(\pi_B)_! s^\ast S,  F).
	\end{equation*}
	Hence $F$ is a sheaf precisely if the latter map is an equivalence for every covering sieve $S\into h(U)$ in context $A$ and every map $s\colon B\to A$ in $\BB$. Together with Remark~\ref{rem:coveringSieveLocal}, this yields the claim.
\end{proof}

\begin{lemma}
	\label{lem:localEquivalencesClosedUnderPullbacks}
	Let $\I{L}$ be a locale and let $S\into h(U)$ be a covering sieve on an object $U\colon A\to\I{L}$. Then for every map $V\to U$ in $\I{L}(A)$ the map $h(V)\times_{h(U)} S\into h(V)$ is a covering sieve.
\end{lemma}
\begin{proof}
	We may assume without loss of generality that $A\simeq 1$. Now if $d\colon\I{G}\to\I{L}$ is a covering of $U$ giving rise to the covering sieve $S$, then universality of colimits in $\iFun(\I{L}^{\op},\ISub_{\BB})$ and the fact that $h$ preserves limits implies that $h (V)\times_{h(U)}S$ is the colimit of the diagram
	\begin{equation*}
		\I{G}\xrightarrow{d}\I{L}\xrightarrow{-\times V}\I{L}\xhookrightarrow{h}\iFun(\I{L}^{\op},\ISub_{\BB}).
	\end{equation*}
	Since universality of colimits in $\I{L}$ implies that the diagram $d(-)\times V\colon\I{G}\to\I{L}$ is a covering of $V$, the claim follows.
\end{proof}

\begin{lemma}
	\label{lem:coveringSievesProducts}
	Let $\I{L}$ be a $\BB$-locale and let $S_0\into h(U)$ and $S_1\into h(U)$ be covering sieves on an object $U\colon A\to\I{L}$. Then $S_0\times_{h(U)} S_1\into h(U)$ is a covering sieve as well.
\end{lemma}
\begin{proof}
	We may assume without loss of generality that $A\simeq 1$.
	Let $d_0\colon \I{G}_0\to\I{L}$ be a covering giving rise to the covering sieve $S_0$, and let $d_1\colon \I{G}_1\to\I{L}$ be a covering giving rise to $S_1$. Define $\I{G}=\I{G}_0\times\I{G}_1$ and let $d\colon \I{G}\to\I{L}$ be the diagram given by the composition
	\begin{equation*}
		\I{G}_0\times\I{G}_1\xrightarrow{d_0\times d_1}\I{L}\times\I{L}\xrightarrow{- \times - } \I{L}.
	\end{equation*}
	Then we have $\colim d \simeq U$ since colimits are universal in $\I{L}$ and since $U\times U\simeq U$ in $\I{L}$. Therefore, it is enough to show that the induced map $\colim h_{\I{L}}d\into h(U)$ in $\iFun(\I{L}^\op, \ISub_{\BB})$ can be identified with $S_0\times_{h(U)} S_1\into h(U)$ This follows from the fact that $h_{\I{L}} d$ is given by the composition
	\begin{equation*}
		\I{G}_0\times\I{G}_1\xrightarrow{h_{\I{L}}d_0\times h_{\I{L}}d_1} \iFun(\I{L}^\op, \ISub_{\BB})\times\iFun(\I{L}^\op,\ISub_{\BB})\xrightarrow{-\times -}\iFun(\I{L}^\op,\ISub_{\BB})
	\end{equation*}
	and the very same argument as above, using that colimits are universal in $\iFun(\I{L}^\op,\ISub_{\BB})$ as well.
\end{proof}

\begin{proof}[{Proof of Proposition~\ref{prop:SheavesFormTopos}}]
	Since the localisation is already accessible (being a Bousfield localisation at a small subcategory, see Theorem~\ref{thm:characterisationPresentableCategories}), the second claim follows from the first by Theorem~\ref{thm:presentationTopoi}. To prove the first, let $T^\prime(A)$ be the class of monomorphisms  $f\colon G\into H$ in the $\infty$-topos $\IPSh[\BB](\I{L})(A)$ (for arbitrary $A\in\BB$) satisfying the condition that for every map $s\colon B\to A$ in $\BB$, every $U\colon B\to\I{L}$ and every map $h(U)\to s^\ast H$ the pullback $s^\ast G\times_{s^\ast H} h(U)\into h(U)$ is a covering sieve in context $B$.  Then $T^\prime(A)$ has the following properties:
	\begin{enumerate}
		\item the maps in $T^\prime(A)$ are closed under pullbacks in $\IPSh(\I{L})(A)$;
		\item the maps in $T^\prime(A)$ are closed under finite limits in $\Fun(\Delta^1,\IPSh(\I{L})(A))$;
		\item every map in $T^\prime(A)$ is inverted by the localisation functor $\IPSh[\BB](\I{L})(A)\to\IShv_{\BB}(\I{L})(A)$;
		\item every covering sieve in context $A\in\BB$ is contained in $T^\prime(A)$. 
	\end{enumerate}
	In fact, the first property is evident, and the second property follows from combining the first one with Lemma~\ref{lem:coveringSievesProducts}. Property~(3) follows from the observation that by descent in the $\Over{\BB}{A}$-topos $\IPSh[\Over{\BB}{A}](\pi_A^\ast\I{L})$, every map in $T^\prime(A)$ is a ($\Over{\BB}{A}$-internal) colimit of covering sieves, which implies (using Remark~\ref{rem:localityPrincipleSheaves}) that it is inverted by the localisation functor $\IPSh[\BB](\I{L})(A)\to\IShv_{\BB}(\I{L})(A)$. The last property is an immediate consequence of Lemma~\ref{lem:localEquivalencesClosedUnderPullbacks}.
	
	Let us set $T^\prime=\bigcup_{A\in\BB}(\pi_A)_! T^\prime(A)$ and let $T$ be the smallest local class of morphisms in $\PSh_{\BB}(\I{L})$ that contains $T^\prime$.
	Then $T$ is bounded since it only contains monomorphisms.
	Moreover, $T$ is closed under finite limits in $\Fun(\Delta^1,\PSh_{\BB}(\I{L}))$.
	To see this, the fact that every map in $T$ is locally (in the $\infty$-topos $\PSh_{\BB}(\I{L})$) contained in $T^\prime$ implies that it suffices to show that for every cospan $f_0\to f \leftarrow f_1$ with $f_0$ and $f_1$ in $T^\prime$, their pullback is in $T^\prime$ as well. Note that if $s\colon B\to A$ is a map in $\BB$ and if $g$ is a map in $\IPSh[\BB](\I{L})(B)$ such that $s_!(g)\in T^\prime(A)$, we have $g\in T^\prime(B)$.
	Therefore, we may assume that both $f_0$ and $f_1$ are in $T^\prime(A)$ for some $A\in\BB$. In this case, the claim immediately follows from properties~(1) and~(2) of $T^\prime(A)$.
	The same argument moreover shows that every map in $T$ is inverted by the localisation functor $\PSh_{\BB}(\I{L})\to\Shv_{\BB}(\I{L})$ as it can be written as a $\Delta^{\op}$-indexed colimits of maps in $T^\prime$.
	
	By employing Lemma~\ref{lem:CoveringSievesExternally} and property~(4) above, we now conclude that there is an equivalence $\Shv_{\BB}(\I{L})\simeq\Loc_{T}(\PSh_{\BB}(\I{L}))$ of Bousfield localisations of $\PSh_{\BB}(\I{L})$.
	Using Proposition~\ref{prop:localClassSubtoposSheafification}, we thus obtain hat $\PSh_{\BB}(\I{L})\to\Shv[\BB](\I{L})$ is left exact.
	In light of Remark~\ref{rem:localityPrincipleSheaves}, this is already sufficient to conclude that the entire functor of $\BB$-categories $\IPSh(\I{L})\to\IShv_{\BB}(\I{L})$ is left exact.
\end{proof}

\subsection{The localic reflection of $\BB$-topoi}
\label{sec:LocalicReflection}
In the previous section, we introduces the $\BB$-topos of sheaves on a $\BB$-locale. In this section, we show that this construction is the universal way to attach a $\BB$-topos to a $\BB$-locale. More precisely, we show:

\begin{proposition}
	\label{prop:UMPSheaves}
	Let $\I{L}$ be a $\BB$-locale. Then the Yoneda embedding $h\colon \I{L}\into\IShv_{\BB}(\I{L})$ induces an equivalence $\I{L}\simeq\ISub(\IShv_{\BB}(\I{L}))$, and for every $\BB$-topos $\I{X}$ precomposition with $h$ induces an equivalence
	\begin{equation*}
		\iFun^\alg(\IShv_{\BB}(\I{L}),\I{X})\simeq \iFun^\alg(\I{L},\ISub(\I{X})),
	\end{equation*}
	where the left-hand side denotes the $\BB$-category of algebraic morphisms between the $\BB$-topoi $\IShv_{\BB}(\I{L})$ and $\I{X}$ and the right-hand side denotes the $\BB$-category of algebraic morphisms between the $\BB$-locales $\I{L}$ and $\ISub(\I{X})$.
\end{proposition}
\begin{proof}
	We begin by showing the first claim. To that end, note that by Lemma~\ref{lem:SubFunctorial} a sheaf $F\colon \I{L}^\op\to\Univ$ is subterminal if and only if it takes values in $\ISub_{\BB}$. Together with the usual base change arguments, this implies that the first claim follows once we verify that every such sheaf $F\colon \I{L}^\op\to\ISub_{\BB}$ is representable. Note that by Example~\ref{ex:subterminalPresheaves}, the associated right fibration $p\colon\Over{\I{L}}{F}\to\I{L}$ is fully faithful. Let $U\colon 1\to\I{L}$ be the colimit of $p$. We then obtain a canonical map $F\to h(U)$ in $\ISub_{\BB}(\IShv_{\BB}(\I{L}))$.
	To show the claim, it is therefore enough to produce a map in the opposite direction, which by Yoneda's lemma is equivalent to show that $F(U)\simeq 1_{\Univ}$. To see this, note that by Proposition~\ref{prop:colimitsBLocales} the object $U$ is the colimit of the restriction of $p$ to $\I{G}=(\Over{\I{L}}{F})^\core$. In other words, we have a covering of $U$ given by $p\vert_{\I{G}}$. Let $S\into h(U)$ be the associated covering sieve. Then, since $F$ is a sheaf, we obtain an equivalence $F(U)\simeq \map{\IPSh(\I{L})}(S, F)$. To complete the proof of the first claim, we therefore need to show that the right-hand side can be identified with $1_{\Univ}$. But as $F$ is subterminal, we may in turn identify $\map{\IPSh(\I{L})}(S, F)$ with $\map{\iFun(\I{L}^{\op},\ISub_{\BB})}(S, F)\simeq\lim F p\vert_{\I{G}}$.
	Thus, the claim follows once we show that $Fp\vert_{\I{G}}\colon\I{G}\to \ISub_{\BB}$ is final in $\iFun(\I{G},\ISub_{\BB})$. Note that the associated object $P\into \I{G}$ in $\Sub(\Over{\BB}{\I{G}})$ is explicitly obtained as the fibre of $p\colon \Over{\I{L}}{F}\to\Univ$ over $p\vert_{\I{G}}$. Thus, the inclusion $\I{G}\into\Over{\I{L}}{F}$ induces a section $\I{G}\to P$, which immediately yields the claim.
	
	We now show the second claim. Let $l\colon \IPSh(\I{L})\to\IShv_{\BB}(\I{L})$ be the localisation functor. We now have maps
	\begin{equation*}
		\iFun^{\alg}(\IShv_{\BB}(\I{L}),\I{X})\xhookrightarrow{l^\ast} 
		\iFun^{\alg}(\IPSh(\I{L}),\I{X})\simeq \iFun^{\lex}(\I{L},\ISub(\I{X})) \hookleftarrow\iFun^\alg(\I{L},\ISub(\I{X}))
	\end{equation*}
	in which the fact that $l^\ast$ is fully faithful follows from the universal property of localisations and where the equivalence in the middle follows from Diaconescu's theorem~\ref{cor:internalDiaconescu} and the straightforward observation that by Remark~\ref{rem:subterminalObjectsDiagonal} every left exact functor $\pi_A^\ast\I{L}\to\pi_A^\ast\I{X}$ necessarily factors through $\pi_A^\ast\ISub(\I{X})$.
	Thus, by using Remarks~\ref{rem:localityPrincipleSheaves} and~\ref{rem:BCSubterminal} together with \cite[Remark~5.2.3]{MWColimits}, the claim follows once we show that a left exact functor $f\colon \I{L}\to \ISub(\I{X})$ is cocontinuous if and only if the left Kan extension $h_!(if)\colon\IPSh(\I{L})\to\I{X}$ (where $i\colon \ISub(\I{X})\into\I{X}$ is the inclusion) carries $\ICov$ into $\I{X}^{\core}$.
	To see this, note that as $h_!(if)$ is an algebraic morphism, it restricts to a functor $\iFun(\I{L}^{\op},\ISub_{\BB})\to\ISub(\I{X})$ which is cocontinuous as well. Consequently, for every object $U\colon A\to\I{L}$ and every covering $d\colon \I{G}\to\pi_A^\ast\I{L}$ of $U$, the image of the associated covering sieve $S_d\into h(U)$ along $h_!(if)$ can be identified with the canonical morphism $\colim f d\to f(U)$. In other words, $h_!(if)$ carries $\ICov$ into $\I{X}^\simeq$ precisely if $f$ is $\Univ$-cocontinuous.
	But by using Proposition~\ref{prop:colimitsBLocales}, this already implies that $f$ is cocontinuous.
	Hence the claim follows.
\end{proof}
\begin{corollary}
	\label{cor:localicReflection}
	The $\BB$-category $\ILLoc_{\BB}$ is a coreflective subcategory of $\ILTop_{\BB}$, where the inclusion is given by carrying a $\BB$-locale to its associated sheaf $\BB$-topos and the coreflection sends a $\BB$-topos to its underlying $\BB$-locale of subterminal objects.
\end{corollary}
\begin{proof}
	Combine Proposition~\ref{prop:coreflectionTopoiLocales} with Proposition~\ref{prop:UMPSheaves}.
\end{proof}

\begin{definition}
	\label{def:localicBTopos}
	A $\BB$-topos $\I{X}$ is \emph{localic} if it is contained in the essential image of $\IShv_{\BB}(-)\colon\ILLoc_{\BB}\into\ILTop_{\BB}$.
\end{definition}

\subsection{Localic $\BB$-topoi as relative locales}
\label{sec:BLocalesVsRelativeLocales}
If $\BB$ is a localic $\infty$-topos, then $\BB$-locales precisely correspond to locales under $\Sub(\BB)$. More precisely, note that as a consequence of Corollary~\ref{cor:localicReflection}, the $\BB$-locale $\ISub_{\BB}$ is the \emph{initial} $\BB$-locale. Since moreover the global sections functor $\Gamma_{\BB}$ restricts to a functor $\LLoc(\BB)\to \LLoc$, we thus obtain an induced functor $\LLoc(\BB)\to \Under{\LLoc}{\Sub(\BB)}$.
\begin{proposition}
	\label{prop:internalVsRelativeLocales}
	If $\BB$ is a localic $\infty$-topos, then the functor $\Gamma\colon \LLoc(\BB)\to\Under{\LLoc}{\Sub(\BB)}$ is an equivalence of $\infty$-categories.
\end{proposition}
\begin{proof}
	Since by Remark~\ref{rem:BLocalesClassically} the $\infty$-category $\LLoc(\BB)$ can be identified with the $1$-category of internal locales in $\Disc(\BB)$, the statement reduces to the analogous result in $1$-topos theory, see~\cite[Theorem~C1.6.3]{johnstone2002}.
\end{proof}

\begin{corollary}
	\label{cor:GlobalSectionsInternalSheaves}
	For every $\BB$-locale $\I{L}$, the $\infty$-topos $\Shv_{\BB}(\I{L})$ can be canonically identified with $\Shv(\Gamma\I{L})$.
\end{corollary}
\begin{proof}
	As a result of Remark~\ref{rem:subterminalObjectsDiagonal}, we have a commutative diagram
	\begin{equation*}
		\begin{tikzcd}
			\LTop(\BB)\arrow[d, "\ISub"]\arrow[r, "\Gamma"] & \Under{(\LTopS)}{\BB}\arrow[d, "\Sub"]\\
			\LLoc(\BB)\arrow[r, "\Gamma"] & \Under{\LLoc}{\Sub(\BB)}
		\end{tikzcd}
	\end{equation*}
	(where we note that as $\LLoc$ is a $1$-category coherence issues do not arise). Therefore, the claim follows from Proposition~\ref{prop:internalVsRelativeLocales} and the fact that by Theorem~\ref{thm:BTopoiRelativeTopoi} the upper horizontal map is an equivalence as well.
\end{proof}

\begin{remark}
	\label{rem:BLocaleFromRelativeLocale}
	The inverse to the equivalence from Proposition~\ref{prop:internalVsRelativeLocales} can be made explicit as follows: Given an algebraic morphism of locales $f^\ast\colon \Sub(\BB)\to L$, let $f^\ast\colon \BB\to\Shv(L)$ be the associated algebraic morphism of $\infty$-topoi. Then $f_\ast\ISub_{\Shv(L)}$ is a $\BB$-locale (as can be easily verified using Proposition~\ref{prop:CharacterisationsLocales}) whose underlying locale recovers $L$ and is therefore the $\BB$-locale associated to $L$. Explicitly, this $\BB$-locale can be described as the sheaf $\Over{L}{f^\ast(-)}$ on $\Sub(\BB)$, i.e.\ the $\CatSS$-valued functor that is classified by the cartesian fibration $\Sub(\BB)\times_{L} \Fun(\Delta^1,\I{L})\to \Sub(\BB)$.
\end{remark}

\subsection{Compactness conditions for $\BB$-locales}
\label{sec:compactnessLocalicTopoi}
In this section, we study how certain compactness properties of $\BB$-locales are inherited by their associated localic $\BB$-topoi. To that end, recall that if $\I{D}$ is a presentable $\BB$-category, we say that $\I{D}$ is \emph{compactly generated} if the inclusion $\I{D}^{\compact}\into\I{D}$ of the full subcategory of compact objects induces via left Kan extension an equivalence $\IInd(\I{D}^{\compact})\simeq\I{D}$. We will furthermore say that $\I{D}$ is \emph{compactly assembled} if $\I{D}$ is a retract (in $\ILPr_{\BB}$) of a compactly generated $\BB$-category. We may now define:
\begin{definition}
	\label{def:locallyStablyCompactBLocale}
	A $\BB$-locale $\I{L}$ is said to be \emph{locally coherent} if it is compactly generated and if $\I{L}^{\compact}$ is closed under binary products in $\I{L}$. We say that $\I{L}$ is \emph{coherent} if it is locally coherent and $1_{\I{L}}$ is compact.
	
	Furthermore, $\I{L}$ is said to be \emph{(locally) stably compact} if it is a retract in $\ILLoc_{\BB}$ of a (locally) coherent $\BB$-locale.
\end{definition}

\begin{remark}
	\label{rem:localityCoherence}
	Since the existence and preservation of limits is local in $\BB$ by \cite[Remark 5.2.3]{MWColimits} and one has $\pi_A^\ast\IInd(\I{L}^{\compact})\simeq\IInd[\Over{\BB}{A}](\pi_A^\ast\I{L}^{\compact})$ by Remark~\ref{rem:BCUComp} and Remark~\ref{rem:BCInd} for every $A\in\BB$, we deduce that for every cover $(\pi_{A_i})\colon\colon\bigsqcup_i A_i\onto 1$ in $\BB$, a $\BB$-locale $\I{L}$ is (locally) coherent if and only if $\pi_{A_i}^\ast \I{L}$ is a (locally) coherent $\Over{\BB}{A_i}$-locale for every $i$.
\end{remark}

Then main goal of this section is to relate sheaves on a locally coherent locale with \emph{finitary sheaves} on its compact objects:
\begin{definition}
	\label{def:finitarySheaf}
	Let $\I{P}$ be a $\BB$-poset with finite colimits and binary products. A presheaf $F\colon\I{P}^\op\to\Univ$ is said to be a  \emph{finitary sheaf} if
	\begin{enumerate}
		\item $F(\varnothing_{\I{P}})\simeq 1_{\Univ}$;
		\item for every two objects $U,V\colon A\rightrightarrows\I{P}$ in arbitrary context $A\in\BB$, the commutative square
		\begin{equation*}
			\begin{tikzcd}
				F(U\vee V)\arrow[r]\arrow[d] & F(V)\arrow[d]\\
				F(U)\arrow[r] & F(U\wedge V)
			\end{tikzcd}
		\end{equation*}
		is a pullback.
	\end{enumerate}
	We let $\IShv_{\BB}^{\fin}(\I{P})$ be the full subcategory of $\IPSh(\I{P})$ that is spanned by those presheaves $\pi_A^\ast\I{P}^\op\to\Univ[\Over{\BB}{A}]$ (in arbitrary context $A\in\BB$) which are finitary sheaves on $\pi_A^\ast\I{P}$.
\end{definition}

\begin{remark}
	\label{rem:localityFinitarySheaves}
	As preservation of (co)limits is a local property \cite[Remark~4.2.1]{MWColimits}, we deduce that for every cover $\bigsqcup_i A_i\onto 1$ in $\BB$ a presheaf $F\colon \I{P}^\op\to\Univ$ is a finitary sheaf if and only if the presheaf $\pi_{A_i}^\ast(F)$ is a finitary sheaf on $\pi_A^\ast\I{P}$.
	In particular, an object $A\to\IPSh(\I{P})$ is contained in $\IShv_{\BB}^{\fin}(\I{P})$ if and only if it transposes to a finitary sheaf on $\pi_A^\ast\I{P}$, and we obtain a canonical equivalence of $\BB$-categories $\pi_A^\ast\IShv_{\BB}^{\fin}(\I{P})\simeq\IShv_{\Over{\BB}{A}}^{\fin}(\pi_A^\ast\I{P})$ for every $A\in\BB$.
\end{remark}

Recall from \S~\ref{sec:sheaves} that if $\I{C}$ is a $\IFilt$-cocomplete $\BB$-category, we denote by $\IShv_{\BB}^{\IFilt}(\I{C})$ the full subcategory of $\IPSh(\I{C})$ that is spanned by the \emph{$\IFilt$-sheaves}, i.e.\ by those functors $\pi_A^\ast\I{C}^\op\to\Univ[\Over{\BB}{A}]$ (in arbitrary context $A\in\BB$) whose opposite is $\IFilt$-cocontinuous. We now obtain the following characterisation of sheaves on a $\BB$-locale:
\begin{proposition}
	\label{prop:characterisationSheavesLimits}
	Let $\I{L}$ be a $\BB$-locale. Then $\IShv_{\BB}(\I{L})\simeq \IShv_{\BB}^{\fin}(\I{L})\cap\IShv_{\BB}^{\IFilt}(\I{L})$ as full subcategories in $\IPSh(\I{L})$.
\end{proposition}
\begin{proof}
	By combining Remarks~\ref{rem:localityFinitarySheaves} and~\ref{rem:localityPrincipleSheaves} with \cite[Remark~5.2.4]{MWColimits}, we need to show that for every $A\in\BB$, a presheaf $F\colon\pi_A^\ast\I{L}^\op\to\Univ[\Over{\BB}{A}]$ is a sheaf if and only if
	\begin{enumerate}
		\item $F^\op\colon \pi_A^\ast\I{L}\to\Univ[\Over{\BB}{A}]^\op$ is $\pi_A^\ast\IFilt$-cocontinuous;
		\item $F(\varnothing_{\pi_A^\ast\I{L}})\simeq 1_{\Univ[\Over{\BB}{A}]}$;
		\item for every two objects $U,V\colon B\rightrightarrows\I{L}$ in arbitrary context $B\in\Over{\BB}{A}$, the commutative square
		\begin{equation*}
			\begin{tikzcd}
				F(U\vee V)\arrow[r]\arrow[d] & F(V)\arrow[d]\\
				F(U)\arrow[r] & F(U\wedge V)
			\end{tikzcd}
		\end{equation*}
		is a pullback.
	\end{enumerate}
	By replacing $\BB$ with $\Over{\BB}{A}$, we may assume that $A\simeq 1$. Now suppose first that $F$ is a sheaf. To show that~(1) is satisfied, we need to verify that for every diagram $d\colon \I{I}\to\pi_A^\ast\I{L}$ where $\I{I}$ is a filtered $\Over{\BB}{A}$-category and $A\in\BB$ is arbitrarily chosen, the natural map $(\pi_A^\ast F)(\colim d)\to\lim (\pi_A^\ast F)d$ is an equivalence. By replacing $\BB$ with $\Over{\BB}{A}$, we may again assume without loss of generality that $A\simeq 1$. As $\I{I}$ is filtered, the functor $\colim_{\I{I}}\colon\iFun(\I{I}, \IPSh(\I{L}))\to\IPSh(\I{L})$ preserves finite limits and therefore (by Remark~\ref{rem:subterminalObjectsDiagonal}) subterminal objects. Therefore, we deduce that $\colim h_{\I{L}}d$ is subterminal. Now by Proposition~\ref{prop:colimitsBLocales}, we may replace $\I{I}$ by $\I{I}^\core$ and can thus assume that $\I{I}$ is a $\BB$-groupoid. Therefore, the sheaf condition implies that we obtain an equivalence
	\begin{equation*}
		F(\colim d)\simeq \map{\IPSh(\I{L})}(\colim h_{\I{L}}d, F) \simeq\lim Fd.
	\end{equation*}
	This shows that $F$ is a $\IFilt$-sheaf. Condition~(2) follows from the observation that as $\varnothing_{\I{L}}$ is the colimit of the unique diagram $\varnothing\to\I{L}$, we obtain an equivalence $F(\varnothing_{\I{L}})\simeq \map{\IPSh(\I{L})}(\varnothing_{\IPSh(\I{L})},F)\simeq 1_{\Univ}$. Lastly, to show that condition~(3) is met, we may again replace $\BB$ with $\Over{\BB}{B}$ and can therefore assume that $B\simeq 1$. Now as $U\vee V$ is the coproduct of $U$ and $V$ in $\I{L}$, the claim follows from the fact that the pushout $h_{\I{L}}(U)\sqcup_{h_{\I{L}}(U\wedge V)}h_{\I{L}}(V)$ in $\IPSh(\I{L})$ computes the coproduct of $h_{\I{L}}(U)$ and $h_{\I{L}}(V)$ in $\iFun(\I{L}^{\op},\ISub_{\BB})$.
	
	Conversely, suppose that $F$ satisfies the three conditions. To show that $F$ is a sheaf, we need to verify that for every covering $d\colon\I{G}\to\pi_A^\ast\I{L}$ of an object $U\colon A\to\I{L}$ the functor $\map{\IPSh(\I{L})}(-,F)$ carries the induces covering sieve $S_{d}\into h(U)$ to an equivalence. By replacing $\BB$ with $\Over{\BB}{A}$, we may again assume that $A\simeq 1$. First, let us show the claim in the case where $\I{G}$ is finite, i.e.\ a locally constant sheaf of finite $\infty$-groupoids. Upon passing to a suitable cover, we can even assume that $\I{G}$ is (the constant $\BB$-category associated with) a finite $\infty$-groupoid. Since $\I{L}$ is a $\BB$-poset, we can even assume that $\I{G}$ is a finite set. By induction, it suffices to cover the cases $\I{G}=\varnothing$ and $\I{G}=1\sqcup 1$. By the above argumentation, these two cases follow immediately from conditions~(2) and~(3).
	
	For the general case, let $\IFin_{\BB}$ be the internal class of finite $\BB$-categories. Since $\IFin_{\BB}$ has the decomposition property (see \S~\ref{sec:decompositionProperty} and \S~\ref{sec:finiteBCategories}), we may find a filtered $\BB$-category $\I{I}$ and a diagram $k\colon \I{I}\to \IFin_{\BB}\into\ICat_{\BB}$ such that $\I{G}\simeq\colim k$.
	Note that since $\I{G}$ is a $\BB$-groupoid and since the groupoidification of a finite $\BB$-category is a finite $\BB$-groupoid, postcomposing $k$ with the groupoidification functor yields a diagram $k^\prime\colon \I{I}\to \Univ\cap\IFin_{\BB}\into\ICat_{\BB}$ that also has colimit $\I{G}$. Therefore, we deduce from Proposition~\ref{prop:decompositionColimitsPSh} and by making use of the subterminal truncation functor $(-)^{\Sub}\colon\IPSh(\I{L})\to\iFun(\I{L}^{\op},\ISub_{\BB})$
	that there is a diagram $d^\prime\colon \I{I}\to\iFun(\I{L}^{\op},\ISub_{\BB})$ such that  (a) we have $\colim d^\prime\simeq \colim h_{\I{L}}d$ and such that (b) for every object $i\colon A\to\I{I}$ in arbitrary context $A\in\BB$ there is a finite $\Over{\BB}{A}$-groupoid $\I{H}_{i}$ together with a diagram $d_{i}\colon \I{H}_{i}\to \pi_{A}^\ast\I{L}$ such that $d^\prime(i)\simeq \colim h_{\I{L}}d_{i}$. From~(a) we deduce that if $l\colon \iFun(\I{L}^{\op},\ISub_{\BB})\to\I{L}$ is the left adjoint of the Yoneda embedding, the unit of the adjunction $l\dashv h_{\I{L}}$ determines morphisms
	\begin{equation*}
		\colim h_{\I{L}} d\simeq \colim d^\prime\xrightarrow{\alpha} \colim h_{\I{L}}l d^\prime\xrightarrow{\beta} h_{\I{L}}(\colim l d^\prime)\simeq h_{\I{L}}(\colim d)
	\end{equation*}
	in $\iFun(\I{L}^{\op},\ISub_{\BB})$. As $\I{I}$ is filtered, the same argumentation as above implies that the colimit in the middle is already the colimit in $\IPSh(\I{L})$. Thus condition~(1) implies that $\map{\IPSh(\I{L})}(-,F)$ carries $\beta$ to an equivalence. To finish the proof, it is therefore enough to show that this functor also sends $\alpha$ to an equivalence. For this, we only need to show that for every object $i\colon A\to\I{I}$ in arbitrary context $A\in\BB$ the map $d^\prime(i)\to h_{\I{L}}ld^\prime(i)$ is sent to an equivalence. By~(b), we find that $d^\prime(i)$ is of the form $\colim h_{\I{L}} d_i$ for some diagram $d_i\colon \I{H}_i\to \pi_A^\ast\I{L}$ where $\I{H}_i$ is a finite $\Over{\BB}{A}$-groupoid. Since this case has already been shown above, the result follows.
\end{proof}

\begin{lemma}
	\label{lem:finitarySheafLocallyCoherent}
	Let $\I{L}$ be a locally coherent $\BB$-locale and let $F\colon\I{L}^\op\to\Univ$ be a $\IFilt$-sheaf on $\I{L}$. Then $F$ is a sheaf on $\I{L}$ if and only if $F\vert_{\I{L}^{\compact}}$ is a finitary sheaf on $\I{L}^{\compact}$.
\end{lemma}
\begin{proof}
	By Proposition~\ref{prop:characterisationSheavesLimits}, we need to show that $F$ is a finitary sheaf on $\I{L}$ if and only if $F\vert_{\I{L}^{\compact}}$ is a finitary sheaf on $\I{L}^{\compact}$.
	As $\I{L}$ is locally coherent and therefore $\I{L}^{\compact}$ is closed under binary products in $\I{L}$, it is clear that the condition is necessary. Moreover, as $\I{L}^{\compact}$ contains the initial object, it is clear that $F$ satisfies condition~(1) of the definition of a finitary sheaf if and only if $F\vert_{\I{L}^{\compact}}$ does.
	Therefore, we only need to show that if $F\vert_{\I{L}^{\compact}}$ is a finitary sheaf, then for every pair of objects $U,V\colon A\rightrightarrows\I{L}$, the map $F(U\vee V)\to F(U)\times_{F(U\wedge V)}F(V)$ is an equivalence.
	Using Remark~\ref{rem:localityFinitarySheaves} and Remark~\ref{rem:BCUComp}, we may replace $\BB$ with $\Over{\BB}{A}$ and can therefore assume that $A\simeq 1$.
	Note that it follows from the bifunctoriality of $-\wedge -$ that for a fixed $U$, both the map $U\wedge V\to V$ and the map $U\wedge V\to U$ are natural in $V$, i.e. define morphisms in $\iFun(\I{L},\I{L})$.
	Therefore, we obtain a cospan $\diag(U)\leftarrow U\wedge -\to \id_{\I{L}}$ in $\iFun(\I{L},\I{L})$ (where $\diag\colon \I{L}\to\iFun(\I{L},\I{L})$ is the diagonal map). By taking the colimit of this diagram, we end up with a commutative square
	\begin{equation*}
		\begin{tikzcd}
			U\wedge- \arrow[r]\arrow[d] & \id_{\I{L}}\arrow[d]\\
			\diag(U)\arrow[r] & U\vee -
		\end{tikzcd}
	\end{equation*}
	in $\iFun(\I{L},\I{L})$. Since colimits are universal in $\I{L}$, the functor $U\wedge -$ is cocontinuous. Furthermore, the functor $\diag(U)$ is $\IFilt$-cocontinuous: in fact, as it can be identified with $U\wedge \diag(1_{\I{L}})(-)$, it suffices to see that $\diag(1_{\I{L}})\simeq 1_{\iFun(\I{L},\I{L})}$ is $\IFilt$-cocontinuous.
	As in the proof of Lemma~\ref{lem:stabilityFiltUCocontinuousFunctors}, this is a consequence of the fact that filtered colimits in $\I{L}$ are left exact, which is easily shown using Lemma~\ref{lem:SubFiltAccessible} below and the fact that $\I{L}$ is a left exact localisation of $\iFun(\I{L}^\op,\ISub_{\BB})$, see Proposition~\ref{prop:CharacterisationsLocales}. Thus, as $\IFilt$-cocontinuous functors are clearly closed under pushouts in $\iFun(\I{L},\I{L})$, the above commutative diagram is a square of $\IFilt$-cocontinuous functors.
	By again using that filtered colimits in $\I{L}$ are left exact, this observation now implies that by postcomposition with (the opposite of) $F$, we end up with a
	morphism $F(U \vee -)\to F(U)\times_{F(U\wedge -)} F(-)$ of $\IFilt$-cocontinuous functors $\I{L}\to\Univ^\op$. Since $\I{L}\simeq\IInd(\I{L}^{\compact})$, the universal property of $\IInd(\I{L}^{\compact})$ thus implies that this morphism is an equivalence already when its restriction to $\I{L}^{\compact}$ is one. Together with our assumption on $F$, it follows that if $U$ is compact, then the map $F(U\vee V)\to F(U)\times_{F(U\wedge V)}F(V)$ is an equivalence for \emph{all} $V\colon 1\to\I{L}$. By symmetry and the fact that the context of $U$ and $V$ has been arbitrarily chosen, this now implies that the morphism $F(-\vee V)\to F(-)\times_{F(-\wedge V)}F(V)$ is an equivalence when restricted to $\I{L}^{\compact}$ and must therefore be an equivalence on all of $\I{L}$. Hence the claim follows.
\end{proof}

For later use we also record the following Lemma:

\begin{lemma}
	\label{lem:finitarySheavesCompactlyGenerated}
	Let $\I{P}$ be a poset with finite colimits and binary products. Then $\IShv_{\BB}^{\fin}(\I{P})$ is closed under $\IFilt$-colimits in $\IPSh(\I{P})$.
\end{lemma}
\begin{proof}
	We need to show that for every $A\in\BB$ and every diagram $d\colon \I{I}\to\pi_A^\ast \IShv_{\BB}^{\fin}(\I{P})$ where $\I{I}$ is a filtered $\Over{\BB}{A}$-category, the colimit of $d$ is contained in $\IShv_{\BB}^{\fin}(\I{P})$. Using Remark~\ref{rem:localityFinitarySheaves}, we may replace $\BB$ with $\Over{\BB}{A}$ and can therefore assume that $A\simeq 1$. We may compute the colimit of $d$ as the composition
	\begin{equation*}
		\I{P}^\op\xrightarrow{d^\prime}\iFun(\I{I},\Univ)\xrightarrow{\colim_{\I{I}}}\Univ
	\end{equation*}
	where $d^\prime$ is the transpose of $d\colon \I{I}\to\IPSh(\I{P})$. As $\I{I}$ is filtered, the functor on the right preserves finite limits. Moreover, the assumption that $d$ takes values in $\IShv_{\BB}^{\fin}(\I{P})$ and the fact that limits in functor $\BB$-categories can be computed objectwise imply that $d^\prime$ is a $\iFun(\I{I},\Univ)$-valued finitary sheaf on $\I{P}$. Hence the claim follows.
\end{proof}

The main result of this section is the following description of sheaves on a locally coherent locale:

\begin{proposition}
	\label{prop:sheavesarefinitarysheaves}
	Let $ \I{L} $ be a locally coherent locale.
	Then there is a canonical equivalence of $ \BB $-topoi
	\[
	\IShv_{\BB}(\I{L})\simeq\IShv_{\BB}^{\fin}(\I{L}^{\compact}).
	\]
\end{proposition}
\begin{proof}
	By Proposition~\ref{prop:characterisationSheavesLimits}, we have an identification $\IShv_{\BB}(\I{L})\simeq\IShv_{\BB}^{\fin}(\I{L})\cap\IShv_{\BB}^{\IFilt}(\I{L})$ of full subcategories of $\IPSh(\I{L})$.
	In particular, we obtain an inclusion $\IShv_{\BB}(\I{L})\into \IShv_{\BB}^{\IFilt}(\I{L})\simeq\IPSh(\I{L}^{\compact})$ (where we use that $\I{L}\simeq\IInd(\I{L}^{\compact})$ and the universal property of $\IInd(\I{L}^{\compact})$).
	Using Remark~\ref{rem:localityPrincipleSheaves} together with Lemma~\ref{lem:finitarySheafLocallyCoherent}, we now find that an object $A\to \IPSh(\I{L}^{\compact})$ is contained in $\IShv_{\BB}(\I{L})$ if and only if it transposes to a finitary sheaf on $\pi_A^\ast\I{L}^{\compact}$, proving the claim.
\end{proof}

\begin{corollary}
	\label{cor:Sheaves_on_loc_coherent}
	Let $ \I{L} $ be a $ \BB $-locale.
	\begin{enumerate}
		\item If $ \I{L} $ is locally coherent then $ \IShv_{\BB}(\I{L}) $ is compactly generated.
		\item If $ \I{L} $ is locally stably compact then $ \Shv_{\BB}(L) $ is compactly assembled.
	\end{enumerate}
\end{corollary}
\begin{proof}
	Note that (2) is an immediate consequence of the definitions and (1).
	To see (1) note that we have an inclusion
	\[
	\IShv_{\BB}(\I{L}) \simeq \IShv_{\BB}^{\fin}(\I{L}^{\compact}) \hookrightarrow \IPSh(\I{L}^{\compact})
	\]
	which by Lemma~\ref{lem:finitarySheavesCompactlyGenerated} preserves filtered colimits.
	Thus the claim follows from Corollary~\ref{cor:AccessibleLocalisation}.
\end{proof}

\subsection{The internal locale of a locally proper and separated map}
\label{sec:locallyProper}
The goal of this section is to show that for a proper and separated map of topological spaces $ f \colon Y \to X $, the $ \Shv(X) $-locale given by $U\mapsto \OO(f^{-1}(U))$ is a stably compact locale.
This is a direct consequence of \cite{Johnstone1981} but we decided to also provide a separate proof of Johnstone's result in the language of $\Shv(X)$-locales that we developed above.
With future applications in mind, we will prove a slightly more general statement about \emph{locally} proper and separated maps of topological spaces (which Johnstone also mentions in~\cite[\S~C4.1]{johnstone2002} but never explicitly spells out).

We begin by recalling the definition of a locally proper map from~\cite{Schnuerer2016}:
\begin{definition}
	\label{def:locallyProper}
	A continuous map $f\colon Y\to X$ of topological spaces is said to be \emph{locally proper} if for every $y\in Y$ and every open neighbourhood $V$ of $y$ there is a neighbourhood $K\subset V$ of $y$ and an open neighbourhood $U$ of $f(y)$ such that $f(K)\subset U$ and such that the induced map $K\to U$ is proper (i.e.\ universally closed).
\end{definition}

\begin{remark}
	\label{rem:locallyProperSeparatedLocalCondition}
	The property of a map $f\colon Y\to X$ to be locally proper and separated is local in the target: if $X=\bigcup_i U_i$ is an open covering, then $f$ is locally proper and separated if and only if each of the restrictions $f^{-1}(U_i)\to U_i$ has that property~\cite[Lemma~2.7]{Schnuerer2016}.
\end{remark}

\begin{remark}
	\label{rem:ProperSeparatedImpliesLocallyProper}
	Every proper and separated morphism is also locally proper~\cite[Proposition~2.12]{Schnuerer2016}. This is the relative version of the fact that compact Hausdorff spaces are locally compact as well.
\end{remark}

\begin{remark}
	\label{rem:localPropernessAlternativeDefinition}
	In the situation of Definition~\ref{def:locallyProper}, if $f$ is separated and locally proper, then for every $y\in Y$ and every open neighbourhood $V$ of $y$ there is an open neighbourhood $V^\prime\subset V$ and an open neighbourhood $U$ of $f(y)$ such that $f(V^\prime)\subset U$ and such that the closure of $V^\prime$ in $f^{-1}(U)$ is proper over $U$. In fact, $f$ being separated implies that its restriction $f^{-1}(U)\to U$ is separated as well. Therefore,~\cite[Lemma~9.12]{Schnuerer2016} implies that if $K\subset V$ is as in Definition~\ref{def:locallyProper}, then $K$ is closed in $f^{-1}(U)$. Hence the closure of the interior of $K$ (again in $f^{-1}(U)$) is a closed subset of $K$ and therefore also proper over $U$.
\end{remark}

To proceed, recall that if $f\colon Y\to X$ is a map of topological spaces, we obtain an algebraic morphism of locales $f^\ast\colon \OO(X)\to\OO(Y)$, where $\OO(X)$ and $\OO(Y)$ denote the locales of open subsets of $X$ and $Y$, respectively. By Proposition~\ref{prop:internalVsRelativeLocales}, $f^\ast$ gives rise to a $\Shv(X)$-locale $\I{O}_{X}(Y)$ that is explicitly given by the sheaf on $X$ that carries an open $U\in\OO(X)$ to the locale $\OO(f^{-1}(U))$ (see Remark~\ref{rem:BLocaleFromRelativeLocale}). Recall, furthermore, that we refer to a $\BB$-locale $\I{L}$ as (locally) stably compact if it arises as a retract in $\ILLoc_{\BB}$ of a (locally) coherent $\BB$-locale (see Definition~\ref{def:locallyStablyCompactBLocale}). Now the main goal of this section is to show:
\begin{proposition}
	\label{prop:continuousInternalLocaleInTopology}
	If $f\colon Y\to X$ is a locally proper and separated morphism of topological spaces, then $\I{O}_{X}(Y)$ is a locally stably compact $\Shv(X)$-locale. If $f$ is even proper, then $\I{O}_{X}(Y)$ is stably compact.
\end{proposition}
The proof of Proposition~\ref{prop:continuousInternalLocaleInTopology} requires a few preparations first. To begin with, we need to construct a candidate for a (locally) coherent $\Shv(X)$-locale of which $\I{O}_{X}(Y)$ is a retract. We will use the following general observation:

\begin{proposition}
	\label{prop:IndLExplicitly}
	Let $\I{L}$ be a $\BB$-locale and let $j\colon\I{P}\into\I{L}$ be a full subposet that is closed under binary products and finite colimits. Then
	\begin{enumerate}
		\item the left Kan extension $h_!(j)\colon\IInd(\I{P})\to\I{L}$ is cocontinuous;
		\item $\IInd(\I{P})$ is a locally coherent $\BB$-locale which is coherent if $\I{P}$ contains the final object of $\I{L}$;
		\item a right fibration over $\I{P}$ (in arbitrary context $A\in\BB$) is contained in the essential image of the inclusion $\IInd(\I{P})\into\IRFib_{\I{P}}$ if and only if it the inclusion of a sieve in $\pi_A^\ast\I{P}$ (i.e.\ a fully faithful right fibration) that is closed under finite colimits.
	\end{enumerate}
\end{proposition}
The proof of Proposition~\ref{prop:IndLExplicitly} requires the following two lemmas:
\begin{lemma}
	\label{lem:SubFiltAccessible}
	The inclusion $\ISub_{\BB}\into\Univ$ preserves filtered colimits
\end{lemma}
\begin{proof}
	Using Remark~\ref{rem:BCSubterminal} and Example~\ref{ex:subterminalPresheaves}, it suffices to show that for every filtered $\BB$-category $\I{I}$, the functor $\colim_{\I{I}}\colon\iFun(\I{I},\Univ)\to\Univ$ restricts to subterminal objects. By Remark~\ref{rem:subterminalObjectsDiagonal}, this is an immediate consequence of $\colim_{\I{I}}$ being left exact.
\end{proof}

\begin{lemma}
	\label{lem:filteredRightFibrationFiniteColimits}
	Let $\I{C}$ be a $\BB$-poset with finite colimits and let $p\colon\I{P}\into\I{C}$ be a sieve (i.e.\ a fully faithful right fibration). Then $\I{P}$ is filtered if and only if it is closed under finite colimits in $\I{C}$.
\end{lemma}
\begin{proof}
	It will be sufficient to show that whenever $d\colon \I{K}\to\I{P}$ is a finite diagram, then $\Under{\I{P}}{d}$ admits an initial object which is carried to the initial object in $\Under{\I{C}}{pd}$ along the induced functor $p_\ast\colon \Under{\I{P}}{d}\to\Under{\I{C}}{pd}$. Note that $p\colon \I{P}\to\I{C}$ being a sieve implies that $p_\ast$ is one as well. Now since $\I{C}$ has finite colimits, $\Under{\I{C}}{pd}$ admits an initial object $\colim (pd)$. Since $p_\ast$ is a right fibration, the inclusion $\Under{\I{P}}{d}\vert_{\colim (pd)}\into\Under{\I{P}}{d}$ of $p_\ast$ over $\colim (pd)$ is initial~\cite[Proposition~4.4.7]{MYoneda}. Since $\I{P}$ is assumed to be filtered, we furthermore have $(\Under{\I{P}}{d})^\gp\simeq 1$. Therefore, we must have $\Under{\I{P}}{d}\vert_{\colim (pd)}\simeq 1$ as this is already a subterminal $\BB$-groupoid (since $p$ is fully faithful, see Example~\ref{ex:subterminalPresheaves}). Hence $\Under{\I{P}}{d}$ admits an initial object which is preserved by $p_\ast$.
\end{proof}

\begin{proof}[{Proof of Proposition~\ref{prop:IndLExplicitly}}]
	The fact that $\I{P}$ has finite colimits implies that the $\BB$-category $\IInd(\I{P})$ is presentable and that the left Kan extension $h_!(j)\colon\IInd(\I{P})\to \I{L}$ is cocontinuous by Corollary~\ref{cor:UCocompleteIndUPresentable}, which shows~(1). 
	
	To show~(2), since $\IInd(\I{P})$ is by definition compactly generated and since we may always identify $\IInd(\I{P})^{\compact}\simeq\I{P}$ (as $\I{P}$ is a $\BB$-poset), we only need to verify that $\IInd(\I{P})$ is indeed a $\BB$-locale.
	To that end, note that $\IInd(\I{P})$ being presentable implies that the inclusion $\IInd(\I{P})\into\IPSh(\I{P})$ admits a left adjoint $l\colon \IPSh(\I{P})\to\IInd(\I{P})$ by \cite[Corollary~7.1.14]{MWColimits}.
	Moreover, Lemma~\ref{lem:SubFiltAccessible} implies that the inclusion $\iFun(\I{P}^{\op},\ISub_{\BB})\into\IPSh(\I{P})$ preserves filtered colimits. Therefore, $\IInd(\I{P})$ must be contained in $\iFun(\I{P}^{\op},\ISub_{\BB})$ and is therefore in particular a $\BB$-poset.
	Hence, we only need to check that $l$ preserves binary products (see Lemma~\ref{lem:BousfieldLocalisationUniversalityColimits}). This is equivalent to $\IInd(\I{P})$ being an exponential ideal in $\IPSh(\I{P})$, i.e.\ that for every object $F\colon A\to \IInd(\I{P})$ and every object $G\colon A\to \IPSh(\I{P})$ (in arbitrary context $A\in\BB$), the internal hom $\ihom_{\IPSh(\I{P})}(G,F)$ is contained in $\IInd(\I{P})$.
	By using \cite[Proposition~7.1.11]{MWColimits}, we can assume that $A\simeq 1$.
	Upon writing $G$ as a colimit of representables and using that the inclusion $\IInd(\I{P})\into\IPSh(\I{P})$ is continuous, we may assume without loss of generality that $G$ is itself representable by an object $U\colon 1\to \I{P}$.
	Thus, Yoneda's lemma and the fact that the Yoneda embedding is continuous imply that $\ihom_{\IPSh(\I{P})}(G,F)$ can be identified with the presheaf $F(U\times -)$.
	Note that by Proposition~\ref{prop:SheavesFlatness} a presheaf is contained in $\IInd(\I{P})$ if and only if it carries finite colimits in $\I{P}$ to limits. Thus, as $F$ by assumption has this property and since colimits are universal in $\I{L}$, the claim follows.
	
	Lastly, in light of Example~\ref{ex:subterminalPresheaves}, statement~(3) is an immediate consequence of Lemma~\ref{lem:filteredRightFibrationFiniteColimits}.
\end{proof}

In light of Proposition~\ref{prop:IndLExplicitly}, our task is now to find a full subposet of $\I{O}_{X}(Y)$ that is closed under binary products and finite colimits. To that end, note that the datum of an object in $\I{O}_{X}(Y)$ in context $U\subset X$ is precisely given by an open subset $V\subset f^{-1}(U)$. With that in mind, we may now define:
\begin{definition}
	\label{def:properClosure}
	Let $f\colon Y\to X$ be a locally proper and separated map of topological spaces. We say that an object $V\subset f^{-1}(U)$  has \emph{proper closure} if its closure $\overline{V}$ in $f^{-1}(U)$ is proper over $U$.
	We define the subposet $\I{O}^{\mathrm{pc}}_X(Y)\into\I{O}_{X}(Y)$ as the full subposet of $\I{O}_{X}(Y)$ that is spanned by these objects.
\end{definition}

\begin{remark}
	\label{rem:closureProperInvariance}
	In the situation of Definition~\ref{def:properClosure}, note that $f$ being separated implies that if $\overline{V}$ is proper over $U$, then $\overline{V}$ is also closed in $Y$ (see~\cite[Lemma~9.12]{Schnuerer2016}). Therefore, $\overline{V}$ is also the closure of $V$ in $Y$ in this case.
\end{remark}

A priori, the subposet $\I{O}^{\mathrm{pc}}_X(Y)$ is only \emph{spanned} by the objects with proper closure, so there could potentially be more objects. Our next result shows that this cannot happen:
\begin{lemma}
	\label{lem:properClosureSheaf}
	An object $V\subset f^{-1}(U)$ in $\I{O}_{X}(Y)$ is contained in $\I{O}^{\mathrm{pc}}_X(Y)$ if and only if it has proper closure.
\end{lemma}
\begin{proof}
	By definition, the condition is sufficient, so it suffices to prove that it is also necessary. This amounts to showing that the property of having proper closure is \emph{local} on the target: if $U=\bigcup_i U_i$ is a covering and if $\overline{V\cap f^{-1}(U_i)}\subset f^{-1}(U_i)$ is proper over $U_i$, then $\overline{V}\subset f^{-1}(U)$ is proper over $U$. Since properness is local on the target~\cite[\S~9.5]{Schnuerer2016}, this follows from the identity $\overline{V\cap f^{-1}(U_i)}= \overline{V}\cap f^{-1}(U_i)$.
\end{proof}

\begin{remark}
	\label{rem:BCproperClosure}
	Note that if $U\subset X$ is an arbitrary open subset, we may identify $\Over{\Shv(X)}{U}\simeq\Shv(U)$. In light of this identification, the $\Shv(U)$-locale $\pi_U^\ast\I{O}_{X}(Y)$ can be identified with $\I{O}_U(f^{-1}(U))$. Moreover, Lemma~\ref{lem:properClosureSheaf} implies that we obtain a canonical equivalence $\pi_U^\ast\I{O}^{\mathrm{pc}}_X(Y)\simeq \I{O}^{\mathrm{pc}}_U(f^{-1}(U))$ of full subposets in $
	\I{O}(f^{-1}(U)/U)$ (see also Remark~\ref{rem:locallyProperSeparatedLocalCondition}).
\end{remark}

Having an explicit description of the full subposet $\I{O}^{\mathrm{pc}}_X(Y)\into\I{O}_{X}(Y)$, we now proceed by showing that it satisfies the conditions of Proposition~\ref{prop:IndLExplicitly}:
\begin{lemma}
	\label{lem:properClosureFiniteLimitsColimits}
	$\I{O}^{\mathrm{pc}}_X(Y)$ is closed under binary products and finite colimits in $\I{O}_{X}(Y)$.
\end{lemma}
\begin{proof}
	Since the map $\varnothing\to X$ is always proper, Lemma~\ref{lem:properClosureSheaf} implies that it is enough to show that for every two objects $V\subset f^{-1}(U)$ and $V^\prime\subset f^{-1}(U)$ whose closure (in $f^{-1}(U)$) is proper over $U$, both $\overline{V\cup V^\prime}\to U$ and $\overline{V\cap V^\prime}\to U$ are proper. The first map is proper by~\cite[\S~9.7]{Schnuerer2016} and the fact that union and closure commute. The second map is proper as it can be decomposed into the composition $\overline{V\cap V^\prime}\into \overline{V}\to U$ where the first map is a closed embedding (hence proper) and the second map is proper by assumption.
\end{proof}

\begin{proposition}
	\label{prop:properClosureDenseInOpens}
	The $\Shv(X)$-category $\IInd[\Shv(X)](\I{O}^{\mathrm{pc}}_X(Y))$ is a locally coherent $\Shv(X)$-locale, and the left Kan extension $\IInd[\Shv(X)](\I{O}^{\mathrm{pc}}_X(Y))\to\I{O}_{X}(Y)$ of the inclusion is a Bousfield localisation.
\end{proposition}
\begin{proof}
	In light of Lemma~\ref{lem:properClosureFiniteLimitsColimits}, the first claim follows from Proposition~\ref{prop:IndLExplicitly}, so that it suffices to show the second one.
	We need to prove that the counit of the adjunction $\I{O}_{X}(Y)\leftrightarrows \IInd[\Shv(X)](\I{O}^{\mathrm{pc}}_X(Y))$ is an equivalence. Using Remark~\ref{rem:BCproperClosure}, it will be enough to check this on a global object $V\subset Y$.
	By \cite[Remark~6.3.6]{MWColimits}, this amounts to showing that $V$ is the colimit of the diagram $\Over{\I{O}^{\mathrm{pc}}_X(Y)}{V}\to \I{O}_{X}(Y)$. Using Proposition~\ref{prop:colimitsBLocales}, we only need to verify that $V\simeq \bigcup_{V^\prime\subset f^{-1}(U)\cap V} V^\prime$, where $U$ runs though all open subsets of $X$ and $V^\prime$ runs through all objects in $\I{O}^{\mathrm{pc}}_X(Y)(U)$ which are contained in $V$. This is an immediate consequence of the fact that $Y$ is locally proper and separated over $X$ (see Remark~\ref{rem:localPropernessAlternativeDefinition}).
\end{proof}

The following Lemma is a suitable relative analogue of the fact that in a locally compact Hausdorff space, every open covering of a compact subset has a finite refinement:

\begin{lemma}
	\label{lem:properSubsetsCoveringsLocallyFinite}
	Let $V\subset f^{-1}(U)$ be an object in $\I{O}^{\mathrm{pc}}_X(Y)(U)$, let $(V^\prime_j\subset f^{-1}(U_j))_{j\in J}$ be a family of objects in $\I{O}_{X}(Y)$ and suppose that $\overline{V}\subset \bigcup_{j\in J} V_j^\prime$. Then there is a covering $U=\bigcup_i U_i$ in $X$ such that for each $i$ there is a finite subset $J_i\subset J$ such that $U_i\subset U_{j}$ for all $j\in J_i$ and such that $\overline{V}\cap f^{-1}(U_i)\subset \bigcup_{j\in J_i} V^\prime_j$.
\end{lemma}
\begin{proof}
	In light of Remark~\ref{rem:BCproperClosure}, we may replace $Y/X$ by $f^{-1}(U)/U$ and each object $V_j^\prime\subset f^{-1}(U_j)$ by its intersection $V_j^\prime\cap f^{-1}(U)\subset f^{-1}(U_j\cap U)$ and can thus assume without loss of generality that $U=X$. Now since $\overline{V}$ is proper over $X$, its fibre $\overline{V}\vert_x$ over every $x\in X$ is compact (as being proper is stable under base change). Therefore, for each $x\in X$ we have a finite subset $J_x\subset J$ such that $\overline{V}\vert_x\subset \bigcup_{j\in J_x} V^\prime_j$. We can assume that $x\in U_j$ for all $j\in J_x$, since otherwise $V^\prime_j\vert_x$ would be empty. Now let $Z$ be the complement of $\bigcup_{j\in J_x} V^\prime_j$ in $Y$. Then $\overline{V}\cap Z$ is closed in $\overline{Y}$, hence $f(\overline{V}\cap Z)$ is closed in $X$ (as proper maps are always closed). By construction, a point $x^\prime\in X$ is contained in $f(\overline{V}\cap Z)$ precisely if $\overline{V}\vert_{x^\prime}$ is not contained in $\bigcup_{j\in J_x} V^\prime_j$. Therefore, if $U$ is the complement of $f(\overline{V}\cap Z)$ in $X$, then $U$ contains precisely those points $x^\prime\in X$ for which $\overline{V}\vert_{x^\prime}\subset \bigcup_{j\in J_x} V^\prime_j$. In other words, we have $\overline{V}\cap f^{-1}(U)\subset \bigcup_{j\in J_x} V^\prime_j$. Since $x\in U$, we may shrink $U$ if necessary so that it is contained in $\bigcap_{j\in J_x} U_x$. Now the claim follows.
\end{proof}

\begin{proof}[{Proof of Proposition~\ref{prop:continuousInternalLocaleInTopology}}]
	By Proposition~\ref{prop:properClosureDenseInOpens}, the left Kan extension $l\colon\IInd(\I{O}^{\mathrm{pc}}_X(Y))\to\I{O}_{X}(Y)$ is a Bousfield localisation. Therefore, we only need to show that $l$ admits a left adjoint $\lambda$ which preserves finite limits.
	
	We begin by showing that $l(X)$ has a left adjoint $\lambda_X$. On account of Proposition~\ref{prop:IndLExplicitly}, this amounts to showing that for every $V\subset Y$, there is sieve $\lambda_X(V)\colon\I{P}\into \I{O}^{\mathrm{pc}}_X(Y)$ which is closed under finite colimits such that for every other sieve $q\colon\I{Q}\into\I{O}^{\mathrm{pc}}_X(Y)$ with the same property and for which $V\subset\colim q$, we have $\I{P}\into \I{Q}$. We define $\I{P}$ to be the full subposet of $\I{O}^{\mathrm{pc}}_X(Y)$ which is spanned by those $V^\prime\subset f^{-1}(U)$ whose closure is contained in $V$. This property is clearly local in $X$, so that every object of $\I{P}$ in context $U\subset X$ will be of this form. Moreover, if $V^{\prime\prime}\subset V^\prime$ and $V^\prime$ is in $\I{P}(U)$, so is $V^{\prime\prime}$. Therefore, $\I{P}\into\I{O}^{\mathrm{pc}}_X(Y)$ is a sieve. Furthermore, $\I{P}$ is closed under finite colimits. Now let $V^\prime\subset f^{-1}(U)$ be an arbitrary object in $\I{P}$ in context $U\subset X$ and let $q\colon \I{Q}\into\I{O}^{\mathrm{pc}}_X(Y)$ be a sieve which is closed under finite colimits such that $V\subset\colim q$. We need to show that $V^\prime$ is contained in $\I{Q}$. By assumption, the closure $\overline{V^\prime}$ is contained in $\colim q$. Using Proposition~\ref{prop:colimitsBLocales}, we may identify
	\begin{equation*}
		\colim q \simeq \bigcup_{\substack{V^{\prime\prime}\in \I{Q}(U)\\U\subset X}} V^{\prime\prime}.
	\end{equation*}
	Therefore, Lemma~\ref{lem:properSubsetsCoveringsLocallyFinite} implies that there is a covering $U=\bigcup_{i\in I} U_i$ such that for each $i$ there are finitely many $V^{\prime\prime}_{i_1},\dots,V^{\prime\prime}_{i_n}\in \I{Q}(U_i)$ with the property that $V^\prime\cap f^{-1}(U_i)\subset\bigcup_{j=1}^n V^{\prime\prime}_{i_j}$. As $\I{Q}$ is closed under finite colimits, the right-hand side is contained in $\I{Q}(U_i)$. Consequently, $V^\prime$ is locally contained in $\I{Q}$ and must therefore also be globally contained in $\I{Q}$. 
	
	Now by carrying out the above argument with $f\vert_{f^{-1}(U)}$ in place of $f$, Remark~\ref{rem:BCproperClosure} implies that $l(U)$ admits a left adjoint $\lambda_U$ for every $U\subset X$. Furthermore, for every pair of opens $U\subset U^\prime\subset X$ and every $V^\prime\subset f^{-1}(U^\prime)$, it follows readily from the constructions that the restriction of $\lambda_{U^\prime}(V^\prime)$ to $U$ can be identified with $\lambda_U(V^\prime\cap f^{-1}(U))$. Therefore, we deduce from \cite[Corollary~3.2.11]{MWColimits} that $l$ admits a left adjoint, as desired. It is then clear from its explicit construction that this left adjoint preserves finite limits.
	
	Lastly, if $f$ is proper, then $\I{O}^{\mathrm{pc}}_X(Y)$ contains the final object of $\I{O}_{X}(Y)$, which immediately implies that $\I{O}_{X}(Y)$ is stably compact.
\end{proof}

\appendix

\chapter{Miscellaneous results on $\BB$-categories}
\numberwithin{equation}{section}
\section{Colimits in slice $\BB$-categories}

It is well-known that if $\CC$ is an $\infty$-category and $c\in\CC$ is an arbitrary object, the colimit of a diagram $d\colon \II\to \Over{\CC}{c}$ can be computed as the colimit of the underlying diagram $(\pi_c)_!d\colon \II\to\CC$. In this section we will establish the analogous statement for $\BB$-categories.

\begin{lemma}
	\label{lem:sliceFibrationInitialObject}
	Let $\I{C}$ be a $\BB$-category and let $f\colon c\to d$ be a map in $\I{C}$ in context $1\in\BB$ such that $c$ is an initial object in $\I{C}$. Then $f$ defines an initial object in $\Over{\I{C}}{d}$.
\end{lemma}
\begin{proof}
	Let $g\colon c^\prime\to d$ be an arbitrary map in $\I{C}$ in context $1\in\BB$. We have an evident commutative square
	\begin{equation*}
		\begin{tikzcd}
			\Over{(\Over{\I{C}}{d})}{g}\arrow[r, "\simeq"]\arrow[d, "(\pi_g)_!"] & \Over{\I{C}}{c^\prime}\arrow[d, "(\pi_{c^\prime})_!"] \\
			\Over{\I{C}}{d}\arrow[r, "(\pi_d)_!"] & \I{C}.
		\end{tikzcd}
	\end{equation*}
	in which the upper horizontal map is an equivalence as it is a right fibration that preserves final objects.
	Moreover, since $c$ is initial, the diagram
	\begin{equation*}
		\begin{tikzcd}
			1\arrow[r, "f"]\arrow[d, "\id"] & \Over{\I{C}}{d}\arrow[d, "(\pi_d)_!"]\\
			1\arrow[r, "c"] & \I{C}
		\end{tikzcd}
	\end{equation*}
	is a pullback. Consequently, we obtain an equivalence $\map{/d}(f,g)\simeq\map{\I{C}}(c,c^\prime)$. Since $c$ is initial, we conclude that $\map{/d}(f,g)\simeq 1$. By replacing $\BB$ with $\Over{\BB}{A}$, the same conclusion holds for every map $g\colon c^\prime\to \pi_A^\ast d$ in context $A$. Hence, we deduce from~\cite[Proposition~4.3.14]{MYoneda} that $f$ is initial.
\end{proof}

\begin{lemma}
	\label{lem:underOverCategoriesCommute}
	Let $\I{C}$ be a $\BB$-category and let $f\colon c\to d$ be a map in $\I{C}$ in context $1\in\BB$. Then there is an equivalence $\Over{(\Under{\I{C}}{c})}{f}\simeq\Under{(\Over{\I{C}}{d})}{f}$ that commutes with the projections to $\Over{\I{C}}{d}$ and $\Under{\I{C}}{c}$.
\end{lemma}
\begin{proof}
	Note that the projection $(\pi_c)_!\colon\Under{\I{C}}{c}\to\I{C}$ induces a left fibration $(\pi_c)_!\colon \Over{(\Under{\I{C}}{c})}{f}\to\Over{\I{C}}{d}$. By considering the commutative square
	\begin{equation*}
		\begin{tikzcd}
			c\arrow[r, "\id"]\arrow[d, "\id"] & c\arrow[d, "f"]\\
			c\arrow[r, "f"] & d
		\end{tikzcd}
	\end{equation*}
	as an object $\phi\colon 1\to \Over{(\Under{\I{C}}{c})}{f}$, we obtain a commutative square
	\begin{equation*}
		\begin{tikzcd}
			1\arrow[r, "\phi"]\arrow[d, "\id_f"] & \Over{(\Under{\I{C}}{c})}{f}\arrow[d, "(\pi_c)_!"]\\
			\Under{(\Over{\I{C}}{d})}{f}\arrow[r, "(\pi_c)_!"]\arrow[ur, dotted] & \Over{\I{C}}{d}.
		\end{tikzcd}
	\end{equation*}
	As the left vertical map is initial, the dotted filler exists, hence the proof is complete once we show that $\phi$ is initial too. By construction, the right fibration $(\pi_f)_!\colon \Over{(\Under{\I{C}}{c})}{f}\to \Under{\I{C}}{c}$ carries $\phi$ to an initial object. The desired result therefore follows from Lemma~\ref{lem:sliceFibrationInitialObject}.
\end{proof}

\begin{proposition}
	\label{prop:sliceFibrationColimits}
	Let $\I{I}$ and $\I{C}$ be $\BB$-categories and let $c\colon 1\to\I{C}$ be an object. Let $d\colon \I{I}\to\Over{\I{C}}{c}$ be a diagram and suppose that the diagram $(\pi_c)_! d\colon\I{I}\to\I{C}$ admits a colimit in $\I{C}$. Then $d$ admits a colimit in $\Over{\I{C}}{c}$, and $(\pi_c)_!$ preserves this colimit.
\end{proposition}
\begin{proof}
	On account of the equivalence $\IFun(\I{I},\Over{\I{C}}{c})\simeq\Over{\IFun(\I{I},\I{C})}{\diag(c)}$, the diagram $d\colon \I{K}\to\Over{\I{C}}{c}$ corresponds to an object $d^\prime=(\pi_c)_! d\to \diag(c)$ in $\Over{\IFun(\I{I},\I{C})}{\diag(c)}$, which can be equivalently regarded as a cocone $\overline{d^\prime}\colon 1\to \Under{\I{C}}{d^\prime}$. One therefore obtains a unique map 
	\begin{equation*}
		\begin{tikzcd}[column sep={3em,between origins}]
			& d^\prime \arrow[dl]\arrow[dr, "\overline{ d^\prime}"]&\\
			\diag(\colim d^\prime)\arrow[rr] && \diag(c)
		\end{tikzcd}
	\end{equation*}
	in $\Under{\I{C}}{d^\prime}$ (by the universal property of initial objects, see~\cite[Proposition~4.3.14]{MYoneda}) which can be regarded as an object in $\Over{(\Under{\I{C}}{d^\prime})}{\overline{ d^\prime}}$.
	Now Lemma~\ref{lem:underOverCategoriesCommute} gives rise to an equivalence
	\begin{equation*}
		\Over{(\Under{\IFun(\I{I},\I{C})}{d^\prime})}{\overline{ d^\prime}}\simeq \Under{\IFun(\I{I},\Over{\I{C}}{c})}{d}
	\end{equation*}
	over $\IFun(\I{I},\Over{\I{C}}{c})$ the pullback of which along the diagonal map determines an equivalence $\Over{(\Under{\I{C}}{d^\prime})}{\overline{ d^\prime}}\simeq \Under{(\Over{\I{C}}{c})}{d}$ that fits into a commutative diagram
	\begin{equation*}
		\begin{tikzcd}[column sep={4.5em,between origins}]
			\Over{(\Under{\I{C}}{d^\prime})}{\overline{ d^\prime}}\arrow[rr, "\simeq"]\arrow[dr, "(\pi_{\overline{d^\prime}})_!"'] && \Under{(\Over{\I{C}}{c})}{d}\arrow[dl, "(\pi_c)_!"]\\
			& \Under{\I{C}}{d}
		\end{tikzcd}
	\end{equation*}
	Consequently, the colimit cocone $d^\prime\to\colim d^\prime$ lifts along $(\pi_c)_!$ to a cocone under $d$. By Lemma~\ref{lem:sliceFibrationInitialObject}, this lift defines an initial object and therefore a colimit cocone, hence the claim follows.
\end{proof}

\begin{corollary}
	\label{cor:sliceCategoryUCocomplete}
	Let $\I{U}$ be an internal class of $\BB$-categories and let $\I{C}$ be a $\I{U}$-cocomplete $\BB$-category. For every object $c\colon 1\to\I{C}$, the slice $\BB$-category $\Over{\I{C}}{c}$ is $\I{U}$-cocomplete, and the forgetful functor $(\pi_c)_!$ is $\I{U}$-cocontinuous.\qed
\end{corollary}

\section{Adjunctions of slice $\BB$-categories}
\label{sec:sliceAdjunctions}

In this section we collect some basic facts about adjunctions between slice $\BB$-categories. More precisely, we show how an adjunction between $\BB$-categories induces an adjunction on slice $\BB$-categories, and we furthermore investigate the relation between the existence of pullbacks in a $\BB$-category and adjunctions between its slices. Everything discussed in this appendix is well-known for $\infty$-categories, and our proofs are straightforward adaptations of their $\infty$-categorical counterparts.
\begin{proposition}
	\label{prop:adjunctionsSliceCategories}
	Let $(l\dashv r)\colon\I{C}\leftrightarrows\I{D}$ be an adjunction between $\BB$-categories, and let $c\colon A\to\I{C}$ be an arbitrary object. Then the induced functor $\Over{r}{c}\colon\Over{\I{C}}{c}\to\Over{\I{D}}{r(c)}$ of $\Over{\BB}{A}$-categories admits a left adjoint $l_c$ that is explicitly given by the composition
	\begin{equation*}
		l_c\colon\Over{\I{D}}{r(c)}\xrightarrow{\Over{l}{r(c)}}\Over{\I{C}}{lr(c)}\xrightarrow{(\epsilon c)_!}\Over{\I{C}}{c},
	\end{equation*}
	where $\epsilon c\colon lr(c)\to c$ is the counit of the adjunction $(l\dashv r)$.
\end{proposition}
\begin{proof}
	Using~\cite[Remark~4.2.2]{MYoneda}, we may replace $\BB$ by $\Over{\BB}{A}$ and the adjunction $l\dashv r$ by its image along $\pi_A^\ast$ and can therefore assume without loss of generality that $A\simeq 1$. Let us fix an adjunction unit $\eta$ and an adjunction counit $\epsilon$. In light of the commutative diagram
	\begin{equation*}
		\begin{tikzcd}[row sep=small]
			&&& \Over{\I{C}}{c}\arrow[rr, "\Over{r}{c}"]\arrow[dddl, bend left, "(\pi_c)_!"] && \Over{\I{D}}{r(c)}\arrow[dddl, bend left, "(\pi_{r(c)})_!"]\\
			\Over{\I{D}}{r(c)}\arrow[rr, "\Over{l}{r(c)}"]\arrow[dd, "(\pi_{r(c)})_!"']\arrow[urrr, bend left, "l_c"'] && \Over{\I{C}}{lr(c)}\arrow[ur, "(\epsilon c)_!"]\arrow[rr, "\Over{r}{lr(c)}", near end, crossing over]\arrow[dd, "(\pi_{lr(c)})_!"'] && \Over{\I{D}}{rlr(c)}\arrow[ur, "(\epsilon r(c))_!"]\arrow[dd, "(\pi_{rlr(c)})_!"']\\
			&&& && \\
			\I{D}\arrow[rr, "l"] && \I{C}\arrow[rr, "r"] && \I{D}
		\end{tikzcd}
	\end{equation*}
	we obtain an equivalence $rl(\pi_{r(c)})_!\simeq(\pi_{r(c)})_!\Over{r}{c}l_c$, which in turn yields a commutative diagram
	\begin{equation*}
		\begin{tikzcd}[column sep=large]
			1\arrow[d, hookrightarrow, "d^0"]\arrow[r, "\id_{r(c)}"]\arrow[rr,bend left, "\epsilon r(c)"]&\Over{\I{D}}{r(c)}\arrow[d, "d^0"]\arrow[r, "\Over{r}{c}l_c"] & \Over{\I{D}}{r(c)}\arrow[d, "(\pi_{r(c)})_!"]\\
			\Delta^1\arrow[r, "\id\otimes \id_{r(c)}"] \arrow[rr, bend right, "r\eta c"]&\Delta^1\otimes\Over{\I{D}}{r(c)}\arrow[r, "\eta(\pi_{r(c)})_!"]\arrow[ur, dashed, "\eta_c"] & \I{D}.
		\end{tikzcd}
	\end{equation*}
	Note that as $d^0$ is a final functor, the lift $\eta_c$ exists. Moreover, since restricting $\eta_c$ along $\id\otimes \id_{r(c)}$ produces a lift of the outer square in the above diagram, the uniqueness of such lifts and the triangle identities for the adjunction $l\dashv r$ imply that $\eta_c$ carries the final object $\id_{r(c)}$ to the map in $\Over{\I{D}}{r(c)}$ that is encoded by the commutative triangle
	\begin{equation*}
		\begin{tikzcd}
			r(c)\arrow[dr, "\id_{r(c)}"']\arrow[r, "r\eta(c)"] & rlr(c)\arrow[d, "\epsilon r(c)"]\\
			& r(c).
		\end{tikzcd}
	\end{equation*}
	In particular, the functor $\Over{\I{D}}{r(c)}\xrightarrow{d^1}\Delta^1\otimes\Over{\I{D}}{r(c)}\xrightarrow{\eta_c}\Over{\I{D}}{r(c)}$ preserves the final object. Since this functor by construction commutes with the projection $(\pi_{r(c)})_!$, it must therefore be equivalent to the identity on $\Over{\I{D}}{r(c)}$, so that $\eta_c$ encodes a map $\id_{\Over{\I{D}}{r(c)}}\to\Over{r}{c}l_c$. 
	
	Dually, we also have an equivalence $lr (\pi_c)_!\simeq (\pi_c)_!l_c\Over{r}{c}$, so that the map $\epsilon(\pi_c)_!\colon\Delta^1\otimes\Over{\I{C}}{c}\to\I{C}$ encodes a morphism of functors $(\pi_c)_!l_c\Over{r}{c}\to(\pi_c)_!$. By an analogous argument as above, we can now construct a lift $\epsilon_c\colon \Delta^1\otimes\Over{\I{C}}{c}\to\Over{\I{C}}{c}$ of $\epsilon(\pi_c)_!$ along $(\pi_c)_!$ that encodes a morphism of functors $l_c\Over{r}{c}\to\id_{\Over{\I{C}}{c}}$. To complete the proof, it now suffices to show that the two compositions $\Over{r}{c}\xrightarrow{\eta_c\Over{r}{c}}\Over{r}{c}l_c\Over{r}{c}\xrightarrow{\Over{r}{c}\epsilon_c}\Over{r}{c}$ and $l_c\xrightarrow{l_c\eta_c}l_c\Over{r}{c}l_c\xrightarrow{\epsilon l_c} l_c$ are equivalences, cf.~\cite[Proposition~3.1.4]{MWColimits}. Using that $(\pi_{r(c)})_!$ and $(\pi_{c})_!$ are right fibrations and therefore in particular conservative, it suffices to show that these two morphisms become equivalences after postcomposition with $(\pi_{r(c)})_!$ and $(\pi_{c})_!$, respectively. Therefore, the claim follows from the triangle identities for $\eta$ and $\epsilon$, together with the observation that by construction we may identify $(\pi_{r(c)})_!\eta_c\simeq \eta (\pi_{r(c)})_!$ and $(\pi_c)_!\epsilon_c\simeq\epsilon (\pi_c)_!$.
\end{proof}

\begin{corollary}
	\label{cor:adjunctionsSliceCategoryPullback}
	Let $(l\dashv r)\colon\I{C}\to\I{D}$ be an adjunction between $\BB$-categories, and suppose that $\I{C}$ admits pullbacks. Then for any object $d\colon A\to\I{D}$ in $\I{D}$ in context $A\in\BB$, the induced functor $\Over{l}{d}\colon\Over{\I{D}}{d}\to\Over{\I{C}}{l(c)}$ admits a right adjoint $r_d$ that is explicitly given by the composition
	\begin{equation*}
		r_d\colon\Over{\I{C}}{l(d)}\xrightarrow{\Over{r}{l(d)}} \Over{\I{D}}{rl(d)}\xrightarrow{(\eta d)^\ast}\Over{\I{D}}{d}
	\end{equation*}
	in which $(\eta d)^\ast$ is the pullback functor along the adjunction unit $\eta d\colon d\to rl(d)$.
\end{corollary}
\begin{proof}
	To begin with, note that the discussion in \S~\ref{sec:descent} shows that since $\I{C}$ has pullbacks, the functor $(\eta d)^\ast$ indeed exists and is right adjoint to the projection $(\eta d)_!$. Now by Proposition~\ref{prop:adjunctionsSliceCategories}, the functor $\Over{r}{l(d)}\colon \Over{\I{C}}{l(d)}\to\Over{\I{D}}{rl(d)}$ admits a left adjoint $l_{r(d)}$ that is given by the composition $(\epsilon l(d))_! \Over{l}{rl(d)}$. Therefore, the functor $r_d$ is right adjoint to the composition
	\begin{equation*}
		\Over{D}{d}\xrightarrow{(\eta d)_!}\Over{\I{D}}{rl(d)}\xrightarrow{\Over{l}{rl(d)}}\Over{C}{lrl(d)}\xrightarrow{(\epsilon l(d))_!}\Over{\I{C}}{l(d)}.
	\end{equation*}
	It now suffices to notice that on account of the triangle identities, this functor is equivalent to $\Over{l}{d}$.
\end{proof}

\begin{lemma}
	\label{lem:mappingGroupoidsSlice}
	Let $\I{C}$ be a $\BB$-category and let $c\colon 1\to\I{C}$ be an arbitrary object. For any map $f\colon d\to c$, there is a pullback square
	\begin{equation*}
		\begin{tikzcd}
			\map{\Over{\I{C}}{c}}(-,f)\arrow[r, "(\pi_c)_!"]\arrow[d] & \map{\I{C}}((\pi_c)_!(-),d)\arrow[d, "f_\ast"]\\
			\diag(1_{\Univ})\arrow[r] & \map{\I{C}}((\pi_c)_!(-),c)
		\end{tikzcd}
	\end{equation*}
	in $\IPSh(\Over{\I{C}}{c})$.
\end{lemma}
\begin{proof}
	Since $(\pi_c)_!$ is a right fibration and therefore a cartesian fibration in which every map is a cartesian morphism (see the discussion in~\cite[\S~3.2]{MCocartesian}), the claim follows from the very definition of cartesian morphisms and the fact that $\id_c$ is final in $\Over{\I{C}}{c}$.
\end{proof}

\begin{proposition}
	\label{prop:characterisationCategoriesWithProducts}
	Let $\I{C}$ be a $\BB$-category with a final object $1_{\I{C}}\colon 1\to\I{C}$. Then the following are equivalent:
	\begin{enumerate}
		\item $\I{C}$ admits finite products;
		\item for every $c\colon A\to\I{C}$ in context $A\in\BB$, the projection $(\pi_c)_!\colon\Over{\I{C}}{c}\to\pi_A^\ast\I{C}$ admits a right adjoint $\pi_c^\ast$.
	\end{enumerate}
	Moreover, if either of these conditions is satisfied, the composition $(\pi_c)_!\pi_c^\ast$ is equivalent to the endofunctor $-\times c$ on $\pi_A^\ast\I{C}$.
\end{proposition}
\begin{proof}
	Let us first assume that $\I{C}$ admits finite products. By replacing $\BB$ with $\Over{\BB}{A}$ and $\I{C}$ with $\pi_A^\ast\I{C}$, we may assume that $A\simeq 1$. Suppose that $d\colon 1\to\I{C}$ is an arbitrary object. On account of the equivalence $1_{\I{C}}\times c\simeq c$, we have a commutative square
	\begin{equation*}
		\begin{tikzcd}
			1\arrow[r, "\id"]\arrow[d, "d_0"] &1\arrow[d, "1_{\I{C}}"']\arrow[r, "\id_{c}"] & \Over{\I{C}}{c}\arrow[d, "(\pi_c)_!"]\\
			\Delta^1\arrow[r, "\pi_d"] & \I{C}\arrow[r, "-\times c"] \arrow[ur, "\pi_c^\ast", dashed]& \I{C}
		\end{tikzcd}
	\end{equation*}
	(in which $\pi_d\colon d\to 1_{\I{C}}$ denotes the unique map), and since $1_{\I{C}}$ is final, the lift $\pi_c^\ast$ exists. Note that the projection $\pr_0\colon -\times c\to \id_{\I{C}}$ defines a map $\epsilon\colon (\pi_c)_!\pi_c^\ast\to\id_{\I{C}}$. Now the fact that $\pi_c^\ast$ by construction preserves final objects implies that this functor carries the unique map $\pi_d\colon d\to 1_{\I{C}}$ to the unique map $\pi_{\pi_c^\ast(d)}\colon \pi_c^\ast(d)\to \id_c$. As this implies that the image of $\pi_{\pi_c^\ast(d)}$ along $(\pi_c)_!$ recovers the projection $\pr_1\colon d\times c\to c$, the commutative square
	\begin{equation*}
		\begin{tikzcd}
			\map{\I{C}}(-,(\pi_c)_!\pi_c^\ast(d))\arrow[r, "\epsilon_\ast"]\arrow[d, "(\pi_c)_!(\pi_{\pi_c^\ast(d)})_\ast"] & \map{\I{C}}(-, d)\arrow[d]\\
			\map{\I{C}}(-, c)\arrow[r] & \diag(1_{\Univ})
		\end{tikzcd}
	\end{equation*}
	is a pullback in $\IPSh(\I{C})$. Together with Lemma~\ref{lem:mappingGroupoidsSlice}, this shows that the composition
	\begin{equation*}
		\map{\Over{\I{C}}{c}}(-,\pi_c^\ast(d))\xrightarrow{(\pi_c)_!}\map{\I{C}}((\pi_c)_!(-),(\pi_c)_!\pi_c^\ast(d))\xrightarrow{\epsilon_\ast}\map{\I{C}}((\pi_c)_!(-),d)
	\end{equation*}
	is an equivalence.  By replacing $\BB$ with $\Over{\BB}{A}$ and $\I{C}$ with $\pi_A^\ast\I{C}$, the same assertion is true for any object $d\colon A\to\I{C}$. Hence $\pi_c^\ast$ is right adjoint to $(\pi_c)_!$.
	
	Conversely, suppose that $(\pi_c)_!$ admits a right adjoint $\pi_c^\ast$ for all objects $c\colon A\to\I{C}$, and let us show that $\I{C}$ admits finite products. By induction, it suffices to consider binary products. Given any pair of objects $(c,d)\colon A\to\I{C}\times\I{C}$, we need to show that the presheaf $\map{\pi_A^\ast\I{C}}(\diag(-), (c,d))$ is representable. We may again assume that $A\simeq 1$. Let us show that the object $(\pi_c)_!\pi_c^\ast(d)$ represents this presheaf. Note that there is a pullback square
	\begin{equation*}
		\begin{tikzcd}
			\map{\I{C}\times\I{C}}(\diag(-),(c,d))\arrow[r]\arrow[d] & \map{\I{C}}(-,d)\arrow[d]\\
			\map{\I{C}}(-,d)\arrow[r] & \diag(1_{\Univ}).
		\end{tikzcd}
	\end{equation*}
	To complete the proof, it therefore suffices to show that the maps $\epsilon_\ast\colon \map{\I{C}}(-, (\pi_c)_!\pi_c^\ast(d))\to\map{\I{C}}(-, d)$ and $(\pi_c)_!(\pi_{\pi_c^\ast(d)})_\ast\colon \map{\I{C}}(-, (\pi_c)_!\pi_c^\ast(d))\to\map{\I{C}}(-, c)$ exhibit $\map{\I{C}}(-, (\pi_c)_!\pi_c^\ast(d))$ as a product of $\map{\I{C}}(-,c)$ and $\map{\I{C}}(-, d)$ in $\IPSh(\I{C})$. By the object-wise criterion for equivalences and~\cite[Corollary~3.1.9]{MWColimits}, this follows once we show that for every $z\colon 1\to\I{C}$ the commutative square
	\begin{equation*}
		\begin{tikzcd}
			\map{\I{C}}(z, (\pi_c)_!\pi_c^\ast(d))\arrow[r, "\epsilon_\ast"]\arrow[d, "(\pi_c)_!(\pi_{\pi_c^\ast(d)})_\ast"] & \map{\I{C}}(z,d)\arrow[d]\\
			\map{\I{C}}(z,c)\arrow[r] & 1
		\end{tikzcd}
	\end{equation*}
	is a pullback square in $\BB$. By descent in $\BB$ and Lemma~\ref{lem:mappingGroupoidsSlice}, this follows once we show that for any map $f\colon \pi_A^\ast(z)\to \pi_A^\ast(c)$ the composition
	\begin{equation*}
		\map{\Over{\I{C}}{c}}(f,\pi_c^\ast(d))\xrightarrow{(\pi_c)_!}\map{\I{C}}(\pi_A^\ast(z), \pi_A^\ast (\pi_c)_!\pi_c^\ast(d))\xrightarrow{\pi_A^\ast(\epsilon)_\ast} \map{\I{C}}(\pi_A^\ast(z), \pi_A^\ast(d))
	\end{equation*}
	is an equivalence. Since this is just the adjunction property of $(\pi_c)_!\dashv \pi_c^\ast$, the claim follows.
\end{proof}
\begin{remark}
	\label{rem:unitCounitSliceAdjunctions}
	In the situation of Proposition~\ref{prop:characterisationCategoriesWithProducts}, the proof shows that in light of the equivalence $(\pi_c)_!\pi_c^\ast\simeq -\times c$, the counit of the adjunction $(\pi_c)_!\dashv\pi_c^\ast$ can be identified with $\pr_0\colon -\times c\to \id_{\pi_A^\ast\I{C}}$. Similarly, if $d\to c$ is an arbitrary map in context $A\in\BB$, the unit $d\to \pi_c^\ast(\pi_c)_!$ is characterised by the condition that the composition $(\pi_c)_!d\to(\pi_c)_!\pi_c^\ast(\pi_c)_! d\simeq ((\pi_c)_! d)\times c\to (\pi_c)_! d$ is equivalent to the identity. It is therefore determined by the map $(\pi_c)_! d\to ((\pi_c)_! d)\times c$ that is given by the identity on the first factor and the structure map $d\to c$ on the second factor.
\end{remark}

\begin{corollary}
	\label{cor:characterisationCategoriesWithPullbacks}
	For any $\BB$-category $\I{C}$, the following are equivalent:
	\begin{enumerate}
		\item $\I{C}$ admits pullbacks;
		\item for every map $f\colon c\to d$ in $\I{C}$ in context $A\in\BB$, the projection $f_!\colon\Over{\I{C}}{c}\to\Over{\I{C}}{d}$ admits a right adjoint $f^\ast$.
	\end{enumerate}
	Moreover, if either of these conditions are satisfied, then the composition $f_! f^\ast$ can be identified with the pullback functor $-\times_d c$.
\end{corollary}
\begin{proof}
	In light of Proposition~\ref{prop:characterisationCategoriesWithProducts}, it will be enough to show that $\I{C}$ admits pullbacks if and only if for every object $c\colon A\to \I{C}$ the $\Over{\BB}{A}$-category $\Over{\I{C}}{c}$ admits binary products. Using Example~\ref{ex:externalLimitsColimits}, this is easily reduced to the corresponding statement for $\infty$-categories, which appears as~\cite[Theorem~6.6.9]{cisinski2019a}.
\end{proof}

\section{Decomposition of colimits}
In~\cite[\S~4.7]{MWColimits} we showed that whenever $\I{C}$ is a $\BB$-category that admits colimits indexed by $\kappa$-small constant $\BB$-categories and $\KK\to \Cat(\BB),~k\mapsto\I{J}_k$ is a diagram that is indexed by a $\kappa$-small constant $\BB$-category $\KK$, then $\I{C}$ admits colimits indexed by $\I{J}=\colim_k \I{J}_k$ as soon as it admits $\I{J}_k$-indexed colimits for all $k\in \KK$. In this section, our goal is to generalise this result by allowing $\KK$ to be an arbitrary $\BB$-category instead of merely a constant one. More precisely, we will show:
\begin{proposition}
	\label{prop:decompositionExistenceColimits}
	Let $\I{U}$ be an internal class, let $d\colon\I{I}\to\I{U}$ be a diagram such that $\I{I}\in\I{U}(1)$, and let $\I{K}=\colim d$. Then every $\I{U}$-cocomplete $\BB$-category admits $\I{K}$-indexed colimits, and every $\I{U}$-cocontinuous functor between $\I{U}$-cocomplete $\BB$-categories preserves $\I{K}$-indexed colimits.
\end{proposition}
Our strategy for the proof of Proposition~\ref{prop:decompositionExistenceColimits} is to take the colimit of a $\I{K}$-indexed diagram in the free cocompletion of $\I{C}$ (i.e.\ in $\IPSh(\I{C}))$ and to show that this colimit can be reflected back into $\I{C}$. We therefore need to study such $\I{K}$-indexed colimits in $\IPSh(\I{C})$ first.

\begin{lemma}
	\label{lem:RFibReflectiveSubcategory}
	For every $\BB$-category $\I{C}$, the large $\BB$-category $\IRFib_{\I{C}}$ is a reflective subcategory of $\Over{(\ICat_{\BB})}{\I{C}}$.
\end{lemma}
\begin{proof}
	To begin with, we note that the sheaf associated with $\Over{(\ICat_{\BB})}{\I{C}}$ is given by $\Over{\Cat(\BB)}{\I{C}\times -}$. In fact, the latter defines a $\PSh_{\SSS}(\BB)$-category, and there is a right fibration of $\PSh_{\SSS}(\BB)$-categories $p\colon \Over{\Cat(\BB)}{\I{C}\times -}\to\Over{\Cat(\BB)}{-}$ that is section-wise given by postcomposition with the projection onto the second factor. By~\cite[Proposition~3.3.5]{MYoneda}, the codomain can be identified with (the underlying $\PSh_{\SSS}(\BB)$-category of) the large $\BB$-category $\ICat_{\BB}$. Since $\Over{\Cat(\BB)}{\I{C}\times -}$ has a final object (in the $\PSh_{\SSS}(\BB)$-categorical sense, which is easily deduced from~\cite[Examples~4.1.11 and~4.1.14]{MWColimits}) that is carried to $\I{C}$ along the right fibration $p$, we thus obtain an equivalence $\Over{(\ICat_{\BB})}{\I{C}}\simeq\Over{\Cat(\BB)}{\I{C}\times -}$ of $\PSh_{\SSS}(\BB)$-categories. Since the domain is a (large) $\BB$-category, so is the codomain, and this equivalence defines an identification of (large) $\BB$-category. Now using this identification, we find that the inclusion $\RFib(\I{C}\times -)\into\Over{\Cat(\BB)}{\I{C}\times -}$ determines a fully faithful functor $i\colon \IRFib_{\I{C}}\into\Over{(\ICat_{\BB})}{\I{C}}$ such that $i(A)$ admits a left adjoint $L_A$ for every $A\in\BB$. Moreover, if $s\colon B\to A$ is an arbitrary map in $\BB$, the fact that $s$ is smooth implies that the natural map $L_Bs^\ast\to s^\ast L_A$ is an equivalence (see the discussion in~\cite[\S~4.4]{MYoneda}), hence the claim follows.
\end{proof}

\begin{proposition}
	\label{prop:decompositionColimitsPSh}
	Let $\I{U}$ be an internal class of $\BB$-categories and let $\I{C}$ be a small $\BB$-category. Then, for any diagram $d\colon \I{I}\to\I{U}$ with colimit $\I{K}$ in $\ICat_{\BB}$ and any diagram $p\colon \I{K}\to\I{C}$ with colimit $F$ in $\IPSh(\I{C})$, there is a diagram $d^\prime\colon \I{I}\to \ISml^{\I{U}}(\I{C})$ such that $F\simeq \colim d^\prime$.
\end{proposition}
\begin{proof}
	The cocone $d\to \diag(\I{C})$ implies that we may view $d$ as a diagram $d\colon\I{I}\to\Over{\I{U}}{\I{C}}\into\Over{(\ICat_{\BB})}{\I{C}}$. By Proposition~\ref{prop:sliceFibrationColimits}, the colimit of this diagram is $p\colon\I{K}\to\I{C}$.
	Since $F\simeq\colim p$, there is a final functor $\I{K}\to\Over{\I{C}}{F}$ over $\I{C}$, hence Lemma~\ref{lem:RFibReflectiveSubcategory} implies that the localisation functor $L\colon\Over{(\ICat_{\BB})}{\I{C}}\to\IRFib_{\I{C}}$ carries $p\colon \I{K}\to\I{C}$ to the right fibration $\Over{\I{C}}{F}\to\I{C}$. In other words, the presheaf $F$ arises as the colimit of the diagram $d^\prime=Ld\colon \I{I}\to \Over{\I{U}}{\I{C}}\to\IRFib_{\I{C}}\simeq\IPSh(\I{C})$. It now suffices to observe that by construction of $L$, this functor takes values in $\ISml^{\I{U}}(\I{C})$.
\end{proof}

\begin{proof}[{Proof of Proposition~\ref{prop:decompositionExistenceColimits}}]
	Let $f\colon\I{C}\to\I{D}$ be a $\I{U}$-cocontinuous functor between $\I{U}$-cocomplete $\BB$-categories. 
	Consider the commutative diagram
	\begin{equation*}
		\begin{tikzcd}
			\I{C}\arrow[r, "f"]\arrow[d, hookrightarrow, "h_{\I{C}}"] & \I{D}\arrow[d, hookrightarrow, "h_{\I{D}}"] \\
			\IPSh^{\I{U}}(\I{C})\arrow[r, "\hat f"] & \IPSh^{\I{U}}(\I{D})
		\end{tikzcd}
	\end{equation*}
	that arises from applying the universal property of $\IPSh^{\I{U}}(\I{C})$ to the composition $\I{C}\to\I{D}\into\IPSh^{\I{U}}(\I{D})$. As $\I{C}$ and $\I{D}$ are $\I{U}$-cocomplete, the vertical inclusions admit left adjoints $L_{\I{C}}$ and $L_{\I{D}}$~\cite[Corollary~7.1.14]{MWColimits}. Now if $p\colon \I{K}\to\I{C}$ is a diagram, Proposition~\ref{prop:decompositionColimitsPSh} implies that there is a diagram $p^\prime\colon \I{I}\to \IPSh^{\I{U}}(\I{C})$ such that $\colim p^\prime$ is equivalent to the colimit of $h_{\I{C}}p$. In particular, the colimit of $h_{\I{C}}p$ is contained in $\IPSh^{\I{U}}(\I{C})$. Consequently, $\colim L_{\I{C}}p^\prime$ defines a colimit of $p$~\cite[Proposition~5.2.5]{MWColimits}. By replacing $\I{C}$ with $\I{D}$, this argument also shows that every diagram $\I{K}\to\I{D}$ admits a colimit in $\I{D}$. Moreover, as $f$ and $\hat f$ are $\I{U}$-cocontinuous, the universal property of $\IPSh^{\I{U}}(\I{C})$ implies that the canonical map $L_{\I{D}}\hat f\to f L_{\I{C}}$ is an equivalence. Consequently, as $L_{\I{D}}\hat f$ preserves the colimit of $h_{\I{C}}p$, so does $fL_{\I{C}}$. As the colimit cocone of $p$ is the image of the colimit cocone of $h_{\I{D}}p$ along $L_{\I{C}}$, we conclude that $f$ preserves the colimit of $p$. Now by replacing $\BB$ with $\Over{\BB}{A}$ and repeating the above argumentation, one concludes that both $\I{C}$ and $\I{D}$ admit $\I{K}$-indexed colimits and that $f$ commutes with such colimits.
\end{proof}

\section{Cocompletion with relations}
\label{sec:cocompletionRelations}
Let $\I{U}$ be an internal class, and let $\I{C}$ be a $\BB$-category. In~\cite{MWColimits}, we constructed the \emph{free $\I{U}$-cocompletion} $\IPSh^{\I{U}}(\I{C})$ of $\I{C}$, i.e.\ the universal $\I{U}$-cocomplete $\BB$-category that is equipped with a functor $\I{C}\to\IPSh^{\I{U}}(\I{C})$. The goal of this section is to generalise this result by imposing that a chosen collection of cocones in $\I{C}$ (that are indexed by objects of $\I{U}$) become colimit cocones in the free $\I{U}$-cocompletion. Our proof of this result is a straightforward adaptation of the discussion in~\cite[\S~5.3.6]{htt}.

Let us fix a small collection $R=(\overline{d}_i\colon \I{K}_i^\triangleright\to \pi_{A_i}^\ast \I{C})_{i\in I}$ of cocones with $A_i\in\BB$ and $\I{K}_i\in \I{U}(A_i)$ for all $i\in I$. Let $\I{S}_R\into\IPSh(\I{C})$ be the (non-full) subcategory that is spanned by the canonical maps $(f_i\colon \colim h_{\I{C}}d_i\to h_{\I{C}}\overline{d}_i(\infty))_{i\in I}$ in $\IPSh(\I{C})$ (with each $f_i$ being in context $A_i$ for $i\in I$), where $d_i$ denotes the restriction of $\overline{d}_i$ along the inclusion $\I{K}_i\into\I{K}_i^\triangleright$ and where $\infty\colon A_i\to \I{K}_i^{\triangleright}$ denotes the cone point.

Let us set $\I{D}=\ILoc_{\I{S}_R}(\IPSh(\I{C}))$. By Corollary~\ref{cor:presentableCategoryBousfieldLocalisation}, the inclusion $i\colon\I{D}\into\IPSh(\I{C})$ admits a left adjoint $L\colon \IPSh(\I{C})\to\I{D}$. In particular, $\I{D}$ is cocomplete~\cite[Proposition~5.2.6]{MWColimits}. We define the (large) $\BB$-category $\IPSh(\I{C})^{(\I{U},R)}(\I{C})$ as the smallest full subcategory of $\I{D}$ that contains the essential image of $Lh_{\I{C}}\colon \I{C}\to\I{D}$ and that is closed under $\I{U}$-colimits in $\I{D}$, and we let $j_{\I{C}}\colon\I{C}\to\IPSh^{(\I{U},R)}(\I{C})$ be the map that is obtained by composing $Lh_{\I{C}}$ with the inclusion. 
\begin{remark}
	\label{rem:BCCocompletionRelations}
	Given any object $A\in\BB$, we denote by $\pi_A^\ast R$ the set of cocones $(\pi_{A}^\ast (\overline{d}_i))_{i\in I}$. We then obtain an equivalence $\pi_A^\ast\I{S}_R\simeq \I{S}_{\pi_A^\ast R}$ of subcategories in $\IPSh[\Over{\BB}{A}](\pi_A^\ast\I{C})$. Hence Remark~\ref{rem:BCLocalObjects} and the same argument as in the proof of~\cite[Proposition~7.1.11]{MWColimits} shows that one obtains a canonical equivalence of large $\Over{\BB}{A}$-categories
	\begin{equation*}
		\pi_A^\ast\IPSh^{(\I{U}, R)}(\I{C})\simeq\IPSh[\Over{\BB}{A}]^{(\pi_A^\ast\I{U}, \pi_A^\ast R)}(\pi_A^\ast\I{C})
	\end{equation*} 
	with respect to which $\pi_A^\ast j_{\I{C}}$ corresponds to the map $j_{\pi_A^\ast\I{C}}$.
\end{remark}

For any $\I{U}$-cocomplete large $\BB$-category $\I{E}$, we will denote by $\IFun(\I{C},\I{E})_R$ the full subcategory of $\IFun(\I{C},\I{E})$ that arises as the pullback
\begin{equation*}
	\begin{tikzcd}
		\IFun(\I{C},\I{E})_R\arrow[r, hookrightarrow]\arrow[d, hookrightarrow] & \IFun(\I{C},\I{E})\arrow[d, "(h_{\I{C}})_!"]\\
		\IFun(\IPSh(\I{C}),\I{E})_{\I{S}_R}\arrow[r, hookrightarrow] & \IFun(\IPSh(\I{C}),\I{E}).
	\end{tikzcd}
\end{equation*}

We now obtain:
\begin{proposition}
	\label{prop:CocompletionRelations}
	For every $i\in I$ the cocone $(j_{\I{C}})_\ast(\overline{d}_i)$ is a colimit cocone in $\IPSh^{(\I{U},R)}(\I{C})$, and for every $\I{U}$-cocomplete large $\BB$-category $\I{E}$, precomposition with $j_{\I{C}}$ induces an equivalence
	\begin{equation*}
		j_{\I{C}}^\ast\colon \IFun^{\cocont{\I{U}}}(\IPSh^{(\I{U},R)}(\I{C}),\I{E})\simeq \IFun(\I{C},\I{E})_R.
	\end{equation*}
\end{proposition}
\begin{proof}
	Note that by construction of $\I{D}$, the map $j_{\I{C}}$ carries each of the cocones $\overline{d}_i$ to a colimit cocone in $\IPSh^{(\I{U},R)}(\I{C})$, hence the first claim follows immediately. The proof of the second claim employs a similar strategy as in the proof of~\cite[Theorem~7.1.13]{MWColimits}. First, if $\I{E}$ is an arbitrary $\I{U}$-cocomplete $\BB$-category, note that the Yoneda embedding induces a $\I{U}$-cocontinuous functor $\I{E}\into \I{E}^\prime=\IFun(\I{E},\Univ[\BB])^{\op}$ into a cocomplete $\BB$-category. By Corollary~\ref{cor:UniversalPropertyLocalObjects} and the universal property of presheaf $\BB$-categories, we now obtain an equivalence
	\begin{equation*}
		(Lh_{\I{C}})_!\colon \IFun(\I{C},\I{E}^\prime)_R\simeq \IFun^\cc(\IPSh(\I{C}),\I{E}^\prime)_{\I{S}_R}\simeq  \IFun^\cc(\I{D},\I{E}^\prime)
	\end{equation*}
	As the inclusion $\IPSh^{(\I{U},R)}(\I{C})\into\I{D}$ is by construction $\I{U}$-cocontinuous, we therefore obtain an induced inclusion
	\begin{equation*}
		(j_{\I{C}})_!\colon \IFun(\I{C},\I{E}^\prime)_R\into\IFun^{\cocont{\I{U}}}(\IPSh^{(\I{U},R)}(\I{C}),\I{E}^\prime).
	\end{equation*}
	Now if $f\colon \IPSh^{(\I{U},R)}(\I{C})\to\I{E}^\prime$ is a $\I{U}$-cocontinuous functor, precisely the same argument as the one employed in the proof of~\cite[Theorem~7.1.13]{MWColimits} shows that the adjunction counit $\epsilon\colon (j_{\I{C}})_! j_{\I{C}}^\ast f\to f$ is an equivalence and that $f$ is therefore contained in the essential image of $(j_{\I{C}})_!$. Together with Remark~\ref{rem:BCCocompletionRelations}, this shows that $(j_{\I{C}})_!$ is an equivalence. Finally, the same argumentation as in the proof of~\cite[Theorem~7.1.13]{MWColimits} also shows that this equivalence restricts to the desired equivalence $\IFun(\I{C},\I{E})_R\into\IFun^{\cocont{\I{U}}}(\IPSh^{(\I{U},R)}(\I{C}),\I{E})$.
\end{proof}

\begin{remark}
	\label{rem:CocompletionRelationsSmall}
	In the situation of Proposition~\ref{prop:CocompletionRelations}, if $\I{U}$ is assumed to be small (i.e.\ a \emph{doctrine} in the terminology of \S~\ref{sec:regularClasses}) implies that $\IPSh^{(\I{U},R)}(\I{C})$ is small as well. In fact, as $\I{D}$ is locally small, the essential image of $Lh_{\I{C}}\colon \I{C}\to\IPSh(\I{C})\to\I{D}$ is small~\cite[Lemma~4.6.5]{MYoneda}. Hence we can make use of the same argument as in~\cite[Remark~7.1.7]{MWColimits} to deduce that $\IPSh^{(\I{U},R)}(\I{C})$ must also be small.
\end{remark}

\begin{construction}
	\label{constr:BoxtimesOfRi}
	For the remainder of this section, let us fix $\BB$-categories $\I{C}^1,\dots,\I{C}^n$ and a doctrine $\I{U}$. Let $\GG\subset\BB$ be a small subcategory of generators, and let us set
	\begin{equation*}
		R_k=\bigsqcup_{G\in\GG}\left\{\overline{d}\colon\I{K}^\triangleright\to \pi_G^\ast\I{C}^k~\vert~\I{K}\in \I{U}(G)\right\}
	\end{equation*}
	as well as
	\begin{equation*}
		S_k=\bigsqcup_{G\in\GG}\left\{\overline{d}\colon\I{K}^\triangleright\to \pi_G^\ast\IPSh^{(\I{U},R_k)}(\I{C}^k)~\vert~\I{K}\in \I{U}(G),~\overline{d}~\text{is a colimit cocone}\right\}.
	\end{equation*}
	Furthermore, let $\square_{k=1}^n R_k$ be the set of all diagrams of the form
	\begin{equation*}
		(c_1,\dots,c_{l-1},\id,c_{l+1},\dots,c_n)\overline{d}\colon\I{K}^\triangleright\to \pi_{G}^\ast\I{C}^l\to \prod_{k=1}^n\pi_{G}^\ast\I{C}^k
	\end{equation*}
	where $\overline{d}$ is an element of $R_l$ and $c_k\colon G\to\I{C}^k$ is an arbitrary object for each $k\neq l$. Let $\square_{k=1}^n S_k$ be defined analogously. Then for any internal class $\I{V}$ that contains $\I{U}$, the composition
	\begin{equation*}
		\I{C}^1\times\cdots\times\I{C}^n\to\IPSh^{(\I{U},R_1)}(\I{C}^1)\times\cdots\times\IPSh^{(\I{U},R_n)}(\I{C}^n)\to \IPSh^{(\I{V}, \square_{k=1}^n S_k)}(\IPSh^{(\I{U},R_1)}(\I{C}^1)\times\cdots\times\IPSh^{(\I{U},R_n)}(\I{C}^n))
	\end{equation*}
	carries each cocone in $\square_{k=1}^n R_k$ to a colimit cocone, hence Proposition~\ref{prop:CocompletionRelations} determines a functor
	\begin{equation*}
		\phi\colon\IPSh^{(\I{V}, \square_{k=1}^n R_k)}( \I{C}^1\times\cdots\times\I{C}^n)\to\IPSh^{(\I{V},\square_{k=1}^n S_k)}(\IPSh^{(\I{U},R_1)}(\I{C}^1)\times\cdots\times \IPSh^{(\I{U},R_n)}(\I{C}^n)).
	\end{equation*}
\end{construction}

\begin{proposition}
	\label{prop:cocompletionRelationsTransitivity}
	The canonical map
	\begin{equation*}
		\phi\colon\IPSh^{(\I{V}, \square_{k=1}^n R_k)}( \I{C}^1\times\cdots\times\I{C}^n)\to\IPSh^{(\I{V},\square_{k=1}^n S_k)}(\IPSh^{(\I{U},R_1)}(\I{C}^1)\times\cdots\times \IPSh^{(\I{U},R_n)}(\I{C}^n))
	\end{equation*}
	is an equivalence.
\end{proposition}
\begin{proof}
	Note that in light of Remark~\ref{rem:CocompletionRelationsSmall}, the map $\phi$ is a well-defined morphism in the $\BB$-category $\ICat_{\BB}^{\cocont{\I{V}}}$ of $\I{V}$-cocomplete $\BB$-categories and $\I{V}$-cocontinuous functors. By combining Yoneda's lemma with Remark~\ref{rem:BCCocompletionRelations} and~\cite[Remark~5.3.4]{MWColimits}, the result thus follows once we verify that for every $\I{V}$-cocomplete $\BB$-category $\I{E}$ the restriction functor
	\begin{equation*}
		\phi^\ast\colon \IFun^{\cocont{\I{V}}}(\IPSh^{(\I{V},\square_{k=1}^n S_k)}(\IPSh^{(\I{U},R_1)}(\I{C}^1)\times\cdots\times \IPSh^{(\I{U},R_n)}(\I{C}^n)),\I{E})\to\IFun^{\cocont{\I{V}}}(\IPSh^{(\I{V}, \square_{k=1}^n R_k)}( \I{C}^1\times\cdots\times\I{C}^n),\I{E})
	\end{equation*}
	is an equivalence. Using Proposition~\ref{prop:CocompletionRelations}, this is in turn equivalent to the map
	\begin{equation*}
		\IFun(\IPSh^{(\I{U},R_1)}(\I{C}^1)\times\cdots\times \IPSh^{(\I{U},R_n)}(\I{C}^n),\I{E})_{\square_{k=1}^n S_k}\to\IFun(\I{C}^1\times\cdots\times\I{C}^n,\I{E})_{\square_{k=1}^n R_k}
	\end{equation*}
	being an equivalence. We will use induction on $n$ to show that this functor is an equivalence. If $n=1$, this is precisely the content of Proposition~\ref{prop:CocompletionRelations}. For $n>1$, the construction of $\square_{k=1}^n R_k$ and $\square_{k=1}^n S_k$ together with Lemma~\ref{lem:limitpreservationswap} imply that the above map can be identified with the morphism
	\begin{equation*}
		\begin{tikzcd}
			\IFun(\IPSh^{(\I{U},R_1)}(\I{C}^1)\times\cdots\times \IPSh^{(\I{U},R_{n-1})}(\I{C}^{n-1}),\IFun^{\cocont{\I{U}}}(\IPSh^{(\I{U}, R_n)}(\I{C}^n),\I{E}))_{\square_{k=1}^{n-1} S_k}\arrow[d]\\
			\IFun( \I{C}^1\times\cdots\times\I{C}^{n-1},\IFun(\I{C}^n,\I{E})_{R_n})_{\square_{k=1}^{n-1} R_k}.
		\end{tikzcd}
	\end{equation*}
	As Proposition~\ref{prop:CocompletionRelations} implies that the map $\IFun(\I{C}^n,\I{E})_{R_n}\to \IFun^{\cocont{\I{U}}}(\IPSh^{(\I{U}, R_n)}(\I{C}^n),\I{E})$ is an equivalence, the claim thus follows by the induction hypothesis.
\end{proof}

\chapter{Locally constant sheaves}
\numberwithin{equation}{chapter}
\label{app:locallyConstantSheaves}

This section is devoted to the study of locally constant sheaves in $\infty$-topoi.
For the entire section, let us fix a compactly generated $ \infty $-category $ \EE $. Recall that we write $ {\Shv_\EE(\BB)} = \Fun^{\lim}(\BB^\op, \EE) = \BB \otimes \EE $, where $ -\otimes- $ denotes the tensor product in $\LPrS$.
By applying $ -\otimes \EE $ to the constant sheaf functor $ \const_\BB \colon \mathcal{S} \rightarrow \BB $ we obtain an adjunction
\[
(\const_{\BB}\dashv\Gamma_\BB)\colon \EE \leftrightarrows \Shv_\EE(\BB).
\]
Similarly, by applying $-\otimes \EE$ to the adjunction $(\pi_{A}^\ast\dashv (\pi_{A})_\ast)\colon  \BB\leftrightarrows \Over{\BB}{A}$ for some $A\in\BB$, we obtain an induced adjunction
\begin{equation*}
	(\pi_{A}^\ast\dashv (\pi_{A})_\ast)\colon \Shv_{\EE}(\BB) \leftrightarrows \Shv_{\EE}(\Over{\BB}{A}).
\end{equation*}
Furhtermore, if there is an accessible left exact localisation $L \dashv i \colon \PSh(\CC) \to \BB$ we get an induced localisation $ L_\EE \dashv i_\EE \colon \PSh_\EE(\CC) \to \Shv_{\EE}(\BB) $.

\begin{definition}
	Let us fix the following terminology:
	\begin{enumerate}
		\item We call $ \const_\BB(K) $ the \emph{constant sheaf associated to} $ K\in \EE $. The objects in the essential image of $ \const_\BB $ are called \emph{constant $\EE$-valued sheaves}.
		\item We call an $\EE$-valued sheaf $ F $ \emph{constant with compact values if} it is of the form $ \const_{\BB}(K) $ for some compact object $ K \in \EE $.
		\item 	An $\EE$-valued sheaf $ F$ is called \emph{locally constant} if there is a cover $ (\pi_{A_i}) \colon \bigsqcup_{i} A_i \onto 1 $ in $\BB$ such that for every $ i $ the $\EE$-valued sheaf $ \pi_{A_i}^* F\in\Shv_{\EE}(\Over{\BB}{A_i})$ is constant.
		\item We call an $\EE$-valued sheaf $ F $ \emph{locally constant with compact values} if we can find a cover $ (s_i)_i \colon\bigsqcup_i A_i\onto 1$ in $\BB$ such that $ s_i^* F $ is constant with compact values.
		\item We denote by $ \LConst^\EE(\BB) $ the full subcategory of $\Shv_{\EE}(\BB)$ spanned by the locally constant sheaves, and by $ \LConst_\compact^{\EE}(\BB) $ the full subcategory spanned by the locally constant sheaves with compact values.
	\end{enumerate}
\end{definition}

The key result that we will show in this section is the following:

\begin{proposition}
	\label{prop:Localsectionsofsheafification}
	Suppose that $ L \colon \PSh(\CC) \to \BB $ is a left exact and accessible localisation, and let $ F$ be an $\EE$-valued presheaf on $\CC$.
	Then for any $ c \in \CC $, any $K \in \EE^\omega$ and any map $ K \to L_\EE F (c) $ there is a collection of morphisms $ (s_i\colon c_i \to c)_{i\in I} $ in $\CC$ such that $ (Ls_i)\colon \bigsqcup_i L(c_i)\onto L(c) $ is a cover in $\BB$ and for any $ i \in I $ the composite $ K \to L_\EE F(c) \xrightarrow{s_i^*} L_\EE F(c_i)  $ factors as a composite $ K \xrightarrow{m_i} F(c_i) \xrightarrow{\varepsilon_F (c_i)} L_\EE F(c_i) $ for some $ m_i \colon K \to F(c_i) $.
\end{proposition}

As an immediate consequence we obtain the following:

\begin{corollary}
	\label{prop:MapsBetweenLconstThingsAreLocallyConst}
	Let $ f \colon \const_\BB(K) \rightarrow \const_\BB(M) $ be a morphism in $ \Shv_\EE(\BB) $ where $ K $ is compact.
	Then there is a cover $ (\pi_{A_i}) \colon \bigsqcup_i A_i\onto 1 $ in $\BB$ and maps $ f_i \colon K \rightarrow M $ in $\EE$ for each $i$ such that $ \pi_{A_i}^* f$ is equivalent to $  \const_{\BB_{/A_i}}(f_i) $.
\end{corollary}
\begin{proof}
	We may pick a left exact accessible localisation $ L \colon \PSh(\CC) \to  \BB$ where $ \CC $ has a final object $ 1 $.
	The morphism $ f $ corresponds to a map $ \tilde{f} \colon K \to \Gamma\const_\BB(M) = \const_\BB(M)(1)$.
	By Proposition~\ref{prop:Localsectionsofsheafification} we may now find a covering $ (\pi_{Lc_i}) \colon \bigsqcup_i L(c_i) \onto  1$ and commutative squares
	\[\begin{tikzcd}
		{ M = \underline{M}(c_i)} & {\const_\BB(M)(c_i)} \\
		K & {\const_\BB(M)(1)}
		\arrow[from=1-1, to=1-2]
		\arrow[from=2-1, to=2-2, "\tilde{f}"]
		\arrow[from=2-2, to=1-2,"\pi_{Lc_i}^*"]
		\arrow["{m_i}", from=2-1, to=1-1]
	\end{tikzcd}\]
	where $ \underline{M} $ denotes the constant $ M $-valued presheaf.
	Let $ f_i = \const_\BB(m_i)$. 
	Then the above square translates to the statement that $ \pi_{Lc_i}^*(f_i) $ is equivalent to $ \const_{\BB_{/Lc_i}}(f_i) $, and the claim follows.
\end{proof}

\begin{corollary}
	\label{cor:LConstPO}
	The full subcategory $  \LConst_\compact^{\EE}(\BB)\into\EE\otimes\BB $ is closed under finite colimits and retracts.
\end{corollary}
\begin{proof}
	We start by showing the claim about finite colimits.
	Since $\LConst_\compact^{\EE}(\BB)$ contains the final object it suffices to see that it is closed under pushouts.
	So let us consider a span $ F \leftarrow G \rightarrow H $ in $ \LConst_\compact^{\EE}(\BB)$. 
	We may pass to a cover in $\BB$ to assume that $ F,G $ and $ H $ are constant.
	Thus by proposition~\ref{prop:MapsBetweenLconstThingsAreLocallyConst} we may further pass to a cover so that we can assume that the span above is given by applying $ \const_\BB $ to a span in $ \EE^\compact $.
	So the claim follows since $ \EE^\compact $ is closed under finite colimits and $ \const_\BB $ preserves finite colimits. 
	The proof that $\LConst_\compact^{\EE}(\BB)$ is closed under retracts proceeds in the same way.
\end{proof}

In order to prove Proposition~\ref{prop:Localsectionsofsheafification}, we first need to treat the special case where $\EE=\SS$ and $K=1$:
\begin{lemma}
	\label{lem:MapsBetweenLConstAreLocallyConstSpaces}
	Let $F \in \PSh(\CC)$ and let $ f \colon 1\rightarrow L F (c) $ be a map for some $ c \in \CC$.
	Then there is a collection of morphisms $ (s_i\colon c_i \to c)_{i\in I} $ in $\CC$ such that $ (Ls_i)\colon \bigsqcup_i L(c_i)\onto L(c) $ is a cover in $\BB$ and maps $ m_i \colon 1 \rightarrow F(c_i) $ for each $i$ such that $ s_i^* f$ is equivalent to the composite $ 1 \xrightarrow{m_i} F(c_i) \xrightarrow{\varepsilon_F (c_i)} LF(c_i)  $.
\end{lemma}
\begin{proof}
	We pick a cover $ (t_j) \colon \bigsqcup_j d_j \onto F $ in $ \PSh(\CC) $.
	Consider the pullback square
	\[\begin{tikzcd}
		{\bigsqcup_j A_j} & {\bigsqcup_i d_j} \\
		c & LF
		\arrow[from=1-2, to=2-2]
		\arrow["f"', from=2-1, to=2-2]
		\arrow[from=1-1, to=2-1]
		\arrow[from=1-1, to=1-2]
	\end{tikzcd}\]
	in $ \PSh(\CC) $.
	Covering each $ A_j \in \PSh(\CC)$ with representables $ c_k^j $ then yields the desired collection of maps $ (s_k^j) \colon \bigsqcup_{j,k}c_k^j \onto c $.
\end{proof}
To reduce the general case to the above lemma we use the ideas of \cite[\S 2]{haine2021}.
Indeed, the fact that $\EE$ is by assumption compactly generated means that we may identify $\EE\simeq\Fun^{\lex}((\EE^\compact)^\op,\SS)$. Consequently, we obtain an equivalence $\Shv_{\EE}(\BB)\simeq \Fun^{\lex}((\EE^\compact)^\op,\BB)$. In light of these identifications, the adjunction $ L_\EE \dashv i_\EE \colon \PSh_{\EE}(\CC) \to \Shv_{\EE}(\BB)$ translates into the adjunction $\Fun^{\lex}((\EE^\compact)^\op,\PSh(\CC))\leftrightarrows\Fun^{\lex}((\EE^\compact)^\op,\BB)$ that is obtained by postcomposition with $L \dashv i$. An analogous observation shows that for $ c \in \CC $ the evaluation functor $ \ev_c^\EE \colon  \Shv_{\EE}(\BB) \to \EE $ is equivalent to the functor $ \ev_{c,*} \colon \Fun^{\lex}((\EE^\compact)^\op,\BB)\to \Fun^{\lex}((\EE^\compact)^\op,\SS) $ given by composing with $ \ev_c \colon \BB \to \SS $.

\begin{proof}[{Proof of Proposition~\ref{prop:Localsectionsofsheafification}}]
	Since $ K $ is compact, the above discussion and Yoneda's lemma allow us to identify $ K \to LF(c) $ with a map $ f \colon 1 \to LF(c)(K) \simeq L(F(K))(c)  $.
	Therefore we are in the situation of Lemma~\ref{lem:MapsBetweenLConstAreLocallyConstSpaces} and get a collection of morphisms $ (s_i\colon c_i \to c)_{i\in I} $ in $\CC$ such that $ (Ls_i)\colon \bigsqcup_i L(c_i)\onto L(c) $ is a cover in $\BB$ and maps $ n_i \colon 1 \rightarrow F(K)(c) $ such that for each $i$ we have a commutative square
	\[\begin{tikzcd}
		{F(K)(c_i)} & {L(F(K))(c_i)} \\
		1 & {L(F(K))(c)}.
		\arrow["f", from=2-1, to=2-2]
		\arrow["{s_i^*}", from=2-2, to=1-2]
		\arrow[from=1-1, to=1-2]
		\arrow["{n_i}"', from=2-1, to=1-1]
	\end{tikzcd}\]
	Via Yoneda's lemma the maps $ n_i $ now yield the desired maps $ m_i \colon K \to F(c_i) $.
\end{proof}

\bibliographystyle{halpha}
\bibliography{references.bib}
\end{document}